\documentclass[a4paper]{book}

\usepackage[english]{babel}
\usepackage{a4}
\usepackage{amsfonts, amssymb, amsmath, amsthm}
\usepackage{caption} 
\usepackage{mathtools} 
\usepackage{fouridx} 
\usepackage{rotating}
\usepackage{url} 
\usepackage{enumitem}

\setlist[itemize]{leftmargin=0.5cm}

\usepackage{tikz} 
\usetikzlibrary{decorations.markings, matrix, arrows, external, fit, calc}

\tikzstyle{map} = [->, font=\scriptsize]
\tikzstyle{linj} = [left hook->, font=\scriptsize]
\tikzstyle{rinj} = [right hook->, font=\scriptsize]
\tikzstyle{sur} = [->>, font=\scriptsize]
\tikzstyle{cell} = [double,double equal sign distance,-implies, shorten >= 3pt, shorten <= 3pt, font=\scriptsize]
\tikzstyle{cell125} = [cell, shorten >= 1pt, shorten <= 1pt]
\tikzstyle{eq} = [double,double equal sign distance]

\tikzstyle{iso} = [above, sloped, inner sep=1.5pt]
\tikzstyle{nat} = [above, sloped, inner sep=2pt]
\tikzstyle{desc} = [fill=white, inner sep=2pt]

\tikzstyle{textbaseline} = [baseline={([yshift=-2.8pt]current bounding box.center)}]

\tikzstyle{barred} = [decoration={markings, mark=at position 0.5 with {\draw[-] (0,-1.5pt) -- (0,1.5pt);}}, postaction ={decorate}]

\tikzstyle{math} = [matrix of math nodes, row sep=1em, column sep=1em, text height=1.5ex, text depth=0.25ex, nodes in empty cells]
\tikzstyle{math125em} = [math, row sep=1.25em, column sep=1.25em]
\tikzstyle{math175em} = [math, row sep=1.75em, column sep=1.75em]
\tikzstyle{math2em} = [math, row sep=2em, column sep=2em]

\providecommand{\nc}{\pgfmatrixnextcell}
\providecommand{\wip}{\tikzset{external/force remake}}

\DeclareFontFamily{U}  {MnSymbolB}{}
\DeclareFontShape{U}{MnSymbolB}{m}{n}{
    <-6>  MnSymbolB5
   <6-7>  MnSymbolB6
   <7-8>  MnSymbolB7
   <8-9>  MnSymbolB8
   <9-10> MnSymbolB9
  <10-12> MnSymbolB10
  <12->   MnSymbolB12}{}

\DeclareFontShape{U}{MnSymbolB}{b}{n}{
    <-6>  MnSymbolB-Bold5
   <6-7>  MnSymbolB-Bold6
   <7-8>  MnSymbolB-Bold7
   <8-9>  MnSymbolB-Bold8
   <9-10> MnSymbolB-Bold9
  <10-12> MnSymbolB-Bold10
  <12->   MnSymbolB-Bold12}{}

\DeclareSymbolFont{MnSyB}         {U}  {MnSymbolB}{m}{n}
\SetSymbolFont{MnSyB}       {bold}{U}  {MnSymbolB}{b}{n}

\DeclareMathSymbol{\nrcirclearrowleft}{\mathrel}{MnSyB}{250}

\makeatletter
\newcommand*{\@old@slash}{}\let\@old@slash\slash
\def\slash{\relax\ifmmode\delimiter"502F30E\mathopen{}\else\@old@slash\fi}
\makeatother

\makeatletter
\def\slashedarrowfill@#1#2#3#4#5{%
  $\m@th\thickmuskip0mu\medmuskip\thickmuskip\thinmuskip\thickmuskip
   \relax#5#1\mkern-7mu%
   \cleaders\hbox{$#5\mkern-2mu#2\mkern-2mu$}\hfill
   \mathclap{#3}\mathclap{#2}%
   \cleaders\hbox{$#5\mkern-2mu#2\mkern-2mu$}\hfill
   \mkern-7mu#4$%
}
\def\rightslashedarrowfill@{%
  \slashedarrowfill@\relbar\relbar\mapstochar\rightarrow}
\newcommand\xslashedrightarrow[2][]{%
  \ext@arrow 0055{\rightslashedarrowfill@}{#1}{#2}}
\makeatother

\def\slashedrightarrow{\xslashedrightarrow{}}

\newtheorem{theorem}{Theorem}[chapter]
\newtheorem{lemma}[theorem]{Lemma}
\newtheorem{corollary}[theorem]{Corollary}
\newtheorem*{ncorollary}{Corollary}
\newtheorem{proposition}[theorem]{Proposition}
\newtheorem*{nproposition}{Proposition}

\theoremstyle{definition}
\newtheorem{definition}[theorem]{Definition}
\newtheorem*{ndefinition}{Definition}

\theoremstyle{remark}
\newtheorem{remark}[theorem]{Remark}
\newtheorem*{nremark}{Remark}
\newtheorem{example}[theorem]{Example}

\newtheorem*{convention}{Convention}

\providecommand{\corref}[1]{Corollary~\ref{#1}}
\providecommand{\defref}[1]{Definition~\ref{#1}}
\providecommand{\exref}[1]{Example~\ref{#1}}
\providecommand{\figref}[1]{Figure~\ref{#1}}
\providecommand{\lemref}[1]{Lemma~\ref{#1}}
\providecommand{\propref}[1]{Proposition~\ref{#1}}
\providecommand{\remref}[1]{Remark~\ref{#1}}
\providecommand{\thmref}[1]{Theorem~\ref{#1}}
\providecommand{\tableref}[1]{Table~\ref{#1}}

\providecommand{\chapref}[1]{Chapter~\ref{#1}}
\providecommand{\secref}[1]{Section~\ref{#1}}
\providecommand{\appref}[1]{Appendix~\ref{#1}}

\theoremstyle{theorem}
\newtheorem*{internalcommaobjects}{\propref{internal double comma objects}}

\newtheorem*{liftingcommaobjects}{\propref{lifting double comma objects}}
\newtheorem*{lefthomlaxstructurecell}{\lemref{lemma:left hom lax structure cell}}
\newtheorem*{maintheorem}{\thmref{right colax forgetful functor lifts all weighted colimits}}

\providecommand{\ndash}{\nobreakdash-}

\providecommand{\of}{\circ}
\providecommand{\iso}{\cong}

\providecommand{\xrar}[1]{\xrightarrow{#1}}
\providecommand{\xlar}[1]{\xleftarrow{#1}}
\providecommand{\xsrar}[1]{\xslashedrightarrow{#1}}
\providecommand{\Rar}{\Rightarrow}
\providecommand{\xRar}[1]{\xRightarrow{#1}}

\providecommand{\into}{\hookrightarrow}

\providecommand{\eps}{\varepsilon}

\DeclareMathOperator{\dash}{--}
\providecommand{\ndash}{\nobreakdash-}

\providecommand{\tens}{\otimes}

\providecommand{\ul}[1]{\underline{#1}{}}
\providecommand{\ull}[1]{\ul{\ul{#1}}{}}
\providecommand{\ulll}[1]{\ul{\ul{\ul{#1}}}{}}

\providecommand{\Sum}{\Sigma}

\providecommand{\brcs}[1]{\lbrace #1 \rbrace}

\providecommand{\brks}[1]{\lbrack #1 \rbrack}
\providecommand{\bigbrks}[1]{\bigl\lbrack #1 \bigr\rbrack}
\providecommand{\Bigbrks}[1]{\Bigl\lbrack #1 \Bigr\rbrack}
\providecommand{\pars}[1]{\left(#1\right)}
\providecommand{\bigpars}[1]{\bigl(#1\bigr)}
\providecommand{\Bigpars}[1]{\Bigl(#1\Bigr)}
\providecommand{\lns}[1]{\lvert#1\rvert}

\providecommand{\ls}[2]{{}_{#1} #2} 

\providecommand{\set}[1]{\brcs{#1}}

\providecommand{\isect}{\cap}

\providecommand{\sym}{\Sigma}
\providecommand{\act}[2]{#1 \cdot #2}

\providecommand{\card}[1]{\lns{#1}}

\providecommand{\natarrow}{\stackrel{.}{\rightarrow}}

\providecommand{\map}[3]{#1\colon#2\to#3}

\providecommand{\emb}[3]{#1\colon#2\into#3}

\providecommand{\nat}[3]{#1\colon#2\natarrow#3}
\providecommand{\cell}[3]{#1\colon#2\Rightarrow#3}

\providecommand{\hmap}[3]{#1\colon#2\slashedrightarrow#3}
\providecommand{\mmap}[3]{#1\colon#2\rightrightarrows#3}

\providecommand{\inv}[1]{{#1}^{-1}}

\DeclareMathOperator{\id}{id}

\DeclareMathOperator{\ob}{ob}

\providecommand{\ladj}{\dashv}

\providecommand{\catvar}[1]{\mathcal{#1}}
\providecommand{\fun}[2]{\brks{#1, #2}}

\providecommand{\psfun}[2]{\fun{#1}{#2}_{\textup{ps}}}

\providecommand{\2}{\mathsf 2}
\providecommand{\C}{\catvar C}
\providecommand{\D}{\catvar D}
\providecommand{\E}{\catvar E}
\renewcommand{\L}{\catvar L}
\providecommand{\K}{\catvar K}
\providecommand{\M}{\catvar M}

\providecommand{\V}{\catvar V}

\renewcommand{\SS}{\mathbb S}
\providecommand{\GG}{\mathbb G}
\providecommand{\NN}{\mathbb N}

\providecommand{\Set}{\mathsf{Set}}

\providecommand{\cat}{\mathsf{Cat}} 
\providecommand{\Cat}{\underline{\mathsf{Cat}}} 
\providecommand{\twoCat}{2\text{-}\cat}
\providecommand{\enCat}[1]{#1\text-\cat}
\providecommand{\inCat}[1]{\cat(#1)}
\providecommand{\psCat}[1]{\inCat{\ps{#1}}}

\providecommand{\op}[1]{#1^\textup{op}}
\providecommand{\co}[1]{#1^\textup{co}}

\providecommand{\ps}[1]{\widehat{#1}}
\providecommand{\globe}{\mathbb G}

\providecommand{\colim}{\coliml\nolimits}
\DeclareMathOperator*{\coliml}{colim}

\providecommand{\ucat}[1]{\lns{#1}}
\providecommand{\fren}[2]{#1\brks{#2}}

\providecommand{\alg}[1]{#1\text-\mathsf{Alg}}
\providecommand{\cAlg}[1]{\alg{#1}_\textup c}
\providecommand{\psAlg}[1]{\alg{#1}_\textup{ps}}

\providecommand{\Prof}{\mathsf{Prof}}
\providecommand{\Mat}[1]{#1\text-\mathsf{Mat}}
\providecommand{\enProf}[1]{#1\text-\Prof}
\providecommand{\Span}[1]{\mathsf{Span}(#1)}
\providecommand{\inProf}[1]{\Prof(#1)}
\providecommand{\psProf}[1]{\inProf{\ps{#1}}}
\providecommand{\Mod}[1]{\mathsf{Mod}(#1)}

\providecommand{\hc}{\odot}
\providecommand{\Hc}{\bigodot}
\providecommand{\lhom}{\triangleleft}
\providecommand{\rhom}{\triangleright}

\DeclareMathOperator{\coev}{coev}
\DeclareMathOperator{\ev}{ev}

\providecommand{\lDbl}{\mathsf{Dbl}_\textup l}
\providecommand{\wnDbl}{\mathsf{Dbl}_\textup{wn}}
\providecommand{\nDbl}{\mathsf{Dbl}_\textup n}
\providecommand{\psDbl}{\mathsf{Dbl}_\textup{ps}}
\providecommand{\lEquip}{\mathsf{Equip}_\textup l}
\providecommand{\nEquip}{\mathsf{Equip}_\textup n}
\providecommand{\psEquip}{\mathsf{Equip}_\textup{ps}}
\providecommand{\qlEquip}{\lEquip^\textup q}
\providecommand{\qnEquip}{\nEquip^\textup q}

\providecommand{\cdiag}[5]{#1 \xlar{#2} #3 \xslashedrightarrow{#4} #5}
\providecommand{\ldiag}[5]{#1 \xslashedrightarrow{#2} #3 \xrar{#4} #5}

\DeclareMathOperator{\lan}{lan}

\providecommand{\ol}[1]{\overline{#1}}

\providecommand{\prom}[1]{#1\textup-\mathsf{Prom}}

\providecommand{\rcProm}[1]{\prom{#1}_\textup{rc}}

\providecommand{\rcMonProf}[1]{#1\textup-\mathsf{MonProf}_{\textup{rc}}}
\providecommand{\lMonProf}[1]{#1\textup-\mathsf{MonProf}_{\textup l}}
\providecommand{\rcsMonProf}[1]{#1\textup-\mathsf{sMonProf}_{\textup{rc}}}

\providecommand{\rcDblProf}{\mathsf{DblProf}_{\textup{rc}}}

\providecommand{\Tens}{\bigotimes}

\providecommand{\inhom}[1]{\brks{#1}}

\providecommand{\fmc}{M}
\providecommand{\fsmc}{S}

\makeindex

\begin{document}
\begin{titlepage}
\vspace*{\stretch{1}}
\begin{center}\bf
{\LARGE Algebraic weighted colimits}\\
\vspace{2cm}
{\Large by}\\
\vspace{15mm}
{\LARGE Seerp Roald Koudenburg.}\\
\vspace{5cm} {\Large A thesis submitted for the degree of\\ Doctor of Philosophy.}\\
\vspace{4cm}
{\large Department of Pure Mathematics\\ School of Mathematics and Statistics\\ The University of Sheffield}\\
\bigskip
{\large November 2012.} \vspace{1cm}
\end{center}
\end{titlepage}

\pagenumbering{roman}

\chapter*{Abstract}
\addcontentsline{toc}{chapter}{Abstract}
	In this thesis weighted colimits in $2$-categories equipped with promorphisms are studied. Such colimits include most universal constructions with counits, like ordinary colimits in categories, weighted colimits in enriched categories, and left Kan extensions.
	
	In the first chapter we recall the notion of $2$-categories equipped with promorphisms (also simply called equipments), that provide a coherent way of adding bimodule-like morphisms to a $2$-category. In the second chapter we recall two ways of defining weighted colimits in equipments. Most important to us is their original definition, introduced by Wood, in equipments that are endowed with a closed structure. The second notion, of what we call pointwise weighted colimits, was introduced by Grandis and Par\'e. It requires no extra structure and generalises Street's notion of pointwise left Kan extensions in $2$-categories. The main result of the second chapter gives a condition, on closed equipments, under which these two notions coincide.
	
	In the third chapter we consider monads on equipments. The main idea of this thesis, given in the fourth chapter, generalises the notions of lax and colax morphisms, of algebras over a $2$-monad, to notions of lax and colax promorphisms, of algebras over a monad on an equipment. One of these, that of right colax promorphisms, is well suited to the construction of weighted colimits. In particular, given a monad $T$ on an equipment $\K$, we will show that $T$-algebras, colax $T$-morphisms and right colax $T$-promorphisms form a double category $\rcProm T$. Although weaker than equipments, double categories still allow definition of weighted colimits, and our main result states that the forgetful functor $\rcProm T \to \K$ lifts all weighted colimits whenever $\K$ is closed, under some mild conditions on $T$.

\chapter*{Acknowledgements}
\addcontentsline{toc}{chapter}{Acknowledgements}

	First of all I would like to thank Simon Willerton. He has been a very good supervisor to me, on one hand letting me free to choose what I wanted to study, while on the other hand urging me (sometimes many times, for I did not always listen!) to read papers that turned out crucial in solving the goals I had chosen. Most of all I would like to thank him for his encouragement and patience, especially during the long time before I finally obtained some `decent' results.
	
	Many other people have contributed to this thesis; I would like to thank David Barnes, Emily Burgunder, Thomas Cottrell, Nick Gurski, Richard Hepworth and Bruno Vallette for useful conversations. Also I would like to thank Thomas Athorne, Eugenia Cheng, Thomas, Jonathan Elliott, Nick and Simon, who formed the audience of the category seminars in which I could present a (still primitive) version of the work in this thesis. Especially I would like to thank Ieke Moerdijk for always being interested in my work, and the discussions we had every time he visited Sheffield, often leading me to new ideas.
	
	Finally I thank the University of Sheffield, for their financial support of my PhD.
	
	My time in Sheffield has been great! I have met so many nice people here, and from so many different parts of the world. Let me start by thanking my housemates Arjun, Harry, Leigh and Tim, with whom I enjoyed playing card games while listening to Zee Radio! Even before that, in my first year, Ramesh and Manoj: moving far away from home is so much easier when meeting such good friends so quickly. That I have enjoyed my work in the Hicks building so much is because of Laura, Ian, Tony, Constanze, Bruce, Zacky, Raj, Panos, Tom, James, Phil, Lakshmi, Shazia, Kharia, Nadia, Muhammad, Martin, Julio, Yimin, Susan, Greg, Vikki, Mark, Thomas, Tom, Jonathan, Serena, Ayesha, Mabruka, Boom, Mike and Samuel. And especially the jokers, whose friendship is very dear to me: Suli, Victor and Wahiba.
	
	Throughout my stay in Sheffield, how nice has it been to go home! To go back home for a long weekend every other month, to see my friends and family, to stroke the cat and to share dinner with my family, to eat the meal my mother has prepared! Suddenly everything is again like it used to be, when visiting Marie-Louise, Sjouke, Johnson, Wouter and Bart. I am very happy that now, after moving back home, I am able to see my parents and sister Myrthe every day!

	\tableofcontents

	\mainmatter
	\chapter*{Introduction}\addcontentsline{toc}{chapter}{Introduction}\markboth{INTRODUCTION}{INTRODUCTION}\label{chapter:introduction}
	The main subject of this thesis is a result of Getzler, that was given in \cite{Getzler09}, which gives conditions ensuring that the left Kan extension of a symmetric monoidal functor, between symmetric monoidal categories, is again symmetric monoidal. Briefly speaking, the aim is to generalise this result to similar situations, like that of double functors between double categories, that of continuous order preserving maps between ordered compact Hausdorff spaces, or that of monoidal globular functors between monoidal globular categories.
	
	Getzler's result concerns the well known notion of left Kan extension, which we recall first. Given a pair of functors $\map dAM$ and $\map jAB$, between small categories and with common source, loosely speaking this notion describes the `best approximation' to an extension $B \to M$ of $d$ along $j$. Formally, such approximations are considered to be pairs $(e, \zeta)$, where $\map eBM$ is a functor and $\nat\zeta d{e \of j}$ is a natural transformation, as on the left below. The \emph{left Kan extension} of $d$ along $j$ is the approximation $(l, \eta)$ that is the `best' in the sense that the transformation $\zeta$ of any other approximation $(e, \zeta)$ factors uniquely through $\eta$ as $\nat{\zeta'}le$:
\begin{equation} \label{equation:left Kan extension universality}
	\begin{split}
	\begin{tikzpicture}[textbaseline]
		\matrix(m)[math, column sep=2.5em, row sep=4em]{A \nc \nc B \\ \nc M \nc \\};
		\path[map]	(m-1-1) edge node[above] {$j$} (m-1-3)
												edge node[below left] {$d$} coordinate (p) (m-2-2)
								(m-1-3) edge node[below right] {$e$} (m-2-2)
								(p) edge[shorten >= 18pt, shorten <= 18pt] node[nat, below] {.} node[above, inner sep=2pt, xshift=-1pt] {$\zeta$} (m-1-3);
	\end{tikzpicture}
	\quad = \quad \begin{tikzpicture}[textbaseline]
		\matrix(m)[math, column sep=2.5em, row sep=4em]{A \nc \nc B \\ \nc M \nc \\};
		\path[map]	(m-1-1) edge node[above] {$j$} (m-1-3)
												edge node[below left] {$d$} coordinate (p) (m-2-2)
								(m-1-3) edge[bend right=20] node[above left, inner sep=1pt] {$l$} coordinate (q) (m-2-2)
												edge[bend left=20]  node[below right] {$e$} coordinate (r)(m-2-2)
								(q) edge[shorten >= 3.5pt, shorten <= 3.5pt] node[nat, below] {.} node[above right, inner sep=1pt] {$\zeta'$} (r);
		\path[map, transform canvas={xshift=-5pt, yshift=2pt}]	(p) edge[shorten >= 18pt, shorten <= 18pt] node[nat, below] {.} node[above, inner sep=2pt, xshift=-1pt] {$\eta$} (m-1-3);
		\draw (m-2-2) node[yshift=1pt] {$\phantom M$.};
	\end{tikzpicture}
	\end{split}
\end{equation}
	The pair $(l, \eta)$ is called \emph{universal} among the pairs $(e, \zeta)$, while the transformation $\eta$ is called the \emph{unit}; one also says that $\eta$ \emph{exhibits} $l$ as the left Kan extension of $d$ along $j$. Any two universal pairs are isomorphic.
	
	If the target category $M$ is cocomplete then the left Kan extension $l$ always exists, and can be constructed using `coends' in $M$:
\begin{equation} \label{equation:left Kan extension as coends}
	l(z) = \int^{y \in A} dy \tens B(jy, z),
\end{equation}
	where $dy \tens B(jy, z)$ denotes the copower $\coprod_{f \in B(jy, z)} dy$ in $M$. A coend is a certain colimit, see for example Section IX.6 of \cite{MacLane98}.	Choosing a left Kan extension $\map{l = \lan_j d}BM$ for every $\map dAM$ gives the left adjoint in the adjunction
\begin{equation} \label{equation:left Kan adjunction}
	\begin{split}
	\begin{tikzpicture}
		\matrix(m)[math, column sep=1.5em]{\lan_j \colon \fun AM \nc \fun BM \colon j^*, \\ };
		\path[transform canvas={yshift=3.5pt}, map] (m-1-1) edge (m-1-2);
		\path[transform canvas={yshift=-3.5pt}, map]	(m-1-2) edge (m-1-1);
		\path[color=white] (m-1-1) edge node[color=black, font=\scriptsize] {$\bot$} (m-1-2);
	\end{tikzpicture}
	\end{split}
\end{equation}
	where $\fun AM$ denotes the category of functors $A \to M$ and natural transformations, and where $j^*$ is given by precomposition with $j$. The units $\eta$ in \eqref{equation:left Kan extension universality} form the components of the unit transformation of this adjunction.

	More generally, the notion of left Kan extension above makes sense not only in the $2$-category $\Cat$ of categories, functors and natural transformations, but in any $2$-category. In particular we can consider Kan extensions in the $2$-category $\Cat_\tens$ of symmetric monoidal categories, symmetric monoidal functors and monoidal transformations. To be precise, here we mean `pseudomonoidal' categories and functors, whose coherence maps are invertible. Examples of such categories include the category of vector spaces and their tensor products, or the cartesian monoidal structure on any category with finite products. We thus obtain the notion of `symmetric monoidal left Kan extension' of a symmetric monoidal functor $\map dAM$ along a symmetric monoidal functor $\map jAB$: it is a universal pair $(l_\tens, \eta_\tens)$ among all pairs $(e_\tens, \zeta_\tens)$ that consist of a symmetric monoidal functor $\map{e_\tens}BM$ and a monoidal natural transformation $\nat{\zeta_\tens}d{e_\tens \of j}$.
	
	Writing $\map U{\Cat_\tens}\Cat$ for the forgetful $2$-functor, Getzer's result gives conditions on $j$ and $M$ under which any ordinary left Kan extension $l$ of $Ud$ along $Uj$, in $\Cat$, can be lifted to a symmetric monoidal left Kan extension $l_\tens$ of $d$ along $j$ in $\Cat_\tens$, such that $Ul_\tens = l$. Thus in that case symmetric monoidal left Kan extensions can be computed simply as ordinary left Kan extensions. The following is a reformulation of his result, which we have also restricted to the unenriched setting. The original statement will be recalled later, as \propref{Getzler's proposition}.
	\begin{nproposition}[Getzler]
		Let $\map jAB$ be a symmetric monoidal functor, and suppose that $M$ is cocomplete. If
		\begin{enumerate}[label=(\alph*)]
			\item the binary tensor product $\map{\tens}{M \times M}M$ preserves colimits in both variables;
			\item for each $x$ in $A$ and $z_1, \dotsc, z_n$ in $B$, the canonical map
				\begin{equation} \label{equation:Getzler's canonical map}
					\int^{(y_1, \dotsc, y_n) \in A^{\times n}} A(x, y_1 \tens \dotsb \tens y_n) \times \prod_{i = 1}^n B(jy_i, z_i) \to B(jx, z_1 \tens \dotsb \tens z_n),
				\end{equation}
				that is induced by applying $j$ to the map in $A$, tensoring the maps in $B$, and composing using the coherence map $j(y_1 \tens \dotsb \tens y_n) \xrar\iso jy_1 \tens \dotsb \tens jy_n$, is an isomorphism, 
		\end{enumerate}
		then, for any symmetric monoidal functor $\map dAM$, the ordinary left Kan extension of $Ud$ along $Uj$ can be lifted to a symmetric monoidal left Kan extension of $d$ along $j$. In particular, in this case the adjunction \eqref{equation:left Kan adjunction} lifts to an adjunction
		\begin{displaymath}
			\begin{tikzpicture}
				\matrix(m)[math, column sep=1.5em]{\lan_j \colon \psfun AM & \psfun BM \colon j^*, \\ };
				\path[transform canvas={yshift=3.5pt}, map] (m-1-1) edge (m-1-2);
				\path[transform canvas={yshift=-3.5pt}, map]	(m-1-2) edge (m-1-1);
				\path[color=white] (m-1-1) edge node[color=black, font=\scriptsize] {$\bot$} (m-1-2);
			\end{tikzpicture}
		\end{displaymath}
		where $\psfun AM$ denotes the category of symmetric monoidal functors $A \to M$ and monoidal transformations.
	\end{nproposition}
	Condition (b) means that any map $\map f{jx}{z_1 \tens \dotsb \tens z_n}$ in $B$ can be decomposed as
	\begin{equation} \label{equation:decomposition}
		f = \bigbrks{jx \xrar{jg} j(y_1 \tens \dotsb \tens y_n) \xrar\iso jy_1 \tens \dotsb \tens jy_n \xrar{h_1 \tens \dotsb \tens h_n} z_1 \tens \dotsb \tens z_n},
	\end{equation}
	where $\map gx{y_1 \tens \dotsb \tens y_n}$ lies in $A$ and the maps $\map{h_i}{jy_i}{z_i}$ lie in $B$. This decomposition should be unique in the sense that if a second tuple $(g', h'_1, \dotsc, h'_n)$ also decomposes $f$ in the same way, then it is identified with $(g, h_1, \dotsc, h_n)$ in the coend above. To get a feeling for this condition we shall consider a detailed example, in which we show how Getzler's proposition can be used to prove that the forgetful functor from `bicommutative Hopf algebras' to cocommutative coalgebras has a left adjoint.
	
\section*{Example: Hopf monoids}
	An important source of symmetric monoidal functors are `algebras over PROPs', as follows. Informally, a `PROP' is an algebraic structure that defines a type of algebra that lives in a symmetric monoidal category, and whose operations may have multiple inputs and multiple outputs. A typical example of such an algebra is a `Hopf monoid':
\begin{ndefinition}
	Let $M = (M, \tens, 1)$ be a symmetric monoidal category. A \emph{Hopf monoid} $H$ is a bimonoid
	$(H, \mu, \eta, \Delta, \eps)$ in $M$, with multiplication $\map\mu{H \tens H}H$ and unit $\map\eta 1H$, and comultiplication $\map\Delta H{H \tens H}$ and counit $\map\eps H1$, that is equipped with a morphism $\map SHH$ satisfying the identity
	\begin{displaymath}
		\mu \of (S \tens \id) \of \Delta = \eta \of \eps = \mu \of (\id \tens S) \of \Delta,
	\end{displaymath}
	of maps $H \to H$. An Hopf monoid $H$ is called \emph{bicommutative} if it is both commutative and cocommutative as a bimonoid.
\end{ndefinition}
	
	A \emph{PROP} is a symmetric strict monoidal category $\mathsf P$ that has the set of natural numbers $\NN = \set{0, 1, 2, \dotsc}$ as objects, on which its tensor product acts as addition: $n_1 \tens \dotsb \tens n_m = n_1 + \dotsb + n_m$. A \emph{$\mathsf P$-algebra} $A$, in a symmetric monoidal category $M$, is a symmetric pseudomonoidal functor $\map A{\mathsf P}M$, while morphisms of $\mathsf P$\ndash algebras are monoidal transformations. 
	
	As an example, Wadsley shows in \cite{Wadsley08} that the PROP $\mathsf H$, whose algebras are bicommutative Hopf monoids, has as hom-sets $\mathsf H(m, n)$ the sets of $n \times m$-matrices with integer coefficients. Composition in $\mathsf H$ is the usual matrix multiplication and, if $\map f{m_1}{n_1}$ and $\map g{m_2}{n_2}$ are two matrices, their tensor product $\map{f \tens g}{m_1 + m_2}{n_1 + n_2}$ is the block matrix $\bigpars{\begin{smallmatrix} f & 0 \\ 0 & g \\ \end{smallmatrix}}$. Note that $0$ is both the initial and terminal object of $\mathsf H$, because there is precisely one matrix with $n$ rows and no columns, and one matrix with no rows and $n$ columns. Any $\mathsf H$-algebra $\map A{\mathsf H}M$ induces the structure of a bicommutative Hopf monoid on the image $A = A(1)$ by taking the composites
\begin{align*}
	\mu &= \bigbrks{A \tens A \iso A(2) \xrar{A(11)} A}& \eta &= \bigbrks{1 \iso A(0) \xrar{A(0)} A} \\
	\Delta &= \bigbrks{A \xrar{A\bigpars{\begin{smallmatrix} 1 \\ 1 \\ \end{smallmatrix}}} A(2) \iso A \tens A} &	\eps &= \bigbrks{A \xrar{A(0)} A(0) \iso 1} \\
	S &= \bigbrks{A \xrar{A(-1)} A}, & &
\end{align*}
	where the isomorphisms are the coherence maps of $A$. That these composites satisfy the bimonoid axioms and the Hopf monoid axiom follows directly from various matrix equations. For example, the identity $\mu \of (S \tens \id) \of \Delta = \eta \of \eps$ is the $A$-image of the matrix equation
\begin{displaymath}
	\begin{pmatrix}1 & 1 \\\end{pmatrix} \begin{pmatrix}-1 & 0 \\ 0 & 1 \\\end{pmatrix} \begin{pmatrix}1 \\ 1 \\\end{pmatrix} = \begin{pmatrix}0 \\\end{pmatrix}.
\end{displaymath}
	Thus every strict $\mathsf H$-algebra induces a bicommutative Hopf monoid in $M$; that this extends to an isomorphism between $\psfun{\mathsf H}M$ and the category of bicommutative Hopf monoids in $M$ is shown in \cite{Wadsley08}.
	
	Commutative monoids are modelled by a simpler PROP $\mathsf F$, that is defined as follows. Writing $\brks n$ for the set $\set{1, \dotsc, n}$ if $n \geq 1$, and $\brks 0 = \emptyset$, the morphisms $m \to n$ of $\mathsf F$ are functions $\map f{\brks m}{\brks n}$. Composition is the usual composition of functions, while the tensor product $\map{f \tens g}{m_1 + m_2}{n_1 + n_2}$ is given by
	\begin{displaymath}
		(f \tens g)(i) = \begin{cases}
			f(i) & \text{if $i \leq m_1$;} \\
			g(i - m_1) + n_1 & \text{if $i > m_1$.}
		\end{cases}
	\end{displaymath}
	It is readily seen that $\mathsf F$-algebras in $M$ are precisely the commutative monoids in $M$. Dually algebras for the opposite category $\op{\mathsf F}$ are the cocommutative comonoids in $M$.
	
	Notice that the PROP $\op{\mathsf F}$ can be embedded into $\mathsf H$ by mapping each morphism $m \to n$, corresponding to a function $\map f{\brks n}{\brks m}$, to the $n \times m$-matrix $(f_{ij})$ that is given by $f_{ij} = 1$ if $fi = j$ and $0$ otherwise. Thus the image of $\op{\mathsf F}$ in $\mathsf H$ is the subPROP consisting of all matrices that contain precisely one non-zero entry in each row, whose coefficient is $1$. This gives a symmetric strict monoidal functor $\emb j{\op{\mathsf F}}{\mathsf H}$ whose induced action
	\begin{displaymath}
		\map{j^*}{\psfun{\mathsf H}M}{\psfun{\op{\mathsf F}}M}
	\end{displaymath}
	on the categories of algebras is simply the forgetful functor, that maps each bicommutative Hopf monoid to its underlying cocommutative comonoid.
	
	We claim that the embedding $j$ satisfies condition (b) of Getzler's proposition. To see this we have to show that every $(n_1 + \dotsb + n_s) \times m$-matrix $f$ in $\mathsf H$ decomposes as $f = (h_1 \tens \dotsb \tens h_s) \of jg$, where the $h_i$ are $n_i \times l_i$-matrices in $\mathsf H$ and $jg$ is a $(l_1 + \dotsb + l_s) \times m$-matrix in $j(\op{\mathsf F})$, as in \eqref{equation:decomposition} above. It is easy to find matrices $h_i$, $i = 1, \dotsc, s$, and $jg$ that satisfy this equation: for example we can take $l_i = m$ and let $h_i$ be the blocks of $f$ as shown below. As $jg$ we can then simply take the block matrix of identity matrices $I_m$ of dimension $m$, as on the right. 
	\begin{displaymath}
		f = \begin{pmatrix} h_1 \\ h_2 \\ \vdots \\ h_s \end{pmatrix} =	\begin{pmatrix}
				 h_1 & 0 & \hdots & 0 \\
				 0 & h_2 & \ddots & \vdots \\
				 \vdots & \ddots & \ddots & 0 \\
				 0 & \hdots & 0 & h_s \\
			\end{pmatrix} \begin{pmatrix} I_m \\ I_m \\ \vdots \\ I_m \end{pmatrix}
	\end{displaymath}
	It is now not hard to check that the assignment $f \mapsto (g, h_1, \dotsc, h_s)$ does in fact give an inverse to the canonical map \eqref{equation:Getzler's canonical map}, so that condition (b) is satisfied.

	We conclude that, for every cocomplete symmetric monoidal category $M$ satisfying condition (a), there exists an adjuction of categories of PROP-algebras
\begin{displaymath}
	\begin{tikzpicture}
		\matrix(m)[math, column sep=1.5em]{\lan_j \colon \psfun{\op{\mathsf F}}M & \psfun{\mathsf H}M \colon j^*. \\ };
		\path[transform canvas={yshift=3.5pt}, map] (m-1-1) edge (m-1-2);
		\path[transform canvas={yshift=-3.5pt}, map]	(m-1-2) edge (m-1-1);
		\path[color=white] (m-1-1) edge node[color=black, font=\scriptsize] {$\bot$} (m-1-2);
	\end{tikzpicture}
\end{displaymath}
	Hence we obtain a formal proof for the following corollary.
\begin{ncorollary}
	Let $M$ be a cocomplete symmetric monoidal category, whose binary tensor product $\map{\tens}{M \times M}M$ preserves colimits in both variables. The forgetful functor $\map{j^*}{\psfun{\mathsf H}M}{\psfun{\op{\mathsf F}}M}$, from bicommutative Hopf monoids in $M$ to cocommutative comonoids in $M$, has a left adjoint.
\end{ncorollary}
	Of course, for an explicit description of the free bicommutative Hopf monoid on a cocommutative comonoid $C = (C, \Delta, \eps)$ in $M$, we have to unpack the coend \eqref{equation:left Kan extension as coends}, which may not be easy. In the noncommutative case, with $M$ the category of vector spaces, a description of free Hopf algebras generated by coalgebras has been given by Takeuchi \cite{Takeuchi71}.
	
	On the other hand we may consider the subPROP $\Sigma \subset \mathsf H$ that is generated by the symmetries of $\mathsf H$, and write $\emb i\Sigma{\mathsf H}$ for the embedding. In other words $\Sigma$ is the free symmetric strict monoidal category on the single object $1 \in \mathsf H$, and the composite
\begin{displaymath}
	\psfun{\mathsf H}M \xrar{i^*} \psfun\Sigma M \iso M
\end{displaymath} 
	is the forgetful functor from bicommutative Hopf monoids in $M$ to $M$. That this forgetful functor does not admit a left adjoint is considered `folklore', e.g.\ it is not possible to construct a `free Hopf algebra' on a vector space. In regard to Getzler's proposition, this manifests as the fact that the embedding $\emb i\Sigma{\mathsf H}$ does not satisfy condition (b). Indeed $\Sigma$ contains only isomorphisms, so that it is impossible to decompose the matrix $\map{\bigpars{\begin{smallmatrix} 1 \\ 1 \\ \end{smallmatrix}}}12$, that models the comultiplication, as a composition $1 \xrar g 2 \xrar{h_1 \tens h_2} 2$, with $g \in \Sigma$, as is required in \eqref{equation:decomposition}.

\section*{Similar settings}
	The question that Getzler's proposition answers can be asked in many similar settings, as follows. In abstract terms, we might consider any $2$-monad $T$ on any $2$-category $\C$, and let $\alg T$ denote the $2$-category of $T$-algebras, $T$-morphisms and $T$-cells (see \secref{section:monad examples}). Like in any $2$\ndash category, the notion of left Kan extension can be defined in both $\C$ and $\alg T$. Thus, given $T$-algebras $A$, $B$ and $M$, and $T$-morphisms $\map jAB$ and $\map dAM$, assuming their left Kan extension $l$ exists in $\C$ we may ask the question
	\begin{quote}
		``Under what conditions on $j$ and $M$ does $l$ lift to a left Kan extension in $\alg T$?''
	\end{quote}
	Getzler's result, in the way stated above, answers this question in the specific case of the `free symmetric strict monoidal category'-monad on the $2$-category $\C = \Cat$ of categories, functors and transformations. To convince the reader that this situation is common three interesting examples are listed below.

	As the first example we consider the ordered compact Hausdorff spaces as studied by Tholen \cite[Example 2]{Tholen09}. Such spaces are preordered sets equipped with a compact Hausdorff topology, which is compatible with the ordering. Since preordered sets can be thought of as categories enriched over the category $\2 = \pars{\bot \to \top}$ of truth values, they can be considered as the objects of the $2$-category $\C = \enCat\2$, where left Kan extensions of order preserving maps are defined. Precisely, the left Kan extension of a pair of order preserving maps $\map dAM$ and $\map jAB$ is given by
	\begin{displaymath}
		(\lan_j)d(z) = \sup\set{dx : jx \leq z}
	\end{displaymath}
	provided these suprema exist in $M$, see the discussion following \exref{example:weighted colimits in V-Prof}. Moreover there is a $2$-monad on $\enCat\2$, called the `ultrafilter'-monad, whose algebras are precisely the ordered compact Hausdorff spaces. In this case the question becomes: ``Given continuous order preserving maps $j$ and $d$, when is $\lan_j d$ above a continuous left Kan extension?''

	The second example is that of double categories, which we will use throughout the thesis. Informally, a double category $\K$ consists of
	\begin{itemize}
		\item[-] objects $A, B, C, \dotsc,$
		\item[-] vertical morphisms denoted $\map fAB$,
		\item[-] horizontal morphisms denoted $\hmap JAB$, with a barred arrow,
		\item[-] cells $\phi$ that are shaped as squares, with vertical and horizontal morphisms as edges, as shown below:
		\begin{displaymath}
			\begin{tikzpicture}
				\matrix(m)[math175em]{A & B \\ C & D. \\};
				\path[map]  (m-1-1) edge[barred] node[above] {$J$} (m-1-2)
														edge node[left] {$f$} (m-2-1)
										(m-1-2) edge node[right] {$g$} (m-2-2)
										(m-2-1) edge[barred] node[below] {$K$} (m-2-2);
				\path[transform canvas={shift={($(m-1-2)!(0,0)!(m-2-2)$)}}] (m-1-1) edge[cell] node[right] {$\phi$} (m-2-1);
			\end{tikzpicture}
		\end{displaymath}
	\end{itemize}
	Pairs of vertical morphisms can be composed, as can pairs of horizontal morphisms, while pairs of cells can be composed both vertically, along a common horizontal edge, and horizontally, along a common vertical edge. In `strict' double categories both compositions are strictly associative. Ones in which horizontal composition is only associative up to coherent invertible cells are called `pseudo' double categories; these we will use often. Underlying any double category $\K$ is a diagram
	\begin{displaymath}
		\begin{tikzpicture}
			\matrix(m)[math175em]{\K_1 & \K_0 \\};
    \path[map]  (m-1-1) edge[transform canvas={yshift=2pt}] node[above] {$L$} (m-1-2)
                        edge[transform canvas={yshift=-2pt}] node[below] {$R$} (m-1-2);
		\end{tikzpicture}
	\end{displaymath}
	of categories and functors, as follows. The category $\K_0$ consists of the objects and vertical morphisms of $\K$, while $\K_1$ has horizontal morphisms as objects; its maps $J \to K$ are cells $\phi$ with horizontal source $J$ and target $K$. The functors $L$ and $R$ map $\hmap JAB$ to its source $A$ and target $B$ respectively, while the cell $\phi$ above is mapped to $L\phi = f$ and $R\phi = g$. Thus the assignment $\K \mapsto (\K_1 \rightrightarrows \K_0)$ forgets the horizontal composition of $\K$. The diagram $\K_1 \rightrightarrows \K_0$ above can be considered as an internal category in the category $\ps{\globe_1}$ of presheaves on $\globe_1 = \pars{0 \rightrightarrows 1}$ and thus as an object of the $2$-category $\inCat{\ps{\globe_1}}$ of internal categories, internal functors and internal natural transformations in $\ps{\globe_1}$, where we have a notion of Kan extension. Moreover there is a $2$-monad on $\inCat{\ps{\globe_1}}$, that is induced by the `free category'-monad on graphs, whose algebras are the double categories as introduced above. Thus, taking double functors $\map S\K\L$ and $\map F\K\M$, we can forget the structure of horizontal compositions and regard them as functors $S'$ and $F'$ in $\inCat{\ps{\globe_1}}$. Supposing that the left Kan extension $L'$ of $F'$ along $S'$ exists in $\inCat{\ps{\globe_1}}$, this leads to the same question: ``When can $L'$ be lifted to a left Kan extension of double functors?''
	
	Generalising the previous example, the globular categories considered by Batanin in his paper \cite{Batanin98} are contravariant functors $\op{\globe} \to \cat$, where $\globe$ is the `globe category'. Batanin considers a  monoidal structure on such categories and uses the resulting `monoidal globular categories' to define weak $n$-categories. This situation too is like the ones described above: there is a $2$-monad on $\inCat{\ps\globe}$, induced by the `free $\omega$-category'-monad on globular sets, whose algebras are the monoidal globular categories.

	The aim of the thesis is to answer the question asked above in a general, conceptual way that, among others, can be applied to each of the three examples given above.

\section*{The idea}
	In each of the three examples given above, the $2$-category $\C$ and the $2$-monad $T$ that we considered are part of a bigger structure, in the following sense.
	\begin{quote}
		There exists a `closed' pseudo double category $\K$ that `equips $\C$ with promorphisms'; moreover $T$ can be extended to form a monad $T'$ on $\K$.
	\end{quote}
	To understand what this means, we first remark that each double category $\K$ induces a $2$-category $V(\K)$ consisting of its objects, vertical morphisms and vertical cells (cells whose horizontal source and target are identity morphisms). That $\K$ `equips $\C$ with promorphisms' means that $V(\K) \iso \C$ and that each morphism $\map fAB$ in $\C$ has both a `companion' $\hmap{B(f, \id)}AB$ and a `conjoint' $\hmap{B(\id, f)}BA$ in $\K$ (for the details see the comments following \defref{definition:equipment}). One can think of the companion $B(f, \id)$ as being the horizontal morphism `isomorphic' to $f$, and the conjoint $B(\id, f)$ being the horizontal morphism `adjoint' to $f$. In this case the horizontal morphisms of $\K$ are called the `promorphisms of $\C$', while $\K$ is called an \emph{equipment}.
	
	The typical example of a $2$-category that can be equipped with promorphisms is the $2$-category $\Cat$ of categories, functors and natural transformations. Here the promorphisms are the profunctors: a profunctor \mbox{$\hmap JAB$} is simply a functor $\map J{\op A \times B}\Set$. That is, the pseudo double category $\K = \mathsf{Prof}$ consists of categories, functors, profunctors and natural transformations. The composition of profunctors is similar to the composition of relations (see \exref{example:unenriched profunctors}), while the companion $\hmap{B(f, \id)}AB$ and conjoint $\hmap{B(\id, f)}BA$, of a functor $\map fAB$, are given by the representable profunctors $(x, y) \mapsto B(fx, y)$ and $(y,x) \mapsto B(y, fx)$.
	
	Furthermore, a double category $\K$ is `closed' whenever, for each promorphism $\hmap JAB$, the functors given by pre- and postcomposition with $J$ both have right adjoints, which generalises the notion of a closed monoidal category. The equipment $\mathsf{Prof}$ is closed: the right adjoint of precomposition with a profunctor $\hmap JAB$ maps any profunctor $\hmap KAC$ to the `left hom' profunctor $\hmap{J \lhom K}BC$ that is given by
	\begin{displaymath}
		\pars{J \lhom K}(y, z) = \set{{}\text{natural transformations } J(\dash, y) \natarrow K(\dash, z)}.
	\end{displaymath}
	These left homs allow us to define many kinds of `well behaved' weighted colimits in $\Cat$, as follows. Given a profunctor $\hmap JAB$ (the weight) as well as a functor $\map dAM$ (the diagram), the $J$-weighted colimit of $d$, if it exists, is a functor $\map{\colim_J d}BM$ that `represents' the left hom $J \lhom M(d, \id)$ in the following sense: there is a natural isomorphism of profunctors
	\begin{displaymath}
		M(\colim_J d, \id) \iso J \lhom M(d, \id).
	\end{displaymath}
	One can verify that, taking various kinds of weights $J$, this generalises the usual notions of colimits in $\Cat$, as described in \tableref{table:weighted colimits in Cat}. For us, the most important of these is that colimits weighted by companions are left Kan extensions, that is there is a natural isomorphism
	\begin{equation}\label{equation:left Kan extension and colimits weighted by companions}
		\map{\lan_j d \iso \colim_{B(j, \id)} d}BM,
	\end{equation}
	for any pair of functors $\map dAB$ and $\map jAB$.
	In fact the functor $\colim_{B(j, \id)} d$ satisfies a stronger property, that defines it as the `pointwise' left Kan extension of $d$ along $j$, see \cite[Corollary X.5.4]{MacLane98}. We shall discuss the notion of pointwise left Kan extensions at the end of this introduction.
	\begin{table}
	\begin{center}
	\begin{tabular}{@{}ll@{}}
		choice of weight $\hmap JAB$ & resulting weighted colimit \\
		\hline
		$B = *$ (terminal category) & ordinary weighted colimit \\
		$J = B(j, \id)$ where $\map jAB$ & `pointwise' left Kan extension \\
		$J = *(!, \id)$ where $\map !A*$ & ordinary (conical) colimit
	\end{tabular}
	\caption{Various types of weighted colimits in $\Cat$.} \label{table:weighted colimits in Cat}
	\end{center}
	\end{table}
	
	For well behaved $\V$ the $2$-category $\enCat\V$ of $\V$-enriched categories can be equipped with `$\V$-enriched profunctors', generalising the profunctors of $\Cat$, and the resulting double category $\enProf\V$ will be closed. The same is true for the $2$-category $\inCat\E$ of internal categories in a well behaved category $\E$ with finite limits (e.g.\ $\E$ is a presheaf category). In each of these cases the left homs can be used to define weighted colimits, just like we did in the case of $\Cat$ above.
	
	We remark that weighted colimits can also be defined in pseudo double categories, that are neither equipments nor closed: Grandis and Par\'e consider in \cite{Grandis-Pare08} the notion of, what they call, `Kan extensions' in ordinary double categories. We shall see in \propref{weighted colimits equivalent to top absolute left Kan extensions} that, in a closed pseudo double category, every weighted colimit, in our sense, is a Kan extension in their sense.
	
	Thus, given a $2$-monad $T$ on $\C$, we will assume there exists a closed double category $\K$ that equips $\C$ with promorphisms, together with a monad $T'$ on $\K$ that restricts to $T$ on $\C$. In each of the examples given in the previous section such $\K$ and $T'$ exist. Remember that, in general, a `colax morphism' $A \to B$ between two $T$-algebras $A$ and $B$ is a morphism $\map fAB$ in $\C$ that is equipped with a `structure cell' $\cell{\bar f}{f \of a}{b \of Tf}$, as shown on the left below, where $\map a{TA}A$ and $\map b{TB}B$ define the $T$-algebra structures on $A$ and $B$. This structure cell is required to satisfy an associativity and unit axiom. There is a dual notion of `lax morphism' $\map gAB$, in which the direction of $\bar g$ is reversed, as shown on the right below. Colax and lax morphisms whose structure cells are invertible are called `pseudomorphisms'.
	\begin{displaymath}
		\begin{tikzpicture}[baseline]
			\matrix(m)[math175em]{ TA & A \\ TB & B \\};
			\path[map]	(m-1-1) edge node[above] {$a$} (m-1-2)
													edge node[left] {$Tf$} (m-2-1)
									(m-1-2) edge node[right] {$f$} (m-2-2)
									(m-2-1) edge node[below] {$b$} (m-2-2);
			\path	(m-1-2) edge[cell, shorten >=7pt, shorten <= 7pt] node[below right] {$\bar f$} (m-2-1);
		\end{tikzpicture}
		\qquad\qquad\qquad\qquad\begin{tikzpicture}[baseline]
			\matrix(m)[math175em]{ TA & TB \\ A & B \\};
			\path[map]	(m-1-1) edge node[above] {$Tg$} (m-1-2)
													edge node[left] {$a$} (m-2-1)
									(m-1-2) edge node[right] {$b$} (m-2-2)
									(m-2-1) edge node[below] {$g$} (m-2-2);
			\path	(m-1-2) edge[cell, shorten >=8pt, shorten <=8pt] node[below right] {$\bar g$} (m-2-1);
		\end{tikzpicture}
	\end{displaymath}
	In the case of the `free symmetric strict monoidal category'-monad on $\Cat$, a colax morphism $\map fAB$ is a symmetric colax monoidal functor, that comes equipped with structure maps $f(x_1 \tens \dotsb \tens x_n) \to fx_1 \tens \dotsb \tens fx_n$. Getzler's proposition involves `pseudomonoidal functors', for which these structure maps are invertible.
	
	The main idea of this thesis is to generalise the above notions, of morphisms between $T$-algebras in $\C$, to notions of promorphisms between $T'$-algebras in $\K$. For example, we shall define a `lax promorphism' $A \slashedrightarrow B$ to be a promorphism $\hmap JAB$ that comes equipped with a structure cell
	\begin{displaymath}
		\begin{tikzpicture}
			\matrix(m)[math175em]{ T'A & T'B \\ A & B; \\};
			\path[map]	(m-1-1) edge[barred] node[above] {$T'J$} (m-1-2)
													edge node[left] {$a$} (m-2-1)
									(m-1-2) edge node[right] {$b$} (m-2-2)
									(m-2-1) edge[barred] node[below] {$J$} (m-2-2);
			\path[transform canvas={shift=($(m-1-2)!0.5!(m-2-2)$)}]	(m-1-1) edge[cell] node[right] {$\bar J$} (m-2-1);
		\end{tikzpicture}
	\end{displaymath}
	compare the structure cell $\bar g$ of a lax morphism above. We shall see that colax structures on a morphism $\map fAB$ correspond precisely to lax structures on its companion $\hmap{B(f, \id)}AB$.
	
	Considering lax promorphisms $\hmap JAB$ whose structure cells $\bar J$ are invertible in $\K_1$ (the category of horizontal morphisms and cells) does not give the right notion of pseudopromorphism, because $\bar J$ being invertible in $\K_1$ implies that the structure maps $\map a{T'A}A$ and $\map b{T'B}B$ are invertible. Instead, we shall consider the two horizontal cells (i.e.\ cells with identities as vertical source and target)
	\begin{displaymath}
		\begin{tikzpicture}[textbaseline]
			\matrix(m)[math175em]{T'A & T'B & B \\ T'A & A & B \\};
			\path[map]	(m-1-1) edge[barred] node[above] {$T'J$} (m-1-2)
									(m-1-2) edge[barred] node[above] {$B(b, \id)$} (m-1-3)
									(m-2-1) edge[barred] node[below] {$A(a, \id)$} (m-2-2)
									(m-2-2) edge[barred] node[below] {$J'$} (m-2-3);
			\path				(m-1-1) edge[eq] (m-2-1)
									(m-1-2) edge[cell] node[right] {$\lambda\bar J$} (m-2-2)
									(m-1-3) edge[eq] (m-2-3);
		\end{tikzpicture}
		\qquad\qquad\text{and}\qquad\qquad\begin{tikzpicture}[textbaseline]
			\matrix(m)[math175em]{A & T'A & T'B \\ A & B & T'B \\};
			\path[map]	(m-1-1) edge[barred] node[above] {$A(\id, a)$} (m-1-2)
									(m-1-2) edge[barred] node[above] {$T'J$} (m-1-3)
									(m-2-1) edge[barred] node[below] {$J$} (m-2-2)
									(m-2-2) edge[barred] node[below] {$B(\id, b)$} (m-2-3);
			\path				(m-1-1) edge[eq] (m-2-1)
									(m-1-2) edge[cell] node[right] {$\rho\bar J$} (m-2-2)
									(m-1-3) edge[eq] (m-2-3);
		\end{tikzpicture}
	\end{displaymath}
	that correspond to $\bar J$ under the `orthogonal flipping' operation that was introduced by Grandis and Par\'e in \cite[Section 1.6]{Grandis-Pare04} (see also \propref{left and right cells}), and call the lax promorphism $J$ a `left pseudopromorphism' or `right pseudopromorphism' when respectively $\lambda\bar J$ or $\rho\bar J$ is invertible. For example, condition (b) of Getzler's proposition is equivalent to asking that the lax structure on the companion $B(j, \id)$ of the monoidal functor $\map jAB$ is right pseudo.
	
	Under a mild condition on the monad $T'$, the two notions of pseudopromorphism lead to corresponding notions of colax promorphism: a `left' and a `right' variant. In \chapref{chapter:algebraic promorphisms} we will see that the second variant, that of `right colax promorphisms', is well suited to constructing algebraic weighted colimits. More precisely, we will show in \secref{section:algebraic promorphisms} that $T'$-algebras, colax morphisms, right colax promorphisms form a pseudo double category $\rcProm T$, so that the main result of the thesis can be stated as follows.

\begin{maintheorem}
	Let $T'$ be a `right suitable normal' monad on a closed equipment $\K$. The forgetful functor $\map{U^{T'}}{\rcProm {T'}}\K$ `lifts' all weighted colimits. Moreover its lift of a weighted colimit $\map{\colim_J d}BM$, where $\map dAM$ is a pseudomorphism and $\hmap JAB$ is a right pseudopromorphism, is a pseudomorphism whenever the canonical vertical cell
	\begin{equation} \label{equation:canonical cell}
		\colim_{T'J} (m \of T'd) \Rar m \of T'(\colim_J d)
	\end{equation}
	is invertible, where $\map m{T'M}M$ is the structure map of $M$.
\end{maintheorem}
	As before, that $U^{T'}$ lifts all weighted colimits means that, if the weighted colimit $k = \colim_{U^{T'}J} U^{T'} d$ exists in $\K$ then the weighted colimit $l = \colim_J d$ exists in $\rcProm{T'}$, and can be chosen such that $U^{T'} l = k$. The assumption that $T'$ is a `right suitable normal' monad is the mild condition on $T'$ that was mentioned above, which is necessary to be able to define right colax promorphisms. 
	
	If we take for $T'$ the extension of the `free symmetric strict monoidal category'-monad on $\Cat$, to the pseudo double category $\Prof$ of profunctors, then the second assertion of the theorem above reduces to Getzler's proposition. More precisely, for monoidal functors $\map dAM$ and $\map jAB$, condition (b) of the latter means that the companion $B(j, \id)$ of $j$ is a right pseudopromorphism while condition (a) is equivalent to the invertibility of \eqref{equation:canonical cell}.

\section*{Pointwise weighted colimits}
	In his paper \cite{Street74} Street gave the following refinement of the notion of left Kan extension, in any $2$-category $\C$ that has `comma objects' (a certain $2$-limit; see \defref{definition:comma object}), like the $2$-category $\inCat\E$ of categories, functors and transformations internal to a category $\E$ with finite limits. Street says that the cell $\cell\eta d{l \of j}$ in the diagram below exhibits $l$ as the \emph{pointwise} left Kan extension of $d$ along $j$ if, for each $\map fCB$, the composition of cells below exhibits $l \of f$ as the left Kan extension of $d \of \pi_A$ along $\pi_C$, where $\pi$ defines $j \slash f$ as the comma object of $j$ and $f$.
	\begin{displaymath}
		\begin{tikzpicture}
			\matrix(m)[math2em, row sep=0.75em, column sep=2.7em]{j \slash f \nc A \nc[-0.5em] \\ \nc \nc M \\ C \nc B \nc \\};
			\path[map]	(m-1-1) edge node[above] {$\pi_A$} (m-1-2)
													edge node[left] {$\pi_C$} (m-3-1)
									(m-1-2) edge node[above right] {$d$} (m-2-3)
													edge node[left] {$j$} (m-3-2)
									(m-3-1) edge node[below] {$f$} (m-3-2)
									(m-3-2) edge node[below right] {$l$} (m-2-3);
			\path				(m-1-2) edge[cell, shorten >=14.5pt, shorten <=14.5pt] node[above left] {$\pi$} (m-3-1)
									($(m-1-2)!0.5!(m-2-3)$) edge[cell, shorten >=9pt, shorten <=9pt] node[below right, yshift=3pt] {$\eta$} (m-3-2);
		\end{tikzpicture}
	\end{displaymath}
	For example the left Kan extensions in $\Cat$, that are given as colimits weighted by companions, as in \eqref{equation:left Kan extension and colimits weighted by companions}, are pointwise.
	
	As we remarked before, Grandis and Par\'e have defined left Kan extensions in any pseudo double category in \cite{Grandis-Pare08}; in the same paper they also define the stronger pointwise variants, by generalising Street's notion above and using `double comma objects'. Modifying their definition somewhat, so that it fits in better with our notion of weighted colimits, we likewise consider \emph{pointwise} weighted colimits in \secref{section:pointwise weighted colimits}. The main result of \chapref{chapter:weighted colimits} (\thmref{weighted colimits are pointwise if double commas are strong}) gives a condition on equipments under which pointwise weighted colimits coincide with the ordinary ones. For example, the pseudo double category $\inProf\E$ of internal profunctors in a category $\E$ with finite limits, that equips $\inCat\E$ with promorphisms, satisfies this condition.
	
	Finally, in \thmref{algebraic weighted colimits are pointwise if double commas are strong} we shall prove that, roughly, given a monad $T$ on an equipment $\K$, the equipment  $\rcProm T$ of right colax $T$-profunctors satisfies the condition of \thmref{weighted colimits are pointwise if double commas are strong} whenever $\K$ does, so that in that case weighted colimits in $\rcProm T$ are also pointwise.

  \chapter{$2$-Categories equipped with promorphisms}
	This chapter contains the preliminary definitions that will be used in the coming chapters. We start by recalling the definition of pseudo double categories and, using this, we recall the definition of $2$-categories equipped with promorphisms. These, also simply called `equipments', will form the main setting of the theory presented here. Throughout, the main examples that we consider are the equipments $\enProf\V$ of $\V$\ndash enriched categories and $\V$-profunctors, where $\V$ is some `well behaved' symmetric monoidal category, and the equipments $\inProf\E$ of categories internal to $\E$ and internal profunctors, where $\E$ is some `well behaved' category with finite limits. 
	
	For most parts of this chapter we follow the papers \cite{Shulman08} by Shulman and \cite{Cruttwell-Shulman10} by Cruttwell and Shulman, borrowing much of their notation as well. In particular we name double categories after their horizontal morphisms, like they do. However, we shall have no need to consider the more general `virtual double categories' that are the main subject of \cite{Cruttwell-Shulman10}. The notion of $2$-categories equipped with promorphisms was originally introduced by Wood \cite{Wood82}.

\section{Pseudo double categories}
	In this section we recall the formal definition of the double categories, that we described in the introduction. We will consider the `pseudo' variants, whose horizontal compositions satisfy the associativity and unit axioms only up to coherent invertible cells.
	
	In the introduction double categories were described as generalisations of $2$\ndash cat\-egories where, in contrast to the ordinary case, there are two types of morphisms, one drawn vertically and the other drawn horizontally. Moreover the cells of such double categories are shaped like squares, with both a pair of vertical morphisms as source and target, as well as a pair of horizontal morphisms. In the introduction we briefly mentioned the example of the double category $\Prof$, with categories as objects and functors and profunctors as vertical and horizontal morphisms. Before recalling the formal definition of double categories we shall describe, in more detail, another typical example.
	
	Let us consider rings $A, B, \dotsc$: besides the usual ring homomorphisms $\map fAB$ between rings there are other objects that we may consider as morphisms $A \to B$, namely $(A, B)$-bimodules. Remember that an $(A, B)$-bimodule $J$ is an abelian group $J$ equipped with a left action $\map l{A \tens J}J\colon (a, x) \mapsto \act a x$ and a right action $\map r{J \tens B}J$ that are associative and unital, and that commute (the tensor products here are tensor products of abelian groups). To distinguish it from ring homomorphisms we shall write $\hmap JAB$. Of course, bimodules only deserve to be considered as morphisms if they allow composition of `composable bimodules', as well as `unit bimodules' that act as units for this composition. But they do: given two bimodules $\hmap JAB$ and $\hmap HBC$ we can take the tensor product $J \tens_B H$ as the composite $\hmap{J \hc H}AC$ and, since we have $J \tens_B B \iso J$ and $A \tens_A J \iso J$, we can take the unit $(A,A)$\ndash bimodule $A \slashedrightarrow A$ to be the ring $A$ itself, considered as an $(A, A)$-bimodule. Thus, to define composition for bimodules we have to pick, for every composable pair $J$ and $H$, a tensor product $J \tens_B H$, that is, formally, a coequaliser of the parallel pair
	\begin{equation} \label{tensor product}
		\begin{split}
		\begin{tikzpicture}
			\matrix(m)[math]{ J \tens B \tens H \nc[1.25em] J \tens H \nc J \tens_B H \\ };
			\path[map]	(m-1-1) edge[transform canvas={yshift=2pt}] node[above] {$r \tens \id$} (m-1-2)
													edge[transform canvas={yshift=-2pt}] node[below] {$\id \tens l$} (m-1-2)
									(m-1-2) edge (m-1-3);
		\end{tikzpicture}
		\end{split}
	\end{equation}
	in the category of abelian groups. Notice that, if we just pick any tensor product $J \tens_B H$ for each pair of composable bimodules $J$ and $H$ (and we do), then the isomorphism $A \tens_A J \iso J$ need not be an identity. Thus the composition of bimodules is not as strict as we are used to: instead of the composite of $A \tens_A J$ being equal to $J$ it is isomorphic to $J$. The same is true for associativity: the universal property of the coequalisers above implies that $(J \tens_B H) \tens_C K \iso J \tens_B (H \tens_C K)$, but this will in general not be an identity. However, since all these isomorphisms are `canonical', in the sense that they are obtained by universal properties, the family of all such isomorphisms is `coherent': an example of this coherence is the commuting of the famous pentagon diagram of Mac Lane \cite[Section VII.1]{MacLane98}, which states that the two ways of using the isomorphisms, that compare the tensor products of three factors, to obtain an isomorphism $\bigpars{(J \tens_B H) \tens_C K} \tens_D L \iso J \tens_B \bigpars{H \tens_C (K \tens_D L)}$, that compares the tensor product of four, coincide.
	
	Thus, to summarise, $(A, B)$-bimodules can be considered as morphisms $A \slashedrightarrow B$ which can be composed, but the resulting composition will only be associative and unital `up to coherent isomorphisms'. Now, if we think about the mathematical object containing rings, ring homomorphisms and bimodules, it is natural to wonder how morphisms of bimodules fit in. The most general of these are morphisms $J \to K$ from an $(A, B)$-bimodule $J$ to a $(C, D)$-bimodule $K$, as follows. Given two ring homomorphisms $\map fAC$ and $\map gBD$, an \emph{$(f, g)$-bilinear} map $\map\phi JK$ is a homomorphism $J \to K$ of abelian groups such that $\phi(\act a x) = \act{f(a)}{\phi(x)}$ and $\phi(\act x b) = \act{\phi (x)}{g(b)}$. Drawing ring homomorphisms vertically and bimodules horizontally, we can depict the bilinear map $\phi$ as a square shaped cell
\begin{displaymath}
  \begin{tikzpicture}
    \matrix(m)[math175em]{A & B \\ C & D. \\};
    \path[map]  (m-1-1) edge[barred] node[above] {$J$} (m-1-2)
                        edge node[left] {$f$} (m-2-1)
                (m-1-2) edge node[right] {$g$} (m-2-2)
                (m-2-1) edge[barred] node[below] {$K$} (m-2-2);
    \path[transform canvas={shift={($(m-1-2)!(0,0)!(m-2-2)$)}}] (m-1-1) edge[cell] node[right] {$\phi$} (m-2-1);
  \end{tikzpicture}
\end{displaymath}
	This is a useful representation of such maps, for it allows us to represent the composite $\cell{\psi \of \phi}JL$, of $\phi$ and a second bilinear map $\cell\psi KL$, as the two squares on the left below. We call $\psi \of \phi$ the \emph{vertical composite} of $\phi$ and $\psi$; as the diagram below suggests it is an $(h \of f, k \of g)$-bilinear map $J \to L$. 
\begin{displaymath}
	\begin{tikzpicture}[baseline]
		\matrix(m)[math175em]{A \nc B \\ C \nc D \\ E \nc F \\};
		\path[map]	(m-1-1) edge[barred] node[above] {$J$} (m-1-2)
												edge node[left] {$f$} (m-2-1)
								(m-1-2) edge node[right] {$g$} (m-2-2)
								(m-2-1) edge[barred] node[below] {$K$} (m-2-2)
												edge node[left] {$h$} (m-3-1)
								(m-2-2) edge node[right] {$k$} (m-3-2)
								(m-3-1) edge[barred] node[below] {$L$} (m-3-2);
		\path[transform canvas={shift=(m-2-1)}]
								(m-1-2) edge[cell] node[right] {$\phi$} (m-2-2)
								(m-2-2) edge[cell, transform canvas={yshift=-3pt}] node[right] {$\psi$} (m-3-2);
	\end{tikzpicture}
	\quad\qquad\qquad\qquad\qquad\begin{tikzpicture}[baseline]
		\matrix(m)[math175em]{A \nc B \nc G \\ C \nc D \nc H \\};
		\path[map]	(m-1-1) edge[barred] node[above] {$J$} (m-1-2)
												edge node[left] {$f$} (m-2-1)
								(m-1-2) edge[barred] node[above] {$M$} (m-1-3)
												edge node[right] {$g$} (m-2-2)
								(m-1-3) edge node[right] {$l$} (m-2-3)
								(m-2-1) edge[barred] node[below] {$K$} (m-2-2)
								(m-2-2) edge[barred] node[below] {$N$} (m-2-3);
		\path[transform canvas={shift=($(m-1-1)!0.5!(m-2-2)$)}]
								(m-1-2) edge[cell] node[right] {$\phi$} (m-2-2)
								(m-1-3) edge[cell] node[right] {$\chi$} (m-2-3);
	\end{tikzpicture}
\end{displaymath}
	Likewise given a third $(g, l)$-bilinear map $\cell\chi MN$ as on the right above, we can construct the \emph{horizontal composite} $\cell{\phi \tens_g \chi}{J \tens_B M}{K \tens_D N}$, by taking the unique factorisation of the composite
\begin{displaymath}
	J \tens M \xrar{\phi \tens \chi} K \tens N \to K \tens_D N
\end{displaymath}
	through the coequaliser $J \tens M \to J \tens_B M$. Again this composition is represented nicely by drawing the cells of $\phi$ and $\chi$ side-by-side, as we have done above.
	
	We conclude that to capture rings, ring homomorphisms, bimodules and bilinear maps between bimodules in a `double category' $\mathsf{Mod}$, such a double category should have rings as objects, and two types of morphisms: the vertical ones being ring homomorphisms and the horizontal ones being bimodules. Moreover it should have $2$-dimensional structure that consists of square shaped cells representing the bilinear maps between bimodules. Any two morphisms of the same type can be composed in $\mathsf{Mod}$, while any pair of `adjacent' cells can be composed, either horizontally or vertically. Finally we have seen that the composition of horizontal morphisms need not be strictly associative and unital, but only up to isomorphisms of bimodules (that is cells). This is exactly the description of a `pseudo double category', whose formal definition we shall now recall.
	
	Consider a diagram of an internal category
\begin{equation} \label{diagram:internal category}
  \begin{split}
  \begin{tikzpicture}
    \matrix(m)[math, column sep=2em]{\K_1 \times_{\K_0} \K_1 \nc \K_1 \nc \K_0, \\};
    \path[map]  (m-1-1) edge node[above] {$\hc$} (m-1-2)
                (m-1-2) edge[transform canvas={yshift=5pt}] node[above] {$L$} (m-1-3)
                        edge[transform canvas={yshift=-5pt}] node[below] {$R$} (m-1-3)
                (m-1-3) edge[transform canvas={yshift=-2.5pt}] node[above=-2pt] {$U$} (m-1-2);
  \end{tikzpicture}
  \end{split}
\end{equation}
  denoted $\K$, in the $2$-category $\Cat$ of categories, functors and natural transformations, where $\K_1 \times_{\K_0} \K_1$ is the pullback
\begin{displaymath}
  \begin{tikzpicture}
    \matrix(m)[math2em]
    { \K_1 \times_{\K_0} \K_1 & \K_1 \\
      \K_1 & \K_0 \\ };
    \path[map]  (m-1-1) edge node[above] {$\pi_2$} (m-1-2)
                        edge node[left] {$\pi_1$} (m-2-1)
                (m-1-2) edge node[right] {$L$} (m-2-2)
                (m-2-1) edge node[below] {$R$} (m-2-2);
  \end{tikzpicture}
\end{displaymath}
  and such that the following equalities hold.
\begin{equation}  \label{internal category conditions}
  \begin{split}
    \map{L \of \hc &= L \of \pi_1}{\K_1 \times_{\K_0} \K_1}{\K_0} \\
    \map{R \of \hc &= R \of \pi_2}{\K_1 \times_{\K_0} \K_1}{\K_0} \\
    \map{R \of U &= \id = L \of U}{\K_0}{\K_0}
  \end{split}
\end{equation}
  The objects of $\K_0$ are called \emph{objects} of $\K$ while the morphisms $\map fAC$ of $\K_0$ are called \emph{vertical morphisms} of $\K$. An object $J$ of $\K_1$ such that $LJ = A$ and $RJ = B$ is denoted by the barred arrow
\begin{displaymath}
  \hmap JAB
\end{displaymath}
  and called a \emph{horizontal morphism} of $\K$. A morphism $\map\phi JK$ in $\K_1$, such that $L \phi = \map fAC$ and $R \phi = \map gBD$, is depicted
\begin{displaymath}
  \begin{tikzpicture}
    \matrix(m)[math175em]{A & B \\ C & D \\};
    \path[map]  (m-1-1) edge[barred] node[above] {$J$} (m-1-2)
                        edge node[left] {$f$} (m-2-1)
                (m-1-2) edge node[right] {$g$} (m-2-2)
                (m-2-1) edge[barred] node[below] {$K$} (m-2-2);
    \path[transform canvas={shift={($(m-1-2)!(0,0)!(m-2-2)$)}}] (m-1-1) edge[cell] node[right] {$\phi$} (m-2-1);
  \end{tikzpicture}
\end{displaymath}
  and called a \emph{cell} of $\K$. We will call $J$ and $K$ the \emph{horizontal source} and \emph{target} of $\phi$, while we call $f$ and $g$ its \emph{vertical source} and \emph{target}. A cell $\cell\phi JK$ whose vertical source and target are identities is called \emph{horizontal}.
  
  The fact that $\K_0$ is a category means that we can compose two vertical morphisms $\map fAC$ and $\map hCE$ to form $h \of f$, which we will draw as $A \xrar f C \xrar h E$, as is customary. Likewise two horizontal morphisms $\hmap JAB$ and $\hmap HBF$ can be composed using the functor $\hc$ of \eqref{diagram:internal category}; we will denote the result $J \hc H$ and draw it as $A \xslashedrightarrow J B \xslashedrightarrow H F$. If the horizontal source of a cell $\psi$ coincides with the horizontal target of $\phi$ then their \emph{vertical composite} $\psi \of \phi$ can be formed in $\K_1$; likewise for a cell $\chi$ whose vertical source is equal to the vertical target of $\phi$, the \emph{horizontal composite} $\phi \hc \chi$ is given by the image of $(\phi, \chi)$ under the functor $\hc$. Notice that the vertical composition of cells is strictly associative, and that each horizontal morphism $\hmap JAB$ comes with an identity
\begin{displaymath}
  \begin{tikzpicture}
    \matrix(m)[math175em]{A & B \\ A & B \\};
    \path[map]  (m-1-1) edge[barred] node[above] {$J$} (m-1-2)
                (m-2-1) edge[barred] node[below] {$J$} (m-2-2);
    \path				(m-1-1) edge[eq] (m-2-1)
                (m-1-2) edge[eq] (m-2-2);
    \path[transform canvas={xshift=1.15em}] (m-1-1) edge[cell] node[right] {$\id_J$} (m-2-1);
  \end{tikzpicture}
\end{displaymath}
	which is a unit for vertical composition. The edges of the vertical composite $\phi \of \psi$ can be read off from the diagram obtained by drawing $\phi$ on top of $\psi$, sharing the common horizontal edge, while the diagram of $\phi$ and $\chi$ side-by-side, sharing their common vertical edge, represents the horizontal composite $\phi \hc \chi$. That the edges of these diagrams indeed coincide with the edges of $\psi \of \phi$ and $\phi \hc \chi$ follows from the fact that $L$, $R$ and $\hc$ preserve composition, as well as the identities \eqref{internal category conditions}.
	
	Thus compositions of cells can be naturally represented by joining them into a grid --- we will do so many times. Grids that consist of one or two columns always define a unique composite because $\hc$ is a functor. For example, suppose we are given four composable cells as in
\begin{displaymath}
	\begin{tikzpicture}
		\matrix(m)[math175em]{\phantom{\text x} & \phantom{\text x} & \phantom{\text x} \\ \phantom{\text x} & \phantom{\text x} & \phantom{\text x} \\ \phantom{\text x} & \phantom{\text x} & \phantom{\text x}, \\};
		\path[map]	(m-1-1) edge (m-1-2)
												edge (m-2-1)
								(m-1-2) edge (m-1-3)
												edge (m-2-2)
								(m-1-3) edge (m-2-3)
								(m-2-1) edge (m-2-2)
												edge (m-3-1)
								(m-2-2) edge (m-2-3)
												edge (m-3-2)
								(m-2-3) edge (m-3-3)
								(m-3-1) edge (m-3-2)
								(m-3-2) edge (m-3-3);
		\path[transform canvas={shift={($(m-2-1)!0.5!(m-2-2)$)}}] (m-1-2) edge[cell] node[right] {$\phi$} (m-2-2)
								(m-1-3) edge[cell] node[right] {$\chi$} (m-2-3)
								(m-2-2) edge[cell] node[right] {$\psi$} (m-3-2)
								(m-2-3) edge[cell] node[right] {$\xi$} (m-3-3);
	\end{tikzpicture}
\end{displaymath}
	then $(\psi \of \phi) \hc (\xi \of \chi) = (\psi \hc \xi) \of (\phi \hc \chi)$; this is called the \emph{interchange law}.
  
  The horizontal morphism $U_A = UA$ that is the image of $A$ under $U$ of \eqref{diagram:internal category}, is called the \emph{horizontal identity} on $A$; in diagrams we will depict it simply by $\begin{tikzpicture}[textbaseline] \path (0.2,0) edge[eq] (0.7,0); \node at (0,0) {$A$}; \node at (0.9,0) {$A$};\end{tikzpicture}$. Likewise, given a vertical morphism $\map fAC$, the \emph{horizontal identity cell} $U_f = Uf$ will be depicted as
\begin{displaymath}
  \begin{tikzpicture}
    \matrix(m)[math175em]{A & A \\ C & C. \\};
    \path[map]  (m-1-1) edge node[left] {$f$} (m-2-1)
                (m-1-2) edge node[right] {$f$} (m-2-2);
    \path (m-1-1) edge[eq] (m-1-2)
          (m-2-1) edge[eq] (m-2-2);
    \path[transform canvas={shift={($(m-1-2)!(0,0)!(m-2-2)$)}}] (m-1-1) edge[cell] node[right] {$U_f$} (m-2-1);
  \end{tikzpicture}
\end{displaymath}
  Cells whose horizontal source and target are horizontal identities, like $U_f$ above, are called \emph{vertical}.

\begin{definition} \label{definition:pseudo double category}
  A \emph{pseudo double category} $\K$ is a pseudo category internal to $\Cat$. That is, $\K$ is given by a diagram of functors \eqref{diagram:internal category} satisfying the conditions \eqref{internal category conditions}, together with natural isomorphisms
  \begin{align*}
    \map{\nat{\mathfrak a&}{\hc \of \pars{\hc \times \id}}{\hc \of \pars{\id \times \hc}}}{\K_1 \times_{\K_0} \K_1 \times_{\K_0} \K_1}{\K_1;} \\
    \map{\nat{\mathfrak l&}{\hc \of \pars{U \times \id}}{\pi_2}}{\K_0 \times_{\K_0} \K_1}{\K_1;} \\
    \map{\nat{\mathfrak r&}{\hc \of \pars{\id \times U}}{\pi_1}}{\K_1 \times_{\K_0} \K_0}{\K_1,}
  \end{align*}
  whose components are horizontal cells in $\K$, and such that the usual coherence axioms for a monoidal category or bicategory are satisfied (see e.g.\ \cite[Section VII.1]{MacLane98}).
\end{definition}
	Thus, in contrast to composition of vertical morphisms, composition of horizontal morphisms is only associative up to invertible horizontal cells
\begin{displaymath}
	\cell{\mathfrak a}{(J \hc H) \hc K}{J \hc (H \hc K)},
\end{displaymath}
	for any $A \xslashedrightarrow J B \xslashedrightarrow H C \xslashedrightarrow K D$, called the \emph{associators} of $\K$, and unital up to invertible horizontal cells
\begin{displaymath}
	\cell{\mathfrak l}{U_A \hc J}J \qquad \text{and} \qquad \cell{\mathfrak r}{J \hc U_B}J,
\end{displaymath}
	for any $\hmap JAB$, called the \emph{unitors} of $\K$.

	Given composable horizontal morphisms $\map{J_i}{A_{i-1}}{A_i}$, for $i = 1, \dotsc, n$, we will abbreviate by $J_1 \hc \dotsb \hc J_n$ their composition with right to left bracketing $J_1 \hc \bigpars{\dotsb \hc (J_{n-1} \hc J_n) \dotsb }$, and likewise for horizontal composites of cells. Moreover in writing down compositions of cells we will often leave out the associators and unitors. For example the composite
\begin{displaymath}
	J \hc H \xRar{\inv{\mathfrak l}} U_A \hc (J \hc H) \xRar{\phi \hc \psi} K \hc (L \hc M) \xRar{\mathfrak a} (K \hc L) \hc M \xRar{\chi \hc \xi} (P \hc Q) \hc U_X \xRar{\mathfrak r} P \hc Q,
\end{displaymath}
	will be abbreviated to simply
\begin{displaymath}
	J \hc H \xRar{\phi \hc \psi} K \hc L \hc M \xRar{\chi \hc \xi} P \hc Q.
\end{displaymath}
	That this does not introduce ambiguity follows from the coherence axioms: like for monoidal categories and bicategories, they imply that any two composites of $\hc$-products of (inverses of) associators, (inverses of) unitors and identity cells, that have the same source and target, are equal. Thus, any two ways of completing an abbreviation like above, by adding associators and unitors, will give the same result. In fact, by a result of Grandis and Par\'e \cite[Theorem 7.5]{Grandis-Pare99} every pseudo double category is equivalent to a strict double category, whose associators and unitors are identities.
	
	Likewise when drawing grids to represent compositions we shall always leave out associators and unitors. For example by the grid
\begin{displaymath}
	\begin{tikzpicture}
	  \matrix(m)[math175em]{A & A & B & D \\ A & C & D & D, \\};
	  \path[map]	(m-1-2) edge[barred] node[above] {$J$} (m-1-3)
												edge node[right] {$f$} (m-2-2)
								(m-1-3) edge[barred] node[above] {$D(g, \id)$} (m-1-4)
												edge node[right] {$g$} (m-2-3)
								(m-2-1) edge[barred] node[below] {$C(f, \id)$} (m-2-2)
								(m-2-2) edge[barred] node[below] {$K$} (m-2-3);
		\path				(m-1-1) edge[eq] (m-1-2)
												edge[eq] (m-2-1)
								(m-1-4) edge[eq] (m-2-4)
								(m-2-3) edge[eq] (m-2-4);
		\path[transform canvas={shift=($(m-1-2)!0.5!(m-2-2)$)}]
								(m-1-2) edge[cell] node[right] {$\ls f\eta$} (m-2-2)
								(m-1-3) edge[cell] node[right] {$\phi$} (m-2-3)
								(m-1-4) edge[cell] node[right] {$\ls g\eps$} (m-2-4);
	\end{tikzpicture}
\end{displaymath}
	that is taken from \propref{left and right cells} below, we mean the composite
\begin{displaymath}
	J \hc D(g, \id) \xRar{\inv{\mathfrak l}} U_A \hc \bigpars{J \hc D(g, \id)} \xRar{\ls f\eta \hc (\phi \hc \ls g\eps)} C(f, \id) \hc (K \hc U_D) \xRar{\id \hc \mathfrak r} C(f, \id) \hc K
\end{displaymath}
	or, in its abbreviated version, $J \hc D(g, \id) \xRar{\ls f\eta \hc \phi \hc \ls g\eps} C(f, \id) \hc K$.

	The following are examples of pseudo double categories. In the next section we shall see that each of those are `promorphism equipments'.
\begin{example} \label{example:bimodules}
	As will be clear rings, ring homomorphisms, bimodules and bilinear maps form a pseudo double category $\mathsf{Mod}$. More precisely $\mathsf{Mod}_0 = \mathsf{Rng}$, the category of rings and their homomorphisms, while $\mathsf{Mod}_1$ has bimodules as objects and bilinear maps as morphisms. The functors $L$ and $\map R{\mathsf{Mod}_1}{\mathsf{Rng}}$ map $(A, B)$-bimodules $J$ to $LJ = A$ and $RJ = B$, and each $(f, g)$-bilinear map $\cell\phi JK$ to $L\phi = f$ and $R\phi = g$. As discussed the components of the associator $\mathfrak a$ and unitors $\mathfrak l$ and $\mathfrak r$ are obtained from the universality of tensor products.
\end{example}
	
\begin{example} \label{example:unenriched profunctors}
	In the introduction we briefly mentioned the pseudo double category $\Prof$ of profunctors; in detail it is given as follows. Its objects and vertical morphisms are small categories and functors, i.e.\ $\Prof_0 = \cat$. Its horizontal morphisms \mbox{$\hmap JAB$} are \emph{profunctors}, that is functors $\map J{\op A \times B}\Set$, where $\op A$ is the dual of $A$, obtained by reversing the arrows in $A$. We think of the elements of $J(a, b)$ as morphisms, and denote them $\map jab$. Likewise we think of the actions of $A$ and $B$ on $J$ as compositions; hence, for morphisms $\map s{a'}a$ and $\map tb{b'}$, we will write $t \of j \of s = J(s, t)(j)$. To describe the cells of $\Prof$, first consider a profunctor $\hmap KCD$ and functors $\map fAC$ and $\map gBD$; we shall write $K(f, g)$ for the composite $\hmap{K \of (\op f \times g)}AB$. A cell
	\begin{displaymath}
		\begin{tikzpicture}
			\matrix(m)[math175em]{A & B \\ C & D \\};
			\path[map]  (m-1-1) edge[barred] node[above] {$J$} (m-1-2)
													edge node[left] {$f$} (m-2-1)
									(m-1-2) edge node[right] {$g$} (m-2-2)
									(m-2-1) edge[barred] node[below] {$K$} (m-2-2);
			\path[transform canvas={shift={($(m-1-2)!(0,0)!(m-2-2)$)}}] (m-1-1) edge[cell] node[right] {$\phi$} (m-2-1);
		\end{tikzpicture}
	\end{displaymath}
	of $\Prof$ is simply a natural transformation $\nat\phi J{K(f, g)}$. Such transformations clearly vertically compose so that they form a category $\Prof_1$ with profunctors as objects.
	
	The horizontal composition $J \hc_B H$ of composable profunctors $\hmap JAB$ and $\hmap HBC$ is given by the reflexive coequalisers
	\begin{equation} \label{equation:composition of profunctors}
    \coprod\limits_{b_1, b_2 \in B} J(a, b_1) \times B(b_1, b_2) \times H(b_2, c) \rightrightarrows \coprod\limits_{b \in B} J(a, b) \times H(b, c) \to (J \hc_B H)(a,c)
  \end{equation}
  of sets, for $a \in A$ and $c \in C$, where the pair of maps let $B(b_1, b_2)$ act on $J(a, b_1)$ and $H(b_2, c)$ respectively. This can be thought of as a multi-object (and cartesian) variant of the tensor product of bimodules \eqref{tensor product}. Often it is denoted shortly using the coend notation
  \begin{displaymath}
    (J \hc_B H)(a, c) = \int^{b \in B} J(a, b) \times H(b, c).
  \end{displaymath}
  The associator $\mathfrak a\colon (J \hc_B H) \hc_C K \iso J \hc_B (H \hc_C K)$ is induced by the fact that taking binary products preserves colimits in both factors. The unit profunctor $U_A$ on the category $A$ is given by its hom-objects $U_A(a_1, a_2) = A(a_1, a_2)$, with actions given by composition, while the horizontal unit $U_f$ for a functor $\map fAC$ is simply the restriction \mbox{$A(a_1, a_2) \to B(fa_1, fa_2)$} of $f$ to the set of maps $A(a_1, a_2)$. The unitor $\mathfrak l\colon U_A \hc_A J \iso J$ is induced by the actions $A(a_1, a_2) \times J(a_2, b) \to J(a_1, b)$: it is not hard to check that its inverse has as components the composites
	\begin{displaymath}
	  J(a, b) \to A(a, a) \times J(a,b) \to (U_A \hc_A J)(a,b)
	\end{displaymath}
	whose first maps are given by $j \mapsto (\id_a, j)$ while the second are insertions. The isomorphisms $U_A \hc_A J \iso J$ and $J \hc_B U_B \iso J$ are called the \emph{Yoneda isomorphisms}.
  
  Finally the horizontal composition $\phi \hc_g \psi$ of the natural transformations \begin{displaymath}
    \begin{tikzpicture}
      \matrix(m)[math175em]{A \nc B \nc E \\ C \nc D \nc F \\};
      \path[map]  (m-1-1) edge[barred] node[above] {$J$} (m-1-2)
                          edge node[left] {$f$} (m-2-1)
                  (m-1-2) edge[barred] node[above] {$H$} (m-1-3)
                          edge node[right] {$g$} (m-2-2)
                  (m-1-3) edge node[right] {$h$} (m-2-3)
                  (m-2-1) edge[barred] node[below] {$K$} (m-2-2)
                  (m-2-2) edge[barred] node[below] {$L$} (m-2-3);
      \path[transform canvas={shift={($(m-1-1)!0.5!(m-2-2)$)}}]
                  (m-1-2) edge[cell] node[right] {$\phi$} (m-2-2)
                  (m-1-3) edge[cell] node[right] {$\psi$} (m-2-3);
    \end{tikzpicture}
  \end{displaymath}
  is given by the following composition of transformations
  \begin{align} \label{equation:composition of natural transformations}
    (J \hc_B H)(a, c) =& \int^{b \in B} J(a, b) \times H(b, c) \natarrow \int^{b \in B} K(fa, gb) \times L(gb, hc) \notag \\
      \natarrow& \int^{d \in D} K(fa, d) \times L(d, hc) = \pars{K \hc_D L}\pars{f, h}(a, c),
  \end{align}
  where the first is induced by the products $\phi_{a, b} \times \psi_{b,c}$ and the second exists by the universality of the second coend.
\end{example}
\begin{example} \label{example:enriched profunctor equipments}
	There is a $\V$-enriched variant of $\Prof$ as follows, for any symmetric pseudomonoidal category $\V$ (with invertible coherence maps) that is cocomplete, such that its binary tensor product $\dash \tens \dash$ preserves colimits on both sides. Examples include the category of vector spaces over a field and their tensor products, the cartesian monoidal category $\cat$ of categories and functors or the category $\2 = (\bot \to \top)$ of truth values. The monoidal stucture taken on the latter is cartesian, that is it is given by conjunction of truth values, while its coproducts are given by disjunction.
	
	The pseudo double category $\enProf\V$ of $\V$-profunctors has small $\V$-categories as objects and $\V$-functors as vertical morphisms. Generalising the unenriched case, the horizontal morphisms $\hmap JAB$ of $\enProf\V$ are \emph{$\V$-profunctors}, that is $\V$-functors of the form \mbox{$\map J{\op A \tens B}\V$}. The dual $\op A$ of $A$ and tensor product $\op A \tens B$ here are constructed using the symmetry on $\V$; see \cite[Section 1.4]{Kelly82} for details. We shall again write $K(f, g) = K \of (\op f \tens g)$ where $\hmap KCD$, $\map fAC$ and $\map gBD$; a cell
	\begin{displaymath}
		\begin{tikzpicture}
			\matrix(m)[math175em]{A & B \\ C & D \\};
			\path[map]  (m-1-1) edge[barred] node[above] {$J$} (m-1-2)
													edge node[left] {$f$} (m-2-1)
									(m-1-2) edge node[right] {$g$} (m-2-2)
									(m-2-1) edge[barred] node[below] {$K$} (m-2-2);
			\path[transform canvas={shift={($(m-1-2)!(0,0)!(m-2-2)$)}}] (m-1-1) edge[cell] node[right] {$\phi$} (m-2-1);
		\end{tikzpicture}
	\end{displaymath}
	of $\enProf\V$ is a $\V$-transformation $\nat\phi J{K(f, g)}$.
	
	Finally horizontal composites and units are both enriched variants of those of $\enProf\V$: the horizontal composition of composable $\V$-profunctors $\hmap JAB$ and $\hmap HBC$ is given by the coends
  \begin{displaymath}
    (J \hc_B H)(a, c) = \int^{b \in B} J(a, b) \tens H(b, c),
  \end{displaymath}
  which are again computed as the coequalisers \eqref{equation:composition of profunctors}, but with the cartesian products of $\Set$ replaced by the tensor products of $\V$. The unit $\V$-profunctor $U_A$ on the $\V$\ndash category $A$ is given by its hom-objects $U_A(a_1, a_2) = A(a_1, a_2)$, with actions given by composition; in particular there are enriched versions of the Yoneda isomorphisms $U_A \hc_A J \iso J$ and $J \hc_B U_B \iso J$, see for example \cite[Formula 3.71]{Kelly82}. The horizontal composition $\phi \hc_g \psi$ of $\V$-natural transformations $\phi$ and $\psi$ is likewise given as the enriched variant of \eqref{equation:composition of natural transformations}.
  
  In the simple case $\V = \2$, a $\2$-enriched category $A$ is the same as a preordered set, that is a set $A'$ with an ordering $\leq$ that is reflexive and transitive, by taking $A' = \ob A$ and $x \leq y$ if and only if $A(x, y) = \top$, for any pair of objects $x$ and $y$. Similarly, a $\2$-profunctor $\hmap JAB$ can be thought of as a relation $\sim_J$ between the preordered sets $A$ and $B$, such that if $x_1 \leq x_2$ in $A$, $x_2 \sim_J y_1$ and $y_1 \leq y_2$ in $B$, then $x_1 \sim_J y_2$ as well. The horizontal composite of such relations $\hmap JAB$ and $\hmap HBC$ is given by the usual composite of relations: $x \sim_{J \hc H} z$ if and only if there exists a $y$ in $B$ such that $x \sim_J y$ and $y \sim_H z$. A $\2$-functor $\map fAB$ is simply an order preserving map while a cell $\phi$, as above, exists (and is unique) if and only if $x \sim_J y$ implies $fx \sim_K gy$ for all $x$ in $A$ and $y$ in $B$.
\end{example}
	
	Shulman shows in \cite[Section 11]{Shulman08} that $\V$-categories and $\V$\ndash profunctors can be considered as `monoids' and `bimodules' in the simpler double category $\Mat\V$ of `$\V$-matrices'. This will turn out useful for us, as many of our questions about $\enProf\V$ are easier to answer by considering the `underlying questions' in $\Mat\V$. This is why we shall recall this construction later, in \defref{definition:monoids and bimodules}. In the same way, for every category $\E$ with finite limits, there is a relatively simple pseudo double category $\Span\E$ of `spans in $\E$', and monoids and bimodules in $\Span\E$ are respectively categories and profunctors `internal in $\E$'. The pseudo double categories $\Mat\V$ and $\Span\E$ which, as we shall see in the next section, are promorphism equipments as well, are introduced below.
\begin{example} \label{example:matrices}
	Let $\V$ be a monoidal category that has all coproducts which are preserved by its tensor product. The pseudo double category $\Mat\V$ of \emph{$\V$-matrices} is given as follows. Its objects are sets and its vertical morphisms are maps of sets, that is $\Mat\V_0 = \Set$, while a horizontal morphism $\hmap JAB$, that is an object of $\Mat\V_1$, consists of a `matrix' $\bigpars{J(a, b)}_{a \in A, b \in B}$ of $\V$-objects. A cell
\begin{displaymath}
  \begin{tikzpicture}
    \matrix(m)[math175em]{A & B \\ C & D \\};
    \path[map]  (m-1-1) edge[barred] node[above] {$J$} (m-1-2)
                        edge node[left] {$f$} (m-2-1)
                (m-1-2) edge node[right] {$g$} (m-2-2)
                (m-2-1) edge[barred] node[below] {$K$} (m-2-2);
    \path[transform canvas={shift={($(m-1-2)!(0,0)!(m-2-2)$)}}] (m-1-1) edge[cell] node[right] {$\phi$} (m-2-1);
  \end{tikzpicture}
\end{displaymath}
	is given by a family of maps $\map{\phi_{a, b}}{J(a, b)}{K(fa, gb)}$ in $\V$. Such cells clearly compose vertically, thus forming a category $\Mat\V_1$. Horizontal composition of $\V$-matrices $\hmap JAB$ and $\hmap HBE$ is given by `matrix multiplication':
	\begin{displaymath}
		(J \hc H)(a, e) = \coprod_{b \in B} J(a, b) \tens H(b, e).
	\end{displaymath}
	The existence of the associators follows from the assumption that the tensor product preserves coproducts, while the unit $U_A$ on a set $A$ is given by
	\begin{displaymath}
		U_A(a, a') = \begin{cases}
			1	& \text{if $a' = a$;} \\
			\emptyset & \text{otherwise,}
		\end{cases}			
	\end{displaymath}
	where $1$ denotes the monoidal unit of $\V$, and $\emptyset$ its initial object. Finally, the composition of composable cells
  \begin{displaymath}
    \begin{tikzpicture}
      \matrix(m)[math175em]{A & B & E \\ C & D & F \\};
      \path[map]  (m-1-1) edge[barred] node[above] {$J$} (m-1-2)
                          edge node[left] {$f$} (m-2-1)
                  (m-1-2) edge[barred] node[above] {$H$} (m-1-3)
                          edge node[right] {$g$} (m-2-2)
                  (m-1-3) edge node[right] {$h$} (m-2-3)
                  (m-2-1) edge[barred] node[below] {$K$} (m-2-2)
                  (m-2-2) edge[barred] node[below] {$L$} (m-2-3);
      \path[transform canvas={shift={($(m-1-1)!0.5!(m-2-2)$)}}]
                  (m-1-2) edge[cell] node[right] {$\phi$} (m-2-2)
                  (m-1-3) edge[cell] node[right] {$\psi$} (m-2-3);
    \end{tikzpicture}
  \end{displaymath}
  is given by the compositions
  \begin{displaymath}
	  (J \hc H)(a, e) \to \coprod_{b \in B} K(fa, gb) \tens L(gb, he) \to (K \hc L)(fa, he),
	\end{displaymath}
	where the first map is the coproduct of the tensor products $\phi_{a, b} \tens \psi_{b, e}$, and the second map exists by the universal property of coproducts.
	
	In case $\V = \2$, a $\2$-matrix $\hmap JAB$ is a relation on $A$ and $B$. Writing $a \sim_J b$ whenever $a$ and $b$ are related in $J$, a cell $\phi$ as above exists (and is unique) if and only if $a \sim_J b$ implies $fa \sim_K gb$ for all $a \in A$ and $b \in B$. Horizontal composition of relations is given by the usual composition of relations.
\end{example}
\begin{example} \label{example:spans}
	For a category $\E$ with finite limits, the pseudo double category $\Span\E$ of \emph{spans in $\E$} is defined as follows. The objects and vertical morphisms of $\Span\E$ are those of $\E$, while a horizontal morphism $\hmap JAB$ is a span $A \xlar{j_A} J \xrar{j_B} B$ in $\E$. A cell $\phi$ as on the left below is a map $\map\phi JK$ in $\E$ such that the diagram on the right commutes.
\begin{displaymath}
  \begin{tikzpicture}[baseline]
    \matrix(m)[math175em]{A & B \\ C & D \\};
    \path[map]  (m-1-1) edge[barred] node[above] {$J$} (m-1-2)
                        edge node[left] {$f$} (m-2-1)
                (m-1-2) edge node[right] {$g$} (m-2-2)
                (m-2-1) edge[barred] node[below] {$K$} (m-2-2);
    \path[transform canvas={shift={($(m-1-2)!(0,0)!(m-2-2)$)}}] (m-1-1) edge[cell] node[right] {$\phi$} (m-2-1);
  \end{tikzpicture}
	\qquad\qquad\qquad\begin{tikzpicture}[baseline]
		\matrix(m)[math175em, column sep=1.25em, row sep=0.3em]{& J & \\ A & & B \\ & K & \\ C & & D \\};
		\path[map]	(m-1-2)	edge node[above left] {$j_A$} (m-2-1)
												edge node[right] {$\phi$} (m-3-2)
												edge node[above right] {$j_B$} (m-2-3)
								(m-2-1) edge node[left] {$f$} (m-4-1)
								(m-2-3) edge node[right] {$g$} (m-4-3)
								(m-3-2) edge node[below=4pt, right=-2pt] {$k_C$} (m-4-1)
												edge node[below=4pt, left=-5pt] {$k_D$} (m-4-3);
	\end{tikzpicture}
\end{displaymath}
	Given spans $\hmap JAB$ and $\hmap HBC$, their composition $J \hc H$ is given by the usual composition of spans: after choosing a pullback $J \times_B H$ of $j_B$ and $h_B$ below, it is taken to consist of the two sides of the diagram below.
	\begin{displaymath}
		\begin{tikzpicture}
			\matrix(m)[math175em, row sep=1.2em, column sep=1.2em]
			{	& & \phantom B & & \\
				& J & & H & \\
				A & & B & & C \\ };
			\path[map]	(m-1-3) edge (m-2-2)
													edge (m-2-4)
									(m-2-2) edge node[above left] {$j_A$} (m-3-1)
													edge node[below left] {$j_B$} (m-3-3)
									(m-2-4) edge node[below right] {$h_B$} (m-3-3)
													edge node[above right] {$h_C$} (m-3-5);
			\draw (m-1-3) node {$J \times_B H$};
		\end{tikzpicture}
	\end{displaymath}
	That this composition is associative and unital up to coherent isomorphisms, with spans of the form $A \xlar{\id} A \xrar{\id} A$ as units, follows from the universality of pullbacks; horizontal composition of cells is also given by using this universality.
\end{example}	
	Since we assume $\E$ to have all finite limits we might equivalently regard a span $J$ as an object $\map jJ{A \times B}$ in the slice category $\E \slash A \times B$. In this case we will still write $j = (j_A, j_B)$. A cell $\phi$ as above thus corresponds to a commutative square in $\E$ as follows. 
	\begin{equation} \label{diagram:cell of spans}
		\begin{split}
		\begin{tikzpicture}
			\matrix(m)[math175em, column sep=1.5em]{J \nc K \\ A \times B \nc C \times D \\};
			\path[map]	(m-1-1) edge node[above] {$\phi$} (m-1-2)
													edge node[left] {$j$} (m-2-1)
									(m-1-2) edge node[right] {$k$} (m-2-2)
									(m-2-1) edge node[below] {$f \times g$} (m-2-2);
		\end{tikzpicture}
		\end{split}
	\end{equation}
	Regarding spans in $\E$ as objects in slice categories, their horizontal composition can be given as follows. Given a morphism $\map fXY$ notice that there is an adjunction
\begin{displaymath}
	\begin{tikzpicture}
		\matrix(m)[math175em, column sep=0.5em]{f_* \colon \E \slash X & & \E \slash Y \colon f^* \\};
		\path[map]	(m-1-1)	edge[transform canvas={yshift=3pt}] (m-1-3)
								(m-1-3)	edge[transform canvas={yshift=-3pt}] (m-1-1);
		\draw[font = \tiny]	(m-1-2) node {$\bot$};
	\end{tikzpicture}
\end{displaymath}
	whose left adjoint $f_*$ is given by postcomposition with $f$ and whose right adjoint $f^*$ is given by pullback along $f$.
\begin{lemma} \label{horizontal composition of spans}
	Let $\E$ be a category with finite limits. Given a span $\map jJ{A \times B}$ in $\E$, horizontal composition $\map{J \hc \dash}{\E \slash B \times C}{\E \slash A \times C}$ can be given as
	\begin{displaymath}
		J \hc \dash = \pars{j_A \times \id_C}_* \of \pars{j_B \times \id_C}^*,
	\end{displaymath}
	where $j = (j_A, j_B)$.
\end{lemma}
\begin{proof}
	Let $\map hH{B \times C}$ be a second span in $\E$. By definition $J \hc H$ is given by $\map{(j_A \of p, h_C \of q)}{J \times_B H}{A \times C}$, where $p$ and $q$ are the projections of the pullback $W = J \times_B H$, onto $J$ and $H$ respectively. An easy calculation will show that $(W, p, q)$ is the pullback of $j_B$ and $h_B$ if and only if
	\begin{displaymath}
		\begin{tikzpicture}
			\matrix(m)[math2em]{W & H \\ J \times C & B \times C \\};
			\path[map]	(m-1-1)	edge node[above] {$q$} (m-1-2)
													edge node[left] {$(p, h_C \of q)$} (m-2-1)
									(m-1-2) edge node[right] {$h$} (m-2-2)
									(m-2-1) edge node[below] {$j_B \times \id_C$} (m-2-2);
		\end{tikzpicture}
	\end{displaymath}
	is a pullback square. Here the map on the left is, by definition, $(j_B \times \id_C)^* h$; composing it with $j_A \times \id_C$ does indeed give $\map{J \hc H}W{A \times C}$ as asserted. It is easily seen that this extends to the right action on morphisms of spans as well.
\end{proof}

	Having introduced the main examples, we close this section by recalling that every pseudo double category $\K$ contains an underlying \emph{vertical $2$-category $V(\K)$}, of vertical morphisms and cells, as well as a \emph{horizontal bicategory $H(\K)$}, of horizontal morphisms and cells, as follows (a bicategory is a category `weakly enriched' in $\cat$, see \cite[Definition 1.5.1]{Leinster04}). The objects of both $V(\K)$ and $H(\K)$ are those of $\K$. Given objects $A$ and $B$, the category $H(\K)(A, B)$ is the subcategory of $\K_1$ consisting of all horizontal morphisms $A \slashedrightarrow B$ and horizontal cells, that is cells with identities as vertical maps, between them. The composition functors
\begin{displaymath}
  H(\K)(A, B) \times H(\K)(B,C) \to H(\K)(A, C)
\end{displaymath}
  are restrictions of $\map\hc{\K_1 \times_{\K_0} \K_1}{\K_1}$, and we will denote composition in $H(\K)$ again by $\hc$.

  On the other hand, the category $V(\K)(A, B)$ consists of vertical maps $A \to B$ in $\K_0$, while the morphisms $\cell\phi fg$ are the vertical cells of $\K$, that is the horizontal source and target of $\phi$ are horizontal identities. The vertical (in terms of the $2$\ndash category $V(\K)$) composition $\psi \of \phi$ of $\phi$ with a second vertical cell $\cell\psi gh$ is given by the composition of the diagram
\begin{displaymath}
  \begin{tikzpicture}
    \matrix(m)[math175em]
      { A & \phantom A & A \\
        A & A & A \\
        B & B & B \\
        B & \phantom B & B \\ };
    \path[map]  (m-2-1) edge node[left] {$f$} (m-3-1)
                (m-2-2) edge node[right] {$g$} (m-3-2)
                (m-2-3) edge node[right] {$h$} (m-3-3);
    \path       (m-1-1) edge[eq] (m-1-3)
                        edge[eq] (m-2-1)
                (m-1-2) edge[cell] (m-2-2)
                (m-1-3) edge[eq] (m-2-3)
                (m-2-1) edge[eq] (m-2-2)
                (m-2-2) edge[eq] (m-2-3)
                (m-3-1) edge[eq] (m-3-2)
                        edge[eq] (m-4-1)
                (m-3-2) edge[eq] (m-3-3)
                        edge[cell] (m-4-2)
                (m-3-3) edge[eq] (m-4-3)
                (m-4-1) edge[eq] (m-4-3);
    \path[transform canvas={shift={($(m-2-1)!0.5!(m-3-2)$)}}]
                (m-2-2) edge[cell] node[right] {$\phi$} (m-3-2)
                (m-2-3) edge[cell] node[right] {$\psi$} (m-3-3);
  \end{tikzpicture}
\end{displaymath}
  where the top and bottom cells are given by (inverses of) unitors of $\K$. That this composition is strictly associative follows from the fact that the unitors satisfy the identity coherence axiom (see \cite[Diagram (7) of Section VII.1]{MacLane98}). The horizontal composition functors
\begin{displaymath}
  \map\hc{V(\K)(A, B) \times V(\K)(B,C)}{V(\K)(A, C)}
\end{displaymath}
  are given by vertical composition in $\K$, which is strictly associative.
	
\begin{example} \label{example:vertical cells in Prof}
	Remember that a vertical cell $\cell\phi fg$, of functors $f$ and $\map gAB$, in the pseudo double category $\Prof$ of (unenriched) profunctors, is a natural transformation of profunctors $\nat\phi{U_A}{U_B(f, g)}$, where $U_B(f, g)(a_1, a_2) = B(fa_1, ga_2)$. By naturality $\phi$ is completely determined by its images $\map{\phi_a = \phi(\id_a)}{fa}{ga}$ in $B$; indeed we have $\phi(u) = \phi_{a_2} \of f(u) = g(u) \of \phi_{a_1}$ for any $\map u{a_1}{a_2}$ in $A$. From this identity it follows that the maps $\phi_a$ form a natural transformation of functors $\nat\phi fg$, in the usual sense; in fact this gives a bijective correspondence between the vertical cells $f \Rightarrow g$ of $\Prof$ and the natural transformations $f \natarrow g$.
	
	It follows from the descriptions of the unitors of $\Prof$, in \exref{example:unenriched profunctors}, that the vertical composition of two vertical cells $\cell\phi fg$ and $\cell\psi gh$ corresponds to the usual vertical composition $\psi \of \phi$ of natural transformations of functors, given by $(\psi \of \phi)_a = \psi_a \of \phi_a$. Moreover, the horizontal composition of two vertical cells $\cell\phi fg$ and $\cell\psi hk$, where $h$ and $k$ are functors $B \to C$, is determined by the images
	\begin{displaymath}
		\map{(\phi \hc \psi)_a = \psi(\phi_a) = \psi_{ga} \of h(\phi_a) = k(\phi_a) \of \psi_{fa}}{(h \of f)(a)}{(k \of g)(a)},
	\end{displaymath}
	that is it corresponds to the usual horizontal composite $\nat{\phi \hc \psi}{h \of f}{k \of g}$. We conclude that $V(\Prof) \iso \Cat$, the $2$-category of small categories, functors and natural transformations. On the other hand clearly the horizontal bicategory $H(\Prof)$ of $\Prof$ is the bicategory of categories and profunctors.
	
	Likewise in the enriched case we get $V(\enProf\V) \iso \enCat\V$, the $2$-category of small $\V$-categories, $\V$-functors and $\V$-natural transformations. 
\end{example}

\section{$2$-Categories equipped with promorphisms} \label{section:2-categories equipped with promorphisms}
In the introduction we described $2$-categories equipped with promorphisms, `promorphism equipments' in short, as pseudo double categories $\K$ in which every vertical map $\map fAC$ is accompanied by two horizontal morphisms $\hmap{C(f, \id)}AC$ and $\hmap{C(\id, f)}CA$, respectively its `companion' and `conjoint'. This description is close to the original one given by Wood \cite{Wood82}. However we will keep on following Shulman's paper \cite{Shulman08} and define promorphism equipments as pseudo double categories in which all `cartesian fillers' exist. Afterwards we will recall \cite[Theorem A.2]{Shulman08} which shows that the two definitions coincide.

  In a pseudo double category $\K$, a diagram of the form
\begin{displaymath}
  \begin{tikzpicture}
    \matrix(m)[math175em]{A & B \\ C & D, \\};
    \path[map]  (m-1-1) edge node[left] {$f$} (m-2-1)
                (m-1-2) edge node[right] {$g$} (m-2-2)
                (m-2-1) edge[barred] node[below] {$K$} (m-2-2);
  \end{tikzpicture}
\end{displaymath}
  also written as $A \xrar f C \xslashedrightarrow K D \xlar g B$, is called a \emph{niche}.
\begin{definition} \label{definition:cartesian filler}
	A \emph{cartesian filler} for the niche above is a cell
\begin{displaymath}
  \begin{tikzpicture}
    \matrix(m)[math175em]{A & B \\ C & D \\};
    \path[map]  (m-1-1) edge[barred] node[above] {$K(f,g)$} (m-1-2)
                        edge node[left] {$f$} (m-2-1)
                (m-1-2) edge node[right] {$g$} (m-2-2)
                (m-2-1) edge[barred] node[below] {$K$} (m-2-2);
    \path[transform canvas={shift={($(m-1-1)!0.5!(m-2-1)$)}}] (m-1-2) edge[cell] node[right] {$\eps$} (m-2-2);
  \end{tikzpicture}
\end{displaymath}
  that is universal in that every other filler $\cell\phi JK$ factorises uniquely as
\begin{displaymath}
  \begin{tikzpicture}[textbaseline]
    \matrix(m)[math175em]{X & Y \\ A & B \\ C & D \\};
    \path[map]  (m-1-1) edge[barred] node[above] {$J$} (m-1-2)
                        edge node[left] {$h$} (m-2-1)
                (m-1-2) edge node[right] {$k$} (m-2-2)
                (m-2-1) edge node[left] {$f$} (m-3-1)
                (m-2-2) edge node[right] {$g$} (m-3-2)
                (m-3-1) edge[barred] node[below] {$K$} (m-3-2);
    \path[transform canvas={shift={($(m-2-1)!0.5!(m-1-1)$)}}] (m-2-2) edge[cell] node[right] {$\phi$} (m-3-2);
  \end{tikzpicture}
  =
  \begin{tikzpicture}[textbaseline]
    \matrix(m)[math175em]{X & Y \\ A & B \\ C & \phantom D \\};
    \path[map]  (m-1-1) edge[barred] node[above] {$J$} (m-1-2)
                        edge node[left] {$h$} (m-2-1)
                (m-1-2) edge node[right] {$k$} (m-2-2)
                (m-2-1) edge[barred] node[below] {$K(f,g)$} (m-2-2)
                        edge node[left] {$f$} (m-3-1)
                (m-2-2) edge node[right] {$g$} (m-3-2)
                (m-3-1) edge[barred] node[below] {$K$} (m-3-2);
    \path[transform canvas={shift=(m-2-1))}]
                (m-1-2) edge[transform canvas={xshift=-0.6em}, cell] node[right] {$\ls f\phi_g$} (m-2-2)
                (m-2-2) edge[transform canvas={yshift=-0.4em}, cell] node[right] {$\eps$} (m-3-2);
    \draw (m-3-2) node {$D$.};
  \end{tikzpicture}
\end{displaymath}
	The horizontal morphism $K(f, g)$ is called the \emph{restriction of $K$ along $f$ and $g$}.
\end{definition}
  Cartesian fillers are unique up to precomposition with invertible horizontal cells.
\begin{definition} \label{definition:equipment}
  A pseudo double category in which every niche has a cartesian filler is called a \emph{$2$-category equipped with promorphisms}, or simply a \emph{(promorphism) equipment}.
\end{definition}
	We will call the vertical morphisms of a promorphism equipment simply \emph{morphisms}; the horizontal morphisms will be called \emph{promorphisms}.	In the example below we will see all the examples of the previous section are promorphism equipments, but first we recall the definition of the companion and conjoint of a morphism.

  Given a morphism $\map fAC$, assume that the cartesian fillers
\begin{alignat*}{5}
  \begin{tikzpicture}[textbaseline]
    \matrix(m)[math175em]{A \nc C \\ C \nc C \\};
    \path[map]  (m-1-1) edge[barred] node[above=3pt] {$U_C(f,\id)$} (m-1-2)
                        edge node[left] {$f$} (m-2-1);
    \path       (m-1-2) edge[eq] (m-2-2)
                (m-2-1) edge[eq] (m-2-2);
    \path[transform canvas={shift={($(m-1-1)!0.5!(m-2-1)$)}}] (m-1-2) edge[cell] node[right] {$\ls f \eps$} (m-2-2);
  \end{tikzpicture}
  &\quad\text{and}\quad&&
  \begin{tikzpicture}[textbaseline]
    \matrix(m)[math175em]{C \nc A \\ C \nc C \\};
    \path[map]  (m-1-1) edge[barred] node[above=3pt] {$U_C(\id,f)$} (m-1-2)
                (m-1-2) edge node[right] {$f$} (m-2-2);
    \path       (m-1-1) edge[eq] (m-2-1)
                (m-2-1) edge[eq] (m-2-2);
    \path[transform canvas={shift={($(m-1-1)!0.5!(m-2-1)$)}}] (m-1-2) edge[cell] node[right] {$\eps_f$} (m-2-2);
  \end{tikzpicture}\\
  \intertext{exist; we will abbreviate $C(f, \id) = U_C(f, \id)$ and $C(\id, f) = U_C(\id, f)$. Factorising the identity cells $U_f$ through these cartesian fillers gives cells}
  \begin{tikzpicture}[textbaseline]
    \matrix(m)[math175em]{A \nc A \\ A \nc C \\};
    \path       (m-1-1) edge[eq] (m-1-2)
                        edge[eq] (m-2-1);
    \path[map]  (m-1-2) edge node[right] {$f$} (m-2-2)
                (m-2-1) edge[barred] node[below] {$C(f, \id)$} (m-2-2);
    \path[transform canvas={shift={($(m-1-1)!0.5!(m-2-1)$)}}] (m-1-2) edge[cell] node[right] {$\ls f \eta$} (m-2-2);
  \end{tikzpicture}
  &\quad\text{and}\quad&&
  \begin{tikzpicture}[textbaseline]
    \matrix(m)[math175em]{A \nc A \\ C \nc A \\};
    \path       (m-1-1) edge[eq] (m-1-2)
                (m-1-2) edge[eq] (m-2-2);
    \path[map]  (m-1-1) edge node[left] {$f$} (m-2-1)
                (m-2-1) edge[barred] node[below] {$C(\id, f)$} (m-2-2);
    \path[transform canvas={shift={($(m-1-1)!0.5!(m-2-1)$)}}] (m-1-2) edge[cell] node[right] {$\eta_f$} (m-2-2);
  \end{tikzpicture} \\
  \intertext{that satisfy}
  \begin{tikzpicture}[textbaseline]
    \matrix(m)[math175em]{A \nc A \\ C \nc C \\};
    \path[map]  (m-1-1) edge node[left] {$f$} (m-2-1)
                (m-1-2) edge node[right] {$f$} (m-2-2);
    \path       (m-1-1) edge[eq] (m-1-2)
                (m-2-1) edge[eq] (m-2-2);
    \path[transform canvas={shift={($(m-1-1)!0.5!(m-2-1)$)}}] (m-1-2) edge[cell] node[right] {$U_f$} (m-2-2);
  \end{tikzpicture}
  =
  \begin{tikzpicture}[textbaseline]
    \matrix(m)[math175em]{A \nc A \\ A \nc C \\ C \nc C \\};
    \path[map]  (m-1-2) edge node[right] {$f$} (m-2-2)
                (m-2-1) edge[barred] node[above] {$C(f, \id)$} (m-2-2)
                        edge node[left] {$f$} (m-3-1);
    \path       (m-1-1) edge[eq] (m-1-2)
                        edge[eq] (m-2-1)
                (m-2-2) edge[eq] (m-3-2)
                (m-3-1) edge[eq] (m-3-2);
    \path[transform canvas={shift=(m-2-1))}]
                (m-1-2) edge[cell, shorten >= 7pt, shorten <= -1pt] node[above=3pt, right] {$_f\eta$} (m-2-2)
                (m-2-2) edge[cell] node[right] {$\ls f\eps$} (m-3-2);
  \end{tikzpicture}
  &\quad\text{and}\quad&&
    \begin{tikzpicture}[textbaseline]
    \matrix(m)[math175em]{A \nc A \\ C \nc C \\};
    \path[map]  (m-1-1) edge node[left] {$f$} (m-2-1)
                (m-1-2) edge node[right] {$f$} (m-2-2);
    \path       (m-1-1) edge[eq] (m-1-2)
                (m-2-1) edge[eq] (m-2-2);
    \path[transform canvas={shift={($(m-1-1)!0.5!(m-2-1)$)}}] (m-1-2) edge[cell] node[right] {$U_f$} (m-2-2);
  \end{tikzpicture}
  =
  \begin{tikzpicture}[textbaseline]
    \matrix(m)[math175em]{A \nc A \\ C \nc A \\ C \nc \phantom C \\};
    \path[map]  (m-1-1) edge node[left] {$f$} (m-2-1)
                (m-2-1) edge[barred] node[above] {$C(\id, f)$} (m-2-2)
                (m-2-2) edge node[right] {$f$} (m-3-2);
    \path       (m-1-1) edge[eq] (m-1-2)
                (m-1-2) edge[eq] (m-2-2)
                (m-2-1) edge[eq] (m-3-1)
                (m-3-1) edge[eq] (m-3-2);
    \path[transform canvas={shift=(m-2-1)}]
                (m-1-2) edge[cell, shorten >= 7pt, shorten <= -1pt] node[above=3pt, right] {$\eta_f$} (m-2-2)
                (m-2-2) edge[cell] node[right] {$\eps_f$} (m-3-2);
    \draw (m-3-2) node {$C$.};
  \end{tikzpicture} \\
  \intertext{Moreover the uniqueness of factorisations through cartesian fillers implies that the composites}
  \begin{tikzpicture}[textbaseline]
    \matrix(m)[math175em]{A \nc A \nc C \\ A \nc C \nc C \\};
    \path[map]  (m-1-2) edge[barred] node[above] {$C(f, \id)$} (m-1-3)
                        edge node[right] {$f$} (m-2-2)
                (m-2-1) edge[barred] node[below] {$C(f, \id)$} (m-2-2);
    \path       (m-1-1) edge[eq] (m-1-2)
                        edge[eq] (m-2-1)
                (m-1-3) edge[eq] (m-2-3)
                (m-2-2) edge[eq] (m-2-3);
    \path[transform canvas={shift={($(m-1-1)!0.5!(m-2-2)$)}}]
                (m-1-2) edge[cell] node[right] {$\ls f\eta$} (m-2-2)
                (m-1-3) edge[cell] node[right] {$\ls f\eps$} (m-2-3);
  \end{tikzpicture}
  &\quad\text{and}\quad&&
  \begin{tikzpicture}[textbaseline]
    \matrix(m)[math175em]{C \nc A \nc A \\ C \nc C \nc A \\};
    \path[map]  (m-1-1) edge[barred] node[above] {$C(\id, f)$} (m-1-2)
                (m-1-2) edge node[right] {$f$} (m-2-2)
                (m-2-2) edge[barred] node[below] {$C(\id, f)$} (m-2-3);
    \path       (m-1-2) edge[eq] (m-1-3)
                (m-1-1) edge[eq] (m-2-1)
                (m-1-3) edge[eq] (m-2-3)
                (m-2-1) edge[eq] (m-2-2);
    \path[transform canvas={shift={($(m-1-1)!0.5!(m-2-2)$)}}]
                (m-1-2) edge[cell] node[right] {$\eps_f$}
                 (m-2-2)
                (m-1-3) edge[cell] node[right] {$\eta_f$} (m-2-3);
  \end{tikzpicture}
\end{alignat*}
  are, after both pre- and postcomposition with (inverses of) unitors, equal to the identity cells of $C(f, \id)$ and $C(\id, f)$ respectively. The two equalities involving $\ls f \eps$ and $\ls f \eta$ above define $C(f, \id)$ as the \emph{companion} for $f$, while those involving $\eps_f$ and $\eta_f$ define $C(\id, f)$ as the \emph{conjoint} for $f$; the equalities themselves are called the \emph{companion} and \emph{conjoint identities} respectively. One can think of the companion $C(f, \id)$ as the promorphism `isomorphic' to $f$, while the conjoint $C(\id, f)$ can be thought of as being `adjoint' to $f$. Companions and conjoints were introduced by Grandis and Par\'e in \cite[Sections 1.2 and 1.3]{Grandis-Pare04} where they are called `orthogonal companions' and `orthogonal adjoints'. As we remarked in the introduction, Wood's original definition of equipments, given in \cite{Wood82}, is close to defining equipments as pseudo double categories in which every vertical morphism has a companion and conjoint.

\begin{example} \label{example:cartesian fillers}
	For a niche $A \xrar f C \xslashedrightarrow K D \xlar g B$ in $\mathsf{Mod}$, where $f$ and $g$ are ring homomorphisms and $K$ is a $(C, D)$-bimodule, the restriction $\hmap{K(f, g)}AB$ is $K$ regarded as a $(A, B)$-bimodule via $f$ and $g$, with left action \mbox{$A \tens K \xrar{f \tens \id} C \tens K \xrar l K$} and similar right action. The cartesian filler $\cell\eps{K(f, g)}K$ is simply the identity for $K$: it is clear that any $(f \of h, g \of k)$-bilinear map $\cell\phi JK$, as in \defref{definition:cartesian filler}, can be regarded as a $(h, k)$-bilinear map $J \Rightarrow K(f,g)$, which is the unique factorisation $\ls f\phi_g$ of $\phi$ through $\eps$. In particular the factorisation $\ls f\eta$ of the horizontal unit $U_f$ through the companion cell $\ls f\eps$, where $U_f$ is the ring homomorphism $\map fAC$ regarded as a $(f,f)$-bilinear map, is again simply $f$, now regarded as a $(\id, f)$-bilinear map $A \Rightarrow C(f, \id)$. 
	
	For a niche $A \xrar f C \xslashedrightarrow K D \xlar g B$ in $\Prof$ we have, in \exref{example:unenriched profunctors}, already used the notation $K(f, g)$ to denote the composite $\hmap{K \of (\op f \times g)}AB$. It is readily seen that $K \of (\op f \times g)$, together with the cell $\cell\eps{K(f,g)}K$ that is the identity transformation on $K\of (\op f \times g)$, do indeed form the cartesian filler: any other filler
	\begin{equation} \label{equation:457490}
		\begin{split}
		\begin{tikzpicture}[textbaseline]
			\matrix(m)[math175em]{X \nc Y \\ A \nc B \\ C \nc D \\};
			\path[map]  (m-1-1) edge[barred] node[above] {$J$} (m-1-2)
													edge node[left] {$h$} (m-2-1)
									(m-1-2) edge node[right] {$k$} (m-2-2)
									(m-2-1) edge node[left] {$f$} (m-3-1)
									(m-2-2) edge node[right] {$g$} (m-3-2)
									(m-3-1) edge[barred] node[below] {$K$} (m-3-2);
			\path[transform canvas={shift={($(m-2-1)!0.5!(m-1-1)$)}}] (m-2-2) edge[cell] node[right] {$\phi$} (m-3-2);
		\end{tikzpicture}
		\end{split}
  \end{equation}
  in $\Prof$, that is a transformation $\nat\phi J{K(f \of h, g \of k)}$, can be regarded as a cell $\cell\phi J{K(f,g)}$, because $K\of (\op f \times g) \of (\op h \times k) = K \of \bigpars{\op{(f \of h)} \times (g \of k)}$. In particular the companion $C(f, \id)$ of a functor $\map fAC$ is given by the representable profunctor $C(f, \id)(a,c) = C(fa, c)$. Conjoints are given likewise.
  
  The previous example generalises directly to the enriched case: in $\enProf\V$ the restriction $K(f, g)$ of a $\V$-profunctor $K$ is given by $\hmap{K(f, g) = K \of (\op f \tens g)}AB$.
	
	In $\Mat\V$ the restriction $\hmap{K(f, g)}AB$ is simply given by $K(f, g)(a, b) = K(fa, gb)$; the cartesian filler $K(f, g) \Rar K$ has identities as components. It follows that the companion $C(f, \id)$ of $f$ is given by
	\begin{displaymath}
		C(f, \id)(a, c) = \begin{cases}
			1 & \text{if $fa = c$;} \\
			\emptyset & \text{otherwise,}
		\end{cases}
	\end{displaymath}
	while both $\ls f\eps$ and $\ls f\eta$ consist of identities. The conjoint $C(\id, f)$ is given likewise.
	
	In $\Span\E$ the restriction $K(f, g)$ is given by the pullback $K(f, g) = (f \times g)^* k$, where $\map kK{C \times D}$; the cartesian filler is the projection $K(f, g) \to K$. The universal property follows immediately from the fact that cells $\phi$ in $\Span\E$ are commutative squares like \eqref{diagram:cell of spans}, which factor through the pullback $K(f, g)$. Hence the companion $C(f, \id)$ is the span $A \xlar\id A \xrar f C$, with $\ls f\eps = f$ and $\ls f\eta = \id_A$, while the conjoint $C(\id, f)$ is the span $C \xlar f A \xrar\id A$.
\end{example}

  A pseudo double category has all cartesian fillers if and only if it has both companions and conjoints for vertical morphisms, as the following shows.
\begin{proposition}[Shulman] \label{cartesian fillers in terms of companions and conjoints}
	Let $A \xrar f C \xslashedrightarrow K D \xlar g B$ be a niche in a pseudo double category $\K$. If $f$ has a companion $C(f, \id)$ and $g$ has a conjoint $D(\id, g)$ then the composition
	\begin{displaymath}
		C(f, \id) \hc K \hc D(\id, g) \xRar{\ls f\eps \hc \id_K \hc \eps_g} U_C \hc K \hc U_D \iso K,
  \end{displaymath}
  is a cartesian filler. In particular there is a canonical isomorphism $C(f, \id) \hc K \hc D(\id, g) \iso K(f, g)$.
\end{proposition}
\begin{proof}[Sketch of the proof.]
	This is shown in the proof of \cite[Theorem 4.1]{Shulman08}, and follows directly from the companion and conjoint identities. For example, any other filler $\cell\phi JK$ as in \eqref{equation:457490} factorises uniquely through the cell above as the composition
	\begin{displaymath}
		J \iso U_X \hc J \hc U_Y \xRar{(\ls f \eta \of U_h) \hc \phi \hc (\eta_g \of U_k)} C(f, \id) \hc K \hc D(\id, g).
	\end{displaymath}
	Factorising the cartesian fillers $C(f, \id) \hc K \hc D(\id, g)$ and $K(f, g)$ through each other we obtain the canonical isomorphism $C(f, \id) \hc K \hc D(\id, g) \iso K(f,g)$.
\end{proof}
\begin{example}
	In $\Prof$ the isomorphism $C(f, \id) \hc_C K \hc_D D(\id, g) \iso K(f, g)$ is simply an instance of the Yoneda isomorphisms, which were described in \exref{example:unenriched profunctors}. The same holds for $\enProf\V$.
\end{example}
  
  Remember that an \emph{adjunction} $l \ladj r$ in a bicategory $\catvar B$ consists of morphisms $\map lXY$ and $\map rYX$, as well as cells $\cell\eta{\id_X}{r \of l}$ (the \emph{unit}) and $\cell\eps{l \of r}{\id_Y}$ (the \emph{counit}), satisfying the triangle identities
  \begin{displaymath}
	  \bigbrks{l \xRar{\id_l \hc \eta} l \of r \of l \xRar{\eps \hc \id_l} l} = \id_l \qquad \text{and} \qquad \bigbrks{r \xRar{\eta \hc \id_r} r \of l \of r \xRar{\id_r \hc \eps} r} = \id_r.
	\end{displaymath}
	These generalise the usual adjunctions between categories, in the $2$-category $\Cat$ of categories, functors and natural transformations. The following, which is a direct consequence of the companion and conjoint identities, will be used throughout.
\begin{proposition} \label{companion-conjoint adjunction}
	For any morphism $\map fAC$ in an equipment $\K$, the companion $C(f, \id)$ is left adjoint to the conjoint $C(\id, f)$ in the horizontal bicategory $H(\K)$. The unit $\ls f\eta_f$ and counit $\ls f\eps_f$ of this adjunction are given by
	\begin{flalign*}
		&& \ls f\eta_f &= \bigbrks{U_A \iso U_A \hc U_A \xRar{\ls f\eta \hc \eta_f} C(f, \id) \hc C(\id, f)}& \\
		\text{and} &&\ls f\eps_f &= \bigbrks{C(\id, f) \hc C(f, \id) \xRar{\eps_f \hc \ls f\eps} U_C \hc U_C \iso U_C}.&
	\end{flalign*}
\end{proposition}
\begin{example}
	For a functor $\map fAC$ the unit $\ls f\eta_f$ of the adjunction $C(f, \id) \ladj C(\id, f)$ is, under the Yoneda isomorphism $C(f, \id) \hc_C C(\id, f) \iso C(f, f)$, simply given by the action $\map f{A(a_1, a_2)}{C(fa_1, fa_2)}$ of $f$ on maps in $A$, while the counit $\ls f\eps_f$ is induced by composition $C(c_1, fb) \times C(fb, c_2) \to C(c_1, c_2)$ in $C$.
\end{example}

	The following definition recalls the notions of `monoid' and `bimodule' in pseudo double categories, from \cite[Section 11]{Shulman08}; see also \cite[Section 5.3]{Leinster04}. These are generalisations of the usual notions of monoids and bimodules in a monoidal category, including the famous example of rings being monoids in the monoidal category $\mathsf{Ab}$ of abelian groups and their tensor product.
\begin{definition}\label{definition:monoids and bimodules}
	Let $\K$ be a pseudo double category.
	\begin{itemize}[label=-]
		\item	A \emph{monoid} $A$ in $\K$ consists of $A = (A_0, A, m, e)$ where $\hmap A{A_0}{A_0}$ is a horizontal morphism in $\K$ and $\cell m{A \hc A}A$ and $\cell e{U_{A_0}}A$ are horizontal cells satisfying the associativity and unit axioms below. The cells $m$ and $e$ are called the \emph{multiplication} and \emph{unit} of $A$.
		\begin{gather*}
			m \of (\id_A \hc m) = m \of (m \hc \id_A) \\
			m \of (e \hc \id_A) = \id_A = m \of (\id_A \hc e)
		\end{gather*}
		\item	Given monoids $A$ and $B$, a \emph{morphism of monoids} $\map fAB$ consists of a vertical morphism $\map{f_0}{A_0}{B_0}$ and a cell
		\begin{displaymath}
			\begin{tikzpicture}
				\matrix(m)[math175em]{A_0 & A_0 \\ B_0 & B_0 \\};
				\path[map]	(m-1-1) edge[barred] node[above] {$A$} (m-1-2)
														edge node[left] {$f_0$} (m-2-1)
										(m-1-2) edge node[right] {$f_0$}  (m-2-2)
										(m-2-1) edge[barred] node[below] {$B$} (m-2-2);
				\path[transform canvas={shift=($(m-1-1)!0.5!(m-2-1)$)}]	(m-1-2) edge[cell] node[right] {$f$} (m-2-2);
			\end{tikzpicture}
		\end{displaymath}
		that is compatible with the multiplications and units of $A$ and $B$:
		\begin{displaymath}
			m_B \of (f \hc f) = f \of m_A \qquad \text{and} \qquad f \of e_A = e_B \of U_{f_0}.
		\end{displaymath}
		\item	Given monoids $A$ and $B$, an \emph{$(A,B)$-bimodule} $J$ consists of $J = (J, l, r)$, where $\hmap J{A_0}{B_0}$ is a horizontal morphism in $\K$ and $\cell l{A \hc J}J$ and $\cell r{J \hc B}J$ are horizontal cells defining the \emph{actions} of $A$ and $B$ on $J$. These actions each satisfy the associativity and unit axioms
		\begin{displaymath}
			l \of (\id \hc l) = l \of (m \hc \id) \qquad \text{and} \qquad l \of (e \hc \id) = \id,
		\end{displaymath}
		likewise for $r$, and together satisfy the commutativity axiom
		\begin{displaymath}
			l \of (\id \hc r) = r \of (l \hc \id).
		\end{displaymath}
		\item Given monoids $A$, $B$, $C$ and $D$, morphisms $\map fAB$ and $\map gCD$ and bimodules $\hmap JAB$ and $\hmap KCD$, a \emph{cell} $\phi$ as on the left below
		\begin{displaymath}
			\begin{tikzpicture}[baseline]
				\matrix(m)[math175em]{A & B \\ C & D \\};
				\path[map]	(m-1-1) edge[barred] node[above] {$J$} (m-1-2)
														edge node[left] {$f$} (m-2-1)
										(m-1-2) edge node[right] {$g$}  (m-2-2)
										(m-2-1) edge[barred] node[below] {$K$} (m-2-2);
				\path[transform canvas={shift=($(m-1-1)!0.5!(m-2-1)$)}]	(m-1-2) edge[cell] node[right] {$\phi$} (m-2-2);
			\end{tikzpicture}
			\qquad\qquad\qquad\begin{tikzpicture}[baseline]
				\matrix(m)[math175em]{A_0 & B_0 \\ C_0 & D_0 \\};
				\path[map]	(m-1-1) edge[barred] node[above] {$J$} (m-1-2)
														edge node[left] {$f_0$} (m-2-1)
										(m-1-2) edge node[right] {$g_0$}  (m-2-2)
										(m-2-1) edge[barred] node[below] {$K$} (m-2-2);
				\path[transform canvas={shift=($(m-1-1)!0.5!(m-2-1)$)}]	(m-1-2) edge[cell] node[right] {$\phi$} (m-2-2);
			\end{tikzpicture}
		\end{displaymath}
		is a cell in $\K$ as on the right, satisfying the following compatibility axioms.
		\begin{displaymath}
			l_K \of (f \hc \phi) = \phi \of l_J \qquad \qquad r_K \of (\phi \hc g) = \phi \of r_J
		\end{displaymath}
	\end{itemize}
\end{definition}
	
	Before giving examples we recall \cite[Proposition 11.10]{Shulman08}, which asserts that monoids and bimodules in an equipment $\K$ form themselves an equipment $\Mod\K$, under a mild condition. Recall that a \emph{reflexive pair} in a category $\C$ is a pair of parallel maps $f$, $\mmap gXY$ together with a common section $\map sYX$, that is $f \of s = \id_Y = g \of s$. A \emph{reflexive coequaliser} is a coequaliser of a reflexive pair.
\begin{proposition}[Shulman] \label{bimodule equipments}
	Let $\K$ be an equipment such that the categories $H(\K)(A, B)$ have reflexive coequalisers, that are preserved by horizontal composition on both sides. Then the monoids and bimodules of $\K$, together with their morphisms and cells, again form an equipment $\Mod\K$.
\end{proposition}
\begin{proof}[Sketch of the proof.]
	We briefly recall the constructions that we will use here; for the details see Shulman's article \cite[Proposition 11.10]{Shulman08}.	The vertical composition of morphisms of monoids and that of cells between bimodules is given by vertical composition in $\K$. For the horizontal composition, consider two composable bimodules $\hmap JAB$ and $\hmap HBC$; their horizontal composite $\hmap{J \hc_B H}AC$ is constructed as the reflexive coequaliser
	\begin{displaymath}
		\begin{tikzpicture}
			\matrix(m)[math175em, column sep=2.5em]{J \hc B \hc H & J \hc H & J \hc_B H \\};
			\path	(m-1-1)	edge[transform canvas={yshift=3.5pt}, cell] node[above] {$r \hc \id$} (m-1-2)
										edge[transform canvas={yshift=-3.5pt}, cell] node[below] {$\id \hc l$} (m-1-2)
						(m-1-2) edge[cell] (m-1-3);
		\end{tikzpicture}
	\end{displaymath}
	in $H(\K)(A_0, C_0)$. By the unit axioms for the actions of $B$ on $J$ and $H$, a section for the parallel cells can be given by the composite
	\begin{displaymath}
		J \hc H \iso J \hc U_{B_0} \hc H \xRar{\id \hc e \hc \id} J \hc B \hc H.
	\end{displaymath}
	The unit for horizontal composition, on a monoid $A$, is itself considered as an $(A, A)$\ndash bimodule.
	
	Given horizontally composable cells
	\begin{displaymath}
    \begin{tikzpicture}
      \matrix(m)[math175em]{A \nc B \nc E \\ C \nc D \nc F \\};
      \path[map]  (m-1-1) edge[barred] node[above] {$J$} (m-1-2)
                          edge node[left] {$f$} (m-2-1)
                  (m-1-2) edge[barred] node[above] {$H$} (m-1-3)
                          edge node[right] {$g$} (m-2-2)
                  (m-1-3) edge node[right] {$h$} (m-2-3)
                  (m-2-1) edge[barred] node[below] {$K$} (m-2-2)
                  (m-2-2) edge[barred] node[below] {$L$} (m-2-3);
      \path[transform canvas={shift={($(m-1-1)!0.5!(m-2-2)$)}}]
                  (m-1-2) edge[cell] node[right] {$\phi$} (m-2-2)
                  (m-1-3) edge[cell] node[right] {$\psi$} (m-2-3);
    \end{tikzpicture}
  \end{displaymath}
  of bimodules, their horizontal composition $\phi \hc_g \psi$ is defined as follows. First consider the composition $J \hc H \xRar{\phi \hc \psi} K \hc L \Rar K \hc_D L$, which factors as a horizontal cell $J \hc H \Rar (K \hc_D L)(f, h)$ through the cartesian filler for the niche $A \xrar f C \xslashedrightarrow{K \hc_D L} F \xlar h E$. The fact that the cells $\phi$ and $\psi$ are compatible with the actions implies that this factorisation factors further as $J \hc_B H \Rar (K \hc_D L)(f, h)$ which, composed with the cartesian filler, corresponds to a cell $J \hc_B H \Rar K \hc_D L$. This is what $\phi \hc_g \psi$ is defined to be.
  
  Finally consider a niche $A \xrar f C \xslashedrightarrow K D \xlar g B$ consisting of morphisms of monoids $f$ and $g$ and a bimodule $K$. The underlying niche $A_0 \xrar{f_0} C_0 \xslashedrightarrow K D_0 \xlar{g_0} B_0$ has a cartesian filler $\cell\eps{K(f_0, g_0)}K$ in $\K$, which can be made into a cell of bimodules as follows. The action $A \hc K(f_0, g_0) \Rar K(f_0, g_0)$ of $A$ is given by the factorisation of $A \hc K(f_0, g_0) \xRar{f \hc \eps} C \hc K \xRar l K$ through $\eps$; the action of $B$ is given similarly. That $\eps$ is compatible with these actions follows immediately. Now the underlying cell of any other filler $\phi$ will factor through $\eps$ as a cell $\phi'$ in $\K$. It is easy to see that the universality of $\eps$ in $\K$ implies that $\phi'$ is compatible with the actions on its source and target bimodules, so that it is a cell of bimodules: we conclude that $\eps$ is universal in $\Mod\K$ as well.
\end{proof}

\begin{example}
	If $\V$ is a cocomplete symmetric pseudomonoidal category whose binary tensor product $\dash \tens \dash$ preserves colimits in both variables, then the equipment $\Mat\V$ of $\V$-matrices satisfies the condition of the previous proposition, so that the monoids and bimodules in $\Mat\V$ form an equipment $\Mod{\Mat\V}$.
	
	A monoid $A$ in $\Mat\V$ consists of a set $A_0$ and a $\V$-object $A(a, b)$ for every pair $a$, $b \in A_0$, such that the maps $\map m{A(a, b) \tens A(b, c)}{A(a, c)}$ and \mbox{$\map e1{A(a, a)}$}, induced by its multiplication and unit, make $A$ into a $\V$-category. Likewise a morphism \mbox{$\map fAB$} of monoids is a $\V$-functor. Given $\V$-categories $A$ and $B$, an $(A, B)$\ndash bimodule \mbox{$\hmap JAB$} consists of $\V$-objects $J(a, b)$, for $a \in A$, $b \in B$, together with action maps
	\begin{displaymath}
		\map l{A(a_1, a_2) \tens J(a_2, b)}{J(a_1, b)} \qquad \text{and} \qquad \map r{J(a, b_1) \tens B(b_1, b_2)}{J(a, b_2)},
	\end{displaymath}
	which can be reconsidered as $\V$\ndash functors $\map{J(\dash, b)}{\op A}\V$ and $\map{J(a, \dash)}B\V$ respectively, for each $a \in A$ and $b \in B$. The compatibility axiom for $l$ and $r$ implies that giving such families of $\V$-functors is equivalent to giving a single $\V$-profunctor $\map J{\op A \tens B}\V$ (see \cite[Section 1.4]{Kelly82}). Likewise, a cell of bimodules is precisely a natural transformation of $\V$-profunctors.
	
	One readily checks that the horizontal compositions, that make $\Mod{\Mat\V}$ into a pseudo double category and that are given in the proof above, coincide with the horizontal compositions of $\enProf\V$ as given in \exref{example:enriched profunctor equipments}. We conclude that the equipments $\Mod{\Mat\V}$ and $\enProf\V$ are `isomorphic', even though we do not yet know what that means; functors between pseudo double categories will be treated in \chapref{chapter:monads on equipments}.
\end{example}

	Monoids in $\Span\E$ are `internal categories in $\E$', as follows.
\begin{example} \label{example:internal profunctor equipments}
	A monoid $A$ in $\Span\E$ is given by a span $\map dA{A_0 \times A_0}$ in $\E$ equipped with a multiplication and unit as follows; we will write $d = (d_0, d_1)$, and $\map{d_2 = (d_0 \of p, d_1 \of q)}{A \times_{A_0} A}{A_0 \times A_0}$ for the induced map on the composite of $A$ with itself, where $p$ and $q$ are the projections onto the first and second factor of $A \times_{A_0} A$. The multiplication and unit of $A$ are given by morphisms
	\begin{displaymath}
		\begin{tikzpicture}[baseline]
			\matrix(m)[math175em, column sep=0.5em]{A \times_{A_0} A & & A \\ & A_0 \times A_0 & \\};
			\path[map]	(m-1-1) edge node[above] {$m$} (m-1-3)
													edge node[below left] {$d_2$} (m-2-2)
									(m-1-3) edge node[below right] {$d$} (m-2-2);
		\end{tikzpicture}
		\qquad\qquad\begin{tikzpicture}[baseline]
			\matrix(m)[math175em, column sep=0.5em]{A_0 & & A \\ & A_0 \times A_0 & \\};
			\path[map]	(m-1-1) edge node[above] {$e$} (m-1-3)
													edge node[below left] {$(\id, \id)$} (m-2-2)
									(m-1-3) edge node[below right] {$d$} (m-2-2);
		\end{tikzpicture}
	\end{displaymath}
	in the slice category $\E \slash A_0 \times A_0$. Satisfying the associativity and unit axioms, these make $A$ into a \emph{category internal in $\E$} (see \cite[Section XII.1]{MacLane98}), with an `object of objects' $A_0$ and an `object of morphisms' $A$. The maps $\map{d_0}A{A_0}$ and $\map{d_1}A{A_0}$ are the source and target maps, so that the pullback $A \times_{A_0} A$ of $d_0$ and $d_1$ forms the `object of composable morphisms'. The map $m$ then gives the composition of $A$ while $e$ supplies its identities. In the case that $\E = \Set$ this recovers ordinary categories as categories internal in $\Set$, with the difference that where an ordinary category $A'$ has sets of maps $A'(a, b)$ for any pair of objects $a$ and $b$, a category $A$ internal in $\Set$ has a single set of maps $A$, and comes with a map $\map dA{A_0 \times A_0}$ that specifies sources and targets. 
	
	Likewise a map of monoids $\map fAB$ is an \emph{internal functor}: it consists of a pair of maps $\map fAB$ and $\map{f_0}{A_0}{B_0}$ making
	\begin{displaymath}
		\begin{tikzpicture}
			\matrix(m)[math175em]{A & B \\ A_0 \times A_0 & B_0 \times B_0 \\};
			\path[map]	(m-1-1) edge node[above] {$f$} (m-1-2)
													edge node[left] {$d_A$} (m-2-1)
									(m-1-2) edge node[right] {$d_B$} (m-2-2)
									(m-2-1) edge node[below] {$f_0 \times f_0$} (m-2-2);
		\end{tikzpicture}
	\end{displaymath}
	commute, and which are compatible with the composition and units of $A$ and $B$.
	
	An $(A, B)$-bimodule $J$ is given by a span $\map dJ{A_0 \times B_0}$, which can be thought of as an object of `$\E$-objects indexed by pairs of objects in $A_0$ and $B_0$', with action maps $\map l{A \times_{A_0} J}J$ and $\map r{J \times_{B_0} B}J$, over $A_0 \times B_0$, that are associative and unital. Bimodules in $\Span\E$ are called \emph{internal profunctors} (`internal hom-functors' in \cite{MacLane98}). Again profunctors $J$ that are internal in $\Set$ are equivalent to the usual profunctors, but are given as a single set $J$ of morphisms, as opposed to giving a set $J(a, b)$ of morphisms for each pair of objects $a \in A$ and $c \in C$. A cell $\phi$ as on the left below is given by a morphism $\map\phi JK$ in $\E$ on the right, compatible with actions of the internal categories involved. Such a cell is called an \emph{internal transformation} of internal profunctors.
	\begin{equation} \label{diagram:natural transformation of internal profunctors}
		\begin{tikzpicture}[baseline]
				\matrix(m)[math175em]{A \nc B \\ C \nc D \\};
				\path[map]	(m-1-1) edge[barred] node[above] {$J$} (m-1-2)
														edge node[left] {$f$} (m-2-1)
										(m-1-2) edge node[right] {$g$}  (m-2-2)
										(m-2-1) edge[barred] node[below] {$K$} (m-2-2);
				\path[transform canvas={shift=($(m-1-1)!0.5!(m-2-1)$)}]	(m-1-2) edge[cell] node[right] {$\phi$} (m-2-2);
		\end{tikzpicture}
		\qquad\qquad\qquad\begin{tikzpicture}[baseline]
			\matrix(m)[math175em]{J \nc K \\ A_0 \times B_0 \nc C_0 \times D_0 \\};
			\path[map]	(m-1-1) edge node[above] {$\phi$} (m-1-2)
													edge node[left] {$d_J$} (m-2-1)
									(m-1-2) edge node[right] {$d_K$} (m-2-2)
									(m-2-1) edge node[below] {$f_0 \times g_0$} (m-2-2);
		\end{tikzpicture}
	\end{equation}
	Just like there is a notion of natural transformations between ordinary functors, there is a notion of `internal transformation' between a pair of internal functors $f$ and $\map gAB$ as well: see \eqref{diagram:natural transformation of internal functors} below. Like in the ordinary case (\exref{example:vertical cells in Prof}), such transformations correspond to vertical cells $\phi$ between internal profunctors, as we shall see in \propref{internal transformations}.
	
	If $\E$ has reflexive coequalisers preserved by pullbacks then $\Span\E$ satisfies the condition of the proposition above, so that categories and profunctors internal to $\E$ form an equipment, which we shall denote $\inProf\E$. Its horizontal composition $J \hc_B H$ of internal profunctors $\hmap JAB$ and $\hmap HBC$ can be given by the reflexive coequaliser of pullbacks
	\begin{displaymath}
		J \times_{B_0} B \times_{B_0} H \rightrightarrows J \times_{B_0} H \to J \hc_B H,
	\end{displaymath}
	where the parallel pair of maps is given by the actions of $B$ on $J$ and $H$. The internal unit profunctor $U_A$ is simply $\map dA{A_0 \times A_0}$ itself, with actions given by composition. Moreover, the horizontal composite $\phi \hc_g \psi$ of two cells on the left below is given by the unique factorisation on the right.
	\begin{displaymath}
    \begin{tikzpicture}[baseline]
      \matrix(m)[math175em]{A \nc B \nc E \\ C \nc D \nc F \\};
      \path[map]  (m-1-1) edge[barred] node[above] {$J$} (m-1-2)
                          edge node[left] {$f$} (m-2-1)
                  (m-1-2) edge[barred] node[above] {$H$} (m-1-3)
                          edge node[right] {$g$} (m-2-2)
                  (m-1-3) edge node[right] {$h$} (m-2-3)
                  (m-2-1) edge[barred] node[below] {$K$} (m-2-2)
                  (m-2-2) edge[barred] node[below] {$L$} (m-2-3);
      \path[transform canvas={shift={($(m-1-1)!0.5!(m-2-2)$)}}]
                  (m-1-2) edge[cell] node[right] {$\phi$} (m-2-2)
                  (m-1-3) edge[cell] node[right] {$\psi$} (m-2-3);
    \end{tikzpicture}
		\qquad\qquad\qquad\begin{tikzpicture}[baseline]
			\matrix(m)[math175em]{J \times_{B_0} H & J \hc_B H \\ K \times_{D_0} L & K \hc_D L \\};
			\path[map]	(m-1-1) edge (m-1-2)
													edge node[left] {$\phi \times_{g_0} \psi$} (m-2-1)
									(m-1-2) edge[dashed] node[right] {$\phi \hc_g \psi$} (m-2-2)
									(m-2-1) edge (m-2-2);
		\end{tikzpicture}
  \end{displaymath}
  
	Finally, the proof of the proposition above shows that the cartesian filler $K(f, g)$ of a niche $A \xrar f C \xslashedrightarrow K D \xlar g B$ in $\inProf\E$ can be constructed in $\Span\E$, so the span underlying the profunctor $K(f, g)$ is the pullback $(f_0 \times g_0)^* k$, where $\map kK{C_0 \times D_0}$ is the span underlying $K$. The factorisation $\map{\ls f\phi_g}J{K(f, g)}$ of any internal transformation $\phi$ of the form \eqref{diagram:natural transformation of internal profunctors}, is obtained by factoring its underlying map $\map \phi JK$ through the pullback $K(f, g)$.
\end{example}
	
	Although internal categories can be considered in any category $\E$ with finite limits, we will mostly be interested in the case of $\E$ being the category $\ps\SS = \fun{\op\SS} \Set$ of presheaves on a small category $\SS$. 
\begin{example}\label{example:monoids in presheaf span equipments}
	Let $\SS$ be a small category. Since pullbacks of presheaves are given pointwise, it follows that giving the multiplication $m$ and unit $e$ of a monoid $\hmap A{A_0}{A_0}$ in $\Span{\ps\SS}$ is the same as giving the sets $(A_0)_s$ and $A_s$, for each $s \in \ob S$, the structure of a category. Moreover, the fact that $m$ and $e$ of $A$ are maps of presheaves is equivalent to the two images of each $\map urs$ in $\SS$, under $A$ and $A_0$, constituting a functor $\map{A_u}{A_s}{A_r}$. Thus monoids in $\Span{\ps\SS}$ are functors $\map A{\op\SS}\cat$, which we shall call \emph{$\SS$-indexed categories}\footnote{We should warn the reader that the same terminology is used in literature to mean \emph{pseudo}functors $\op\SS \to \Cat$, where $\Cat$ is the $2$-category of categories, functors and transformations.}. Similarly a map of monoids $\map fAB$ is an \emph{$\SS$-indexed functor}, that is a natural transformation $\nat fAB$.
	
	An $(A, B)$-bimodule $J$ in $\Span{\ps\SS}$ is, by definition, a natural transformation of presheaves $\map dJ{A_0 \times B_0}$ together with actions $\map l{A \times_{A_0} J}J$ and $\map r{J \times_{B_0} B}J$, which are transformations over $A_0 \times B_0$ as well, and satisfy the bimodule axioms. We will call such bimodules \emph{$\SS$-indexed profunctors}. An elementary calculation will show that the coequalisers of $\E = \Set$ are preserved by pullbacks; that this is the case for any category of presheaves $\E = \ps\SS$ follows since limits and colimits of presheaves are computed pointwise. Hence we can apply \propref{bimodule equipments}, so that $\SS$-indexed categories and $\SS$-indexed profunctors form an equipment $\psProf{\SS} = \Mod{\Span{\ps\SS}}$.
\end{example}

	As an example we consider the equipment $\psProf{\GG_1}$ of categories and profunctors indexed by $\GG_1 = (0 \rightrightarrows 1)$. Briefly speaking, categories indexed by $\GG_1$ are `double categories without horizontal compositions'.
\begin{example}\label{example:GG1-indexed bimodules}	
	A $\GG_1$-indexed category $\map A{\op{\GG_1}}\cat$ consists of two categories $A_0$ and $A_1$ together with a pair of functors $L$ and $\map R{A_1}{A_0}$, that is it consists of the data underlying a category internal in $\Cat$ as in \eqref{diagram:internal category}. Hence, like a double category, we think of the objects $a, a', \dotsc$ of $A_0$ as objects of $A$, the morphisms $\map pa{a'}$ of $A_0$ as vertical morphisms of $A$, the objects $j$ of $A_1$, with $Lj = a_1$ and $Rj = a_2$, as horizontal morphisms $\hmap j{a_1}{a_2}$ of $A$ and, lastly, the morphisms $\map xjk$ of $A_1$, with $Lx = p$ and $Rx = q$, as cells $x$ of $A$ of the form as shown below. The composition of $A_0$ and $A_1$ gives vertical composition of vertical morphisms and cells respectively. In other words, a $\GG_1$-indexed category $A$ is a `double category without horizontal compositions'. In fact, in \secref{section:monad examples} we shall consider the `free strict double category'-monad $D$ on the double category of $\GG_1$-indexed categories and their profunctors, and a $D$-algebra structure on a $\GG_1$-indexed category $A$ will equip $A$ with horizontal compositions.
	\begin{displaymath}
		\begin{tikzpicture}
			\matrix(m)[math175em]{a_1 \nc a_2 \\ a'_1 \nc a'_2 \\};
			\path[map]  (m-1-1) edge[barred] node[above] {$j$} (m-1-2)
													edge node[left] {$p$} (m-2-1)
									(m-1-2) edge node[right] {$q$} (m-2-2)
									(m-2-1) edge[barred] node[below] {$k$} (m-2-2);
			\path[transform canvas={shift={($(m-1-2)!(0,0)!(m-2-2)$)}}] (m-1-1) edge[cell] node[right] {$x$} (m-2-1);
		\end{tikzpicture}
	\end{displaymath}
	A $\GG_1$-indexed functor $\map fAB$ is simply given by two functors $\map {f_0}{A_0}{B_0}$ and $\map{f_1}{A_1}{B_1}$ such that $L \of f_1 = f_0 \of L$ and $R \of f_1 = f_0 \of R$. Thus $f$ maps objects, vertical morphisms and horizontal morphisms, as well as cells, of $A$ to those of $B$, in a way that preserves vertical and horizontal sources and targets, as well as vertical compositions.

	A $\GG_1$-indexed profunctor $\hmap JAB$ is given as follows. Its underlying span in $\Span{\ps{\GG_1}}$ consists of sets $J_0$ and $J_1$, together with functions $L$, $R$, $d_0$ and $d_1$, that make both the `$L$-square' and the `$R$-square' in the diagram on the left below commute.
	\begin{displaymath}
		\begin{tikzpicture}[baseline]
			\matrix(m)[math2em]
			{ J_1 & \ob A_1 \times \ob B_1 \\
				J_0 & \ob A_0 \times \ob B_0 \\	};
			\path[map]	(m-1-1) edge node[above] {$d_1$} (m-1-2)
													edge[transform canvas={xshift=-2.5pt}] node[left] {$L$} (m-2-1)
													edge[transform canvas={xshift=2.5pt}] node[right] {$R$} (m-2-1)
									(m-1-2) edge[transform canvas={xshift=-2.5pt}] node[left] {$L \times L$} (m-2-2)
													edge[transform canvas={xshift=2.5pt}] node[right] {$R \times R$} (m-2-2)
									(m-2-1) edge node[below] {$d_0$} (m-2-2);
		\end{tikzpicture}
		\qquad\qquad\qquad\qquad\begin{tikzpicture}[baseline]
				\matrix(m)[math175em]{a_1 \nc a_2 \\ b_1 \nc b_2 \\};
				\path[map]  (m-1-1) edge[barred] node[above] {$h$} (m-1-2)
														edge node[left] {$s$} (m-2-1)
										(m-1-2) edge node[right] {$t$} (m-2-2)
										(m-2-1) edge[barred] node[below] {$k$} (m-2-2);
				\path[transform canvas={shift={($(m-1-2)!(0,0)!(m-2-2)$)}}] (m-1-1) edge[cell] node[right] {$u$} (m-2-1);
		\end{tikzpicture}
	\end{displaymath}
	Given a pair of objects $a \in A$ and $b \in B$ we write $J_0(a, b) = \inv d_0\set{(a, b)}$, and likewise $J_1(h, k) = \inv d_1\set{(h,k)}$ for horizontal morphisms $\hmap h{a_1}{a_2}$ in $A$ and $\hmap k{b_1}{b_2}$ in $B$.	As with ordinary profunctors (\exref{example:unenriched profunctors}) we shall think of elements $s  \in J_0(a_1, b_1)$ as vertical morphisms $\map s{a_1}{b_1}$, and of $u \in J_1(h, k)$, with $Lu = s$ and $Ru = t$, as a cell $u$ as shown on right above. That the sources and targets of $k$, $h$, $s$ and $t$ indeed coincide as drawn follows from the commuting of the two squares on the left above. The left action of $A$ on $J$ is given by functions $\map\of{A_i \times_{\ob A_i} J_i}{J_i}$, for $i = 0, 1$, that make both squares in the diagram on the left below commute, and that satisfy associativity and unit axioms.
	\begin{displaymath}
		\begin{tikzpicture}[baseline]
			\matrix(m)[math2em]
			{	A_1 \times_{\ob A_1} J_1 & J_1 \\
				A_0 \times_{\ob A_0} J_0 & J_0 \\ };
			\path[map]	(m-1-1) edge node[above] {$\of$} (m-1-2)
													edge[transform canvas={xshift=-2.5pt}] node[left] {$L \times_L L$} (m-2-1)
													edge[transform canvas={xshift=2.5pt}] node[right] {$R \times_R R$} (m-2-1)
									(m-1-2) edge[transform canvas={xshift=-2.5pt}] node[left] {$L$} (m-2-2)
													edge[transform canvas={xshift=2.5pt}] node[right] {$R$} (m-2-2)
									(m-2-1) edge node[below] {$\of$} (m-2-2);
		\end{tikzpicture}
		\qquad\qquad\qquad\qquad\begin{tikzpicture}[baseline]
			\matrix(m)[math175em]{a'_1 \nc a'_2 \\ a_1 \nc a_2 \\};
			\path[map]  (m-1-1) edge[barred] node[above] {$j$} (m-1-2)
													edge node[left] {$p$} (m-2-1)
									(m-1-2) edge node[right] {$q$} (m-2-2)
									(m-2-1) edge[barred] node[below] {$h$} (m-2-2);
			\path[transform canvas={shift={($(m-1-2)!(0,0)!(m-2-2)$)}}] (m-1-1) edge[cell] node[right] {$x$} (m-2-1);
		\end{tikzpicture}
	\end{displaymath}
	This lets us `precompose' the morphisms and cells of $A$ with those of $J$: for example we can compose the cell $x$ in $A$ with $u$ above to obtain a cell $\cell{u \of x}jk$ in $J$. The commuting of the diagram on the left states that these compositions need to be compatible with $L$ and $R$ (e.g.\ $L(u \of x) = s \of p$); of course they also need to be associative and unital. Analogously $J$ comes equipped with an action of $B$, given by functions $\map\of{B_i \times_{\ob B_i} J_i}{J_i}$ that satisfy similar axioms, while the actions of $A$ and $B$ commute.

	Given a second $\GG_1$-profunctor $\hmap HBC$, the horizontal composition $J \hc_B H$ is given by the coequalisers
	\begin{equation} \label{equation:composition of internal profunctors}
		J_i \times_{\ob B_i} B_i \times_{\ob B_i} H_i \rightrightarrows J_i \times_{\ob B_i} H_i \to (J \hc_B H)_i,
	\end{equation}
	for $i = 0, 1$, where the parallel pair of maps is given by the actions of $B$ on $J$ and $H$. Thus, for an object $a$ in $A$ and $c$ in $C$, the set of vertical maps $(J \hc_B H)_0(a, c)$ is the set of equivalence classes of formal vertical composites
	\begin{displaymath}
		a \xrar s b \xrar t c, \qquad\qquad\qquad\qquad \text{$s \in J$ and $t \in H$,}
	\end{displaymath}
	modulo the usual coend relation, which coequalises post- and precomposition with maps of $B_0$. The sets of cells $(J \hc_B H)_1$ are given similarly: they consist of equivalence classes of formal vertical composites of cells.
	
	This leaves the $\GG_1$-indexed cells in $\psProf{\GG_1}$, as on the left below, where $J$ and $K$ are $\GG_1$-indexed profunctors and $f$ and $g$ are $\GG_1$-indexed functors.
\begin{displaymath}
	\begin{tikzpicture}[baseline]
		\matrix(m)[math175em]{A & B \\ C & D \\};
		\path[map]  (m-1-1) edge[barred] node[above] {$J$} (m-1-2)
												edge node[left] {$f$} (m-2-1)
								(m-1-2) edge node[right] {$g$} (m-2-2)
								(m-2-1) edge[barred] node[below] {$K$} (m-2-2);
		\path[transform canvas={shift={($(m-1-2)!(0,0)!(m-2-2)$)}}] (m-1-1) edge[cell] node[right] {$\phi$} (m-2-1);
	\end{tikzpicture}
	\qquad\qquad\begin{tikzpicture}[baseline]
		\matrix(m)[math175em]{J_i \nc \ob A_i \times \ob B_i \\ K_i \nc \ob C_i \times \ob D_i \\};
		\path[map]  (m-1-1) edge node[above] {$d_i$} (m-1-2)
												edge node[left] {$\phi_i$} (m-2-1)
								(m-1-2) edge node[right] {$f_i \times g_i$} (m-2-2)
								(m-2-1) edge node[below] {$d_i$} (m-2-2);
	\end{tikzpicture}
	\qquad\qquad\begin{tikzpicture}[baseline]
		\matrix(m)[math175em]{J_1 \nc J_0 \\ K_1 \nc K_0 \\};
		\path[map]  (m-1-1) edge[transform canvas={yshift=2.5pt}] node[above] {$L$} (m-1-2)
												edge[transform canvas={yshift=-2.5pt}] node[below] {$R$} (m-1-2)
												edge node[left] {$\phi_1$} (m-2-1)
								(m-1-2) edge node[right] {$\phi_0$} (m-2-2)
								(m-2-1) edge[transform canvas={yshift=2.5pt}] node[above] {$L$} (m-2-2)
												edge[transform canvas={yshift=-2.5pt}] node[below] {$R$} (m-2-2);
	\end{tikzpicture}
\end{displaymath}
	Such a cell $\phi$ is given by a pair of maps $\map{\phi_i}{J_i}{K_i}$, for $i = 0, 1$, that make the two diagrams on the right above commute, and that are compatible with the actions of $A$, $B$, $C$ and $D$. This means that $\phi_0$ and $\phi_1$ map the vertical morphisms and vertical cells of $J$ to those of $K$, in a way that is compatible with horizontal and vertical sources and targets, as shown below.
\begin{align*}
	\phi &\colon \bigpars{a \xrar s b} \mapsto \bigpars{fa \xrar{\phi s} gb} \\
	\phi &\colon \begin{tikzpicture}[textbaseline]
		\matrix(m)[math175em]{a_1 \nc a_2 \\ b_1 \nc b_2 \\};
		\path[map]  (m-1-1) edge[barred] node[above] {$h$} (m-1-2)
												edge node[left] {$s_1$} (m-2-1)
								(m-1-2) edge node[right] {$s_2$} (m-2-2)
								(m-2-1) edge[barred] node[below] {$k$} (m-2-2);
		\path[transform canvas={shift={($(m-1-2)!(0,0)!(m-2-2)$)}}] (m-1-1) edge[cell] node[right] {$u$} (m-2-1);
	\end{tikzpicture} \mapsto \begin{tikzpicture}[textbaseline]
	\matrix(m)[math175em]{fa_1 \nc fa_2 \\ gb_1 \nc gb_2 \\};
		\path[map]  (m-1-1) edge[barred] node[above] {$fh$} (m-1-2)
												edge node[left] {$\phi s_1$} (m-2-1)
								(m-1-2) edge node[right] {$\phi s_2$} (m-2-2)
								(m-2-1) edge[barred] node[below] {$gk$} (m-2-2);
		\path[transform canvas={shift={($(m-1-2)!(0,0)!(m-2-2)$)}}] (m-1-1) edge[cell] node[right] {$\phi u$} (m-2-1);
	\end{tikzpicture}
\end{align*}
	These assignments are compatible with respect to the actions of $A$ and $B$ in the usual sense: for example $\phi(u \of x) = \phi u \of fx$, for the composite of a cell $u$ in $J$ and $x$ in $A$.
\end{example}

	We now briefly return to the comment in \exref{example:internal profunctor equipments}, about the vertical cells in the equipment $\inProf\E$ of profunctors internal to $\E$. Just like natural transformations between ordinary functors, there is a notion of \emph{internal transformations} $\nat\phi fg$ between internal functors $f$ and $\map gAB$, as defined in e.g.\ \cite[page 107]{Street74}. Such transformations are given by a map $\phi$ as on the left below, such that the naturality diagram on the right commutes. In the proposition below we shall prove that, like in $\Prof$ (see \exref{example:vertical cells in Prof}), the vertical cells of $\inProf\E$ are exactly such internal transformations.
	\begin{equation} \label{diagram:natural transformation of internal functors}
		\begin{tikzpicture}[baseline]
			\matrix(m)[math175em, column sep=0.5em]{A_0 \nc \nc B \\ \nc B_0 \times B_0 \nc \\};
			\path[map]	(m-1-1) edge node[above] {$\phi$} (m-1-3)
													edge node[below left] {$(f_0, g_0)$} (m-2-2)
									(m-1-3) edge node[below right] {$d_B$} (m-2-2);
		\end{tikzpicture}
		\qquad\qquad\quad\begin{tikzpicture}[baseline]
			\matrix(m)[math175em]
			{	A \nc B \times_{B_0} B \\
				B \times_{B_0} B \nc B \\};
			\path[map]	(m-1-1) edge node[above] {$(f, \phi \of d_1)$} (m-1-2)
													edge node[left] {$(\phi \of d_0, g)$} (m-2-1)
									(m-1-2) edge node[right] {$m_B$} (m-2-2)
									(m-2-1) edge node[below] {$m_B$} (m-2-2);
		\end{tikzpicture}
	\end{equation}
	Internal categories, functors and the transformations above form a $2$-category $\inCat\E$, where the natural transformations compose as follows. The vertical composition of the transformations $\nat\phi fg$ and $\nat\psi gh$, where $f$, $g$, $\map hAB$, is given by
	\begin{displaymath}
		A_0 \iso A_0 \times_{A_0} A_0 \xrar{\phi \times \psi} B \times_{B_0} B \xrar{m_B} B,
	\end{displaymath}
	while the horizontal composite of $\nat\phi fg$ and $\nat\psi hk$, now with $f$, $\map gAB$ and $h$, $\map kBC$, is given by
	\begin{displaymath}
		A_0 \iso A_0 \times_{A_0} A_0 \xrar{\phi \times g_0} B \times_{B_0} B_0 \xrar{h \times \psi} C \times_{C_0} C \xrar{m_C} C.
	\end{displaymath}
\begin{proposition} \label{internal transformations}
	Let $\E$ be a category with finite limits. Transformations $\phi$ of internal profunctors (see \exref{example:internal profunctor equipments}), that are of the form as on the left below, correspond to maps $\phi_0$ in $\E$, as in the middle, that make the diagram on the right commute.
	\begin{equation} \label{equation:126483}
		\begin{split}
		\begin{tikzpicture}[baseline]
			\matrix(m)[math175em]{A \nc A \\ C \nc D \\};
			\path[map]	(m-1-1) edge node[left] {$f$} (m-2-1)
									(m-1-2) edge node[right] {$g$} (m-2-2)
									(m-2-1) edge[barred] node[below] {$K$} (m-2-2);
			\path				(m-1-1) edge[eq] (m-1-2)
													edge[transform canvas={shift=($(m-1-2)!0.5!(m-2-2)$)}, cell] node[right] {$\phi$} (m-2-1);				
		\end{tikzpicture}
		\quad\quad\begin{tikzpicture}[baseline]
			\matrix(m)[math175em, column sep=0em]{A_0 \nc \nc K \\ \nc C_0 \times D_0 \nc \\};
			\path[map]	(m-1-1) edge node[above] {$\phi_0$} (m-1-3)
													edge node[below left] {$(f_0, g_0)$} (m-2-2)
									(m-1-3) edge node[below right] {$k$} (m-2-2);
		\end{tikzpicture}
		\quad\quad\begin{tikzpicture}[baseline]
			\matrix(m)[math175em]
			{	A \nc C \times_{C_0} K \\
				K \times_{D_0} D \nc K \\};
			\path[map]	(m-1-1) edge node[above] {$(f, \phi_0 \of d_1)$} (m-1-2)
													edge node[left] {$(\phi_0 \of d_0, g)$} (m-2-1)
									(m-1-2) edge node[right] {$l$} (m-2-2)
									(m-2-1) edge node[below] {$r$} (m-2-2);
		\end{tikzpicture}
		\end{split}
	\end{equation}
	This correspondence preserves the factorisations through the restriction $K(f, g)$ (again see \exref{example:internal profunctor equipments}) in the sense that $\map{(\ls f\phi_g)_0}{A_0}{K(f, g)}$ coincides with the factorisation of $\phi_0$ through the pullback $K(f, g)$. Moreover it preserves vertical composition in the sense that $(\chi \of \phi)_0 = \chi \of \phi_0$, for any transformation $\chi$ with horizontal source $K$, and horizontal composition in the sense that
	\begin{displaymath}
		\begin{tikzpicture}[baseline]
			\matrix(m)[math175em, column sep=2.25em]{A_0 \times_{A_0} H & K \times_{D_0} L \\ A \hc_A H & K \hc_D L \\};
			\path[map]	(m-1-1) edge node[above] {$\phi_0 \times_{g_0} \psi$} (m-1-2)
									(m-1-2) edge (m-2-2)
									(m-2-1) edge node[below] {$\phi \hc_g \psi$} (m-2-2);
			\path[color=white, font=]	(m-1-1) edge node[color=black, sloped] {$\iso$} (m-2-1);
		\end{tikzpicture}
	\end{displaymath}
	commutes, for any transformation $\cell\psi KL$ with left vertical map $g$.
	
	If $\E$ has coequalisers preserved by pullbacks, so that internal profunctors in $\E$ from an equipment $\inProf\E$ by \propref{bimodule equipments}, then this correspondence induces an isomorphism of $2$\ndash categories $V(\inProf\E) \iso \inCat\E$.
\end{proposition}
	We will often abuse notation and denote $\phi_0$ by $\phi$ as well; it will always be clear which of the two is meant.
\begin{proof}
	To a cell $\phi$ above, which is given by a map $\map\phi AK$ in $\E$, we assign the composition $\phi_0 = \bigbrks{A_0 \xrar{e_A} A \xrar\phi K}$. That it makes the diagram on the right in \eqref{equation:126483} commute follows from the fact that
	\begin{displaymath}
		\begin{tikzpicture}
			\matrix(m)[math175em, column sep=0.35em, row sep=0.75em]
				{	&[-0.75em] A \times_{A_0} A & & C \times_{C_0} K \\
					A \times_{A_0} A_0 & & & \\
					& A & & K \\
					A_0 \times_{A_0} A & & & \\
					& A \times_{A_0} A & & K \times_{D_0} D \\};
			\path[map]	(m-1-2) edge node[above] {$f \times \phi$} (m-1-4)
													edge node[right] {$m_A$} (m-3-2)
									(m-1-4) edge node[right] {$l$} (m-3-4)
									(m-2-1) edge node[above left] {$\id \times e_A$} (m-1-2)
									(m-3-2) edge node[below] {$\phi$} (m-3-4)
									(m-4-1) edge node[below left] {$e_A \times \id$} (m-5-2)
									(m-5-2) edge node[right] {$m_A$} (m-3-2)
													edge node[below] {$\phi \times g$} (m-5-4)
									(m-5-4) edge node[right] {$r$} (m-3-4);
			\path[color=white, font=]	(m-2-1) edge node[color=black, sloped, right=-4pt] {$\iso$} (m-3-2)
																(m-4-1) edge node[color=black, sloped, right=-4pt] {$\iso$} (m-3-2);
		\end{tikzpicture}
	\end{displaymath}
	commutes, where the commuting squares describe the naturality of $\phi$ and the commuting triangles form the unit axiom for $A$. Conversely a map $\map{\phi_0}{A_0}K$ induces a cell $\cell\phi{U_A}K$ given by the composite $A \to K$ of either leg of the naturality diagram \eqref{equation:126483} of $\phi_0$ above. That the resulting map is natural follows from the facts that $f$ and $g$ are compatible with composition and that the actions $l$ and $r$ are associative. That the two assignments $\phi \mapsto \phi_0$ and $\phi_0 \mapsto \phi$ thus given are inverse to each other follows from the commutativity diagram above, together with the fact that $f$ and $l$ (or $g$ and $r$) are compatible with the units of $A$, $B$ and $C$.
	
	We saw in \exref{example:internal profunctor equipments} that the restriction $K(f, g)$ is the pullback of $\map kK{C_0 \times D_0}$ along $f_0 \times g_0$; that the correspondence above preserves the factorisations through $K(f, g)$ in the way as claimed follows directly from the universal property of the pullback $K(f, g)$. That the diagram for the horizontal composition commutes follows from the commuting of
	\begin{displaymath}
		\begin{tikzpicture}
			\matrix(m)[math2em, column sep=2.25em]
			{	A_0 \times_{A_0} H & A \times_{A_0} H & K \times_{D_0} L \\
				H & A \hc_A H & K \hc_D L, \\ };
			\path[map]	(m-1-1) edge[bend left=35] node[above] {$\phi_0 \times_{g_0} \psi$} (m-1-3)
													edge node[above] {$e \times \id$} (m-1-2)
									(m-1-2) edge node[above] {$\phi \times_{g_0} \psi$} (m-1-3)
													edge (m-2-2)
													edge[shorten >= 4pt, shorten <= 2pt] node[below right, inner sep=1.5pt] {$l$} (m-2-1)
									(m-1-3) edge (m-2-3)
									(m-2-2) edge node[below] {$\phi \hc_g \psi$} (m-2-3);
			\path[color=white, font=]	(m-1-1) edge node[color=black, sloped] {$\iso$} (m-2-1)
																(m-2-1) edge node[color=black, sloped] {$\iso$} (m-2-2);
		\end{tikzpicture}
	\end{displaymath}
	where the square on the right commutes by definition of $\phi \hc_g \psi$, while the triangles making up the square on the left commute by the unit axiom for $H$ and the definition of the Yoneda isomorphism $H \iso A \hc_A H$.
	
	Of course if $\phi$ is vertical, i.e.\ $K = U_B$ from some internal category $B$, then $\phi_0$ is an internal transformation as given in \eqref{diagram:natural transformation of internal functors}; checking that this induces an isomorphism $V(\inProf\E) \iso \inCat\E$ is straightforward.
\end{proof}
  Having introduced the main examples, the following propositions record some useful properties of equipments. Some of these can be found in \cite{Shulman08}; they all follow as specialisations of the Theorems and Corollaries 7.16 -- 7.22  of \cite{Cruttwell-Shulman10} for `virtual equipments' to the ordinary equipments that we consider. Remember that the companion $\hmap{B(f,\id)}AB$ of a map $\map fAB$ is defined by two cells $\ls f\eps$ and $\ls f\eta$ that satisfy the companion identities; see the discussion following \defref{definition:equipment}.
\begin{proposition}[Shulman] \label{companion compositors}
	If $B(f, \id)$ and $C(g, \id)$ are the companions of composable morphisms $\map fAB$ and $\map gBC$, in an equipment $\K$, then the compositions of the diagrams below define the composition $B(f, \id) \hc C(g, \id)$ as the companion of $g \of f$.
	\begin{displaymath}
    \begin{tikzpicture}[textbaseline]
      \matrix(m)[math175em]{A & B & C \\ B & B & C \\ C & C & C \\};
      \path[map]  (m-1-1) edge[barred] node[above] {$B(f,\id)$} (m-1-2)
                          edge node[left] {$f$} (m-2-1)
                  (m-1-2) edge[barred] node[above] {$C(g,\id)$} (m-1-3)
                  (m-2-1) edge node[left] {$g$} (m-3-1)
                  (m-2-2) edge[barred] node[above] {$C(g,\id)$} (m-2-3)
                          edge node[right] {$g$} (m-3-2);
      \path       (m-1-2) edge[eq] (m-2-2)
                  (m-1-3) edge[eq] (m-2-3)
                  (m-2-1) edge[eq] (m-2-2)
                  (m-2-3) edge[eq] (m-3-3)
                  (m-3-1) edge[eq] (m-3-2)
                  (m-3-2) edge[eq] (m-3-3);
      \path[transform canvas={shift=($(m-2-1)!0.5!(m-2-2)$)}]
                  (m-1-2) edge[cell] node[right] {$\ls f \eps$} (m-2-2)
                  (m-1-3) edge[cell, shorten >= 7pt, shorten <= -1pt] node[above right] {$\id$} (m-2-3)
                  (m-2-2) edge[cell] node[right] {$U_g$} (m-3-2)
                  (m-2-3) edge[cell] node[right] {$\ls g \eps$} (m-3-3);
    \end{tikzpicture}
    \qquad \qquad
    \begin{tikzpicture}[textbaseline]
      \matrix(m)[math175em]{A & A & A \\ A & B & B \\ A & B & C \\};
      \path[map]  (m-1-2) edge node[right] {$f$} (m-2-2)
                  (m-1-3) edge node[right] {$f$} (m-2-3)
                  (m-2-1) edge[barred] node[below] {$B(f, \id)$} (m-2-2)
                  (m-2-3) edge node[right] {$g$} (m-3-3)
                  (m-3-1) edge[barred] node[below] {$B(f, \id)$} (m-3-2)
                  (m-3-2) edge[barred] node[below] {$C(g, \id)$} (m-3-3);
      \path       (m-1-1) edge[eq] (m-1-2)
                          edge[eq] (m-2-1)
                  (m-1-2) edge[eq] (m-1-3)
                  (m-2-1) edge[eq] (m-3-1)
                  (m-2-2) edge[eq] (m-2-3)
                          edge[eq] (m-3-2);
			\path[transform canvas={shift=($(m-2-1)!0.5!(m-2-2)$)}]
                  (m-1-2) edge[cell] node[right] {$\ls f \eta$} (m-2-2)
                  (m-1-3) edge[cell] node[right] {$U_f$} (m-2-3)
                  (m-2-2) edge[cell, shorten >= -1pt, shorten <= 7pt] node[below=1.5pt, right] {$\id$} (m-3-2)
                  (m-2-3) edge[cell] node[right] {$\ls g \eta$} (m-3-3);
    \end{tikzpicture}
  \end{displaymath}
  Consequently, any other companion $C(g \of f, \id)$ of the composite $g \of f$ induces a canonical horizontal invertible cell $B(f, \id) \hc C(g, \id) \Rar C(g \of f, \id)$ that is the factorisation of the cartesian filler above through $\cell{\ls{g \of f} \eps}{C(g \of f, \id)}{U_C}$.
  
  Analogously the composition $C(\id, g) \hc B(\id, f)$ can be taken as the conjoint of $g \of f$, and any other conjoint $C(\id, g \of f)$ for $g \of f$ induces a canonical horizontal invertible cell $C(\id, g) \hc B(\id, f) \Rar C(\id, g \of f)$.
\end{proposition}
	Notice that, using the unit axiom for $\K$, the cartesian filler on the left above can also be written as the composition
\begin{displaymath}
	B(f, \id) \hc C(g, \id) \xRar{\ls f\eps \hc \id} U_B \hc C(g, \id) \xRar{\mathfrak l} C(g, \id) \xRar{\ls g\eps} U_C.
\end{displaymath}
\begin{proof}
	It is clear that the companion identities imply that the compositions of the two diagrams above vertically compose to the identity of $g \of f$ and horizontally compose to the horizontal identity on $B(f, \id) \hc C(g, \id)$, that is they define $B(f, \id) \hc C(g, \id)$ as the companion of $g \of f$.  The existence of the canonical isomorphisms follows from the fact that any two companions of $g \of f$ are isomorphic (via invertible horizontal cells).
\end{proof}
	Although the bijective correspondence itself, below, is given by Grandis and Par\'e in \cite[Section 1.6]{Grandis-Pare04} (see also \cite[Corollary 7.21]{Cruttwell-Shulman10}), the author has not seen the explicit descriptions of its functoriality before.
\begin{proposition}[Grandis and Par\'e] \label{left and right cells}
	In any equipment $\K$ there are bijective correspondences between the three sets consisting of all cells that are of the form
	\begin{displaymath}
		\begin{tikzpicture}[textbaseline]
      \matrix(m)[math175em]{A & B & D \\ A & C & D, \\};
      \path[map]  (m-1-1) edge[barred] node[above] {$J$} (m-1-2)
                  (m-1-2) edge[barred] node[above] {$D(g, \id)$} (m-1-3)
                  (m-2-1) edge[barred] node[below] {$C(f, \id)$} (m-2-2)
                  (m-2-2) edge[barred] node[below] {$K$} (m-2-3);
      \path       (m-1-1) edge[eq] (m-2-1)
                  (m-1-2) edge[cell] (m-2-2)
                  (m-1-3) edge[eq] (m-2-3);
    \end{tikzpicture}
    \quad \quad
    \begin{tikzpicture}[textbaseline]
      \matrix(m)[math175em]{A & B \\ C & D \\};
        \path[map]  (m-1-1) edge[barred] node[above] {$J$} (m-1-2)
                            edge node[left] {$f$} (m-2-1)
                    (m-1-2) edge node[right] {$g$} (m-2-2)
                    (m-2-1) edge[barred] node[below] {$K$} (m-2-2);
        \path[transform canvas={shift={($(m-1-2)!0.5!(m-2-2)$)}}] (m-1-1) edge[cell] (m-2-1);
    \end{tikzpicture}
    \qquad \text{and} \qquad
    \begin{tikzpicture}[textbaseline]
      \matrix(m)[math175em]{C & A & B \\ C & D & B \\};
      \path[map]  (m-1-1) edge[barred] node[above] {$C(\id, f)$} (m-1-2)
                  (m-1-2) edge[barred] node[above] {$J$} (m-1-3)
                  (m-2-1) edge[barred] node[below] {$K$} (m-2-2)
                  (m-2-2) edge[barred] node[below] {$D(\id, g)$} (m-2-3);
      \path       (m-1-1) edge[eq] (m-2-1)
                  (m-1-2) edge[cell] (m-2-2)
                  (m-1-3) edge[eq] (m-2-3);
    \end{tikzpicture}
  \end{displaymath}
  respectively. For a cell $\phi$ that is of the second form above, we write $\lambda\phi$ for the corresponding cell of the first form and $\rho\phi$ for that of the third form; they are given by the compositions
  \begin{displaymath}
	  \lambda\phi = \begin{tikzpicture}[textbaseline]
		  \matrix(m)[math175em]{A & A & B & D \\ A & C & D & D \\};
		  \path[map]	(m-1-2) edge[barred] node[above] {$J$} (m-1-3)
													edge node[right] {$f$} (m-2-2)
									(m-1-3) edge[barred] node[above] {$D(g, \id)$} (m-1-4)
													edge node[right] {$g$} (m-2-3)
									(m-2-1) edge[barred] node[below] {$C(f, \id)$} (m-2-2)
									(m-2-2) edge[barred] node[below] {$K$} (m-2-3);
			\path				(m-1-1) edge[eq] (m-1-2)
													edge[eq] (m-2-1)
									(m-1-4) edge[eq] (m-2-4)
									(m-2-3) edge[eq] (m-2-4);
			\path[transform canvas={shift=($(m-1-2)!0.5!(m-2-2)$)}]
									(m-1-2) edge[cell] node[right] {$\ls f\eta$} (m-2-2)
									(m-1-3) edge[cell] node[right] {$\phi$} (m-2-3)
									(m-1-4) edge[cell] node[right] {$\ls g\eps$} (m-2-4);
		\end{tikzpicture}
		\quad \text{and} \quad
		\rho\phi = \begin{tikzpicture}[textbaseline]
			\matrix(m)[math175em]{C & A & B & B \\ C & C & D & B. \\};
		  \path[map]	(m-1-1) edge[barred] node[above] {$C(\id, f)$} (m-1-2)
									(m-1-2) edge[barred] node[above] {$J$} (m-1-3)
													edge node[right] {$f$} (m-2-2)
									(m-1-3)	edge node[right] {$g$} (m-2-3)
									(m-2-2) edge[barred] node[below] {$K$} (m-2-3)
									(m-2-3) edge[barred] node[below] {$D(\id, g)$} (m-2-4);
			\path				(m-1-1) edge[eq] (m-2-1)
									(m-1-3) edge[eq] (m-1-4)
									(m-1-4) edge[eq] (m-2-4)
									(m-2-1) edge[eq] (m-2-2);
			\path[transform canvas={shift=($(m-1-2)!0.5!(m-2-2)$)}]
									(m-1-2) edge[cell] node[right] {$\eps_f$} (m-2-2)
									(m-1-3) edge[cell] node[right] {$\phi$} (m-2-3)
									(m-1-4) edge[cell] node[right] {$\eta_g$} (m-2-4);
		\end{tikzpicture}
	\end{displaymath}
	The assignments $\phi \mapsto \lambda\phi$ and $\phi \mapsto \rho\phi$ preserve horizontal identity cells as well as compositions, the latter in the following sense. For vertically composable cells as on the left below the diagram of horizontal cells on the right commutes, where the canonical cells $\lambda_\hc$ are given by the previous proposition.
	\begin{equation} \label{diagram:vertical naturality of lambda}
		\begin{split}
		\begin{tikzpicture}[baseline]
			\matrix(m)[math175em]{A \nc B \\ C \nc D \\ E \nc F \\};
			\path[map]	(m-1-1) edge[barred] node[above] {$J$} (m-1-2)
													edge node[left] {$f$} (m-2-1)
									(m-1-2) edge node[right] {$g$} (m-2-2)
									(m-2-1) edge[barred] node[below] {$K$} (m-2-2)
													edge node[left] {$h$} (m-3-1)
									(m-2-2) edge node[right] {$k$} (m-3-2)
									(m-3-1) edge[barred] node[below] {$L$} (m-3-2);
			\path[transform canvas={shift=(m-2-1)}]
									(m-1-2) edge[cell] node[right] {$\phi$} (m-2-2)
									(m-2-2) edge[cell, transform canvas={yshift=-3pt}] node[right] {$\psi$} (m-3-2);
		\end{tikzpicture}
		\qquad\begin{tikzpicture}[baseline]
			\matrix(k)[math, yshift=3.5em]{C(f, \id) \hc K \hc F(k, \id) \\};
			\matrix(l)[math2em]
			{	J \hc D(g, \id) \hc F(k, \id) \nc C(f, \id) \hc E(h, \id) \hc L \\ };
			\matrix(m)[math2em, yshift=-4em]{J \hc F(k \of g, \id) \nc E(h \of f, \id) \hc L \\};
			\path	(l-1-1)	edge[cell] node[left] {$\id \hc \lambda_\hc$} (m-1-1)
										edge[cell] node[above left] {$\lambda\phi \hc \id$} (k-1-1)
						(k-1-1) edge[cell] node[above right] {$\id \hc \lambda\psi$} (l-1-2)
						(m-1-1) edge[cell] node[below] {$\lambda(\psi \of \phi)$} (m-1-2)
						(l-1-2) edge[cell] node[right] {$\lambda_\hc \hc \id$} (m-1-2);
		\end{tikzpicture}
		\end{split}
	\end{equation}
	If $\phi$ is a horizontal cell, that is $f = \id_C$ and $g = \id_D$, this reduces to
	\begin{displaymath}
		\cell{\lambda(\psi \of \phi) = \lambda \psi \of (\phi \hc \id)}{J \hc D(g, \id)}{C(f, \id) \hc L};
	\end{displaymath}
	likewise when $\psi$ is a horizontal cell. Secondly, given horizontally composable cells as on the left below, the diagram on the right commutes.
	\begin{equation} \label{diagram:horizontal naturality of lambda}
		\begin{split}
		\begin{tikzpicture}[baseline]
			\matrix(m)[math175em]{A \nc B \nc G \\ C \nc D \nc H \\};
			\path[map]	(m-1-1) edge[barred] node[above] {$J$} (m-1-2)
													edge node[left] {$f$} (m-2-1)
									(m-1-2) edge[barred] node[above] {$M$} (m-1-3)
													edge node[right] {$g$} (m-2-2)
									(m-1-3) edge node[right] {$l$} (m-2-3)
									(m-2-1) edge[barred] node[below] {$K$} (m-2-2)
									(m-2-2) edge[barred] node[below] {$N$} (m-2-3);
			\path[transform canvas={shift=($(m-1-1)!0.5!(m-2-2)$)}]
									(m-1-2) edge[cell] node[right] {$\phi$} (m-2-2)
									(m-1-3) edge[cell] node[right] {$\chi$} (m-2-3);
		\end{tikzpicture}
	\qquad\begin{tikzpicture}[baseline]
			\matrix(k)[math, yshift=-1.625em]{J \hc D(g, \id) \hc N \\};
			\matrix(l)[math2em, yshift=1.625em, column sep=3em]{J \hc M \hc H(l, \id) \nc C(f, \id) \hc K \hc N \\};
			\path	(l-1-1) edge[cell] node[below left] {$\id \hc \lambda\chi$} (k-1-1)
										edge[cell] node[above] {$\lambda(\phi \hc \chi)$} (l-1-2)
						(k-1-1) edge[cell] node[below right] {$\lambda\phi \hc \id$} (l-1-2);
		\end{tikzpicture}
		\end{split}
	\end{equation}
		
	Analogously for $\phi \mapsto \rho\phi$ the following hold, where the cells $\rho_\hc$ are given by previous proposition.
	\begin{gather*}
		\rho(\phi \of \psi) \of (\rho_\hc \hc \id) = (\id \hc \rho_\hc) \of (\rho\psi \hc \id) \of (\id \hc \rho\phi) \\
		\rho(\phi \hc \chi) = (\id \hc \rho\chi) \of (\rho\phi \hc \id)
	\end{gather*}
\end{proposition}
\begin{proof}
	Let $\psi$ be a horizontal cell $J \hc D(g, \id) \Rar C(f, \id) \hc K$, i.e.\ $\psi$ is of the first form above. It is clear that the companion identities for $C(f, \id)$ and $D(g, \id)$ imply that mapping $\psi$ to the composition of the diagram below gives an inverse to $\phi \mapsto \lambda\phi$. Similarly composing a cell, that is of the third form above, with the conjoint cells $\eta_f$ and $\eps_g$ gives an inverse to $\phi \mapsto \rho\phi$.
	\begin{displaymath}
		\begin{tikzpicture}
			\matrix(m)[math175em]{A & B & B \\ A & B & D \\ A & C & D \\ C & C & D \\};
			\path[map]	(m-1-1) edge[barred] node[above] {$J$} (m-1-2)
									(m-1-3) edge node[right] {$g$} (m-2-3)
									(m-2-1) edge[barred] node[below] {$J$} (m-2-2)
									(m-2-2) edge[barred] node[below] {$D(g, \id)$} (m-2-3)
									(m-3-1) edge[barred] node[above] {$C(f, \id)$} (m-3-2)
													edge node[left] {$f$} (m-4-1)
									(m-3-2) edge[barred] node[above] {$K$} (m-3-3)
									(m-4-2) edge[barred] node[below] {$K$} (m-4-3);
			\path				(m-1-1) edge[eq] (m-2-1)
									(m-1-2) edge[eq] (m-1-3)
													edge[eq] (m-2-2)
									(m-2-1) edge[eq] (m-3-1)
									(m-2-2) edge[cell] node[right] {$\psi$} (m-3-2)
									(m-2-3) edge[eq] (m-3-3)
									(m-3-2) edge[eq] (m-4-2)
									(m-3-3) edge[eq] (m-4-3)
									(m-4-1) edge[eq] (m-4-2);
			\path[transform canvas={shift=($(m-2-1)!0.5!(m-3-2)$)}]
                  (m-1-2) edge[cell] node[right] {$\id$} (m-2-2)
                  (m-1-3) edge[cell] node[right] {$\ls g \eta$} (m-2-3)
                  (m-3-2) edge[cell] node[right] {$\ls f \eps$} (m-4-2)
                  (m-3-3) edge[cell] node[right] {$\id$} (m-4-3);
		\end{tikzpicture}
	\end{displaymath}
	
	That the assignments $\lambda$ and $\rho$ preserve horizontal identity cells follows from the companion and conjoint identities. To show that $\lambda$ preserves vertical compositions, consider the cells $\phi$ and $\psi$ of \eqref{diagram:vertical naturality of lambda}; we will show that the diagram \eqref{diagram:vertical naturality of lambda} commutes. The composite of $\lambda\phi \hc \id$ and $\id \hc \lambda \psi$ in the top leg of the diagram can be written as the composition of the following grid of cells, where the empty cells denote identity cells.
	\begin{displaymath}
		\begin{tikzpicture}
			\matrix(m)[math175em]
			{	A & A & A & B & D & F \\
				A & C & C & D & D & F \\
				A & C & E & F & F & F \\ };
			\path[map]	(m-1-2)	edge node[right] {$f$} (m-2-2)
									(m-1-3) edge[barred] node[above] {$J$} (m-1-4)
													edge node[left] {$f$} (m-2-3)
									(m-1-4) edge[barred] node[above] {$D(g, \id)$} (m-1-5)
													edge node[right] {$g$} (m-2-4)
									(m-1-5) edge[barred] node[above] {$F(k, \id)$} (m-1-6)
									(m-2-1) edge[barred] node[below] {$C(f, \id)$} (m-2-2)
									(m-2-3) edge[barred] node[below] {$K$} (m-2-4)
													edge node[right] {$h$} (m-3-3)
									(m-2-4) edge node[right] {$k$} (m-3-4)
									(m-2-5) edge[barred] node[above] {$F(k, \id)$} (m-2-6)
													edge node[left] {$k$} (m-3-5)
									(m-3-1) edge[barred] node[below] {$C(f, \id)$} (m-3-2)
									(m-3-2) edge[barred] node[below] {$E(h, \id)$} (m-3-3)
									(m-3-3) edge[barred] node[below] {$M$} (m-3-4);
			\path				(m-1-1) edge[eq] (m-1-2)
													edge[eq] (m-2-1)
									(m-1-2) edge[eq] (m-1-3)
									(m-1-5) edge[eq] (m-2-5)
									(m-1-6) edge[eq] (m-2-6)
									(m-2-1) edge[eq] (m-3-1)
									(m-2-2) edge[eq] (m-2-3)
													edge[eq] (m-3-2)
									(m-2-4) edge[eq] (m-2-5)
									(m-2-6) edge[eq] (m-3-6)
									(m-3-4) edge[eq] (m-3-5)
									(m-3-5) edge[eq] (m-3-6);
			\path[transform canvas={shift=(m-2-3)}]
									(m-1-2) edge[cell] node[right] {$\ls f\eta$} (m-2-2)
									(m-1-4) edge[cell] node[right] {$\phi$} (m-2-4)
									(m-1-5) edge[cell] node[right] {$\ls g\eps$} (m-2-5)
									(m-2-3) edge[cell] node[right] {$\ls h\eta$} (m-3-3)
									(m-2-4) edge[cell] node[right] {$\psi$} (m-3-4)
									(m-2-6) edge[cell] node[right] {$\ls k\eps$} (m-3-6);
		\end{tikzpicture}
	\end{displaymath}
	It follows from the definition (see \propref{companion compositors}) of the canonical cells $\lambda_\hc$ that postcomposing the square of companion cells on the left with $\lambda_\hc$ gives the companion cell $\ls{h \of f}\eta$; in the same way precomposing the companion cell $\ls{k \of g}\eps$ in $\lambda(\psi \of \phi)$ with $\lambda_\hc$ gives the square of companion cells on the right above. This shows that both legs of the diagram \eqref{diagram:vertical naturality of lambda} are equal.
	 
	Now suppose that $\phi$ in \eqref{diagram:vertical naturality of lambda} is a horizontal cell. In that case postcomposing the left subsquare in the diagram above with $\cell{\lambda_\hc}{C(\id, \id) \hc E(h, \id)}{E(h, \id)}$ reduces it to the single companion cell $\ls h\eta$, since both are factorisations of $U_h$ through $\ls h\eps$. Likewise, precomposing the right subsquare with the composite
	\begin{displaymath}
		F(k, \id) \xRar{\inv{\mathfrak l}} U_D \hc F(k, \id) \xRar{\ls\id\eta \hc \id} D(\id, \id) \hc F(k, \id)
	\end{displaymath}
	reduces it to $\ls k\eps$. It follows that precomposing the top leg of \eqref{diagram:vertical naturality of lambda} with the above composition gives $\lambda\psi \of (\phi \hc \id)$. On the other hand, the universal property of the companion cell $\ls k\eps$ implies that the composite above, followed by $\lambda_\hc$, is equal to the identity on $F(k, \id)$, so that the bottom leg of \eqref{diagram:vertical naturality of lambda}, precomposed with the composite above, reduces to $\lambda(\psi \of \phi)$. We conclude that $\lambda(\psi \of \phi) = \lambda\psi \of (\phi \hc \id)$.
	 
	Finally it is easily seen that the diagram \eqref{diagram:horizontal naturality of lambda} for the vertically composable cells $\phi$ and $\chi$ commutes because, in the vertical composition of $\id \hc \lambda\chi$ and $\lambda\phi \hc \id$, the companion cell $\ls g\eps$ in $\lambda \phi$ cancels with the companion cell $\ls g\eta$ in $\lambda \chi$.
\end{proof}

	Recall that any equipment $\K$ induces a $2$-category $V(\K)$ of morphisms and vertical cells, as well as a bicategory $H(\K)$ of promorphisms and horizontal cells. The assignments $\phi \mapsto \lambda\phi$ and $\phi \mapsto \rho\phi$ restrict to `pseudofunctors' between $V(\K)$ and $H(\K)$, as follows. Recall that a \emph{pseudofunctor} $\map F{\catvar B}{\catvar C}$ between bicategories is a `weakly' $\cat$-enriched functor, preserving composition and identities only up to coherent invertible cells in $\catvar C$, which are called \emph{compositors} and \emph{unitors} (see \cite[Definition 1.5.8]{Leinster04}). A pseudofunctor $F$ is called \emph{locally fully faithful} whenever it restricts to fully faithful functors $\catvar B(X, Y) \to \catvar C(FX, FY)$ on the morphism-categories. Given a bicategory $\catvar B$ we write $\op{\catvar B}$ and $\co{\catvar B}$ for the bicategories obtained by reversing respectively the morphisms or the cells of $\catvar B$, that is $\op{\catvar B}(X, Y) = \catvar B(Y, X)$ and $\co{\catvar B}(X, Y) = \op{\catvar B(X, Y)}$.
\begin{proposition}[Shulman] \label{companion and conjoint pseudofunctors}
	Choosing a companion $B(f, \id)$ for every morphism $\map fAB$, together with the assignment
	\begin{equation} \label{equation:24513}
		\phi \mapsto \bigl\lbrack B(g, \id) \xRar{\inv{\mathfrak l}} U_A \hc B(g, \id) \xRar{\lambda\phi} B(f, \id) \hc U_B \xRar{\mathfrak r} B(f, \id)\bigr\rbrack,
	\end{equation}
	for each vertical cell $\cell\phi fg$ and where $\lambda\phi$ is given in the previous proposition, gives a pseudofunctor $\map\lambda{\co{V(\K)}}{H(\K)}$ that restricts to the identity on objects and that is locally fully faithful. Its compositor components
  \begin{displaymath}
		\cell{\lambda_\hc}{B(f, \id) \hc C(h, \id)}{C(h \of f, \id)}
	\end{displaymath}
	are the canonical cells given by \propref{companion compositors}, while its unitors $U_A \Rar A(\id, \id)$ are given by the companion cells $\ls\id\eta$.
	
	Dually, choosing a conjoint $B(\id, f)$ for each $\map fAB$ induces a locally fully faithful pseudofunctor $\map\rho{\op{V(\K)}}{H(\K)}$ that restricts to the identity on objects; its action on a vertical cell $\cell\phi fg$ is given by
	\begin{displaymath}
		\phi \mapsto \bigl\lbrack B(\id, f) \xRar{\inv{\mathfrak r}} B(\id, f) \hc U_A \xRar{\rho\phi} U_B \hc B(\id, g) \xRar{\mathfrak l} B(\id, g) \bigr\rbrack.
	\end{displaymath}
\end{proposition}
	We are abusing notation by denoting the composition \eqref{equation:24513} again by $\lambda\phi$. This will not lead to confusion however, since we will never (except in the proof below) use the original cell $\cell{\lambda\phi}{U_A \hc B(g, \id)}{B(f, \id) \hc U_B}$, as given by the previous proposition, when $\phi$ is vertical: it is far more natural to consider its composition with the unitors, as given above. Since in the proof below we will use both the original and the new definition of $\lambda\phi$, we will denote \eqref{equation:24513} temporarily by $\lambda'\phi$.
\begin{proof}
	To see that the action of $\lambda'$ on vertical cells preserves vertical composition, consider composable vertical cells $\cell\phi fg$ and $\cell\chi gh$, where $f$, $g$ and $h$ are morphisms $A \to B$; we need
	\begin{equation} \label{equation:641891}
	\cell{\lambda'(\mathfrak r \of (\phi \hc \chi) \of \inv{\mathfrak l}) = \lambda'\phi \of \lambda'\chi}{B(h, \id)}{B(f, \id)}
	\end{equation}
	where, by definition of $V(\K)$, the left-hand side is the image of the vertical composite of $\phi$ and $\chi$. Under the isomorphisms $U_A^{\hc 3} \hc B(h, \id) \iso B(h, \id)$ and $B(f, \id) \hc U_B^{\hc 3} \iso B(f, \id)$, given by the unitors of $\K$, the left-hand side equals
	\begin{displaymath}
		U_A \hc U_A \hc U_A \hc B(h, \id) \xRar{\ls f\eta \hc \phi \hc \chi \hc \ls h\eps} B(f, \id) \hc U_B \hc U_B \hc U_B,
	\end{displaymath}
	while the right-hand side, under the isomorphisms $U_A^{\hc 4} \hc B(h, \id) \iso B(h,\id)$ and $B(f, \id) \hc U_B^{\hc 4} \iso B(f, \id)$, equals the composite
	\begin{multline*}
		U_A \hc U_A \hc U_A \hc U_A \hc B(h, \id) \xRar{\id \hc \ls g\eta \hc \chi \hc \ls h\eps} U_A \hc U_A \hc B(g, \id) \hc U_B \hc U_B \\
		\xRar{\ls f\eta \hc \phi \hc \ls g\eps \hc \id} B(f, \id) \hc U_B \hc U_B \hc U_B \hc U_B.
	\end{multline*}
	Cancelling $\ls g\eta$ against $\ls g\eps$ here we see that the two cells above coincide up to composition with unitors, so that the equality \eqref{equation:641891} follows.
	
	Given vertically composable vertical cells $\phi$ and $\psi$, the fact that the compositors $\lambda_\hc$ are natural, that is $\lambda_\hc \of (\lambda' \phi \hc \lambda' \psi) = \lambda'(\psi \of \phi) \of \lambda_\hc$, is equivalent to the commuting, for $\phi$ and $\psi$, of the diagram \eqref{diagram:vertical naturality of lambda} of the previous proposition. That the compositors satisfy the associativity axiom follows from the naturality of the associator $\mathfrak a$ and the uniqueness of factorisations through cartesian fillers. Finally it is easy to show that the cells $\cell{\ls{\id}\eta}{U_A}{A(\id, \id)}$ can be taken as the unitors for $\lambda$. 
\end{proof}
\begin{example} \label{example:left cell corresponding to vertical cell}
	Recall, from \exref{example:vertical cells in Prof}, that a vertical cell $\cell\phi fg$ between functors $f$ and $\map gAB$, in the equipment $\Prof$ of (unenriched) profunctors, is simply an ordinary natural transformation $\nat\phi fg$. In this case the corresponding horizontal cell $\cell{\lambda\phi}{B(g, \id)}{B(f, \id)}$ is given by the composite
	\begin{displaymath}
		B(g, \id) \iso U_A \hc_A U_A \hc_A B(g, \id) \xRar{\eta_f \hc_f \phi \hc \ls g\eps} B(f, \id) \hc_B U_B \hc_B U_B \iso B(f, \id).
	\end{displaymath}
	Recalling the definition of the Yoneda isomorphisms here (\exref{example:unenriched profunctors}) we find that under this composite a map $\map s{ga}b$ in $B(ga, b)$ is mapped as $s \mapsto (\id_a, \id_a, s) \mapsto (\id_{fa}, \phi_a, s) \mapsto s \of \phi_a$, where $s \of \phi_a$ is contained in $B(fa, b)$. Thus $\lambda\phi$ is given by precomposition with the components of $\phi$. In the case of enriched profunctors $\enProf\V$ the action of $\lambda$ on $\V$-natural transformations is given similarly.
\end{example}

	We end this section with a definition that will be crucial later on, as explained by the remark following it.
\begin{definition} \label{definition:invertible cells}
  In an equipment $\K$ consider a cell $\phi$, together with its corresponding horizontal cells $\lambda\phi$ and $\rho\phi$ as given by \propref{left and right cells}. We say that $\phi$ is \emph{left invertible} if $\lambda\phi$ has an inverse in $H(\K)$; if $\rho\phi$ has an inverse then we say that $\phi$ is \emph{right invertible}.
\end{definition}
\begin{remark} \label{remark:why invertible cells}
	Once we have defined what a monad $T$ on an equipment $\K$ is, in \defref{definition:monad}, a `strict' $T$-algebra $A$ will be defined as usual: it is an object $A$ together with a structure morphism $\map a{TA}A$, satisfying associativity and unit axioms. Given a second $T$-algebra $B$, a `lax $T$-promorphism' $J$ can then be defined (see \defref{definition:lax promorphism}) as a promorphism $\hmap JAB$ equipped with a `lax structure cell'
	\begin{displaymath}
		\begin{tikzpicture}
			\matrix(m)[math175em]{TA & TB \\ A & B, \\};
			\path[map]	(m-1-1) edge[barred] node[above] {$TJ$} (m-1-2)
													edge node[left] {$a$} (m-2-1)
									(m-1-2) edge node[right] {$b$} (m-2-2)
									(m-2-1) edge[barred] node[below] {$J$} (m-2-2);
			\path[transform canvas={shift=($(m-1-1)!0.5!(m-2-1)$)}]	(m-1-2) edge[cell] node[right] {$\bar J$} (m-2-2);
		\end{tikzpicture}
	\end{displaymath}
	satisfying associativity and unit axioms; compare the (well-known) definition of a colax $T$-morphism (\defref{definition:colax morphism}). Usually an object with some `lax structure cell' is called `pseudo' whenever the structure cell is invertible. In this case however, asking that $\bar J$ be invertible in $\K_1$ is clearly too strong: it would mean that the structure maps $a$ and $b$ are isomorphisms. The solution to this, it will turn out, is to consider `left pseudo lax' and `right pseudo lax' $T$-promorphisms instead, for which $\bar J$ is respectively left or right invertible.
\end{remark}
\begin{example}
	A cell of $\V$-matrices 
	\begin{displaymath}
		\begin{tikzpicture}
			\matrix(m)[math175em]{A & B \\ C & D \\};
			\path[map]  (m-1-1) edge[barred] node[above] {$J$} (m-1-2)
													edge node[left] {$f$} (m-2-1)
									(m-1-2) edge node[right] {$g$} (m-2-2)
									(m-2-1) edge[barred] node[below] {$K$} (m-2-2);
			\path[transform canvas={shift={($(m-1-2)!(0,0)!(m-2-2)$)}}] (m-1-1) edge[cell] node[right] {$\phi$} (m-2-1);
		\end{tikzpicture}
	\end{displaymath}
	corresponds to the horizontal cell $\cell{\rho\phi}{C(\id, f) \hc J}{K \hc D(\id, g)}$ that is given by 
	\begin{displaymath}
		\bigpars{C(\id, f) \hc J}(c, b) = \coprod_{fa = c} J(a, b) \xrar{\coprod \phi_{a, b}} K(c, gb) = \pars{K \hc D(\id, g)}(c, b);
	\end{displaymath}
	thus $\phi$ is right invertible if and only if the maps $\coprod_{fa = c} \phi_{a, b}$ are isomorphisms, for all $b \in B$ and $c \in C$.
\end{example}
\begin{example} \label{example:right invertible spans}
	In $\Span\E$ the horizontal cell $\cell{\rho\phi}{C(\id, f) \times_A J}{K \times_D D(\id, g)}$, corresponding to a cell $\phi$ as above, is given by the morphism of spans on the left below. Here $K \times_D B$ is the pullback of $k_D$ and $g$, with projections $p$ and $q$. 
	\begin{displaymath}
		\begin{tikzpicture}[baseline]
			\matrix(m)[math175em, column sep=2.25em, row sep=0.6em]{& J & \\ C & & B \\ & \phantom K & \\ C & & B \\};
			\path[map]	(m-1-2)	edge node[above left] {$f \of j_A$} (m-2-1)
													edge node[desc] {$(\phi, j_B)$} (m-3-2)
													edge node[above right] {$j_B$} (m-2-3)
									(m-3-2) edge node[below=4pt, right=-2pt] {$k_C \of p$} (m-4-1)
													edge node[below=4pt, left=-2pt] {$q$} (m-4-3);
			\path				(m-2-1) edge[eq] (m-4-1)
									(m-2-3) edge[eq] (m-4-3);
			\draw	(m-3-2) node {$K \times_D B$};
		\end{tikzpicture}
		\qquad\qquad\qquad\begin{tikzpicture}[baseline]
			\matrix(m)[math175em]{J & B \\ K & D \\};
			\path[map]	(m-1-1) edge node[above] {$j_B$} (m-1-2)
													edge node[left] {$\phi$} (m-2-1)
									(m-1-2) edge node[right] {$g$} (m-2-2)
									(m-2-1) edge node[below] {$k_D$} (m-2-2);
		\end{tikzpicture}
	\end{displaymath}
	We conclude that $\phi$ is right invertible if and only if the square on the right above is a pullback square.
\end{example}
	The final result of this section, which seems to be new, concerns the right invertibility of cells between bimodules. It is the first of several results that reduce statements about bimodules in $\Mod\K$ to statements about the underlying promorphisms in $\K$, which are usually easier to check.
\begin{proposition} \label{right invertible cells between bimodules}
	Let $\K$ be an equipment that satisfies the condition of \propref{bimodule equipments}, so that bimodules in $\K$ form an equipment $\Mod\K$. A cell of bimodules $\phi$, on the left below, is right invertible whenever the underlying cells $\cell fAC$ and $\cell \phi JK$, on the right, are right invertible in $\K$.
	\begin{displaymath}
		\begin{tikzpicture}
			\matrix(m)[math175em]{A & B \\ C & D \\};
			\path[map]  (m-1-1) edge[barred] node[above] {$J$} (m-1-2)
													edge node[left] {$f$} (m-2-1)
									(m-1-2) edge node[right] {$g$} (m-2-2)
									(m-2-1) edge[barred] node[below] {$K$} (m-2-2);
			\path[transform canvas={shift={($(m-1-2)!(0,0)!(m-2-2)$)}}] (m-1-1) edge[cell] node[right] {$\phi$} (m-2-1);
		\end{tikzpicture}
		\qquad\qquad\begin{tikzpicture}
			\matrix(m)[math175em]{A_0 & A_0 \\ C_0 & C_0 \\};
			\path[map]  (m-1-1) edge[barred] node[above] {$A$} (m-1-2)
													edge node[left] {$f_0$} (m-2-1)
									(m-1-2) edge node[right] {$f_0$} (m-2-2)
									(m-2-1) edge[barred] node[below] {$C$} (m-2-2);
			\path[transform canvas={shift={($(m-1-2)!(0,0)!(m-2-2)$)}}] (m-1-1) edge[cell] node[right] {$f$} (m-2-1);
		\end{tikzpicture}
		\qquad\qquad\begin{tikzpicture}
			\matrix(m)[math175em]{A_0 & B_0 \\ C_0 & D_0 \\};
			\path[map]  (m-1-1) edge[barred] node[above] {$J$} (m-1-2)
													edge node[left] {$f_0$} (m-2-1)
									(m-1-2) edge node[right] {$g_0$} (m-2-2)
									(m-2-1) edge[barred] node[below] {$K$} (m-2-2);
			\path[transform canvas={shift={($(m-1-2)!(0,0)!(m-2-2)$)}}] (m-1-1) edge[cell] node[right] {$\phi$} (m-2-1);
		\end{tikzpicture}
	\end{displaymath}
\end{proposition}
\begin{proof}
	To compute $\rho\phi = \eps_f \hc_f \phi \hc_g \eta_g$ we will use the fact that cartesian fillers (and thus companion and conjoint cells) in $\Mod \K$ can be computed in $\K$, and likewise the factorisations through them, as was described in the proof of \propref{bimodule equipments}. Thus $\phi' = \phi \hc_g \eta_g$, which is the factorisation, in $\Mod\K$, of $\phi$ through the restriction $K(\id, g)$, is computed in $\K$ as the factorisation of $\phi$ through $K(\id, g_0)$: we conclude that the cell $\phi \hc \eta_{g_0}$ of $\K$ underlies the cell $\phi'$ of $\Mod\K$. On the other hand the cell underlying $\eps_f$ is the cartesian filler of the niche $C_0 \xrar{\id} C_0 \xsrar C C_0 \xlar{f_0} A_0$ which, by \propref{cartesian fillers in terms of companions and conjoints}, we can take to be the composite $\id_C \hc \eps_{f_0}$.
	
	We recall from the proof of \propref{bimodule equipments} that the composite $\rho\phi = \eps_f \hc_f \phi'$ is defined by the following diagram in $H(\K)$, where both rows are coequalisers.
	\begin{displaymath}
		\begin{tikzpicture}
			\matrix(m)[math2em]
			{ C(\id, f) \hc A \hc J & C(\id, f) \hc J & C(\id, f) \hc_A J \\
				C \hc C \hc K(\id, g) & C \hc K(\id, g) & C \hc_C K(\id, g) \\ };
			\path	(m-1-1)	edge[transform canvas={yshift=3.5pt}, cell] node[above] {$r \hc \id$} (m-1-2)
										edge[transform canvas={yshift=-3.5pt}, cell] node[below] {$\id \hc l$} (m-1-2)
										edge[cell] node[left] {$\eps_f \hc f \hc \phi'$} (m-2-1)
						(m-1-2) edge[cell] (m-1-3)
										edge[cell] node[right] {$\eps_f \hc \phi'$} (m-2-2)
						(m-1-3) edge[cell, dashed] node[right] {$\eps_f \hc_f \phi'$} (m-2-3)
						(m-2-1)	edge[transform canvas={yshift=3.5pt}, cell] node[above] {$m \hc \id$} (m-2-2)
										edge[transform canvas={yshift=-3.5pt}, cell] node[below] {$\id \hc l$} (m-2-2)
						(m-2-2) edge[cell] (m-2-3);
		\end{tikzpicture}
	\end{displaymath}
	As we have seen, the cell underlying $\eps_f \hc \phi'$ is the composite $\id_C \hc \eps_{f_0} \hc \phi \hc \eta_{g_0} = \id_C \hc \rho\phi$, which is invertible by assumption. Likewise the cell underlying $\eps_f \hc f \hc \phi'$ is the composite $\id_C \hc \eps_{f_0} \hc f \hc \phi \hc \eta_{g_0} = \id_C \hc \rho(f \hc \phi)$. The latter is also invertible because $\rho(f \hc \phi)$ can be written as a vertical composite of $\rho f$ and $\rho \phi$ by \propref{left and right cells}, both of which are invertible by assumption. Hence in the diagram above the two cells, that are drawn both solidly and vertically, are invertible and it follows that the dashed cell, which equals $\rho\phi$, is invertible as well. This concludes the proof.
\end{proof}

  \chapter{Weighted colimits} \label{chapter:weighted colimits}
	Having introduced equipments in the previous chapter we can now recall the notion of weighted colimits in such equipments. These are the main objects of our study and include, for example, `pointwise' left Kan extensions. Weighted colimits can be introduced in two ways, both of which will be given here. The first way goes back to Wood's original paper \cite{Wood82}, and defines weighted colimits in so-called `closed' equipments---equipments equipped with extra structure generalising that of closed monoidal categories. For well-behaved $\V$ and $\E$, the equipments $\enProf\V$ and $\inProf\E$, of $\V$\ndash enriched profunctors and internal profunctors in $\E$, are closed, and this first approach to defining weighted colimits is very close to how weighted colimits are defined in the $2$-category $\enCat\V$ of $\V$-categories. In particular, colimits in $\enProf\V$ that are `weighted by companions' are precisely the `enriched' left Kan extensions in $\enCat\V$ that were introduced by Dubuc \cite[Section I.4]{Dubuc70}. Satisfying a stronger universal property than ordinary Kan extensions, this is the usual notion of Kan extensions of enriched functors, see for example Kelly's book \cite[Formula 4.20]{Kelly82}.
	
	The second approach uses the notion of `left Kan extensions in a pseudo double category', as given by Grandis and Par\'e in \cite[Section 2]{Grandis-Pare08}, which generalises the classical notion of Kan extension, in ordinary $2$-categories. We will show that, in a closed equipment, a certain type of such left Kan extensions (which we will call `top absolute' extensions) coincide with the weighted colimits as they were defined in the first approach. Using this alternative definition has several advantages. Firstly it allows us to define weighted colimits in pseudo double categories, which are much more common than equipments (even more so than closed equipments). This turns out to be very useful later on, because the pseudo double categories $\rcProm T$ of `algebraic promorphisms', for a monad $T$ on an equipment (\defref{definition:right colax promorphisms}), will, when $T$ is not a pseudomonad, in general not be an equipment (see \propref{restrictions in T-promrc}).
	
	Secondly we can follow Grandis and Par\'e and, using their `double comma objects', introduce a notion of `pointwise' weighted colimits, generalising Street's notion of pointwise left Kan extensions in $2$-categories, that was introduced in \cite{Street74}. Just like Dubuc's notion of `enriched' extensions is the right notion when considering enriched functors, so is Street's notion of `pointwise' extensions the right one for internal functors. In fact Street's notion can be used for enriched functors as well, but is too strong in general: it is easy to give a pair of $2$-functors (i.e.\ $\cat$-enriched) whose enriched left Kan extension exists, but whose pointwise left Kan extension does not, see \exref{example:stronger pointwise left Kan extension}. Returning to weighted colimits in equipments, the main idea of this chapter is to, in the language of equipments, describe a property of double comma objects whose failure to hold can be thought of as the cause of the difference between pointwise and ordinary weighted colimits. Calling double comma objects with this property `strong' the main result (\thmref{weighted colimits are pointwise if double commas are strong}) of this chapter states that if all double comma objects in an equipment $\K$ are strong then all weighted colimits in $\K$ are pointwise. Since double comma objects in $\inProf\E$ are strong (\propref{internal double comma objects are strong}), we conclude that the notion of weighted colimits in equipments unifies the `right notions' of weighted colimits in $\enProf\V$ and $\inProf\E$: in the first they generalise the notion of enriched left Kan extension while in the second they generalise the notion of pointwise left Kan extension. Later we will prove that, given a monad $T$ on an equipment $\K$, the equipment of `algebraic profunctors' $\rcProm T$ has strong comma objects whenever $\K$ does, so that in that case the weighted colimits in $\rcProm T$ are also pointwise; see \thmref{algebraic weighted colimits are pointwise if double commas are strong}. 
		
\section{Weighted colimits in closed equipments} \label{section:closed equipments}
	We start with the notion of closed equipment, which generalises that of a closed monoidal category. The main ideas here go back to Wood's original paper \cite{Wood82}, while the author was introduced to them by the $n$Lab article on equipments \cite{nLab_on_equipments}, which is mostly written by Shulman. Nothing in this section is new.
	
	Let $\catvar B$ be a bicategory, with horizontal composition denoted $\hc$. Recall that $\catvar B$ is called \emph{closed} if for any morphism $\map HBC$ the functors
	\begin{displaymath}
		\map{H \hc \dash}{\catvar B(C, D)}{\catvar B(B, D)} \qquad \text{and} \qquad \map{\dash \hc H}{\catvar B(A, B)}{\catvar B(A, C)}
	\end{displaymath}
	both have right adjoints, in that case denoting them by $H \lhom \dash$ and $\dash \rhom H$ respectively. This means that, for morphisms $\map JAB$, $\map HBC$ and $\map KAC$, there are correspondences of cells
	\begin{equation} \label{equation:horizontal composition adjunction}
		\catvar B(B, C)(H, J \lhom K) \iso \catvar B(A, C)(J \hc H, K) \iso \catvar B(A, B)(J, K \rhom H)
	\end{equation}
	of which, on first sight, the first is natural in $H$ and $K$, while the second is natural in $J$ and $K$. However, the right adjoints can be made into functors
	\begin{displaymath}
		\map\lhom{\op{\catvar B(A, B)} \times \catvar B(A, C)}{\catvar B(B, C)} \qquad \text{and} \qquad \map\rhom{\catvar B(A, C) \times \op{\catvar B(B, C)}}{\catvar B(A, B)}
	\end{displaymath}
	in the usual way, and with respect to these both the correspondences above are natural in each of the variables $J$, $H$ and $K$. For example, given cells $\cell \phi HJ$ and $\cell \psi KL$, their image $\cell{\phi \lhom \psi}{J \lhom K}{H \lhom L}$ is the adjoint of the composite
	\begin{displaymath}
		H \hc (J \lhom K) \xRar{\phi \hc \id} J \hc (J \lhom K) \xRar{\ev_\lhom} K \xRar\psi L.
	\end{displaymath}
	Here $\cell{\ev_\lhom}{J \hc (J \lhom K)}K$ denotes the counit, called \emph{evaluation}, of the adjunction $J \hc \dash \ladj J \lhom \dash$; the unit, \emph{coevaluation}, will be denoted $\cell{\coev_\lhom}H{J \lhom (J \hc H)}$. Likewise the counit and unit of $\dash \hc H \ladj \dash \rhom H$ are given by cells $\cell{\ev_\rhom}{(K \rhom H) \hc H}K$ and $\cell{\coev_\rhom}J{(J \hc H) \rhom H}$. Often we will suppress the subscripts $\lhom$ and $\rhom$ from these notations and simply write $\ev$ and $\coev$, when no confusion can arise. Finally, we will call $J \lhom K$ the \emph{left hom} of $J$ and $K$, while $K \rhom H$ will be called the \emph{right hom}.
	
	Recall that every pseudo double category $\K$ contains an underlying bicategory $H(\K)$ of horizontal morphisms and horizontal cells.
\begin{definition} \label{definition:closed pseudo double category}
	A pseudo double category $\K$ is called \emph{closed} whenever $H(\K)$ is closed.
\end{definition}
\begin{example}
	The equipment $\Mat\V$ is closed whenever $\V$ is a closed symmetric monoidal category that has products. In that case the left hom $J \lhom K$ of $\V$-matrices $\hmap JAB$ and $\hmap KAC$ can be given by
	\begin{displaymath}
		(J \lhom K)(b, c) = \prod_{a \in A} \brks{J(a, b), K(a, c)},
	\end{displaymath}
	where $\brks{\dash, \dash}$ denotes the inner hom functor of $\V$; the right homs are given similarly. To show that with these definitions $\catvar B = H(\Mat\V)$ satisfies the correspondences \eqref{equation:horizontal composition adjunction} notice that, for any $\hmap HBC$, a map of $\V$-matrices $\cell\phi{J \hc H}K$ is exactly a family of maps $\map{\phi_{a, b, c}}{J(a, b) \tens H(b, c)}{K(a, c)}$, for $a \in A$, $b \in B$ and $c \in C$. Under the tensor-hom adjunction of $\V$, this corresponds to a family $\map{\phi_{a, b, c}^\flat}{H(b, c)}{\brks{J(a, b), K(a, c)}}$, forming a map $\cell{\phi^\flat}H{J \lhom K}$\footnote{It is customary to denote the cell $H \Rar J \lhom K$, that corresponds to $\cell\phi{J \hc H}K$ under \eqref{equation:horizontal composition adjunction}, by $\phi^\flat$. Likewise, the cell $J \hc H \Rar K$ corresponding to $\cell\psi H{J \lhom K}$ is denoted $\psi^\sharp$.}.
\end{example}
	
	To give a condition under which the equipment $\Span\E$ of spans in $\E$ is closed, where $\E$ has finite limits, notice that the slice categories $\E \slash X$, with $X \in \ob \E$, are cartesian: the products of $\E \slash X$ are pullbacks in $\E$, over $X$, while the identity map on $X$ is the terminal object of $\E \slash X$. The category $\E$ is called \emph{locally cartesian closed} whenever all its cartesian slice categories are cartesian closed. In that case each functor $\map{f^*}{\E \slash Y}{\E \slash X}$, given by pullback along any $\map fXY$ in $\E$, has, besides a left adjoint $f_*$ given by postcomposition with $f$, a right adjoint $\map{\prod_f}{\E \slash X}{\E \slash Y}$ as well (see \cite[Theorem I.9.4]{MacLane-Moerdijk92}).
\begin{example}
	The horizontal bicategory underlying the equipment $\Span\E$ is the bicategory of spans and, as we saw in \lemref{horizontal composition of spans}, horizontal composition with a span $\map jJ{A \times B}$ can be given as
	\begin{displaymath}
		\map{J \hc \dash = (j_A \times \id_C)_* \of (j_B \times \id_C)^*}{\E \slash B \times C}{\E \slash A \times C}.
	\end{displaymath}
	It follows from the discussion above that a right adjoint to this functor can be given by $J \lhom \dash = \prod_{j_B \times \id_C} \of (j_A \times \id_C)^*$, whenever $\E$ is locally cartesian closed. For $\map hH{B \times C}$ the right homs $\dash \rhom H$ can be constructed in the same way, as the roles of $J$ and $H$ in $J \hc H$ are completely symmetric. Hence $\Span\E$ is closed if $\E$ is locally cartesian closed.
\end{example}
	The equipment $\Mod\K$ of monoids and bimodules (see \defref{definition:monoids and bimodules}) in a closed equipment $\K$ is often again closed by the following result, which is \cite[Proposition 11.16]{Shulman08}.
\begin{proposition}[Shulman]
	Let $\K$ be a closed equipment such that every category $H(\K)(A, B)$ has reflexive coequalisers and equalisers. Then $\Mod\K$ is again closed.
\end{proposition}
\begin{proof}[Sketch of the proof]
	The fact that $\K$ is closed means that the reflexive coequalisers in $H(\K)(A, B)$ are preserved by $\hc$ on both sides so that, by \propref{bimodule equipments}, the monoids and bimodules in $\K$ form an equipment $\Mod\K$. Given bimodules $\hmap JAB$ and $\hmap KAC$, their left hom $J \lhom_A K$ can be constructed as the equaliser
	\begin{displaymath}
		\begin{tikzpicture}
			\matrix(m)[math175em, column sep=2em]{J \lhom_A K & J \lhom K & (A \hc J) \lhom K \\};
			\path	(m-1-1)	edge[cell] (m-1-2)
						(m-1-2)	edge[transform canvas={yshift=3.5pt}, cell]  (m-1-3)
										edge[transform canvas={yshift=-3.5pt}, cell] (m-1-3);
		\end{tikzpicture}
	\end{displaymath}
	in $H(\K)(B_0, C_0)$, where the pair of parallel morphisms are induced by the action of $A$ on $J$ and $K$ respectively.
\end{proof}
\begin{example}
	If $\V$ is a closed symmetric monoidal category that is both complete and cocomplete then the equipment $\enProf\V$ of $\V$-profunctors is closed by the previous proposition: the left hom $J \lhom_A K$ of $\V$-profunctors $\hmap JAB$ and $\hmap KAC$ is given by the equalisers
	\begin{displaymath}
    (J \lhom_A K)(b, c) \to \prod\limits_{a \in A} \inhom{J(a, b), K(a, c)} \rightrightarrows \prod\limits_{a_1, a_2 \in A} \inhom{A(a_1, a_2) \tens J(a_2, b), K(a_1, c)},
  \end{displaymath}
  where the pair of maps are induced by the action of $A(a_1, a_2)$ on $J(a_2, b)$ and $K(a_2, c)$ respectively; shortly this is written as the end
  \begin{displaymath}
    J \lhom_A K = \int_{a \in A} \inhom{J(a, \dash), K(a, \dash)}.
  \end{displaymath}
  Notice that, by definition (see \cite[Formula 2.10]{Kelly82}), this means that $(J \lhom_A K)(b, c)$ is the object of $\V$-natural transformations $J(\dash, b) \natarrow K(\dash, c)$:
  \begin{equation} \label{equation:left hom and natural transformation object}
	  (J \lhom_A K)(b, c) = \brks{A, \V}(J(\dash, b), K(\dash, c)).
	\end{equation}
	The right homs are given similarly.
	
	The symmetric monoidal category $\2 = (\bot \to \top)$ is closed with respect to implication, that is
	\begin{displaymath}
		\brks{x, y} = \begin{cases}
			\bot & \text{if $(x,y) = (\top, \bot)$;} \\
			\top & \text{otherwise.}
			\end{cases}
	\end{displaymath}
	Notice that equalisers in $\2$ are always trivial, since $\2$ consists of a single non-identity map. Thus, the left hom $J \lhom_A K$ of $\2$-profunctors $\hmap JAB$ and $\hmap KAC$ (see \exref{example:enriched profunctor equipments}) is given by
	\begin{displaymath}
		(J \lhom_A K)(y, z) = \bigwedge_{x \in A} \brks{J(x, y), K(x,z)},
	\end{displaymath}
	in other words $x \sim_{J \lhom_A K} z$ if and only if $x \sim_J y$ implies $x \sim_K z$ for all $x \in A$.
\end{example}
	
	Before defining weighted colimits in closed equipments, we record some useful properties of the adjoints $\lhom$ and $\rhom$.
\begin{proposition} \label{left and right hom properties}
  Given composable horizontal morphisms $A \xslashedrightarrow H B \xslashedrightarrow K C \xslashedrightarrow L D$ in a closed pseudo double category $\K$, as well as $\hmap MAD$, the following pairs of horizontal morphisms are naturally isomorphic in $H(\K)$.
  \begin{itemize}
    \item[-] $K \lhom (H \lhom M) \iso (H \hc K) \lhom M$
    \item[-] $H \lhom (M \rhom L) \iso (H \lhom M) \rhom L$
    \item[-] $(M \rhom L) \rhom K \iso M \rhom (K \hc L)$
  \end{itemize}
\end{proposition}
\begin{proof}
  Using the Yoneda lemma, these follow directly from applying the isomorphisms \eqref{equation:horizontal composition adjunction}, once or twice, to the sides of the natural isomorphism below, that is given by precomposition with the associator for horizontal composition.
  \begin{displaymath}
    H(\K)(A,D)((H \hc K) \hc L, M) \iso H(\K)(A,D)(H \hc (K \hc L), M) \qedhere
  \end{displaymath}
\end{proof}
	The following is \cite[Proposition 5.11]{Shulman08}.
\begin{proposition}[Shulman] \label{cartesian filler properties}
  For a niche $A \xrar f C \xslashedrightarrow K D \xlar g B$ in a closed equipment $\K$, the following triples of promorphisms are naturally isomorphic in $H(\K)$.
  \begin{itemize}
    \item[-] $C(f, \id) \hc K \iso K(f, \id) \iso C(\id, f) \lhom K$
    \item[-] $K \hc D(\id, g) \iso K(\id, g) \iso K \rhom D(g, \id)$
  \end{itemize}
  From left to right, these isomorphisms are given by the horizontal cells that are adjoint to the counits of the companion-conjoint adjunction, see \propref{companion-conjoint adjunction}. Combining them gives isomorphisms
  \begin{displaymath}
	  C(f, \id) \hc K \hc D(\id, g) \iso K(f, g) \iso C(\id, f) \lhom K \rhom D(g, \id).
	\end{displaymath}
\end{proposition}
\begin{proof}
  We will only treat the first triple, the other being dual. The first isomorphism $C(f, \id) \hc K \iso K(f, \id)$ is given by \propref{cartesian fillers in terms of companions and conjoints}. That $C(f, \id) \hc K \iso C(\id, f) \lhom K$ holds as well follows from the correspondences below, where $\hmap JAD$ is any promorphism. They are given by the companion-conjoint adjunction and the composition-left hom adjunction respectively.
  \begin{align*}
	  H(\K)(A, D)(J, C(f, \id) \hc K) &\iso H(\K)(C, D)(C(\id, f) \hc J, K) \\
	  &\iso H(\K)(A, D)(J, C(\id, f) \lhom K) \qedhere
	\end{align*}	
\end{proof}
	Combining the previous propositions we obtain the following corollary. These isomorphisms will be used extensively.
\begin{corollary} \label{companion, conjoint and left hom isomorphisms}
	Given promorphisms $\hmap HAB$, $\hmap KAC$ and $\hmap LCD$ in a closed equipment $\K$, as well as morphisms $\map fAC$ and $\map gDB$, the following natural isomorphisms exist in $H(\K)$.
	\begin{itemize}
		\item[-] $H \lhom \bigpars{C(f, \id) \hc L} \iso \bigpars{C(\id, f) \hc H} \lhom L$
		\item[-] $B(g, \id) \hc \pars{H \lhom K} \iso \bigpars{H \hc B(\id, g)} \lhom K$
	\end{itemize}
	The first is adjoint to evaluation followed by the counit $\ls f\eps_f$ of the companion-conjoint adjunction (see \propref{companion-conjoint adjunction}), while the second is adjoint to $\ls g\eps_g$ followed by evaluation.
\end{corollary}

	In a closed equipment weighted colimits and weighted limits can be defined as follows.
\begin{definition} \label{definition:weighted colimits and limits}
  Let $\K$ be a closed equipment. Given a promorphism $\hmap JAB$ and a morphism $\map dAM$ in $\K$, the \emph{$J$-weighted colimit of $d$}, if it exists, is a morphism $\map{\colim_J d}BM$ together with an isomorphism
  \begin{displaymath}
    M(\colim\nolimits_J d, \id) \iso J \lhom M(d, \id)
  \end{displaymath}
  in $H(\K)$, of promorphisms $B \slashedrightarrow M$. Likewise, given a morphism $\map eBM$, the \emph{$J$-weighted limit of $e$}, if it exists, is a morphism \mbox{$\map{\lim_J e}AM$} together with an isomorphism
  \begin{displaymath}
    M(\id, \lim\nolimits_J e) \iso M(\id, e) \rhom J
  \end{displaymath}
  in $H(\K)$, of promorphisms $M \slashedrightarrow A$.
\end{definition}
	We will call the pair of morphisms $\cdiag MdAJB$ above a \emph{colimit diagram} in $\K$; the pair $\ldiag AJBeM$ is called a \emph{limit diagram}.
\begin{example} \label{example:weighted colimits in V-Prof}
	Given a $\V$-profunctor $\hmap JAB$ and a $\V$-functor $\map dAM$ in $\enProf\V$, the $J$-weighted colimit of $d$ is the $\V$-functor $\map{\colim_J d}BM$ satisfying
	\begin{equation} \label{equation:defining equation enriched weighted colimit}
		M\bigpars{\colim_J d(b), m} \iso \brks{\op A, \V}\bigpars{J(\dash, b), M(d\dash, m)},
	\end{equation}
	natural in both $b$ and $m$, as follows from \eqref{equation:left hom and natural transformation object}, where the $\V$-object on the right is that of $\V$-natural transformations $J(\dash, b) \natarrow M(d \dash, m)$. This means that at each $b$ the image $\colim_J d(b)$ is the usual $J(\dash, b)$-weighted colimit of $d$ in $M$, see \cite[Formula 3.5]{Kelly82}. Conversely if all weighted colimits $\colim_{J(\dash, b)} d$ exist in $M$, then $\colim_J d$ exists and can be given by $\colim_J d (b) = \colim_{J(\dash, b)} d$. This follows from the fact that the assignment $(W, d) \mapsto \colim_W d$, for pairs consisting of a $\V$-weight $\map W{\op A}\V$ and $\V$-diagram $\map dAM$ whose weighted colimit exists, is functorial, see \cite[Formula 3.11]{Kelly82}. Thus $\colim_J d$ can be thought of as being a `parametrised weighted colimit'.
	
	It follows that $\colim_J d$ can be given as
	\begin{equation} \label{equation:weighted colimit in terms of coends of copowers}
		\colim_J d(b) = \int^{a \in A} da \tens J(a, b),
	\end{equation}
	in terms of coends of copowers in $M$, see \cite[Formula 3.70]{Kelly82}, whenever such colimits exist in $M$. Here $da \tens J(a, b)$ is the copower of $da$ by $J(a,b)$, defined by an isomorphism $M(da \tens J(a, b), m) \iso \brks{J(a, b), M(da, m)}$ that is natural in $m$.
	
	In the case that $B = 1$, the unit $\V$-category consisting of a single object $*$ and morphism object $1(*,*) = 1$, the $\V$-profunctor $J$ is simply a $\V$-functor $\map J{\op A}\V$, under the isomorphism $\op A \tens 1 \iso \op A$, and $\map{\colim_J d}1M$ is the ordinary $J$\ndash weighted colimit of $d$.
	
	If $J$ is the companion $B(j, \id)$ of a $\V$-functor $\map jAB$ then the defining isomorphism becomes
	\begin{displaymath}
		M\bigpars{\colim_{B(j,\id)} d(b), m} \iso \brks{\op A, \V}\bigpars{B(j \dash, b), M(d\dash, m)},
	\end{displaymath}
	which is exactly the identity defining $\colim_{B(j, \id)} d$ as the \emph{enriched} left Kan extension of $d$ along $j$, in the sense of \cite[Formula 4.20]{Kelly82}, where they are simply called `left Kan extensions'. Enriched left Kan extensions are in particular (ordinary) left Kan extensions, see e.g.\ \cite[Theorem 4.43]{Kelly82}, and were first introduced (in the case of $M$ having $\V$-copowers) by Dubuc in his lecture notes \cite[Section I.4]{Dubuc70}, where they are called `pointwise' Kan extensions. We shall however reserve the prefix `pointwise' for a different notion of left Kan extension, given by Street, which is generally stronger than Dubuc's notion.
	
	If $\hmap JAB$ is a $\2$-profunctor and $\map dAM$ an order preserving map then the defining isomorphism \eqref{equation:defining equation enriched weighted colimit} is equivalent to
	\begin{displaymath}
		\colim_J d(y) = \sup\set{dx : x \sim_J y}.
	\end{displaymath}
	In the case that $J = B(j, \id)$ we recover the description of the left Kan extension of a pair of order preserving maps that was given in the introduction.
	
	So far we have recovered both ordinary weighted colimits and pointwise left Kan extensions as weighted colimits in $\enCat\V$; compare the unenriched situation as described in \tableref{table:weighted colimits in Cat} of the introduction. To see that ordinary (conical) colimits can be recovered as well, recall that the assignment $M \mapsto \ucat M$ mapping a $\V$-category $M$ to its underlying category $\ucat M$ has a left adjoint $C \mapsto \fren \V C$, as follows. Given an ordinary category $C$, the free $\V$-category $\fren \V C$ has the same objects, while its hom-objects $\fren \V C(x,y)$ are the copowers $C(x, y) \tens 1$, see \cite[Section 2.5]{Kelly82}. Notice that $\fren\V * = 1$, where $*$ is the terminal category. Given an ordinary category $C$, consider the companion $1(\fren \V!, \id)$ of the $\V$-functor $\map{\fren \V!}{\fren \V C}1$ induced by the terminal functor $\map !C*$. Then, for any unenriched functor $\map dC{\ucat M}$, the $1(\fren \V!, \id)$\ndash weighted colimit of the $\V$-functor $\map{d^\sharp}{\fren \V C}M$, that corresponds to $d$ under the adjunction $\fren \V\dash \ladj \ucat\dash$, is the conical colimit of $d$ in $M$. Indeed, the $\V$-profunctor $1(\fren\V !, \id)$, considered as a $\V$-functor $\op{\fren \V C} \to \V$, coincides with the adjoint $\Delta^\sharp_1$ of the constant functor $\map{\Delta_1}{\op C}{\ucat \V}$ at $1 \in \ucat\V$. It follows that $\colim_{1(\fren\V !, \id)} d^\sharp$ is the ordinary $\Delta^\sharp_1$-weighted colimit of $d^\sharp$, which is by definition the conical colimit of $d$, see \cite[Section 3.8]{Kelly82}.
\end{example}

\section{Weighted colimits in pseudo double categories}	
	In this section we compare the definition above, of weighted colimits in a closed equipment, to that of `left Kan extensions in a pseudo double category', as introduced by Grandis and Par\'e in \cite[Section 2]{Grandis-Pare08}. The results in this section are new.
	
	Remember that in a $2$-category $\C$ the (ordinary) left Kan extension of \mbox{$\map dAM$} along $\map jAB$ is a pair $(\map lBM, \cell\eta d{l \of j})$ such that, for any other pair $(\map eBM, \cell\phi d{e \of j})$ there is a unique cell $\cell{\phi'} le$ such that
	\begin{displaymath}
		\bigbrks{d \xRar{\eta} l \of j \xRar{\phi' \of j} e \of j} = \phi.
	\end{displaymath}
	In this case $\eta$ is said to \emph{exhibit} $l$ as the left Kan extension of $d$ along $j$; also $\eta$ is called the unit of $l$. Grandis and Par\'e generalise this notion to pseudo double categories as follows.
\begin{definition} \label{definition:left Kan extensions in pseudo double categories}
	In a pseudo double category $\K$ consider vertical morphisms $\map dAM$ and $\map lBM$, as well as horizontal morphisms $\hmap JAB$ and $\hmap IMN$. The cell $\eta$ below \emph{exhibits $l$ as the left Kan extension of $d$ from $J$ to $I$} if every cell $\phi$ factors uniquely through $\eta$ as a vertical cell $\phi'$, as shown. The cell $\eta$ is called the \emph{unit} of $l$.
	\begin{displaymath}
		\begin{tikzpicture}[textbaseline]
			\matrix(m)[math175em]{A & B \\ M & N \\};
			\path[map]	(m-1-1) edge[barred] node[above] {$J$} (m-1-2)
													edge node[left] {$d$} (m-2-1)
									(m-1-2) edge node[right] {$e$} (m-2-2)
									(m-2-1) edge[barred] node[below] {$I$} (m-2-2);
			\path[transform canvas={shift=($(m-1-2)!0.5!(m-2-2)$)}] (m-1-1) edge[cell] node[right] {$\phi$} (m-2-1);
		\end{tikzpicture}
		= \begin{tikzpicture}[textbaseline]
			\matrix(m)[math175em]{A & B & B \\ M & N & N \\[-3.25em] \phantom M & \phantom M & \phantom M \\};
			\path[map]	(m-1-1) edge[barred] node[above] {$J$} (m-1-2)
													edge node[left] {$d$} (m-2-1)
									(m-1-2) edge node[right] {$l$} (m-2-2)
									(m-1-3) edge node[right] {$e$} (m-2-3)
									(m-2-1) edge[barred] node[below] {$I$} (m-2-2);
			\path				(m-1-2) edge[eq] (m-1-3)
									(m-2-2) edge[eq] (m-2-3);
			\path[transform canvas={shift=($(m-1-2)!0.5!(m-2-3)$)}]	(m-1-1) edge[cell] node[right] {$\eta$} (m-2-1)
									(m-1-2) edge[cell] node[right] {$\phi'$} (m-2-2);
		\end{tikzpicture}
	\end{displaymath}
	
	A left Kan extension $l$ is called an \emph{absolute extension} by Grandis and Par\'e when cells $\phi$ of the more general form below also factor uniquely through $\eta$, as shown. In the case that such factorisations only exist for cells $\phi$ with $K = U_N$ we will call $l$ a \emph{top absolute extension}.
	\begin{displaymath}
		\begin{tikzpicture}[textbaseline]
			\matrix(m)[math175em]{A & B & C \\ M & N & P \\};
			\path[map]	(m-1-1) edge[barred] node[above] {$J$} (m-1-2)
													edge node[left] {$d$} (m-2-1)
									(m-1-2) edge[barred] node[above] {$H$} (m-1-3)
									(m-1-3) edge node[right] {$e$} (m-2-3)
									(m-2-1) edge[barred] node[below] {$I$} (m-2-2)
									(m-2-2) edge[barred] node[below] {$K$} (m-2-3);
			\path				(m-1-2) edge[cell] node[right] {$\phi$} (m-2-2);
		\end{tikzpicture}
		= \begin{tikzpicture}[textbaseline]
			\matrix(m)[math175em]{A & B & C \\ M & N & P \\[-3.25em] \phantom M & \phantom M & \phantom M \\};
			\path[map]	(m-1-1) edge[barred] node[above] {$J$} (m-1-2)
													edge node[left] {$d$} (m-2-1)
									(m-1-2) edge[barred] node[above] {$H$} (m-1-3)
													edge node[right] {$l$} (m-2-2)
									(m-1-3) edge node[right] {$e$} (m-2-3)
									(m-2-1) edge[barred] node[below] {$I$} (m-2-2)
									(m-2-2) edge[barred] node[below] {$K$} (m-2-3);
			\path[transform canvas={shift=($(m-1-2)!0.5!(m-2-3)$)}]	(m-1-1) edge[cell] node[right] {$\eta$} (m-2-1)
									(m-1-2) edge[cell] node[right] {$\phi'$} (m-2-2);
		\end{tikzpicture}
	\end{displaymath}
\end{definition}
\begin{remark} \label{remark:left Kan extensions in pseudo double categories}
	Notice that the weakest kind of factorisation above, that of cells $\cell \phi JI$, implies that any two cells exhibiting the left Kan extension of $d$ from $J$ to $I$ coincide, up to horizontal composition on the right with a vertical isomorphism. In particular, if we know that the (top) absolute left Kan extension of $d$ from $J$ to $I$ exists then any exhibition of $l$ as the (ordinary) left Kan extension of $d$ from $J$ to $I$ also exhibits it as the (top) absolute extension.
\end{remark}

	The proposition below proves that $J$-weighted colimits in a closed equipment coincide with top absolute extensions from $J$ to a unit promorphism. To state it notice that every cell below corresponds to a horizontal cell $\cell{\lambda\eta}{J \hc M(l, \id)}{M(d, \id)}$ (see \propref{left and right cells}) that in turn is adjoint, under the adjunction \eqref{equation:horizontal composition adjunction}, to a horizontal cell $\cell{(\lambda\eta)^\flat}{M(l, \id)}{J \lhom M(d, \id)}$. 
	\begin{displaymath}
		\begin{tikzpicture}
			\matrix(m)[math175em]{A & B \\ M & M \\};
			\path[map]	(m-1-1) edge[barred] node[above] {$J$} (m-1-2)
													edge node[left] {$d$} (m-2-1)
									(m-1-2) edge node[right] {$l$} (m-2-2);
			\path				(m-2-1) edge[eq] (m-2-2);
			\path[transform canvas={shift=($(m-1-2)!0.5!(m-2-2)$)}] (m-1-1) edge[cell] node[right] {$\eta$} (m-2-1);
		\end{tikzpicture}
	\end{displaymath}
	The correspondence above is bijective; in particular any isomorphism $M(l, \id) \iso J \lhom M(d, \id)$, that defines $l$ as the $J$-weighted colimit of $d$ (\defref{definition:weighted colimits and limits}), is, from left to right, of the form $(\lambda\eta)^\flat$ for some unique cell $\eta$ that is of the form above.
\begin{example} \label{example:units of weighted colimits in V-Prof}
	Of course in the case of the equipment $\enProf\V$ of $\V$-profunctors, the cell $\eta$ corresponding to $M(l, \id) \iso J \lhom M(d, \id)$ is the $\V$-natural transformation $J \natarrow M(d, l)$ which, restricted to each $b$ in $B$, is the unit $\nat{\eta_{\dash, b}}{J(\dash, b)}{M(d\dash, lb)}$ that defines $lb$ as the $J(\dash, b)$-weighted colimit of $d$ (see \exref{example:weighted colimits in V-Prof}).
\end{example}
\begin{proposition} \label{weighted colimits equivalent to top absolute left Kan extensions}
	For a cell $\eta$ above, in a closed equipment $\K$, the following are equivalent:
	\begin{itemize}
		\item[-]	$\eta$ defines $l$ as the top absolute left Kan extension of $d$ from $J$ to $U_M$ (\defref{definition:left Kan extensions in pseudo double categories});
		\item[-]	the corresponding horizontal cell $(\lambda\eta)^\flat$, as given above, is an isomorphism, thus defining $l$ as the $J$-weighted colimit of $d$ (\defref{definition:weighted colimits and limits}).
	\end{itemize}
\end{proposition}
\begin{proof}
	The first statement means that we have unique factorisations
	\begin{displaymath}
		\begin{tikzpicture}[textbaseline]
			\matrix(m)[math175em]{A & B & C \\ M & \phantom M & M \\};
			\path[map]	(m-1-1) edge[barred] node[above] {$J$} (m-1-2)
													edge node[left] {$d$} (m-2-1)
									(m-1-2) edge[barred] node[above] {$H$} (m-1-3)
									(m-1-3) edge node[right] {$e$} (m-2-3);
			\path				(m-2-1) edge[eq] (m-2-3)
									(m-1-2) edge[cell] node[right] {$\phi$} (m-2-2);
		\end{tikzpicture}
		= \begin{tikzpicture}[textbaseline]
			\matrix(m)[math175em]{A & B & C \\ M & M & M, \\};
			\path[map]	(m-1-1) edge[barred] node[above] {$J$} (m-1-2)
													edge node[left] {$d$} (m-2-1)
									(m-1-2) edge[barred] node[above] {$H$} (m-1-3)
													edge node[right] {$l$} (m-2-2)
									(m-1-3) edge node[right] {$e$} (m-2-3);
			\path				(m-2-1) edge[eq] (m-2-2)
									(m-2-2) edge[eq] (m-2-3);
			\path[transform canvas={shift=($(m-1-2)!0.5!(m-2-3)$)}]	(m-1-1) edge[cell] node[right] {$\eta$} (m-2-1)
									(m-1-2) edge[cell] node[right] {$\phi'$} (m-2-2);
		\end{tikzpicture}
	\end{displaymath}
	which we claim is equivalent to the existence of the following unique factorisations of horizontal cells. This follows directly from the fact that under the correspondence $\phi \mapsto \lambda\phi$ (which preserves composition, see \propref{left and right cells}) a factorisation of the first type corresponds to one of the second type, and vice versa.
	\begin{displaymath}
		\begin{tikzpicture}[textbaseline]
			\matrix(m)[math175em]{A & B & M \\ A & & M \\};
			\path[map]	(m-1-1) edge[barred] node[above] {$J$} (m-1-2)
									(m-1-2) edge[barred] node[above] {$H'$} (m-1-3)
									(m-2-1) edge[barred] node[below] {$M(d, \id)$} (m-2-3);
			\path				(m-1-1) edge[eq] (m-2-1)
									(m-1-3) edge[eq] (m-2-3)
									(m-1-2) edge[cell] node[right] {$\psi$} (m-2-2);
		\end{tikzpicture}
		= \begin{tikzpicture}[textbaseline]
			\matrix(m)[math175em]{A & B & M \\ A & B & M \\ A & & M \\};
			\path[map]	(m-1-1) edge[barred] node[above] {$J$} (m-1-2)
									(m-1-2) edge[barred] node[above] {$H'$} (m-1-3)
									(m-2-1) edge[barred] node[below] {$J$} (m-2-2)
									(m-2-2) edge[barred] node[below] {$M(l, \id)$} (m-2-3)
									(m-3-1) edge[barred] node[below] {$M(d, \id)$} (m-3-3);
			\path				(m-1-1) edge[eq] (m-2-1)
									(m-1-2) edge[eq] (m-2-2)
									(m-1-3) edge[eq] (m-2-3)
									(m-2-1) edge[eq] (m-3-1)
									(m-2-3) edge[eq] (m-3-3);
			\path[transform canvas={yshift=-0.4em}]	(m-2-2) edge[cell] node[right] {$\lambda\eta$} (m-3-2);
			\path[transform canvas={shift=($(m-2-2)!0.5!(m-2-3)$)}]	(m-1-1) edge[cell] node[right] {$\id$} (m-2-1)
									(m-1-2) edge[cell] node[right] {$\psi'$} (m-2-2);
		\end{tikzpicture}
	\end{displaymath}
	But the factorisation of horizontal cells above means precisely that $\lambda\eta$ forms a `universal arrow from the functor $J \hc \dash$ to the object $M(d, \id)$', as defined on \cite[Page 58]{MacLane98}. Since evaluation $\cell\ev{J \hc J \lhom M(d, \id)}{M(d, \id)}$, which is the counit of the adjunction $J \hc \dash \ladj J \lhom \dash$, forms a universal arrow from $J \hc \dash$ to $M(d, \id)$ as well (see \cite[Theorem IV.1.2]{MacLane98}), and because universal arrows are unique up to isomorphism, we conclude that the factorisation of horizontal cells above exists if and only if there is a unique invertible horizontal cell $\cell\chi{M(l, \id)}{J \lhom M(d, \id)}$ such that $\ev \of (\id_J \hc \chi) = \lambda\eta$. This equation means that $\chi$ is adjoint to $\lambda\eta$, so that the latter statement is equivalent to $(\lambda\eta)^\flat = \chi$ being invertible, which is what we wanted to show.
\end{proof}
	
	Consequently we extend \defref{definition:weighted colimits and limits}, of weighted colimits in closed equipments, to one for pseudo double categories, as follows.
\begin{definition} \label{definition:weighted colimits in double categories}
	Given morphisms $\hmap JAB$ and $\map dAM$ in a pseudo double category $\K$, the \emph{$J$-weighted colimit of $d$} is the top absolute left Kan extension of $d$ from $J$ to $U_M$, as defined in \defref{definition:left Kan extensions in pseudo double categories}.
\end{definition}
	
	In \exref{example:weighted colimits in V-Prof} we saw that colimits weighted by companions in $\enProf\V$ are left Kan extensions. Recall that in any double category $\K$, the companion of a vertical map $\map jAB$ is defined as the horizontal morphism $\hmap{B(j, \id)}AB$ that comes equipped with a cartesian filler $\cell{\ls j\eps}{B(j, \id)}B$ of the niche $B \xrar\id B \xsrar{U_B} B \xlar j A$, that defines $B(j, \id)$ as the restriction of $U_B$ along $j$ and $\id$; see the discussion following \defref{definition:cartesian filler}. Factorising the horizontal unit $U_j$ through $\ls j\eps$ gives $\cell{\ls j\eta}{U_A}{B(j, \id)}$.
\begin{proposition} \label{colimits weighted by companions}
	Let $\map jAB$, $\map dAM$ and $\map lBM$ be vertical maps in a pseudo double category $\K$ and assume that $j$ has a companion. If the cell
	\begin{displaymath}
		\begin{tikzpicture}
			\matrix(m)[math175em]{A & B \\ M & M \\};
			\path[map]	(m-1-1) edge[barred] node[above] {$B(j, \id)$} (m-1-2)
													edge node[left] {$d$} (m-2-1)
									(m-1-2) edge node[right] {$l$} (m-2-2);
			\path				(m-2-1) edge[eq] (m-2-2);
			\path[transform canvas={shift=($(m-1-2)!0.5!(m-2-2)$)}] (m-1-1) edge[cell] node[right] {$\zeta$} (m-2-1);
		\end{tikzpicture}
	\end{displaymath}
	exhibits $l$ as the $B(j, \id)$-weighted colimit of $d$ then the vertical cell $\cell{\zeta \of \ls j\eta}d{l \of j}$ exhibits $l$ as the left Kan extension of $d$ along $j$ in the vertical $2$-category $V(\K)$. The converse holds whenever the $B(j,\id)$-weighted colimit of $d$ is known to exist.
\end{proposition}
\begin{proof}
	The implication follows directly from the fact that there is a correspondence between cells of the form
	\begin{displaymath}
		\begin{tikzpicture}[textbaseline]
			\matrix(m)[math175em]{A & B \\ M & M \\};
			\path[map]	(m-1-1) edge[barred] node[above] {$B(j, \id)$} (m-1-2)
													edge node[left] {$d$} (m-2-1)
									(m-1-2) edge node[right] {$e$} (m-2-2);
			\path				(m-2-1) edge[eq] (m-2-2);
			\path[transform canvas={shift=($(m-1-2)!0.5!(m-2-2)$)}] (m-1-1) edge[cell] node[right] {$\phi$} (m-2-1);
		\end{tikzpicture}
		\qquad \text{and} \qquad \begin{tikzpicture}[textbaseline]
			\matrix(m)[math175em]{A & A \\ & B \\ M & M, \\ };
			\path[map]	(m-1-1) edge node[left] {$d$} (m-3-1)
									(m-1-2) edge node[right] {$j$} (m-2-2)
									(m-2-2) edge node[right] {$e$} (m-3-2);
			\path				(m-1-1) edge[eq] (m-1-2)
									(m-3-1) edge[eq] (m-3-2);
			\path[transform canvas={shift=($(m-2-2)!0.5!(m-3-2)$)}] (m-1-1) edge[cell] node[right] {$\chi$} (m-2-1);
		\end{tikzpicture}
	\end{displaymath}
	given by $\phi \mapsto \phi \of \ls j\eta$ and $\chi \mapsto \chi \hc (U_e \of \ls j\eps)$, that is natural with respect to horizontal composition on the right with vertical maps $e \Rar f$, for any second map $\map fBM$. The converse follows from \remref{remark:left Kan extensions in pseudo double categories}.
\end{proof}
	
	Consequently we make the following definition.
\begin{definition} \label{definition:weighted left Kan extensions}
	Consider morphisms $\map jAB$ and $\map dAM$ in a pseudo double category $\K$, such that $j$ has a companion $B(j, \id)$. We define the \emph{weighted left Kan extension} $\map{\lan_j d}BM$ of $d$ along $j$ to be the weighted colimit $\lan_j d = \colim_{B(j, \id)} d$.
\end{definition}
	As we saw in \exref{example:weighted colimits in V-Prof}, the weighted left Kan extensions in the equipment $\enProf\V$ of $\V$-profunctors are precisely the enriched Kan extensions in $\enCat\V$. Proving that the weighted left Kan extensions in the equipment $\inProf\E$ of profunctors internal to $\E$ are precisely the `pointwise' Kan extensions in $\inCat\E$ will be the aim of the final section of this chapter.
	
	We close this section by giving two useful properties of weighted colimits. The first is a `Fubini theorem' for iterated weighted colimits, like the Fubini theorems for iterated limits in ordinary categories (\cite[Section IX.8]{MacLane98}) and for iterated ordinary weighted limits in enriched categories (\cite[Section 3.3]{Kelly82}). Its proof is easy and omitted.
\begin{proposition}
	Consider horizontally composable cells
	\begin{displaymath}
		\begin{tikzpicture}
			\matrix(m)[math175em]{A & B & C \\ M & M & M \\};
			\path[map]	(m-1-1) edge[barred] node[above] {$J$} (m-1-2)
													edge node[left] {$d$} (m-2-1)
									(m-1-2) edge[barred] node[above] {$H$} (m-1-3)
													edge node[right] {$l$} (m-2-2)
									(m-1-3) edge node[right] {$k$} (m-2-3);
			\path	(m-2-1) edge[eq] (m-2-2)
						(m-2-2) edge[eq] (m-2-3);
			\path[transform canvas={shift={($(m-1-2)!0.5!(m-2-3)$)}}] (m-1-1) edge[cell] node[right] {$\eta$} (m-2-1)
			(m-1-2) edge[cell] node[right] {$\zeta$} (m-2-2);
		\end{tikzpicture}
	\end{displaymath}
	in a pseudo double category. If $\eta$ exhibits $l$ as the $J$-weighted colimit of $d$ and $\zeta$ exhibits $k$ as the $H$-weighted colimit of $l$ then $\eta \hc \zeta$ exhibits $k$ as the $J \hc H$-weighted colimit of $d$.
\end{proposition}
	Remember that a cell $\phi$ in an equipment is called right invertible whenever its corresponding horizontal cell $\rho\phi$, that is obtained by horizontally composing $\phi$ with conjoint cells (see \propref{left and right cells}), is invertible.
\begin{proposition} \label{weighted colimits are preserved by precomposition with invertible cells}
	Consider vertically composable cells
	\begin{displaymath}
		\begin{tikzpicture}
			\matrix(m)[math175em]{A & B \\ C & D \\ M & M \\};
			\path[map]	(m-1-1) edge[barred] node[above] {$J$} (m-1-2)
													edge node[left] {$f$} (m-2-1)
									(m-1-2) edge node[right] {$g$} (m-2-2)
									(m-2-1) edge[barred] node[below] {$K$} (m-2-2)
													edge node[left] {$d$} (m-3-1)
									(m-2-2) edge node[right] {$l$} (m-3-2);
			\path	(m-3-1) edge[eq] (m-3-2);
			\path[transform canvas={shift=(m-2-2)}]	(m-1-1) edge[cell] node[right] {$\phi$} (m-2-1);
			\path[transform canvas={shift=(m-2-2), yshift=-3pt}]	(m-2-1) edge[cell] node[right] {$\zeta$} (m-3-1);
		\end{tikzpicture}
	\end{displaymath}
	in an equipment. If $\phi$ is right invertible then $\zeta$ exhibits $l$ as a weighted colimit if and only if $\zeta \of \phi$ exhibits $l \of g$ as a weighted colimit.
\end{proposition}
\begin{proof}
	Consider cells that are of one of the four forms shown below. 
	\begin{displaymath}
		\begin{tikzpicture}
			\matrix(k)[math175em]{A & B & E \\ C & & \\ M & \phantom M & M \\};
			\path[map]	(k-1-1) edge[barred] node[above] {$J$} (k-1-2)
													edge node[left] {$f$} (k-2-1)
									(k-1-2) edge[barred] node[above] {$H$} (k-1-3)
									(k-1-3) edge node[right] {$e$} (k-3-3)
									(k-2-1) edge node[left] {$d$} (k-3-1);
			\path	(k-3-1) edge[eq] (k-3-3);
			\path[transform canvas={shift={($(k-2-2)!0.5!(k-3-2)$)}}]	(k-1-2) edge[cell] node[right] {$\psi$} (k-2-2);
			
			\matrix(l)[math175em, xshift=17.5em]{C & D & B & E \\ M & \phantom M & \phantom M & M \\};
			\path[map]	(l-1-1) edge[barred] node[above] {$K$} (l-1-2)
													edge node[left] {$d$} (l-2-1)
									(l-1-2) edge[barred] node[above] {$D(\id, g)$} (l-1-3)
									(l-1-3) edge[barred] node[above] {$H$} (l-1-4)
									(l-1-4) edge node[right] {$e$} (l-2-4);
			\path	(l-2-1) edge[eq] (l-2-4);
			\path[transform canvas={shift={($(l-1-3)!0.5!(l-2-3)$)}, xshift=-17.5em}]	(l-1-2) edge[cell] node[right] {$\chi$} (l-2-2);
			
			\matrix(m)[math175em, yshift=-12.5em]{B & E \\ D & \\ M & M \\};
			\path[map]	(m-1-1) edge[barred] node[above] {$H$} (m-1-2)
													edge node[left] {$g$} (m-2-1)
									(m-1-2) edge node[right] {$e$} (m-3-2)
									(m-2-1) edge node[left] {$l$} (m-3-1);
			\path (m-3-1) edge[eq] (m-3-2);
			\path[transform canvas={shift={($(m-2-2)!0.5!(m-3-2)$)}, yshift=12.5em}]	(m-1-1) edge[cell] node[right] {$\psi'$} (m-2-1);
			
			\matrix(n)[math175em, xshift=17.5em, yshift=-12.5em]{D & B & E \\ M & \phantom M & M \\};
			\path[map]	(n-1-1) edge[barred] node[above] {$D(\id, g)$} (n-1-2)
													edge node[left] {$l$} (n-2-1)
									(n-1-2) edge[barred] node[above] {$H$} (n-1-3)
									(n-1-3) edge node[right] {$e$} (n-2-3);
			\path	(n-2-1) edge[eq] (n-2-3)
						(n-1-2) edge[cell] node[right] {$\chi'$} (n-2-2);
			
			\path[<->]	(k) edge[shorten >= 2em, shorten <= 2em] (l)
									(m) edge[shorten >= 2em, shorten <= 2em] (n);
			\path[<->, dashed] (k) edge[shorten <= 0.75em, shorten >= 1.25em] (m)
									(l) edge[shorten >= 2em, shorten <= 2em] (n);
		\end{tikzpicture}
	\end{displaymath}
	Here the solid double headed arrows denote the following bijective correspondences. Cells that are of form of $\psi'$ correspond to those of the form of $\chi'$ under the assignments $\psi' \mapsto (U_l \of \eps_g) \hc \psi'$ and $\chi' \mapsto \chi' \of (\eta_g \hc \id_H)$; that these give a bijective correspondence follows from the conjoint identities. Likewise cells of the form of $\psi$ correspond to those of the form of $\chi$, by assigning $\psi \mapsto \bigpars{(U_d \of \eps_f) \hc \psi} \of \bigpars{\inv{(\rho\phi)} \hc \id_H}$ and $\chi \mapsto \chi \of (\phi \hc \eta_g \hc \id_H)$. That the latter give a bijective correspondence too follows from the conjoint identities and the fact that $\rho\phi = \eps_f \hc \phi \hc \eta_g$, see \propref{left and right cells}.
	
	Moreover $\zeta \of \phi$ exhibits $l \of g$ as a weighted colimit if and only if every cell $\psi$ above factors uniquely through $\zeta \of \phi$ as a cell $\psi'$ as above, such that $(\zeta \of \phi) \hc \psi' = \psi$, as is denoted by the dashed double headed arrow on the left. Likewise if $\zeta$ exhibits $l$ as a weighted colimit of $d$ then every cell $\chi$ above factors uniquely through $\zeta$ as a cell $\chi'$ as above, such that $\zeta \hc \chi' = \chi$, as denoted on the right by the dashed double headed arrow. The latter are in fact equivalent: by taking $g = \id_B$ and vertically precomposing the cells $\chi$ and $\chi'$ with the isomorphisms $H \iso B(\id, \id) \hc H$ we see that, in terms of \propref{weighted colimits equivalent to top absolute left Kan extensions}, $\zeta$ exhibits $l$ as the top absolute left Kan extension of $d$ along $J$.
	
	We claim that these two conditions are equivalent. Indeed consider four cells $\psi$, $\chi$, $\psi'$ and $\chi'$ above, such that the pairs $(\psi, \chi)$ and $(\psi', \chi')$ correspond under the solid correspondences. To see that the correspondence on the right, that is $\zeta \hc \chi' = \chi$, implies that on the left, i.e.\ $(\zeta \of \phi) \hc \psi' = \psi$, consider the simplified grids of cells below: still representing composites of cells these only describe the shape of non-identity cells and how they fit together; to save space objects, morphisms and promorphisms are left out. The identities follow from the assumptions that $\psi'$ and $\chi'$ correspond, $\zeta \hc \chi' = \chi$ and that $\psi$ and $\chi$ correspond.
	\begin{displaymath}
		\begin{tikzpicture}[textbaseline, x=0.7cm, y=0.6cm, font=\scriptsize]
			\draw (1,0) -- (0,0) -- (0,2) -- (1,2) -- (1,0) -- (2,0) -- (2,2) -- (1,2)
				(0,1) -- (1,1);
			\draw[shift={(0.5,0)}]
				(1,1) node {$\psi'$};
			\draw[shift={(0.5,0.5)}]
				(0,0) node {$\zeta$}
				(0,1) node {$\phi$};
		\end{tikzpicture}
		\quad = \quad \begin{tikzpicture}[textbaseline, x=0.7cm, y=0.6cm, font=\scriptsize]
			\draw	(0,1) -- (0,0) -- (3,0) -- (3,1) -- (0,1) -- (0,2) -- (2,2) -- (2,1)
						(1,0) -- (1,2);
			\draw[shift={(0,0.5)}]
				(2,0) node {$\chi'$};
			\draw[shift={(0.5,0.5)}]
				(0,0) node {$\zeta$}
				(0,1) node {$\phi$}
				(1,1) node {$\eta_g$};
		\end{tikzpicture}
		\quad = \quad \begin{tikzpicture}[textbaseline, x=0.7cm, y=0.6cm, font=\scriptsize]
			\draw (2,1) -- (3,1) -- (3,0) -- (0,0) -- (0,2) -- (2,2) -- (2,1) -- (0,1)
						(1,2) -- (1,1);
			\draw[shift={(0.5,0.5)}]
				(1,0) node {$\chi$}
				(0,1) node {$\phi$}
				(1,1) node {$\eta_g$};
		\end{tikzpicture}
		\quad = \quad \begin{tikzpicture}[textbaseline, x=0.7cm, y=0.6cm, font=\scriptsize]
			\draw (0,0) -- (0,2) -- (2,2) -- (2,0) -- (0,0)
				(1,1) node {$\psi$};
		\end{tikzpicture}
	\end{displaymath}
	For the converse assume that $(\zeta \of \phi) \hc \psi' = \psi$ holds. Then
	\begin{align*}
		\begin{tikzpicture}[textbaseline, x=0.7cm, y=0.6cm, font=\scriptsize]
			\draw (1,1) -- (0,1) -- (0,0) -- (3,0) -- (3,1) -- (1,1) -- (1,0);
			\draw[shift={(0,0.5)}]
				(2,0) node {$\chi'$};
			\draw[shift={(0.5,0.5)}]
				(0,0) node {$\zeta$};
		\end{tikzpicture}
		&\quad = \quad \begin{tikzpicture}[textbaseline, x=0.7cm, y=0.6cm, font=\scriptsize]
			\draw	(1,1) -- (0,1) -- (0,0) -- (3,0) -- (3,1) -- (1,1) -- (1,0)
						(0,1) -- (0,2) -- (2,2) -- (2,1)
						(0,2) -- (0,3) -- (2,3) -- (2,2);
			\draw[shift={(0,0.5)}]
				(2,0) node {$\chi'$}
				(1,1) node {$\rho\phi$}
				(1,2) node {$\inv{(\rho\phi)}$};
			\draw[shift={(0.5,0.5)}]
				(0,0) node {$\zeta$};
		\end{tikzpicture}
		\quad = \quad \begin{tikzpicture}[textbaseline, x=0.7cm, y=0.6cm, font=\scriptsize]
			\draw (2,1) -- (0,1) -- (0,3) -- (2,3) -- (2,0) -- (3,0) -- (3,2) -- (0,2)
						(1,2) -- (1,0) -- (2,0);
			\draw[shift={(0.5,0)}]
				(2,1) node {$\psi'$};
			\draw[shift={(0.5,0.5)}]
				(1,0) node {$\zeta$}
				(0,1) node {$\eps_f$}
				(1,1) node {$\phi$};
			\draw[shift={(0,0.5)}]
				(1,2) node {$\inv{(\rho\phi)}$};
		\end{tikzpicture} \\[1em]
		&\quad = \quad \begin{tikzpicture}[textbaseline, x=0.7cm, y=0.6cm, font=\scriptsize]
			\draw (1,1) -- (0,1) -- (0,2) -- (3,2) -- (3,0) -- (1,0) -- (1,2)
						(0,2) -- (0,3) -- (2,3) -- (2,2)
						(2,1) node {$\psi$};
			\draw[shift={(0.5,0.5)}]
				(0,1) node {$\eps_f$};
			\draw[shift={(0, 0.5)}]
				(1,2) node {$\inv{(\rho\phi)}$};
		\end{tikzpicture}
		\quad = \quad \begin{tikzpicture}[textbaseline, x=0.7cm, y=0.6cm, font=\scriptsize]
			\draw (0,0) -- (3,0) -- (3,1) -- (0,1) -- (0,0);
			\draw[shift={(0.5,0.5)}]
				(1,0) node {$\chi$};
		\end{tikzpicture}\,,
	\end{align*}
	where the identities follow from $\rho\phi = \eps_f \hc \phi \hc \eta_g$ (\propref{left and right cells}) and the assumption that $\psi'$ and $\chi'$ correspond; $(\zeta \of \phi) \hc \psi' = \psi$; the assumption that $\psi$ and $\chi$ correspond. This completes the proof.
\end{proof}
\begin{nremark}
	The economic representation of composites that was introduced in the proof above will be used regularly from now on. This will save some space as the composites of cells that we need to consider are getting bigger.
\end{nremark}
	
\section{Pointwise weighted colimits}	\label{section:pointwise weighted colimits}
	In his paper \cite{Street74} Street uses `comma objects' to give a notion of `pointwise' left Kan extensions in $2$-categories, which was generalised by Grandis and Par\'e in \cite[Section 4.1]{Grandis-Pare08} to a notion of `pointwise' Kan extensions in pseudo double categories, using `double comma objects'. By modifying the latter slightly, we will introduce our notion of `pointwise' weighted colimits, given in \defref{definition:pointwise weighted colimits}. Extending the notion of pointwise left Kan extensions, these are in general stronger than the ordinary weighted colimits, as we will see in \exref{example:stronger pointwise left Kan extension}. To understand this difference we will consider `strong' double comma objects in equipments, that satisfy a property isolating the cause of this difference. The main result of this section (\thmref{weighted colimits are pointwise if double commas are strong}) proves that if each double comma object in an equipment $\K$ is strong, then the pointwise and ordinary notion of weighted colimit coincide. In particular this holds for the equipment $\K = \inProf\E$ of profunctors internal to $\E$ so that, where weighted colimits in $\enProf\V$ generalise Dubuc's enriched left Kan extensions (see \exref{example:weighted colimits in V-Prof}), in $\inProf\E$ they generalise Street's pointwise left Kan extensions. The notion of strong double comma objects introduced here as well as all results in this section about them seem to be new.
	
	We begin with the well-known notion of comma objects in $2$-categories.
\begin{definition} \label{definition:comma object}
	Given two morphisms $\map fAC$ and $\map gBC$ in a $2$-category $\C$, the \emph{comma object} $f \slash g$ of $f$ and $g$ consists of an object $f \slash g$ equipped with projections $\pi_A$ and $\pi_B$, as in the diagram below, as well as a cell $\cell\pi{f \of \pi_A}{g \of \pi_B}$, satisfying the following $1$-dimensional and $2$-dimensional universal properties.
	\begin{displaymath}
		\begin{tikzpicture}
			\matrix(m)[math175em]{f \slash g & A \\ B & C \\};
			\path[map]	(m-1-1) edge node[above] {$\pi_A$} (m-1-2)
													edge node[left] {$\pi_B$} (m-2-1)
									(m-1-2) edge node[right] {$f$} (m-2-2)
									(m-2-1) edge node[below] {$g$} (m-2-2);
			\path				(m-1-2) edge[cell, shorten >= 8pt, shorten <= 8pt] node[below right] {$\pi$} (m-2-1);
		\end{tikzpicture}
	\end{displaymath}
	Given any other cell $\phi$ in $\C$ as on the left below, the $1$-dimensional property states that there exists a unique morphism $\map{\phi'}U{f \slash g}$ such that $\pi_A \of \phi' = \phi_A$, \mbox{$\pi_B \of \phi' = \phi_B$} and $\pi \of \phi' = \phi$.
	\begin{displaymath}
		\begin{tikzpicture}[textbaseline]
			\matrix(m)[math175em]{U & A \\ B & C \\};
			\path[map]	(m-1-1) edge node[above] {$\phi_A$} (m-1-2)
													edge node[left] {$\phi_B$} (m-2-1)
									(m-1-2) edge node[right] {$f$} (m-2-2)
									(m-2-1) edge node[below] {$g$} (m-2-2);
			\path				(m-1-2) edge[cell, shorten >= 6pt, shorten <= 6pt] node[below right] {$\phi$} (m-2-1);
		\end{tikzpicture}
		\qquad\qquad\qquad\qquad\qquad\begin{tikzpicture}[textbaseline]
			\matrix(m)[math175em]{U & A \\ B & C \\};
			\path[map]	(m-1-1) edge node[above] {$\psi_A$} (m-1-2)
													edge node[left] {$\psi_B$} (m-2-1)
									(m-1-2) edge node[right] {$f$} (m-2-2)
									(m-2-1) edge node[below] {$g$} (m-2-2);
			\path				(m-1-2) edge[cell, shorten >= 6pt, shorten <= 6pt] node[below right] {$\psi$} (m-2-1);
		\end{tikzpicture}
	\end{displaymath}
	The $2$-dimensional property is the following. Suppose we are given a further cell $\psi$ as on the right above, which factors through $\pi$ as $\map{\psi'}U{f \slash g}$, like $\phi$ factors as $\phi'$. Then for any pair of cells $\cell{\xi_A}{\phi_A}{\psi_A}$ and $\cell{\xi_B}{\phi_B}{\psi_B}$, such that $\psi \of (f \of \xi_A) = (g \of \xi_B) \of \phi$, there exists a unique cell $\cell{\xi'}{\phi'}{\psi'}$ such that $\pi_A \of \xi' = \xi_A$ and $\pi_B \of \xi' = \xi_B$.
\end{definition}
	We remark that formally the comma object $f \slash g$ can also be defined as the $\cat$\ndash enriched $J$-weighted limit of the diagram of $f$ and $g$, with the weight as follows, where $*$ denotes the terminal category and $\2 = (\bot \to \top)$.
	\begin{displaymath}
		J = \begin{tikzpicture}[textbaseline]
			\matrix(m)[math175em]{ & * \\ * & \2 \\};
			\path[map]	(m-1-2) edge node[right] {$\top$} (m-2-2)
									(m-2-1) edge node[below] {$\bot$} (m-2-2);
		\end{tikzpicture}
	\end{displaymath}
	
\begin{example} \label{example:comma objects}
	In the case that $f$ and $g$ are ordinary functors in $\C = \Cat$ the comma object $f \slash g$ is the usual comma category, whose objects are triples $(a, x, b)$ with $a \in A$, $b \in B$ and $\map x{fa}{gb}$ in $C$, while a map $(a, x, b) \to (a', x', b')$ is a pair $\map{(u, v)}{(a,b)}{(a', b')}$ in $A \times B$ making the diagram
	\begin{displaymath}
		\begin{tikzpicture}
			\matrix(m)[math175em]{fa & gb \\ fa' & gb' \\};
			\path[map]	(m-1-1) edge node[above] {$x$} (m-1-2)
													edge node[left] {$fu$} (m-2-1)
									(m-1-2) edge node[right] {$gv$} (m-2-2)
									(m-2-1) edge node[below] {$x'$} (m-2-2);
		\end{tikzpicture}
	\end{displaymath}
	commute. The natural transformation $\nat \pi{f \of \pi_A}{g \of \pi_B}$ is given by $\pi_{(a, x, b)} = x$.
	
	Comma categories can be generalised in several ways: firstly if $\map fAC$ and $\map gBC$ are $2$-functors then the comma object $f \slash g$ in $\twoCat$, the $2$-category of $2$\ndash categories, $2$-functors and $2$-natural transformations, has as objects and morphisms the same ones as given above, while a cell \mbox{$(u,v) \Rar (u', v')$}, for morphisms $(u,v)$ and $\map{(u', v')}{(a, b)}{(a', b')}$, consists of a pair of cells \mbox{$\cell mu{u'}$} in $A$ and $\cell nv{v'}$ in $B$ that make the diagram below commute. The $2$-natural transformation $\nat \pi{f \of \pi_A}{g \of \pi_B}$ is again given by $\pi_{(a, x, b)} = x$.
	\begin{displaymath}
		\begin{tikzpicture}
			\matrix(m)[math2em, row sep=2.5em, column sep=3em]{fa & gb \\ fa' & gb' \\};
			\path[map]	(m-1-1) edge node[above] {$x$} (m-1-2)
													edge[bend right=33] node[left] {$fu$} (m-2-1)
													edge[bend left=33] node[right] {$fu'$} (m-2-1)
									(m-1-2) edge[bend right=33] node[left] {$gv$} (m-2-2)
													edge[bend left=33] node[right] {$gv'$} (m-2-2)
									(m-2-1) edge node[below] {$x'$} (m-2-2)
									($(m-1-1)!0.5!(m-2-1)+(-0.75em,-0.4em)$) edge[cell] node[above] {$fm$} ($(m-1-1)!0.5!(m-2-1)+(0.75em,-0.4em)$)
                  ($(m-1-2)!0.5!(m-2-2)-(0.75em,0.4em)$) edge[cell] node[above] {$gn$} ($(m-1-2)!0.5!(m-2-2)+(0.75em,-0.4em)$);
		\end{tikzpicture}
	\end{displaymath}

	Secondly if $\C = \psCat\SS$ then $f \slash g$, for two $\SS$-indexed functors $\map fAB$ and $\map gBC$, is given pointwise, that is $(f \slash g)_s = f_s \slash g_s$, where $f_s \slash g_s$ is the comma category of the functors $f_s$ and $g_s$ as given above. In general the comma object of two internal functors $\map fAC$ and $\map gBC$ in any category $\E$ with finite limits can be similarly constructed. We shall not do so but instead treat, in detail, the more general construction of internal double comma objects below.
\end{example}

	We can now give Street's definition of pointwise left Kan extensions that was introduced in \cite[Section 4]{Street74}.
\begin{definition} \label{definition:pointwise left Kan extensions}
	Let $\map dAM$ and $\map jAB$ be morphisms in a $2$-category that has all comma objects. The cell $\cell\eta d{l \of j}$ in the diagram below is said to exhibit $l$ as the \emph{pointwise} left Kan extension of $d$ along $j$ if, for each $\map fCB$, the composition of cells below exhibits $l \of f$ as the left Kan extension of $d \of \pi_A$ along $\pi_C$.
	\begin{equation} \label{diagram:pointwise Kan extension}
		\begin{split}
		\begin{tikzpicture}
			\matrix(m)[math2em, row sep=0.75em, column sep=2.7em]{j \slash f \nc A \nc[-0.5em] \\ \nc \nc M \\ C \nc B \nc \\};
			\path[map]	(m-1-1) edge node[above] {$\pi_A$} (m-1-2)
													edge node[left] {$\pi_C$} (m-3-1)
									(m-1-2) edge node[above right] {$d$} (m-2-3)
													edge node[left] {$j$} (m-3-2)
									(m-3-1) edge node[below] {$f$} (m-3-2)
									(m-3-2) edge node[below right] {$l$} (m-2-3);
			\path				(m-1-2) edge[cell, shorten >=14.5pt, shorten <=14.5pt] node[above left] {$\pi$} (m-3-1)
									($(m-1-2)!0.5!(m-2-3)$) edge[cell, shorten >=9pt, shorten <=9pt] node[below right, yshift=3pt] {$\eta$} (m-3-2);
		\end{tikzpicture}
		\end{split}
	\end{equation}
\end{definition}
\begin{remark}
	Taking $f = \id$ one can show that the cell $\eta$ above exhibits $l$ as the ordinary left Kan extension of $d$ along $j$, see \cite[Corollary 19]{Street74}. The following example shows that, in the case of $\C = \enCat\V$, Street's notion of pointwise left Kan extension is, in general, stronger than Dubuc's notion of enriched left Kan extension, that was given in \exref{example:weighted colimits in V-Prof}. Thus enriched left Kan extensions in $\enProf\V$, as defined in \defref{definition:weighted left Kan extensions}, will in general not be pointwise. In fact, Street remarks at the end of the introduction to \cite{Street74} that ``For $\V = \Set$ and $\V = \2$ the definitions do agree; for $\V = \mathsf{AbGp}$ and $\V = \cat$, they do not.''
\end{remark}
\begin{example} \label{example:stronger pointwise left Kan extension}
	To construct an enriched left Kan extension of $2$-functors that is not pointwise consider the $2$-categories
	\begin{displaymath}
		B = \bigpars{\begin{tikzpicture}[textbaseline, font=\scriptsize]
			\matrix(m)[math2em, column sep=2.5em]{\nc \\};
			\path[map]	(m-1-1) edge[bend left=40] node[above] {$u$} (m-1-2)
													edge[bend right=40] node[below] {$v$} (m-1-2);
			\fill	(m-1-1) circle (1pt) node[left] {$x$}
						(m-1-2) circle (1pt) node[right] {$y$};
			\path	($(m-1-1)!0.5!(m-1-2)+(0,0.8em)$) edge[cell] ($(m-1-1)!0.5!(m-1-2)-(0,0.8em)$);
		\end{tikzpicture}}
		\qquad \text{and} \qquad
		M = \bigpars{\begin{tikzpicture}[textbaseline, font=\scriptsize]
			\matrix(m)[math2em, column sep=2.5em]{\nc \\};
			\path[map]	(m-1-1) edge[bend left=40] node[above] {$u'$} (m-1-2)
													edge[bend right=40] node[below] {$v'$} (m-1-2);
			\fill	(m-1-1) circle (1pt) node[left] {$x'$}
						(m-1-2) circle (1pt) node[right] {$y'$};
		\end{tikzpicture}}.
	\end{displaymath}
	Writing $\map x*B$ and $\map{x'}*M$ for the $2$-functors that map the single object of the terminal $2$-category $*$ to the objects $x$ in $B$ and $x'$ in $M$ respectively, we claim that the enriched left Kan extension $\map lBM$ of $x'$ along $x$ is the constant $2$-functor mapping all of $B$ to $x'$. Indeed since $*$ has no non-identity morphisms the coends in \eqref{equation:weighted colimit in terms of coends of copowers} reduce to $l(x) = x'$ and $l(y) = \2 \tens x'$, where $\2 = (\bot \to \top)$. It is easily checked that the latter copower exists and is equal to $x'$. Notice that $l \of x = x'$; the unit $\eta$ exhibiting $l$ is the identity transformation for $x'$.
	
	To show that $l$ fails to be pointwise take $\map{f = y}*B$ in \eqref{diagram:pointwise Kan extension}: we claim that the composite of $\pi$ and $\eta$ does not exhibit $l \of y = x'$ as the left Kan extension of $x' \of \pi_A$ along $\pi_C$. To see this notice that the comma object $f \slash g$ (as given in \exref{example:comma objects}) is the discrete $2$-category consisting of the triples $(*, \map uxy, *)$ and $(*, \map vxy, *)$, so that the assignments $(*, u, *) \mapsto u'$ and $(*, v, *) \mapsto v'$ trivially form a $2$-natural transformation $\nat\phi{x' \of \pi_A}{y' \of \pi_C}$. However both possible candidates $\nat{\phi'}{x' = l \of y}{y'}$, for a factorisation of $\phi$ through \eqref{diagram:pointwise Kan extension}, consist of a single map in $M$ that is either $u'$ or $v'$ so that, precomposed with \eqref{diagram:pointwise Kan extension}, both $\phi'$ do not equal $\phi$. We conclude that $l$ is a left Kan extension that is not pointwise.
	
	If a pointwise left Kan extension of $x$ and $x'$ exists then it coincides with $l$ (indeed, both are in particular ordinary left Kan extensions of $x$ and $x'$, which are unique up to isomorphism), so we can conclude that the pointwise left Kan extension of $x$ and $x'$ does not exist.
\end{example}
	
	The following definition introduces comma objects of a horizontal and vertical morphism, in a double category. It is one of several possible definitions, the strongest of which is used by Grandis and Par\'e in \cite[Section 3.2]{Grandis-Pare08}, to define pointwise left Kan extensions in double categories; details are given in \remref{remark:pointwise weighted colimits}.
\begin{definition} \label{definition:double comma object}
	Given a horizontal morphism $\hmap JAB$ and a vertical morphism $\map fCB$ in a pseudo double category $\K$, the \emph{(vertical) double comma object} $J \slash f$ of $J$ and $f$ consists of an object $J \slash f$ equipped with projections $\pi_A$ and $\pi_C$ as well as a cell $\pi$, as in the diagram below, satisfying the following $1$-dimensional and $2$-dimensional universal properties.
	\begin{displaymath}
		\begin{tikzpicture}
			\matrix(m)[math175em]{J \slash f & J \slash f \\ & C \\ A & B \\};
			\path[map]	(m-1-1) edge node[left] {$\pi_A$} (m-3-1)
									(m-1-2)	edge node[right] {$\pi_C$} (m-2-2)
									(m-2-2) edge node[right] {$f$} (m-3-2)
									(m-3-1) edge[barred] node[below] {$J$} (m-3-2);
			\path[transform canvas={shift={($(m-2-2)!0.5!(m-3-2)$)}}] (m-1-1) edge[cell] node[right] {$\pi$} (m-2-1);
			\path				(m-1-1) edge[eq] (m-1-2);
		\end{tikzpicture}
	\end{displaymath}
	Given any other cell $\phi$ in $\K$ as on the left below, the $1$-dimensional property states that there exists a unique vertical morphism $\map{\phi'}X{J \slash f}$ such that $\pi_A \of \phi' = \phi_A$, \mbox{$\pi_C \of \phi' = \phi_C$} and $\pi \of U_{\phi'} = \phi$.
	\begin{displaymath}
		\begin{tikzpicture}[baseline]
			\matrix(m)[math175em]{X & X \\ \phantom C & C \\ A & B \\};
			\path[map]	(m-1-1) edge node[left] {$\phi_A$} (m-3-1)
									(m-1-2)	edge node[right] {$\phi_C$} (m-2-2)
									(m-2-2) edge node[right] {$f$} (m-3-2)
									(m-3-1) edge[barred] node[below] {$J$} (m-3-2);
			\path[transform canvas={shift={($(m-2-2)!0.5!(m-3-2)$)}}] (m-1-1) edge[cell] node[right] {$\phi$} (m-2-1);
			\path				(m-1-1) edge[eq] (m-1-2);
		\end{tikzpicture}
		\qquad\qquad\begin{tikzpicture}[baseline]
			\matrix(m)[math175em]{Y & Y \\ & C \\ A & B \\};
			\path[map]	(m-1-1) edge node[left] {$\psi_A$} (m-3-1)
									(m-1-2)	edge node[right] {$\psi_C$} (m-2-2)
									(m-2-2) edge node[right] {$f$} (m-3-2)
									(m-3-1) edge[barred] node[below] {$J$} (m-3-2);
			\path[transform canvas={shift={($(m-2-2)!0.5!(m-3-2)$)}}] (m-1-1) edge[cell] node[right] {$\psi$} (m-2-1);
			\path				(m-1-1) edge[eq] (m-1-2);
		\end{tikzpicture}
		\qquad\qquad\begin{tikzpicture}[baseline]
			\matrix(m)[math175em]{X & Y \\ J \slash f & J \slash f \\};
			\path[map]	(m-1-1) edge[barred] node[above] {$K$} (m-1-2)
													edge node[left] {$\phi'$} (m-2-1)
									(m-1-2) edge node[right] {$\psi'$} (m-2-2);
			\path				(m-2-1) edge[eq] (m-2-2);
			\path[transform canvas={shift=($(m-1-2)!0.5!(m-2-2)$)}] (m-1-1) edge[cell] node[right] {$\xi'$} (m-2-1);
		\end{tikzpicture}
	\end{displaymath}
	The $2$-dimensional property is the following. Suppose we are given a further cell $\psi$ as in the middle above, which factors through $\pi$ as $\map{\psi'}Y{J \slash f}$, like $\phi$ factors as $\phi'$. Then for any horizontal map $\hmap KXY$ and any pair of cells $\xi_A$ and $\xi_C$, such that the identity below holds, there exists a unique cell $\xi'$, as on the right above, such that $U_{\pi_A} \of \xi' = \xi_A$ and $U_{\pi_C} \of \xi' = \xi_C$.
	\begin{displaymath}
		\begin{tikzpicture}[textbaseline]
			\matrix(m)[math175em]{X & Y & Y \\ \phantom Y & & C \\ A & A & B \\};
			\path[map]	(m-1-1) edge[barred] node[above] {$K$} (m-1-2)
													edge node[left] {$\phi_A$} (m-3-1)
									(m-1-2) edge node[right] {$\psi_A$} (m-3-2)
									(m-1-3) edge node[right] {$\psi_C$} (m-2-3)
									(m-2-3) edge node[right] {$f$} (m-3-3)
									(m-3-2) edge[barred] node[below] {$J$} (m-3-3);
			\path				(m-1-2) edge[eq] (m-1-3)
									(m-3-1) edge[eq] (m-3-2);
			\path[transform canvas={shift={($(m-2-3)!0.5!(m-3-2)$)}}]	(m-1-1) edge[cell] node[right] {$\xi_A$} (m-2-1);
			\path[transform canvas={shift={($(m-2-3)!0.5!(m-3-2)$)}, xshift=2pt}]	(m-1-2) edge[cell] node[right] {$\psi$} (m-2-2);
		\end{tikzpicture}
		= \begin{tikzpicture}[textbaseline]
			\matrix(m)[math175em]{X & X & Y \\ \phantom Y & C & C \\ A & B & B \\};
			\path[map]	(m-1-1) edge node[left] {$\phi_A$} (m-3-1)
									(m-1-2) edge[barred] node[above] {$K$} (m-1-3)
													edge node[left] {$\phi_C$} (m-2-2)
									(m-1-3) edge node[right] {$\psi_C$} (m-2-3)
									(m-2-2) edge node[right] {$f$} (m-3-2)
									(m-2-3) edge node[right] {$f$} (m-3-3)
									(m-3-1) edge[barred] node[below] {$J$} (m-3-2);
			\path				(m-1-1) edge[eq] (m-1-2)
									(m-2-2) edge[eq] (m-2-3)
									(m-3-2) edge[eq] (m-3-3);
			\path[transform canvas={shift={($(m-2-3)!0.5!(m-3-2)$)}}]	(m-1-1) edge[cell] node[right] {$\phi$} (m-2-1);
			\path[transform canvas={shift={($(m-2-2)!0.5!(m-2-3)$)}}]	(m-1-2) edge[cell] node[right] {$\xi_C$} (m-2-2);
		\end{tikzpicture}
	\end{displaymath}
\end{definition}
	Just like comma objects can be defined as weighted limits, the vertical double comma objects $J \slash f$ is the `double limit' of the functor $\map DI\K$, where $I$ is the double category $0 \slashedrightarrow 1 \leftarrow 2$, that maps onto $J$ and $f$. Analogous to the definition of classical limits, the double limit of $D$ is defined in \cite[Section 4]{Grandis-Pare99} as the `double terminal object' of the `double comma category' $\Delta \slash D$, where $\map\Delta \K{\K^I}$ is the diagonal `double functor' and $\map D *{\K^I}$ is the constant functor on $D \in \K^I$. Note that Grandis and Par\'e draw horizontal morphisms (i.e.\ promorphisms) vertically and vertical morphisms horizontally. We discuss the dual notion, that of `double colimit', in more detail in \secref{section:applications of the main theorem}.
\begin{example} \label{example:double comma objects}
	The construction of double comma objects is similar to that of comma objects, as given in \exref{example:comma objects}. For example, the double comma object $J \slash f$ of an (unenriched) profunctor $\hmap JAB$ and a functor $\map fCB$ is the category with triples $(a, \map ja{fc}, c)$, where $j \in J(a, fc)$, as objects and pairs $\map{(x, y)}{(a, j, c)}{(a', j', c')}$, that make the diagram below commute in $J$, as morphisms.
	\begin{equation} \label{diagram:morphism in a double comma category}
		\begin{split}
		\begin{tikzpicture}
			\matrix(m)[math175em]{ a \nc fc \\ a' \nc fc' \\ };
			\path[map]	(m-1-1) edge node[above] {$j$} (m-1-2)
													edge node[left] {$x$} (m-2-1)
									(m-1-2) edge node[right] {$fy$} (m-2-2)
									(m-2-1) edge node[below] {$j'$} (m-2-2);
		\end{tikzpicture}
		\end{split}
	\end{equation}
	The transformation $\nat\pi{U_{J \slash f}}{J(\pi_A, f \of \pi_C)}$ maps each morphism $(x, y)$ to the diagonal $j' \of x = fy \of j$ of the square above. To see that $\pi$ satisfies the $1$\ndash dimensional universal property consider a second transformation $\nat\phi{U_X}{J(\phi_A, f \of \phi_B)}$; it factors uniquely through $\pi$ as the functor $\map{\phi'}X{J \slash f}$ that is given by $\phi'(v) = \bigpars{\phi_A(v), \phi_{(v,v)}(\id_v), \phi_C(v)}$ on objects, with $\map{\phi_{(v,v)}(\id_v)}{\phi_A v}{f \of \phi_C(v)}$ in $J$, and by $\phi'(s) = (\phi_A s, \phi_B s)$ on morphisms $\map suv$.
	
	Moreover given a further natural transformation $\nat\psi{U_Y}{J(\psi_A, f \of \psi_B)}$ that factors as $\map{\psi'}Y{J \slash f}$, as well as transformations $\nat{\xi_A}K{A(\phi_A, \psi_A)}$ and $\nat{\xi_C}K{C(\phi_C, \psi_C)}$, as in the definition above, the transformation $\nat{\xi'}K{J \slash f(\phi', \psi')}$, that is unique such that $U_{\pi_A} \of \xi' = \xi_A$ and $U_{\pi_C} \of \xi' = \xi_C$, can be given by mapping each morphism $\map kuv$ in $K(u,v)$ to the morphism $(\xi_A k, \xi_B k)$ in $J \slash f (\phi' u, \psi' v)$. This shows that $\pi$ satisfies the $2$-dimensional universal property as well.
	
	If $J$ is a $2$-profunctor (i.e.\ $\cat$-enriched) and $f$ is a $2$-functor, then the double comma object $J \slash f$ has again the same objects and morphisms as the ones given above, while a cell $(x, y) \Rightarrow (x', y')$, consists of a pair of cells $\cell mx{x'}$ in $A$ and $\cell ny{y'}$ in $C$, so that the diagram below commutes in $J$.
	\begin{displaymath}
		\begin{tikzpicture}
			\matrix(m)[math2em, row sep=2.5em, column sep=3em]{a & fc \\ a' & fc' \\};
			\path[map]	(m-1-1) edge node[above] {$j$} (m-1-2)
													edge[bend right=33] node[left] {$x$} (m-2-1)
													edge[bend left=33] node[right] {$x'$} (m-2-1)
									(m-1-2) edge[bend right=33] node[left] {$fy$} (m-2-2)
													edge[bend left=33] node[right] {$fy'$} (m-2-2)
									(m-2-1) edge node[below] {$j'$} (m-2-2)
									($(m-1-1)!0.5!(m-2-1)+(-0.75em,-0.2em)$) edge[cell] node[above] {$m$} ($(m-1-1)!0.5!(m-2-1)+(0.75em,-0.2em)$)
                  ($(m-1-2)!0.5!(m-2-2)-(0.75em,0.4em)$) edge[cell] node[above] {$fn$} ($(m-1-2)!0.5!(m-2-2)+(0.75em,-0.4em)$);
		\end{tikzpicture}
	\end{displaymath}
\end{example}

	We now turn to constructing the double comma object $J \slash f$ of an internal profunctor $\hmap JAB$ and an internal functor $\map fCB$ in a category $\E$ with finite limits, leaving the details to \appref{appendix:double comma object proofs}. As its $\E$-object of objects take the wide pullback $(J \slash f)_0 = A_0 \times_{A_0} J \times_{B_0} C_0$ of the diagram $A_0 \xrar{\id} A_0 \xlar{d_0} J \xrar{d_1} B_0 \xlar{f_0} C_0$ and define its object of morphisms $J \slash f$ to be the pullback
	\begin{equation} \label{diagram:morphism object of internal double comma object}
		\begin{split}
		\begin{tikzpicture}
			\matrix(m)[math2em, column sep=0.75em]{J \slash f \nc (J \slash f)_0 \times_{C_0} C \\ A \times_{A_0} (J \slash f)_0 \nc (J \slash f)_0, \\};
			\path[map]	(m-1-1)	edge (m-1-2)
													edge (m-2-1)
									(m-1-2) edge node[right] {$r$} (m-2-2)
									(m-2-1) edge node[below] {$l$} (m-2-2);
			\coordinate (hook) at ($(m-1-1)+(0.4,-0.4)$);
			\draw (hook)+(0,0.17) -- (hook) -- +(-0.17,0);
		\end{tikzpicture}
		\end{split}
	\end{equation}
	where the corners are pullbacks of $(J \slash f)_0 \to C_0 \xlar{d_0} C$ and $A \xrar{d_1} A_0 \leftarrow (J \slash f)_0$ respectively, and the maps $r$ and $l$ are the compositions
	\begin{flalign} \label{diagram:actions on the double comma object}
		&& (J \slash f)_0 \times_{C_0} C &\to A_0 \times_{A_0} J \times_{B_0} B \times_{C_0} C_0 \xrar{\id \times r \times \id} (J \slash f)_0 & \notag \\
		\text{and} && A \times_{A_0} (J \slash f)_0 &\iso A_0 \times_{A_0} A \times_{A_0} J \times_{B_0} C_0 \xrar{\id \times l \times \id} (J \slash f)_0. & 
	\end{flalign}
	Here the unlabelled map is given by applying $f$ to $C$ after using $C_0 \times_{C_0} C \iso C \times_{C_0} C_0$, while the isomorphism is induced by $A \times_{A_0} A_0 \iso A_0 \times_{A_0} A$ and the maps $l$ and $r$ denote the actions of $A$ and $B$ on $J$. Notice that $(J \slash f)_0$ and $J \slash f$ are the internal versions of the sets of objects and morphisms that define the comma category of ordinary functors, as given above. Indeed $(J \slash f)_0$ is the internal version of the set of triples $(a, \map ja{fc}, c)$, while $J \slash f$ is the $\E$-object of commutative squares of the form \eqref{diagram:morphism in a double comma category}. The source and target maps \mbox{$\map{d = (d_0, d_1)}{J \slash f}{(J \slash f)_0 \times (J \slash f)_0}$} are given by the projections
	\begin{flalign*}
		&& d_0 &= \bigbrks{J \slash f \to (J \slash f)_0 \times_{C_0} C \to (J \slash f)_0} & \\
		\text{and} && d_1 & = \bigbrks{J \slash f \to A \times_{A_0} (J \slash f)_0 \to (J \slash f)_0}. &
	\end{flalign*}
	So far we have defined an $\E$-span $\map d{f \slash g}{(f \slash g)_0 \times (f \slash g)_0}$. That it can be made into an internal category that is the double comma object of $J$ and $f$ is asserted by the following proposition, whose proof is given in \appref{appendix:double comma object proofs}.
\begin{proposition} \label{internal double comma objects}
	The $\E$-span $\map d{J \slash f}{(J \slash f)_0 \times (J \slash f)_0}$ above can be given the structure of a category internal to $\E$. Moreover, $J \slash f$ is made into the double comma object of $J$ and $f$ in $\inProf\E$ by the internal functors $\map{\pi_A}{J \slash f}A$, given by the projections
	\begin{displaymath}
		\map{(\pi_A)_0}{(J \slash f)_0}{A_0} \qquad \text{and} \qquad \pi_A = \bigbrks{J \slash f \to A \times_{A_0} (J \slash f)_0 \to A},
	\end{displaymath}
	and $\map{\pi_C}{J \slash f}C$, given similarly, together with the cell
	\begin{displaymath}
		\begin{tikzpicture}
			\matrix(m)[math175em]{J \slash f & J \slash f \\ & C \\ A & B \\};
			\path[map]	(m-1-1) edge node[left] {$\pi_A$} (m-3-1)
									(m-1-2)	edge node[right] {$\pi_C$} (m-2-2)
									(m-2-2) edge node[right] {$f$} (m-3-2)
									(m-3-1) edge[barred] node[below] {$J$} (m-3-2);
			\path[transform canvas={shift={($(m-2-2)!0.5!(m-3-2)$)}}] (m-1-1) edge[cell] node[right] {$\pi$} (m-2-1);
			\path				(m-1-1) edge[eq] (m-1-2);
		\end{tikzpicture}
	\end{displaymath}
	of internal profunctors that, under the correspondence of \propref{internal transformations}, is given by the projection $(J \slash f)_0 \to J$.
\end{proposition}
\begin{remark} \label{remark:comma object and cartesian filler}
	Remember from \exref{example:internal profunctor equipments} that the restriction $J(\id, f)$ is given by the pushout of the maps $\map dJ{A_0 \times B_0}$ and $\map{\id \times f_0}{A_0 \times C_0}{A_0 \times B_0}$. It will be useful to notice that we can take $J(\id, f) = (J \slash f)_0$ as this pullback, such that the cartesian filler $J(\id, f) \natarrow J$ is simply the projection $(J \slash f)_0 \to J$. It is easily checked that with this choice the actions of $A$ and $C$ (as given in the proof of \propref{bimodule equipments}), that make $J(\id, f)$ into an internal profunctor, are the maps $l$ and $r$ given in \eqref{diagram:actions on the double comma object}, that were used to define $J \slash f$.
\end{remark}
	
	Double comma objects of companions are closely related to comma objects, as follows.
\begin{proposition} \label{double comma objects of companions}
	Let $\map jAB$ and $\map fCB$ be morphisms in an equipment $\K$. If the cell
	\begin{displaymath}
		\begin{tikzpicture}
			\matrix(m)[math175em]{X & X \\ & C \\ A & B \\};
			\path[map]	(m-1-1) edge node[left] {$\pi_A$} (m-3-1)
									(m-1-2) edge node[right] {$\pi_C$} (m-2-2)
									(m-2-2) edge node[right] {$f$} (m-3-2)
									(m-3-1) edge[barred] node[below] {$B(j, \id)$} (m-3-2);
			\path				(m-1-1) edge[eq] (m-1-2);
			\path[transform canvas={shift={($(m-2-2)!0.5!(m-3-2)$)}}] (m-1-1) edge[cell] node[right] {$\pi$} (m-2-1);
		\end{tikzpicture}
	\end{displaymath}
	defines $X$ as the double comma object of $B(j,\id)$ and $f$ then the vertical cell $\ls j\eps \of \pi$ defines $X$ as the comma object of $j$ and $f$ in the vertical $2$-category $V(\K)$. The converse holds whenever the double comma object $B(j, \id) \slash f$ is known to exist.
\end{proposition}
\begin{proof}
	This follows from the correspondence between cells of the form
	\begin{displaymath}
		\begin{tikzpicture}[textbaseline]
			\matrix(m)[math175em]{Y & Y \\ & C \\ A & B \\};
			\path[map]	(m-1-1) edge node[left] {$\phi_A$} (m-3-1)
									(m-1-2) edge node[right] {$\phi_C$} (m-2-2)
									(m-2-2) edge node[right] {$f$} (m-3-2)
									(m-3-1) edge[barred] node[below] {$B(j, \id)$} (m-3-2);
			\path				(m-1-1) edge[eq] (m-1-2);
			\path[transform canvas={shift={($(m-2-2)!0.5!(m-3-2)$)}}] (m-1-1) edge[cell] node[right] {$\phi$} (m-2-1);
		\end{tikzpicture}
		\qquad\text{and}\qquad\begin{tikzpicture}[textbaseline]
			\matrix(m)[math175em]{Y & Y \\ A & C \\ B & B, \\};
			\path[map]	(m-1-1) edge node[left] {$\phi_A$} (m-2-1)
									(m-1-2) edge node[right] {$\phi_C$} (m-2-2)
									(m-2-1) edge node[left] {$j$} (m-3-1)
									(m-2-2) edge node[right] {$f$} (m-3-2);
			\path				(m-3-1) edge[eq] (m-3-2)
									(m-1-1) edge[eq] (m-1-2);
			\path[transform canvas={shift={($(m-2-2)!0.5!(m-3-2)$)}}] (m-1-1) edge[cell] node[right] {$\psi$} (m-2-1);
		\end{tikzpicture}
	\end{displaymath}
	that is given by $\phi \mapsto \ls j\eps \of \phi$ and $\psi \mapsto (\ls j\eta \of U_{\phi_A}) \hc \psi$, and that is natural with respect to precomposition with vertical cells and horizontal composition with vertical cells \mbox{$h \Rar \phi_A$} or $\phi_C \Rar k$, where $\map hYA$ and $\map kYC$. It follows that the $1$\ndash dimensional property of $\pi$ coincides with that of $\ls j\eps \of \pi$, while the $2$-dimensional property of $\pi$ coincides with that of $\ls j\eps \of \pi$ when restricted to cells $\xi_A$ and $\xi_C$ with $K = U_X$ (see \defref{definition:double comma object}). The implication follows, and its converse follows from the fact that double comma objects are uniquely determined, up to isomorphism, by their $1$-dimensional property alone.
\end{proof}

	Following Grandis and Par\'e \cite[Section 4.1]{Grandis-Pare08} we generalise Street's notion of pointwise left Kan extensions (\defref{definition:pointwise left Kan extensions}) to pointwise weighted colimits. The remark below explains the differences between our and Grandis and Par\'e's definition. Suppose that the cell $\pi$ defines the double comma object $J \slash f$ of morphisms $\map fCB$ and $\hmap JAB$, in a pseudo double category $\K$. If the companion of $\pi_C$ exists in $\K$, and is defined by the cartesian filler $\ls{\pi_C}\eps$, then we will denote the composite $\pi \hc (U_f \of \ls{\pi_C}\eps)$ again by $\pi$; it is of the form
	\begin{displaymath}
		\begin{tikzpicture}
			\matrix(m)[math175em, column sep=2em]{J \slash f & C \\ A & B. \\[-3.25em] \phantom{J\slash f} & \phantom{J\slash f} \\};
			\path[map]	(m-1-1) edge[barred] node[above] {$C(\pi_C, \id)$} (m-1-2)
													edge node[left] {$\pi_A$} (m-2-1)
									(m-1-2) edge node[right] {$f$} (m-2-2)
									(m-2-1) edge[barred] node[below] {$J$} (m-2-2);
			\path[transform canvas={shift=($(m-1-2)!0.5!(m-2-2)$)}]	(m-1-1) edge[cell] node[right] {$\pi$} (m-2-1);
		\end{tikzpicture}
	\end{displaymath}
\begin{definition} \label{definition:pointwise weighted colimits}
	Consider morphisms $\map dAM$ and $\hmap JAB$ in a pseudo double category that has companions and double comma objects. The cell $\eta$ in the diagram below exhibits $l$ as the \emph{pointwise} $J$-weighted colimit if, for each $\map fCB$, the composition below exhibits $l \of f$ as the $C(\pi_C, \id)$-weighted colimit of $d \of \pi_A$.
	\begin{equation} \label{diagram:pointwise weighted colimit}
		\begin{split}
		\begin{tikzpicture}
			\matrix(m)[math175em, column sep=2em]
			{ J \slash f \nc C \\ A \nc B \\ M \nc M \\[-3.25em] \phantom{J \slash f} \nc \phantom{J \slash f} \\ };
			\path[map]	(m-1-1) edge[barred] node[above] {$C(\pi_C, \id)$} (m-1-2)
													edge node[left] {$\pi_A$} (m-2-1)
									(m-1-2) edge node[right] {$f$} (m-2-2)
									(m-2-1) edge[barred] node[below] {$J$} (m-2-2)
													edge node[left] {$d$} (m-3-1)
									(m-2-2) edge node[right] {$l$} (m-3-2);
			\path				(m-3-1) edge[eq] (m-3-2);
			\path[transform canvas={shift=(m-2-2)}] (m-1-1) edge[cell] node[right] {$\pi$} (m-2-1);
			\path[transform canvas={shift=(m-2-2), yshift=-2pt}]	(m-2-1) edge[cell] node[right] {$\eta$} (m-3-1);
		\end{tikzpicture}
		\end{split}
	\end{equation}
	In terms of \defref{definition:weighted left Kan extensions} this means that $l \of f$ is the weighted left Kan extension of $d \of \pi_A$ along $\pi_C$. Also we will, when $J = B(j, \id)$ for $\map jAB$, say that $\eta$ exhibits $l$ as the \emph{pointwise} weighted left Kan extension of $d$ along $j$.
\end{definition}
\begin{remark} \label{remark:pointwise weighted colimits}
	Grandis and Par\'e use in their definition of pointwise left Kan extensions in strict double categories the following stronger notion of double comma objects $J \slash f$, that is given in \cite[Section 3.2]{Grandis-Pare08}. If the cell $\pi$ defines $J \slash f$ then, besides asking that $\pi$ satisfies the $1$-dimensional `vertical' universal property of \defref{definition:double comma object}, they ask the corresponding cell $\pi' = \ls{\pi_A}\eta \hc \pi \hc (U_f \of \ls{\pi_C}\eps)$ below to satisfy the following $1$-dimensional `horizontal' universal property: each cell $\phi$ below factors uniquely through $\pi'$ as a horizontal morphism $\hmap{\phi'}X{J \slash f}$, such that $\phi' \hc A(\pi_A, \id) = \phi_A$, $\phi' \hc C(\pi_C, \id) = \phi_C$ and
	\begin{displaymath}
		\begin{tikzpicture}[textbaseline]
			\matrix(m)[math175em]{ X & \phantom A & C \\ X & A & B \\ };
			\path[map]	(m-1-1) edge[barred] node[above] {$\phi_C$} (m-1-3)
									(m-1-3) edge node[right] {$f$} (m-2-3)
									(m-2-1) edge[barred] node[below] {$\phi_A$} (m-2-2)
									(m-2-2) edge[barred] node[below] {$J$} (m-2-3);
			\path				(m-1-1) edge[eq] (m-2-1)
									(m-1-2) edge[cell] node[right] {$\phi$} (m-2-2);
		\end{tikzpicture}
		= \begin{tikzpicture}[textbaseline]
			\matrix(m)[math175em]{ X & J \slash f & \phantom A & C \\ X & J \slash f & A & B. \\ };
			\path[map]	(m-1-1) edge[barred] node[above] {$\phi'$} (m-1-2)
									(m-1-2) edge[barred] node[above] {$C(\pi_C, \id)$} (m-1-4)
									(m-1-4) edge node[right] {$f$} (m-2-4)
									(m-2-1) edge[barred] node[below] {$\phi'$} (m-2-2)
									(m-2-2) edge[barred] node[below] {$A(\pi_A, \id)$} (m-2-3)
									(m-2-3) edge[barred] node[below] {$J$} (m-2-4);
			\path				(m-1-1) edge[eq] (m-2-1)
									(m-1-2) edge[eq] (m-2-2)
									(m-1-3) edge[cell] node[right] {$\pi'$} (m-2-3);
		\end{tikzpicture}
	\end{displaymath}
	Similarly there is a `horizontal' counterpart to the `vertical' $2$-dimensional property of $J \slash f$ given in \defref{definition:double comma object}, and the double comma objects used by Grandis and Par\'e satisfy a `global' $2$-dimensional property, that combines both the horizontal and vertical directions. They remark that these are very strong conditions and write that one ``\dots can only expect [their] comma objects to exist in very particular double categories, having some sort of symmetry in themselves \dots''. Their main example is the double category $\mathsf{Dbl}$, that has pseudo double categories as objects, lax functors (\defref{definition:lax functor}) as vertical morphisms and colax functors (see \secref{section:monad examples}) as horizontal morphisms, see \cite[Section 1.4]{Grandis-Pare08}. On the other hand, the double comma objects in the typical equipment $\Prof$ of profunctors seem not to satisfy the horizontal $1$-dimensional property above.
	
	The second difference is that Grandis and Par\'e ask the composite \eqref{diagram:pointwise weighted colimit} to exhibit $l \of f$ just as an ordinary extension of $d \of \pi_A$, whereas we ask it to be a top absolute extension, which fits in better with our notion of weighted colimits (\propref{weighted colimits equivalent to top absolute left Kan extensions}). We conclude that our notion of pointwise weighted colimits can be applied to more double categories than that of Grandis and Par\'e, because our requirements on double comma objects are weaker, and that it is stronger, because we ask the extension $l \of f$ to be top absolute. A further consequence of our weaker double comma objects is that it does not seem possible to prove that, in general, every pointwise weighted colimit is a weighted colimit. In the case of Grandis and Par\'e's definition this does hold, as is shown in \cite[Theorem 4.2]{Grandis-Pare08}, whose proof uses the horizontal $1$-dimensional property of their comma objects. We will however be able to prove such a result in the case of the equipments $\inProf\E$. In fact the main result (\thmref{weighted colimits are pointwise if double commas are strong}) of this section implies that in $\inProf\E$ the notion of pointwise weighted colimits (\defref{definition:pointwise weighted colimits}) coincides with that of weighted colimits (\defref{definition:weighted colimits and limits}).
\end{remark}
	
	The following is a consequence of \propref{colimits weighted by companions} and \propref{double comma objects of companions}.
\begin{proposition}
	Let $\K$ be a pseudo double category that has companions and double comma objects. If the cell
	\begin{displaymath}
		\begin{tikzpicture}
			\matrix(m)[math175em]{A & B \\ M & M \\};
			\path[map]	(m-1-1) edge[barred] node[above] {$B(j, \id)$} (m-1-2)
													edge node[left] {$d$} (m-2-1)
									(m-1-2) edge node[right] {$l$} (m-2-2);
			\path				(m-2-1) edge[eq] (m-2-2);
			\path[transform canvas={shift=($(m-1-2)!0.5!(m-2-2)$)}] (m-1-1) edge[cell] node[right] {$\zeta$} (m-2-1);
		\end{tikzpicture}
	\end{displaymath}
	exhibits $l$ as the pointwise weighted left Kan extension of $d$ along $j$ then the vertical composite $\zeta \of \ls j\eta$ exhibits $l$ as the pointwise left Kan extension of $d$ along $j$ in $V(\K)$.
\end{proposition}
\begin{proof}
	Let $\map fCB$ be any morphism and let us write $\zeta' = \zeta \of \ls j\eta$. If the cell $\pi$ defines the double comma object $B(j, \id) \slash f$ then the vertical composite $\pi' = \ls j\eps \of \pi$ defines $B(j, \id) \slash f$ as the comma object of $j$ and $f$ by \propref{double comma objects of companions}, and we claim that the vertical cell $(\zeta' \of U_{\pi_A}) \hc (U_l \of \pi') = \zeta \of \pi$ (i.e.\ \eqref{diagram:pointwise Kan extension} with $\eta = \zeta'$ and $\pi = \pi'$) exhibits $l \of f$ as the left Kan extension of $d \of \pi_A$ along $\pi_C$, which completes the proof. Indeed $(\zeta \of \pi) \hc (U_{l \of f} \of \ls{\pi_C} \eps)$ (i.e.\ \eqref{diagram:pointwise weighted colimit} with $\eta = \zeta$) exhibits $l \of f$ as the weighted left Kan extension of $d \of \pi_A$ along $\pi_C$ by assumption, so that the claim follows from \propref{colimits weighted by companions}.
\end{proof}

	The following definition is the main idea of this chapter. It describes a condition on double comma objects that, when satisfied by all double comma objects, implies that all weighted colimits in an equipment $\K$ are pointwise. Remember that, if
	\begin{displaymath}
		\begin{tikzpicture}
			\matrix(m)[math175em, column sep=2em]{J \slash f & C \\ A & B \\[-3.25em] \phantom{J\slash f} & \phantom{J\slash f} \\};
			\path[map]	(m-1-1) edge[barred] node[above] {$C(\pi_C, \id)$} (m-1-2)
													edge node[left] {$\pi_A$} (m-2-1)
									(m-1-2) edge node[right] {$f$} (m-2-2)
									(m-2-1) edge[barred] node[below] {$J$} (m-2-2);
			\path[transform canvas={shift=($(m-1-2)!0.5!(m-2-2)$)}]	(m-1-1) edge[cell] node[right] {$\pi$} (m-2-1);
		\end{tikzpicture}
	\end{displaymath}
	defines $J \slash f$ as the double comma object of $J$ and $f$ (as in \defref{definition:pointwise weighted colimits}), then we write $\pi_f$ for its factorisation $C(\pi_C, \id) \Rar J(\id, f)$ through the restriction $J(\id, f)$, see \defref{definition:cartesian filler}.
\begin{definition} \label{definition:strong double comma objects}
	Assume that the cell $\pi$ above defines $J \slash f$ as the double comma object of $\hmap JAB$ and $\map fCB$. We will call $J \slash f$ \emph{strong} if every cell $\phi$ below factors uniquely through $\pi_f$ as follows.
	\begin{equation} \label{equation:strong double comma objects}
		\begin{split}
		\begin{tikzpicture}[textbaseline]
			\matrix(m)[math175em]
			{ J \slash f \nc C \nc D \\
				A \nc \nc \\
				M \nc \nc M \\[-3.25em]
				\phantom{J \slash f} \nc \phantom{J \slash f} \nc \phantom{J \slash f} \\ };
			\path[map]	(m-1-1) edge[barred] node[above] {$C(\pi_C, \id)$} (m-1-2)
													edge node[left] {$\pi_A$} (m-2-1)
									(m-1-2) edge[barred] node[above] {$H$} (m-1-3)
									(m-1-3) edge node[right] {$k$} (m-3-3)
									(m-2-1) edge node[left] {$h$} (m-3-1);
			\path				(m-3-1) edge[eq] (m-3-3);
			\path[transform canvas={shift=($(m-2-2)!0.5!(m-3-2)$)}]	(m-1-2) edge[cell] node[right] {$\phi$} (m-2-2);
		\end{tikzpicture}
		= \begin{tikzpicture}[textbaseline]
			\matrix(m)[math175em]
			{ J \slash f \nc C \nc D \\
				A \nc C \nc D \\
				M \nc \nc M \\[-3.25em]
				\phantom{J \slash f} \nc \phantom{J \slash f} \nc \phantom{J \slash f} \\ };
			\path[map]	(m-1-1) edge[barred] node[above] {$C(\pi_C, \id)$} (m-1-2)
													edge node[left] {$\pi_A$} (m-2-1)
									(m-1-2) edge[barred] node[above] {$H$} (m-1-3)
									(m-2-1) edge[barred] node[below] {$J(\id, f)$} (m-2-2)
													edge node[left] {$h$} (m-3-1)
									(m-2-2) edge[barred] node[above] {$H$} (m-2-3)									
									(m-2-3) edge node[right] {$k$} (m-3-3);									
			\path				(m-1-2) edge[eq] (m-2-2)
									(m-1-3) edge[eq] (m-2-3)
									(m-3-1) edge[eq] (m-3-3)
									(m-2-2) edge[cell] node[right] {$\phi'$} (m-3-2);
			\path[transform canvas={shift=($(m-2-2)!0.5!(m-2-3)$)}]	(m-1-1) edge[cell] node[right] {$\pi_f$} (m-2-1);
		\end{tikzpicture}
		\end{split}
	\end{equation}
\end{definition}
\begin{example}
	In $\Prof$ any double comma object $J \slash f$ of a profunctor $\hmap JAB$ and a functor $\map fCB$ is strong. To see this recall that $J \slash f$ has triples $(a, \map ja{fc}, c)$ as objects and pairs $\map{(x,y)}{(a, j, c)}{(a', j', c')}$, with $\map xa{a'}$ and $\map yc{c'}$ such that $fy \of j = j' \of x$, as morphisms, and notice that the natural transformation $\nat{\pi_f}{C(\pi_C, \id)}{J(\pi_A, f)}$ is simply given by
	\begin{displaymath}
		(\pi_f)_{((a, j, c), c')} (\map xc{c'}) = \bigbrks{a \xrar j fc \xrar{fx} fc'}.
	\end{displaymath}
	For a natural transformation $\nat\phi{H(\pi_C, \id)}{M(h \of \pi_A, k)}$ as above, a factorisation $\nat{\phi'}{J(\id, f) \hc_C H}{M(h, k)}$ consists of components
	\begin{displaymath}
		\map{\phi_{(a, d)}}{\int^{c \in C} J(a, fc) \times H(c, d)}{M(ha, kd)}
	\end{displaymath}
	which we take to be induced by the composites
	\begin{displaymath}
		J(a, fc) \times H(c, d) \iso \coprod_{j \in J(a, fc)} H(c,d) \xrar{\coprod_j \phi_{((a, j, c), d)}} M(ha, kd).
	\end{displaymath}
	That these induce a well-defined map on the coend above follows from the naturality of $\phi$ with respect to maps of the form $(\id, x)$ in $J \slash f$, where $\map xc{c'}$. Indeed this naturality means that the diagrams
	\begin{displaymath}
		\begin{tikzpicture}
			\matrix(m)[math175em, column sep=0.5em]{H(c', d) & \phantom{M(ha, kd)} & M(ha, kd) \\ \phantom{M(ha, kd)} & H(c, d) & \\};
			\path[map]	(m-1-1) edge node[above] {$\phi_{((a, fx \of j, c'), d)}$} (m-1-3)
													edge node[below left] {$H(x, d)$} (m-2-2)
									(m-2-2) edge node[below right] {$\phi_{((a, j, c), d)}$} (m-1-3);
		\end{tikzpicture}
	\end{displaymath}
	commute, for every $\map jac$ in $J$. Finally notice that the $((a, j, b), d)$-component of the composite $\nat{\pi_f \hc_C \id_H}{C(\pi_C, \id) \hc_C H}{J(\id, f) \hc_C H}$, under the isomorphism $H(\pi_C, \id) \iso C(\pi_C, \id) \hc_C H$, is simply the composite of insertions
	\begin{displaymath}
		H(b, d) \to J(a, fb) \times H(b,d) \to \int^{c \in C} J(a, fc) \times H(c,d),
	\end{displaymath}
	the first at $j \in J(a, fb)$ and the second at $c = b$. From this it easily follows that $\phi'$ is the unique transformation such that $\phi = \phi' \of (\pi_f \hc_C \id_H)$.
\end{example}
\begin{example}
	To give an example of a comma object in the equipment $\enProf\cat$, of $2$-categories, $2$-functors, $2$-profunctors and $2$-natural transformations, that is not strong, we can again consider the $2$-categories of \exref{example:stronger pointwise left Kan extension}:
	\begin{displaymath}
		B = \bigpars{\begin{tikzpicture}[textbaseline, font=\scriptsize]
			\matrix(m)[math2em, column sep=2.5em]{\nc \\};
			\path[map]	(m-1-1) edge[bend left=40] node[above] {$u$} (m-1-2)
													edge[bend right=40] node[below] {$v$} (m-1-2);
			\fill	(m-1-1) circle (1pt) node[left] {$x$}
						(m-1-2) circle (1pt) node[right] {$y$};
			\path	($(m-1-1)!0.5!(m-1-2)+(0,0.8em)$) edge[cell] ($(m-1-1)!0.5!(m-1-2)-(0,0.8em)$);
		\end{tikzpicture}}
		\qquad \text{and} \qquad
		M = \bigpars{\begin{tikzpicture}[textbaseline, font=\scriptsize]
			\matrix(m)[math2em, column sep=2.5em]{\nc \\};
			\path[map]	(m-1-1) edge[bend left=40] node[above] {$u'$} (m-1-2)
													edge[bend right=40] node[below] {$v'$} (m-1-2);
			\fill	(m-1-1) circle (1pt) node[left] {$x'$}
						(m-1-2) circle (1pt) node[right] {$y'$};
		\end{tikzpicture}}.
	\end{displaymath}
	Taking $A = * = C$ to be the terminal $2$-category, we choose \mbox{$\map{f = y}CB$}, the inclusion that maps the single object of $A = *$ to $y$, and likewise $\map{j = x}AB$, \mbox{$\map{h = x'}AM$} and $\map{k = y'}CM$. Then the double comma object $B(j, \id) \slash f$ (as in \exref{example:double comma objects}) is the discrete $2$-category consisting of the triples \mbox{$(*, \map uxy, *)$} and \mbox{$(*, \map vxy, *)$}, and the assignments $\bigpars{(*, u, *), *} \mapsto u'$ and $\bigpars{(*, v, *),*} \mapsto v'$ trivially form a $2$-natural transformation $\nat\phi{C(\pi_C, \id)}{M(h \of \pi_A, k)}$. However, giving a cell
	\begin{displaymath}
		\begin{tikzpicture}
			\matrix(m)[math175em]{A & B \\ M & M \\};
			\path[map]	(m-1-1) edge[barred] node[above] {$B(j, f)$} (m-1-2)
													edge node[left] {$h$} (m-2-1)
									(m-1-2) edge node[right] {$k$} (m-2-2);
			\path				(m-2-1) edge[eq] (m-2-2)
									(m-1-1) edge[transform canvas={shift=($(m-1-2)!0.5!(m-2-2)$)}, cell] node[right] {$\phi'$} (m-2-1);
		\end{tikzpicture}
	\end{displaymath}
	is equivalent to giving a functor $\map{\phi'}{B(x, y)}{M(x', y')}$, and $\phi' \of \pi_f = \phi$ implies $\phi' u = u'$ and $\phi' v = v'$. We conclude that such a $\phi'$ cannot exist, since there is no cell from $u'$ to $v'$ in $M$.
\end{example}
\begin{proposition} \label{internal double comma objects are strong}
	Let $\E$ be a category with finite limits and reflexive coequalisers preserved by pullbacks, so that internal profunctors in $\E$ form an equipment $\inProf\E$ by \propref{bimodule equipments}. The double comma objects of $\inProf\E$, that exist by \propref{internal double comma objects}, are all strong.
\end{proposition}
\begin{proof}
	Let $\hmap JAB$ be an internal profunctor and $\map fCB$ be an internal functor, and consider an internal transformation $\nat\phi{H(\pi_C, \id)}M$ as in \eqref{equation:strong double comma objects}, which we consider to be precomposed with the isomorphism $H(\pi_C, \id) \iso C(\pi_C, \id) \hc_C H$ given by \propref{cartesian fillers in terms of companions and conjoints}. We have to show that there exists a unique factorisation $\nat{\phi'}{J(\id, f) \hc_C H}M$ such that \eqref{equation:strong double comma objects} holds.
	
	We begin with computing the horizontal composite
	\begin{equation} \label{equation:366623}
		H(\pi_C, \id) \iso C(\pi_C, \id) \hc_C H \xrar{\pi_f \hc_C \id_H} J(\id, f) \hc_C H,
	\end{equation}
	where the cell $\pi_f$, after being composed with unitors, coincides with the horizontal composite
	\begin{displaymath}
		\begin{tikzpicture}
			\matrix(m)[math175em]
			{	J \slash f & J \slash f & C \\
				A & C & C \\[-0.3em]
				\phantom{J \slash f} \nc \phantom{J \slash f} \nc \phantom{J \slash f} \\ };
			\path[map]	(m-1-1) edge node[left] {$\pi_A$} (m-2-1)
									(m-1-2) edge[barred] node[above] {$C(\pi_C, \id)$} (m-1-3)
													edge node[right] {$\pi_C$} (m-2-2)
									(m-2-1) edge[barred] node[below] {$J(\id, f)$} (m-2-2);
			\path				(m-1-1) edge[eq] (m-1-2)
									(m-1-3) edge[eq] (m-2-3)
									(m-2-2) edge[eq] (m-2-3);
			\path[transform canvas={shift=($(m-2-1)!0.5!(m-2-2)$)}]	(m-1-2) edge[cell] node[right] {$\pi'$} (m-2-2)
									(m-1-3) edge[cell] node[right] {$\ls{\pi_C}\eps$} (m-2-3);
		\end{tikzpicture}
	\end{displaymath}
	where $\pi'$ denotes the factorisation of the cell $\pi$, that defines $J \slash f$, through the restriction $J(\id, f)$, and $\ls{\pi_C}\eps$ is the cartesian filler defining the companion $C(\pi_C, \id)$.
	
	Recall from \remref{remark:comma object and cartesian filler} that we can take $J(\id, f) = (J \slash f)_0$. With that choice the cartesian filler $J(\id, f) \to J$ is the projection $(J \slash f)_0 \to J$ and the factorisation $\nat{\pi'}{J \slash f}{J(\id, f)}$, as a map $(J \slash f)_0 \to (J \slash f)_0$ (see \propref{internal transformations}), is simply the identity. Likewise the companion $C(\pi_C, \id)$ is the pullback of $\map dC{C_0 \times C_0}$ along $\map{(\pi_C)_0 \times \id}{(J \slash f)_0 \times C_0}{C_0 \times C_0}$, so that we can take $C(\pi_C, \id) = (J \slash f)_0 \times_{C_0} C$; it follows that the transformation $\nat{\pi_f = \pi' \hc \ls{\pi_C}\eps}{C(\pi_C, \id)}{J(\id, f)}$ is the right action $\map r{(J \slash f)_0 \times_{C_0} C}{(J \slash f)_0}$ given by \eqref{diagram:actions on the double comma object}. Finally we also take $H(\pi_C, \id) = (J \slash f)_0 \times_{C_0} H$. Recall from the proof of \propref{cartesian fillers in terms of companions and conjoints} that the isomorphism $H(\pi_C, \id) \xrar\iso C(\pi_C, \id) \hc_C H$ is the composite
	\begin{displaymath}
		H(\pi_C, \id) \iso U_{J \slash f} \hc_{J \slash f} H(\pi_C, \id) \xrar{\ls{\pi_C} \eta \hc \eps} C(\pi_C, \id) \hc_C H,
	\end{displaymath}
	where $\ls{\pi_C}\eta$ corresponds to $\map{\bigpars{\id, e \of (\pi_C)_0}}{(J \slash f)_0}{(J \slash f)_0 \times_{C_0} C}$ under \propref{internal transformations} while $\eps$ is the cartesian filler, i.e.\ the projection $(J \slash f)_0 \times_{C_0} H \to H$. Thus, using the remark on horizontal compositions in \propref{internal transformations}, we conclude that the composite \eqref{equation:366623} equals
	\begin{multline*}
		(J \slash f)_0 \times_{C_0} H \xrar{\bigpars{\id, e \of (\pi_C)_0} \times_{C_0} \id_H} \bigpars{(J \slash f)_0 \times_{C_0} C} \times_{C_0} H \\
		\xrar{r \times_{C_0} \id_H} (J \slash f)_0 \times_{C_0} H \to (J \slash f)_0 \hc_C H.
	\end{multline*}
	Here the first two maps cancel by the unit axiom for actions, so that the coequaliser $(J \slash f)_0 \times_{C_0} H \to (J \slash f)_0 \hc_C H$ remains.
	
	It follows that the identity \eqref{equation:strong double comma objects} is equivalent to the commuting of the diagram
	\begin{displaymath}
		\begin{tikzpicture}
			\matrix(m)[math2em]
			{	(J \slash f)_0 \times_{C_0} C \times_{C_0} H & (J \slash f)_0 \times_{C_0} H &[1.5em] (J \slash f)_0 \hc_C H & M \\
				& & A_0 \times D_0 & M_0 \times M_0. \\ };
			\path[map, transform canvas={yshift=3pt}]	(m-1-1)	edge node[above] {$r \times \id$} (m-1-2);
			\path[map, transform canvas={yshift=-3pt}]	(m-1-1) edge node[below] {$\id \times l$} (m-1-2);
			\path[map]	(m-1-2) edge[bend left=30] node[above] {$\phi$} (m-1-4)
													edge node[above] {$\pi_f \hc_C \id_H$} (m-1-3)
													edge node[below left] {$(\pi_A)_0 \times h_D$} (m-2-3)
									(m-1-3) edge node[above] {$\phi'$} (m-1-4)
													edge (m-2-3)
									(m-1-4) edge node[right] {$d$} (m-2-4)
									(m-2-3) edge node[below] {$h_0 \times k_0$} (m-2-4);
		\end{tikzpicture}
	\end{displaymath}
	That is, to prove that every internal transformation $\phi$ induces a unique $\phi'$ that satisfies \eqref{equation:strong double comma objects} we have to show that $\phi$ forks the actions $r \times \id$ and $\id \times l$, that is $\phi \of (r \times \id) = \phi \of (\id \times l)$. But this is a consequence of the naturality of $\phi$, as follows. Its naturality with respect to $J \slash f$ (defined as the pullback \eqref{diagram:morphism object of internal double comma object}) means that the diagram
	\begin{displaymath}
		\begin{tikzpicture}
			\matrix(m)[math2em]
			{ J \slash f \times_{(J \slash f)_0} \bigpars{(J \slash f)_0 \times_{C_0} H} & (J \slash f)_0 \times_{C_0} C \times_{C_0} H & (J \slash f)_0 \times_{C_0} H \\
				A \times_{A_0} \bigpars{(J \slash f)_0 \times_{C_0} H} & M \times_{M_0} M & M \\ };
			\path[map]	(m-1-1) edge (m-1-2)
													edge (m-2-1)
									(m-1-2) edge node[above] {$\id \times_{C_0} l$} (m-1-3)
									(m-1-3) edge node[right] {$\phi$} (m-2-3)
									(m-2-1) edge node[below] {$h \times_{h_0} \phi$} (m-2-2)
									(m-2-2) edge node[below] {$m$} (m-2-3);
		\end{tikzpicture}
	\end{displaymath}
	commutes, where the unlabelled maps are induced by the projections $J \slash f \to (J \slash f)_0 \times_{C_0} C$, $(J \slash f)_0 \times_{C_0} H \to H$ and $J \slash f \to A \times_{A_0} (J \slash f)_0 \to A$ and where $m$ is the composition of $M$. Next consider the map
	\begin{displaymath}
		(J \slash f)_0 \times_{C_0} C \times_{C_0} H \to J \slash f \times_{(J \slash f)_0} \bigpars{(J \slash f)_0 \times_{C_0} H},
	\end{displaymath}
	 that is induced by $(J \slash f)_0 \times_{C_0} C \to J \slash f$, which is in turn induced by the identity on $(J \slash f)_0 \times_{C_0} C$ and the composite $(J \slash f)_0 \times_{C_0} C \xrar r (J \slash f)_0 \xrar{(e \of (\pi_A)_0, \id)} A \times_{A_0} (J \slash f)_0$, the action $\map r{(J \slash f)_0 \times_{C_0} C}{(J \slash f)_0}$ and the identity on $H$. It is easy to see that composing this map with the diagram above gives a commuting square whose top leg is $\phi \of (\id \times_{C_0} l)$ and (using the unit axioms for $h$ and $m$) whose bottom leg is $\phi \of (r \times_{C_0} \id)$, so that $\phi$ forks $\id \times_{C_0} l$ and $r \times_{C_0} \id$. This completes the proof.
\end{proof}

	We close this chapter with its main result.
\begin{theorem} \label{weighted colimits are pointwise if double commas are strong}
	Consider a pseudo double category $\K$ that has companions as well as double comma objects, which are all strong. A cell
	\begin{displaymath}
		\begin{tikzpicture}
			\matrix(m)[math175em]{A & B \\ M & M \\};
			\path[map]	(m-1-1) edge[barred] node[above] {$J$} (m-1-2)
													edge node[left] {$d$} (m-2-1)
									(m-1-2) edge node[right] {$l$} (m-2-2);
			\path				(m-2-1) edge[eq] (m-2-2);
			\path[transform canvas={shift=($(m-1-2)!0.5!(m-2-2)$)}] (m-1-1) edge[cell] node[right] {$\zeta$} (m-2-1);
		\end{tikzpicture}
	\end{displaymath}
	exhibits $l$ as the $J$-weighted colimit of $d$ (\defref{definition:weighted colimits in double categories}) if and only if it exhibits $l$ as the pointwise $J$-weighted colimit of $d$ (\defref{definition:pointwise weighted colimits}).
\end{theorem}
\begin{proof}
	Let $\hmap HCD$, $\map fCB$ and $\map eDM$ be any morphisms in $\K$, and consider cells that are of the form as in the diagram below.
	\begin{displaymath}
		\begin{tikzpicture}
			\matrix(k)[math175em]{J \slash f & C & D \\	A & & \\ M & & M \\[-3.25em] \phantom{J \slash f} & \phantom{J \slash f} & \phantom{J \slash f} \\ };
			\path[map]	(k-1-1) edge[barred] node[above] {$C(\pi_C, \id)$} (k-1-2)
													edge node[left] {$\pi_A$} (k-2-1)
									(k-1-2) edge[barred] node[above] {$H$} (k-1-3)
									(k-1-3) edge node[right] {$e$} (k-3-3)
									(k-2-1) edge node[left] {$d$} (k-3-1);
			\path				(k-3-1) edge[eq] (k-3-3);
			\path[transform canvas={shift=($(k-2-2)!0.5!(k-3-2)$)}]	(k-1-2) edge[cell] node[right] {$\phi$} (k-2-2);
			
			\matrix(l)[math175em, xshift=17.5em]{A & B & C & D \\ M & \phantom M & \phantom M & M \\};
			\path[map]	(l-1-1) edge[barred] node[above] {$J$} (l-1-2)
													edge node[left] {$d$} (l-2-1)
									(l-1-2) edge[barred] node[above] {$B(\id, f)$} (l-1-3)
									(l-1-3) edge[barred] node[above] {$H$} (l-1-4)
									(l-1-4) edge node[right] {$e$} (l-2-4);
			\path	(l-2-1) edge[eq] (l-2-4);
			\path[transform canvas={shift={($(l-1-3)!0.5!(l-2-3)$)}, xshift=-17.5em}]	(l-1-2) edge[cell] node[right] {$\psi$} (l-2-2);
			
			\matrix(m)[math175em, yshift=-12.5em]{C & D \\ B & \\ M & M \\};
			\path[map]	(m-1-1) edge[barred] node[above] {$H$} (m-1-2)
													edge node[left] {$f$} (m-2-1)
									(m-1-2) edge node[right] {$e$} (m-3-2)
									(m-2-1) edge node[left] {$l$} (m-3-1);
			\path (m-3-1) edge[eq] (m-3-2);
			\path[transform canvas={shift={($(m-2-2)!0.5!(m-3-2)$)}, yshift=12.5em}]	(m-1-1) edge[cell] node[right] {$\phi'$} (m-2-1);
			
			\matrix(n)[math175em, xshift=17.5em, yshift=-12.5em]{B & C & D \\ M & \phantom M & M \\};
			\path[map]	(n-1-1) edge[barred] node[above] {$B(\id, f)$} (n-1-2)
													edge node[left] {$l$} (n-2-1)
									(n-1-2) edge[barred] node[above] {$H$} (n-1-3)
									(n-1-3) edge node[right] {$e$} (n-2-3);
			\path	(n-2-1) edge[eq] (n-2-3)
						(n-1-2) edge[cell] node[right] {$\psi'$} (n-2-2);
			
			\path[<->]	(k) edge[shorten >= 2em, shorten <= 2em] (l)
									(m) edge[shorten >= 2em, shorten <= 2em] (n);
			\path[<->, dashed] (k) edge[shorten <= 0.75em, shorten >= 1.25em] (m)
									(l) edge[shorten >= 2em, shorten <= 2em] (n);
		\end{tikzpicture}
	\end{displaymath}
	Here the top double headed arrow denotes the correspondence between cells of the form $\phi$ and $\psi$ that follows from the fact that $J \slash f$ is a strong double comma object: for each $\phi$ there is a unique $\psi$ such that $\phi = \psi \of (\pi \hc \eta_f \hc \id_H)$ (here we have taken $J(\id, f) = J \hc B(\id, f)$ in \defref{definition:strong double comma objects} so that $\pi_f = \pi \hc \eta_f$). Also the bottom arrow denotes a bijective correspondence, that is given by the assignments $\phi' \mapsto (U_l \of \eps_f) \hc \phi'$ and $\psi' \mapsto \psi' \of (\eta_f \hc \id_H)$, where $\eps_f$ and $\eta_f$ are the conjoint cells defining $B(\id, f)$.
	
	Moreover $\zeta$ exhibits $l$ as the pointwise $J$-weighted colimit of $d$ if and only if, for any $f$, there is a correspondence between cells of the form $\phi$ and $\phi'$ as is denoted by the dashed arrow on the left: for each $\phi$ there is a unique $\phi'$ such that $\phi = (\zeta \of \pi) \hc \phi'$. Likewise if $\zeta$ exhibits $l$ as the ordinary $J$-weighted colimit of $d$ then there exists a correspondence between cells of the form $\psi$ and $\psi'$ as on the right: for each $\psi$ there is a unique $\psi'$ such that $\psi = \zeta \hc \psi'$. The latter are in fact equivalent: by taking $f = \id_B$ and vertically precomposing the cells $\psi$ and $\psi'$ with the isomorphisms $H \iso B(\id, \id) \hc H$ we see that, in terms of \propref{weighted colimits equivalent to top absolute left Kan extensions}, $\zeta$ exhibits $l$ as the top absolute left Kan extension of $d$ along $J$.
	
	We claim that the two `dashed' correspondences above are equivalent, thus proving the theorem. To see this, consider pairs of cells $(\phi, \psi)$ and $(\phi', \psi')$ that correspond under the top and bottom correspondence above. Then $\psi = \zeta \hc \psi'$ implies
	\begin{displaymath}
		\begin{tikzpicture}[textbaseline, x=0.7cm, y=0.6cm, font=\scriptsize]
			\draw	(1,1) -- (0,1) -- (0,0) -- (2,0) -- (2,2) -- (1,2) -- (1,0)
						(1,2) -- (0,2) -- (0,1);
			\draw[shift={(0.5,0.5)}]
				(0,0) node {$\zeta$}
				(0,1) node {$\pi$};
			\draw[shift={(0.5,0)}]
				(1,1) node {$\phi'$};
		\end{tikzpicture}
		\quad = \quad \begin{tikzpicture}[textbaseline, x=0.7cm, y=0.6cm, font=\scriptsize]
			\draw	(2,1) -- (2,2) -- (0,2) -- (0,0) -- (3,0) -- (3,1) -- (0,1)
						(1,0) -- (1,2);
			\draw[shift={(0.5,0.5)}]
				(0,0) node {$\zeta$}
				(0,1) node {$\pi$}
				(1,1) node {$\eta_f$};
			\draw[shift={(0,0.5)}]
				(2,0) node {$\psi'$};
		\end{tikzpicture}
		\quad = \quad \begin{tikzpicture}[textbaseline, x=0.7cm, y=0.6cm, font=\scriptsize]
			\draw	(2,1) -- (2,2) -- (0,2) -- (0,0) -- (3,0) -- (3,1) -- (0,1)
						(1,1) -- (1,2);
			\draw[shift={(0.5,0.5)}]
				(0,1) node {$\pi$}
				(1,1) node {$\eta_f$}
				(1,0) node {$\psi$};
		\end{tikzpicture}
		\quad = \quad \begin{tikzpicture}[textbaseline, x=0.7cm, y=0.6cm, font=\scriptsize]
			\draw (0,0) -- (2,0) -- (2,2) -- (0,2) -- (0,0);
			\draw (1,1) node {$\phi$};
		\end{tikzpicture}\,.
	\end{displaymath}
	For the converse assume $\phi = (\zeta \of \pi) \hc \phi'$. Then
	\begin{align*}
		\begin{tikzpicture}[textbaseline, x=0.7cm, y=0.6cm, font=\scriptsize]
			\draw	(1,1) -- (3,1) -- (3,0) -- (0,0) -- (0,1) -- (1,1) -- (1,0)
						(2,1) -- (2,2) -- (0,2) -- (0,1);
			\draw[shift={(0.5,0.5)}]
				(0,0) node {$\zeta$};
			\draw[shift={(0,0.5)}]
				(2,0) node {$\psi'$}
				(1,1) node {$\pi_f$};
		\end{tikzpicture}
		&\quad = \quad \begin{tikzpicture}[textbaseline, x=0.7cm, y=0.6cm, font=\scriptsize]
			\foreach \i in {0,1,2}
				\draw (\i,0) -- (\i,3);
			\draw (0,2) -- (3,2) -- (3,0) -- (0,0)
						(0,1) -- (2,1)
						(0,3) -- (2,3);
			\draw[shift={(0.5,0.5)}]
				(0,0) node {$\zeta$}
				(0,2) node {$\pi$}
				(1,1) node {$\eps_f$}
				(1,2) node {$\eta_f$};
			\draw[shift={(0.5,0)}]
				(2,1) node {$\phi'$};
		\end{tikzpicture}
		\quad = \quad \begin{tikzpicture}[textbaseline, x=0.7cm, y=0.6cm, font=\scriptsize]
			\draw	(1,1) -- (0,1) -- (0,0) -- (2,0) -- (2,2) -- (1,2) -- (1,0)
						(1,2) -- (0,2) -- (0,1);
			\draw[shift={(0.5,0.5)}]
				(0,0) node {$\zeta$}
				(0,1) node {$\pi$};
			\draw[shift={(0.5,0)}]
				(1,1) node {$\phi'$};
		\end{tikzpicture} \\[1em]
		&\quad = \quad \begin{tikzpicture}[textbaseline, x=0.7cm, y=0.6cm, font=\scriptsize]
			\draw (0,0) -- (2,0) -- (2,2) -- (0,2) -- (0,0);
			\draw (1,1) node {$\phi$};
		\end{tikzpicture}
		\quad = \quad \begin{tikzpicture}[textbaseline, x=0.7cm, y=0.6cm, font=\scriptsize]
			\draw	(2,1) -- (2,2) -- (0,2) -- (0,0) -- (3,0) -- (3,1) -- (0,1);
			\draw[shift={(0.5,0.5)}]
				(1,0) node {$\psi$};
			\draw[shift={(0,0.5)}]
				(1,1) node {$\pi_f$};
		\end{tikzpicture}
	\end{align*}
	which, since factorisations through $\pi_f$ are unique, implies $\psi = \zeta \hc \psi'$. This concludes the proof.
\end{proof}

	Combined with \propref{internal double comma objects are strong} we obtain:
\begin{corollary}
	Let $\E$ be a category with finite limits and reflexive coequalisers preserved by pullbacks, so that internal profunctors in $\E$ form an equipment $\inProf\E$ by \propref{bimodule equipments}. The pointwise weighted colimits in $\inProf\E$ coincide with the ordinary weighted colimits.
\end{corollary}

  \chapter{Monads on equipments} \label{chapter:monads on equipments}
	Here we recall the notion of monads on equipments which, analogous to monads on an ordinary category, allow us to describe algebraic structures in equipments. To do so, we will first need to consider functors of equipments and their transformations: these will be recalled in the first section below, and again mostly from \cite[Section 6]{Shulman08}. Following this, \secref{section:monad examples} introduces the two main examples of monads that we will consider: the `free symmetric strict monoidal $\V$-category'-monad $\fsmc$ on $\enProf\V$ and the `free strict double category'\ndash monad $D$ on $\psProf{\GG_1}$, where $\GG_1$ is the category $\pars{0 \rightrightarrows 1}$. Both of them were briefly discussed in the introduction to this thesis. In general the functor underlying a monad $T$, on an equipment $\K$, does need not preserve the horizontal units and horizontal composition of $\K$ strictly. It will be important that the monads $T$ we consider are `normal', meaning that they preserve horizontal units. Many monads, including $\fsmc$ and $D$, are normal, and every normal monad $T$ induces a `vertical $2$-monad' $V(T)$ on the vertical $2$-category $V(\K)$. Given a $2$-monad $S$ on a $2$-category $\C$, we will recall, in \secref{section:monad examples}, the definition of the $2$-category of `pseudo' algebras, $S$\ndash morphisms and $S$-cells in $\C$. Thus, for a normal monad $T$ on $\K$, taking $S = V(T)$ allows us to define $T$-algebras, vertical $T$-morphisms and vertical $T$-cells in $\K$ as the $V(T)$\ndash algebras, $V(T)$-morphisms and $V(T)$-cells in $V(\K)$. Of course in the case $T = \fsmc$ we obtain symmetric monoidal $\V$-categories, symmetric monoidal $\V$-functors and monoidal $\V$\ndash natural transformations, whereas with $T = D$ we recover double categories, functors and transformations.

\section{Functors, transformations and monads}
  We start with the definition of functors.
\begin{definition} \label{definition:lax functor}
  Let $\K$ and $\L$ be pseudo double categories. A \emph{lax functor} $\map F\K\L$ consists of a pair of functors $\map{F_0}{\K_0}{\L_0}$ and $\map{F_1}{\K_1}{\L_1}$, such that $L F_1 = F_0 L$ and $R F_1 = F_0 R$, together with natural transformations
  \begin{displaymath}
    \begin{tikzpicture}
      \matrix(m)[math, column sep=3em, row sep=3em]
        { \K_1 \times_{\K_0} \K_1 & \K_1 & \K_0 \\
          \L_1 \times_{\L_0} \L_1 & \L_1 & \phantom{\L_0} \\};
      \path[map]  (m-1-1) edge node[above] {$\hc_\K$} (m-1-2)
                          edge node[left] {$F_1 \times_{F_0} F_1$} (m-2-1)
                  (m-1-3) edge node[above] {$U$} (m-1-2)
                          edge node[right] {$F_0$} (m-2-3)
                  (m-2-1) edge node[below] {$\hc_\L$} (m-2-2)
                  (m-1-2) edge node[left] {$F_1$} (m-2-2)
                  (m-2-3) edge node[below] {$U$} (m-2-2);
      \path[map]  (m-2-1) edge[shorten >= 8pt, shorten <= 8pt] node[above left] {$F_\hc$} node[nat] {.} (m-1-2)
                  (m-2-3) edge[shorten >= 8pt, shorten <= 8pt] node[above right] {$F_U$} node[nat] {.} (m-1-2);

      \draw       (m-2-3) node {$\L_0$,};
    \end{tikzpicture}
  \end{displaymath}
  whose components are horizontal cells, and that satisfy the usual associativity and unit axioms for lax monoidal functors, see e.g.\ \cite[Section XI.2]{MacLane98}. The transformations $F_\hc$ and $F_U$ are called the \emph{compositor} and \emph{unitor} of $F$. A \emph{normal functor} $\map F\K\L$ is a lax functor for which $F_U$ is the identity; a normal functor whose compositor $F_\hc$ is invertible is called a \emph{pseudofunctor}.
\end{definition}
	Unpacking this, a lax functor $\map F\K\L$ maps the objects, vertical and horizontal morphisms, as well as cells, of $K$ to those of $\L$, in a way that preserves horizontal and vertical sources and targets. Moreover, composition of vertical maps is preserved while horizontal compositions and units are preserved only up to natural coherent cells $\cell{F_\hc}{FJ \hc FH}{F(J \hc H)}$ and $\cell{F_U}{U_{FA}}{FU_A}$.

	The notions of normal functors and pseudofunctors above are slightly stronger than the ones in \cite[Section 6]{Shulman08}, where the unitors $F_U$ are only required to be invertible; such functors we will call \emph{weakly normal}. The stronger notions are preferable as they are easier to work with, while \propref{weakly normal functors coherence} below shows that every weakly normal functor is equivalent to a normal one. Most of the functors that we encounter will be normal.
	
	We will often use the unit axiom: for a normal functor $F$ it means that for any horizontal morphism $\hmap JAB$, the composition
	\begin{displaymath}
		U_{FA} \hc FJ = FU_A \hc FJ \xRar{F_\hc} F(U_A \hc J) \xRar{F\mathfrak l} FJ
	\end{displaymath}
	is equal to $\cell{\mathfrak l}{U_{FA} \hc FJ}{FJ}$, and similar for $\mathfrak r$. Also notice that, with $F$ still normal, the naturality of the unitor implies that $\cell{FU_f = U_{Ff}}{U_{FA}}{U_{FC}}$ for any $\map fAC$.

\begin{example}
	Any pullback preserving functor $\map F\D\E$ between categories with finite limits induces a pseudofunctor $\map{\Span F}{\Span\D}{\Span\E}$, simply by applying $F$ to the spans in $\D$. Clearly this preserves unit spans strictly, while the existence of invertible compositors follows from the fact that $F$ preserves pullbacks. The functor underlying the `free strict category'-monad is an example.
\end{example}

	Recall that the companion $C(f, \id)$ of a morphism $\map fAC$ is defined by two cells $\cell{\ls f \eps}{C(f, \id)}{U_C}$ and $\cell{\ls f \eta}{U_A}{C(f, \id)}$ that satisfy the companion identities, see the discussion following \defref{definition:cartesian filler}. That, between equipments, normal functors are the right morphisms to consider is because they preserve companions and conjoints, as follows.
	
	The following is a direct consequence of \cite[Proposition 6.8]{Shulman08}. Although the proof is relatively simple, it is important to our theory: without the inverses to the compositors below the main result (\thmref{right colax forgetful functor lifts all weighted colimits}) could not have been proven for monads that are not pseudo.
\begin{proposition}[Shulman] \label{normal functors preserve companions}
	Let $\map F\K\L$ be a normal functor between equipments. For any morphism $\map fAC$ in $\K$ the images
	\begin{displaymath}
		\cell{F\ls f\eps}{FC(f, \id)}{U_{FC}} \qquad \text{and} \qquad \cell{F\ls f\eta}{U_{FA}}{FC(f, \id)}
	\end{displaymath}
	define $FC(f, \id)$ as the companion of $Ff$ in $\L$. In particular the canonical map $FC(f, \id) \Rar FC(Ff,\id)$, that is induced by the cartesian filler of $FC(Ff, \id)$, is invertible. Moreover the compositor
	\begin{displaymath}
		\cell{F_\hc}{FC(f, \id) \hc FK}{F\bigpars{C(f, \id) \hc K}}
	\end{displaymath}
	is invertible for any promorphism $\hmap KCD$. Analogous results hold for conjoints: for each morphism $\map gBD$ the image $FD(\id, g)$ is the conjoint of $Fg$ and the compositor
	\begin{displaymath}
		\cell{F_\hc}{FK \hc FD(\id, g)}{F\bigpars{K \hc D(\id, g)}}
	\end{displaymath}
	is invertible.
\end{proposition}
\begin{proof}
	The images $F\ls f\eps$ and $F\ls f\eta$ satisfy the vertical companion identity because $F$ preserves vertical composition; to see that they satisfy the horizontal one as well notice that
	\begin{multline*}
		\mathfrak r \of (F \ls f\eta \hc F\ls f \eps) \of \inv{\mathfrak l} = F\mathfrak r \of F_\hc \of (F\ls f\eta \hc F\ls f\eps) \of \inv{\mathfrak l} = F\mathfrak r \of F(\ls f\eta \hc \ls f\eps) \of F_\hc \of \inv{\mathfrak l} \\
		= F\mathfrak r \of F(\ls f\eta \hc \ls f\eps) \of F\inv{\mathfrak l} = F\bigpars{\mathfrak r \of (\ls f\eta \hc \ls f\eps) \of \inv{\mathfrak l}} = \id_{FC(f, \id)},
	\end{multline*}
	which follows from the unit axiom for $F$, the naturality of $F_\hc$ and the horizontal companion identity for $C(f, \id)$. We claim that the inverse to $F_\hc$ above is given by the composite
	\begin{align*}
		F\bigpars{C(f, \id) \hc K} &\xRar{\inv{\mathfrak l}} U_{FA} \hc F\bigpars{C(f, \id) \hc K} \\
		&\xRar{F\ls f\eta \hc F(\ls f\eps \hc \id)} FC(f, \id) \hc F(U_C \hc K) \xRar{\id \hc F\mathfrak l} FC(f, \id) \hc FK.
	\end{align*}
	Indeed, using the unit axiom for $F$ we find that pre- and postcomposing this with $F_\hc$ gives respectively
	\begin{multline*}
		F\bigpars{C(f, \id) \hc K} \xRar{F\inv{\mathfrak l}}  F\bigpars{U_A \hc C(f, \id) \hc K} \\
		\xRar{F(\ls f\eta \hc \ls f\eps \hc \id)} F\bigpars{C(f, \id) \hc U_C \hc K} \xRar{F(\id \hc \mathfrak l)} F\bigpars{C(f, \id) \hc K}
	\end{multline*}
	and
	\begin{multline*}
		FC(f, \id) \hc FK \xRar{\inv{\mathfrak l}} U_{FA} \hc FC(f, \id) \hc FK \\
		\xRar{F\ls f\eta \hc F\ls f\eps \hc \id} FC(f, \id) \hc U_{FC} \hc K \xRar{\id \hc \mathfrak l} FC(f, \id) \hc K
	\end{multline*}
	so that the claim follows from the horizontal companion identities for $FC(f, \id)$ and $C(f, \id)$ and the fact that $(\id \hc \mathfrak l) \of \mathfrak a = \mathfrak r \hc \id$, which follows from the coherence axioms for pseudo double categories.
\end{proof}
	
	Normal functors between equipments are compatible with the correspondences of \propref{left and right cells}, as follows.
\begin{proposition} \label{lax functors and left and right cells}
	Let $\map F\K\L$ be a normal functor. The diagram on the left commutes for any cell $\phi$ on the right, where the isomorphisms are given by the previous proposition, and the cells $\lambda\phi$ and $\lambda F\phi$ are given by \propref{left and right cells}.
	\begin{displaymath}
		\begin{tikzpicture}[baseline]
			\matrix(m)[math2em]
			{ FJ \hc FD(Fg, \id) & FJ \hc FD(g, \id) \\
				FC(Ff, \id) \hc FK & F\bigpars{J \hc D(g, \id)} \\
				FC(f, \id) \hc FK & F\bigpars{C(f, \id) \hc K} \\ };
			\path[map]	(m-1-1) edge[cell] node[left] {$\lambda F\phi$} (m-2-1)
									(m-1-2) edge[cell] node[right] {$F_\hc$} (m-2-2)
									(m-2-2) edge[cell] node[right] {$F\lambda\phi$} (m-3-2)
									(m-3-1) edge[cell] node[below] {$F_\hc$} (m-3-2);
			\path[color=white]	(m-1-1) edge node[color=black, font=] {$\iso$} (m-1-2)
									(m-2-1) edge node[color=black, sloped, font=] {$\iso$} (m-3-1);
		\end{tikzpicture}
		\qquad\qquad\begin{tikzpicture}[baseline]
		\matrix(m)[math175em]{ A & B \\ C & D \\};
		\path[map]	(m-1-1) edge[barred] node[above] {$J$} (m-1-2)
												edge node[left] {$f$} (m-2-1)
								(m-1-2) edge node[right] {$g$} (m-2-2)
								(m-2-1) edge[barred] node[below] {$K$} (m-2-2);
		\path[transform canvas={shift=($(m-1-2)!0.5!(m-2-2)$)}] (m-1-1) edge[cell] node[right] {$\phi$} (m-2-1);
	\end{tikzpicture}
	\end{displaymath}
	A similar diagram commutes for the cells $\rho F\phi$ and $F\rho \phi$.
\end{proposition}
\begin{proof}
	Writing $\ls f\eps$ and $\ls{Ff}\eps$ for the companion cells defining the companions $C(f, \id)$ and $FC(Ff, \id)$ respectively, the isomorphism on the left is the unique factorisation $FC(Ff, \id) \Rar FC(f, \id)$ of $F\ls f\eps$ through $\ls{Ff}\eps$. To show that the diagram commutes, we will postcompose it with $\map{F(\ls f\eps \hc \id)}{F\bigpars{C(f, \id) \hc K}}{FK}$. Notice that, since the composite $F(\ls f\eps \hc \id) \of F_\hc = F\ls f\eps \hc \id$ is a cartesian filler, and $F_\hc$ is an isomorphism by the previous proposition, the cell $F(\ls f\eps \hc \id)$ is a cartesian filler as well. Hence, by uniqueness of factorisations through cartesian fillers, if the diagram commutes after postcomposition with $F(\ls f\eps \hc \id)$, then it commutes itself as well. The bottom leg of the diagram above, postcomposed with $F(\ls{f}\eps \hc \id)$, is the composite
	\begin{multline*}
		\bigbrks{FJ \hc FD(Fg, \id) \xRar{\lambda F\phi} FC(Ff, \id) \hc FK \xRar{\ls{Ff}\eps \hc \id} FK} \\
		= \bigbrks{FJ \hc FD(Fg, \id) \xRar{F\phi \hc \ls{Fg}\eps} FK},
	\end{multline*}
	where the equality follows from $\lambda F\phi = \ls{Ff}\eta \hc \phi \hc \ls{Fg} \eps$ and the vertical companion identity $\ls{Ff}\eps \of \ls{Ff}\eta = U_{Ff}$. Likewise, precomposing the bottom leg with $F(\ls f\eps \hc \id)$ gives
	\begin{multline*}
		FJ \hc FD(Fg, \id) \Rar FJ \hc FD(g, \id) \xRar{F_\hc} F\bigpars{J \hc D(g, \id)} \xRar{F(\phi \hc \ls g\eps)} FK \\
		= \bigbrks{FJ \hc FD(Fg, \id) \xRar{F\phi \hc \ls{Fg}\eps} FK}
	\end{multline*}
	as well, and the proof follows.
\end{proof}

	Lax functors between closed equipments are compatible with left homs, as follows.
\begin{proposition} \label{lax functors induce left hom coherence cell}
	Let $\map F\K\L$ be a lax functor between closed equipments. The coherence cells $F_\hc$ induce canonical horizontal cells
	\begin{displaymath}
		F(J \lhom H) \hc K \Rar FJ \lhom (FH \hc K)
	\end{displaymath}
	in $\L$, for any triple of promorphisms $\hmap JAB$ and $\hmap HAC$ in $\K$, and \mbox{$\hmap K{FC}D$} in $\L$.
\end{proposition}
\begin{proof}
	Simply take the adjoint of
	\begin{displaymath}
		FJ \hc F(J \lhom H) \hc K \xRar{F_\hc \hc \id} F(J \hc J \lhom H) \hc K \xRar{F\ev \hc \id} FH \hc K. \qedhere
	\end{displaymath}
\end{proof}

	Having introduced functors we turn to transformations.
\begin{definition} \label{definition:transformation}
  A \emph{transformation} $\nat\xi FG$ between lax functors $F$ and $\map G\K\L$ is given by natural transformations $\nat{\xi_0}{F_0}{G_0}$ and $\nat{\xi_1}{F_1}{G_1}$, with $L\xi_1 = \xi_0 L$ and $R\xi_1 = \xi_0R$, and such that the following diagrams commute, which we will call the composition and unit axiom respectively.
  \begin{displaymath}
	  \begin{tikzpicture}[textbaseline]
			\matrix(m)[math2em]
			{	\hc_\L \of \pars{F_1 \times_{F_0} F_1} & F_1 \of \hc_\K \\
				\hc_\L \of \pars{G_1 \times_{G_0} G_1} & G_1 \of \hc_\K \\ };
			\path[map]	(m-1-1) edge node[nat, below] {.} node[above] {$F_\hc$} (m-1-2)
													edge node[nat] {.} node[left] {$\hc_\L \of \pars{\xi_1 \times_{\xi_0} \xi_1}$} (m-2-1)
									(m-1-2) edge node[nat, below] {.} node[right] {$\xi_1 \of \hc_\K$} (m-2-2)
									(m-2-1) edge node[nat] {.} node[below] {$G_\hc$} (m-2-2);
	  \end{tikzpicture}
	  \qquad
		\begin{tikzpicture}[textbaseline]
			\matrix(m)[math2em]
			{	U \of F_0 & F_1 \of U \\
				U \of G_0 & G_1 \of U \\ };
			\path[map]	(m-1-1) edge node[nat, below] {.} node[above] {$F_U$} (m-1-2)
													edge node[nat] {.} node[left] {$U \of \xi_0$} (m-2-1)
									(m-1-2) edge node[nat, below] {.} node[right] {$\xi_1 \of U$} (m-2-2)
									(m-2-1) edge node[nat] {.} node[below] {$G_U$} (m-2-2);
	  \end{tikzpicture}
	\end{displaymath}
\end{definition}
	Unpacking this, a natural transformation $\xi$ consists of a vertical morphism $\map{\xi_A}{FA}{GA}$ for every object $A$ of $\K$, as well as a cell $\xi_J$ for every horizontal morphism $\hmap JAB$ in $\K$, as on the left below. The composition and unit axioms state that the diagrams on the right commute. In the case of normal functors $F$ and $G$ the unit diagram below reduces to $(\xi)_{U_A} = U_{\xi_A}$, for every object $A$ of $\K$.
\begin{displaymath}
	\begin{tikzpicture}[baseline]
		\matrix(m)[math175em]{ FA & FB \\ GA & GB \\};
		\path[map]	(m-1-1) edge[barred] node[above] {$FJ$} (m-1-2)
												edge node[left] {$\xi_A$} (m-2-1)
								(m-1-2) edge node[right] {$\xi_B$} (m-2-2)
								(m-2-1) edge[barred] node[below] {$GJ$} (m-2-2);
		\path[transform canvas={shift=($(m-1-2)!0.5!(m-2-2)$)}] (m-1-1) edge[cell] node[right] {$\xi_J$} (m-2-1);
	\end{tikzpicture}
	\qquad\begin{tikzpicture}[baseline]
		\matrix(m)[math175em]{ FJ \hc FH & F(J \hc H) \\ GJ \hc GH & G(J \hc H) \\};
		\path	(m-1-1) edge[cell] node[above] {$F_\hc$} (m-1-2)
												edge[cell] node[left] {$\xi_J \hc \xi_H$} (m-2-1)
								(m-1-2) edge[cell] node[right] {$\xi_{J \hc H}$} (m-2-2)
								(m-2-1) edge[cell] node[below] {$G_\hc$} (m-2-2);
	\end{tikzpicture}
	\qquad\begin{tikzpicture}[baseline]
		\matrix(m)[math175em]{ U_{FA} & FU_A \\ U_{GA} & GU_A \\};
		\path	(m-1-1) edge[cell] node[above] {$F_U$} (m-1-2)
												edge[cell] node[left] {$U_{\xi_A}$} (m-2-1)
								(m-1-2) edge[cell] node[right] {$\xi_{U_A}$} (m-2-2)
								(m-2-1) edge[cell] node[below] {$G_U$} (m-2-2);
	\end{tikzpicture}
\end{displaymath}

	Like normal functors between equipments, transformations between such functors are also compatible with companions and conjoints, in the following sense.
\begin{proposition} \label{transformations and companions}
	Let $\nat\xi FG$ be a transformation of normal functors $F$ and $\map G\K\L$ between equipments. For any morphism $\map fAC$ the components $\xi_{C(f, \id)}$ and $\xi_{C(\id, f)}$ coincide with the composites
	\begin{displaymath}
		\begin{tikzpicture}[baseline]
			\matrix(m)[math175em]
			{	FA & FA & FA & FC \\
				GA & GA & FC & FC \\
				GA & GC & GC & GC \\ };
			\path[map]	(m-1-1) edge node[left] {$\xi_A$} (m-2-1)
									(m-1-2) edge node[left] {$\xi_A$} (m-2-2)
									(m-1-3) edge[barred] node[above] {$FC(f, \id)$} (m-1-4)
													edge node[left] {$Ff$} (m-2-3)
									(m-2-2) edge node[right] {$Gf$} (m-3-2)
									(m-2-3) edge node[right] {$\xi_C$} (m-3-3)
									(m-2-4) edge node[right] {$\xi_C$} (m-3-4)
									(m-3-1) edge[barred] node[below] {$GC(f, \id)$} (m-3-2);
			\path	(m-1-1)	edge[eq] (m-1-2)
						(m-1-2) edge[eq] (m-1-3)
						(m-1-4) edge[eq] (m-2-4)
						(m-2-1) edge[eq] (m-2-2)
										edge[eq] (m-3-1)
						(m-2-3) edge[eq] (m-2-4)
						(m-3-2) edge[eq] (m-3-3)
						(m-3-3) edge[eq] (m-3-4);
			\path[transform canvas={shift=(m-2-3), xshift=-0.7em}]	(m-2-1) edge[cell] node[right] {$G\ls f\eta$} (m-3-1)
						(m-1-3) edge[cell] node[right] {$F\ls f\eps$} (m-2-3);
		\end{tikzpicture}
		\qquad\text{and}\qquad\begin{tikzpicture}[baseline]
			\matrix(m)[math175em]
			{	FC & FA & FA & FA \\
				FC & FC & GA & GA \\
				GC & GC & GC & GA. \\ };
			\path[map]	(m-1-1) edge[barred] node[above] {$FC(\id, f)$} (m-1-2)
									(m-1-2) edge node[right] {$Ff$} (m-2-2)
									(m-1-3) edge node[right] {$\xi_A$} (m-2-3)
									(m-1-4) edge node[right] {$\xi_A$} (m-2-4)
									(m-2-1) edge node[left] {$\xi_C$} (m-3-1)
									(m-2-2) edge node[left] {$\xi_C$} (m-3-2)
									(m-2-3) edge node[left] {$Gf$} (m-3-3)
									(m-3-3) edge[barred] node[below] {$GC(\id, f)$} (m-3-4);
			\path	(m-1-1)	edge[eq] (m-2-1)
						(m-1-2) edge[eq] (m-1-3)
						(m-1-3) edge[eq] (m-1-4)
						(m-2-1) edge[eq] (m-2-2)
						(m-3-1)	edge[eq] (m-3-2)
						(m-2-3) edge[eq] (m-2-4)
						(m-3-2) edge[eq] (m-3-3)
						(m-2-4) edge[eq] (m-3-4);
			\path[transform canvas={shift=(m-2-3), xshift=-0.7em}]	(m-2-3) edge[cell] node[right] {$G\eta_f$} (m-3-3)
						(m-1-1) edge[cell] node[right] {$F\eps_f$} (m-2-1);
		\end{tikzpicture}
	\end{displaymath}
\end{proposition}
\begin{proof}
	That $\xi_{C(f, \id)}$ equals the composite on the left follows directly by horizontally composing the following identity, which follows from the unit axiom of $\xi$ and its naturality, on the right by $F\ls f\eps$.
	\begin{displaymath}
		G \ls f\eta \of U_{\xi_A} = G \ls f\eta \of \xi_{U_A} = \xi_{C(f, \id)} \of F\ls f\eta
	\end{displaymath}
	A similar argument shows that $\xi_{C(\id, f)}$ coincides with the composite on the right.
\end{proof}

	We recall \cite[Proposition 6.17]{Shulman08}.
\begin{proposition}[Shulman]
  Small equipments, lax functors and transformations form a strict $2$-category $\lEquip$. Restricting to normal functors and pseudofunctors we obtain respectively sub-$2$-categories $\nEquip$ and $\psEquip$.
\end{proposition}
	Of course $\lEquip$ is contained in the larger $2$-category $\lDbl$ of small pseudo double categories, lax functors and transformations, which in turn contains sub-$2$\ndash cat\-egories $\nDbl$ and $\psDbl$ consisting of its normal functors and pseudofunctors respectively. The following is \cite[Example 6.20]{Shulman08}, where we have denoted by $\Cat_{\,\textup{fl}}$ the $2$\ndash cat\-egory of categories with finite limits, pullback preserving functors, and natural transformations.
\begin{proposition}[Shulman] \label{span is a 2-functor}
	The assignment $\E \mapsto \Span\E$ given by \exref{example:spans} extends to a $2$-functor $\map{\mathsf{Span}}{\Cat_{\,\textup{fl}}}{\psEquip}$.
\end{proposition}
	
	Following \cite{Shulman08} we denote by $\qlEquip$ and $\qnEquip$ the sub-$2$-categories of $\lEquip$ and $\nEquip$ consisting of equipments $\K$ for which every $H(\K)(A, B)$ has reflexive coequalisers that are preserved by $\dash \hc \dash$ on both sides, that is those equipments that satisfy the conditions of \propref{bimodule equipments}. The following result is \cite[Proposition 11.12]{Shulman08}.
\begin{proposition}[Shulman] \label{mod is a 2-functor}
	The assignment $\K \mapsto \Mod\K$ given by \propref{bimodule equipments} extends to a $2$-functor $\map{\mathsf{Mod}}\qlEquip\qnEquip$. The image $\Mod F$ of a pseudofunctor $\map F\K\L$ is again a pseudofunctor whenever $F$ preserves the reflexive coequalisers of $H(\K)(A, B)$.
\end{proposition}
\begin{proof}[Sketch of the proof]
	A lax functor $\map F\K\L$ between equipments $\K$ and $\L$ preserves monoids, morphisms of monoids, bimodules and cells of bimodules in the same way that lax monoidal functors do. Now suppose that $\K$ and $\L$ are contained in $\qlEquip$, so that their monoids and bimodules form equipments $\Mod\K$ and $\Mod\L$ by \propref{bimodule equipments}. Since the horizontal unit $U_A$ for a monoid $\hmap A{A_0}{A_0}$ in $\K$ is $A$ itself, we have $FU_A = U_{FA}$; so we can take the unitor of $\Mod F$ to be the identity. To construct the compositor $\cell{\Mod F_\hc}{FJ \hc_{FB} FH}{F(J \hc_B H)}$ for composable bimodules $\hmap JAB$ and $\hmap HBC$, consider the following diagram in $H(\L)(A_0, C_0)$.
	\begin{displaymath}
		\begin{tikzpicture}
			\matrix(m)[math2em]
			{	FJ \hc FB \hc FH & FJ \hc FH & FJ \hc_{FB} FH \\
				F(J \hc B \hc H) & F(J \hc H) & F(J \hc_B H) \\};
			\path	(m-1-1)	edge[transform canvas={yshift=3.5pt}, cell] (m-1-2)
										edge[transform canvas={yshift=-3.5pt}, cell] (m-1-2)
										edge[cell] (m-2-1)
						(m-1-2) edge[cell] (m-1-3)
										edge[cell] node[right] {$F_\hc$} (m-2-2)
						(m-1-3) edge[cell, dashed] node[right] {$\Mod F_\hc$} (m-2-3)
						(m-2-1) edge[transform canvas={yshift=3.5pt}, cell] (m-2-2)
										edge[transform canvas={yshift=-3.5pt}, cell] (m-2-2)
						(m-2-2) edge[cell] (m-2-3);
		\end{tikzpicture}
	\end{displaymath}
	Here the top row is the coequaliser defining $FJ \hc_{FB} FH$ in $\Mod\L$, the bottom row is the image of the coequaliser defining $J \hc_B H$ and the solid vertical cells are given by the compositors of $F$. We can take the compositor $\Mod F_\hc$ to be the unique cell given by the universal property of the coequaliser $FJ \hc_{FB} FH$, making $\Mod F$ into a normal functor $\Mod\K \to \Mod\L$. Moreover, if $F$ is a pseudofunctor, so that the solid vertical cells are invertible, then $\Mod F_\hc$ is again invertible when coequaliser at the bottom row is preserved by $F$, thus proving the second assertion. Similarly the components of a transformation $\nat\xi FG$ induce a transformation $\nat{\Mod\xi}{\Mod F}{\Mod G}$.
\end{proof}

	Remember that every pseudo double category $\K$ induces a $2$-category $V(\K)$ of vertical morphisms and vertical cells. In the proposition below $\twoCat$ denotes the $2$-category of $2$-categories, $2$-functors and $2$\ndash transformations.
\begin{proposition}\label{normal functors induce vertical functors}
  The assignment $\K \mapsto V(\K)$ extends to a strict $2$-functor
  \begin{displaymath}
	  \map V\nDbl \twoCat.
	\end{displaymath}
\end{proposition}
\begin{proof}
  A normal functor $\map F\K\L$ preserves vertical cells, so $V(F)$ can be defined simply by applying $F$ to the objects, vertical morphisms and vertical cells. That $V(F)$ preserves composition strictly follows from the unit axiom for $F$. Likewise the vertical part $\nat{\xi_0}{F_0}{G_0}$, of any transformation $\nat\xi FG$, forms a $2$-transformation $\nat{V(\xi)}{V(F)}{V(G)}$: the naturality with respect to vertical cells follows from the unit axiom for $\xi$.
\end{proof}

	Having introduced functors and transformations we are ready to define monads on pseudo double categories.
\begin{definition} \label{definition:monad}
  A \emph{lax monad} $T$ on a pseudo double category $\K$ is a monad $T = (T, \mu, \eta)$ on $\K$ in $\lDbl$, consisting of a lax functor $\map T\K\K$, a multiplication $\nat\mu{T^2}T$ and a unit $\nat\eta{\id_\K}T$, satisfying the usual axioms (see \cite[Section VI.1]{MacLane98}). A lax monad $T$ on $\K$ is called a \emph{normal} monad or \emph{pseudomonad} whenever the underlying endofunctor $\map T\K\K$ is a normal functor or pseudofunctor.
\end{definition}
  By \propref{normal functors induce vertical functors}, a normal monad $T$ on $\K$ induces a $2$-monad $V(T)$ on $V(\K)$. As mentioned in the introduction, this allows us in the next section to define $T$\ndash algebras, $T$-morphisms and vertical $T$-cells as $V(T)$-algebras, $V(T)$-morphisms and $V(T)$\ndash cells in $V(\K)$. Notice that for $\K$ in $\qlEquip$, any lax monad $S$ on $\K$ induces a normal monad $\Mod S$ on $\Mod\K$, by \propref{mod is a 2-functor}.
\begin{example} \label{example:monads inducing monads on internal profunctors}
	We saw in \exref{example:internal profunctor equipments} that if $\E$ is a category with finite limits and reflexive coequalisers, the latter preserved by pullbacks, then categories internal to $\E$, together with internal profunctors, form an equipment $\inProf\E = \Mod{\Span\E}$. In that case any pullback preserving monad $T$ on $\E$ induces a pseudomonad $\Span T$ on $\Span\E$ by \propref{span is a 2-functor} which, by \propref{mod is a 2-functor}, induces a normal monad $\inProf T$ on $\inProf\E$. This is how, in the next section, the `free strict double category'-monad on $\psProf{\GG_1}$ will be obtained from the `free category'-monad on the category $\ps{\GG_1}$ of presheaves on $\GG_1 = (0 \rightrightarrows 1)$. Even though the latter is a pseudomonad the former is not, as we will see in \propref{free strict double category monad is not pseudo}.
	
	Many interesting monads on bicategories of spans arise in this way, see for example \cite[Section 4.1]{Leinster04} where Leinster discusses `cartesian monads': a monad $(T, \mu, \eta)$ on a category $\E$ with finite limits is \emph{cartesian} whenever $T$ preserves pullbacks and the naturality squares of $\mu$ and $\eta$ are pullbacks. In our terms the second condition means that for each span $J$ the cells $\mu_J$ and $\eta_J$ are right invertible, see \exref{example:right invertible spans}.
\end{example}

	Remember that weighted colimits can be defined in any pseudo double category, see \defref{definition:weighted colimits in double categories}. The following two propositions record some useful consequences of the universality of weighted colimits, with respect to normal functors of pseudo double categories and the transformations between them.
\begin{proposition} \label{universality result for colimits with respect to functors}
  Let $\map F\K\L$ be a normal functor of pseudo double categories. Given a colimit diagram $\cdiag MdAJB$ in $\K$, any vertical morphism $\map f{FM}N$ in $\L$ induces a canonical vertical cell
  \begin{displaymath}
	  \colim_{FJ}(f \of Fd) \Rar f \of F(\colim_J d)
	\end{displaymath}
	of morphisms $FB \to N$ in $\L$, provided these weighted colimits exist. In the case that $\K$ and $\L$ are closed equipments this cell corresponds, after applying $\lambda$ (\propref{left and right cells}) and under the canonical isomorphisms
	\begin{flalign*}
		&& &N\bigpars{f \of F(\colim_J d), \id} \iso F\bigpars{J \lhom M(d, \id)} \hc N(f, \id) & \\
		\text{and} && &N\bigpars{\colim_{FJ} (f \of Fd), \id} \iso FJ \lhom \bigpars{FM(d, \id) \hc N(f, \id)},
	\end{flalign*}
	to the horizontal cell $F\bigpars{J \lhom M(d, \id)} \hc N(f, \id) \Rar FJ \lhom \bigpars{FM(d, \id) \hc N(f, \id)}$ given by \propref{lax functors induce left hom coherence cell}.
\end{proposition}
\begin{proof}
	Supposing that the weighted colimits exist, denote by $\eta$ and $\zeta$ the cells that exhibit respectively $l = \colim_J d$ as the $J$-weighted colimit of $d$ and $k = \colim_{FJ} (f \of Fd)$ as the $FJ$-weighted colimit of $f \of Fd$, as in \propref{weighted colimits equivalent to top absolute left Kan extensions}. It follows from the universality of $\zeta$ (\defref{definition:left Kan extensions in pseudo double categories}) that there exists a unique vertical cell $\cell\phi k{f \of Fl}$ such that
	\begin{displaymath}
		\begin{tikzpicture}[textbaseline]
			\matrix(m)[math175em]{FA & FB \\ FM & FM \\ N & N \\};
			\path[map]	(m-1-1) edge[barred] node[above] {$FJ$} (m-1-2)
													edge node[left] {$Fd$} (m-2-1)
									(m-1-2) edge node[right] {$Fl$} (m-2-2)
									(m-2-1) edge node[left] {$f$} (m-3-1)
									(m-2-2) edge node[right] {$f$} (m-3-2);
			\path				(m-2-1) edge[eq] (m-2-2)
									(m-3-1) edge[eq] (m-3-2);
			\path[transform canvas={shift=(m-2-1)}]	(m-1-2) edge[cell] node[right] {$F\eta$} (m-2-2)
									(m-2-2) edge[cell] node[right] {$U_f$} (m-3-2);
		\end{tikzpicture}
		= \begin{tikzpicture}[textbaseline]
			\matrix(m)[math175em]{FA & FB & FB \\ FM & \phantom{FM} & FM \\ N & N & N, \\};
			\path[map]	(m-1-1) edge[barred] node[above] {$FJ$} (m-1-2)
													edge node[left] {$Fd$} (m-2-1)
									(m-1-2) edge node[right] {$k$} (m-3-2)
									(m-1-3) edge node[right] {$Fl$} (m-2-3)
									(m-2-1) edge node[left] {$f$} (m-3-1)
									(m-2-3) edge node[right] {$f$} (m-3-3);
			\path				(m-1-2) edge[eq] (m-1-3)
									(m-3-1) edge[eq] (m-3-2)
									(m-3-2) edge[eq] (m-3-3);
			\path[transform canvas={shift=($(m-2-1)!0.5!(m-3-2)$)}]	(m-1-2) edge[cell] node[right] {$\zeta$} (m-2-2)
									(m-1-3) edge[cell] node[right] {$\phi$} (m-2-3);
		\end{tikzpicture}
	\end{displaymath}
	which is the canonical vertical cell that we seek. To prove the second assertion we apply $\lambda$ to the latter identity. Using the functoriality of $\lambda$ (\propref{left and right cells}), we obtain the top half of the following commuting diagram, while the bottom half is given by \propref{lax functors and left and right cells}.
	\begin{displaymath}
		\begin{tikzpicture}
			\matrix(m)[math2em, yshift=3.5em]
			{	FJ \hc N(f \of Fl, \id) & FJ \hc N(k, \id) \\};
			\matrix(n)[math]
			{ FJ \hc FM(l, \id) \hc N(f, \id) & FJ \hc FM(Fl, \id) \hc N(f, \id) & N(f \of Fd, \id) \\ };
			\matrix(p)[math2em, column sep=6em, yshift=-3.5em, xshift=-4em]
			{ F\bigpars{J \hc M(l, \id)} \hc N(f, \id) & FM(d, \id) \hc N(f, \id) \\};
			\matrix(q)[math, yshift=-6.5em, xshift=-3em]{FM(Fd, \id) \hc N(f,\id) \\};
			\path	(m-1-1) edge[cell] node[above] {$\id \hc \lambda\phi$} (m-1-2)
						(m-1-2) edge[cell] node[above right] {$\lambda\zeta$} (n-1-3)
						(n-1-1) edge[cell] node[left, inner sep=4pt] {$F_\hc \hc \id$} (p-1-1)
						(p-1-1) edge[cell] node[below left, inner sep=4pt] {$F\lambda\eta \hc \id$} (q-1-1)
						(n-1-2) edge[cell] node[left, inner sep=8pt] {$\lambda F\eta \hc \id$} (p-1-2);
			\path[white]	(m-1-1) edge node[black, sloped, desc] {$\iso$} (n-1-1)
										(n-1-1) edge node[black, sloped, desc] {$\iso$} (n-1-2)
										(n-1-3) edge node[black, sloped, desc] {$\iso$} (p-1-2)
										(q-1-1) edge node[black, sloped, desc] {$\iso$} (p-1-2);
		\end{tikzpicture}
	\end{displaymath}
	We claim that under \eqref{equation:horizontal composition adjunction} this diagram is adjoint to
	\begin{displaymath}
		\begin{tikzpicture}
			\matrix(m)[math2em]
			{ N(f \of Fl, \id) & N(k, \id) \\
				FM(l, \id) \hc N(f, \id) & FJ \lhom N(f \of Fd, \id) \\
				F\bigpars{J \lhom M(d, \id)} \hc N(f, \id) & FJ \lhom \bigpars{FM(d, \id) \hc N(f, \id)}, \\ };
			\path (m-1-1) edge[cell] node[above] {$\lambda\phi$} (m-1-2)
						(m-2-1)	edge[cell] node[left] {$F(\lambda\eta)^\flat$} (m-3-1)
						(m-1-2) edge[cell] node[right] {$(\lambda\zeta)^\flat$} (m-2-2)
						(m-3-1) edge[cell] (m-3-2);
			\path[white]	(m-1-1) edge node[black, sloped, desc] {$\iso$} (m-2-1)
										(m-2-2) edge node[black, sloped, desc] {$\iso$} (m-3-2);
		\end{tikzpicture}
	\end{displaymath}
	where the bottom cell is that given by \propref{lax functors induce left hom coherence cell}. Indeed it is clear that the top legs are adjoint; to see that the bottom ones are too notice that, in general, the following (which is \lemref{lemma:left hom facts}(c)) holds: given any horizontal cell $H \Rar J \lhom K$ that is adjoint to $\cell\psi{J \hc H}K$, the composite (where the second cell is given by \propref{lax functors induce left hom coherence cell})
	\begin{displaymath}
		FH \hc L \Rar F(J \lhom K) \hc L \Rar FJ \lhom (FK \hc L)
	\end{displaymath}
	is adjoint to $FJ \hc FH \hc L \xRar{F_\hc \hc \id} F(J \hc H) \hc L \xRar{F\psi \hc \id} FK \hc L$. Applying this to the bottom leg of the bottom diagram above we find it is adjoint to that of the top diagram. This completes the proof, since the sides of the diagram above are the isomorphisms of the assertion. 
\end{proof}

\begin{example} \label{example:canonical transformations between colimits induced by V-functors}
	If $\K = \enProf\V$ and $F$ is the identity, then the canonical transformation $\colim_J (f \of d) \natarrow f \of \colim_J d$ is the usual unique transformation $\phi$ whose components make the diagrams
	\begin{displaymath}
		\begin{tikzpicture}
			\matrix(m)[math2em, column sep=2.25em]
			{ J(\id, b) & M(d, \colim_J d(b)) \\
				M\bigpars{f \of d, \colim_J (f \of d)(b)} & M\bigpars{f \of d, (f \of \colim_J d)(b)} \\ };
			\path[map]	(m-1-1) edge node[above] {$\eta$} node[nat,below] {.} (m-1-2)
													edge node[left] {$\zeta$} node[nat] {.} (m-2-1)
									(m-1-2) edge node[right] {$f$} node[nat,below] {.} (m-2-2)
									(m-2-1) edge node[below] {$M(\id, \phi_b)$} node[nat] {.} (m-2-2);
		\end{tikzpicture}
	\end{displaymath}
	commute for each $b$ in $B$, where $\eta$ and $\zeta$ are the units of the weighted colimits $\colim_J d(b) \iso \colim_{J(\dash, b)} d$ and $\colim_J (f \of d)(b) \iso \colim_{J(\dash, b)} (f \of d)$, see \exref{example:units of weighted colimits in V-Prof}. In particular, if the target $N$ of $f$ is cocomplete, then $\phi_b$, as a map $\int^x (f \of d)(x) \tens J(x, b) \to f\bigpars{\int^y dy \tens J(y, b)}$ is induced by the composites
	\begin{displaymath}
		(f \of d)(x) \tens J(x, b) \to f\bigpars{dx \tens J(x, b)} \to f\bigpars{\int^y dy \tens J(y, b)},
	\end{displaymath}
	where the first map exists by universality of the copower $(f \of d)(x) \tens J(x, b)$ and the second map is the $f$-image of insertion into the coend.
\end{example}

\begin{proposition} \label{universality result for colimits with respect to transformations}
	Let $\nat\xi FG$ be a natural transformation between normal functors $F$ and $\map G\K\L$ of pseudo double categories. Given a colimit diagram $\cdiag MdAJB$ in $\K$, any vertical morphism $\map g{GM}N$ in $\L$ induces a canonical vertical cell
	\begin{displaymath}
		\colim_{FJ} (g \of Gd \of \xi_A) \Rar \colim_{GJ} (g \of Gd) \of \xi_B,
	\end{displaymath}
	provided that these colimits exist. This cell makes the following diagram commute, where the other unlabelled cells are given by the previous proposition. Moreover it is invertible whenever $\xi_J$ is right invertible.
	\begin{displaymath}
		\begin{tikzpicture}
			\matrix(k)[math, yshift=6.5em]{ \colim_{GJ} (g \of Gd) \of \xi_B \\};
			\matrix(l)[math2em, yshift=3.25em, column sep=4em]
			{	\colim_{FJ} (g \of Gd \of \xi_A) & g \of G(\colim_J d) \of \xi_B \\ };
			\matrix(m)[math2em]
			{ \colim_{FJ} (g \of \xi_M \of Fd) & g \of \xi_M \of F(\colim_J d) \\};
			\path	(k-1-1) edge[cell] (l-1-2)
						(l-1-1) edge[cell] (k-1-1)
										edge[eq] (m-1-1)
						(l-1-2) edge[eq] (m-1-2)
						(m-1-1) edge[cell] (m-1-2);
		\end{tikzpicture}
	\end{displaymath}
\end{proposition}
\begin{proof}
	Denoting by $\zeta$ and $\theta$ the cells that define respectively the weighted colimits $l = \colim_{GJ} (g \of Gd)$ and $k = \colim_{FJ} (g \of Gd \of \xi_A)$, then the canonical vertical cell above is the unique factorisation $\phi$ of $\zeta \of \xi_J$ through $\theta$:
	\begin{displaymath}
		\begin{tikzpicture}[textbaseline]
			\matrix(m)[math175em]{ FA & FB \\ GA & GB \\ GM & \phantom{GM} \\ N & N \\ };
			\path[map]	(m-1-1) edge[barred] node[above] {$FJ$} (m-1-2)
													edge node[left] {$\xi_A$} (m-2-1)
									(m-1-2) edge node[right] {$\xi_B$} (m-2-2)
									(m-2-1) edge[barred] node[below] {$GJ$} (m-2-2)
													edge node[left] {$Gd$} (m-3-1)
									(m-2-2) edge node[right] {$l$} (m-4-2)
									(m-3-1) edge node[left] {$g$} (m-4-1);
			\path				(m-4-1) edge[eq] (m-4-2);
			\path[transform canvas={shift={($(m-2-2)!0.5!(m-3-2)$)}}]	(m-1-1) edge[cell] node[right] {$\xi_J$} (m-2-1);
			\path[transform canvas={shift=(m-3-2)}]	(m-2-1) edge[cell] node[right] {$\zeta$} (m-3-1);
		\end{tikzpicture}
		= \begin{tikzpicture}[textbaseline]
			\matrix(m)[math175em]{ FA & FB & FB \\ GA & & GB \\ GM & \phantom{GM} & \phantom{GM} \\ N & N & N. \\ };
			\path[map]	(m-1-1) edge[barred] node[above] {$FJ$} (m-1-2)
													edge node[left] {$\xi_A$} (m-2-1)
									(m-1-2) edge node[right] {$k$} (m-4-2)
									(m-1-3) edge node[right] {$\xi_B$} (m-2-3)
									(m-2-1) edge node[left] {$Gd$} (m-3-1)
									(m-2-3) edge node[right] {$l$} (m-4-3)
									(m-3-1) edge node[left] {$g$} (m-4-1);
			\path				(m-1-2) edge[eq] (m-1-3)
									(m-4-1) edge[eq] (m-4-2)
									(m-4-2) edge[eq] (m-4-3);
			\path[transform canvas={shift={($(m-2-2)!0.5!(m-3-3)$)}}]	(m-2-1) edge[cell] node[right] {$\theta$} (m-3-1)
									(m-2-2) edge[cell] node[right] {$\phi$} (m-3-2);
		\end{tikzpicture}
	\end{displaymath}
	That it makes the diagram above commute follows from the fact that, precomposed with $\theta$, the top and bottom leg equal respectively the left and right-hand side of
	\begin{displaymath}
		U_g \of G\kappa \of \xi_J = U_g \of U_{\xi_M} \of F\kappa,
	\end{displaymath}
	where $\kappa$ is the cell defining $\colim_J d$. The last assertion follows directly from \propref{weighted colimits are preserved by precomposition with invertible cells}, which shows that $\zeta \of \xi_J$ exhibits $l \of \xi_B$ as weighted colimit whenever $\xi_J$ is right invertible, so that in that case the canonical cell $\phi$ is an isomorphism.
\end{proof}
	
	To close this section we give the following simple coherence result, which shows that every weakly normal functor can be replaced by a normal one, so that our choice of considering just normal functors is not much of a restriction. Recall that a lax functor is called weakly normal if its unitor is invertible; we will denote by $\wnDbl$ the $2$-category of pseudo double categories, weakly normal functors and transformations.
	
	Recall that a $2$\ndash functor $\map F\C\D$ is called a \emph{local equivalence} if each of its hom-functors $\map F{\C(x, y)}{\D(Fx, Fy)}$ is an equivalence, and that it is called \emph{essentially surjective on objects} if for each object $y$ in $\D$ there is an object $x$ in $\C$ with $Fx \simeq y$. If $F$ is both a local equivalence and essentially surjective on objects then it is called a \emph{biequivalence}. Analogous to the case of ordinary categories, one can show (see \cite[Proposition 1.5.13]{Leinster04}) that the latter is equivalent to the existence of a pseudofunctor $\map G\D\C$ together with equivalences $\id_\C \simeq GF$ in $\fun \C\C$ and $FG \simeq \id_\D$ in $\fun \D\D$, where $\fun \C\C$ is the $2$\ndash category of pseudofunctors $\C \to \C$, pseudonatural transformations and modifications; likewise for $\fun \D\D$.
\begin{proposition} \label{weakly normal functors coherence}
	The inclusion $\nDbl \hookrightarrow \wnDbl$ is a biequivalence.
\end{proposition}
\begin{proof}
	Denote the inclusion by $j$. Clearly $j$ is essentially surjective on objects so we have to prove that, for each pair of pseudo double categories $\K$ and $\L$, the inclusion $\emb j{\nDbl(\K, \L)}{\wnDbl(\K, \L)}$ is an equivalence. To do so consider a weakly normal functor $\map F\K\L$ between pseudo double categories. We will construct a normal functor $\map{F'}\K\L$ from $F$ simply by replacing the $F_1$-image of each horizontal unit $U_A$ by $U_{FA}$. Simultaneously we will construct an invertible transformation $\nat{\rho = \rho_F}{F'}F$. First we take $F_0' = F_0$ and $\rho_0 = \id$; we define $F_1'$ and $\rho_1$ on horizontal morphisms by
	\begin{displaymath}
		F_1' J = \begin{cases}
			U_{FA} & \text{if $J = U_A$;} \\
			F_1 J & \text{otherwise}
		\end{cases}
		\qquad \text{and} \qquad
		(\rho_1)_J = \begin{cases}
			\cell{F_U}{U_{FA}}{FU_A} & \text{if $J = U_A$;} \\
			\id_J & \text{otherwise}.
		\end{cases}
	\end{displaymath}
	For a cell $\cell\phi JK$ in $\K_1$, take $F_1' \phi = \inv\rho_K \of F_1\phi \of \rho_J$. It is clear that this makes $F_1'$ into a functor $\K_1 \to \L_1$ and that it makes $\rho_1$ natural. That $F_1'$ is compatible with $L$, $\map R{\K_1}{\K_0}$ and $L$, $\map R{\L_1}{\L_0}$ follows from the fact that the components of $\rho_1$ are horizontal cells, thus $LF_1' = LF_1 = F_0L$; likewise for $R$. Of course the unitor $F_U'$ we take to be the identity, making $\rho$ compatible with the unitors of $F_1'$ and $F_1$. Finally let the compositor $F_\hc'$ be given by the composites
	\begin{displaymath}
		F'J \hc F'H \xRar{\rho_J \hc \rho_H} FJ \hc FH \xRar{F_\hc} F(J \hc H) \xRar{\inv \rho_{J \hc H}} F'(J \hc H),
	\end{displaymath}
	for any pair of horizontally composable cells $J$ and $H$ in $\K$. Checking that these satisfy the unit and associativity axiom is straightforward, while $\rho$ is compatible with the compositors by definition. This completes the definition of $F'$ and $\nat{\rho_F}{F'}F$. It is clear now that the assignment $F \mapsto F'$ extends to a functor $\map{(\dash)'}{\wnDbl(\K, \L)}{\nDbl(\K, \L)}$ that maps a transformation $\nat\xi FG$ to the composite $\nat{\xi' = \inv\rho_G \of \xi \of \rho_F}{F'}{G'}$, so that the invertible transformations $\rho_F$ form a natural isomorphism $j \of (\dash)' \iso \id_{\wnDbl(\K, \L)}$. On the other hand clearly $(\dash)' \of j = \id_{\nDbl(\K, \L)}$, and we conclude that the inclusion $\emb j{\nDbl(\K, \L)}{\wnDbl(\K, \L)}$ is an equivalence.
\end{proof}

\section{Examples}\label{section:monad examples}
	Here we recall the main examples: the `free strict monoidal $\V$-category'-monad $\fmc$ and the `free symmetric strict monoidal $\V$\ndash category'-monad $\fsmc$ on the equipment $\enProf\V$ of $\V$-profunctors, as well as the `free strict double category'-monad $D$ on the equipment $\psProf{\GG_1}$ of $\GG_1$-categories, where $\GG_1 = (0 \rightrightarrows 1)$. Using \propref{mod is a 2-functor} the monads $\fmc$ and $\fsmc$ will be obtained from monads on $\Mat\V$, while $D$ is obtained from the `free category'-monad on $\Span{\ps{\GG_1}}$. After introducing each of them we will describe their (colax) algebras in the $2$-categories $\enCat\V = V(\enProf\V)$ and $\psCat{\GG_1} = V(\psProf{\GG_1})$ respectively. All three monads are (of course) well-known, so not much in this section is original. However, our way of defining the `free symmetric strict monoidal category'-monad $S$, by letting it be induced by a monad on $\Mat\V$, might be new. In fact, Cruttwell and Shulman discuss $S$ immediately after remarking that ``Not every monad on $\enProf\V$ or $\inProf\E$ is induced by one on $\Mat\V$ or $\Span\E$ however. The following examples are also important.'' \cite[Example 3.14]{Cruttwell-Shulman10}, suggesting that $S$ is not induced by a monad on $\Mat\V$. That it is induced by such a monad is useful, for it allows a relatively easy proof of the fact that $S$ is a pseudomonad. Although stated in \cite[Example 5.12]{Cruttwell-Shulman10} without a proof, the author has not yet seen a proof of this fact in the literature.
	
\subsection{Monoidal $\V$-categories}
	We start with the monad for monoidal $\V$-categories. Throughout this subsection we assume $\V$ to be a cocomplete symmetric pseudomonoidal category whose colimits are preserved by the binary monoidal product $\tens$ in both variables, so that $\V$\ndash matrices and $\V$-profunctors form equipments $\Mat\V$ and $\enProf\V$; see \exref{example:matrices} and \exref{example:enriched profunctor equipments}. As mentioned above the `free strict monoidal $\V$-category'-monad on $\enProf\V$ will be obtained from a monad on the pseudo double category $\Mat\V$ of $\V$-matrices, namely the `free monoid'-monad, that we will also denote by $\fmc$ and that is given as follows.
	
	Restricted to the objects of $\Mat\V$, which are simply sets, $\map\fmc{\Mat\V}{\Mat\V}$ is the usual `free monoid'-monad: each set $A$ is mapped to the set
\begin{displaymath}
	\fmc A = \coprod_{n \geq 0} A^{\times n}
\end{displaymath}
	of finite sequences of its elements. Such sequences will be denoted by $\ul x = (\ul x_1, \dotsc, \ul x_n)$ and we will write $\card{\ul x} = n$ for the length of $\ul x$; the empty sequence will be denoted $()$. The image $\fmc f$ of a function $\map fAB$ applies $f$ indexwise, so that $(\fmc f \ul x)_i = f\ul x_i$. The image of a $\V$-matrix $\hmap JAB$ is given by
\begin{displaymath}
	\fmc J(\ul x, \ul y) = \begin{cases}
		\fmc_nJ(\ul x, \ul y) & \text{if $\card{\ul x} = n = \card{\ul y}$;} \\
		\emptyset	& \text{otherwise,}
	\end{cases}
\end{displaymath}	
	where $\fmc_n J(\ul x, \ul y) = \Tens^n_{i=1} J(\ul x_i, \ul y_i)$ when $n > 0$ and $\fmc_0((),()) = 1$ (the monoidal unit of $\V$), and where $\emptyset$ is the initial object of $\V$. The assignment $J \mapsto \fmc J$ clearly extends to cells $\cell\phi JK$ of $\V$-matrices by tensoring the components of $\phi$. One readily checks that the assignments above form two endofunctors $\fmc_0$ and $\fmc_1$ on the categories $\Mat\V_0 = \Set$ and $\Mat\V_1$.
	
	Notice that the unitors of $\V$ induce natural isomorphisms $\fmc_U \colon \fmc U_A \iso U_{\fmc A}$ for each set $A$. Likewise, using the fact that $\tens$ preserves coproducts on both sides, the symmetric structure of $\V$ induces natural isomorphisms $\fmc_\hc \colon \fmc J \hc \fmc H \iso \fmc(J \hc H)$. The coherence axioms for $\V$ induce the associativity and unit axioms for $\fmc_U$ and $\fmc_\hc$, and we conclude that $\fmc_0$ and $\fmc_1$ combine to form a pseudo-endofunctor $\fmc$ on the equipment $\Mat\V$. Finally the multiplication $\nat\mu{\fmc^2}\fmc$ and unit $\nat\eta\id\fmc$, that make $\fmc$ into a pseudomonad on $\Mat\V$, are given as follows. Given a set $A$, the set $\fmc^2 A$ consists of sequences of sequences of elements in $A$, which we will call double sequences and denote by $\ull x$. The map $\map{\mu_A}{\fmc^2 A}{\fmc A}$ is given by concatenation of double sequences, while $\map{\eta_A}A{\fmc A}$ maps elements of $A$ to the corresponding one-element sequences. Likewise the cell $\cell{\mu_J}{\fmc^2 J}{\fmc J}$, for a $\V$-matrix $\hmap JAB$, is induced by the isomorphisms
\begin{equation} \label{equation:multiplication of fmc}
	\fmc_{m_1} J(\ull x_1, \ull y_1) \tens \dotsb \tens \fmc_{m_n} J(\ull x_n, \ull y_n) \iso \fmc_{m_1 + \dotsb + m_n} J(\mu_A \ull x, \mu_B \ull y)
\end{equation}
	that are given by the associator of $\V$; the cell $\cell{\eta_J}J{\fmc J}$ is simply given by the identities $J(x, y) = J(\eta_A x, \eta_B y)$. That $\mu$ satisfies the associativity axiom follows form the associativity axiom for $\V$; that the unit axiom holds is clear. This completes the definition of the `free monoid'-monad $\fmc$ on $\Mat\V$.
	
	Having given $\fmc$ we turn to the normal monad $\Mod\fmc$, on the equipment $\Mod{\Mat\V} = \enProf\V$ of $\V$-profunctors, that is induced by $\fmc$ by \propref{mod is a 2-functor}. The monad $\Mod\fmc$ is the `free strict monoidal $\V$-category'-monad; we will again write $\fmc = \Mod\fmc$. It is shown below that $\Mod\fmc$ is a pseudomonad as well. Given a $\V$\ndash category $A$, the $\V$-category $\fmc A$ has finite sequences of $A$\ndash objects as objects, i.e.\ $\ob \fmc A = \coprod_{n \geq 0} (\ob A)^{\times n}$, while its hom-objects are given by
\begin{equation} \label{equation:action of fmc on categories}
	\fmc A(\ul x, \ul y) = \begin{cases}
		\fmc_nA(\ul x, \ul y) & \text{if $\card{\ul x} = n = \card{\ul y}$;} \\
		\emptyset	& \text{otherwise.}
	\end{cases}
\end{equation}
	The composition of $A$ is given indexwise after reordering: 
\begin{displaymath}
	\fmc A(\ul x, \ul y) \tens \fmc A(\ul y, \ul z) \iso \Tens_{i = 1}^n A(\ul x_i, \ul y_i) \tens A(\ul y_i, \ul z_i) \to \Tens_{i=1}^n A(\ul x_i, \ul z_i) = \fmc A(\ul x, \ul z),
\end{displaymath}
	where $\ul x$, $\ul y$ and $\ul z$ are of equal length $n$, so that the identity on $\ul x$ is given by the composition $1 \iso 1^{\tens n} \xrar{\id_{\ul x_1} \tens \dotsb \tens \id_{\ul x_n}} \fmc A(\ul x, \ul x)$. The assignment $A \mapsto \fmc A$ extends to $\V$-functors $\map fAB$ by letting $\fmc f$ act indexwise: the functor \mbox{$\map{\fmc f}{\fmc A}{\fmc B}$} is given on objects by $(\fmc f \ul x)_i = f(\ul x_i)$, while on hom-objects it is given by tensor products of the maps $A(\ul x_i, \ul y_i) \to B\bigpars{f(\ul x_i), f(\ul y_i)}$ that define the action of $f$ on the hom-objects of $A$. Of course the action of $\fmc$ on a $\V$-profunctor \mbox{$\hmap JAB$} is given precisely like \eqref{equation:action of fmc on categories}, when $A$ is replaced by $J$ (both $A$ and $J$ are $\V$-matrices). Likewise, its action on natural transformations between $\V$-profunctors is given by taking tensor products of their components, as in the definition of $\fmc f$. The definitions of the multiplication $\mu$ and unit $\eta$ can be easily read off from those of the multiplication and unit of `free monoid'-monad $\fmc$ on $\Mat\V$ above.
	
	Before describing the algebras for `free strict monoidal $\V$-category'-monad $\fmc = \Mod \fmc$ on $\enProf\V$, we remark that the `free monoid'-monad $\fmc$ on $\Mat\V$ preserves the reflexive coequalisers of $H(\Mat\V)(A, B)$, for any sets $A$ and $B$, so that by \propref{mod is a 2-functor} $\Mod\fmc$ is a pseudomonad. Since such coequalisers are computed pointwise this follows directly from the following proposition, which is \cite[Lemma 2.3.2(a)]{Rezk96}. In general this result is very helpful when considering monads that involve tensor products.
\begin{proposition}[Rezk] \label{tensor products preserve split coequalisers}
	Let $\V$ be a symmetric monoidal category whose tensor product $\map\tens{\V \times \V}\V$ preserves reflexive coequalisers in both variables. For each $i = 1, \dotsc, n$, let $\map{k_i}{Y_i}{Z_i}$ be the coequaliser of the reflexive pair $f_i$, $\mmap{g_i}{X_i}{Y_i}$ in $\V$, with common section $\map{s_i}{Y_i}{X_i}$. Then the induced diagram
	\begin{displaymath}
		X_1 \tens \dotsb \tens X_n \rightrightarrows Y_1 \tens \dotsb \tens Y_n \to Z_1 \tens \dotsb \tens Z_n
	\end{displaymath}
	is again a reflexive coequaliser.
\end{proposition}
	Writing $\D$ for the category generated by three maps $\map u01$, $\map v01$ and $\map s10$, under the relations $u \of s = \id_1 = v \of s$, notice that a parallel pair $f, g \colon X \rightrightarrows Y$ in any category $\C$ is a reflexive pair if and only if there exists a functor $\map R\D\C$ with $Ru = f$ and $Rv = g$; in that case $Rs$ is a common section for $f$ and $g$. The proof of the proposition above follows from the fact that the diagonal functor $\D \to \D \times \dotsb \times \D$ is `final', see \cite[Section IX.3]{MacLane98}.
	
	Recall that $\fmc$ restricts to a $2$-monad $V(\fmc)$ on the $2$-category $\enCat\V = \enProf\V$, of $\V$-categories, $\V$-functors and $\V$-natural transformations, by \propref{normal functors induce vertical functors}. We shall now recall the definition of colax algebras for $2$-monads, as well as those of colax morphisms and cells between them, describing each of these notions in the case of the $2$-monad $V(\fmc)$. Even though the algebras considered there are lax algebras we will follow \cite[Section 1]{Lack02}, using largely the same notation.
	
	Consider a $2$-monad $T = (T, \mu, \eta)$ on a $2$-category $\C$, consisting of a $2$-functor $\map T\C\C$ and $2$-natural transformations $\nat\mu{T^2}T$ and $\nat\eta {\id_\C}T$ satisfying the usual axioms (see \cite[Section VI.1]{MacLane98}): $\mu \of T\mu = \mu \of \mu T$ and $\mu \of T\eta = \id_\C = \mu \of \eta T$.
\begin{definition} \label{definition:colax algebra}
  Let $T = (T, \mu, \eta)$ be a $2$-monad on a $2$-category $\C$. A \emph{colax $T$\ndash algebra} $A$ is a quadruple $A = (A, a, \alpha, \alpha_0)$ consisting of an object $A$ of $\C$ equipped with a morphism $\map a{TA}A$, its \emph{structure map}, and cells $\cell\alpha{a \of \mu_A}{a \of Ta}$ and $\cell{\alpha_0}{a \of \eta_A}{\id_A}$, respectively the \emph{associator} and the \emph{unitor} of $A$. The latter satisfy the following coherence conditions:
  
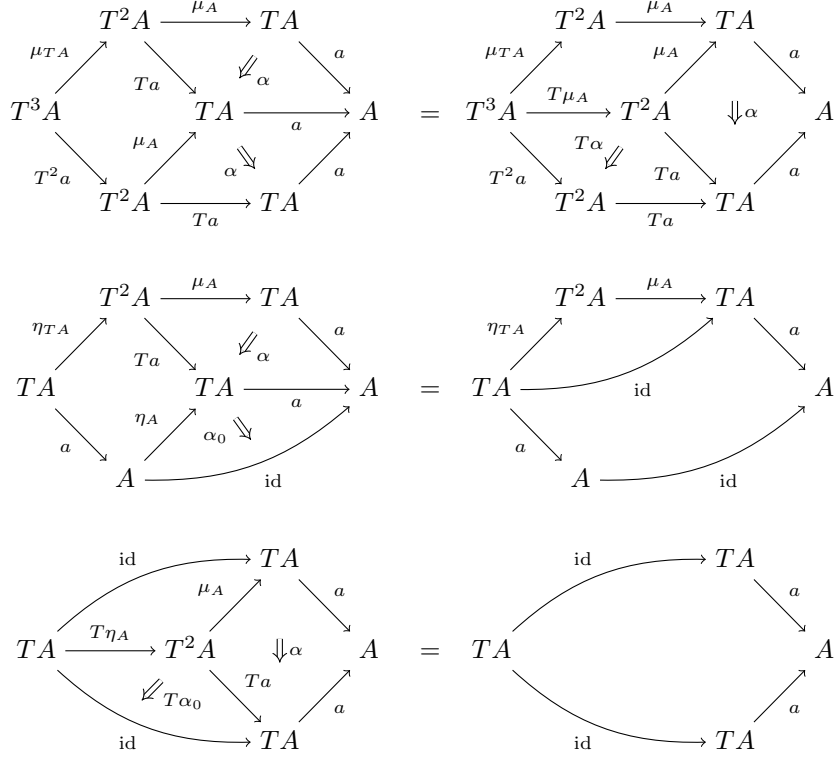
\begin{displaymath}
    \begin{tikzpicture}[textbaseline]
      \matrix(m)[math, column sep=0.75em, row sep=2em]
      { & T^2 A & \phantom{T^2A} &[-0.9em] TA & \phantom{T^2A}\\
        T^3 A & & TA & \phantom{T^2A} & A \\
        & T^2 A & & TA & \\ };
      \path[map]  (m-1-2) edge node[above] {$\mu_A$} (m-1-4)
                          edge node[below left] {$Ta$} (m-2-3)
                  (m-1-4) edge node[above right] {$a$} (m-2-5)
                  (m-2-1) edge node[above left] {$\mu_{TA}$} (m-1-2)
                          edge node[below left] {$T^2 a$} (m-3-2)
                  (m-2-3) edge node[below] {$a$} (m-2-5)
                  (m-3-2) edge node[above left] {$\mu_A$} (m-2-3)
                          edge node[below] {$Ta$} (m-3-4)
                  (m-3-4) edge node[below right] {$a$} (m-2-5)
                  (m-1-4) edge[cell, shorten >= 7pt, shorten <= 7pt] node[below right] {$\alpha$} (m-2-3)
                  (m-2-3) edge[cell, shorten >= 7pt, shorten <= 7pt] node[below left] {$\alpha$} (m-3-4);
    \end{tikzpicture}
    = \begin{tikzpicture}[textbaseline]
      \matrix(m)[math, column sep=0.75em, row sep=2em]
      { & T^2 A &[-0.9em] \phantom{T^2A} & TA & \phantom{T^2A} \\
        T^3 A & & T^2 A & \phantom{T^2A} & A \\
        & T^2 A & & TA & \\ };
      \path[map]  (m-1-2) edge node[above] {$\mu_A$} (m-1-4)
                  (m-1-4) edge node[above right] {$a$} (m-2-5)
                  (m-2-1) edge node[above left] {$\mu_{TA}$} (m-1-2)
                          edge node[above] {$T\mu_A$} (m-2-3)
                          edge node[below left] {$T^2 a$} (m-3-2)
                  (m-2-3) edge node[above left] {$\mu_A$} (m-1-4)
                          edge node[below left] {$Ta$} (m-3-4)
                  (m-3-2) edge node[below] {$Ta$} (m-3-4)
                  (m-3-4) edge node[below right] {$a$} (m-2-5)
                  (m-2-3) edge[cell, shorten >= 7pt, shorten <= 7pt] node[above left] {$T\alpha$} (m-3-2)
                  ($(m-2-3)!0.5!(m-2-5)+(0,0.75em)$) edge[cell] node[right] {$\alpha$} ($(m-2-3)!0.5!(m-2-5)-(0,0.75em)$);
    \end{tikzpicture}
  \end{displaymath}
  \begin{displaymath}
    \begin{tikzpicture}[textbaseline]
      \matrix(m)[math, column sep=0.75em, row sep=2em]
      { \phantom{T^2A} & T^2A & \phantom{T^2A} &[-0.9em] TA & \phantom{T^2A} \\
        TA & & TA & \phantom{T^2A} & A \\
        & A & & & \\ };
      \path[map]  (m-1-2)	edge node[above] {$\mu_A$} (m-1-4)
													edge node[below left] {$Ta$} (m-2-3)
									(m-1-4) edge node[above right] {$a$} (m-2-5)
                  (m-2-1) edge node[above left] {$\eta_{TA}$} (m-1-2)
                          edge node[below left] {$a$} (m-3-2)
                  (m-2-3) edge node[below] {$a$} (m-2-5)
                  (m-3-2) edge node[above left] {$\eta_A$} (m-2-3)
                          edge[bend right=21] node[below right] {$\id$} (m-2-5)
                  (m-1-4) edge[cell, shorten >= 7pt, shorten <= 7pt] node[below right] {$\alpha$} (m-2-3)
                  (m-2-3) edge[cell, shorten >= 12.5pt, shorten <= 4.5pt] node[left, inner sep=8pt] {$\alpha_0$} (m-3-4);
    \end{tikzpicture}
    = \begin{tikzpicture}[textbaseline]
      \matrix(m)[math, column sep=0.75em, row sep=2em]
      { \phantom{T^2A} & T^2A & \phantom{T^2A} &[-0.9em] TA & \phantom{T^2A} \\
        TA & & & \phantom{T^2A} & A \\
        & A & & & \\ };
      \path[map]  (m-1-2)	edge node[above] {$\mu_A$} (m-1-4)
									(m-1-4) edge node[above right] {$a$} (m-2-5)
                  (m-2-1) edge node[above left] {$\eta_{TA}$} (m-1-2)
                          edge node[below left] {$a$} (m-3-2)
                          edge[bend right=21] node[below right] {$\id$} (m-1-4)
                  (m-3-2) edge[bend right=21] node[below right] {$\id$} (m-2-5);
    \end{tikzpicture}
  \end{displaymath}
  \begin{displaymath}
    \begin{tikzpicture}[textbaseline]
      \matrix(m)[math, column sep=0.75em, row sep=2em]
      { \phantom{T^2A} &[-0.9em] \phantom{T^2A} & \phantom{T^2A} & TA & \phantom{T^2A} \\
        TA & & T^2 A & \phantom{T^2A} & A \\
        & & & TA & \\ };
      \path[map]  (m-1-4) edge node[above right] {$a$} (m-2-5)
                  (m-2-1) edge[bend left=22] node[above left] {$\id$} (m-1-4)
                          edge node[above] {$T\eta_A$} (m-2-3)
                          edge[bend right=22] node[below left] {$\id$} (m-3-4)
                  (m-2-3) edge node[above left] {$\mu_A$} (m-1-4)
                          edge node[above right] {$Ta$} (m-3-4)
                  (m-3-4) edge node[below right] {$a$} (m-2-5)
                  (m-2-3) edge[cell, shorten >= 16.5pt, shorten <= 5.5pt] node[right, inner sep=8pt] {$T\alpha_0$} (m-3-2)
                  ($(m-2-3)!0.5!(m-2-5)+(0,0.75em)$) edge[cell] node[right] {$\alpha$} ($(m-2-3)!0.5!(m-2-5)-(0,0.75em)$);
    \end{tikzpicture}
    = \begin{tikzpicture}[textbaseline]
      \matrix(m)[math, column sep=0.75em, row sep=2em]
      { \phantom{T^2A} &[-0.9em] \phantom{T^2A} & \phantom{T^2A} & TA & \phantom{T^2A} \\
        TA & & & \phantom{T^2A} & A \\
        & & & TA & \\ };
      \path[map]  (m-1-4) edge node[above right] {$a$} (m-2-5)
                  (m-2-1) edge[bend left=22] node[above left] {$\id$} (m-1-4)
                          edge[bend right=22] node[below left] {$\id$} (m-3-4)
                  (m-3-4) edge node[below right] {$a$} (m-2-5);
    \end{tikzpicture}
  \end{displaymath}
  We call $A$ a \emph{pseudo} $T$-algebra if the cells $\alpha$ and $\alpha_0$ are isomorphisms; if they are identities then $A$ is called \emph{strict}. If just the unitor $\alpha_0$ is an identity then $A$ is called \emph{normal}.
\end{definition}
	Hence a normal colax $V(\fmc)$-algebra $A$, where $V(\fmc)$ is the vertical restriction of the `free strict monoidal $\V$-category'\ndash monad on $\enProf\V$, as given above, comes equipped with a structure map that is a $\V$-functor $\map a{\fmc A}A$. The latter defines the tensor product $a(\ul x)$ of each sequence $\ul x$ of $A$-objects; we will write $\ul x_1 \tens \dotsb \tens \ul x_n = a(\ul x)$ and $1 = a()$. The fact that $A$ is normal means that $a(\ul x) = \ul x_1$ for each sequence $\ul x$ of length $1$. Moreover $A$ comes with an associator $\nat{\mathfrak a}{a \of \mu_A}{a \of \fmc a}$, that consists of $\V$-natural maps
\begin{multline*}
	\mathfrak a_{\ull x} \colon \ull x_{11} \tens \dotsb \tens \ull x_{1m_1} \tens \dotsb \tens \ull x_{n1} \tens \dotsb \tens \ull x_{nm_n} \\
	\to (\ull x_{11} \tens \dotsb \tens \ull x_{1m_1}) \tens \dotsb \tens (\ull x_{n1} \tens \dotsb \tens \ull x_{nm_n}),
\end{multline*}
	for each double sequence $\ull x$. At $\ull x = \bigpars{(x_1, x_2), (x_3, x_4, x_5), ()}$ for example, the component $\mathfrak a_{\ull x}$ is a map
\begin{displaymath}
	 x_1 \tens x_2 \tens x_3 \tens x_4 \tens x_5 \to (x_1 \tens x_2) \tens (x_3 \tens x_4 \tens x_5) \tens 1.
\end{displaymath}
	
	Roughly speaking the associativity axiom above, for $\mathfrak a$, says that the two ways of adding brackets to the tensor product of $\mu_A \of \mu_{\fmc A}(\ulll x)$, where $\ulll x$ is a triple sequence, either starting with the `outer' brackets or with the `inner' ones, coincide. Formally it states that the following diagram commutes, for each triple sequence $\ulll x$, where we have written $\card{\ulll x} = n$, $\card{\ulll x_i} = m_i$ and $\card{\ulll x_{ij}} = l_{ij}$ and where $\Tens_{j = 1, k = 1}^{j = m_i, k = l_{ij}} \ulll x_{ijk}$ denotes the tensor product $\ulll x_{i11} \tens \dotsb \tens \ulll x_{i1l_{i1}} \tens \dotsb \tens \ulll x_{im_i1} \tens \dotsb \tens \ulll x_{im_il_{im_i}}$. In the top leg the inner brackets are added first while in the bottom one it is the outer brackets.
\begin{displaymath}
	\begin{tikzpicture}
		\matrix(m)[math2em, column sep=2.5em]
		{ \Tens_{i = 1, j = 1, k = 1}^{i = n, j = m_i, k = l_{ij}} \ulll x_{ijk} & \Tens_{i = 1, j = 1}^{i = n, j = m_i} (\Tens_{k=1}^{l_{ij}} \ulll x_{ijk}) \\
			\Tens_{i = 1}^n (\Tens_{j = 1, k = 1}^{j = m_i, k = l_{ij}} \ulll x_{ijk}) &  \Tens_{i = 1}^n \bigpars{\Tens_{j=1}^{m_i} (\Tens_{k=1}^{l_{ij}} \ulll x_{ijk})} \\ };
		\path[map]	(m-1-1) edge node[above] {$\mathfrak a$} (m-1-2)
												edge node[left] {$\mathfrak a$} (m-2-1)
								(m-1-2) edge node[right] {$\mathfrak a$} (m-2-2)
								(m-2-1) edge node[below] {$\Tens_{i=1}^n \mathfrak a$} (m-2-2);
	\end{tikzpicture}
\end{displaymath}
	The unit axioms for $\mathfrak a$ mean that for double sequences $\ull x$ that are of the form $\ull x = \bigpars{(\ull x_{11}, \dotsc, \ull x_{1n})}$ or $\ull x = \bigpars{(\ull x_{11}), \dotsc, (\ull x_{n1})}$ the component $\mathfrak a_{\ull x}$ is the identity.
	
	In the case that $\V = \Set$ the description of normal colax $\fmc$-algebras above is closely related to that of `lax monoidal categories' of e.g.\ \cite[Definition 3.1.1]{Leinster04}. Indeed reversing the direction of the associator $\mathfrak a$ in the discussion above we obtain the notion of a lax monoidal category whose unitors are identities. This is why, for general $\V$, normal colax $\fmc$-algebras are called \emph{normal colax monoidal $\V$\ndash categories}; we will often leave out the prefix `normal'. Returning to $\V = \Set$, pseudo $\fmc$\ndash algebras, i.e.\ whose associators are invertible, are often called \emph{unbiased monoidal categories}. The prefix `unbiased' refers to the fact that the tensor products of all arities are part of the structure of $A$, in contrast to the classical definition of a monoidal category, where only the nullary (unit) and binary tensor products are given. In the case of unenriched monoidal categories it is known that the biased and unbiased definitions are equivalent, see for example \cite[Section 3.2]{Leinster04}. That this result can be extended to the monoidal enriched categories that are treated here seems plausible, but the author does not know of such an extension.
	
	Next is the definition of morphisms between colax algebras.
\begin{definition} \label{definition:colax morphism}
  Given colax $T$-algebras $A = (A, a, \alpha, \alpha_0)$ and $B = (B, b, \beta, \beta_0)$, a \emph{colax $T$-morphism} $\map fAB$ is a morphism $\map fAB$ equipped with a cell $\cell{\bar f}{f \of a}{b \of Tf}$, its \emph{structure cell}, satisfying the following associativity and unit coherence conditions.
  \begin{displaymath}
    \begin{tikzpicture}[textbaseline]
      \matrix(m)[math, column sep=0.75em, row sep=2em]
      { \phantom{T^2B} & TA & \phantom{T^2B} &[-0.9em] A & \phantom{T^2B} \\
        T^2 A & & TB & \phantom{T^2B} & B \\
        & T^2 B & & TB & \\ };
      \path[map]  (m-1-2) edge node[above] {$a$} (m-1-4)
                          edge node[below left] {$Tf$} (m-2-3)
                  (m-1-4) edge node[above right] {$f$} (m-2-5)
                  (m-2-1) edge node[above left] {$\mu_A$} (m-1-2)
                          edge node[below left] {$T^2 f$} (m-3-2)
                  (m-2-3) edge node[below] {$b$} (m-2-5)
                  (m-3-2) edge node[above left] {$\mu_B$} (m-2-3)
                          edge node[below] {$Tb$} (m-3-4)
                  (m-3-4) edge node[below right] {$b$} (m-2-5)
                  (m-1-4) edge[cell, shorten >= 7pt, shorten <= 7pt] node[below right] {$\bar f$} (m-2-3)
                  (m-2-3) edge[cell, shorten >= 7pt, shorten <= 7pt] node[below left] {$\beta$} (m-3-4);
    \end{tikzpicture}
    = \begin{tikzpicture}[textbaseline]
      \matrix(m)[math, column sep=0.75em, row sep=2em]
      { \phantom{T^2B} & TA &[-0.9em] \phantom{T^2B} & A & \phantom{T^2B} \\
        T^2 A & & TA & \phantom{T^2B} & B \\
        & T^2 B & & TB & \\ };
      \path[map]  (m-1-2) edge node[above] {$a$} (m-1-4)
                  (m-1-4) edge node[above right] {$f$} (m-2-5)
                  (m-2-1) edge node[above left] {$\mu_A$} (m-1-2)
                          edge node[above] {$Ta$} (m-2-3)
                          edge node[below left] {$T^2 f$} (m-3-2)
                  (m-2-3) edge node[below right] {$a$} (m-1-4)
                          edge node[below left] {$Tf$} (m-3-4)
                  (m-3-2) edge node[below] {$Tb$} (m-3-4)
                  (m-3-4) edge node[below right] {$b$} (m-2-5)
                  (m-1-2) edge[cell, shorten >= 7pt, shorten <= 7pt] node[above right] {$\alpha$} (m-2-3)
                  (m-2-3) edge[cell, shorten >= 7pt, shorten <= 7pt] node[above left] {$T\bar f$} (m-3-2)
                  ($(m-2-3)!0.5!(m-2-5)+(0,0.75em)$) edge[cell] node[right] {$\bar f$} ($(m-2-3)!0.5!(m-2-5)-(0,0.75em)$);
    \end{tikzpicture}
  \end{displaymath}
  \begin{displaymath}
		\begin{tikzpicture}[textbaseline]
      \matrix(m)[math, column sep=0.75em, row sep=2em]
      { \phantom{T^2B} & TA & \phantom{T^2B} &[-0.9em] A & \phantom{T^2B} \\
        A & \phantom{T^2B} & TB & \phantom{T^2B} & B \\
        & B & & & \\ };
      \path[map]  (m-1-2)	edge node[above] {$a$} (m-1-4)
													edge node[below left] {$Tf$} (m-2-3)
									(m-1-4) edge node[above right] {$f$} (m-2-5)
                  (m-2-1) edge node[above left] {$\eta_A$} (m-1-2)
                          edge node[below left] {$f$} (m-3-2)
                  (m-2-3) edge node[below] {$b$} (m-2-5)
                  (m-3-2) edge node[above left] {$\eta_B$} (m-2-3)
                          edge[bend right=21] node[below right] {$\id$} (m-2-5)
                  (m-1-4) edge[cell, shorten >= 7pt, shorten <= 7pt] node[below right] {$\bar f$} (m-2-3)
                  (m-2-3) edge[cell, shorten >= 12.5pt, shorten <= 4.5pt] node[left, inner sep=8pt] {$\beta_0$} (m-3-4);
    \end{tikzpicture}
    = \begin{tikzpicture}[textbaseline]
      \matrix(m)[math, column sep=0.75em, row sep=2em]
      { \phantom{T^2B} & TA & \phantom{T^2B} &[-0.9em] A & \phantom{T^2B} \\
        A & \phantom{T^2B} & & \phantom{T^2B} & B \\
        & B & & & \\ };
      \path[map]  (m-1-2)	edge node[above] {$a$} (m-1-4)
									(m-1-4) edge node[above right] {$f$} (m-2-5)
                  (m-2-1) edge node[above left] {$\eta_A$} (m-1-2)
                          edge node[below left] {$f$} (m-3-2)
                          edge[bend right=21] node[below right] {$\id$} (m-1-4)
                  (m-3-2) edge[bend right=21] node[below right] {$\id$} (m-2-5)
                  (m-1-2) edge[cell, shorten >= 16.5pt, shorten <= 5.5pt] node[left, inner sep=8pt] {$\alpha_0$} (m-2-3);
    \end{tikzpicture}
  \end{displaymath}
  If the structure cell $\bar f$ is invertible then $f$ is called a \emph{pseudo $T$-morphism}; if it is an identity cell then $f$ is called \emph{strict}.
\end{definition}
	For $T = V(\fmc)$ a colax $V(\fmc)$-morphism \mbox{$\map fAB$} as above, between normal colax monoidal $\V$-categories $A$ and $B$, is a $\V$-functor $\map fAB$ equipped with a $\V$-natural transformation $\nat{f_\tens}{f \of a}{b \of \fmc f}$, the \emph{compositor}, that consists of maps
\begin{displaymath}
	\map{f_\tens}{f(\ul x_1 \tens \dotsb \tens \ul x_n)}{f\ul x_1 \tens \dotsb \tens f\ul x_n}	
\end{displaymath}
	for any sequence $\ul x$; in particular this includes a map $f1_A \to 1_B$. The unit axiom means that $f_\tens$ restricts to the identity on sequences of length $1$, while the associativity axiom states that the diagram below commutes, for each double sequence $\ull x$. Colax $V(\fmc)$-morphisms are called \emph{colax monoidal $\V$-functors}.
\begin{displaymath}
	\wip\begin{tikzpicture}
		\matrix(m)[math2em, column sep=3em, inner sep=6pt]
			{	f\bigpars{\Tens_{i=1, j=1}^{i=n, j=m_i} \ull x_{ij}} & \Tens_{i=1, j=1}^{i=n, j=m_i} f\ull x_{ij} \\
				f\bigpars{\Tens_{i=1}^n (\Tens_{j=1}^{m_i} \ull x_{ij})} & \Tens_{i=1}^n \bigpars{\Tens_{j=1}^{m_i} f\ull x_{ij}} \\ };
			\matrix(n)[math, yshift=-5.5em, inner sep=6pt]{\Tens_{i=1}^n \bigpars{f(\Tens_{j=1}^{m_i} \ull x_{ij})} \\};
			\path[map]	(m-1-1) edge node[above] {$f_\tens$} (m-1-2)
													edge node[left] {$f\mathfrak a_A$} (m-2-1)
									(m-1-2) edge node[right] {$\mathfrak a_B$} (m-2-2)
									(m-2-1) edge node[below left, yshift=2pt] {$f_\tens$} (n-1-1)
									(n-1-1) edge node[below right, yshift=2pt] {$\Tens_{i=1}^n f_\tens$} (m-2-2);
	\end{tikzpicture}
\end{displaymath}
	
	Thus, like pseudo $V(\fmc)$-algebras are unbiased variants of ordinary monoidal $\V$\ndash categories, colax monoidal $\V$-functors between them are unbiased variants of ordinary colax monoidal $\V$-functors which, instead of being equipped with comparison maps $f(x \tens y) \to fx \tens fy$ and $f1_A \to 1_B$, come with comparison maps for each arity $n \geq 0$. 

  This leaves the definition of cells between $T$-morphisms.
\begin{definition} \label{definition:colax cell}
  Given colax $T$-morphisms $f$ and $\map gAB$, a \emph{$T$-cell} between $f$ and $g$ is a cell $\cell\phi fg$ satisfying
  \begin{displaymath}
    \begin{tikzpicture}[textbaseline]
      \matrix(m)[math, column sep=3em, row sep=3em]{TA & A \\ TB & B \\};
      \path[map]  (m-1-1) edge node[above] {$a$} (m-1-2)
                          edge[bend right=35] node[left] {$Tf$} (m-2-1)
                          edge[bend left=35] node[right] {$Tg$} (m-2-1)
                  (m-1-2) edge[bend left=35] node[right] {$g$} (m-2-2)
                  (m-2-1) edge node[below] {$b$} (m-2-2)
                  ($(m-1-1)!0.5!(m-2-1)+(0.75em,-0.3em)$) edge[cell] node[above] {$T\phi$} ($(m-1-1)!0.5!(m-2-1)-(0.75em,0.3em)$)
                  (m-1-2) edge[cell, shorten >= 16.75pt, shorten <= 16.75pt, transform canvas = {xshift=9pt}] node[below right] {$\bar g$} (m-2-1);
    \end{tikzpicture}
    = 
    \begin{tikzpicture}[textbaseline]
			\matrix(m)[math, column sep=3em, row sep=3em]{TA & A \\ TB & B. \\};
      \path[map]  (m-1-1) edge node[above] {$a$} (m-1-2)
                          edge[bend right=35] node[left] {$Tf$} (m-2-1)
                  (m-1-2) edge[bend right=35] node[left] {$f$} (m-2-2)
                          edge[bend left=35] node[right] {$g$} (m-2-2)
                  (m-2-1) edge node[below] {$b$} (m-2-2)
                  ($(m-1-2)!0.5!(m-2-2)+(0.75em,-0.3em)$) edge[cell] node[above] {$\phi$} ($(m-1-2)!0.5!(m-2-2)-(0.75em,0.3em)$)
                  (m-1-2) edge[cell, shorten >= 17pt, shorten <= 17pt, transform canvas = {xshift=-5pt}] node[above left] {$\bar f$} (m-2-1);
    \end{tikzpicture}
  \end{displaymath}
\end{definition}
	In particular a $V(\fmc)$-cell $\cell\phi fg$ between colax monoidal $\V$-functors $f$ and $g$, called a \emph{monoidal $\V$-natural transformation}, is simply a $\V$-natural transformation $\nat\phi fg$ that is compatible with the structure transformations $f_\tens$ and $g_\tens$: the diagrams
	\begin{displaymath}
		\begin{tikzpicture}
			\matrix(m)[math2em]
			{ f(\ul x_1 \tens \dotsb \tens \ul x_n) & f\ul x_1 \tens \dotsb \tens f\ul x_n \\
				g(\ul x_1 \tens \dotsb \tens \ul x_n) & g\ul x_1 \tens \dotsb \tens g\ul x_n \\ };
			\path[map]	(m-1-1) edge node[above] {$f_\tens$} (m-1-2)
													edge node[left] {$\phi_{\ul x_1 \tens \dotsb \tens \ul x_n}$} (m-2-1)
									(m-1-2) edge node[right] {$\phi_{\ul x_1} \tens \dotsb \tens \phi_{\ul x_n}$} (m-2-2)
									(m-2-1) edge node[below] {$g_\tens$} (m-2-2);
		\end{tikzpicture}
	\end{displaymath}
	commute for every sequence $\ul x$.
	
	For any $2$-monad $T$, we shall write $\cAlg T$ for the $2$-category of colax $T$-algebras, colax $T$-morphisms and $T$-cells, and $\map{U^T}{\cAlg T}\C$ for the forgetful $2$-functor. Recall that every normal monad $T$ on an equipment $\K$ induces a $2$-monad $V(T)$ on the vertical $2$-category $V(\K)$ that is contained in $\K$, by \propref{normal functors induce vertical functors}.
\begin{definition} \label{definition:colax V(T)-algebras and morphisms}
	Let $T$ be a normal monad on an equipment $\K$. By \emph{colax $T$\ndash al\-ge\-bras} (in $\K$), \emph{colax $T$-morphisms} and \emph{vertical $T$-cells} between them, we simply mean colax $V(T)$\ndash al\-ge\-bras, colax $V(T)$-morphisms and $V(T)$-cells (in $V(\K)$).
\end{definition}
	Thus in case $T = \fmc$, the `free strict monoidal $\V$-category'-monad, normal colax $\fmc$-algebras and colax $\fmc$-morphisms are respectively normal colax monoidal $\V$-categories and colax monoidal $\V$-functors, as we have seen.

\subsection{Symmetric monoidal $\V$-categories}
	Here we discuss the `free symmetric strict monoidal $\V$-category'-monad $\fsmc$ on the equipment $\enProf\V$ of $\V$-profunctors, which is also briefly mentioned in \cite[Example 3.14]{Cruttwell-Shulman10} and whose restriction to $\enCat\V$ is described in \cite[page 683]{Getzler09}. Its colax algebras are `colax symmetric monoidal $\V$-categories', as we will see. Like the `free strict monoidal $\V$-category'-monad $\fmc$, the monad $\fsmc$ is also induced by a monad on $\Mat\V$, also denoted $\fsmc$, that we will introduce first.
	
	On $\Mat\V_0 = \Set$ we take $\fsmc_0 = \fmc_0$ while the endofunctor $\fsmc_1$ on $\Mat\V_1$ is given as follows. Let $\sym_n$ be the group of permutations of the set $\set{1, \dotsc, n}$; given $\sigma \in \sym_n$ we will write $\act{\ul x}\sigma$ for the permuted sequence given by $(\act{\ul x}\sigma)_i = \ul x_{\sigma i}$. This gives a right action on the sequences of length $n$, i.e.\ $\act{(\act{\ul x}\sigma)}\tau = \act{\ul x}(\sigma \of \tau)$. The image $\fsmc J$ of a $\V$-matrix $\hmap JAB$ is given by
	\begin{displaymath}
		\fsmc J(\ul x, \ul y) = \begin{cases}
			\coprod_{\sigma \in \Sigma_n} \fmc_n J(\act{\ul x}\sigma, \ul y) & \text{if $\card{\ul x} = n = \card{\ul y}$}; \\
			\emptyset & \text{otherwise,}
		\end{cases}
	\end{displaymath}
	where $\fmc_n J(\act{\ul x}\sigma, \ul y) = \Tens_{i=1}^n J(\ul x_{\sigma i}, \ul y_i)$. As before the image $\fsmc \phi$ of a cell $\cell \phi JK$ is given by taking tensor products of the components of $\phi$; this extends $J \mapsto \fsmc J$ to a endofunctor $\fsmc_1$ on $\Mat\V_1$. 
	
	To combine the endofunctors $\fsmc_0$ and $\fsmc_1$ into a lax endofunctor $\fsmc$ on $\Mat\V$ we have to supply horizontal cells $\cell{\fsmc_\hc}{\fsmc J \hc \fsmc K}{\fsmc(J \hc K)}$, for $\V$-matrices $\hmap JAB$ and $\hmap KBC$, that form the compositor of $\fsmc$, as well as horizontal cells $\cell{\fsmc_U}{U_{\fsmc A}}{\fsmc U_A}$, for each set $A$, that form the unitor of $\fsmc$. For the compositor notice that there is an induced action of $\sym_n$ on the tensor products  $\fmc_n J(\ul x, \ul y)$ given by the isomorphisms
	\begin{displaymath}
		\sigma \colon \fmc_n J(\ul x, \ul y) \xrar\iso \fmc_n J(\act{\ul x}\sigma, \act{\ul y}\sigma).
	\end{displaymath}
	that permute the factors, using the symmetric structure of $\V$. Using this we define the components of $\fsmc_\hc$ to be the compositions below, for each pair of sequences $\ul x$ and $\ul z$ of length $n$.
	\begin{align*}
		(\fsmc J \hc \fsmc K)(\ul x, \ul z) &\iso \coprod_{\substack{\rho, \sigma \in \sym_n \\ \ul y \in \fmc_n B}} \fmc_n J(\act{\ul x}\rho, \ul y) \tens \fmc_n K(\act{\ul y}\sigma, \ul z) \\
		&\xrar{\coprod{\sigma \tens \id}} \coprod_{\substack{\rho, \sigma \in \sym_n \\ \ul y \in \fmc_n B}} \fmc_n J(\act{\ul x}{(\rho \of \sigma)}, \act{\ul y}\sigma) \tens \fmc_n K(\act{\ul y}\sigma, \ul z) \\
		&\to \coprod_{\substack{\tau \in \sym_n \\ \ul w \in \fmc_n B}} \fmc_n J(\act{\ul x}\tau, \ul w) \tens \fmc_n K(\ul w, \ul z) \iso \fsmc(J \hc K)(\ul x, \ul z)
	\end{align*}
	Here the isomorphisms exists by the symmetric structure of $\V$ and the fact that $\tens$ preserves coproducts in both variables, while the unlabelled map is induced by inserting each component, that is indexed by $\rho$, $\sigma$ and $\ul y$, at $\tau = \rho \of \sigma$ and $\ul w = \act{\ul y}\sigma$. Checking that this makes $\fsmc_\hc$ a natural transformation, and that it satisfies the associativity axiom, is straightforward. The component of the unitor $\fsmc_U$ for a set $A$ is given by
	\begin{displaymath}
		U_{\fsmc A}(\ul x, \ul y) = 1 \iso 1^{\tens n} \to \fsmc U_A(\ul x, \ul y)
	\end{displaymath}
	if $\ul x = \ul y$, where the unlabelled map is insertion at $\sigma = \id$; if $\ul x \neq \ul y$ then it is simply the initial map $\emptyset \to \fsmc U_A(\ul x, \ul y)$. Checking that, with this definition, the unitors and associators satisfy the unit axioms is again straightforward. This concludes the definition of the lax endofunctor $\fsmc$ on $\Mat\V$.
	
	The endofunctor $\fsmc$ is made into a lax monad as follows. Its multiplication  $\nat{\mu_A}{\fsmc^2 A}{\fsmc A}$ for a set $A$ concatenates the double sequences of elements in $A$. To define the cell $\cell{\mu_J}{\fsmc^2 J}{\fsmc J}$ for a $\V$-matrix $\hmap JAB$ notice that $\fsmc^2 J$ is given by
\begin{displaymath}
	\fsmc^2 J(\ull x, \ull y) = \coprod_{\sigma \in \sym_n(\ull x, \ull y)} \Tens_{i = 1}^n \coprod_{\sigma_i \in \sym_{m_i}} \Tens_{j_i = 1}^{m_i} J\bigpars{(\ull x_{\sigma i})_{\sigma_i j_i}, (\ull y_i)_{j_i}},
\end{displaymath}
	where $\sym_n(\ull x, \ull y) \subset \sym_n$ is the subgroup of permutations $\sigma$ for which $\card{\ull x_{\sigma i}} = m_i = \card{\ull y_i}$, for each index $i$. In here we can rewrite $(\ull y_i)_{j_i} = (\mu_B \ull y)_{k(i, j_i)}$ with
\begin{displaymath}
	k(i, j_i) = m_1 + \dotsb + m_{i-1} + j_i;
\end{displaymath}
	similarly we have $(\ull x_{\sigma i})_{\sigma_i j_i} = (\mu_A \ull x)_{(\sigma_1, \dotsc, \sigma_n) \of \sigma_{(m_1, \dotsc, m_n)} k(i, j_i)}$, where the permutations $\sigma_{\pars{m_1, \dotsc, m_n}}$ and $\pars{\sigma_1, \dotsc, \sigma_n}$ are given as follows. Writing $N = \Sum_i m_i$, the \emph{block permutation} $\sigma_{\pars{m_1, \dotsc, m_n}} \in \sym_N$ permutes the $n$ blocks of the $\pars{m_1, \dotsc, m_n}$\ndash par\-ti\-tion of $\set{1, \dotsc, N}$ exactly like $\sigma$ permutes the elements of $\set{1, \dotsc, n}$, while preserving the ordering of the elements in each block. The permutation $\pars{\sigma_1, \dotsc, \sigma_n}$ of $\sym_N$ is simply the disjoint union of the $\sigma_i$, letting each $\sigma_i$ act on the subset
	\begin{displaymath}
		\set{l :  \Sum_{i' < i} m_{i'} + 1 \leq l \leq \Sum_{i' \leq i} m_{i'}} \subset \set{1, \dotsc, N}.
	\end{displaymath}
	It follows that, using the fact that $\tens$ preserves coproducts in both variables, we can rewrite $\fsmc^2 J(\ull x, \ull y)$ as
\begin{equation} \label{equation:rewriting double fsmc image}
	\fsmc^2 J(\ull x, \ull y) \iso \coprod_{\substack{\sigma \in \sym_n(\ull x, \ull y) \\ \sigma_i \in \sym_{m_i}}} \fmc_N J\Bigpars{\act{\mu_A(\ull x)}{\bigpars{(\sigma_1, \dotsc, \sigma_n) \of \sigma_{(m_1, \dotsc, m_n)}}}, \mu_B\ull y}
\end{equation}
	the multiplication $\map{\mu_J}{\fsmc^2 J}{\fsmc J}$ is given by this isomorphism followed by inserting each component above, that is indexed by $(\sigma, \sigma_1, \dotsc, \sigma_n)$, into $\fsmc J(\mu_A \ull x, \mu_B \ull y)$ as the component indexed by $(\sigma_1, \dotsc, \sigma_n) \of \sigma_{(m_1, \dotsc, m_n)}$. That the components $\mu_J$ are natural with respect to the cells of $\Mat\V$, and that they satisfy the unit axiom (\defref{definition:transformation}), is clear. To see that they satisfy the composition axiom as well, that is
\begin{displaymath}
	\begin{tikzpicture}
		\matrix(m)[math2em]{ \fsmc^2 J \hc \fsmc^2 K & \fsmc^2(J \hc K) \\ \fsmc J \hc \fsmc K & \fsmc (J \hc K) \\};
		\path	(m-1-1) edge[cell] node[above] {$\fsmc^2_\hc$} (m-1-2)
									edge[cell] node[left] {$\mu_J \hc \mu_K$} (m-2-1)
					(m-1-2) edge[cell] node[right] {$\mu_{J \hc K}$} (m-2-2)
					(m-2-1) edge[cell] node[below] {$\fsmc_\hc$} (m-2-2);
	\end{tikzpicture}
\end{displaymath}
	commutes for every pair of $\V$-matrices $\hmap JAB$ and $\hmap KBC$, notice that under the isomorphisms \eqref{equation:rewriting double fsmc image} the action of $\fsmc^2_\hc = \fsmc\fsmc_\hc \of \fsmc_\hc (\fsmc \times \fsmc)$ is given as follows. Consider the component of $(\fsmc^2 J \hc \fsmc^2 K)(\ull x, \ull z)$ that is indexed by $\ull y \in \fmc_N B$ (for the horizontal composite), $(\sigma, \sigma_1, \dotsc, \sigma_n)$ (for $\fsmc^2 J$) and $(\tau, \tau_1, \dotsc, \tau_n)$ (for $\fsmc^2 K$). The first factor of $\fsmc^2_\hc$ permutes the tensor product of $J$-components in $\fsmc^2 J \hc \fsmc^2 K$ with $\tau_{(k_1, \dotsc, k_n)}$, where $\card{\ull z_i} = k_i$, while the second factor permutes the same product with $(\tau_1, \dotsc, \tau_m)$. We conclude that the top leg of the diagram above is given by permuting the $J$-components with $(\tau_1, \dotsc, \tau_m) \of \tau_{(k_1, \dotsc, k_n)}$, followed by insertion into $\fsmc(J \hc K)$, which is exactly what the bottom leg does.

	The components $\map{\eta_A}A{\fsmc A}$ of the unit, for any set $A$, are simply given by mapping the elements of $A$ to their corresponding $1$-element sequences; the cells $\eta_J$ for $\V$-matrices $J$ are given by inserting the components of $J$ into those of $\fsmc J$. That this makes $\eta$ into a well-defined transformation $\id \natarrow \fsmc$ is clear. Finally, to see that $\mu$ and $\eta$ satisfy the associativity and unit axioms of a monad, notice that their actions on the permutations, that are the indices of the coproducts in $\fsmc^2 J$ and $\fsmc J$, coincide with the composition and unit of the `operad of symmetries', see e.g.\ \cite[Example 2.2.20]{Leinster04}: the operadic associativity and unit axioms for this operad imply the associativity and unit axioms for $\mu$ and $\eta$. This completes the definition of the lax monad $\fsmc$ on the equipment $\Mat\V$ of $\V$-matrices.
	
	We will often use the fact that there exists a transformation $\nat{\theta}\fmc\fsmc$ embedding the `free monoid'-monad $\fmc$ on $\Mat\V$ into $\fsmc$, that restricts to the identity on $\Mat\V_0 = \Set$ and whose cells $\cell{\theta_J}{\fmc J}{\fsmc J}$ insert at $\sigma = \id$. Notice that $\theta$ is compatible with the multiplications and units, in the sense that $\theta \of \mu_\fmc = \mu_\fsmc \of \theta^2$ (where $\theta^2 = \theta\fsmc \of \fmc\theta$) and $\theta \of \eta_\fmc = \eta_\fsmc$, i.e.\ $\theta$ is a \emph{morphism of monads}.
	
	As mentioned in the introduction, we are interested in the monad $\Mod\fsmc$ on the equipment $\enProf\V$ of $\V$-profunctors that is induced by $\fsmc$, see \propref{mod is a 2-functor}. The monad $\Mod\fsmc$ is the `free symmetric strict monoidal $\V$-category'-monad; we will again write $\fsmc = \Mod\fsmc$. Its image $\fsmc A$ of a $\V$-category $A$ has finite sequences of elements of $A$ as objects, while the hom-objects $\fsmc A(\ul x, \ul y)$ are given by
	\begin{displaymath}
		\fsmc A(\ul x, \ul y) = \begin{cases}
			\coprod_{\sigma \in \Sigma_n} \fmc_n A(\act{\ul x}\sigma, \ul y) & \text{if $\card{\ul x} = n = \card{\ul y}$}; \\
			\emptyset & \text{otherwise.}
		\end{cases}
	\end{displaymath}
	Thus generalised morphisms $\ul x \to \ul y$ in $\fsmc A$ (that is $\V$-morphisms $1 \to \fsmc A(\ul x, \ul y)$) can be thought of as a permutation of the sequence $\ul x$, followed by a sequence of maps from the permuted $\ul x$ to $\ul y$. Given $\ul x$, $\ul y$ and $\ul z$ of equal length $n$, the composition $\fsmc A(\ul x, \ul y) \tens \fsmc A(\ul y, \ul z) \to \fsmc A(\ul x, \ul z)$ is determined by
	\begin{multline*}
		\fmc_n A(\act{\ul x}\sigma, \ul y) \tens \fmc_n A(\act{\ul y}\tau, \ul z) \xrar{\tau \tens \id} \fmc_n A\bigpars{\act{\ul x}{(\sigma \of \tau)}, \act{\ul y}\tau} \tens \fmc_n A(\act{\ul y}\tau, \ul z) \\ \to \fmc_n A\bigpars{\act{\ul x}{(\sigma \of \tau)}, \ul z},
	\end{multline*}
	where the second map is given by composition in $A$. Given a sequence $x$ of length $n$ and a permutation $\sigma$ in $\sym_n$ we denote by $\map{\sigma_{\ul x}}{\ul x}{\act{\ul x}\sigma}$ the isomorphism in $\fsmc A$ that is given by the compostion
	\begin{equation} \label{equation:permutation in fsmc A}
		1 \iso 1^{\tens n} \xrar{\Tens_i \id_{\ul x_i}} \fmc_n A(\ul x, \ul x) \xrar\sigma \fmc_n A(\act{\ul x}\sigma, \act{\ul x}\sigma) \to \fsmc A(\ul x, \act{\ul x}\sigma),
	\end{equation}
	where the last map is insertion at $\sigma$. Note that the assignment $(\sigma, \ul x) \mapsto \sigma_{\ul x}$ is functorial in the sense that $\tau_{\act{\ul x}\sigma} \of \sigma_{\ul x} = (\sigma \of \tau)_{\ul x}$, while taking $\sigma = \id$ gives the identity for $\ul x$. In particular $(\inv\sigma)_{\act{\ul x}\sigma}$ is the inverse of $\sigma_{\ul x}$.
	
	Given a $\V$-functor $\map fAB$ the image $\map{\fsmc f}{\fsmc A}{\fsmc B}$ is given on objects by $(\fsmc f \ul x)_i = f(\ul x_i)$, while on hom-objects it is determined by tensor products of the maps $A(\ul x_{\sigma i}, \ul y_i) \to B\bigpars{f(\ul x_{\sigma i}), f(\ul y_i)}$ that define the action of $f$ on the hom-objects of $A$. The formula for the action of $\fsmc$ on $\V$-profunctors is identical to that on the hom-objects of $\V$-categories above. Likewise the actions of $\op{(\fsmc A)}$ and $\fsmc B$ on $\fsmc J$ are given in a similar fashion to the composition maps of $\fsmc A$ above, using the actions of $\op A$ and $B$ on $J$ instead of the composition maps of $A$. Finally the images of $\V$-natural transformations of $\V$-profunctors are given by taking tensor products of their components.

	Notice that, by \propref{mod is a 2-functor}, the monad $\fsmc$ is in fact normal. More is true: the following was stated in \cite[Example 5.12]{Cruttwell-Shulman10} without proof.
\begin{proposition} \label{fsmc is a pseudofunctor}
	The `free symmetric strict monoidal $\V$-category'-monad $\fsmc$ on $\enProf\V$, as described above, is a pseudomonad.
\end{proposition}

	Before giving a proof we will describe the algebras of $\fsmc$. Notice that the embedding $\nat\theta\fmc\fsmc$ of monads on $\Mat\V$ induces a morphism of monads $\nat{\theta = \Mod\theta}{\Mod\fmc}{\Mod\fsmc}$ on $\enProf\V$ and therefore a morphism of $2$-monads $\nat{V(\theta)}{V(\fmc)}{V(\fsmc)}$ on $\enCat\V = V(\enProf\V)$. In general, such a morphism of monads induces a functor $\cAlg{V(S)} \to \cAlg{V(M)}$ between the categories of colax algebras over $V(S)$ and $V(M)$, see \cite[Section 6.1]{Leinster04}. It follows that the algebra structure on a normal colax $\fsmc$-algebra $A = (A, a', \mathfrak a')$ (see \defref{definition:colax algebra}) restricts to normal colax monoidal $\V$-category structure on $A$, whose tensor products are given by the $\V$-functor $a = \brks{\fmc A \xrar{\theta_A} \fsmc A \xrar{a'} A}$ and whose associator is the $\V$\ndash natural transformation $\nat{\mathfrak a = \mathfrak a' \of \theta^2}{a \of \mu_A}{a \of \fmc a}$.
	
	Furthermore, given a sequence $\ul x$ of length $n$ and a permutation $\sigma$ in $\sym_n$, we write $\map{(\mathfrak s_\sigma)_{\ul x} = a'(\sigma_{\ul x})}{\ul x_1 \tens \dotsb \tens \ul x_n}{\ul x_{\sigma 1} \tens \dotsb \tens \ul x_{\sigma n}}$, where $\map{\sigma_{\ul x}}{\ul x}{\act{\ul x}\sigma}$ is the isomorphism in $\fsmc A$ that is given by the composition \eqref{equation:permutation in fsmc A}. Writing $a_n$ for the restriction of $a$ to the full subcategory $\fmc_n A$ of $\fmc A$ consisting of sequences of length $n$, checking that the maps $(\mathfrak s_\sigma)_{\ul x}$ combine to form a $\V$-natural isomorphism \mbox{$\nat{\mathfrak s_\sigma}{a_n}{a_n \of \sigma}$}, where the action of $\sigma$ on $\fmc_n A$ is given by permutation of tensor factors, is straightforward. The functoriality of \eqref{equation:permutation in fsmc A} implies that $\sigma \mapsto \mathfrak s_\sigma$ is functorial in the sense that $(\mathfrak s_\tau \of \sigma) \of \mathfrak s_\sigma = \mathfrak s_{\sigma \of \tau}$, and $\mathfrak s_{\id} = \id$. Thus the original $\fsmc$-algebra structure $\V$-functor $\map{a'}{\fsmc A}A$ restricts to a $\V$-functor $\map a{\fmc A}A$ that makes $A$ into a colax monoidal $\V$-category, as well as coherent natural isomorphisms $\mathfrak s_\sigma$ that permute the factors of this monoidal structure. In fact giving $a'$ is equivalent to giving both $a$ and natural transformations $\mathfrak s_\sigma$, provided that they satisfy the properties just described. Indeed this follows from the fact that they are related by the commuting of the following diagram.
\begin{equation} \label{diagram:monoidal structure determines symmetric monoidal structure}
	\begin{split}
	\begin{tikzpicture}
		\matrix(m)[math2em]
		{ \fmc_n A(\act{\ul x}\sigma, \ul y) \nc A\bigpars{a_n(\act{\ul x}\sigma), a_n (\ul y)} \\
			\fsmc A(\ul x, \ul y) \nc A\bigpars{a_n (\ul x), a_n (\ul y)} \\ };
		\path[map]	(m-1-1) edge node[above] {$a_n$} (m-1-2)
												edge[linj] (m-2-1)
								(m-1-2) edge node[right] {$A(\mathfrak s_\sigma, \id)$} (m-2-2)
								(m-2-1) edge node[below] {$a'$} (m-2-2);
	\end{tikzpicture}
	\end{split}
\end{equation}	
	
	Since the embedding $\nat\theta\fmc\fsmc$ is the identity on objects, the original associator $\mathfrak a'$ is completely determined by its restriction $\mathfrak a = \mathfrak a' \of \theta^2$. In particular the associativity axioms for  $\mathfrak a'$ and its restriction $\mathfrak a$ are equivalent. Moreover the $\V$-naturality of $\mathfrak a'$, as a transformation of $\V$-functors $\fsmc A \natarrow A$, splits into the $\V$-naturality of $\mathfrak a$ and the naturality of $\mathfrak a$ with respect to the symmetries $\mathfrak s_\sigma$. The latter means that for each double sequence $\ull x$, with $\card{\ull x} = n$ and $\card{\ull x_i} = m_i$, as well as permutations $\sigma \in \sym_n$ and $\tau_i \in \sym_{m_i}$, for each $i = 1, \dotsc, n$, the following diagram commutes. Here $\Tens_{i=1, j=1}^{i= n, j=m_i} \ull x_{ij}$ denotes the tensor product $\ull x_{11} \tens \dotsb \tens \ull x_{1m_1} \tens \dotsb \tens \ull x_{n1} \tens \dotsb \tens \ull x_{nm_n}$, and $(\tau_1, \dotsc, \tau_n) \of \sigma_{(m_1, \dotsc, m_n)}$ is the permutation that was used in the definition of the multiplication $\mu$ of the monad $\fsmc$.
\begin{displaymath}
	\begin{tikzpicture}
		\matrix(m)[math2em, column sep=4.5em, inner sep=6pt]
		{ \Tens_{i=1, j=1}^{i= n, j=m_i} \ull x_{ij} & \Tens_{i=1}^n \bigpars{\Tens_{j=1}^{m_i} \ull x_{ij}} \\
			\Tens_{i=1, j=1}^{i=n, j=m_{\sigma i}} \ull x_{(\sigma i)(\tau_{\sigma i} j)} & \Tens_{i=1}^n \bigpars{\Tens_{j=1}^{m_i} \ull x_{i(\tau_i j)}} \\ };
		\matrix(n)[math, yshift=-5em, inner sep=6pt]{\Tens_{i=1}^n \bigpars{\Tens_{j=1}^{m_{\sigma i}} \ull x_{(\sigma i)(\tau_{\sigma i} j)}} \\ };
		\path[map]	(m-1-1) edge node[above] {$\mathfrak a$} (m-1-2)
												edge node[left] {$\mathfrak s_{\sigma_{(m_1, \dotsc, m_n)} \of (\tau_1, \dotsc, \tau_n)}$} (m-2-1)
								(m-1-2) edge node[right] {$\Tens_{i=1}^n \mathfrak s_{\tau_i}$} (m-2-2)
								(m-2-1) edge node[below left, yshift=2pt] {$\mathfrak a$} (n-1-1)
								(m-2-2) edge node[below right, yshift=2pt] {$\mathfrak s_\sigma$} (n-1-1);
	\end{tikzpicture}
\end{displaymath}

	Summarising, a normal colax $\fsmc$-algebra $A$ is a normal colax monoidal $\V$-category $A$ equipped with $\V$-natural symmetry isomorphisms $\nat{\mathfrak s_\sigma}{a_n}{a_n \of \sigma}$, for each permutation $\sigma$ in $\sym_n$, where $a_n$ denotes the restriction of the structure $\V$-functor $\map a{\fmc A}A$ to the full subcategory $\fmc_n A$ of sequences of length $n$. The assignment $\sigma \mapsto \mathfrak s_\sigma$ is required to be functorial in the sense that $(\mathfrak s_\tau \of \sigma) \of \mathfrak s_\sigma = \mathfrak s_{\sigma \of \tau}$, and $\mathfrak s_{\id} = \id$, while the associator $\nat{\mathfrak a}{a \of \mu_A}{a \of \fmc a}$ is required to be natural with respect to the symmetries $\mathfrak s_\sigma$, in that the diagrams above commute. Normal colax $\fsmc$-algebras are called \emph{symmetric normal colax monoidal $\V$-categories}.
	
	Likewise a colax $\fsmc$-morphism $\map fAB$ between symmetric normal colax monoidal $\V$-categories $A$ and $B$ is a colax monoidal $\V$-functor $\map fAB$ whose $\V$-natural compositor $\nat{f_\tens}{f \of a}{b \of \fmc f}$ is natural with respect to the symmetries of $A$ and $B$, in the sense that the diagrams below commute for any sequence $\ull x$ of length $n$ and permutation $\sigma$ in $\sym_n$. Colax $\fsmc$-morphisms are called \emph{symmetric colax monoidal $\V$-functors}. Finally $\fsmc$-cells between such functors are simply monoidal $\V$-natural transformations.
\begin{displaymath}
	\begin{tikzpicture}
		\matrix(m)[math2em]
		{	f(\ul x_1 \tens \dotsb \tens \ul x_n) & f\ul x_1 \tens \dotsb \tens f\ul x_n \\
			f(\ul x_{\sigma 1} \tens \dotsb \tens \ul x_{\sigma n}) & f\ul x_{\sigma 1} \tens \dotsb \tens f\ul x_{\sigma n} \\ };
		\path[map]	(m-1-1) edge node[above] {$f_\tens$} (m-1-2)
												edge node[left] {$f\mathfrak s_\sigma$} (m-2-1)
								(m-1-2) edge node[right] {$\mathfrak s_\sigma$} (m-2-2)
								(m-2-1) edge node[below] {$f_\tens$} (m-2-2);
	\end{tikzpicture}
\end{displaymath}
	
	We now turn to the proof of \propref{fsmc is a pseudofunctor} which, the author believes, has not appeared in the literature before. Where Cruttwell and Shulman remark in \cite[Example 5.12]{Cruttwell-Shulman10} that such a proof can be given using ``a more involved computation with coequalisers'', we will give a simple proof that uses the fact that $\fmc$ is a pseudomonad together with the following property of the embedding $\nat\theta\fmc\fsmc$. Remember that the cell $\theta_J$, where $J$ is a $\V$-profunctor \mbox{$A \slashedrightarrow B$}, corresponds under \propref{left and right cells} to a horizontal cell $\cell{\rho \theta_J}{\fsmc A(\id, \theta_A) \hc_{\fmc A} \fmc J}{\fsmc J(\id, \theta_B)}$, and that $\theta_J$ is called right invertible whenever $\rho \theta_J$ is invertible.
\begin{proposition} \label{theta cells are right invertible}
	For each $\V$-profunctor $J$ the cell $\theta_J$ is right invertible.
\end{proposition}
\begin{proof}
	Recall that horizontal composites of $\V$-profunctors are given by the reflexive coequalisers \eqref{equation:composition of profunctors}; in particular the source $\fsmc A(\id, \theta_A) \hc_{\fmc A} \fmc J$ of $\rho\theta_J$ is the reflexive coequaliser of the two parallel maps below, that are given by letting $\fmc A$ act on $\fsmc A(\id, \theta_A)$ and $\fmc J$ respectively. 
	\begin{multline*}
		\coprod_{\substack{\sigma \in \sym_n \\ \ul u, \ul v \in \fmc_n A}} \fmc_n A(\act{\ul x}\sigma, \ul u) \tens \fmc_n A(\ul u, \ul v) \tens \fmc_n J(\ul v, \ul z) \\
		\rightrightarrows \coprod_{\substack{\sigma \in \sym_n \\ \ul y \in \fmc_n A}} \fmc_n A(\act{\ul x}\sigma, \ul y) \tens \fmc_n J(\ul y, \ul z) \to \coprod_{\sigma \in \sym_n} \fmc_n J(\act{\ul x}\sigma, \ul z)
	\end{multline*}
	Working out the definition of the horizontal cell $\rho\theta_J$ (\propref{left and right cells}), we find that its component for $(\ul x, \ul z)$ is induced by the second, single map above, which is given by letting $\fmc_n A$ act on $\fmc_n J$. But this means that, by the definition of the enriched Yoneda isomorphism (see \exref{example:unenriched profunctors} for the analogous unenriched variant), the diagram above is the coproduct of a number of reflexive coequalisers, one for each $\sigma \in \sym_n$, and hence the diagram itself is a reflexive coequaliser. From this we conclude that each component of $\rho\theta_J$ is an isomorphism.
\end{proof}
\begin{proof}[Proof of \propref{fsmc is a pseudofunctor}]
	Consider $\V$-profunctors $\hmap JAB$ and $\hmap KBC$; we will show that the compositor $\cell{\fsmc_\hc}{\fsmc J \hc_{\fsmc B} \fsmc K}{\fsmc (J \hc_B K)}$, which is a $\V$\ndash natural transformation $\op{(\fsmc A)} \tens \fsmc C \to \V$, is invertible, that is each of the components of $\fsmc_\hc$ is an isomorphism. Notice that since the embedding $\nat\theta\fmc\fsmc$ is the identity on objects we may equivalently prove that all of the components of the composite $\fsmc_\hc \of (\id \tens \theta_C)$ are isomorphisms. Applying $\rho$ to the composition axiom for $\theta$ (\defref{definition:transformation}), where we use the functoriality of $\rho$ (\propref{left and right cells}), we obtain the following commuting diagram of $\V$-natural transformations.
	\begin{displaymath}
		\begin{tikzpicture}
			\matrix(m)[math2em, row sep=2.5em]
			{ \fsmc J \hc_{\fsmc B} \fsmc K \hc_{\fsmc C} \fsmc C(\id, \theta_C) & \fsmc (J \hc_B K) \hc_{\fsmc C} \fsmc C(\id, \theta_C) \\
				\fsmc J \hc_{\fsmc B} \fsmc B(\id, \theta_B) \hc_{\fmc B} \fmc K & \fsmc A(\id, \theta_A) \hc_{\fmc A} \fmc(J \hc_B K) \\ };
			\matrix(n)[math, yshift=-5.25em]{\fsmc A(\id, \theta_A) \hc_{\fmc A} \fmc J \hc_{\fmc B} \fmc K \\};
			\path	(m-1-1)	edge[cell] node[above] {$\fsmc_\hc \hc \id$} (m-1-2)
										edge[cell] node[left] {$\id \hc \rho\theta_K$} (m-2-1)
						(m-1-2) edge[cell] node[right] {$\rho\theta_{J \hc K}$} (m-2-2)
						(m-2-1) edge[cell] node[below left, yshift=2pt] {$\rho\theta_J \hc \id$} (n-1-1)
						(n-1-1) edge[cell] node[below right, yshift=2pt] {$\id \hc \fmc_\hc$} (m-2-2);
		\end{tikzpicture}
	\end{displaymath}
	Here the top cell coincides with the restriction $\fsmc_\hc \of (\id \tens \theta_C)$ after it is composed with the isomorphisms $\fsmc J \hc_{\fsmc B} \fsmc K \hc_{\fsmc C} \fsmc C(\id, \theta_C) \iso (\fsmc J \hc_{\fsmc B} \fsmc K)(\id, \theta_C)$ and $\fsmc (J \hc_B K) \hc_{\fsmc C} \fsmc C(\id, \theta_C) \iso \fsmc (J \hc_B K)(\id, \theta_C)$ of \propref{cartesian fillers in terms of companions and conjoints}. All of the cells above but the top one are known to be invertible: hence the top one, and therefore $\fsmc_\hc$, are invertible as well.
\end{proof}
	
\subsection{Pseudo double categories}
	As the third example we recall the `free strict double category'-monad $D$ on the equipment $\psProf{\GG_1} = \Mod{\Span{\ps{\GG_1}}}$ of $\GG_1$-indexed profunctors (see \exref{example:internal profunctor equipments}), whose pseudoalgebras are `unbiased pseudo double categories'. Its definition will involve fewer `details' than that of $\fsmc$, because it is induced by the relatively simple monad $F$ of free categories on the category $\ps{\GG_1}$ of presheaves on $\GG_1 = (0 \rightrightarrows 1)$. The latter induces $D$ by the formula $D = \Prof (F)$, as described in \exref{example:monads inducing monads on internal profunctors}.

	A presheaf $\map A{\op\GG_1}\Set$ consists of two sets $A_0$ and $A_1$ as well as a pair of functions $A_1 \rightrightarrows A_0$ that we will denote $L$ and $R$. Such a presheaf is a \emph{directed graph}, with a set of vertices $A_0$ and a set of edges $A_1$: every $j \in A_1$ is a directed edge from the vertex $Lj$ to the vertex $Rj$. The free category monad $F = (F, \mu, \eta)$ on $\ps{\GG_1}$ is given as follows. The directed graph $F A$ has the same vertices as $A$, while its vertices are (possible empty) paths of the directed edges in $A$. In formulae:
	\begin{equation} \label{equation:free category monad}
		(FA)_0 = A_0 \qquad \text{and} \qquad (FA)_1 = A_0 \amalg \coprod_{n \geq 1} \overbrace{A_1 \times_{A_0} A_1 \times_{A_0} \dotsb \times_{A_0} A_1}^{\text{$n$ times}},
	\end{equation}
	where the set under the brace denotes the limit, called a \emph{wide pullback}, of the diagram $A_1 \xrar R A_0 \xlar L A_1 \xrar R \dotsb \xlar L A_1$ consisting of $n$ copies of $A_1$ and $n-1$ copies of $A_0$. On these limits the functions $L$ and $\map R{(FA)_1}{A_0}$ restrict to applying $L$ to the leftmost copy of $A_1$ and $R$ to the rightmost copy respectively. It is clear that this assignment extends to an endofunctor on $\ps{\GG_1}$. The edges of the directed graph $F^2A$ consist of paths of paths of edges in $A$ and the multiplication $\map{\mu_A}{F^2 A}{FA}$ is given by concatenation; the unit $\map{\eta_A}A{FA}$ maps each edge to its corresponding $1$-edge path. Of course the strict algebras for $F$ are categories.
	
	An elementary calculation shows that $F$ preserves pullbacks, so that we can apply the $2$-functor $\Span \dash$ of \propref{span is a 2-functor} to obtain a monad $\Span F$ on the equipment $\Span{\ps{\GG_1}}$ of spans in $\ps{\GG_1}$. Following this we apply the $2$-functor $\Mod\dash$ of \propref{mod is a 2-functor}, obtaining a monad $D = \Prof (F)$ on the equipment $\psProf{\GG_1} = \Mod{\Span{\ps{\GG_1}}}$ of $\GG_1$-indexed categories and $\GG_1$-indexed profunctors. We have already described the equipment $\psProf{\GG_1}$ in detail in \exref{example:GG1-indexed bimodules}: its objects are `double categories without horizontal composition', while a $\GG_1$-indexed profunctor $\hmap JAB$ consists of vertical morphisms $\map pab$, where $a \in A$ and $b \in B$, and cells $\cell ujk$, where $j$ and $k$ are horizontal morphisms in $A$ and $B$. The vertical morphisms and cells of $A$ and $B$ act on those of $J$.
	
	Applying the $2$-functors $\Span{\dash}$ and $\Mod{\dash}$ we find that the monad $D = \Mod{\Span F}$ is given as follows. The image $DA$ of a $\GG_1$-indexed category $A$ is again given by the formulae \eqref{equation:free category monad} where now the wide pullback is a limit of categories. The functors $L$ and $\map R{(DA)_1}{A_0}$ are again given by applying $L$ and $R$ respectively to the leftmost and rightmost copy of $A_1$. Thus the objects and vertical maps of $DA$ are those of $A$, while its horizontal morphisms and cells are (possibly empty) `formal horizontal composites'
	\begin{flalign} \label{equation:formal horizontal composites}
		&& \ul j &= (a_0 \xsrar{j_1} a_1, a_1 \xsrar{j_2} a_2, \dotsc, a_{n-1} \xsrar{j_n} a_n)& \notag \\
		\text{and} &&\ul u &= \Bigpars{\begin{tikzpicture}[textbaseline]
		\matrix(m)[math175em, column sep=0.9em, row sep=0.9em]{a_0 \nc a_1 \\ c_0 \nc c_1 \\};
				\path[map]  (m-1-1) edge[barred] node[above] {$j_1$} (m-1-2)
														edge node[left] {$p_1$} (m-2-1)
										(m-1-2) edge node[right] {$p_2$} (m-2-2)
										(m-2-1) edge[barred] node[below] {$k_1$} (m-2-2);
				\path[transform canvas={shift={($(m-1-2)!(0,0)!(m-2-2)-(0.35em, 0)$)}}] (m-1-1) edge[cell, shorten >= 0, shorten <= 0] node[right] {$u_1$} (m-2-1);
		\end{tikzpicture}, \begin{tikzpicture}[textbaseline]
		\matrix(m)[math175em, column sep=0.9em, row sep=0.9em]{a_1 \nc a_2 \\ c_1 \nc c_2 \\};
				\path[map]  (m-1-1) edge[barred] node[above] {$j_2$} (m-1-2)
														edge node[left] {$p_2$} (m-2-1)
										(m-1-2) edge node[right] {$p_3$} (m-2-2)
										(m-2-1) edge[barred] node[below] {$k_2$} (m-2-2);
				\path[transform canvas={shift={($(m-1-2)!(0,0)!(m-2-2)-(0.35em, 0)$)}}] (m-1-1) edge[cell, shorten >= 0, shorten <= 0] node[right] {$u_2$} (m-2-1);
				\end{tikzpicture}, \dotsc, \begin{tikzpicture}[textbaseline]
		\matrix(m)[math175em, column sep=0.9em, row sep=0.9em]{a_{n-1} \nc a_n \\ c_{n-1} \nc c_n \\};
				\path[map]  (m-1-1) edge[barred] node[above] {$j_n$} (m-1-2)
														edge node[left] {$p_{n-1}$} (m-2-1)
										(m-1-2) edge node[right] {$p_n$} (m-2-2)
										(m-2-1) edge[barred] node[below] {$k_n$} (m-2-2);
				\path[transform canvas={shift={($(m-1-2)!(0,0)!(m-2-2)-(0.15em, 0)$)}}] (m-1-1) edge[cell, shorten >= 0, shorten <= 0] node[right] {$u_n$} (m-2-1);
		\end{tikzpicture}}&
	\end{flalign}
	of horizontal morphisms and cells in $A$. Vertical composition of such sequences of cells is given coordinatewise. Again $\card{\ul j} = n$ denotes the length of $\ul j$, while the empty sequence at the object $a$ will be denoted by $(a)$, and that at the vertical morphism $p$ by $(p)$.
	
	The action of $D$ on a $\GG_1$-indexed profunctor $\hmap JAB$ is given similarly: on vertical morphisms by $(DJ)_0(a, b) = J_0(a, b)$ and on cells by
	\begin{displaymath}
		(DJ)_1(\ul j, \ul k) = \begin{cases}
		J_0(a, b) & \text{if $\ul j = (a)$, $\ul k = (b)$;} \\
		J_1(\ul j_1, \ul k_1) \times_{J_0} J_1(\ul j_2, \ul k_2) \times_{J_0} \dotsb \times_{J_0} J_1(\ul j_n, \ul k_n) & \text{if $\card{\ul j} = n = \card{\ul k}$;} \\
		\emptyset & \text{otherwise}
		\end{cases}
	\end{displaymath}
	where, in the second case, the wide pullback is of the maps $L$ and $\map R{J_1(\ul j_i, \ul k_i)}{J_0}$ into the set $J_0$ of all vertical morphisms of $J$. Thus the cells $\ul x$ of $DJ$ look like the cells $\ul u$ of $DA$ above, except that now the horizontal morphisms $j_i$ and $k_i$ are horizontal morphisms of $A$ and $B$ respectively, while the vertical morphisms $p_i$ are those of $J_0$. The actions of $DA$ and $DB$ are given coordinatewise.
	
	The assignments $A \mapsto DA$ and $J \mapsto DJ$ above extend to a normal endofunctor on the equipment $\psProf{\GG_1}$: the natural $\GG_1$-indexed transformations $DJ \hc_{DB} DH \natarrow D(J \hc_B H)$ making up the compositors of $D$ are given by the universal property of coends. Exactly as for $F$ the multiplication $\nat\mu{D^2}D$ is given by concatenation, while the unit $\nat\eta\id D$ maps horizontal morphisms and cells to the corresponding sequences of length $1$. Unlike $\Span F$, which is a pseudofunctor, the compositors of $D = \Mod{\Span F}$ need not be invertible, as \propref{free strict double category monad is not pseudo} below shows. We shall now describe the colax algebras of $D$. As with monoidal categories we simplify matters slightly by considering normal colax $D$-algebras, whose unitors $\cell{\alpha_0}{a \of \eta_A}{\id}$ are identities.
	
	A normal colax $D$-algebra $A = (A, \hc, \mathfrak a)$ consists of a $\GG_1$-indexed category $A$ that is equipped with a $\GG_1$-indexed functor $\map \hc{DA}A$. The strict unitor $\hc \of \eta = \id$ implies that $\hc$ restricts to the identity on objects, vertical morphisms and sequences consisting of a single horizontal morphism or cell. It maps each sequence $\ul j = (\ul j_1, \dotsc, \ul j_n)$ of composable horizontal morphisms of length $n > 1$, as in \eqref{equation:formal horizontal composites}, to its composite $\hmap{\ul j_1 \hc \dotsb \hc \ul j_n = \hc(\ul j)}{a_0}{a_n}$, while it picks a horizontal unit $\hmap{1_a = \hc(a)}aa$ for each object $a$. Likewise the sequence $\ul u = (\ul u_1, \dotsc, \ul u_n)$ of composable cells in \eqref{equation:formal horizontal composites} is mapped to a cell $\ul u_1 \hc \dotsb \hc \ul u_n = \hc(\ul u)$ that is of the form as on the left below, and it picks a horizontal unit cell $\map{1_p = \hc(p)}{\id_a}{\id_c}$ for each vertical morphism $\map pac$, as on the right.
	\begin{displaymath}
		\begin{tikzpicture}[baseline]
			\matrix(m)[math175em, column sep=5em]{a_1 \nc a_n \\ c_1 \nc c_n \\};
				\path[map]  (m-1-1) edge[barred] node[above] {$\ul j_1 \hc \dotsb \hc \ul j_n$} (m-1-2)
														edge node[left] {$p_1$} (m-2-1)
										(m-1-2) edge node[right] {$p_n$} (m-2-2)
										(m-2-1) edge[barred] node[below] {$\ul k_1 \hc \dotsb \hc \ul k_n$} (m-2-2);
				\path[transform canvas={shift=($(m-1-2)!0.5!(m-2-2)$), xshift=-2.3em}] (m-1-1) edge[cell] node[right] {$\ul u_1 \hc \dotsb \hc \ul u_n$} (m-2-1);
		\end{tikzpicture}
		\qquad\qquad\qquad\qquad\begin{tikzpicture}[baseline]
		\matrix(m)[math175em]{a \nc a \\ c \nc c \\};
				\path[map]  (m-1-1) edge node[left] {$p$} (m-2-1)
										(m-1-2) edge node[right] {$p$} (m-2-2);
				\path				(m-1-1) edge[eq] (m-1-2)
										(m-2-1) edge[eq] (m-2-2);
				\path[transform canvas={shift={($(m-1-2)!(0,0)!(m-2-2)$)}}] (m-1-1) edge[cell] node[right] {$1_p$} (m-2-1);
		\end{tikzpicture}
	\end{displaymath}
	The associator $\mathfrak a$ is a natural transformation $\nat{\mathfrak a}{\hc \of \mu_A}{\hc \of D\hc}$ consisting of cells
	\begin{displaymath}
		\mathfrak a \colon \ull j_{11} \hc \dotsb \hc \ull j_{1m_1} \hc \dotsb \hc \ull j_{n1} \hc \dotsb \hc \ull j_{nm_n} \to (\ull j_{11} \hc \dotsb \hc \ull j_{1m_1}) \hc \dotsb \hc (\ull j_{n1} \hc \dotsb \hc \ull j_{nm_n})
	\end{displaymath}
	which, by the unit axiom for $\mathfrak a$ (see \defref{definition:colax algebra}), are horizontal. Its associativity axiom says that the two ways of adding brackets in the composite of a triple sequence of horizontal morphisms coincide, exactly like the associativity axiom for normal colax monoidal categories. In \cite[Example 9.3]{Cruttwell-Shulman10} normal colax $D$-algebras are called `normal oplax double categories' by Cruttwell and Shulman; we shall call them \emph{normal colax double categories}. In \cite[Section 5.2]{Leinster04}, Leinster calls non-normal pseudo $D$-algebras `weak double categories'. The latter can be considered as the `unbiased' versions of the the familiar `biased' pseudo double categories (\defref{definition:pseudo double category}), that we have been using throughout, in the same way that unbiased monoidal categories are related to their biased counterparts.
	
	Likewise colax morphisms $\map fAB$ between unbiased pseudo double categories $A$ and $B$ are \emph{colax double functors}: they are $\GG_1$-indexed functors that come equipped with natural vertical cells $\cell{f_\hc}{f(\ul j_1 \hc \dotsb \hc \ul j_n)}{f\ul j_1 \hc \dotsb \hc f\ul j_n}$ which satisfy an associativity axiom, exactly as for unbiased colax monoidal functors (also compare \defref{definition:lax functor} of biased lax functors).
	
\begin{proposition}\label{free strict double category monad is not pseudo}
	The `free strict double category'-monad $D$ on $\psProf{\GG_1}$, as given above, is not a pseudomonad.
\end{proposition}
\begin{proof}
	Remember that the cells of $(DJ)_1$ are formal horizontal composites of cells in $J$. Therefore, to find $\GG_1$-indexed profunctors $\hmap JAB$ and $\hmap HBC$ whose compositor $DJ \hc_{DB} DH \natarrow D(J \hc_B H)$ is not invertible, we should try to find $J$ and $H$ that, each on its own, have no horizontally composable cells but so that, by taking the quotients that form the composite $J \hc_B H$, such composable cells are created. This can be easily achieved as follows. First we choose $A$ and $C$ to be the terminal $\GG_1$\ndash indexed category, with single object $*$ and single horizontal morphism $* \slashedrightarrow *$. Let $B$ be given by $B_0 = (\bot \to \top)$ and with $B_1$ consisting of a single horizontal morphism $\bot \slashedrightarrow \top$. Now choose $J$ and $H$ to consist of the single cells
	\begin{displaymath}
		J = \big(\begin{tikzpicture}[textbaseline]
			\matrix(m)[math175em, column sep=1.5em, row sep=1.5em]{* & * \\ \bot & \top \\};
			\path[map]	(m-1-1) edge[barred] (m-1-2)
													edge (m-2-1)
									(m-1-2) edge (m-2-2)
									(m-2-1) edge[barred] (m-2-2);
			\path				(m-1-1) edge[transform canvas={shift=($(m-1-2)!0.5!(m-2-2)$)}, cell] (m-2-1);
		\end{tikzpicture}\bigr) \qquad\qquad\text{and}\qquad\qquad H = \bigl(\begin{tikzpicture}[textbaseline]
			\matrix(m)[math175em, column sep=1.5em, row sep=1.5em]{\bot & \top \\ * & * \\};
			\path[map]	(m-1-1) edge[barred] (m-1-2)
													edge (m-2-1)
									(m-1-2) edge (m-2-2)
									(m-2-1) edge[barred] (m-2-2);
			\path				(m-1-1) edge[transform canvas={shift=($(m-1-2)!0.5!(m-2-2)$)}, cell] (m-2-1);
		\end{tikzpicture}\bigr).
	\end{displaymath}
	Those cells are not horizontally composable with themselves, hence $DJ = J$ and $DH = H$. On the other hand the coequaliser computing $J \hc_B H$ (see \eqref{equation:composition of internal profunctors}) identifies the two formal composites, $(* \to \bot, \bot \to *)$ and $(* \to \top, \top \to *)$, of vertical morphisms in $J \hc H$, so that $J \hc_B H$ only contains a single vertical morphism $* \to *$. Under this identification the formal vertical composite of the cells above has the form
	\begin{displaymath}
		\begin{tikzpicture}[textbaseline]
			\matrix(m)[math175em, column sep=1.5em, row sep=1.5em]{* & * \\ * & * \\};
			\path[map]	(m-1-1) edge[barred] (m-1-2)
													edge (m-2-1)
									(m-1-2) edge (m-2-2)
									(m-2-1) edge[barred] (m-2-2);
			\path				(m-1-1) edge[transform canvas={shift=($(m-1-2)!0.5!(m-2-2)$)}, cell] (m-2-1);
		\end{tikzpicture}
	\end{displaymath}
	and we conclude that $J \hc_B H$ is the terminal $\GG_1$-indexed profunctor $\hmap *AC$. We conclude that the compositor of $J$ and $H$ is the natural $\GG_1$-indexed transformation $* \natarrow D*$, which is clearly not invertible because the single cell of $*$ above is horizontally composable with itself.
\end{proof}

	To close this chapter we now briefly look at the `free strict $\omega$-category'-monad and the `ultrafilter'-monad, whose algebras are respectively `monoidal globular categories' and `ordered Hausdorff spaces', leaving out most of the details.
\subsection{Monoidal globular categories}
	As a generalisation of the category $\GG_1$, that we used above to define double categories, we here consider the \emph{globe category} $\GG$ that has the natural numbers as objects and the maps
	\begin{displaymath}
		\begin{tikzpicture}
			\matrix(m)[math2em]{0 & 1 & 2 & \cdots \\};
			\path[map, transform canvas={yshift=2pt}]	(m-1-1) edge node[above] {$\sigma_1$} (m-1-2)
					(m-1-2) edge node[above] {$\sigma_2$} (m-1-3)
					(m-1-3) edge node[above] {$\sigma_3$} (m-1-4);
			\path[map, transform canvas={yshift=-2pt}]	(m-1-1) edge node[below] {$\tau_1$} (m-1-2)
					(m-1-2) edge node[below] {$\tau_2$} (m-1-3)
					(m-1-3) edge node[below] {$\tau_3$} (m-1-4);
		\end{tikzpicture}
	\end{displaymath}
	as morphisms, that satisfy the relations
	\begin{displaymath}
		\sigma_n \of \sigma_{n-1} = \tau_n \of \sigma_{n-1} \qquad \text{and} \qquad \sigma_n \of \tau_{n-1} = \tau_n \of \tau_{n-1}.
	\end{displaymath}
	A presheaf $A$ on $\GG$ is called a \emph{globular set}; it consists of sets $A_n$, $n \in \NN$, whose elements are thought of as `$n$-dimensional globular cells', and it comes equipped with `source and target functions' $s$ and $\map t{A_n}{A_{n-1}}$. For example an element of $A_2$ is thought of as a cell
	
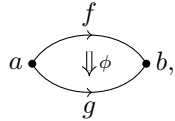
\begin{displaymath}
		\begin{tikzpicture}[decoration={markings, mark=at position 0.52 with \arrow{>}}]
			\path	(0,0) edge[bend right = 60, postaction={decorate}] node[below] {$g$} (1.5,0)
												edge[bend left = 60, postaction={decorate}] node[above] {$f$} (1.5,0)
						(0.75,0.27) edge[cell] node[right] {$\phi$} (0.75,-0.27);
			\fill (0,0) circle(1.5pt) node[left] {$a$}
						(1.5,0) circle(1.5pt) node[right] {$b,$};
		\end{tikzpicture}
	\end{displaymath}
	where $s\phi = f$, $t\phi = g$, $sf = a = sg$ and $tf = b = tg$ (notice that the images of relations above, on the $\sigma_i$ and $\tau_i$, hold).
	
	More general, a $\GG$-indexed category $A$, that is a presheaf $\map A{\op\GG}\cat$, is called a \emph{$\omega$-globular category} by Batanin, who uses `monoidal' variants of such globular categories in \cite{Batanin98} to introduce weak $\omega$-categories. There exists a `free strict $\omega$-category'-monad $T$ on $\ps\GG$, which preserves pullbacks and thus induces a monad $\inProf T$ on the equipment $\psProf\GG$ by \propref{span is a 2-functor} and \propref{mod is a 2-functor}, and pseudoalgebras for $\inProf T$ are precisely the `unbiased' versions of the \emph{monoidal globular categories} of \cite{Batanin98}; see \cite[Example 9.11]{Cruttwell-Shulman10}. Loosely speaking, the image $TA$ of a globular set $A$ consists of `formal pasting-diagrams' of the cells of $A$: a typical element of $(TA)_2$ looks like
	
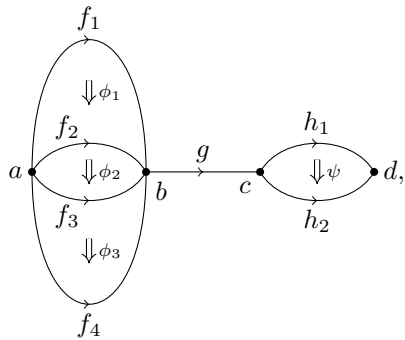
\begin{displaymath}
		\begin{tikzpicture}[decoration={markings, mark=at position 0.5 with \arrow{>}}]
			\path	(0,0) edge[bend right = 60, postaction={decorate}] node[below left, yshift=2pt] {$f_3$} (1.5,0)
												edge[bend left = 60, postaction={decorate}] node[above left, yshift=-2pt] {$f_2$} (1.5,0)
									edge[bend left = 90, looseness=4, postaction={decorate}] node[above] {$f_1$} (1.5,0)
									edge[bend right = 90, looseness=4, postaction={decorate}] node[below] {$f_4$} (1.5,0)
						(3,0) edge[bend right = 60, postaction={decorate}] node[below] {$h_2$} (4.5,0)
												edge[bend left = 60, postaction={decorate}] node[above] {$h_1$} (4.5,0)
						(0.75,1.32) edge[cell] node[right] {$\phi_1$} (0.75,0.78)
						(1.5,0) edge[postaction={decorate}] node[above] {$g$} (3,0)
						(0.75,0.27) edge[cell] node[right] {$\phi_2$} (0.75,-0.27)
						(0.75,-0.78) edge[cell] node[right] {$\phi_3$} (0.75,-1.32)
						(3.75,0.27) edge[cell] node[right] {$\psi$} (3.75,-0.27);
			\fill (0,0) circle(1.5pt) node[left] {$a$}
						(1.5,0) circle(1.5pt) node[below right] {$b$}
						(3,0) circle(1.5pt) node[below left] {$c$}
						(4.5,0) circle(1.5pt) node[right] {$d,$};
		\end{tikzpicture}
	\end{displaymath}
	where the $0$-cells, $1$-cells and $2$-cells those of $A$. The details can be found in \cite[Chapter 8]{Leinster04}.

\subsection{Ordered compact Hausdorff spaces}
	The second example that we briefly mention are the `ordered compact Hausdorff spaces' considered by Tholen in \cite{Tholen09}. Different from the usual way of defining a topological space, a topology on a set $A$ can be given using `ultrafilters' in $A$, as follows. An \emph{ultrafilter} $\mathfrak f$ in $A$ is defined to be a subset of the powerset $P(A)$ on $A$, that is upward-closed, closed under intersection and `satisfies dichotomy', as follows:
	\begin{itemize}[label=-]
		\item if $S \in \mathfrak f$ and $T \supseteq S$ then $T \in \mathfrak f$;
		\item if $S, T \in \mathfrak f$ then $S \isect T \in \mathfrak f$;
		\item for any $S \subseteq A$, either $S \in \mathfrak f$ or $A - S \in \mathfrak f$.
	\end{itemize}
	The set of ultrafilters in $A$ is denoted $\beta A$. One can think of ultrafilters in $A$ as families of subsets that `move in a certain direction within $A$': for each $S \in \mathfrak f$ and any smaller subset $T \subsetneq S$, either $T \in \mathfrak f$ or $S - T \in \mathfrak f$, so that $\mathfrak f$ either `moves into $T$' or `moves into $S - T$'.

	If $A$ is a topological space then an ultrafilter $\mathfrak f$ in $A$ is said to \emph{converge} to a point $x$ of $A$ if each open subset $U$ containing $x$ belongs to $\mathfrak f$. If $A$ is a compact Hausdorff space then every ultrafilter in $A$ converges to precisely one point, giving a function $\map a{\beta A}A$, that is called the \emph{convergence} of $A$. Manes showed in his thesis \cite{Manes69} that the endofunctor $\map \beta\Set\Set$ can be extended to a monad $(\beta, \mu, \eta)$, the \emph{ultrafilter-monad}, such that strict $\beta$-algebra structures $\map a{\beta A}A$ on a set $A$ correspond to compact Hausdorff topologies on $A$.
	
	Now consider the equipment $\Mat\2$, where $\2 = (\bot \to \top)$, that consists of sets and relations. The ultrafilter-monad $\beta$ can be extended to a pseudomonad on $\Mat\2$ by mapping the relation $\hmap JAB$ to the relation $\hmap{\beta J}{\beta A}{\beta B}$ that is given by
	\begin{displaymath}
		(\beta J)(\mathfrak f, \mathfrak g) = \bigwedge_{S \in \mathfrak f, T \in \mathfrak g} \bigvee_{x \in S, y \in T} J(x,y).
	\end{displaymath}
	Hence, applying \propref{mod is a 2-functor}, we obtain a monad $\Mod \beta$ on the equipment $\enProf\2$ of preordered sets and modular relations (see \exref{example:enriched profunctor equipments}). A pseudoalgebra $A$ for $\Mod\beta$ consists of a preordered set equipped with a compact Hausdorff topology, so that the convergence $\map a{\beta A}A$ is monotone. In \cite[Example 2]{Tholen09} such spaces are called \emph{ordered compact Hausdorff spaces}.

	\chapter{Algebraic promorphisms}\label{chapter:algebraic promorphisms}
	In this final chapter we propose generalisations of the well-known notions of lax and colax morphisms, of algebras over a $2$-monad, to notions of lax and colax promorphisms, of algebras over a monad on an equipment. The lax notion recovers natural notions of algebraic promorphisms, some of which have been considered before. For example a `lax monoidal profunctor' $\hmap JAB$, of monoidal categories $A$ and $B$, comes equipped with a coherent tensor product
	\begin{displaymath}
		(\map {j_1}{a_1}{b_1}, \dotsc, \map{j_n}{a_n}{b_n}) \mapsto \map{j_1 \tens \dotsb \tens j_n}{a_1 \tens \dotsb \tens a_n}{b_1 \tens \dotsb \tens b_n}
	\end{displaymath}
	of morphisms $\map jab$ in $J$. Another example, that of lax double profunctors between double categories, is discussed on the $n$Lab page \cite{nLab_on_double_profunctors}.
	
	When defining lax promorphisms we shall, following \remref{remark:why invertible cells}, also consider two kinds of `pseudo lax promorphisms': `left pseudo' and `right pseudo' ones. These correspond to the two ways of making a cell into a horizontal cell, either by using companions or by using conjoint cells. Right pseudopromorphisms generalise the situation of Getzler's proposition, that was discussed in the introduction, and satisfy a decomposition condition. For example the lax monoidal profunctor $\hmap JAB$ above is a right pseudopromorphism if each morphism $\map ja{c_1 \tens \dotsb \tens c_n}$ in $J$ decomposes, in some sense uniquely, as a composite
	\begin{displaymath}
		a \xrar f b_1 \tens \dotsb \tens b_n \xrar{j_1 \tens \dotsb \tens j_n} c_1 \tens \dotsb \tens c_n,
	\end{displaymath}
	where $\map fa{b_1 \tens \dotsb \tens b_n}$ belongs to $A$ and each $\map{j_i}{b_i}{c_i}$ belongs to $J$.
	
	In \defref{definition:right colax promorphisms} the notion of right pseudopromorphisms leads to that of right colax promorphisms. Roughly speaking the latter are equipped with `splittings' like the decomposition above, without being equipped with a lax structure. For example, a right colax monoidal profunctor $\hmap HAB$ does not define tensor products of its maps, like $J$ does above, but only gives a coherent way of `splitting' each morphism $\map ha{c_1 \tens \dotsb \tens c_n}$ into a morphism $\map fa{b_1 \tens \dotsb \tens b_n}$ in $A$ and morphisms $\map{h_1}{b_1}{c_1}, \dotsc, \map{h_n}{b_n}{c_n}$ in $H$. Notice that, since we cannot take the tensor product of the maps $h_1, \dotsc, h_n$ in $H$, there is no way of comparing $h$ with its `splitting' $(f, h_1, \dotsc, h_n)$.
	
	Although the author does not yet know of any `true' right colax promorphisms, that are not right pseudopromorphisms, they have the following formal advantage over lax promorphisms. It is, in general, not possible to coherently horizontally compose lax promorphisms, for a monad that is not pseudo. In particular it is not possible to  compose lax double profunctors between double categories in a coherent way, as discussed in \cite{nLab_on_double_profunctors}. Right colax promorphisms do not suffer from this: in fact the main result (\propref{right colax promorphisms form a pseudo double category}) of \secref{section:algebraic promorphisms} below asserts that colax algebras, colax morphisms and right colax promorphisms over a normal monad $T$, that satisfies some mild conditions, form a pseudo double category $\rcProm T$.
	
	Since $\rcProm T$ is a pseudo double category it allows a notion of `algebraic weighted colimits', as given in \defref{definition:weighted colimits in double categories}. Moreover, using the results of \chapref{chapter:weighted colimits} and by generalising a notion of Grandis and Par\'e, we shall consider `pointwise weighted colimits' in $\rcProm T$ as well. In fact \thmref{algebraic weighted colimits are pointwise if double commas are strong} below enhances \thmref{weighted colimits are pointwise if double commas are strong}, by proving that all weighted colimits in $\rcProm T$ are pointwise whenever the underlying equipment $\K$ has strong double comma objects (\defref{definition:strong double comma objects}).
	
	Finally in \secref{section:main result} we consider the main result (\thmref{right colax forgetful functor lifts all weighted colimits}), which states that if $T$ is a `suitable' normal monad on a closed equipment $\K$, then the forgetful functor
	\begin{displaymath}
		\map U{\rcProm T}\K
	\end{displaymath}
	lifts all weighted colimits. This means that an algebraic weighted colimit $\colim_J d$ in $\rcProm T$ can be computed as the ordinary weighted colimit $\colim_{UJ} Ud$ in $\K$. Following this, in \secref{section:applications of the main theorem}, applications of the main result are considered, while it is also compared to results in \cite{Getzler09}, \cite{Grandis-Pare07} and \cite{Mellies-Tabareau08}. The author believes that all results in this chapter are original.
	
\section{Algebraic promorphisms} \label{section:algebraic promorphisms}
	We start with the notion of lax promorphisms which, even though being a straightforward generalisation of that of lax morphisms, appears to be new. That it is the right notion follows from \propref{colax morphisms correspond to lax companions} below, which shows that colax algebra structures on a morphism $\map fAC$ correspond to lax algebra structures on its companion $C(f, \id)$. Given a normal monad $T$ on an equipment $\K$, recall that by a colax $T$-algebra we mean a colax algebra for the $2$-monad $V(T)$ on $V(\K)$.
	
\begin{definition} \label{definition:lax promorphism}
	Let $T$ be a normal monad on an equipment $\K$. Given colax $T$\ndash algebras $A = (A, a, \alpha, \alpha_0)$ and $B = (B, b, \beta, \beta_0)$, a \emph{lax $T$-promorphism} $A \slashedrightarrow B$ is a promorphism $\hmap JAB$ equipped with a \emph{structure cell}
	\begin{displaymath}
		\begin{tikzpicture}
			\matrix(m)[math175em]{TA & TB \\ A & B \\};
			\path[map]  (m-1-1) edge[barred] node[above] {$TJ$} (m-1-2)
													edge node[left] {$a$} (m-2-1)
									(m-1-2) edge node[right] {$b$} (m-2-2)
									(m-2-1) edge[barred] node[below] {$J$} (m-2-2);
			\path[transform canvas={shift={($(m-1-2)!(0,0)!(m-2-2)$)}}] (m-1-1) edge[cell] node[right] {$\bar J$} (m-2-1);			
		\end{tikzpicture}
	\end{displaymath}
	satisfying the following coherence conditions, respectively the associativity axiom and the unit axiom.
	\begin{displaymath}
		\begin{tikzpicture}[textbaseline]
      \matrix(m)[math175em]
      {	T^2 A & T^2 B & T^2 B \\
				TA & TB & TB \\
				A & B & B \\ };
      \path[map]	(m-1-1) edge[barred] node[above] {$T^2 J$} (m-1-2)
													edge node[left] {$\mu_A$} (m-2-1)
									(m-1-2) edge node[right] {$\mu_B$} (m-2-2)
									(m-1-3) edge node[right] {$Tb$} (m-2-3)
									(m-2-1) edge[barred] node[below] {$TJ$} (m-2-2)
													edge node[left] {$a$} (m-3-1)
									(m-2-2) edge node[right] {$b$} (m-3-2)
									(m-2-3) edge node[right] {$b$} (m-3-3)
									(m-3-1) edge[barred] node[below] {$J$} (m-3-2);
			\path				(m-1-2) edge[eq] (m-1-3)
									(m-3-2) edge[eq] (m-3-3);
			\path[transform canvas={shift={($(m-2-1)!0.5!(m-2-2)$)}}]
									(m-1-2) edge[cell] node[right] {$\mu_J$} (m-2-2);
			\path[transform canvas={shift={($(m-2-1)!0.5!(m-2-2)-(0,0.25em)$)}}]
									(m-2-2) edge[cell] node[right] {$\bar J$} (m-3-2);
			\path[transform canvas={shift={($(m-1-1)!0.5!(m-2-2)$)}}]
									(m-2-3) edge[cell] node[right] {$\beta$} (m-3-3);
    \end{tikzpicture}
    \quad = \quad \begin{tikzpicture}[textbaseline]
      \matrix(m)[math175em]
      {	T^2 A & T^2 A & T^2 B \\
				TA & TA & TB \\
				A & A & B \\ };
			\path[map]	(m-1-1) edge node[left] {$\mu_A$} (m-2-1)
									(m-1-2) edge[barred] node[above] {$T^2 J$} (m-1-3)
													edge node[left] {$Ta$} (m-2-2)
									(m-1-3) edge node[right] {$Tb$} (m-2-3)
									(m-2-1) edge node[left] {$a$} (m-3-1)
									(m-2-2) edge[barred] node[below] {$TJ$} (m-2-3)
													edge node[left] {$a$} (m-3-2)
									(m-2-3) edge node[right] {$b$} (m-3-3)
									(m-3-2) edge[barred] node[below] {$J$} (m-3-3);
			\path				(m-1-1) edge[eq] (m-1-2)
									(m-3-1) edge[eq] (m-3-2);
			\path[transform canvas={shift={($(m-1-1)!0.5!(m-2-2)$)}}]
									(m-2-2) edge[cell] node[right] {$\alpha$} (m-3-2);
			\path[transform canvas={shift={($(m-2-1)!0.5!(m-2-2)$)}}]
									(m-1-3) edge[cell] node[right] {$T\bar J$} (m-2-3);
			\path[transform canvas={shift={($(m-2-1)!0.5!(m-2-2)-(0,0.25em)$)}}]
									(m-2-3) edge[cell] node[right] {$\bar J$} (m-3-3);
    \end{tikzpicture}
  \end{displaymath}
  \begin{displaymath}
    \begin{tikzpicture}[textbaseline]
      \matrix(m)[math175em]
      { A & B & B \\
				TA & TB & \phantom{TA} \\
				A & B & B \\ };
			\path[map]	(m-1-1)	edge[barred] node[above] {$J$} (m-1-2)
													edge node[left] {$\eta_A$} (m-2-1)
									(m-1-2) edge node[right] {$\eta_B$} (m-2-2)
									(m-2-1) edge[barred] node[below] {$TJ$} (m-2-2)
													edge node[left] {$a$} (m-3-1)
									(m-2-2) edge node[right] {$b$} (m-3-2)
									(m-3-1) edge[barred] node[below] {$J$} (m-3-2);
			\path				(m-1-2) edge[eq] (m-1-3)
									(m-1-3)	edge[eq] (m-3-3)
									(m-3-2) edge[eq] (m-3-3);
			\path[transform canvas={shift={($(m-1-1)!0.5!(m-2-2)$)}}]
									(m-2-3) edge[cell] node[right] {$\beta_0$} (m-3-3);
			\path[transform canvas={shift={($(m-2-1)!0.5!(m-2-2)$)}}]
									(m-1-2) edge[cell] node[right] {$\eta_J$} (m-2-2);
			\path[transform canvas={shift={($(m-2-1)!0.5!(m-2-2)-(0,0.25em)$)}}]
									(m-2-2) edge[cell] node[right] {$\bar J$} (m-3-2);
    \end{tikzpicture}
    \quad = \quad \begin{tikzpicture}[textbaseline]
      \matrix(m)[math175em]
      { A & A & B \\
				TA & \phantom{TA} & \phantom{TA} \\
				A & A & B \\ };
			\path[map]	(m-1-2) edge[barred] node[above] {$J$} (m-1-3)
									(m-1-1) edge node[left] {$\eta_A$} (m-2-1)
									(m-2-1) edge node[left] {$a$} (m-3-1)
									(m-3-2) edge[barred] node[below] {$J$} (m-3-3);
			\path				(m-1-2) edge[eq] (m-3-2)
									(m-1-1) edge[eq] (m-1-2)
									(m-3-1)	edge[eq] (m-3-2)
									(m-1-3) edge[eq] (m-3-3);
			\path[transform canvas={shift={($(m-1-1)!0.5!(m-2-2)$)}}]
									(m-2-2) edge[cell] node[right] {$\alpha_0$} (m-3-2);
    \end{tikzpicture}
	\end{displaymath}
	
	Following \defref{definition:invertible cells} we call $J$ a \emph{left pseudo} $T$-promorphism when the horizontal cell $\cell{\lambda\bar J}{TJ \hc B(b, \id)}{A(a, \id) \hc J}$ is invertible, and \emph{right pseudo} when $\cell{\rho\bar J}{A(\id, a) \hc TJ}{J \hc B(\id, b)}$ is invertible.
\end{definition}
	
	Taking for $T$ the `free strict monoidal $\V$-category'-monad $\fmc$ or the `free strict double category'-monad $D$, both introduced in \secref{section:monad examples}, we obtain the notion of `lax monoidal $\V$-profunctors' between monoidal $\V$-categories and that of `lax double profunctors' between pseudo double categories, as follows.
\begin{example} \label{example:lax monoidal profunctors}
	Choosing $T = \fmc$, the `free strict monoidal $\V$-category'-monad, let $A$ and $B$ be normal colax monoidal $\V$-categories with structure $\V$-functors \mbox{$\map a{\fmc A}A$} and $\map b{\fmc B}B$. A lax $\fmc$\ndash promorphism $\hmap JAB$, as defined above, will be called a \emph{lax monoidal $\V$-profunctor}. It consists of a $\V$-profunctor $\map J{\op A \tens B}\V$ equipped with a $\V$-natural transformation $\nat{J_\tens = \bar J}{\fmc J}{J(a, b)}$ of $\V$-profunctors $\fmc A \slashedrightarrow \fmc B$, that satisfies the coherence axioms above. Unpacking this we find that $J_\tens$ consists of morphisms
	\begin{displaymath}
		\map{J_\tens}{J(\ul x_1, \ul y_1) \tens \dotsb \tens J(\ul x_n, \ul y_n)}{J(\ul x_1 \tens \dotsb \tens \ul x_n, \ul y_1 \tens \dotsb \tens \ul y_n)}
	\end{displaymath}
	that are natural in $\ul x \in \fmc_n A$ and $\ul y \in \fmc_n B$, as well as $\map{J_1}1{J(1_A, 1_B)}$, such that the following associativity and unit axioms are satisfied. The associativity axiom states that, for each pair of double sequences $\ull x$ and $\ull y$, of equal length and such that $\card{\ull x_i} = m_i = \card{\ull y_i}$, the following diagram commutes. Here $\Tens_{i = 1, j = 1}^{i = n, j = m_i} \ull x_{ij}$ denotes the tensor product $\ull x_{11} \tens \dotsb \tens \ull x_{1m_1} \tens \dotsb \tens \ull x_{n1} \tens \dotsb \tens \ull x_{nm_n}$.
	\begin{displaymath}
		\begin{tikzpicture}
			\matrix(m)[math2em, inner sep=6pt]
			{ \Tens_{i=1}^n \bigpars{\Tens_{j=1}^{m_i} J(\ull x_{ij}, \ull y_{ij})} & \Tens_{i=1}^n J\bigpars{\Tens_{j=1}^{m_i} \ull x_{ij}, \Tens_{j=1}^{m_i} \ull y_{ij}} \\
				\Tens_{i=1, j = 1}^{i=n, j=m_i} J(\ull x_{ij}, \ull y_{ij}) & J\bigpars{\Tens_{i=1}^n (\Tens_{j=1}^{m_i} \ull x_{ij}), \Tens_{i=1}^n (\Tens_{j=1}^{m_i} \ull y_{ij})} \\
				J\bigpars{\Tens_{i = 1, j = 1}^{i = n, j = m_i} \ull x_{ij}, \Tens_{i = 1, j = 1}^{i = n, j = m_i} \ull y_{ij}} & J\bigpars{\Tens_{i = 1, j = 1}^{i = n, j = m_i} \ull x_{ij}, \Tens_{i=1}^n (\Tens_{j=1}^{m_i} \ull y_{ij})} \\};
			\path[map]	(m-1-1) edge node[above] {$\Tens_{i=1}^n J_\tens$} (m-1-2)
									(m-1-2) edge node[right] {$J_\tens$} (m-2-2)
									(m-2-1)	edge node[left] {$J_\tens$} (m-3-1)
									(m-2-2) edge node[right] {$J(\mathfrak a_{\ull x}, \id)$} (m-3-2)
									(m-3-1) edge node[below] {$J(\id, \mathfrak a_{\ull y})$} (m-3-2);
			\path[white]	(m-1-1) edge node[black, sloped, desc] {$\iso$} (m-2-1);
		\end{tikzpicture}
	\end{displaymath}
	The unit axiom says that the map $\map{J_\tens}{J(x, y)}{J(x, y)}$, for any pair of objects $x$ and $y$, is the identity.
	
	Formally the lax monoidal $\V$-profunctor $\hmap JAB$ above is a right pseudoprofunctor whenever the horizontal composite $\rho J_\tens = \eps_a \hc J_\tens \hc \eta_b$ is invertible. Computing $\rho J_\tens$ we find that this means that the composites
	\begin{align} \label{equation:right pseudo monoidal profunctor}
		\int^{\ul y} A(x, \ul y_1 &{} \tens \dotsb \tens \ul y_n) \tens J(\ul y_1, \ul z_1) \tens \dotsb  \tens J(\ul y_n, \ul z_n) \notag \\
		& \xrar{\int \id \tens J_\tens} \int^{\ul y} A(x, \ul y_1 \tens \dotsb \tens \ul y_n) \tens J(\ul y_1 \tens \dotsb \tens \ul y_n, \ul z_1 \tens \dotsb \tens \ul z_n) \notag \\
		& \to J(x, \ul z_1 \tens \dotsb \tens \ul z_n)
	\end{align}
	are isomorphisms, for all $x \in A$ and $\ul z \in \fmc_n B$, where the coends range over $\ul y \in \fmc_n A$ and where the second map is induced by the action of $A$ on $J$. When $n = 0$ this reduces to $A(x, 1_A) \tens 1 \xrar{\id \tens J_\tens} A(x, 1_A) \tens J(1_A, 1_B) \to J(x, 1_B)$.
	
	In the case that $\V = \Set$ this means that every morphism $\map jx{\ul z_1 \tens \dotsb \tens \ul z_n}$ in $J$ can be decomposed as
	\begin{equation} \label{equation:decomposition of right pseudoprofunctors}
		j = (\ul j_1 \tens \dotsb \tens \ul j_n) \of f,
	\end{equation}
	with $\map{\ul j_i}{\ul y_i}{\ul z_i}$ and $\map fx{\ul y_1 \tens \dotsb \tens \ul y_n}$, and that this decomposition is unique in the sense that another pair $(f', \ul j')$ also decomposes $j$ if and only if $(f, \ul j)$ and $(f', \ul j')$ are identified in the first coend above.
	
	Lax symmetric monoidal $\V$-profunctors, that is lax $\fsmc$-profunctors for the `free symmetric strict monoidal $\V$-category'-monad $\fsmc$, can be described similarly.
\end{example}

\begin{example} \label{example:lax double profunctors}
	Consider normal colax double categories $A$ and $B$, that is normal colax $D$-algebras, with $\GG_1$-indexed structure functors $\map{\hc_A}{DA}A$ and $\map{\hc_B}{DB}B$. A lax $D$-promorphism $\hmap JAB$ will be called a \emph{lax double profunctor}: it consists of a $\GG_1$-indexed profunctor $\hmap JAB$ (see \exref{example:GG1-indexed bimodules}) equipped with a structure transformation $\nat{\hc_J = \bar J}{DJ}{J(\hc_A, \hc_B)}$, that satisfies the associativity and unit axioms. Giving $\hc_J$ is equivalent to specifying horizontal composites
	\begin{displaymath}
		\Bigpars{\begin{tikzpicture}[textbaseline]
			\matrix(m)[math175em]{a_0 \nc a_1 \\ b_0 \nc b_1 \\};
				\path[map]  (m-1-1) edge[barred] node[above] {$\ul k_1$} (m-1-2)
														edge node[left] {$p_0$} (m-2-1)
										(m-1-2) edge node[right] {$p_1$} (m-2-2)
										(m-2-1) edge[barred] node[below] {$\ul l_1$} (m-2-2);
				\path[transform canvas={shift={($(m-1-2)!(0,0)!(m-2-2)$)}}] (m-1-1) edge[cell] node[right] {$\ul w_1$} (m-2-1);
		\end{tikzpicture}, \dotsc, \begin{tikzpicture}[textbaseline]
			\matrix(m)[math175em]{a_{n-1} \nc a_n \\ b_{n-1} \nc b_n \\};
				\path[map]  (m-1-1) edge[barred] node[above] {$\ul k_n$} (m-1-2)
														edge node[left] {$p_{n-1}$} (m-2-1)
										(m-1-2) edge node[right] {$p_n$} (m-2-2)
										(m-2-1) edge[barred] node[below] {$\ul l_n$} (m-2-2);
				\path[transform canvas={shift={($(m-1-2)!(0,0)!(m-2-2)-(0.3em, 0)$)}}] (m-1-1) edge[cell] node[right] {$\ul w_n$} (m-2-1);
		\end{tikzpicture}} \mapsto \begin{tikzpicture}[textbaseline]
			\matrix(m)[math175em, column sep=6.5em]{a_0 \nc a_n \\ b_0 \nc b_n \\};
				\path[map]	(m-1-1) edge[barred] node[above] {$\ul k_1 \hc_A \dotsb \hc_A \ul k_n$} (m-1-2)
														edge node[left] {$p_0$} (m-2-1)
										(m-1-2) edge node[right] {$p_n$} (m-2-2)
										(m-2-1) edge[barred] node[below] {$\ul l_1 \hc_B \dotsb \hc_B \ul l_n$} (m-2-2);
				\path[transform canvas={shift={($(m-1-2)!0.5!(m-2-2)-(3em,0)$)}}] (m-1-1) edge[cell] node[right] {$\ul w_1 \hc_J \dotsb \hc_J \ul w_n$} (m-2-1);
		\end{tikzpicture}
	\end{displaymath}
	for the cells of $J$, as well as horizontal units which are of the form as on the left below. The associativity axiom says that the diagram on the right commutes for every double sequence $\cell{\ull w}{\ull k}{\ull l}$.
	\begin{displaymath}
		\begin{tikzpicture}[baseline]
			\matrix(m)[math175em]{a & a \\ b & b \\};
				\path[map]	(m-1-1) edge node[left] {$p$} (m-2-1)
										(m-1-2) edge node[right] {$p$} (m-2-2);
				\path				(m-1-1) edge[eq] (m-1-2)
										(m-2-1) edge[eq] (m-2-2);
				\path[transform canvas={shift={($(m-1-2)!0.5!(m-2-2)$)}}] (m-1-1) edge[cell] node[right] {$1_p$} (m-2-1);
		\end{tikzpicture}
		\qquad\qquad\begin{tikzpicture}[baseline]
			\matrix(m)[math2em, column sep=5.5em]
			{	\Hc_{i=1}^n \bigpars{\Hc_{j=1}^{m_i} \ull k_{ij}} & \Hc_{i=1}^n \bigpars{\Hc_{j=1}^{m_i} \ull l_{ij}} \\
				\Hc_{i=1, j=1}^{i=n, j=n_i} \ull k_{ij} & \Hc_{i=1, j=1}^{i=n, j=n_i} \ull l_{ij} \\ };
			\path	(m-1-1) edge[cell] node[above] {$\Hc_{i=1}^n \bigpars{\Hc_{j=1}^{m_i} \ull w_{ij}}$} (m-1-2)
										edge[cell] node[left] {$\mathfrak a_A$} (m-2-1)
						(m-1-2) edge[cell] node[right] {$\mathfrak a_B$} (m-2-2)
						(m-2-1) edge[cell] node[below] {$\Hc_{i=1, j=1}^{i=n, j=m_i} \ull w_{ij}$} (m-2-2);
		\end{tikzpicture}
	\end{displaymath}
	In the case that $A$ and $B$ are strict double categories this recovers the description that is given on the $n$Lab page on double profunctors \cite{nLab_on_double_profunctors}.
	
	That $J$ is right pseudo means that the morphisms
	\begin{displaymath}
		\int^{\ul k \in DA} A(h, \ul k_1 \hc_A \dotsb \hc_A \ul k_n) \times J(\ul k_1, \ul l_1) \times \dotsb \times J(\ul k_n, \ul l_n) \to J(h, \ul l_1 \hc_B \dotsb \hc_B \ul l_n),
	\end{displaymath}
	induced by $\hc_J$ and the action of $A$ on $J$, are invertible for each horizontal morphism $h$ in $A$ and sequence $\ul l$ in $DB$. This means that every cell $w$ in $J$, that is of the form below, can be decomposed as
	\begin{displaymath}
		\begin{tikzpicture}[textbaseline]
			\matrix(m)[math175em]
			{ a & & & & b \\
				c_0 & c_1 & \dotsb & c_{n'} & c_n \\ };
			\path[map]	(m-1-1) edge[barred] node[above] {$h$} (m-1-5)
													edge node[left] {$p$} (m-2-1)
									(m-1-5) edge node[right] {$q$} (m-2-5)
									(m-2-1) edge[barred] node[below] {$\ul l_1$} (m-2-2)
									(m-2-2) edge[barred] (m-2-3)
									(m-2-3) edge[barred] (m-2-4)
									(m-2-4) edge[barred] node[below] {$\ul l_n$} (m-2-5);
			\path[transform canvas={shift={($(m-1-1)!0.5!(m-2-1)$)}}]	(m-1-5) edge[cell] node[right] {$w$} (m-2-5);
		\end{tikzpicture}
		= \begin{tikzpicture}[textbaseline]
			\matrix(m)[math175em]
			{ a & & & & b \\
				a_0 & a_1 & \dotsb & a_{n'} & a_n \\
				c_0 & c_1 & \dotsb & c_{n'} & c_n \\ };
			\path[map]	(m-1-1) edge[barred] node[above] {$h$} (m-1-5)
													edge node[left] {$f$} (m-2-1)
									(m-1-5) edge node[right] {$g$} (m-2-5)
									(m-2-1) edge[barred] node[above] {$\ul k_1$} (m-2-2)
													edge node[left] {$p_0$} (m-3-1)
									(m-2-2) edge[barred] (m-2-3)
													edge node[right] {$p_1$} (m-3-2)
									(m-2-3) edge[barred] (m-2-4)
									(m-2-4) edge[barred] node[above] {$\ul k_n$} (m-2-5)
													edge node[left] {$p_{n'}$} (m-3-4)
									(m-2-5) edge node[right] {$p_n$} (m-3-5)
									(m-3-1) edge[barred] node[below] {$\ul l_1$} (m-3-2)
									(m-3-2) edge[barred] (m-3-3)
									(m-3-3) edge[barred] (m-3-4)
									(m-3-4) edge[barred] node[below] {$\ul l_n$} (m-3-5);
			\path[transform canvas={shift=(m-2-1)}]	(m-1-5) edge[cell] node[right] {$u$} (m-2-5);
			\path[transform canvas={shift=($(m-2-2)!0.5!(m-2-3)$)}] (m-2-2) edge[cell] node[right] {$\ul w_1$} (m-3-2)
									(m-2-5) edge[cell] node[right] {$\ul w_n$} (m-3-5);
		\end{tikzpicture}
	\end{displaymath}
	where $u \in A$ and $\ul w \in DJ$, and that this decomposition is unique up to the identification of the pairs $(u, \ul w)$ in the coend above.
\end{example}
	
	The following proposition shows that colax morphisms correspond to companions with lax promorphism structures. If $f$ is a colax morphism for which the lax promorphism $C(f, \id)$ is right pseudo then we will say that $f$ \emph{has a right pseudo companion}. Such colax morphisms have appeared several times in literature. For example, they are crucial to the main results of \cite{Mellies-Tabareau08} and \cite{Grandis-Pare07}, where such morphisms respectively are called `$T$\ndash operadic' or said to satisfy the `left Conduch\'e property'. Also the second condition of the result of Getzler, that is the motivation for this thesis and that was recalled in the introduction, simply means that the symmetric monoidal functor, along which the Kan extension is taken, has a right pseudo companion. For us they will be important because they are the vertical morphisms of $\rcProm T$ (\propref{right colax promorphisms form a pseudo double category}) that admit companions, and thus can be used to define left Kan extensions.
\begin{proposition} \label{colax morphisms correspond to lax companions}
	Let $T$ be a normal monad on an equipment $\K$. For any morphism $\map fAC$ in $\K$, there is bijective correspondence between vertical cells \mbox{$\cell{\bar f}{f \of a}{c \of Tf}$}, that make $f$ into a colax morphism, and cells $\ol{C(f, \id)}$, of the form as on the left below, that make the companion $C(f, \id)$ into a lax promorphism.
	\begin{displaymath}
		\begin{tikzpicture}[baseline]
			\matrix(m)[math175em, column sep=2.5em]{TA & TC \\ A & C \\};
			\path[map]	(m-1-1) edge[barred] node[above] {$TC(f, \id)$} (m-1-2)
													edge node[left] {$a$} (m-2-1)
									(m-1-2) edge node[right] {$c$} (m-2-2)
									(m-2-1) edge[barred] node[below] {$C(f, \id)$} (m-2-2);
			\path[transform canvas={shift={($(m-1-2)!0.5!(m-2-2)$)}, xshift=-1.35em}]	(m-1-1) edge[cell] node[right] {$\ol{C(f, \id)}$} (m-2-1);
		\end{tikzpicture}
		\qquad\qquad\begin{tikzpicture}[baseline]
			\matrix(m)[math175em]
			{	TA & TA & TA & TC \\
				A & A & TC & TC \\
				A & C & C & C \\ };
			\path[map]	(m-1-1) edge node[left] {$a$} (m-2-1)
									(m-1-2) edge node[right] {$a$} (m-2-2)
									(m-1-3) edge[barred] node[above, yshift=2pt] {$TC(f, \id)$} (m-1-4)
													edge node[left] {$Tf$} (m-2-3)
									(m-2-2) edge node[right] {$f$} (m-3-2)
									(m-2-3) edge node[left] {$c$} (m-3-3)
									(m-2-4) edge node[right] {$c$} (m-3-4)
									(m-3-1) edge[barred] node[below] {$C(f, \id)$} (m-3-2);
			\path	(m-1-1) edge[eq] (m-1-2)
						(m-1-2) edge[eq] (m-1-3)
						(m-1-4) edge[eq] (m-2-4)
						(m-2-1) edge[eq] (m-2-2)
										edge[eq] (m-3-1)
						(m-2-3) edge[eq] (m-2-4)
						(m-3-2) edge[eq] (m-3-3)
						(m-3-3) edge[eq] (m-3-4);
			\path[transform canvas={shift=(m-2-3), xshift=-0.7em}]	(m-1-3) edge[cell] node[right] {$T\ls f\eps$} (m-2-3);
			\path[transform canvas={shift=(m-2-3)}]
									(m-2-1) edge[cell] node[right] {$\ls f\eta$} (m-3-1);
			\path[transform canvas={shift=($(m-2-3)!0.5!(m-3-3)$)}] (m-1-2) edge[cell] node[right] {$\bar f$} (m-2-2);
		\end{tikzpicture}
	\end{displaymath}
	Under this assignment a structure cell $\bar f$ is mapped to the composite on the right.
\end{proposition}
\begin{proof}
	First notice that mapping $\bar f$ to the composite above gives a bijection: it follows from the companion identities that an inverse is given by the assignment $\ol{C(f, \id)} \mapsto \ls f\eps \of \ol{C(f, \id)} \of T\ls f\eta$. We have to check that the associativity and unit axioms of $\bar f$ are equivalent to those of $\ol{C(f, \id)}$. To see this, notice that the unit axiom of $T$ implies that the image $T\ol{C(f, \id)}$ is equal to the horizontal composite of $T\ls f\eta$, $T\bar f$ and $T^2 \ls f\eps$. Substituting this into the right-hand side of the associativity axiom for $\ol{C(f, \id)}$, the cells $T\ls f\eta$ (of $T\ol{C(f, \id)}$) and $T\ls f\eps$ (of $\ol{C(f, \id)}$) cancel and what remains is the right-hand side of the associativity axiom for $\bar f$, horizontally composed on the left with $\ls f\eta$ and on the right with $T^2 \ls f\eps$.
	
	The same is true for the left-hand side: recall from \propref{transformations and companions} that $\mu_{C(f, \id)}$ equals to the horizontal composition of $T\ls f\eta$, the identity cell of $Tf \of \mu_A = \mu_C \of T^2 f$ and $T^2\ls f\eps$. Hence after replacing $\mu_{C(f, \id)}$ with this composite, the cells $T\ls f\eta$ and $T\ls f\eps$ cancel in the left-hand side of the axiom for $\ol{C(f, \id)}$ and we see that it too is equal to the left-hand side of the corresponding axiom for $\bar f$, again after composing the latter on the left with $\ls f\eta$ and on the right with $T^2\ls f\eps$. It thus follows from the companion identities that the associativity axioms for $\bar f$ and $\ol{C(f, \id)}$ are equivalent. A similar, easier argument shows that the unit axioms for $\bar f$ and $\ol{C(f, \id)}$ are equivalent as well, finishing the proof.
\end{proof}
\begin{example}
	Given a colax monoidal $\V$-functor $\map fAC$ the structure cell $\cell{\ol{C(f, \id)}}{\fmc C(f, \id)}{C(f \of \tens, \tens)}$  consists of the composites
	\begin{multline} \label{equation:companion monoidal functor}
		C(f \ul x_1, \ul y_1) \tens \dotsb \tens C(f \ul x_n, \ul y_n) \xrar{\tens} C(f \ul x_1 \tens \dotsb \tens f \ul x_n, \ul y_1 \tens \dotsb \tens \ul y_n) \\
		\xrar{C(\bar f, \id)} C\bigpars{f(\ul x_1 \tens \dotsb \tens \ul x_n), \ul y_1 \tens \dotsb \tens \ul y_n}.
	\end{multline}
	In this case, replacing $J_\tens$ in \eqref{equation:right pseudo monoidal profunctor} by the composite above, we find that $C(f, \id)$ being right pseudo coincides with the second condition of Getzler's proposition \cite[Proposition 2.3]{Getzler09}, that was recalled in the introduction, and which has been the main motivation for developing the theory presented here. In particular the companion of the symmetric monoidal (unenriched) functor $\map j{\op{\mathsf F}}{\mathsf H}$ of PROPs, that was given as an example in the introduction, is right pseudomonoidal.
	
	As in the previous example, the companion $C(f, \id)$ of every colax double functor $\map fAC$ admits a lax double profunctor structure, which maps each sequence $\ul w = (\cell{\ul w_1}{f\ul k_1}{\ul l_1}, \dotsc, \cell{\ul w_n}{f\ul k_n}{\ul l_n})$ in $DC(f, \id)$ to the composite
	\begin{displaymath}
		f(\ul k_1 \hc \dotsb \hc \ul k_n) \xRar{\bar f} f\ul k_1 \hc \dotsb \hc f\ul k_n \xRar{\ul w_1 \hc \dotsb \hc \ul w_n} \ul l_1 \hc \dotsb \hc \ul l_n
	\end{displaymath}
	in $C(f \of \hc, \hc)$, where $\bar f$ is the compositor of $f$. (Biased) colax double functors $f$ for which $C(f, \id)$ is right pseudo play an important role in the main theorem of \cite{Grandis-Pare07}, where the existence of a decomposition as in \exref{example:lax double profunctors}, for $J = C(f, \id)$ and restricted to the biased case of $n = 2$, is called the `left Conduch\'e condition'. This result will be discussed in \secref{section:applications of the main theorem}.
\end{example}
\begin{remark}
	Notice that, in general, there is no obvious way to define the horizontal composition of two lax promorphisms $\hmap JAB$ and $\hmap HBC$: for that we need an inverse to $\cell{T_\hc}{TJ \hc TH}{T(J \hc H)}$ so that we can take
	\begin{displaymath}
		\ol{J \hc H} = \bigbrks{T(J \hc H) \xRar{\inv T_\hc} TJ \hc TH \xRar{\bar J \hc \bar H} J \hc H}
	\end{displaymath}
	as a structure cell for the horizontal composite. But, since $T$ is not a pseudo functor in general, such inverses need not exist, so that colax algebras, colax morphisms and lax promorphisms will not in general form a pseudo double category. Instead one can prove that they admit a weaker structure, that of a `virtual' double category. Virtual double categories are related to pseudo double categories in the same way that multicategories are related to monoidal categories: instead of cells $J \Rar K$ with a single morphism as horizontal source, a virtual double category has `multicells' $(J_1, \dotsc, J_n) \Rar K$, with a sequence of morphisms as horizontal source. Their definition can be found in \cite[Section 2]{Cruttwell-Shulman10}; it is also given in \cite[Section 5.1]{Leinster04} where they are called `fc-multicategories'.
	
	We shall not use that lax promorphisms form a virtual double category. Instead, we will work with `right colax promorphisms', that are defined below, which, as we will see, do form a pseudo double category. Specifically, in the case of the lax double functors of the previous example, it is not clear how to arrange double categories, colax (or pseudo) double functors and lax double profunctors into a pseudo double category, as is also remarked on the $n$Lab page \cite{nLab_on_double_profunctors}. Instead we will consider double categories, colax double functors and `right colax double profunctors' (see \exref{example:right colax double profunctors}), which do form a pseudo double category.
\end{remark}

	The idea is to choose the notion of right colax profunctors in such a way that right colax promorphisms with invertible structure cells correspond precisely to right pseudo lax promorphisms, as defined in \defref{definition:right colax promorphisms}. Remember that the structure cell of a right pseudo lax promorphism $\hmap JAB$ corresponds to a horizontal invertible cell $\cell{\rho J}{A(\id, a) \hc TJ}{J \hc B(\id, b)}$, which suggests that structure cells for right colax promorphisms $\hmap HAB$ should be horizontal cells of the form
\begin{displaymath}
	\cell{\bar H}{H \hc B(\id, b)}{A(\id, a) \hc TH}.
\end{displaymath}
	The associativity and unit axioms for $\bar H$ should then be chosen in such a way that when $\bar H = \inv{(\rho J)}$, for a right pseudo lax promorphism $\hmap JAB$, they are equivalent to the corresponding axioms for $J$. To do this we first have to transform the associativity and unit axioms for lax promorphisms into identities in terms of $\rho J$ and $T \rho J$. First we introduce the following convention.
\begin{convention}
	Let $T = (T, \mu, \eta)$ be a normal monad on an equipment $\K$. In the remainder of this chapter we will assume that a companion cell \mbox{$\cell{\ls f\eps}{C(f, \id)}{U_C}$} and conjoint cell $\cell{\eps_f}{C(\id, f)}{U_C}$ have been specified for every morphism $\map fAC$ in $\K$, such that $\ls{\id_C}\eps = \id_{U_C} = \eps_{\id_C}$. Moreover, whenever we use (e.g.\ in \propref{left and right cells}) companions or conjoints for the vertical morphisms of structure cells of algebras, colax morphisms or lax promorphisms (which either consist of a single morphism or a composite of two) we will always make the following choices.
	\begin{itemize}
		\item[-] For composable morphisms $\map fAB$ and $\map gBC$ we choose as companion and conjoint for $g \of f$ the composites $B(f, \id) \hc C(g, \id)$ and $C(\id, g) \hc B(\id, f)$, whose corresponding companion and conjoint cells are compositions of the specified cells $\ls f\eps$, $\ls g\eps$, $\eps_f$ and $\eps_g$, as in \propref{companion compositors}.
		\item[-] For any morphism $\map fAB$ we choose as companion and conjoint cells for $Tf$ the $T$-images $T\ls f\eps$ and $T\eps_f$, see \propref{normal functors preserve companions}.
	\end{itemize}
	
	Consequently for any colax $T$-algebra $A = (A, a, \alpha, \alpha_0)$ the horizontal cells that correspond, under \propref{left and right cells}, to the vertical structure cells $\alpha$ and $\alpha_0$ will be of the form
	\begin{flalign*}
			&&\cell{\lambda\alpha&}{TA(a, \id) \hc A(a, \id)}{TA(\mu_A, \id) \hc A(a, \id)},& \\ 
			&& \cell{\lambda\alpha_0&}{U_A}{TA(\eta_A, \id) \hc A(a, \id)},& \\
			&&\cell{\rho\alpha&}{A(\id, a) \hc TA(\id, \mu_A)}{A(\id, a) \hc TA(\id, a)}, & \\ 
			\text{and}&& \cell{\rho\alpha_0&}{A(\id, a) \hc TA(\id, \eta_A)}{U_A}.&
	\end{flalign*}
	Likewise the horizontal cells $\lambda\bar f$ and $\rho\bar f$ corresponding to the structure cell $\bar f$ of a colax morphism $\map fAC$ will be of the form
	\begin{flalign*}
		&&\cell{\lambda\bar f&}{TC(f, \id) \hc C(c, \id)}{A(a, \id) \hc C(f, \id)} & \\
		\text{and}&& \cell{\rho\bar f&}{C(\id, f) \hc A(\id, a)}{C(\id, c) \hc TC(\id, f)}, &
	\end{flalign*}
	while the $T$-image of the structure cell of a lax $T$-promorphism $\map JAB$ corresponds to the horizontal cells
		\begin{flalign*}
			&&\cell{\rho T\bar J&}{TA(\id, a) \hc T^2 J}{TJ \hc TB(\id, b)}& \\
			\text{and}&& \cell{\lambda T\bar J&}{T^2 J \hc TB(b, \id)}{TA(a, \id) \hc TJ}.&
	\end{flalign*}
\end{convention}
	
\begin{proposition} \label{left and right structure cells}
	Let $T$ be a normal monad on an equipment $\K$, and let $A$ and $B$ be colax $T$-algebras. Given a promorphism $\hmap JAB$ and a cell
	\begin{displaymath}
		\begin{tikzpicture}
			\matrix(m)[math175em]{TA & TB \\ A & B\text, \\};
			\path[map]  (m-1-1) edge[barred] node[above] {$TJ$} (m-1-2)
													edge node[left] {$a$} (m-2-1)
									(m-1-2) edge node[right] {$b$} (m-2-2)
									(m-2-1) edge[barred] node[below] {$J$} (m-2-2);
			\path[transform canvas={shift={($(m-1-2)!(0,0)!(m-2-2)$)}}] (m-1-1) edge[cell] node[right] {$\bar J$} (m-2-1);			
		\end{tikzpicture}
	\end{displaymath}
	the following are equivalent.
	\begin{itemize}
		\item[-] The cell $\bar J$ makes $J$ into a lax $T$-promorphism, that is it satisfies the associativity and unit axioms.
		
		\item[-] The horizontal cell $\cell{\rho\bar J}{A(\id, a) \hc TJ}{J \hc B(\id, b)}$ satisfies the following coherence conditions in $H(\K)$.
	\end{itemize}
	\begin{displaymath}
		\begin{tikzpicture}[textbaseline]
      \matrix(m)[math, column sep=0.75em, row sep=2em]
      { \phantom{T^2B} & TA & \phantom{T^2B} &[-0.8em] T^2 A & \\
        A &\phantom{T^2B} & TB &\phantom{T^2B} & T^2 B \\
        & B & & TB & \\ };
      \path[map]  (m-1-2) edge[barred] node[above] {$TA(\id, \mu_A)$} (m-1-4)
                          edge[barred] node[desc] {$TJ$} (m-2-3)
                  (m-1-4) edge[barred] node[above right] {$T^2 J$} (m-2-5)
                  (m-2-1) edge[barred] node[above left] {$A(\id, a)$} (m-1-2)
                          edge[barred] node[below left] {$J$} (m-3-2)
                  (m-2-3) edge[barred] node[below] {$TB(\id, \mu_B)$} (m-2-5)
                  (m-3-2) edge[barred] node[desc] {$B(\id, b)$} (m-2-3)
                          edge[barred] node[below] {$B(\id, b)$} (m-3-4)
                  (m-3-4) edge[barred] node[below right] {$TB(\id, b)$} (m-2-5);
			\path				($(m-2-1)!0.5!(m-2-3)+(0,0.8em)$) edge[cell] node[right] {$\rho\bar J$} ($(m-2-1)!0.5!(m-2-3)-(0,0.8em)$)
                  (m-1-4) edge[cell, shorten >= 7pt, shorten <= 7pt] node[below right] {$\rho\mu_J$} (m-2-3)
                  (m-2-3) edge[cell, shorten >= 7pt ,shorten <= 7pt] node[below left] {$\rho\beta$} (m-3-4);
    \end{tikzpicture}
    = \begin{tikzpicture}[textbaseline]
      \matrix(m)[math, column sep=0.75em, row sep=2em]
      { \phantom{T^2B} & TA &[-0.8em] \phantom{T^2B} & T^2 A & \\
        A &\phantom{T^2B} & TA &\phantom{T^2B} & T^2 B \\
        & B & & TB & \\ };
      \path[map]  (m-1-2) edge[barred] node[above] {$TA(\id, \mu_A)$} (m-1-4)
                  (m-1-4) edge[barred] node[above right] {$T^2 J$} (m-2-5)
                  (m-2-1) edge[barred] node[above left] {$A(\id, a)$} (m-1-2)
                          edge[barred] node[below] {$A(\id, a)$} (m-2-3)
                          edge[barred] node[below left] {$J$} (m-3-2)
                  (m-2-3) edge[barred] node[desc] {$TA(\id, a)$} (m-1-4)
                          edge[barred] node[desc] {$TJ$} (m-3-4)
                  (m-3-2) edge[barred] node[below] {$B(\id, b)$} (m-3-4)
                  (m-3-4) edge[barred] node[below right] {$TB(\id, b)$} (m-2-5);
      \path				(m-1-2) edge[cell, shorten >= 7pt, shorten <= 7pt] node[below left] {$\rho\alpha$} (m-2-3)
                  (m-2-3) edge[cell, shorten >= 7pt, shorten <= 7pt] node[below right] {$\rho\bar J$} (m-3-2)
                  ($(m-2-3)!0.5!(m-2-5)+(-0.75em,0.8em)$) edge[cell] node[right] {$\rho T\bar J$} ($(m-2-3)!0.5!(m-2-5)-(0.75em,0.8em)$);
    \end{tikzpicture}
  \end{displaymath}
  \begin{displaymath}
    \begin{tikzpicture}[textbaseline]
      \matrix(m)[math, column sep=0.75em, row sep=2em]
      { \phantom{T^2B} & TA & \phantom{T^2B} &[-0.8em] A & \phantom{T^2B} \\
        A & \phantom{T^2B} & TB & \phantom{T^2B} & B \\
        & B & & & \\ };
      \path[map]  (m-1-2) edge[barred] node[above] {$TA(\id, \eta_A)$} (m-1-4)
                          edge[barred] node[desc] {$TJ$} (m-2-3)
                  (m-1-4) edge[barred] node[above right] {$J$} (m-2-5)
                  (m-2-1) edge[barred] node[above left] {$A(\id, a)$} (m-1-2)
                          edge[barred] node[below left] {$J$} (m-3-2)
                  (m-2-3) edge[barred] node[above, inner sep=1.5pt] {$TB(\id, \eta_B)$} (m-2-5)
                  (m-3-2) edge[barred] node[desc] {$B(\id, b)$} (m-2-3);
			\path				(m-3-2) edge[eq, bend right=22] (m-2-5)
									($(m-2-1)!0.5!(m-2-3)+(0,0.8em)$) edge[cell] node[right] {$\rho\bar J$} ($(m-2-1)!0.5!(m-2-3)-(0,0.8em)$)
                  (m-1-4) edge[cell, shorten >= 7pt, shorten <= 7pt] node[above left] {$\rho\eta_J$} (m-2-3)
                  (m-2-3) edge[cell, shorten >= 12.5pt, shorten <= 4.5pt] node[above right] {$\rho\beta_0$} (m-3-4);
    \end{tikzpicture}
    = \begin{tikzpicture}[textbaseline]
      \matrix(m)[math, column sep=0.75em, row sep=2em]
      { \phantom{T^2B} & TA & \phantom{T^2B} &[-0.8em] A & \phantom{T^2B} \\
        A & \phantom{T^2B} & & \phantom{T^2B} & B \\
        & B & & & \\ };
      \path[map]  (m-1-2) edge[barred] node[above] {$TA(\id, \eta_A)$} (m-1-4)
                  (m-1-4) edge[barred] node[above right] {$J$} (m-2-5)
                  (m-2-1) edge[barred] node[above left] {$A(\id, a)$} (m-1-2)
                          edge[barred] node[below left] {$J$} (m-3-2);
			\path				(m-2-1) edge[eq, bend right=22] (m-1-4)
									(m-3-2) edge[eq, bend right=22] (m-2-5)
									(m-1-2) edge[cell, shorten >= 17pt, shorten <= 5pt] node[left, inner sep=7pt] {$\rho\alpha_0$} (m-2-3);
    \end{tikzpicture}
	\end{displaymath}
	\begin{itemize}
			\item[-] The horizontal cell $\cell{\lambda\bar J}{TJ \hc B(b, \id)}{A(a, \id) \hc J}$ satisfies the following coherence conditions in $H(\K)$.
	\end{itemize}
	\begin{displaymath}
    \begin{tikzpicture}[textbaseline]
      \matrix(m)[math, column sep=0.75em, row sep=2em]
      { \phantom{T^2B} & T^2 B & \phantom{T^2B} &[-0.9em] TB & \phantom{T^2B} \\
        T^2 A & & TB & \phantom{T^2B} & B \\
        & TA & & A & \\ };
      \path[map]  (m-1-2) edge[barred] node[above] {$TB(b, \id)$} (m-1-4)
                          edge[barred] node[desc] {$TB(\mu_B, \id)$} (m-2-3)
                  (m-1-4) edge[barred] node[above right] {$B(b, \id)$} (m-2-5)
                  (m-2-1) edge[barred] node[above left] {$T^2 J$} (m-1-2)
                          edge[barred] node[below left] {$TA(\mu_A, \id)$} (m-3-2)
                  (m-2-3) edge[barred] node[below] {$B(b, \id)$} (m-2-5)
                  (m-3-2) edge[barred] node[desc] {$TJ$} (m-2-3)
                          edge[barred] node[below] {$A(a, \id)$} (m-3-4)
                  (m-3-4) edge[barred] node[below right] {$J$} (m-2-5);
      \path				($(m-2-1)!0.5!(m-2-3)+(0,0.8em)$) edge[cell] node[right] {$\lambda\mu_J$} ($(m-2-1)!0.5!(m-2-3)-(0,0.8em)$)
                  (m-1-4) edge[cell, shorten >= 7pt, shorten <= 7pt] node[below right] {$\lambda\beta$} (m-2-3)
                  (m-2-3) edge[cell, shorten >= 7pt, shorten <= 7pt] node[below left] {$\lambda\bar J$} (m-3-4);
    \end{tikzpicture}
    = \begin{tikzpicture}[textbaseline]
      \matrix(m)[math, column sep=0.75em, row sep=2em]
      { \phantom{T^2B} & T^2B &[-0.9em] \phantom{T^2B} & TB & \phantom{T^2B} \\
        T^2 A & & TA & \phantom{T^2B} & B \\
        & TA & & A & \\ };
      \path[map]  (m-1-2) edge[barred] node[above] {$TB(b, \id)$} (m-1-4)
                  (m-1-4) edge[barred] node[above right] {$B(b, \id)$} (m-2-5)
                  (m-2-1) edge[barred] node[above left] {$T^2 J$} (m-1-2)
                          edge[barred] node[below] {$TA(a, \id)$} (m-2-3)
                          edge[barred] node[below left] {$TA(\mu_A, \id)$} (m-3-2)
                  (m-2-3) edge[barred] node[desc] {$TJ$} (m-1-4)
                          edge[barred] node[desc] {$A(a, \id)$} (m-3-4)
                  (m-3-2) edge[barred] node[below] {$A(a, \id)$} (m-3-4)
                  (m-3-4) edge[barred] node[below right] {$J$} (m-2-5);
      \path				(m-1-2) edge[cell, shorten >= 7pt, shorten <= 7pt] node[below left] {$\lambda T\bar J$} (m-2-3)
                  (m-2-3) edge[cell, shorten >= 7pt, shorten <= 7pt] node[below right] {$\lambda\alpha$} (m-3-2)
                  ($(m-2-3)!0.5!(m-2-5)+(0,0.75em)$) edge[cell] node[right] {$\lambda\bar J$} ($(m-2-3)!0.5!(m-2-5)-(0,0.75em)$);
    \end{tikzpicture}
  \end{displaymath}
	\begin{displaymath}
		\begin{tikzpicture}[textbaseline]
      \matrix(m)[math, column sep=0.75em, row sep=2em]
      { \phantom{T^2B} & B & \phantom{T^2B} &[-0.9em] & \phantom{T^2B} \\
        A & \phantom{T^2B} & TB & \phantom{T^2B} & B \\
        & TA & & A & \\ };
      \path[map]  (m-1-2) edge[barred] node[desc] {$TB(\eta_B, \id)$} (m-2-3)
                  (m-2-1) edge[barred] node[above left] {$J$} (m-1-2)
                          edge[barred] node[below left] {$TA(\eta_A, \id)$} (m-3-2)
                  (m-2-3) edge[barred] node[below] {$B(b, \id)$} (m-2-5)
                  (m-3-2) edge[barred] node[desc] {$TJ$} (m-2-3)
                          edge[barred] node[below] {$A(a, \id)$} (m-3-4)
                  (m-3-4) edge[barred] node[below right] {$J$} (m-2-5);
      \path				(m-1-2) edge[eq, bend left=22] (m-2-5)
									($(m-2-1)!0.5!(m-2-3)+(0,0.8em)$) edge[cell] node[right] {$\lambda\eta_J$} ($(m-2-1)!0.5!(m-2-3)-(0,0.8em)$)
                  (m-1-4) edge[cell, shorten >= 4.5pt, shorten <= 12.5pt] node[below right] {$\lambda\beta_0$} (m-2-3)
                  (m-2-3) edge[cell, shorten >= 7pt, shorten <= 7pt] node[below left] {$\lambda\bar J$} (m-3-4);
    \end{tikzpicture}
    = \begin{tikzpicture}[textbaseline]
      \matrix(m)[math, column sep=0.75em, row sep=2em]
      { \phantom{T^2B} & B &[-0.9em] \phantom{T^2B} & & \phantom{T^2B} \\
        A & \phantom{T^2B} & & \phantom{T^2B} & B \\
        & TA & & A & \\ };
      \path[map]  (m-2-1) edge[barred] node[above left] {$J$} (m-1-2)
                          edge[barred] node[below left] {$TA(\eta_A, \id)$} (m-3-2)
                  (m-3-2) edge[barred] node[below] {$A(a, \id)$} (m-3-4)
                  (m-3-4) edge[barred] node[below right] {$J$} (m-2-5);
      \path				(m-1-2) edge[eq, bend left=22] (m-2-5)
									(m-2-1) edge[eq, bend left=22] (m-3-4)
									(m-2-3) edge[cell, shorten >= 4.5pt, shorten <= 12.5pt] node[below right] {$\lambda\alpha_0$} (m-3-2);
    \end{tikzpicture}
  \end{displaymath}
	Moreover in the associativity condition for $\rho\bar J$ the horizontal cell $\rho T\bar J$ can be rewritten in terms of $T\rho\bar J$ as below, where the inverse of the compositor exists by \propref{normal functors preserve companions}. Likewise $\lambda T\bar J = \inv T_\hc \of T\lambda \bar J \of T_\hc$.
	\begin{multline*}
		\rho T\bar J = \bigbrks{TA(\id, a) \hc T^2 J \xRar{T_\hc} T\bigpars{A(\id, a) \hc TJ} \\ \xRar{T \rho\bar J} T\bigpars{J \hc B(\id, b)} \xRar{\inv T_\hc} TJ \hc TB(\id, b)}
	\end{multline*}
\end{proposition}
\begin{proof}
	For each lax $T$-promorphism $\hmap JAB$ it follows, from the naturality of $\rho$ (\propref{left and right cells}), that we obtain the two identities for $\rho\bar J$ above when, under the convention above, we apply $\rho$ to the associativity and unit axioms for the structure cell $\bar J$ (\defref{definition:lax promorphism}). On the other hand, since the assignment $\rho$ is invertible, the two identities above imply the associativity and unit axioms for $\bar J$. The second assertion is \propref{lax functors and left and right cells}. For $\lambda\bar J$ the arguments are symmetric.
\end{proof}
	
	To be able to rewrite the coherence conditions of \propref{left and right structure cells} in terms of $\inv{(\rho \bar J)}$ (provided this inverse exists) we need the inverses of $\rho\mu_J$ and $\rho\eta_J$ to exist. This leads to the following definition.
\begin{definition} \label{definition:right suitable monads}
	A normal monad $T = (T, \mu, \eta)$ on an equipment $\K$ is called \emph{right suitable} if for each promorphism $\hmap JAB$ the cells $\mu_J$ and $\eta_J$ are right invertible.
\end{definition}
	Cruttwell and Shulman consider pseudomonads with the property that all cells $\mu_J$ and $\eta_J$ are left invertible in \cite[Appendix A]{Cruttwell-Shulman10}. They call such monads `horizontally strong'.
\begin{example}
	Remember the `free monoid'-monad $\fmc$ on the equipment $\Mat\V$ of $\V$-matrices that was introduced in \secref{section:monad examples}: it maps a set $A$ to the free monoid $\fmc A$ generated by $A$ while the image $\fmc J$ of a $\V$-matrix $\hmap JAB$ is given by $\fmc J(\ul x, \ul y) = \fmc_n J(\ul x, \ul y) = \Tens_{i = 1}^n J(\ul x_i, \ul y_i)$; the monad $\fmc$ induces the `free strict monoidal $\V$\ndash category'-monad $\fmc = \Mod\fmc$ on the equipment $\enProf\V$ of $\V$\ndash profunctors. Also recall the modification $\fsmc$ of $\fmc$, that is given by $\fsmc J(\ul x, \ul y) = \coprod_{\sigma \in \sym_n} \fmc_n J(\act{\ul x}\sigma, \ul y)$, which induces the `free symmetric strict monoidal $\V$-ca\-te\-go\-ry'\ndash monad $\fsmc = \Mod\fsmc$ on $\enProf\V$. Both monads $\fmc$ and $\fsmc$ on $\Mat\V$ are right suitable, as is shown below, while it follows from the proposition further below that the monads $\fmc = \Mod\fmc$ and $\fsmc = \Mod\fsmc$, that they induce, are right suitable as well.
	
	Let $\hmap JAB$ be a $\V$-matrix; we have to show that the $\V$-natural transformation $\cell{\rho\mu_J}{\fmc A(\id, \mu_A) \hc \fmc^2 J}{\fmc J(\id, \mu_B)}$ is invertible. Working out the components of its source we find
	\begin{displaymath}
		\bigpars{\fmc A(\id, \mu_A) \hc \fmc^2 J}(\ul x, \ull y) = \coprod_{\ull v \in \fmc^2 A} \fmc A(\ul x, \mu_A\ull v) \tens \fmc^2 J(\ull v, \ull y),
	\end{displaymath}
	where $\fmc A(\ul x, \mu_A\ull v) \iso 1$ only if $\mu_A\ull v = \ul x$, and $\emptyset$ otherwise, while $\fmc^2 J(\ull v, \ull y) \neq \emptyset$ only if $\card{\ull v} = n = \card{\ull y}$ and $\card{\ull v_i} = m_i = \card{\ull y_i}$, for all $i = 1, \dotsc, n$. So the conditions, under which the $\ull v$-component in the coproduct above is not the initial object, determine $\ull v$ uniquely, if it exists. Notice that it does exist if and only if $\card x = \Sum_{i=1}^n \card{\ull y_i}$, since in that case we can take $\ull v_i = (\ul x_{M_i + 1}, \dotsc, \ul x_{M_i + m_i})$ where $M_i = m_1 + \dotsb + m_{i - 1}$. But the latter is exactly the condition under which $MJ(\ul x, \mu_B\ull y)$, the target of $(\rho\mu_J)_{\ul x, \ull y}$, is not the initial object, and we conclude that the component $(\rho\mu_J)_{\ul x, \ull y}$ is either the initial map $\emptyset \to \emptyset$, when $\card{\ul x} \neq \Sum_{i=1}^n\card{\ull y_i}$, or, in the case of equality, reduces to the isomorphism $\mu_J \colon \fmc^2 J(\ull v, \ull y) \xrar\iso \fmc J(\mu_A \ull v, \mu_B \ull y)$, where $\mu_A \ull v = \ul x$, that is given by \eqref{equation:multiplication of fmc}. This proves that each of the cells $\mu_J$ is right invertible. That the same holds for $\eta_J$ is easily seen: the components $(\rho\eta_J)_{\ul x, y}$ either equal the initial map $\emptyset \to \emptyset$, when $\card{\ul x} \neq 1$, or the identity on $J(\ul x_1, y)$, when $\card{\ul x} = 1$. We conclude that $\fmc$ is right suitable. As mentioned above this implies, by the proposition below, that $\fmc = \Mod\fmc$ on $\enProf\V$ is right suitable as well. In particular the inverse of $\rho\mu_J$, for a $\V$-profunctor $\hmap JAB$, consists of the maps
	\begin{equation} \label{morphism:inverse of right multiplication cell for fmc}
		\fmc J(\ul x, \mu_B \ull y) \iso \fmc^2 J(\ull v, \ull y) \to \int^{\ull w \in \fmc^2 A} \fmc A(\ul x, \mu_A \ull w) \tens \fmc^2 J(\ull w, \ull y),
	\end{equation}
	for $\ul x$ and $\ull y$ such that $\card{\ull y_1} + \dotsb + \card{\ull y_n} = \card{\ul x}$, where $\ull v$ and the isomorphism are as given above and where the second map inserts at $\ull w = \ull v$, using the identity $1 \to \fmc A(\ul x, \mu_A \ull v)$.
	
	That the monad $\fsmc$ on $\Mat\V$ is right suitable as well follows from applying $\rho$ to the identities $(\mu_\fsmc)_J \of \theta^2_J = \theta_J \of (\mu_\fmc)_J$ and $(\eta_\fsmc)_J = \theta_J \of (\eta_\fmc)_J$, that hold because $\nat\theta MS$ is a morphism of monads. Indeed, they give commuting diagrams in which we know of all but the cells $\rho(\mu_\fsmc)_J$ and $\rho(\eta_\fsmc)_J$ that they are invertible, by \propref{theta cells are right invertible}; the latter are therefore invertible as well.
\end{example}
\begin{example}
	A monad $T = (T, \mu, \eta)$ on a category $\E$ with finite limits is called \emph{cartesian} if $T$ preserves pullbacks and all the naturality squares of $\mu$ and $\eta$ are pullbacks. It follows from \exref{example:right invertible spans} that every such cartesian monad $T$ induces a right suitable monad $\Span T$ on $\Span \E$. For example the free category monad $F$ on $\ps{\GG_1}$ is cartesian (see \cite[Example 6.5.3]{Leinster04} or for an elementary proof \cite[Proposition 2.3]{Dawson-Pare-Pronk06}) so that the induced monad $\Span F$ on $\Span{\ps{\GG_1}}$ is right suitable. It follows from the proposition below that $D = \Mod{\Span F}$, the free strict double category monad, is right suitable as well.
\end{example}

\begin{example} \label{example:other monads}
	The `free strict $\omega$-category'-monad $T$, that we briefly mentioned at the end of \secref{section:monad examples}, is cartesian as well. This is shown by Leinster in \cite[Appendix F]{Leinster04}. It follows that the induced monad $\inProf T$, for monoidal globular categories, is right suitable. The ultrafilter-monad $\beta$, that we also mentioned, is not right suitable: one can verify that components of its multiplication are right invertible, but those of its unit are not, see \cite[Example A.7]{Cruttwell-Shulman10}
\end{example}

	Recall that $\qlEquip$ denotes the $2$-category of equipments $\K$, for which every $H(\K)(A, B)$ has reflexive coequalisers preserved by $\hc$ on both sides, lax functors and the transformations between them. Its sub-$2$-category $\qnEquip$ consists of normal functors.
\begin{proposition}
	The $2$-functor $\map{\mathsf{Mod}}\qlEquip\qnEquip$ of \propref{mod is a 2-functor} preserves right suitable monads.
\end{proposition}
\begin{proof}
	The multiplication $\Mod \mu$ of the image of a monad $T = (T, \mu, \eta)$ on $\K$ is given by the cells $\mu_A$ that form morphisms of monoids $\map{\mu_A}{T^2 A}{TA}$ (remember a monoid $A$ is a promorphism $\hmap A{A_0}{A_0}$ equipped with horizontal multiplication and unit cells), and by the cells $\mu_J$ that form cells of bimodules $\cell{\mu_J}{T^2 J}{TJ}$ for every bimodule $\hmap JAB$. If $T$ is right suitable then the $\K$-cells underlying $\mu_A$ and $\mu_J$ are right invertible, so that by \propref{right invertible cells between bimodules} the cell $\mu_J$, as a cell of bimodules, is right invertible in $\Mod\K$. In the same way we see that each $\eta_J$ is right invertible in $\Mod\K$ as well, which completes the proof.
\end{proof}
	We are now ready to define right colax promorphisms.
\begin{definition} \label{definition:right colax promorphisms}
	Let $T$ be a right suitable normal monad on an equipment $\K$. Given colax $T$-algebras $A$ and $B$, a \emph{right colax $T$-promorphism} $\hmap JAB$ is a promorphism $\hmap JAB$ equipped with a horizontal structure cell
	\begin{displaymath}
		\begin{tikzpicture}
			\matrix(m)[math175em]{A & B & TB \\ A & TA & TB \\};
			\path[map]	(m-1-1) edge[barred] node[above] {$J$} (m-1-2)
									(m-1-2) edge[barred] node[above] {$B(\id, b)$} (m-1-3)
									(m-2-1) edge[barred] node[below] {$A(\id, a)$} (m-2-2)
									(m-2-2) edge[barred] node[below] {$TJ$} (m-2-3);
			\path				(m-1-1) edge[eq] (m-2-1)
									(m-1-2) edge[cell] node[right] {$\bar J$} (m-2-2)
									(m-1-3) edge[eq] (m-2-3);
		\end{tikzpicture}
	\end{displaymath}
	satisfying the following associativity and unit axioms in $H(\K)$.
	\begin{displaymath}
		\begin{tikzpicture}[baseline={([yshift=12pt]current bounding box.center)}]
      \matrix(m)[math, column sep=0.75em, row sep=2em]
      { \phantom{T^2B} & B &[-0.8em] \phantom{T^2B} & TB & \\
        A &\phantom{T^2B} & TA &\phantom{T^2B} & T^2 B \\
        & TA & & T^2 A & \\ };
      \path[map]  (m-1-2) edge[barred] node[above] {$B(\id, b)$} (m-1-4)
                  (m-1-4) edge[barred] node[above right] {$TB(\id, \mu_B)$} (m-2-5)
                  (m-2-1) edge[barred] node[above left] {$J$} (m-1-2)
                          edge[barred] node[above, inner sep=2pt] {$A(\id, a)$} (m-2-3)
                          edge[barred] node[below left] {$A(\id, a)$} (m-3-2)
                  (m-2-3) edge[barred] node[desc] {$TJ$} (m-1-4)
                          edge[barred] node[desc] {$TA(\id, \mu_A)$} (m-3-4)
                  (m-3-2) edge[barred] node[below, inner sep=2pt] {$TA(\id, a)$} (m-3-4)
													edge[barred, bend right=65, looseness=1.4] node[below, inner sep=6pt] {$T\bigpars{A(\id, a) \hc TJ}$} coordinate (p) (m-2-5)
                  (m-3-4) edge[barred] node[desc] {$T^2 J$} (m-2-5);
      \path				(m-1-2) edge[cell, shorten >= 7pt, shorten <= 7pt] node[above right] {$\bar J$} (m-2-3)
                  (m-2-3) edge[cell, shorten >= 7pt, shorten <= 7pt] node[above left] {$\rho\alpha$} (m-3-2)
                  ($(m-2-3)!0.5!(m-2-5)+(-1.25em,0.8em)$) edge[cell] node[right] {$\inv{(\rho \mu_J)}$} ($(m-2-3)!0.5!(m-2-5)-(1.25em,0.8em)$)
                  (m-3-4) edge[cell, shorten >= 6pt, shorten <=0pt] node[left, yshift=1pt] {$T_\hc$} (p);
    \end{tikzpicture}
    = \begin{tikzpicture}[baseline={([yshift=12pt]current bounding box.center)}, trim right=8.75em]
      \matrix(m)[math, column sep=0.75em, row sep=2em]
      { \phantom{T^2B} & B & \phantom{T^2B} &[-0.8em] TB & \\
        A &\phantom{T^2B} & TB &\phantom{T^2B} & T^2 B \\
        & TA & & & \\ };
      \path[map]  (m-1-2) edge[barred] node[above] {$B(\id, b)$} (m-1-4)
                          edge[barred] node[desc, xshift=-3pt] {$B(\id, b)$} (m-2-3)
                  (m-1-4) edge[barred] node[above right] {$TB(\id, \mu_B)$} (m-2-5)
                  (m-2-1) edge[barred] node[above left] {$J$} (m-1-2)
                          edge[barred] node[below left] {$A(\id, a)$} (m-3-2)
                  (m-2-3) edge[barred] node[above, inner sep=1pt] {$TB(\id, b)$} (m-2-5)
                  (m-3-2) edge[barred] node[desc] {$TJ$} (m-2-3)
                          edge[barred, bend right =20, looseness=1.25] node[desc] {$T\bigpars{J \hc B(\id, b)}$} coordinate (q) (m-2-5)
                  (m-3-2) edge[barred, bend right =65, looseness=1.4] node[below, inner sep=6pt] {$T\bigpars{A(\id, a) \hc TJ}$} coordinate (p) (m-2-5);
			\path				($(m-2-1)!0.5!(m-2-3)+(0,0.8em)$) edge[cell] node[right] {$\bar J$} ($(m-2-1)!0.5!(m-2-3)-(0,0.8em)$)
                  (m-1-4) edge[cell, shorten >= 7pt, shorten <= 7pt] node[above left, inner sep=1pt] {$\rho\beta$} (m-2-3)
                  (m-2-3) edge[cell, shorten >= 12pt, shorten <= 4pt] node[above right] {$T_\hc$} (q)
                  (m-2-3) edge[cell, shorten >= 8pt, shorten <= 37pt] node[below, inner sep=10pt, xshift=-5pt] {$T\bar J$} (p);
    \end{tikzpicture}
  \end{displaymath}
  \begin{displaymath}
    \begin{tikzpicture}[textbaseline]
      \matrix(m)[math, column sep=0.75em, row sep=2em]
      { \phantom{T^2 B} & B &[-0.8em] \phantom{T^2 B} & TB & \phantom{T^2 B} \\
        A & \phantom{T^2 B} & TA & \phantom{T^2 B} & B \\
        & & & A & \\ };
      \path[map]  (m-1-2) edge[barred] node[above] {$B(\id, b)$} (m-1-4)
									(m-1-4) edge[barred] node[above right] {$TB(\id, \eta_B)$} (m-2-5)
                  (m-2-1) edge[barred] node[above left] {$J$} (m-1-2)
                          edge[barred] node[above] {$A(\id, a)$} (m-2-3)
                  (m-2-3) edge[barred] node[desc] {$TJ$} (m-1-4)
													edge[barred] node[desc] {$TA(\id, \eta_A)$} (m-3-4)
                  (m-3-4) edge[barred] node[below right] {$J$} (m-2-5);
      \path				(m-1-2) edge[cell, shorten >= 7pt, shorten <= 7pt] node[above right] {$\bar J$} (m-2-3)
                  (m-2-3) edge[cell, shorten >= 12.5pt, shorten <= 4.5pt] node[above left, xshift=2pt, yshift=2pt] {$\rho\alpha_0$} (m-3-2)
                  ($(m-2-3)!0.5!(m-2-5)+(-1.25em,0.8em)$) edge[cell] node[right] {$\inv{(\rho\eta_J)}$} ($(m-2-3)!0.5!(m-2-5)-(1.25em,0.8em)$)
									(m-2-1) edge[eq, bend right=22] (m-3-4);
    \end{tikzpicture}
    = \begin{tikzpicture}[textbaseline]
      \matrix(m)[math, column sep=0.75em, row sep=2em]
      { \phantom{T^2 B} & B & \phantom{T^2 B} &[-0.8em] TB & \phantom{T^2 B} \\
        A & \phantom{T^2 B} & & \phantom{T^2 B} & B \\
        & & & A & \\ };
      \path[map]  (m-1-2) edge[barred] node[above] {$B(\id, b)$} (m-1-4)
									(m-1-4) edge[barred] node[above right] {$TB(\id, \eta_B)$} (m-2-5)
									(m-2-1) edge[barred] node[above left] {$J$} (m-1-2)
									(m-3-4) edge[barred] node[below right] {$J$} (m-2-5);
      \path	(m-1-2) edge[eq, bend right=22] (m-2-5)
						(m-2-1) edge[eq, bend right=22] (m-3-4)
						(m-1-4) edge[cell, shorten >= 12.5pt, shorten <= 4.5pt] node[below right, xshift=3pt, yshift=5pt] {$\rho\beta_0$} (m-2-3);
    \end{tikzpicture}
	\end{displaymath}
\end{definition}
\begin{remark}
	It is readily seen that, for a right pseudopromorphism $\hmap HAB$, the axioms above for $\bar J = \inv{(\rho\bar H)}$ are indeed equivalent to those of \propref{left and right structure cells} for $\rho\bar H$, which was our goal. It follows that right pseudo lax promorphisms correspond to pseudo right colax promorphisms. We will not distinguish these two notions and call promorphisms of either type simply \emph{right pseudopromorphisms}.
	
	Also notice the necessity of the existence of the inverses of $\rho\mu_J$ and $\rho\eta_J$, and the necessity of considering colax algebras, in the above definition. Indeed, in the associativity axiom both $\inv{(\rho\mu_J)}$ and $\rho\alpha$ are `kept in place' by $\bar J$ and $T_\hc$: when neither $\bar J$ nor $T_\hc$ is invertible it is not possible to rewrite the axiom in terms of $\rho\mu_J$ instead of $\inv{(\rho\mu_J)}$, or in terms of a lax algebra structure cell $\cell{\alpha'}{a \of Ta}{a \of \mu_A}$ instead of $\alpha$. In the same way $\inv{(\rho\eta_J)}$ is kept in place in the unit axiom.
\end{remark}

	We shall first describe the relatively simple notion of (unenriched) right colax monoidal profunctors. Before we start however, the author must admit that the only such profunctors he knows are all right pseudomonoidal profunctors, that is with invertible structure cells, as described in \exref{example:lax monoidal profunctors}.
	
	Consider an ordinary profunctor $\hmap JAB$ between normal colax unenriched monoidal categories $A$ and $B$. Loosely speaking, a right colax structure on $J$ is a weakening of a right pseudo structure on $J$: while the latter gives a decomposition \eqref{equation:decomposition of right pseudoprofunctors} of each map $\map jx{\ul z_1 \tens \dotsb \tens \ul z_n}$ in $J$ as $j = (\ul j_1 \tens \dotsb \tens \ul j_n) \of f$, where $\map fx{\ul y_1 \tens \dotsb \tens \ul y_n}$ and $\map{j_s}{\ul y_s}{\ul z_s}$, a right colax structure only gives a coherent assignment $j \mapsto (f, \ul j)$, with no further relation between $(f, \ul j)$ and $j$. In particular a right colax structure does not define tensor products $\ul j_1 \tens \dotsb \tens \ul j_n$ of maps in $J$.
	
	To make this formal, by a \emph{right splitting} of a map $\map jx{\ul z_1 \tens \dotsb \tens \ul z_n}$ in $J$ we mean a pair $(f, \ul j)$ consisting of a map $\map fx{\ul y_1 \tens \dotsb \tens \ul y_n}$ in $A$ and a sequence $\map{\ul j}{\ul y}{\ul z}$ in $\fmc_n J$, of maps $\map{\ul j_s}{\ul y_s}{\ul z_s}$ in $J$. Such a right splitting $(f, \ul j)$ of $j$ is drawn as
\begin{displaymath}
	\begin{tikzpicture}
		\matrix(m)[math, row sep=1.5em]
		{	x & & \ul z_1 \tens \dotsb \tens \ul z_n \\
			& \ul y_1 \tens \dotsb \tens \ul y_n & \\ };
		\path[map]	(m-1-1) edge node[above] {$j$} (m-1-3)
												edge node[below left] {$f$} (m-2-2);
		\draw	($(m-1-3)+(3.25em,-1.25em)$) node {$\ul z$.} coordinate (z)
					($(m-2-2)+(3.25em,-1.25em)$) node {$\ul y$} coordinate (y);
		\path[map] (y) edge[shorten >=7pt, shorten <= 5pt] node[below right] {$\ul j$} (z);
	\end{tikzpicture}
\end{displaymath}
	Two splittings $(f, \ul j)$ and $(f', \ul j')$ of $j$ are called equivalent if they are identified in the coend  $\int^{\ul y \in \fmc_n A} A(x, \ul y_1 \tens \dotsb \tens \ul y_n) \times \fmc_n J(\ul y, \ul z)$. Thus the equivalence relation on splittings $(f, \ul j)$ is generated by the pairs $(f, \ul j) \sim (f', \ul j')$ for which there exists a sequence $\map{\ul g}{\ul y}{\ul y'}$ such that $(\ul g_1 \tens \dotsb \tens \ul g_n) \of f = f'$ and $\ul j' \of \ul g = \ul j$.
	
	Likewise we consider \emph{(right) double splittings} of maps
\begin{displaymath}
	\map jx{(\ull z_{11} \tens \dotsb \tens \ull z_{1m_1}) \tens \dotsb \tens (\ull z_{n1} \tens \dotsb \tens \ull z_{nm_n})}
\end{displaymath}
	to be pairs $(f, \ull j)$ where $\map fx{(\ull y_{11} \tens \dotsb \tens \ull y_{1m_1}) \tens \dotsb \tens (\ull y_{n1} \tens \dotsb \tens \ull y_{nm_n})}$ is a map in $A$ and $\map{\ull j}{\ull y}{\ull z}$ is a double sequence in $\fmc^2 J$. Again we call $(f, \ull j)$ and $(f', \ull j')$ equivalent if they are identified in the coend $\int^{\ull y \in \fmc^2 A} A\bigpars{x, (\mathrm\tens \of \fmc\mathrm\tens)(\ull y)} \times M^2 J(\ull y, \ull z)$.
\begin{example} \label{example:right colax unenriched promorphisms}
	A \emph{right colax monoidal} profunctor $\hmap JAB$, between normal colax monoidal categories $A$ and $B$, comes equipped with a right splitting $\bar J(j) = (f, \ul j)$ of each map $\map jx{\ul z_1 \tens \dotsb \tens \ul z_n}$, as above, such that the following naturality, associativity and unit axioms hold.
	\begin{itemize}[label=-]
		\item For $\map gyx$ in $A$, $\map jx{\ul z_1 \tens \dotsb \tens \ul z_n}$ in $J$ and $\map{\ul h}{\ul z}{\ul t}$ in $\fmc_n B$, if $\bar J(j) = (f, \ul j)$ then the splittings $\bar J\bigpars{(\ul h_1 \tens \dotsb \tens \ul h_n) \of j \of g}$ and $(f \of g, \ul h \of \ul j)$ are equivalent.
		\item For any $\map jx{\ull z_{11} \tens \dotsb \tens \ull z_{1m_1} \tens \dotsb \tens \ull z_{n1} \tens \dotsb \tens \ull z_{nm_n}}$ the double splittings of $\mathfrak a_B \of j$
		\begin{displaymath}
			(\mathfrak a_A \of f, \ull j) \qquad \text{and} \qquad \bigpars{(\ul h_1 \tens \dotsb \tens \ul h_n) \of g, \ull l}
		\end{displaymath}
		as in the diagrams below, are equivalent. Here the left-hand side is obtained by first splitting $j$ as $\bar J(j) = (f, \ul j)$, where the target of $f$ is $\ul y_1 \tens \dotsb \tens \ul y_N$, with $N = m_1 + \dotsb + m_n$. Then there is a unique way of decomposing $\ul j$ into a double sequence $\map{\ull j}{\ull y}{\ull z}$ such that $\mu_B \ull j = \ul j$\footnote{In fact $\ull j$ is the image of $\ul j$ under the cell $\cell{\inv{(\rho\mu_J)}}{TJ(\id, \mu_B)}{TA(\id, \mu_A) \hc_{T^2 A} T^2 J}$ in the associativity axiom of \defref{definition:right colax promorphisms}.}, making $(\mathfrak a_A \of f, \ull j)$ a double splitting of $\mathfrak a_B \of j$.
	\end{itemize}
	\begin{displaymath}
		\begin{tikzpicture}
			\matrix(m)[math, column sep=-5em, row sep=3em]
			{	x & \ull z_{11} \tens \dotsb \tens \ull z_{1m_1} \tens \dotsb \tens \ull z_{n1} \tens \dotsb \tens \ull z_{nm_n} \\
				\ul y_1 \tens \dotsb \tens \ul y_N & \\
				& (\ull z_{11} \tens \dotsb \tens \ull z_{1m_1}) \tens \dotsb \tens (\ull z_{n1} \tens \dotsb \tens \ull z_{nm_n})\\
				(\ull y_{11} \tens \dotsb \tens \ull y_{1m_1}) \tens \dotsb \tens (\ull y_{n1} \tens \dotsb \tens \ull y_{nm_n}) & \\ };
			\path[map]	(m-1-1) edge node[above] {$j$} (m-1-2)
													edge node[left] {$f$} (m-2-1)
									(m-2-1) edge node[left] {$\mathfrak a_A$} (m-4-1);
			\path[map, transform canvas={xshift=4.5em}]	(m-1-2) edge node[right] {$\mathfrak a_B$} (m-3-2);
			\draw	($(m-1-2)+(-2.75em,-1.75em)$) node {$\mu_B\ull z$} coordinate (z)
						($(m-2-1)+(3.25em,-1.75em)$) node {$\ul y$} coordinate (y)
						($(m-3-2)+(4.5em,-1.75em)$) node {$\ull z$} coordinate (zz)
						($(m-4-1)+(10.5em,-1.75em)$) node {$\ull y$} coordinate (yy);
			\path[map]	(y) edge[shorten >=12pt, shorten <= 5pt] node[below right] {$\ul j$} (z)
						(yy) edge[shorten >=6pt, shorten <=5pt] node[below right] {$\ull j$} (zz);
		\end{tikzpicture}
	\end{displaymath}
	\begin{itemize}[label=\phantom{-}]
		\item The right-hand side is given by first splitting $\mathfrak a_B \of j$ into $\bar J(\mathfrak a_B \of j) = (g, \ul k)$ followed by splitting each of the $\ul k_i$ as $\bar J(\ul k_i) = (h_i, \ul l_i)$. Together the latter form sequences $\ul h$ in $\fmc_n A$ and $\ull l$ in $\fmc^2 J$, making $\bigpars{(\ul h_1 \tens \dotsb \tens \ul h_n) \of g, \ull l}$ a double splitting of $\mathfrak a_B \of j$ as well.
	\end{itemize}
	\begin{displaymath}
		\begin{tikzpicture}
			\matrix(m)[math, column sep=-6em, row sep=3em]
			{	x & \ull z_{11} \tens \dotsb \tens \ull z_{1m_1} \tens \dotsb \tens \ull z_{n1} \tens \dotsb \tens \ull z_{nm_n} \\
				& (z_{11} \tens \dotsb \tens \ull z_{1m_1}) \tens \dotsb \tens (\ull z_{n1} \tens \dotsb \tens \ull z_{nm_n})\\
				\ul v_1 \tens \dotsb \tens \ul v_n & \\
				(\ull w_{11} \tens \dotsb \tens \ull w_{1m_1}) \tens \dotsb \tens (\ull w_{n1} \tens \dotsb \tens \ull w_{nm_n}) & \\ };
			\path[map]	(m-1-1) edge node[above] {$j$} (m-1-2)
													edge node[left] {$g$} (m-3-1)
									(m-3-1) edge node[left] {$\ul h_1 \tens \dotsb \tens \ul h_n$} (m-4-1)
									(m-1-2) edge node[right] {$\mathfrak a_B$} (m-2-2);
			\draw	($(m-2-2)+(-2.75em,-1.75em)$) node {$(\fmc_n\mathrm{\tens})(\ull z)$} coordinate (z)
						($(m-3-1)+(3.25em,-1.75em)$) node {$\ul v$} coordinate (y)
						($(m-2-2)+(4.5em,-1.75em)$) node {$\ull z$} coordinate (zz)
						($(m-4-1)+(11em,-1.75em)$) node {$\ull w$} coordinate (yy);
			\path[map]	(y) edge[shorten >=12pt, shorten <= 5pt] node[below right] {$\ul k$} (z)
						(yy) edge[shorten >=6pt, shorten <=7pt] node[below right] {$\ull l$} (zz);
		\end{tikzpicture}
	\end{displaymath}
	\begin{itemize}[label=-]
		\item For each $\map jx{\ul z_1}$ the splitting $\bar J(j) = (f, \ul j)$ satisfies $\ul j_1 \of f = j$.
	\end{itemize}
	It is easily seen that, using the axiom of choice if necessary, choosing a right splitting of each map $\map jx{\ul z_1 \tens \dotsb \tens \ul z_n}$ in $J$, such that the first axiom above holds, is equivalent to giving a natural transformation $\nat{\bar J}{J(\id, \mathrm\tens)}{A(\id, \mathrm\tens) \times_{\fmc A} \fmc J}$ as in \defref{definition:right colax promorphisms}. It is then straightforward to check that the associativity axiom above corresponds to that of \defref{definition:right colax promorphisms}, once we have postcomposed the latter with the inverse of the compositor $\fmc_\hc$ (which exists because $\fmc$ is a pseudomonad), and that the unit axioms coincide as well. 
\end{example}
	
	The enriched version is as follows.
\begin{example} \label{example:right colax enriched promorphisms}
	Given normal colax monoidal $\V$-categories $A$ and $B$, a \emph{right colax monoidal $\V$\ndash profunctor} $\hmap JAB$ is a $\V$-profunctor $\hmap JAB$ equipped with $\V$-natural maps
	\begin{displaymath}
		\map{\bar J}{J(x, \ul z_1 \tens \dotsb \tens \ul z_n)}{\int^{\ul y \in \fmc_n A} A(x, \ul y_1 \tens \dotsb \tens \ul y_n) \tens J(\ul y_1, \ul z_1) \tens \dotsb \tens J(\ul y_n, \ul z_n)}
	\end{displaymath}
	for $x \in A$ and $\ul z \in \fmc_n B$, that satisfy the following associativity and unit axioms. 
	
	The associativity axiom states that, for every $x \in A$ and $\ull z \in \fmc^2 B$, the composite
	\begin{align*}
		J(x&{}, \ull z_{11} \tens \dotsb \tens \ull z_{nm_n}) \\
		&\xrar{\bar J} \int^{\ul y} A(x, \ul y_1 \tens \dotsb \tens \ul y_N) \tens J(\ul y_1, \ull z_{11}) \tens \dotsb \tens J(\ul y_N, \ull z_{nm_n}) \\
		&\iso \int^{\ull y} A(x, \ull y_{11} \tens \dotsb \tens \ull y_{nm_n}) \tens J(\ull y_{11}, \ull z_{11}) \tens \dotsb \tens J(\ull y_{nm_n}, \ull z_{nm_n}) \\
		&\xrar{\int A(\id, \mathfrak a_A) \tens \id} \int^{\ull y} A\bigl(x, (\ull y_{11} \tens \dotsb \tens \ull y_{1m_1}) \!\begin{lgathered}[t]{}\tens \dotsb \tens (\ull y_{n1} \tens \dotsb \tens \ull y_{nm_n})\bigr) \\
		{}\tens J(\ull y_{11}, \ull z_{11}) \tens \dotsb \tens J(\ull y_{nm_n}, \ull z_{nm_n}), \end{lgathered}
	\end{align*}
	where the coends range over $\ul y \in \fmc A$ or $\ull y \in \fmc^2 A$ and where the isomorphism is given by \eqref{morphism:inverse of right multiplication cell for fmc}, equals
	\begin{align*}
		&J(x, \ull z_{11} \tens \dotsb \tens \ull z_{nm_n}) \xrar{J(\id, \mathfrak a_B)} J\bigpars{x, (\ull z_{11} \tens \dotsb \tens \ull z_{1m_1}) \tens \dotsb \tens (\ull z_{n1} \tens \dotsb \tens \ull z_{nm_n})} \\
		&\xrar{\bar J} \int^{\ul v} A(x, \ul v_1 \tens \dotsb \tens \ul v_n) \tens \Tens_{i=1}^n J(\ul v_i, \ull z_{i1} \tens \dotsb \tens \ull z_{im_i}) \\
		&\xrar{\int\id \tens \bar J^{\tens n}} \mspace{-6mu}\int^{\ul v, \ull w} \mspace{-3mu} A(x, \ul v_1 \tens \dotsb \tens \ul v_n) \tens \Tens_{i=1}^n \bigpars{A(\ul v_i, \ull w_{i1} \tens \dotsb \tens \ull w_{im_i}) \tens \Tens_{j=1}^{m_i} J(\ull w_{ij}, \ull z_{ij})} \\
		&\to \int^{\ull y} A\bigl(x, (\ull y_{11} \tens \dotsb \tens \ull y_{1m_1}) \!\begin{lgathered}[t]{}\tens \dotsb \tens (\ull y_{n1} \tens \dotsb \tens \ull y_{nm_n})\bigr) \\
		{}\tens J(\ull y_{11}, \ull z_{11}) \tens \dotsb \tens J(\ull y_{nm_n}, \ull z_{nm_n}). \end{lgathered}
	\end{align*}
	Here the last map is induced by maps that, after permuting the factors, use the mo\-noi\-dal structure of $A$ to map the tensor product of the factors \mbox{$A(\ul v_i, \ull w_{i1} \tens \dotsb \tens \ull w_{im_i})$} into \mbox{$A\bigpars{\ul v_1 \tens \dotsb \tens \ul v_n, \Tens_{i=1}^n (\ull w_{i1} \tens \dotsb \tens \ull w_{im_i})}$}, followed by composing the latter with \mbox{$A(x, \ul v_1 \tens \dotsb \tens \ul v_n)$}. The unit axiom (remember that $A$ and $B$ are assumed to be normal) means that the composites
	\begin{displaymath}
		J(x, z) \xrar{\bar J} \int^{y \in A} A(x, y) \tens J(y, z) \iso J(x, z),
	\end{displaymath}
	where the isomorphism is the Yoneda isomorphism, equal the identity.
\end{example}

	Right colax structures on double profunctors are similar to those on monoidal profunctors, as is shown in the following example. However, as a consequence of the fact that the `free strict double category'-monad $D$ is not a pseudofunctor, the associativity axiom is slightly harder to state.
\begin{example} \label{example:right colax double profunctors}
	Let $A$ and $B$ be normal colax double categories. A \emph{right colax double profunctor} $\hmap JAB$ is a $\GG_1$-indexed profunctor (\exref{example:GG1-indexed bimodules}), that comes equipped with \emph{right splittings} $\bar J(w) = (u, \ul w)$: for every cell $\cell wh{\ul l_1 \hc \dotsb \hc \ul l_n}$ in $J$, where $\cell uh{\ul k_1 \hc \dotsb \hc \ul k_n}$ is a cell in $A$ and $\cell{\ul w}{\ul k}{\ul l}$ is a sequence of cells in $D_n J$, as follows, that satisfies the axioms below.
	\begin{displaymath}
		\begin{tikzpicture}[textbaseline]
			\matrix(m)[math175em]
			{ a & & & & b \\
				c_0 & c_1 & \dotsb & c_{n'} & c_n \\ };
			\path[map]	(m-1-1) edge[barred] node[above] {$h$} (m-1-5)
													edge node[left] {$p$} (m-2-1)
									(m-1-5) edge node[right] {$q$} (m-2-5)
									(m-2-1) edge[barred] node[below] {$\ul l_1$} (m-2-2)
									(m-2-2) edge[barred] (m-2-3)
									(m-2-3) edge[barred] (m-2-4)
									(m-2-4) edge[barred] node[below] {$\ul l_n$} (m-2-5);
			\path[transform canvas={shift={($(m-1-1)!0.5!(m-2-1)$)}}]	(m-1-5) edge[cell] node[right] {$w$} (m-2-5);
		\end{tikzpicture}
		\mapsto \begin{tikzpicture}[textbaseline]
			\matrix(m)[math175em]
			{ a & & & & b \\
				a_0 & a_1 & \dotsb & a_{n'} & a_n \\[-0.5em]
				a_0 & a_1 & & a_{n'} & a_n \\
				c_0 & c_1 & & c_{n'} & c_n \\ };
			\path[map]	(m-1-1) edge[barred] node[above] {$h$} (m-1-5)
													edge node[left] {$f$} (m-2-1)
									(m-1-5) edge node[right] {$g$} (m-2-5)
									(m-2-1) edge[barred] node[below] {$\ul k_1$} (m-2-2)
									(m-3-1) edge[barred] node[above] {$\ul k_1$} (m-3-2)
													edge node[left] {$p_0$} (m-4-1)
									(m-2-2) edge[barred] (m-2-3)
									(m-3-2)	edge node[right] {$p_1$} (m-4-2)
									(m-2-3) edge[barred] (m-2-4)
									(m-2-4) edge[barred] node[below] {$\ul k_n$} (m-2-5)
									(m-3-4) edge[barred] node[above] {$\ul k_n$} (m-3-5)
													edge node[left] {$p_{n'}$} (m-4-4)
									(m-3-5) edge node[right] {$p_n$} (m-4-5)
									(m-4-1) edge[barred] node[below] {$\ul l_1$} (m-4-2)
									(m-4-4) edge[barred] node[below] {$\ul l_n$} (m-4-5);
			\path[transform canvas={shift=($(m-2-1)!0.5!(m-3-1)$)}]	(m-1-5) edge[cell] node[right] {$u$} (m-2-5);
			\path[transform canvas={shift=($(m-2-2)!0.5!(m-3-3)$)}] (m-3-2) edge[cell] node[right] {$\ul w_1$} (m-4-2)
									(m-3-5) edge[cell] node[right] {$\ul w_n$} (m-4-5);
			\draw	($(m-3-2)!0.5!(m-4-4)$) node {$\dotsb$};
		\end{tikzpicture}
	\end{displaymath}
	The axioms are:
	\begin{itemize}[label=-]
		\item	for each splitting $\bar J(w) = (u, \ul w)$ as above, $p_0 \of f = p$ and $p_n \of g = q$;
		\item for each $\bar J(w) = (u, \ul w)$ above, $\cell xjh$ in $A$ and $\cell{\ul y}{\ul l}{\ul m}$ in $D_nB$, the images of $\bar J\bigpars{(\ul y_1 \hc \dotsb \hc \ul y_n) \of w \of x}$ and $(u \of x, \ul y \of \ul w)$ coincide in the coend \mbox{$\int^{\ul k} A(j, \ul k_1 \hc \dotsb \hc \ul k_n) \times DJ(\ul k, \ul m)$};
		\item for each cell $\cell wh{\ull l_{11} \hc \dotsb \hc \ull l_{1m_1} \hc \dotsb \hc \ull l_{n1} \hc \dotsb \hc \ull l_{nm_n}}$ in $J$, the images of the double splittings\footnote{Notice that while the image of the first double splitting here lies in the image of $A(\id, \hc) \hc_{DA} DA(\id, \hc) \hc_{D^2A} D^2 A$, this is generally not true for the image of the second splitting. Indeed successive cells $y_s$ and $y_{s+1}$ need not have common vertical maps, so that they may not form a sequence $\ul y \in DA$. In fact this failure is the reason for the monad $D$ not being pseudo.}
		\begin{flalign*}
			&& &\Bigpars{h \xRar u (\mathrm\hc \of \mu_A)(\ull j) \xRar{\mathfrak a_A} (\mathrm\hc \of D\mathrm\hc)(\ull j),\, \bigpars{(D\mathrm\hc(\ull j))_s \xRar\id \mathrm\hc(\ull j_s),\, \ull j_s \xRar{\ull w_s} \ull l_s}_{1 \leq s \leq n}}& \\
			\text{and} && &\Bigpars{h \xRar x \mathrm\hc(\ul k), \, \bigpars{\ul k_s \xRar{ y_s} \ul m_s,\, \ul m_s \xRar{\ul v'_s} \ull l_s}_{1 \leq s \leq n}},&
		\end{flalign*}
		that are given as follows, coincide in $A(\id, \hc) \times_{DA} D\bigpars{A(\id, \hc) \times_{DA} DA}$. The first is obtained by splitting $w$ into $\bar J(w) = (u, \ul w)$, where $\cell uh{\ul j_1 \hc \dotsb \hc \ul j_N}$ with $N = m_1 + \dotsb + m_n$, and $\ull w$ is the unique decomposition of $\ul w$ such that $\mu_J(\ull w) = \ul w$, while the second is obtained by first taking $\bar J(\mathfrak a_A \of w) = (x, \ul v)$ and then splitting each $\cell{\ul v_s}{\ul k_s}{\ull l_{s1} \hc \dotsb \hc \ull l_{sm_s}}$ into $\bar J(\ul v_s) = (y_s, \ul v'_s)$;
		\item for each $\cell wh{\ul l_1}$ the splitting $\bar J(w) = (x, \ul w)$ satisfies $\ul w_1 \of x = w$.
	\end{itemize}
\end{example}

	The following definition introduces $T$-cells between right colax promorphisms. Afterwards we will prove that, together with colax algebras and colax morphisms, these form a pseudo double category. In the definition we will use the following notation: for any colax $T$-morphism $\map fAC$, with vertical structure cell \mbox{$\cell{\bar f}{f \of a}{c \of Tf}$}, we will denote the composite of the cells below by $\cell{\rho_1\bar f}{A(\id, a)}{C(\id, c)}$. The subscript $1$ in the notation represents the fact that in each vertical edge of $\bar f$ we only `slid one map around the corner', as opposed to $\rho\bar f$, where both maps are moved. Notice that $f \mapsto \rho_1\bar f$ is functorial, in the sense that $\rho_1(\ol{g \of f}) = (\rho_1\bar g) \of (\rho_1\bar f)$ for any other colax morphism $\map gCE$.
	\begin{displaymath}
		\begin{tikzpicture}
			\matrix(m)[math175em]
			{	A \nc TA \nc TA \nc TA \\
				A \nc A \nc TC \nc TC \\
				C \nc C \nc C \nc TC \\[-3em]
				\phantom{TC} \nc\nc\nc \\ };
			\path[map]	(m-1-1) edge[barred] node[above] {$A(\id, a)$} (m-1-2)
									(m-1-2) edge node[right] {$a$} (m-2-2)
									(m-1-3) edge node[left] {$Tf$} (m-2-3)
									(m-1-4) edge node[right] {$Tf$} (m-2-4)
									(m-2-1) edge node[left] {$f$} (m-3-1)
									(m-2-2) edge node[right] {$f$} (m-3-2)
									(m-2-3) edge node[left] {$c$} (m-3-3)
									(m-3-3) edge[barred] node[below] {$C(\id, c)$} (m-3-4);
			\path				(m-1-1) edge[eq] (m-2-1)
									(m-1-2) edge[eq] (m-1-3)
									(m-1-3) edge[eq] (m-1-4)
									(m-2-1) edge[eq] (m-2-2)
									(m-2-3) edge[eq] (m-2-4)
									(m-2-4) edge[eq] (m-3-4)
									(m-3-1) edge[eq] (m-3-2)
									(m-3-2) edge[eq] (m-3-3);
			\path[transform canvas={shift=(m-2-3)}]	(m-1-1) edge[cell] node[right] {$\eps_a$} (m-2-1)
									(m-2-3) edge[cell] node[right] {$\eta_c$} (m-3-3)
									(m-2-1) edge[cell] node[right] {$U_f$} (m-3-1);
			\path[transform canvas={shift=(m-2-3), xshift=-0.7em}]
									(m-1-3) edge[cell] node[right] {$U_{Tf}$} (m-2-3);
			\path[transform canvas={shift=($(m-2-3)!0.5!(m-3-3)$)}] (m-1-2) edge[cell] node[right] {$\bar f$} (m-2-2);
		\end{tikzpicture}
	\end{displaymath}
\begin{example}
	For a colax monoidal $\V$-functor $\map fAC$ the $\V$-natural transformation $\rho_1 f_\tens$ is given by the composites
	\begin{displaymath}
		A(x, \ul y_1 \tens \dotsb \tens \ul y_n) \xrar f C\bigpars{fx, f(\ul y_1 \tens \dotsb \tens \ul y_n)} \xrar{C(\id, f_\tens)} C(fx, f\ul y_1 \tens \dotsb \tens f\ul y_n).
	\end{displaymath}
\end{example}
\begin{definition} \label{definition:cells of right colax promorphisms}
	Let $T$ be a right suitable normal monad on an equipment $\K$. Consider a cell
	\begin{displaymath}
		\begin{tikzpicture}
			\matrix(m)[math175em]{A & B \\ C & D \\};
			\path[map]  (m-1-1) edge[barred] node[above] {$J$} (m-1-2)
													edge node[left] {$f$} (m-2-1)
									(m-1-2) edge node[right] {$g$} (m-2-2)
									(m-2-1) edge[barred] node[below] {$K$} (m-2-2);
			\path[transform canvas={shift={($(m-1-2)!(0,0)!(m-2-2)$)}}] (m-1-1) edge[cell] node[right] {$\phi$} (m-2-1);			
		\end{tikzpicture}
	\end{displaymath}
	where $f$ and $g$ are colax $T$-morphisms and where $J$ and $K$ are right colax $T$\ndash pro\-mor\-phisms. We will call $\phi$ a \emph{$T$-cell} whenever the following identity is satisfied, where $\rho_1 \bar f$ is defined above.
	\begin{displaymath}
		\begin{tikzpicture}[textbaseline]
			\matrix(m)[math175em]
			{ A & B & TB \\
				A & TA & TB \\
				C & TC & TD \\[-3em]
				\phantom{TD} & \phantom{TD} & \phantom{TD} \\ };
			\path[map]	(m-1-1) edge[barred] node[above] {$J$} (m-1-2)
									(m-1-2) edge[barred] node[above] {$B(\id, b)$} (m-1-3)
									(m-2-1) edge[barred] node[above] {$A(\id, a)$} (m-2-2)
													edge node[left] {$f$} (m-3-1)
									(m-2-2) edge[barred] node[above] {$TJ$} (m-2-3)
													edge node[right] {$Tf$} (m-3-2)
									(m-2-3) edge node[right] {$Tg$} (m-3-3)
									(m-3-1) edge[barred] node[below] {$C(\id, c)$} (m-3-2)
									(m-3-2) edge[barred] node[below] {$TK$} (m-3-3);
			\path				(m-1-1) edge[eq] (m-2-1)
									(m-1-2) edge[cell] node[right] {$\bar J$} (m-2-2)
									(m-1-3) edge[eq] (m-2-3);
			\path[transform canvas={shift={($(m-2-1)!0.5!(m-2-2)$)}, xshift=-0.675em}]	(m-2-2) edge[cell] node[right] {$\rho_1\bar f$} (m-3-2);
			\path[transform canvas={shift={($(m-2-1)!0.5!(m-2-2)$)}}] (m-2-3) edge[cell] node[right] {$T\phi$} (m-3-3);
		\end{tikzpicture}
		= \begin{tikzpicture}[textbaseline]
			\matrix(m)[math175em]
			{ A & B & TB \\
				C & D & TD \\
				C & TC & TD \\[-3em]
				\phantom{TD} & \phantom{TD} & \phantom{TD} \\ };
			\path[map]	(m-1-1) edge[barred] node[above] {$J$} (m-1-2)
													edge node[left] {$f$} (m-2-1)
									(m-1-2) edge[barred] node[above] {$B(\id, b)$} (m-1-3)
													edge node[right] {$g$} (m-2-2)
									(m-1-3) edge node[right] {$Tg$} (m-2-3)
									(m-2-1) edge[barred] node[below] {$K$} (m-2-2)
									(m-2-2) edge[barred] node[below] {$D(\id, d)$} (m-2-3)
									(m-3-1) edge[barred] node[below] {$C(\id, c)$} (m-3-2)
									(m-3-2) edge[barred] node[below] {$TK$} (m-3-3);
			\path				(m-2-1) edge[eq] (m-3-1)
									(m-2-2) edge[cell] node[right] {$\bar K$} (m-3-2)
									(m-2-3) edge[eq] (m-3-3);
			\path[transform canvas={shift={($(m-2-1)!0.5!(m-2-2)$)}, xshift=-0.675em}]	(m-1-3) edge[cell] node[right] {$\rho_1\bar g$} (m-2-3);
			\path[transform canvas={shift={($(m-2-1)!0.5!(m-2-2)$)}}] (m-1-2) edge[cell] node[right] {$\phi$} (m-2-2);
		\end{tikzpicture}
	\end{displaymath}
\end{definition}
	Notice that if $\phi$ is a vertical cell $\cell\phi fg$ then the $T$-cell axiom for $\phi$ above is equivalent to the vertical $T$-cell axiom given in \defref{definition:colax cell}: the first equals the horizontal composite of the latter with conjoint cells $\eps_a$ on the left and $\eta_c$ on the right. Thus the definition above generalises that of vertical cells between colax morphisms.

	As an example we describe cells between right colax monoidal functors; cells between right colax double functors are similar.
\begin{example} \label{example:monoidal transformations}
	When $T$ is the `free strict monoidal $\V$-category'-monad $\fmc$ we will call a $\fmc$-cell $\nat\phi J{K(f, g)}$ above, where $J$ and $K$ are right colax monoidal $\V$\ndash profunctors, a \emph{monoidal transformation}. The $\fmc$-cell axiom above means that the following diagram commutes, for objects $x \in A$ and $\ul z \in \fmc_nB$, where the coends range over $\ul y \in \fmc_nA$ and $\ul w \in \fmc_nC$.
	\begin{displaymath}
		\begin{tikzpicture}
			\matrix(m)[math2em, column sep=2.5em, row sep=3.25em, inner sep=6pt]
			{ J(x, \ul z_1 \tens \dotsb \tens \ul z_n) & \displaystyle\int^{\ul y} A(x, \ul y_1 \tens \dotsb \tens \ul y_n) \tens \Tens_{i=1}^n J(\ul y_i, \ul z_i) \\
				K\bigpars{fx, g(\ul z_1 \tens \dotsb \tens \ul z_n)} & \displaystyle\int^{\ul y} C(fx, f\ul y_1 \tens \dotsb \tens f\ul y_n) \tens \Tens_{i=1}^n K(f\ul y_i, g\ul z_i) \\
				K(fx, g\ul z_1 \tens \dotsb \tens g\ul z_n) & \displaystyle\int^{\ul w} C(fx, \ul w_1 \tens \dotsb \tens \ul w_n) \tens \Tens_{i=1}^n K(\ul w_i, g\ul z_i) \\ };
			\path[map]	(m-1-1) edge node[above] {$\bar J$} (m-1-2)
													edge node[left] {$\phi$} (m-2-1)
									(m-1-2) edge node[right] {$\int \rho_1 f_\tens \tens \phi^{\tens n}$} (m-2-2)
									(m-2-1) edge node[left] {$K(\id, g_\tens)$} (m-3-1)
									(m-2-2) edge (m-3-2)
									(m-3-1) edge node[below] {$\bar K$} (m-3-2);
		\end{tikzpicture}
	\end{displaymath}
\end{example}
	Recall that every normal monad $T$ on an equipment $\K$ induces a $2$-monad $V(T)$ on the vertical $2$-category $V(\K)$, by \propref{normal functors induce vertical functors}.
\begin{proposition} \label{right colax promorphisms form a pseudo double category}
	Let $T$ be a right suitable normal monad on an equipment $\K$. Colax $T$\ndash algebras, colax $T$-morphisms, right colax $T$-promorphisms and $T$-cells form a pseudo double category $\rcProm T$, in which the horizontal composite $J \hc H$ of two right colax promorphisms $\hmap JAB$ and $\hmap HBC$ is equipped with the structure cell
	\begin{multline*}
		\ol{J \hc H} = \bigbrks{J \hc H \hc C(\id, c) \xRar{\id \hc \bar H} J \hc B(\id, b) \hc TH \\
		\xRar{\bar J \hc \id} A(\id, a) \hc TJ \hc TH \xRar{\id \hc T_\hc} A(\id, a) \hc T(J \hc H)}.
	\end{multline*}
	Moreover the vertical $2$-category $V(\rcProm T)$ is the $2$-category $\cAlg{V(T)}$ of colax $V(T)$\ndash algebras, colax $V(T)$-morphisms and vertical $V(T)$-cells.
\end{proposition}
\begin{proof}
	We know that colax $T$-algebras (remember these are in fact colax $V(T)$\ndash al\-ge\-bras, see \defref{definition:colax V(T)-algebras and morphisms}) and colax morphisms form a category $(\rcProm T)_0$. On the other hand it follows readily from the functoriality of $f \mapsto \rho_1 \bar f$ that $T$-cells can be vertically composed by vertically composing their underlying cells in $\K$, so that they too form a category $(\rcProm T)_1$, with right colax promorphisms as objects. The last assertion follows from the remark following \defref{definition:cells of right colax promorphisms} above, once we have shown that $(\rcProm T)_0$ and $(\rcProm T)_1$ combine to a pseudo double category $\rcProm T$.
	
	We claim that the functor $\map{\hc}{\K_1 \times_{\K_0} \K_1}{\K_1}$, that equips $\K$ with horizontal composites, lifts to a functor
	\begin{displaymath}
		 \map{\hc}{(\rcProm T)_1 \times_{(\rcProm T)_0} (\rcProm T)_1}{(\rcProm T)_1},
	\end{displaymath}
	by equipping each horizontal composite $J \hc H$ with the structure cell $\ol{J \hc H}$ above. First we have to check that $\ol{J \hc H}$ satisfies the associativity and unit axioms. That it satisfies the associativity axiom is shown in \figref{figure:associativity axiom for the horziontal composite of right colax promorphisms}; proving the unit axiom is similar. Thus the structure cell $\ol{J \hc H}$ above makes $J \hc H$ into a well defined right colax promorphism.
	\begin{figure}
	\begin{flalign*}
		\begin{tikzpicture}[textbaseline, x=0.75cm, y=0.6cm, font=\scriptsize]
			\foreach \i in {0,...,4}
				\draw (1, \i) -- (4, \i);
			\draw (1,0) -- (1,4)
						(4,0) -- (4,4)
						(1,4) -- (1,5) -- (3,5) -- (3,4)
						(1,5) -- (0,5) -- (0,6) -- (2,6) -- (2,5)
						(1,6) -- (1,7) -- (3,7) -- (3,6) -- (2,6)
						(2,7) -- (2,8) -- (4,8) -- (4,7) -- (3,7);
			\draw[shift={(0.5,0.5)}]
						(2,0) node {$T(\id \hc T_\hc)$}
						(2,1) node {$T(\bar J \hc \id)$}
						(2,2) node {$T(\id \hc \bar H)$}
						(2,3) node {$T_\hc$};
			\draw[shift={(0,0.5)}]
						(2,4) node {$T_\hc$}
						(1,5) node {$\bar J$}
						(2,6) node {$\bar H$}
						(3,7) node {$\rho\gamma$};
		\end{tikzpicture}
		&\quad\overset{(1)}=\quad \begin{tikzpicture}[textbaseline, x=0.75cm, y=0.6cm, font=\scriptsize]
			\foreach \i in {0,...,3}
				\draw (1,\i) -- (4,\i);
			\draw (1,0) -- (1,3)
						(4,0) -- (4,3)
						(2,3) -- (2,4) -- (4,4) -- (4,3)
						(4,4) -- (4,5) -- (0,5) -- (0,4) -- (2,4) -- (2,5)
						(1,5) -- (1,6) -- (3,6) -- (3,5)
						(2,6) -- (2,7) -- (4,7) -- (4,6) -- (3,6);
			\draw[shift={(0.5,0.5)}]
						(2,0) node {$T(\id \hc T_\hc)$}
						(2,1) node {$T(\bar J \hc \id)$}
						(2,2) node {$T_\hc$};
			\draw[shift={(0,0.5)}]
						(3,3) node {$T\bar H$}
						(1,4) node {$\bar J$}
						(3,4) node {$T_\hc$}
						(2,5) node {$\bar H$}
						(3,6) node {$\rho\gamma$};
		\end{tikzpicture}
		\quad\overset{(2)}=\quad \begin{tikzpicture}[textbaseline, x=0.75cm, y=0.6cm, font=\scriptsize]
			\foreach \i in {0,...,3}
				\draw (1,\i) -- (4,\i);
			\draw (1,0) -- (1,3)
						(4,0) -- (4,3)
						(1,3) -- (0,3) -- (0,4) -- (2,4) -- (2,3)
						(2,4) -- (4,4) -- (4,3)
						(1,4) -- (1,5) -- (3,5) -- (3,4)
						(2,5) -- (2,6) -- (4,6) -- (4,5) -- (3,5)
						(2,6) -- (1,6) -- (1,7) -- (3,7) -- (3,6);
			\draw[shift={(0.5,0.5)}]
						(2,0) node {$T(\id \hc T_\hc)$}
						(2,1) node {$T(\bar J \hc \id)$}
						(2,2) node {$T_\hc$};
			\draw[shift={(0,0.5)}]
						(1,3) node {$\bar J$}
						(3,3) node {$T_\hc$}
						(2,4) node {$\rho\beta$}
						(3,5) node {$\inv{(\rho\mu_H)}$}
						(2,6) node {$\bar H$};
		\end{tikzpicture} \\
		&\quad\overset{(1)}=\quad \begin{tikzpicture}[textbaseline, x=0.75cm, y=0.6cm, font=\scriptsize]
			\foreach \i in {0,...,2}
				\draw (1,\i) -- (4,\i);
			\draw (1,0) -- (1,2)
						(4,0) -- (4,2)
						(1,2) -- (1,3) -- (3,3) -- (3,2)
						(1,3) -- (1,4) -- (3,4) -- (3,3)
						(1,4) -- (0,4) -- (0,5) -- (2,5) -- (2,4)
						(1,5) -- (1,6) -- (3,6) -- (3,5) -- (2,5)
						(2,6) -- (2,7) -- (4,7) -- (4,6) -- (3,6)
						(2,7) -- (1,7) -- (1,8) -- (3,8) -- (3,7);
			\draw[shift={(0.5,0.5)}]
						(2,0) node {$T(\id \hc T_\hc)$}
						(2,1) node {$T_\hc$};
			\draw[shift={(0,0.5)}]
						(2,2) node {$T\bar J$}
						(2,3) node {$T_\hc$}
						(1,4) node {$\bar J$}
						(2,5) node {$\rho\beta$}
						(3,6) node {$\inv{(\rho\mu_H)}$}
						(2,7) node {$\bar H$};
		\end{tikzpicture}
		\quad\overset{(3)}=\quad \begin{tikzpicture}[textbaseline, x=0.75cm, y=0.6cm, font=\scriptsize]
			\foreach \i in {0,...,2}
				\draw (1,\i) -- (4,\i);
			\draw (1,0) -- (1,2)
						(4,0) -- (4,2)
						(1,2) -- (1,3) -- (3,3) -- (3,2)
						(1,3) -- (0,3) -- (0,4) -- (2,4) -- (2,3)
						(1,4) -- (1,5) -- (3,5) -- (3,4) -- (2,4)
						(1,5) -- (0,5) -- (0,6) -- (2,6) -- (2,5)
						(2,6) -- (4,6) -- (4,5) -- (3,5)
						(1,6) -- (1,7) -- (3,7) -- (3,6);
			\draw[shift={(0.5,0.5)}]
						(2,0) node {$T(\id \hc T_\hc)$}
						(2,1) node {$T_\hc$};
			\draw[shift={(0,0.5)}]
						(2,2) node {$T_\hc$}
						(1,3) node {$\rho\alpha$}
						(2,4) node {$\inv{(\rho\mu_J)}$}
						(1,5) node {$\bar J$}
						(3,5) node {$\inv{(\rho\mu_H)}$}
						(2,6) node {$\bar H$};
		\end{tikzpicture} \\
		&\quad\overset{(1)}=\quad \begin{tikzpicture}[textbaseline, x=0.75cm, y=0.6cm, font=\scriptsize]
			\draw (1,1) -- (1,0) -- (4,0) -- (4,2) -- (0,2) -- (0,1) -- (4,1)
						(2,1) -- (2,2)
						(1,2) -- (1,3) -- (3,3) -- (3,2)
						(1,3) -- (0,3) -- (0,4) -- (2,4) -- (2,3)
						(2,4) -- (4,4) -- (4,3) -- (3,3)
						(1,4) -- (1,5) -- (3,5) -- (3,4);
			\draw[shift={(0.5,0.5)}]
						(2,0) node {$T_\hc$};
			\draw[shift={(0,0.5)}]
						(1,1) node {$\rho\alpha$}
						(3,1) node {$T^2_\hc$}
						(2,2) node {$\inv{(\rho\mu_J)}$}
						(1,3) node {$\bar J$}
						(3,3) node {$\inv{(\rho\mu_H)}$}
						(2,4) node {$\bar H$};
		\end{tikzpicture}
		\quad\overset{(4)}=\quad \begin{tikzpicture}[textbaseline, x=0.75cm, y=0.6cm, font=\scriptsize]
			\draw (1,0) -- (1,1) -- (4,1) -- (4,0) -- (1,0)
						(1,1) -- (0,1) -- (0,2) -- (2,2) -- (2,1)
						(1,2) -- (1,3) -- (4,3) -- (4,2) -- (2,2)
						(1,3) -- (1,4) -- (3,4) -- (3,3)
						(1,4) -- (0,4) -- (0,5) -- (2,5) -- (2,4)
						(1,5) -- (1,6) -- (3,6) -- (3,5) -- (2,5); 
			\draw[shift={(0.5,0.5)}]
						(2,0) node {$T_\hc$}
						(2,2) node {$\inv{(\rho\mu_{J \hc H})}$};
			\draw[shift={(0,0.5)}]
						(1,1) node {$\rho\alpha$}
						(2,3) node {$T_\hc$}
						(1,4) node {$\bar J$}
						(2,5) node {$\bar H$};
		\end{tikzpicture}
	\end{flalign*}
	\captionsetup{singlelinecheck=off}
	\caption[.]{A schematic proof of the associativity axiom for the composite $J \hc H$ of right colax promorphisms $\hmap JAB$ and $\hmap HBC$, with structure cell as given by \propref{right colax promorphisms form a pseudo double category}. Each of the diagrams above represents a vertical composition $J \hc H \hc C(\id, c) \hc TC(\id, \mu_C) \Rar A(\id, a) \hc T\bigpars{A(\id, a) \hc T(J \hc H)}$ of horizontal cells: to save space only non-identity cells are denoted, while objects and promorphisms are left out. For example the first two blocks in the upper left diagram represent the cells (with $\gamma$ the associator of $C$)
	\begin{align*}
		J \hc H  \hc C(\id, c) \hc TC(\id, \mu_C) {}&\xRar{\id \hc \id \hc \rho\gamma} J \hc H \hc C(\id, c) \hc TC(\id, c) \\
		&\xRar{\id \hc \bar H \hc \id} J \hc B(\id, b) \hc TH \hc TC(\id, c).
	\end{align*}
	In this way the top left diagram represents the right-hand side of the associativity axiom for $J \hc H$, while the bottom right diagram represents its left-hand side. The equalities follow from (1) the associativity and naturality of $T_\hc$; (2) the associativity axiom for $H$; (3) the associativity axiom for $J$; (4) the composition axiom for $\mu$ (see \defref{definition:transformation}).}
	\label{figure:associativity axiom for the horziontal composite of right colax promorphisms}
	\end{figure}
	
	To show that the horizontal composite of the cells underlying two $T$-cells
	\begin{displaymath}
    \begin{tikzpicture}
      \matrix(m)[math175em]{A & B & C \\ D & E & F \\};
      \path[map]  (m-1-1) edge[barred] node[above] {$J$} (m-1-2)
                          edge node[left] {$f$} (m-2-1)
                  (m-1-2) edge[barred] node[above] {$H$} (m-1-3)
                          edge node[right] {$g$} (m-2-2)
                  (m-1-3) edge node[right] {$h$} (m-2-3)
                  (m-2-1) edge[barred] node[below] {$K$} (m-2-2)
                  (m-2-2) edge[barred] node[below] {$L$} (m-2-3);
      \path[transform canvas={shift={($(m-1-1)!0.5!(m-2-2)$)}}]
                  (m-1-2) edge[cell] node[right] {$\phi$} (m-2-2)
                  (m-1-3) edge[cell] node[right] {$\psi$} (m-2-3);
    \end{tikzpicture}
  \end{displaymath}
  is again a $T$-cell as well consider the schematic diagrams below. Each diagram represents a composition $J \hc H \hc C(\id ,c) \Rar D(\id, d) \hc T(K \hc L)$ and, in particular, the first and the last represent the left and right-hand side of the $T$-cell axiom for $\phi \hc \psi$. The equalities follow from the naturality of $T_\hc$, the $T$-cell axiom for $\phi$ and the $T$-cell axiom for $\psi$.
  \begin{displaymath}
	  \begin{tikzpicture}[textbaseline, x=0.75cm, y=0.6cm, font=\scriptsize]
			\draw (1,1) -- (1,0) -- (0,0) -- (0,1) -- (3,1) -- (3,0) -- (1,0)
						(1,1) -- (1,2) -- (3,2) -- (3,1)
						(1,2) -- (0,2) -- (0,3) -- (2,3) -- (2,2)
						(1,3) -- (1,4) -- (3,4) -- (3,3) -- (2,3)
						(0.5,0.5) node {$\rho_1 \bar f$};
			\draw[shift={(0,0.5)}]
						(2,0) node {$T(\phi \hc \psi)$}
						(2,1) node {$T_\hc$}
						(1,2) node {$\bar H$}
						(2,3) node {$\bar J$};
	  \end{tikzpicture}
	  \mspace{13mu}=\mspace{13mu} \begin{tikzpicture}[textbaseline, x=0.75cm, y=0.6cm, font=\scriptsize]
			\draw (2,1) -- (1,1) -- (1,0) -- (3,0) -- (3,1)
						(0,2) -- (0,1) -- (1,1) -- (1,2) -- (3,2) -- (3,1) -- (2,1) -- (2,2)
						(1,2) -- (0,2) -- (0,3) -- (2,3) -- (2,2)
						(1,3) -- (1,4) -- (3,4) -- (3,3) -- (2,3)
						(0.5,1.5) node {$\rho_1 \bar f$}
						(1.5,1.5) node {$T\phi$}
						(2.5,1.5) node {$T\psi$};
			\draw[shift={(0,0.5)}]
						(2,0) node {$T_\hc$}
						(1,2) node {$\bar H$}
						(2,3) node {$\bar J$};
	  \end{tikzpicture}
	  \mspace{13mu}=\mspace{13mu} \begin{tikzpicture}[textbaseline, x=0.75cm, y=0.6cm, font=\scriptsize]
			\draw (1,0) -- (3,0) -- (3,1) -- (1,1) -- (1,0)
						(1,1) -- (0,1) -- (0,2) -- (2,2) -- (2,1)
						(0,2) -- (0,3) -- (3,3) -- (3,2) -- (2,2) -- (2,3)
						(3,3) -- (3,4) -- (1,4) -- (1,2)
						(0.5,2.5) node {$\phi$}
						(1.5,2.5) node {$\rho_1\bar g$}
						(2.5,2.5) node {$T\psi$};
			\draw[shift={(0,0.5)}]
						(2,0) node {$T_\hc$}
						(1,1) node {$\bar L$}
						(2,3) node {$\bar J$};
	  \end{tikzpicture}
	  \mspace{13mu}=\mspace{13mu} \begin{tikzpicture}[textbaseline, x=0.75cm, y=0.6cm, font=\scriptsize]
			\draw (1,0) -- (3,0) -- (3,1) -- (1,1) -- (1,0)
						(1,1) -- (0,1) -- (0,2) -- (2,2) -- (2,1)
						(1,2) -- (1,3) -- (3,3) -- (3,2) -- (2,2)
						(1,3) -- (0,3) -- (0,4) -- (3,4) -- (3,3)
						(1,3) -- (1,4)
						(2,3) -- (2,4)
						(0.5,3.5) node {$\phi$}
						(1.5,3.5) node {$\psi$}
						(2.5,3.5) node {$\rho_1\bar h$};
			\draw[shift={(0,0.5)}]
						(2,0) node {$T_\hc$}
						(1,1) node {$\bar L$}
						(2,2) node {$\bar K$};
		\end{tikzpicture}
	\end{displaymath}
	
	Finally notice that, for each colax algebra $A$, the unit cell $U_a$ trivially makes $U_A$ into a lax promorphism, which is right pseudo because $\rho U_a = \id_{A(\id, a)}$. Thus each unit promorphism $U_A$ admits a right colax promorphism structure; to finish the proof we remark that the associativity and unit axiom for $T$ imply that the components of the associator $\mathfrak a$ and unitors $\mathfrak l$ and $\mathfrak r$ of $\K$ are $T$-cells, so that they can be taken as the associator and unitors for $\rcProm T$.
\end{proof}

\begin{example} \label{example:pseudo double category of right colax monoidal profunctors}
	Applying the previous proposition to the `free strict monoidal $\V$\ndash ca\-te\-go\-ry'-monad $\fmc$, on the equipment $\enProf\V$ of $\V$-profunctors, we obtain the pseudo double category $\rcMonProf\V$ of colax monoidal $\V$-categories, colax monoidal $\V$\ndash func\-tors, right colax monoidal $\V$-profunctors (\exref{example:right colax enriched promorphisms}) and monoidal $\V$-natural transformations (\exref{example:monoidal transformations}) between them. Analogously, by applying the proposition to the `free symmetric strict monoidal $\V$-category'-monad $\fsmc$, we obtain the pseudo double category $\rcsMonProf\V$ of symmetric right colax monoidal $\V$-pro\-func\-tors.
	
	Since $\fmc$ is a pseudomonad it is possible to replace the right colax monoidal $\V$\ndash profunctors above by lax monoidal $\V$-profunctors (\exref{example:lax monoidal profunctors}), obtaining a pseudo double category $\lMonProf\V$. However $\rcMonProf\V$ is best suited to constructing weighted colimits, as we shall see in \secref{section:applications of the main theorem}.
\end{example}

\begin{example} \label{example:pseudo double category of right colax double functors}
	Applying the previous proposition to the `free strict double category'\ndash monad $D$, on the equipment $\psProf{\GG_1}$ of $\GG_1$-indexed profunctors, we obtain the pseudo double category $\rcDblProf$ of normal colax double categories, colax functors, right colax double profunctors (\exref{example:right colax double profunctors}) and transformations between them. Unlike the monads $M$ and $S$, the monad $D$ is not pseudo (\propref{free strict double category monad is not pseudo}), and it is not clear how to construct a pseudo double category of double categories and lax double profunctors (\exref{example:lax double profunctors}); see also the discussion in \cite{nLab_on_double_profunctors}.
\end{example}

\begin{example}
	At the end of \secref{section:monad examples} we described how the monad for monoidal globular categories is induced by the `free strict $\omega$-category'-monad on $\Span{\ps\GG}$. Since the latter is right suitable, see \exref{example:other monads}, we obtain a pseudo double category of `right colax monoidal globular profunctors'.
	
	On the other hand, the ultrafilter-monad $\beta$ on $\Mat\2$, that induces the monad for ordered compact Hausdorff spaces, is not right suitable, as the components of its unit are in general not right invertible. However, since $\beta$ is a pseudo monad and since the components of its multiplication $\mu$ are right invertible, it seems likely that it is still possible to arrange ordered compact Hausdorff spaces, their morphisms and right colax promorphisms, into a pseudo double category. Furthermore, it seems likely that a variant of our main theorem (\thmref{right colax forgetful functor lifts all weighted colimits}) then holds, hence giving conditions under which the left Kan extension of a pair of morphisms between ordered compact Hausdorff spaces, as considered in the introduction, is again continuous.
\end{example}

	Knowing that right colax promorphisms form a pseudo double category, we wonder whether it is an equipment. The proposition below shows that $\rcProm T$ has all conjoints, but need not have all companions.
	
	In general, consider a normal functor $\map F\K\L$ between pseudo double categories $\K$ and $\L$, and let $f$ be a vertical morphism in $\K$. We say that $F$ \emph{lifts the companion of $f$} if the existence of a companion for $Ff$ in $\L$ implies the existence of a companion for $f$ in $\K$ that is preserved by $F$. Functors that \emph{lift conjoints} are defined analogously.
\begin{proposition} \label{restrictions in T-promrc}
	Let $T$ be a right suitable normal monad on an equipment $\K$. The forgetful functor $\map {U^T}{\rcProm T}\K$ lifts all conjoints, while it lifts companions of colax $T$-morphisms $\map fAC$ whose corresponding companion $C(f, \id)$ is right pseudo (see the comments preceeding \propref{colax morphisms correspond to lax companions}). In particular $\rcProm T$ has restrictions $K(f, g)$ for such $f$.
\end{proposition}
\begin{proof}
	First consider a colax morphism $\map fAC$. We saw in \propref{colax morphisms correspond to lax companions} that its companion $C(f, \id)$ is a lax promorphism, and we assume that it is right pseudo. This means that the structure cell $\rho\ol{C(f, \id)}$ is invertible and hence its inverse $\inv{\bigpars{\rho\ol{C(f, \id)}}}$ forms a right colax promorphism structure on $C(f, \id)$. Thus, if we prove that the companion cells $\ls f\eps$ and $\ls f\eta$, that define $C(f, \id)$, are $T$-cells with respect this structure, then it follows that $\ls f\eps$ is a companion cell in $\rcProm T$ as well, because $\ls f\eps$ and $\ls f\eta$ will still satisfy the companion identities in $\rcProm T$. But this is easy: by precomposing (resp. postcompsing) with $\rho\ol{C(f, \id)}$, the $T$-cell axioms for $\ls f\eps$ and $\ls f\eta$ are equivalent to $\rho_1 \bar f \hc T\ls f\eps = (\ls f\eps \hc \id_{C(\id, c)}) \of \rho\ol{C(f, \id)}$ and $\ls f\eta \hc \rho_1 \bar f = \rho\ol{C(f, \id)} \of (\id_{A(\id, a)} \hc T\ls f\eta)$, which both follow from the fact that $\rho\ol{C(f, \id)} = \rho\bar f$.
	
	Secondly let $\map gBD$ be a colax morphism with vertical structure cell $\bar g$. We claim that the horizontal cell $\cell{\rho\bar g}{D(\id, g) \hc B(\id, b)}{D(\id, d) \hc TD(\id, g)}$ makes $D(\id, g)$ into a right colax promorphism. Indeed, to see that it satisfies the associativity axiom we apply $\rho$ to the associativity axiom for $g$. This gives the following identity of cells $D(\id, g) \hc B(\id, b) \hc TB(\id, \mu_B) \Rar D(\id, d) \hc TD(\id, \mu_D) \hc T^2D(\id, g)$, where $\beta$ and $\delta$ denote the associators of $B$ and $D$, and where $U$ denotes the unit cell for $Tg \of \mu_B = \mu_D \of T^2 g$.
	\begin{displaymath}
		\begin{tikzpicture}[textbaseline, x=0.75cm, y=0.6cm, font=\scriptsize]
			\draw	(1,1) -- (0,1) -- (0,0) -- (2,0) -- (2,1)
						(1,2) -- (1,1) -- (3,1) -- (3,2) -- (2,2)
						(0,2) -- (2,2) -- (2,3) -- (0,3) -- (0,2);
			\draw[shift={(0,0.5)}]
						(1,0) node {$\rho\delta$}
						(2,1) node {$\rho U$}
						(1,2)	node {$\rho\bar g$};
		\end{tikzpicture}
		\quad = \quad\begin{tikzpicture}[textbaseline, x=0.75cm, y=0.6cm, font=\scriptsize]
			\draw	(1,1) -- (1,0) -- (3,0) -- (3,1) -- (2,1)
						(1,2) -- (0,2) -- (0,1) -- (2,1) -- (2,2)
						(1,3) -- (1,2) -- (3,2) -- (3,3) -- (1,3);
			\draw[shift={(0,0.5)}]
						(2,0) node {$\rho T\bar g$}
						(1,1) node {$\rho\bar g$}
						(2,2)	node {$\rho\beta$};
		\end{tikzpicture}
	\end{displaymath}
	In the composition above $U$ is considered as a vertical cell $Tg \of \mu_B \Rar \mu_D \of T^2 g$, so that $\rho U$ is a horizontal cell $TD(\id, g) \hc TB(\id, \mu_B) \Rar TD(\id, \mu_D) \hc T^2 D(\id, g)$. Of course $U$ is its own inverse, but now considered as a cell $\mu_D \of T^2 g \Rar Tg \of \mu_B$. It follows that applying $\rho$ to the latter gives the inverse of $\rho U$ on one hand, while it gives $\rho \mu_{D(\id, g)}$ on the other, by \propref{transformations and companions}. Therefore $\rho U = \inv{(\rho \mu_{D(\id, g)})}$ in the left-hand side above. Finally our convention implies that $\rho T\bar g = \inv T_\hc \of T\rho\bar g \of T_\hc$, and we conclude that postcomposing both sides with $T_\hc$ gives the associativity axiom for $D(\id, g)$. That the unit axiom holds for $D(\id, g)$ follows in the same way from that of $g$. To show that $\eps_g$ is a conjoint cell in $\rcProm T$, it remains to check that the cells $\eps_g$ and $\eta_g$ form $T$-cells with respect to the algebra structure on $D(\id, g)$. But this follows easily from the fact that we have taken $\ol{D(\id, g)} = \rho\bar g$. To complete the proof notice that the last assertion follows from \propref{cartesian fillers in terms of companions and conjoints}.
\end{proof}

	In closing this section we consider weighted colimits in $\rcProm T$ and, following \cite[Section 4.1]{Grandis-Pare08}, define pointwise weighted colimits in $\rcProm T$ (even though it may not be an equipment). Remember that, in \defref{definition:weighted colimits in double categories}, the definition of weighted colimits in closed equipments was extended to pseudo double categories as follows. Given morphisms $\hmap JAB$ and \mbox{$\map dAM$} in a pseudo double category, the $J$-weighted colimit of $d$ is a vertical morphism $\map {l = \colim_J d}BM$ equipped with a cell $\eta$, the unit, as below, such that any cell $\phi$ factors uniquely through $\eta$ as shown.
\begin{displaymath}
	\begin{tikzpicture}[textbaseline]
		\matrix(m)[math175em]{A & B & C \\ M & \phantom M & M \\};
		\path[map]	(m-1-1) edge[barred] node[above] {$J$} (m-1-2)
												edge node[left] {$d$} (m-2-1)
								(m-1-2) edge[barred] node[above] {$H$} (m-1-3)
								(m-1-3) edge node[right] {$e$} (m-2-3);
		\path				(m-2-1) edge[eq] (m-2-3)
								(m-1-2) edge[cell] node[right] {$\phi$} (m-2-2);
	\end{tikzpicture}
	= \begin{tikzpicture}[textbaseline]
		\matrix(m)[math175em]{A & B & C \\ M & M & M, \\};
		\path[map]	(m-1-1) edge[barred] node[above] {$J$} (m-1-2)
												edge node[left] {$d$} (m-2-1)
								(m-1-2) edge[barred] node[above] {$H$} (m-1-3)
												edge node[right] {$l$} (m-2-2)
								(m-1-3) edge node[right] {$e$} (m-2-3);
		\path				(m-2-1) edge[eq] (m-2-2)
								(m-2-2) edge[eq] (m-2-3);
		\path[transform canvas={shift=($(m-1-2)!0.5!(m-2-3)$)}]	(m-1-1) edge[cell] node[right] {$\eta$} (m-2-1)
								(m-1-2) edge[cell] node[right] {$\phi'$} (m-2-2);
	\end{tikzpicture}
\end{displaymath}
	
	Applied to $\rcProm T$ we obtain the notion of a \emph{weighted colimit} \mbox{$\map{\colim_J d}BM$} of a colax morphism $\map dAM$, weighted by a right colax promorphism $\hmap JAB$. In particular we can use the previous proposition in \defref{definition:weighted left Kan extensions} to obtain, for any colax morphism $\map jAB$ that has a right pseudo companion $B(j, \id)$, the definition of a \emph{weighted left Kan extension} $\map{\lan_j d}BM$, of any colax morphism $\map dAM$ along $j$, as the weighted colimit $\lan_j d = \colim_{B(j, \id)} d$ in $\rcProm T$. We may then apply \propref{colimits weighted by companions} to find that every weighted left Kan extension $\lan_j d$ is in particular a left Kan extension of $d$ along $l$ in the vertical $2$-category $V(\rcProm T) = \cAlg{V(T)}$ of colax $V(T)$-algebras, in the usual sense.	Our main result, given in the next section, states that the forgetful functor $\map{U^T}{\rcProm T}\K$ lifts weighted colimits. This means that if the weighted colimit $l' = \colim_{U^T J} U^T d$ exists in $\K$ then it can be lifted to a colax morphism $l$ (that is $U^T l = l'$) such that $l = \colim_J d$.
	
	We now turn to the notion of pointwise weighted colimits in $\rcProm T$. To generalise \defref{definition:pointwise weighted colimits} we first need double comma objects in $\rcProm T$. The following proposition shows that the forgetful functor $\map{U^T}{\rcProm T}\K$ lifts double comma objects $J \slash f$ of right pseudopromorphisms $\hmap JAB$ and pseudomorphisms $\map fCB$. This generalises \cite[Proposition 4.6]{Lack05} which shows that, for any $2$-monad $T$ on a $2$-category $\C$, the forgetful $2$-functor $\map{U^T}{\cAlg T}\C$ lifts comma objects $f \slash g$ of colax $T$-morphisms $f$ and $g$ between strict $T$-algebras, when $g$ is a strict $T$-morphism. There it follows as an easy consequence of a more general theorem, and Lack remarks that giving a direct proof would strengthen his result to include the case of $g$ being only a pseudomorphism. We will give a direct proof of the proposition that is somewhat long (but straightforward) and therefore postponed to \appref{appendix:double comma object proofs}.
	
	Given a functor $\map F\K\L$ between pseudo double categories and morphisms $\hmap JAB$ and $\map fCB$ in $\K$, we say that $F$ \emph{lifts} the double comma object of  $J$ and $f$ when, if a cell $\pi$ defines the double comma object of $FJ$ and $Ff$ in $\L$, then there exists a cell $\zeta$ in $\K$ such that $F\zeta = \pi$, that defines the double comma object of $J$ and $f$.
\begin{proposition} \label{lifting double comma objects}
	Consider a right suitable normal monad $T$ on an equipment $\K$. The forgetful functor $\map{U^T}{\rcProm T}\K$ lifts double comma objects $J \slash f$ where $\hmap JAB$ is a right pseudopromorphism and $\map fCB$ is a pseudomorphism. The projections of $J \slash f$ are strict $T$-morphisms, and $J \slash f$ is a pseudoalgebra whenever both $A$ and $C$ are.
\end{proposition}

	As an example we consider the double comma object $J \slash f$ of a right pseudo double profunctor $\hmap JAB$ and a pseudo double functor $\map fCB$. In case $J = B(\id, g)$ is the companion of a colax double profunctor $\map gAB$ this recovers the unbiased variants of the `comma double categories' of \cite[Section 1.3]{Grandis-Pare07}.
\begin{example}
	Consider a right pseudo double profunctor $\hmap JAB$ and a pseudo double functor $\map fCB$ between normal colax double categories $A$, $B$ and $C$. Forgetting the horizontal composites for a moment, remember that the comma object $J \slash f$ in $\K = \psCat{\GG_1}$ is constructed indexwise. Thus $J \slash f$ is the $\GG_1$-indexed category that has objects $(a, \map pa{fc}, c)$, where $a \in A$ and $c \in C$ are objects and $p$ is a vertical morphism in $J$. A vertical morphism $(a, \map pa{fc}, c) \to (a', \map{p'}{a'}{fc'}, c')$ is given by a pair $(\map sa{a'}, \map tc{c'})$ of vertical morphisms $s$ in $A$ and $t$ in $C$ making the diagram on the left below commute, while a horizontal morphism $(a_0, \map{p_0}{a_0}{fc_0}, c_0) \slashedrightarrow (a_1, \map{p_1}{a_1}{fc_1}, c_1)$ consists of a triple $(\hmap j{a_0}{a_1}, w, \hmap k{c_0}{c_1})$ of horizontal morphisms $j$ in $A$ and $k$ in $C$, and a cell $w$ in $J$ that is of the form as on the right below.
	\begin{flalign*}
		&&\begin{tikzpicture}[baseline]
			\matrix(m)[math175em]{a \nc fc \\ a' \nc fc' \\};
			\path[map]	(m-1-1) edge node[above] {$p$} (m-1-2)
													edge node[left] {$s$} (m-2-1)
									(m-1-2) edge node[right] {$ft$} (m-2-2)
									(m-2-1) edge node[below] {$q$} (m-2-2);
		\end{tikzpicture}
		&&\begin{tikzpicture}[baseline]
			\matrix(m)[math175em]{a_0 \nc a_1 \\ fc_0 \nc fc_1 \\};
			\path[map]	(m-1-1) edge[barred] node[above] {$j$} (m-1-2)
													edge node[left] {$p_0$} (m-2-1)
									(m-1-2) edge node[right] {$p_1$} (m-2-2)
									(m-2-1) edge[barred] node[below] {$fk$} (m-2-2);
			\path[transform canvas={shift={($(m-1-2)!(0,0)!(m-2-2)$)}}] (m-1-1) edge[cell] node[right] {$w$} (m-2-1);
		\end{tikzpicture}&&
	\end{flalign*}
	Finally, given a second horizontal morphism $\hmap{(j', w', k')}{(a'_0, p'_0, c'_0)}{(a'_1, p'_1, c'_1)}$, a cell $w \Rar w'$ in $J \slash f$ consists of a pair $(u, v)$ of cells
	\begin{displaymath}
		\begin{tikzpicture}[textbaseline]
			\matrix(m)[math175em]{a_0 & a_1 \\ a'_0 & a'_1 \\};
			\path[map]	(m-1-1) edge[barred] node[above] {$j$} (m-1-2)
													edge node[left] {$s_0$} (m-2-1)
									(m-1-2) edge node[right] {$s_1$} (m-2-2)
									(m-2-1) edge[barred] node[below] {$j'$} (m-2-2);
			\path[transform canvas={shift={($(m-1-2)!(0,0)!(m-2-2)$)}}] (m-1-1) edge[cell] node[right] {$u$} (m-2-1);
		\end{tikzpicture}
		\qquad\text{and}\qquad\begin{tikzpicture}[textbaseline]
			\matrix(m)[math175em]{c_0 & c_1 \\ c'_0 & c'_1, \\};
			\path[map]	(m-1-1) edge[barred] node[above] {$k$} (m-1-2)
													edge node[left] {$t_0$} (m-2-1)
									(m-1-2) edge node[right] {$t_1$} (m-2-2)
									(m-2-1) edge[barred] node[below] {$k'$} (m-2-2);
			\path[transform canvas={shift={($(m-1-2)!(0,0)!(m-2-2)$)}}] (m-1-1) edge[cell] node[right] {$v$} (m-2-1);
		\end{tikzpicture}
	\end{displaymath}
	in $A$ and $C$ respectively, such that $w' \of u = fv \of w$. With vertical composition given indexwise, this completes the description of $J \slash f$ as a $\GG_1$-indexed category.
	
	The $\GG_1$-indexed category $J \slash f$ can be made into a normal colax double category as follows. On horizontal morphisms the structure functor $\map\hc{D(J \slash f)}{J \slash f}$ is given by the assignment
	\begin{displaymath}
		\Bigpars{\begin{tikzpicture}[textbaseline]
		\matrix(m)[math125em]{a_0 \nc a_1 \\ fc_0 \nc fc_1 \\};
				\path[map]  (m-1-1) edge[barred] node[above] {$j_1$} (m-1-2)
														edge node[left] {$p_0$} (m-2-1)
										(m-1-2) edge node[right] {$p_1$} (m-2-2)
										(m-2-1) edge[barred] node[below] {$fk_1$} (m-2-2);
				\path[transform canvas={shift={($(m-1-2)!(0,0)!(m-2-2)$)}}] (m-1-1) edge[cell125] node[right] {$w_1$} (m-2-1);
		\end{tikzpicture}, \dotsc, \begin{tikzpicture}[textbaseline]
		\matrix(m)[math125em]{a_{n'} \nc a_n \\ fc_{n'} \nc fc_n \\};
				\path[map]  (m-1-1) edge[barred] node[above] {$j_n$} (m-1-2)
														edge node[left] {$p_{n'}$} (m-2-1)
										(m-1-2) edge node[right] {$p_n$} (m-2-2)
										(m-2-1) edge[barred] node[below] {$fk_n$} (m-2-2);
				\path[transform canvas={shift={($(m-1-2)!(0,0)!(m-2-2)-(0.15em, 0)$)}}] (m-1-1) edge[cell125] node[right] {$w_n$} (m-2-1);
		\end{tikzpicture}} \mapsto \begin{tikzpicture}[textbaseline]
			\matrix(m)[math125em]
			{	a_0 & a_1 & \dotsb & a_{n'} & a_n \\
				fc_0 & fc_1 & \dotsb & fc_{n'} & fc_n \\
				fc_0 & & & & fc_n \\ };
			\path[map]	(m-1-1) edge[barred] node[above] {$j_1$} (m-1-2)
													edge node[left] {$p_0$} (m-2-1)
									(m-1-2) edge[barred] (m-1-3)
									(m-1-3) edge[barred] (m-1-4)
									(m-1-4) edge[barred] node[above] {$j_n$} (m-1-5)
									(m-1-5) edge node[right] {$p_n$} (m-2-5)
									(m-2-1) edge[barred] node[below] {$fk_1$} (m-2-2)
									(m-2-2) edge[barred] (m-2-3)
									(m-2-3) edge[barred] (m-2-4)
									(m-2-4) edge[barred] node[below] {$fk_n$} (m-2-5)
									(m-3-1) edge[barred] node[below] {$f(k_1 \hc \dotsb \hc k_n)$} (m-3-5);
			\path				(m-2-1) edge[eq] (m-3-1)
									(m-2-5) edge[eq] (m-3-5);
			\path				(m-2-3) edge[cell125] node[right] {$\inv f_\hc$} (m-3-3);
			\path[transform canvas={xshift=-2.55em}]	(m-1-3) edge[cell125] node[right] {$w_1 \hc \dotsb \hc w_n$} (m-2-3);
		\end{tikzpicture}
	\end{displaymath}
	where $w_1 \hc \dotsb \hc w_n$ is the horizontal composite given by the lax structure on $J$, and where $\inv f_\hc$ is the inverse of the colax structure cell of $f$. By mapping a sequence $\bigpars{(u_1, v_1), \dotsc, (u_n, v_n)}$ of composable cells in $D(J \slash f)$ to the pair of composites $(u_1 \hc \dotsb \hc u_n, v_1 \hc \dotsb \hc v_n)$ this extends to a $\GG_1$-indexed functor; that the latter is a well-defined cell follows from the naturality of $f_\hc$ and of the horizontal composition of $J$. Lastly the associator $\nat{\mathfrak a}{\hc \of \mu_{J \slash f}}{\hc \of D\hc}$ is given by $\mathfrak a_{(\ull j, \ull w, \ull k)} = \bigpars{(\mathfrak a_A)_{\ull j}, (\mathfrak a_B)_{\ull k}}$, where $\mathfrak a_A$ and $\mathfrak a_B$ denote the associators of $A$ and $B$.
\end{example}
	
	Recall that, in a pseudo double category $\K$ that has companions and double comma objects, a cell $\eta$ below exhibits $l$ as the pointwise $J$-weighted colimit of $d$ if, for each morphism $\map fCB$, the composite $\eta \of \pi$ exhibits $l \of f$ as the weighted colimit, where $\pi$ corresponds to the universal cell defining the double comma object $J \slash f$, see \defref{definition:strong double comma objects}.
\begin{displaymath}
	\begin{tikzpicture}
		\matrix(m)[math175em, column sep=2em]
		{ J \slash f \nc C \\ A \nc B \\ M \nc M \\[-3.25em] \phantom{J \slash f} \nc \phantom{J \slash f} \\ };
		\path[map]	(m-1-1) edge[barred] node[above] {$C(\pi_C, \id)$} (m-1-2)
												edge node[left] {$\pi_A$} (m-2-1)
								(m-1-2) edge node[right] {$f$} (m-2-2)
								(m-2-1) edge[barred] node[below] {$J$} (m-2-2)
												edge node[left] {$d$} (m-3-1)
								(m-2-2) edge node[right] {$l$} (m-3-2);
		\path				(m-3-1) edge[eq] (m-3-2);
		\path[transform canvas={shift=(m-2-2)}] (m-1-1) edge[cell] node[right] {$\pi$} (m-2-1);
		\path[transform canvas={shift=(m-2-2), yshift=-2pt}]	(m-2-1) edge[cell] node[right] {$\eta$} (m-3-1);
	\end{tikzpicture}
\end{displaymath}
	In $\rcProm T$ we only know that the double comma object $J \slash f$ exists when $J$ is right pseudo and $f$ is pseudo. Notice that in that case the cell $\pi$ above can be constructed as well, since the companion of $\map{\pi_C}{J \slash f}C$ exists in $\rcProm T$ by \propref{restrictions in T-promrc} because $\pi_C$ is a strict morphism. This leads to the following definition.
\begin{definition} \label{definition:pointwise algebraic weighted colimits}
	Let $T$ be a right suitable normal monad on an equipment $\K$ that has double comma objects. Given a right pseudopromorphism $\hmap JAB$ and colax morphisms $\map dAM$ and $\map lBM$, we say that the cell $\eta$ above exhibits $l$ as the \emph{pointwise} $J$-weighted colimit of $d$ if, for each pseudomorphism $\map fCB$, the composition $\eta \of \pi$ above exhibits $l \of f$ as the $C(\pi_C, \id)$-weighted colimit of $d \of \pi_A$.
	
	As before, if $J$ is the companion $B(j, \id)$ of a colax morphism $\map jAB$ then we call $l$ the \emph{pointwise} weighted left Kan extension of $d$ along $j$.
\end{definition}
	The idea of the previous definition is \cite[Section 6.4]{Grandis-Pare08} where `pointwise normal left Kan extensions' of pairs of colax double functors are defined as follows. They consider a double category $\mathsf{CxDbl}^{\textup u}$ with normal colax double functors as vertical morphisms and colax double functors as horizontal ones and, after having shown that the `comma double category $J \slash f$' (see \cite[Section 3.4]{Grandis-Pare08}) can be constructed in $\mathsf{CxDbl}^{\textup u}$ whenever $f$ is a pseudo double functor, define the `pointwise normal left Kan extension' of a colax double functor $\hmap DAM$ along a colax double functor $\hmap JAB$ as a normal colax double functor $\map lBM$, equipped with a cell $\eta$ as below, such that for every normal pseudo double functor $f$ the composite $\eta \of \pi$ exhibits $l \of f$ as the left Kan extension of $\pi_A$, from $C(\pi_C, \id)$ to $D$ (see \defref{definition:left Kan extensions in pseudo double categories}).
\begin{displaymath}
	\begin{tikzpicture}
		\matrix(m)[math175em, column sep=2em]
		{ J \slash f \nc C \\ A \nc B \\ A \nc M \\[-3.25em] \phantom{J \slash f} \nc \phantom{J \slash f} \\ };
		\path[map]	(m-1-1) edge[barred] node[above] {$C(\pi_C, \id)$} (m-1-2)
												edge node[left] {$\pi_A$} (m-2-1)
								(m-1-2) edge node[right] {$f$} (m-2-2)
								(m-2-1) edge[barred] node[below] {$J$} (m-2-2)
								(m-2-2) edge node[right] {$l$} (m-3-2)
								(m-3-1) edge[barred] node[below] {$D$} (m-3-2);
		\path				(m-2-1) edge[eq] (m-3-1);
		\path[transform canvas={shift=(m-2-2)}] (m-1-1) edge[cell] node[right] {$\pi$} (m-2-1);
		\path[transform canvas={shift=(m-2-2), yshift=-2pt}]	(m-2-1) edge[cell] node[right] {$\eta$} (m-3-1);
	\end{tikzpicture}
\end{displaymath}

	Assuming that the double comma objects of $\K$ are all strong, our final aim is to apply \thmref{weighted colimits are pointwise if double commas are strong} to $\rcProm T$, thus proving that all colimits in $\rcProm T$, that are weighted by right pseudopromorphisms, are pointwise. To do so we first have to show that $\rcProm T$ has strong double comma objects. Notice that the cells $\cell{\pi_f}{C(\pi_C, \id)}{J(\id, f)}$, that are used in \defref{definition:strong double comma objects} to define strong double comma objects, exist when $\hmap JAB$ is right pseudo and $\map fCB$ is pseudo, since both the companion $C(\pi_C, \id)$ and the restriction $J(\id, f)$ exist in $\rcProm T$ by \propref{restrictions in T-promrc}. The following enhances \propref{lifting double comma objects}.
\begin{proposition} \label{lifting strong comma objects}
	Consider a right suitable normal monad $T$ on an equipment $\K$. The forgetful functor $\map{U^T}{\rcProm T}\K$ lifts strong double comma objects $J \slash f$ where $\hmap JAB$ is a right pseudopromorphism and $\map fCB$ is a pseudomorphism.
\end{proposition}
\begin{proof}
	In $\K$ consider the double comma object $J \slash f$ of a right pseudopromorphism $\hmap JAB$ and a pseudomorphism $\map fCB$, defined by a universal cell $\pi$. Then $J \slash f$ can be given the structure of a colax algebra so that the projections \mbox{$\map{\pi_A}{J \slash f}A$} and $\map{\pi_C}{J \slash f}C$ become strict morphisms and the cell $\pi$ becomes a $T$-cell, by \propref{lifting double comma objects}. As discussed above, the corresponding cell $\pi_f = \pi \hc (\eta_f \of \ls{\pi_C}\eps)$, that is used in the factorisations \eqref{equation:strong double comma objects} that define $J \slash f$ as a strong double comma object, is a $T$-cell as well. Now consider such a factorisation $\phi = \phi' \of (\pi_f \hc \id_H)$ in $\K$: to prove the proposition, we have to show that if $\phi$ is a $T$-cell then so is $\phi'$. That the $T$-cell axiom for $\phi'$, when precomposed with $\pi_f$, holds follows from the following identities, that follow from the definition of $\phi'$; the $T$-cell axiom for $\phi$; the definition of $\phi'$ again as well as the naturality of $T_\hc$; the $T$-cell axiom for $\pi_f$.
	\begin{align*}
		\begin{tikzpicture}[textbaseline, x=0.75cm, y=0.6cm, font=\scriptsize]
			\draw	(0,1) -- (3,1) -- (3,0) -- (0,0) -- (0,2) -- (1,2) -- (1,1)
				(2,1) -- (2,0);
			\draw[shift={(0,0.5)}]
				(1,0) node {$\phi'$};
			\draw[shift={(0.5,0.5)}]
				(0,1) node {$\pi_f$}
				(2,0) node {$\rho_1\bar k$};
		\end{tikzpicture}
		&\quad = \quad\begin{tikzpicture}[textbaseline, x=0.75cm, y=0.6cm, font=\scriptsize]
			\draw	(2,0) -- (2,2) -- (0,2) -- (0,0) -- (3,0) -- (3,2) -- (2,2)
				(1,1) node {$\phi$};
			\draw[shift={(0.5,0)}]
				(2,1) node {$\rho_1\bar k$};
		\end{tikzpicture}
		\quad = \quad\begin{tikzpicture}[textbaseline, x=0.75cm, y=0.6cm, font=\scriptsize]
			\draw	(0,1) -- (1,1) -- (1,2) -- (0,2) -- (0,0) -- (1,0) -- (1,1)
				(1,2) -- (3,2) -- (3,0) -- (1,0)
				(1,2) -- (1,3) -- (3,3) -- (3,2)
				(2,3) -- (2,4) -- (0,4) -- (0,3) -- (1,3)
				(2,4) -- (3,4) -- (3,5) -- (1,5) -- (1,4)
				(2,1) node {$T\phi$};
			\draw[shift={(0,0.5)}]
				(2,2) node {$T_\hc$}
				(1,3) node {$\ol{C(\pi_C, \id)}$}
				(2,4) node {$\bar H$};
			\draw[shift={(0.5,0.5)}]
				(0,0) node {$\rho_1h$}
				(0,1) node {$\rho_1\pi_A$};
		\end{tikzpicture} \\
		&\quad = \quad\begin{tikzpicture}[textbaseline, x=0.75cm, y=0.6cm, font=\scriptsize]
			\draw	(3,1) -- (0,1) -- (0,0) -- (3,0) -- (3,2) -- (1,2) -- (1,0)
				(2,2) -- (2,3) -- (1,3) -- (1,2) -- (0,2) -- (0,3)
				(2,3) -- (2,4) -- (0,4) -- (0,3) -- (1,3)
				(2,4) -- (3,4) -- (3,5) -- (1,5) -- (1,4);
			\draw[shift={(0,0.5)}]
				(2,0) node {$T\phi'$}
				(2,1) node {$T_\hc$}
				(1,3) node {$\ol{C(\pi_C, \id)}$}
				(2,4) node {$\bar H$};
			\draw[shift={(0.5,0.5)}]
				(0,0) node {$\rho_1h$}
				(0,2) node {$\rho_1\pi_A$}
				(1,2) node {$T\pi_f$};
			\end{tikzpicture}
		\quad = \quad\begin{tikzpicture}[textbaseline, x=0.75cm, y=0.6cm, font=\scriptsize]
			\draw	(3,1) -- (0,1) -- (0,0) -- (3,0) -- (3,2) -- (1,2) -- (1,0)
				(2,2) -- (2,3) -- (0,3) -- (0,2) -- (1,2)
				(0,3) -- (0,4) -- (1,4)
				(2,3) -- (3,3) -- (3,4) -- (1,4) -- (1,3);
			\draw[shift={(0,0.5)}]
				(2,0) node {$T\phi'$}
				(2,1) node {$T_\hc$}
				(1,2) node {$\ol{J(\id, f)}$}
				(2,3) node {$\bar H$};
			\draw[shift={(0.5,0.5)}]
				(0,0) node {$\rho_1h$}
				(0,3) node {$\pi_f$};
			\end{tikzpicture}
	\end{align*}
	The $T$-cell axiom for $\phi'$ follows from the uniqueness of factorisations through $\pi_f$.
\end{proof}

	Finally let us consider a right suitable normal monad $T$ on an equipment $\K$ that has strong double comma objects. Then, by restricting the proof of \thmref{weighted colimits are pointwise if double commas are strong} to pseudomorphisms $\map fCB$ and right pseudopromorphisms $\hmap JAB$, as in \defref{definition:pointwise algebraic weighted colimits} and so that the double comma object $J \slash f$ exists by \propref{lifting double comma objects}, and is strong by \propref{lifting strong comma objects}, it can be applied to the pseudo double category $\rcProm T$ without change. We thus obtain the following variation of \thmref{weighted colimits are pointwise if double commas are strong}.
\begin{theorem} \label{algebraic weighted colimits are pointwise if double commas are strong}
	Let $T$ be a right suitable normal monad on an equipment $\K$, that has strong double comma objects. Given a colax morphism $\map dAM$ as well as a right pseudopromorphism $\hmap JAB$, a $T$-cell
	\begin{displaymath}
		\begin{tikzpicture}
			\matrix(m)[math175em]{A & B \\ M & M \\};
			\path[map]	(m-1-1) edge[barred] node[above] {$J$} (m-1-2)
													edge node[left] {$d$} (m-2-1)
									(m-1-2) edge node[right] {$l$} (m-2-2);
			\path				(m-2-1) edge[eq] (m-2-2);
			\path[transform canvas={shift=($(m-1-2)!0.5!(m-2-2)$)}] (m-1-1) edge[cell] node[right] {$\zeta$} (m-2-1);
		\end{tikzpicture}
	\end{displaymath}
	exhibits $l$, in $\rcProm T$, as the $J$-weighted colimit of $d$ (\defref{definition:weighted colimits in double categories}) if and only if it exhibits $l$ as the pointwise $J$-weighted colimit of $d$ (\defref{definition:pointwise algebraic weighted colimits}).
\end{theorem}

\section{Algebraic weighted colimits} \label{section:main result}
	We are now ready to state the main result. The following definition is motivated by \cite[Definition 3.7]{Lack-Shulman12}, where the notion of `lifting limits along enriched functors' is considered.
\begin{definition} \label{definition:lifts of weighted colimits in equipments}
	Consider a normal functor $\map F\K\L$ between pseudo double categories. Given a horizontal morphism $\hmap JAB$ in $\K$, we say that $F$ \emph{lifts} $J$\ndash weighted colimits if, for any vertical morphism $\map dAM$ in $\K$, the following holds. If the cell $\zeta$ on the left below exhibits $k$ as the $FJ$-weighted colimit of $Fd$ in $\L$ then there exists a cell $\eta$, as on the right, that exhibits $l$ as the $J$-weighted colimit of $d$ in $\K$, such that $F\eta = \zeta$ (and hence $Fl = k$).
	\begin{displaymath}
		\begin{tikzpicture}[baseline]
			\matrix(m)[math175em]{FA & FB \\ FM & FM \\};
			\path[map]	(m-1-1) edge[barred] node[above] {$FJ$} (m-1-2)
													edge node[left] {$Fd$} (m-2-1)
									(m-1-2) edge node[right] {$k$} (m-2-2);
			\path				(m-2-1) edge[eq] (m-2-2);
			\path[transform canvas={shift=($(m-1-2)!0.5!(m-2-2)$)}] (m-1-1) edge[cell] node[right] {$\zeta$} (m-2-1);
		\end{tikzpicture}
		\qquad\qquad\qquad\qquad\begin{tikzpicture}[baseline]
			\matrix(m)[math175em]{A & B \\ M & M \\};
			\path[map]	(m-1-1) edge[barred] node[above] {$J$} (m-1-2)
													edge node[left] {$d$} (m-2-1)
									(m-1-2) edge node[right] {$l$} (m-2-2);
			\path				(m-2-1) edge[eq] (m-2-2);
			\path[transform canvas={shift=($(m-1-2)!0.5!(m-2-2)$)}] (m-1-1) edge[cell] node[right] {$\eta$} (m-2-1);
		\end{tikzpicture}
	\end{displaymath}
\end{definition}
	The main result is as follows.
\begin{theorem} \label{right colax forgetful functor lifts all weighted colimits}
	Let $T$ be a right suitable normal monad on a closed equipment $\K$. The forgetful functor $\map{U^T}{\rcProm T}\K$ lifts all weighted colimits. Moreover its lift of a weighted colimit $\map{\colim_J d}BM$, where $\map dAM$ is a pseudomorphism and $\hmap JAB$ is a right pseudopromorphism, is a pseudomorphism whenever the canonical vertical cell (see \propref{universality result for colimits with respect to functors})
	\begin{displaymath}
		\colim_{TJ} (m \of Td) \Rar m \of T(\colim_J d)
	\end{displaymath}
	is invertible, where $\map m{TM}M$ is the structure map of $M$.
\end{theorem}

By \thmref{weighted colimits are pointwise if double commas are strong} and \thmref{algebraic weighted colimits are pointwise if double commas are strong}, in both $\rcProm T$ and $\K$ a cell $\eta$ exhibits a morphism $l$ as a pointwise weighted colimit if and only if it exhibits $l$ as an ordinary weighted colimit. Combining this with the main result we obtain the following corollary.
\begin{corollary} \label{right colax forgetful functor lifts all pointwise weighted colimits}
	Let $T$ be a right suitable monad on a closed equipment $\K$ that has strong double comma objects. The forgetful functor $\map{U^T}{\rcProm T}\K$ lifts all pointwise weighted colimits $\colim_J d$, where $J$ is a right pseudopromorphism.
\end{corollary}

	In the proof of the main result the following lemma is crucial. Recall that a right colax $T$-pro\-mor\-phism $\hmap JAB$ comes equipped with a horizontal structure cell $\cell{\bar J}{J \hc B(\id, b)}{A(\id, a) \hc TJ}$, while the structure cell $\bar K$ of a lax $T$-promorphism $\hmap KAC$ corresponds, by \propref{left and right cells}, to a horizontal cell $\cell{\lambda\bar K}{TK \hc C(c, \id)}{A(a, \id) \hc K}$ that satisfies the $\lambda$-images of the associativity and unit axioms for $K$, see \propref{left and right structure cells}.
\begin{lemma} \label{lemma:left hom lax structure cell} \label{APPLEMMA}
	Under the hypothesis of \thmref{right colax forgetful functor lifts all weighted colimits}, consider a right colax $T$\ndash pro\-mor\-phism $\hmap JAB$ as well as a lax $T$-promorphism $\hmap KAC$. The left hom $\hmap{J \lhom K}BC$ admits a lax $T$-promorphism structure cell $\ol{J \lhom K}$ that corresponds, under \propref{left and right cells}, to the horizontal cell $\lambda\ol{J \lhom K}$ that is given by
	\begin{multline} \label{equation:left hom lax structure cell}
		T(J \lhom K) \hc C(c, \id) \Rar TJ \lhom \bigpars{TK \hc C(c, \id)} \xRar{\id \lhom \lambda\bar K} TJ \lhom \bigpars{A(a, \id) \hc K} \\
		\iso \bigpars{A(\id, a) \hc TJ} \lhom K \xRar{\bar J \lhom \id} \bigpars{J \hc B(\id, b)} \lhom K \iso B(b, \id) \hc J \lhom K,
	\end{multline}
	where the first cell is given by \propref{lax functors induce left hom coherence cell} and where the two isomorphisms are given by \corref{companion, conjoint and left hom isomorphisms}.
\end{lemma}
	The proof of the above lemma consists of two (the first lengthy) calculations showing that the composite above satisfies the associativity and unit axioms, and will be postponed to \appref{appendix:lemma}. We will denote the composition of the first four cells of \eqref{equation:left hom lax structure cell} above by $\cell{\omega(\bar J, \bar K)}{T(J \lhom K) \hc C(c, \id)}{\bigpars{J \hc B(\id, b)} \lhom K}$. It is useful to compute the adjoint $\omega(\bar J, \bar K)^\sharp$: using item (b) of \lemref{lemma:left hom facts} we find that it is the composite
\begin{align} \label{equation:adjoint of left hom lax structure cell}
	J \hc B(\id, b)&{} \hc T(J \lhom K) \hc C(c, \id) \xRar{\bar J \hc \id} A(\id, a) \hc TJ \hc T(J \lhom K) \hc C(c, \id) \notag \\
	&\xRar{\id \hc T_\hc \hc \id} A(\id, a) \hc T(J \hc J \lhom K) \hc C(c, \id) \notag \\
	&\xRar{\id \hc T\ev \hc \id} A(\id, a) \hc TK \hc C(c, \id) \notag \\
	&\xRar{\id \hc \lambda\bar K} A(\id, a) \hc A(a, \id) \hc K \xRar{\ls a\eps_a \hc \id} K.
\end{align}
	Here $\ls a\eps_a$ is the counit of the companion-conjoint adjunction $A(a, \id) \ladj A(\id, a)$, see \propref{companion-conjoint adjunction}.
	
	The proof of the main theorem takes up the remainder of this section. In the next section we consider its applications for the algebras for some of the monads that we described in \secref{section:monad examples}.
\begin{proof}[Proof of \thmref{right colax forgetful functor lifts all weighted colimits}]
	Let $\hmap JAB$ be a right colax promorphism and \mbox{$\map dAM$} be a colax morphism, and assume that the $J$-weighted colimit $\map{l = \colim_J d}BM$ exists in $\K$. By \propref{weighted colimits equivalent to top absolute left Kan extensions} this means that it comes with a unit
	\begin{displaymath}
		\begin{tikzpicture}
			\matrix(m)[math175em]{A & B \\ M & M \\};
			\path[map]	(m-1-1) edge[barred] node[above] {$J$} (m-1-2)
													edge node[left] {$d$} (m-2-1)
									(m-1-2) edge node[right] {$l$} (m-2-2);
			\path				(m-2-1) edge[eq] (m-2-2);
			\path[transform canvas={shift=($(m-1-2)!0.5!(m-2-2)$)}]	(m-1-1) edge[cell] node[right] {$\eta$} (m-2-1);
		\end{tikzpicture}
	\end{displaymath}
	such that $\cell{(\lambda\eta)^\flat}{M(l, \id)}{J \lhom M(d, \id)}$ is invertible. We have to supply $l$ with a colax structure such that $\eta$ becomes a $T$-cell, and then show that $\eta$ defines $l$ as the $J$-weighted colimit of $d$ in $\rcProm T$. Recall from \propref{colax morphisms correspond to lax companions} that giving $l$ a colax structure is the same as giving $M(l, \id)$ a lax structure. Likewise the colax structure on $d$ corresponds to a lax structure on $M(d, \id)$ so that by the lemma above the left hom $J \lhom M(d, \id) \iso M(l, \id)$ admits a lax structure. We thus obtain a colax structure on $l$, whose vertical structure cell $\cell{\bar l}{l \of b}{m \of Tl}$ corresponds to the composite
	\begin{multline} \label{equation:212976}
		\lambda\bar l = \bigbrks{TM(l, \id) \hc M(m ,\id) \xRar{T(\lambda\eta)^\flat \hc \id} T\bigpars{J \lhom M(d, \id)} \hc M(m, \id) \\
		\xRar{\ol{J \lhom M(d, \id)}} B(b, \id) \hc J \lhom M(d, \id) \xRar{\id \hc \inv{(\lambda\eta)^\flat}} B(b, \id) \hc M(l, \id)}.
	\end{multline}
	Notice that $\ol{J \lhom M(d, \id)}$ is invertible whenever $\rho J$, $\bar l$ and the canonical cell $T\bigpars{J \lhom M(d, \id)} \hc M(m, \id) \Rar TJ \lhom \bigpars{TM(d, \id) \hc M(m, \id)}$ are invertible. As we have seen in \propref{universality result for colimits with respect to functors} the latter is, up to isomorphism, the $\lambda$-image of the canonical cell $\colim_{TJ} (m \of Td) \Rar m \of T(\colim_J d)$, so that the last assertion of the theorem follows.
	
	Having obtained a colax structure on $\map lBM$, we show that $\eta$ is a $T$-cell, which amounts to proving that the following equality holds.
	\begin{displaymath}
		\begin{tikzpicture}[textbaseline]
			\matrix(m)[math175em]
			{ A & B & TB \\
				A & TA & TB \\
				M & TM & TM \\[-3em]
				\phantom{TM} & \phantom{TM} & \phantom{TM} \\ };
			\path[map]	(m-1-1) edge[barred] node[above] {$J$} (m-1-2)
									(m-1-2) edge[barred] node[above] {$B(\id, b)$} (m-1-3)
									(m-2-1) edge[barred] node[above] {$A(\id, a)$} (m-2-2)
													edge node[left] {$d$} (m-3-1)
									(m-2-2) edge[barred] node[above] {$TJ$} (m-2-3)
													edge node[right] {$Td$} (m-3-2)
									(m-2-3) edge node[right] {$Tl$} (m-3-3)
									(m-3-1) edge[barred] node[below] {$M(\id, m)$} (m-3-2);
			\path				(m-1-1) edge[eq] (m-2-1)
									(m-1-2) edge[cell] node[right] {$\bar J$} (m-2-2)
									(m-1-3) edge[eq] (m-2-3)
									(m-3-2) edge[eq] (m-3-3);
			\path[transform canvas={shift={($(m-2-1)!0.5!(m-2-2)$)}, xshift=-0.675em}]	(m-2-2) edge[cell] node[right] {$\rho_1\bar d$} (m-3-2);
			\path[transform canvas={shift={($(m-2-1)!0.5!(m-2-2)$)}}] (m-2-3) edge[cell] node[right] {$T\eta$} (m-3-3);
		\end{tikzpicture}
		= \begin{tikzpicture}[textbaseline]
			\matrix(m)[math175em]
			{ A & B & TB \\
				M & M & TM \\[-3em]
				\phantom{TM} & \phantom{TM} & \phantom{TM} \\ };
			\path[map]	(m-1-1) edge[barred] node[above] {$J$} (m-1-2)
													edge node[left] {$d$} (m-2-1)
									(m-1-2) edge[barred] node[above] {$B(\id, b)$} (m-1-3)
													edge node[right] {$l$} (m-2-2)
									(m-1-3) edge node[right] {$Tl$} (m-2-3)
									(m-2-2) edge[barred] node[below] {$M(\id, m)$} (m-2-3);
			\path				(m-2-1) edge[eq] (m-2-2);
			\path[transform canvas={shift={($(m-1-1)!0.5!(m-2-2)$)}, xshift=-0.5em}]	(m-1-3) edge[cell] node[right] {$\rho_1\bar l$} (m-2-3);
			\path[transform canvas={shift={($(m-1-1)!0.5!(m-2-2)$)}}] (m-1-2) edge[cell] node[right] {$\eta$} (m-2-2);
		\end{tikzpicture}
	\end{displaymath}
	Equivalently we may prove that its image under $\lambda$ holds, which is the equality of cells $J \hc B(\id, b) \hc TM(l, \id) \Rar M(d, \id) \hc M(\id, m)$ that is represented by the following diagram, in which all details but the non-identity cells are left out.
	\begin{equation} \label{equation:179425}
		\begin{split}
		\begin{tikzpicture}[textbaseline, x=0.75cm, y=0.6cm, font=\scriptsize]
			\draw	(1,1) -- (0,1) -- (0,0) -- (2,0) -- (2,1)
						(1,2) -- (1,1) -- (3,1) -- (3,2) -- (2,2)
						(0,2) -- (2,2) -- (2,3) -- (0,3) -- (0,2);
			\draw[shift={(0,0.5)}]
						(1,0) node {$\lambda\rho_1\bar d$}
						(2,1) node {$\lambda T\eta$}
						(1,2)	node {$\bar J$};
		\end{tikzpicture}
		\quad = \quad\begin{tikzpicture}[textbaseline, x=0.75cm, y=0.6cm, font=\scriptsize]
			\draw	(1,1) -- (0,1) -- (0,0) -- (2,0) -- (2,1)
						(1,2) -- (1,1) -- (3,1) -- (3,2) -- (1,2);
			\draw[shift={(0,0.5)}]
						(2,1) node {$\lambda\rho_1\bar l$}
						(1,0)	node {$\lambda\eta$};
		\end{tikzpicture}
		\end{split}
	\end{equation}
	It follows from the definitions of $\rho_1$ and $\lambda$ (\defref{definition:cells of right colax promorphisms} and \propref{left and right cells}), as well as the vertical companion identity, that the horizontal cell $\lambda\rho_1\bar l$ here can be rewritten in terms of $\lambda\bar l$ as
	\begin{align*}
		\lambda\rho_1\bar l = \bigbrks{B(\id, b) \hc TM(l, \id) &\xRar{\id \hc \ls m\eta_m} B(\id, b) \hc TM(l, \id) \hc M(m, \id) \hc M(\id, m) \\
		&\xRar{\id \hc \lambda\bar l \hc \id} B(\id, b) \hc B(b, \id) \hc M(l, \id) \hc M(\id, m) \\
		&\xRar{\ls b\eps_b \hc \id} M(l, \id) \hc M(\id, m)},
	\end{align*}
	where $\ls m\eta_m = \ls m\eta \hc \eta_m$ and $\ls b\eps_b = \eps_b \hc \ls b\eps$ are the unit and counit of the companion-conjoint adjunctions $M(m, \id) \ladj M(\id, m)$ and $B(b, \id) \ladj B(\id, b)$ of \propref{companion-conjoint adjunction}. In the same way $\lambda\rho_1\bar d$ can be written in terms of $\lambda\bar d$. Substituting \eqref{equation:212976} into the above we find that $\lambda\rho_1\bar l$ equals the composite
	\begin{align*}
		B(\id, b)&{} \hc TM(l, \id) \\
		&\xRar{\id \hc T(\lambda\eta)^\flat \hc \ls m\eta_m} B(\id, b) \hc T\bigpars{J \lhom M(d, \id)} \hc M(m, \id) \hc M(\id, m) \\
		&\xRar{\id \hc \ol{J \lhom M(d, \id)} \hc \id} B(\id, b) \hc B(b, \id) \hc J \lhom M(d, \id) \hc M(\id, m) \\
		&\xRar{\ls b\eps_b \hc \id} J \lhom M(d, \id) \hc M(\id, m) \xRar{\inv{(\lambda\eta)^\flat} \hc \id} M(l, \id) \hc M(\id, m).
	\end{align*}
	Now in the right-hand side of \eqref{equation:179425} we can rewrite $\lambda\eta \hc \id = {(\lambda\eta)^\flat}^\sharp \hc \id$. Precomposed with the last cell of the composite above this gives $\bigpars{{(\lambda\eta)^\flat}^\sharp \of (\id \hc \inv{(\lambda\eta)^\flat})} \hc \id = \bigpars{(\lambda\eta)^\flat \of \inv{(\lambda\eta)^\flat}}^\sharp \hc \id = \id^\sharp \hc \id = \ev \hc \id$, where we have used the naturality of $\dash^\sharp$ in its second argument. It follows that the right-hand side of \eqref{equation:179425} is given by
	\begin{align} \label{equation:610993}
		J \hc {}&B(\id, b) \hc TM(l, \id) \notag \\
		&\xRar{\id \hc T(\lambda\eta)^\flat \hc \ls m\eta_m} J \hc B(\id, b) \hc T\bigpars{J \lhom M(d, \id)} \hc M(m, \id) \hc M(\id, m) \notag \\
		&\xRar{\id \hc \ol{J \lhom M(d, \id)} \hc \id} J \hc B(\id, b) \hc B(b, \id) \hc J \lhom M(d, \id) \hc M(m, \id) \notag \\
		&\xRar{\id \hc \ls b\eps_b \hc \id} J \hc J \lhom M(d, \id) \hc M(m, \id) \xRar{\ev \hc \id} M(d, \id) \hc M(m, \id).
	\end{align}
	In the above consider the composition of the last two cells together with the last cell of the composite $\ol{J \lhom M(d, \id)}$, as given by \eqref{equation:left hom lax structure cell}:
	\begin{multline*}
		J \hc B(\id, b) \hc \bigpars{J \hc B(\id, b)} \lhom M(d, \id) \iso J \hc B(\id, b) \hc B(b, \id) \hc J \lhom M(d, \id) \\
		\xRar{\id \hc \ls b\eps_b \hc \id} J \hc J \lhom M(d, \id) \xRar{\ev} M(d, \id).
	\end{multline*}
	Writing $\phi = \ev \of (\id \hc \ls b\eps_b \hc \id)$ for the last two cells here, notice that the isomorphism, from right to left, equals $\id_{J \hc B(\id, b)} \hc \phi^\flat$, by \corref{companion, conjoint and left hom isomorphisms}. Of course $\ev \of (\id_{J \hc B(\id, b)} \hc \phi^\flat) = \phi$, and we conclude that the composite above is simply given by evaluation. This implies that \eqref{equation:610993} equals $\id \hc T(\lambda\eta)^\flat \hc \ls m\eta_m$ followed by the adjoint of $\omega(\bar J, \lambda\bar d)$, which is the composite of the first four cells of \eqref{equation:left hom lax structure cell}. The latter we have already computed in \eqref{equation:adjoint of left hom lax structure cell}, from which it follows that the right-hand side of \eqref{equation:179425} can be rewritten as
	\begin{align*}
		J \hc B(\id, b)&{} \hc TM(l, \id) \xRar{\bar J \hc T(\lambda\eta)^\flat} A(\id, a) \hc TJ \hc T\bigpars{J \lhom M(d, \id)} \\
		&\xRar{\id \hc T_\hc} A(\id, a) \hc T\bigpars{J \hc J \lhom M(d, \id)} \\
		&\xRar{\id \hc T\ev \hc \ls m\eta_m} A(\id, a) \hc TM(d, \id) \hc M(m, \id) \hc M(\id, m) \\
		&\xRar{\id \hc \lambda\bar d \hc \id} A(\id, a) \hc A(a, \id) \hc M(d, \id) \hc M(\id, m) \\
		&\xRar{\ls a\eps_a \hc \id} M(d, \id) \hc M(\id, m).
	\end{align*}
	We have already seen that the composition of $\ls m \eta_m$, $\lambda\bar d$ and $\ls a\eps_a$ here equals $\lambda\rho_1\bar d$. Also notice that $T\ev \of T_\hc \of \bigpars{\id \hc T(\lambda\eta)^\flat} = T\ev \of T\bigpars{\id \hc (\lambda\eta)^\flat} \of T_\hc = T\lambda\eta \of T_\hc = \lambda T\eta$, where the last equality follows from the unit axioms for $T$. So, finally, the composite above equals
	\begin{multline*}
		J \hc B(\id, b) \hc TM(l, \id) \xRar{\bar J \hc \id} A(\id, a) \hc TJ \hc TM(l, \id) \\
		\xRar{\id \hc \lambda T\eta} A(\id, a) \hc TM(d, \id) \xRar{\lambda\rho_1\bar d} M(d, \id) \hc M(\id, m),
	\end{multline*}
	which is the left-hand side of \eqref{equation:179425}. This shows that, with respect to the chosen colax structure on $l$, the cell $\eta$ is a $T$-cell.
	
	To complete the proof we have to show that $\eta$ defines $l$ as the $J$-weighted colimit of $d$ in $\rcProm T$. Thus we have to show that every $T$-cell $\phi$ as below factors uniquely through $\eta$, as shown.
	\begin{displaymath}
		\begin{tikzpicture}[textbaseline]
			\matrix(m)[math175em]{A & B & C \\ M & \phantom M & M \\};
			\path[map]	(m-1-1) edge[barred] node[above] {$J$} (m-1-2)
													edge node[left] {$d$} (m-2-1)
									(m-1-2) edge[barred] node[above] {$H$} (m-1-3)
									(m-1-3) edge node[right] {$e$} (m-2-3);
			\path				(m-2-1) edge[eq] (m-2-3)
									(m-1-2) edge[cell] node[right] {$\phi$} (m-2-2);
		\end{tikzpicture}
		= \begin{tikzpicture}[textbaseline]
			\matrix(m)[math175em]{A & B & C \\ M & M & M, \\};
			\path[map]	(m-1-1) edge[barred] node[above] {$J$} (m-1-2)
													edge node[left] {$d$} (m-2-1)
									(m-1-2) edge[barred] node[above] {$H$} (m-1-3)
													edge node[right] {$l$} (m-2-2)
									(m-1-3) edge node[right] {$e$} (m-2-3);
			\path				(m-2-1) edge[eq] (m-2-2)
									(m-2-2) edge[eq] (m-2-3);
			\path[transform canvas={shift=($(m-1-2)!0.5!(m-2-3)$)}]	(m-1-1) edge[cell] node[right] {$\eta$} (m-2-1)
									(m-1-2) edge[cell] node[right] {$\phi'$} (m-2-2);
		\end{tikzpicture}
	\end{displaymath}
	Forgetting about the $T$-algebra structures, such a factorisation does exist in $\K$, since $\eta$ defines $l$ as a weighted colimit. To show that this factorisation lifts to $\rcProm T$ we have to prove that $\phi'$ is a $T$-cell. That the $T$-cell axiom for $\phi'$ holds after it is horizontally precomposed with $\eta$ is shown below, where the identities are the $T$-cell axiom for $\eta$, the factorisation $T \phi = T(\eta \hc \phi')$, the $T$-cell axiom for $\phi$, and again the factorisation $\phi = \eta \hc \phi'$.
	\begin{align*}
		\begin{tikzpicture}[textbaseline, x=0.75cm, y=0.6cm, font=\scriptsize]
			\draw	(2,1) -- (0,1) -- (0,0) -- (3,0) -- (3,1) -- (2,1) -- (2,0)
						(1,0) -- (1,2) -- (3,2) -- (3,1);
			\draw[shift={(0.5,0.5)}]
						(0,0) node {$\eta$}
						(1,0) node {$\rho_1\bar l$}
						(2,0) node {$T\phi'$};
			\draw[shift={(0,0.5)}]
						(2,1) node {$\bar H$};
		\end{tikzpicture}
		&\quad=\quad \begin{tikzpicture}[textbaseline, x=0.75cm, y=0.6cm, font=\scriptsize]
			\draw (1,0) -- (1,1) -- (0,1) -- (0,0) -- (3,0) -- (3,1) -- (2,1) -- (2,0)
						(0,1) -- (0,2) -- (2,2) -- (2,1) -- (1,1)
						(1,2) -- (1,3) -- (3,3) -- (3,2) -- (2,2);
			\draw[shift={(0.5,0.5)}]
						(0,0) node {$\rho_1\bar d$}
						(1,0) node {$T\eta$}
						(2,0) node {$T\phi'$};
			\draw[shift={(0,0.5)}]
						(1,1) node {$\bar J$}
						(2,2) node {$\bar H$};
		\end{tikzpicture}
		\quad=\quad \begin{tikzpicture}[textbaseline, x=0.75cm, y=0.6cm, font=\scriptsize]
			\draw (0,0) -- (3,0) -- (3,1) -- (0,1) -- (0,0)
						(1,0) -- (1,2) -- (3,2) -- (3,1)
						(1,2) -- (0,2) -- (0,3) -- (2,3) -- (2,2)
						(1,3) -- (1,4) -- (3,4) -- (3,3) -- (2,3);
			\draw[shift={(0.5,0.5)}]
						(0,0) node {$\rho_1\bar d$};
			\draw[shift={(0,0.5)}]
						(2,0) node {$T\phi$}
						(2,1) node {$T_\hc$}
						(1,2) node {$\bar J$}
						(2,3) node {$\bar H$};
		\end{tikzpicture} \\[1em]
		&\quad=\quad \begin{tikzpicture}[textbaseline, x=0.75cm, y=0.6cm, font=\scriptsize]
			\draw (2,0) -- (0,0) -- (0,1) -- (3,1) -- (3,0) -- (2,0) -- (2,1);
			\draw[shift={(0.5,0.5)}]
						(2,0) node {$\rho_1\bar e$};
			\draw[shift={(0,0.5)}]
						(1,0) node {$\phi$};
		\end{tikzpicture}
		\quad=\quad \begin{tikzpicture}[textbaseline, x=0.75cm, y=0.6cm, font=\scriptsize]
			\draw	(2,1) -- (0,1) -- (0,0) -- (3,0) -- (3,1) -- (2,1) -- (2,0)
						(1,1) -- (1,0);
			\draw[shift={(0.5,0.5)}]
						(0,0) node {$\eta$}
						(1,0) node {$\phi'$}
						(2,0) node {$\rho_1\bar e$};
		\end{tikzpicture}
	\end{align*}
	Since factorisations through $\eta$ are unique the $T$-cell axiom for $\phi'$ follows, completing the proof.
\end{proof}

\section{Applications} \label{section:applications of the main theorem}
	In this section we apply the main theorem to some of the monads that were described in \secref{section:monad examples}. We shall also compare our main result to a result of Melli\`es and Tabareau, given in \cite{Mellies-Tabareau08}, which gives a different approach to obtaining algebra structures on left Kan extensions.
	
\subsection{Monoidal $\V$-categories}
	We start with the `free (symmetric) strict monoidal $\V$-category'-monads $\fmc$ and $\fsmc$ on the equipment $\enProf\V$ of $\V$-profunctors. Applying the main theorem we find that the forgetful functors of pseudo double categories
	\begin{displaymath}
		\map{U^\fmc}{\rcMonProf\V}{\enProf\V} \qquad \text{and} \qquad \map{U^\fsmc}{\rcsMonProf\V}{\enProf\V}
	\end{displaymath}
	lift weighted colimits, where $\rcMonProf\V$ is the pseudo double category of right colax monoidal $\V$-profunctors and $\rcsMonProf\V$ is that of symmetric right colax monoidal $\V$-profunctors, see \exref{example:pseudo double category of right colax monoidal profunctors}.
	
	Unpacking this result for $U^\fmc$ we obtain the following. Recall that weighted colimits in a cocomplete $\V$-category can be computed using its coends and copowers, see \exref{example:weighted colimits in V-Prof}.
\begin{proposition} \label{colax monoidal functors application}
	Given a colax monoidal $\V$-functor $\map dAK$ into a cocomplete $\V$-category $K$ and a right colax monoidal $\V$-profunctor $\hmap JAB$. Then the weighted colimit $\map{l = \colim_J d}BK$ exists in $\enProf\V$ and can be given by the coends $l(z) = \int^x dx \tens J(x, z)$. The composites below, where $\ul z \in \fmc_n B$, give $l$ the structure of a colax monoidal $\V$-functor, with which it is the $J$-weighted colimit of $d$ in $\rcMonProf\V$.
	\begin{align} \label{equation:colax monoidal functors application}
		l&{}(\ul z_1 \tens \dotsb \tens \ul z_n) = \int^x dx \tens J(x, \ul z_1 \tens \dotsb \tens \ul z_n) \notag \\
		&\xrar{\int \id \tens \bar J} \int^x dx \tens \Bigpars{\int^{\ul y} A(x, \ul y_1 \tens \dotsb \tens \ul y_n) \tens J(\ul y_1, \ul z_1) \tens \dotsb \tens J(\ul y_n, \ul z_n)} \notag \\
		&\iso \int^{\ul y} \Bigpars{\int^x dx \tens A(x, \ul y_1 \tens \dotsb \tens \ul y_n)} \tens J(\ul y_1, \ul z_1) \tens \dotsb \tens J(\ul y_n, \ul z_n) \notag \\
		&\iso \int^{\ul y} d(\ul y_1 \tens \dotsb \tens \ul y_n) \tens \bigpars{J(\ul y_1, \ul z_1) \tens \dotsb \tens J(\ul y_n, \ul z_n)} \notag \\
		&\xrar{\int d_\tens \tens \id} \int^{\ul y} \bigpars{d\ul y_1 \tens \dotsb \tens d\ul y_n} \tens \bigpars{J(\ul y_1, \ul z_1) \tens \dotsb \tens J(\ul y_n, \ul z_n)} \notag \\
		&\to \int^{y_1} \bigpars{dy_1 \tens J(y_1, \ul z_1)} \tens \dotsb \tens \int^{y_n} \bigpars{dy_n \tens J(y_n, \ul z_n)} = l\ul z_1 \tens \dotsb \tens l\ul z_n
	\end{align}
	The second isomorphism here is the Yoneda isomorphism while the unlabelled map is a component of the canonical transformation 
	\begin{displaymath}
		\colim_{\fmc_n J}(k_n \of \fmc_n d) \natarrow k_n \of \fmc_n(\colim_J d)
	\end{displaymath}
	given by \propref{universality result for colimits with respect to functors}, where $\map{k_n}{\fmc_n K}K$ is a restriction of the structure map of $K$, that defines the $n$-fold tensor products of $K$. The coends above range over single objects in $A$ and sequences in $\fmc_n A$.
	
	The monoidal structure for $l$ above is invertible whenever $d$ is a pseudomonoidal $\V$-functor, $J$ is a right pseudomonoidal $\V$-profunctor and the canonical cells above are invertible for each $n \geq 1$. For the latter to hold it suffices that $K$ is a pseudomonoidal $\V$-category such that $\map{k_2}{K \tens K}K$ preserves weighted colimits in both variables.
\end{proposition}

	Before giving the proof we compare Getzler's Proposition 2.3 of \cite{Getzler09}. In stating it Getzler uses the fact that the $\V$-category of $\V$-presheaves $\op A \to \V$ can be given a symmetric pseudomonoidal structure, as follows. The monoidal product of $\V$-presheaves $d_1, \dotsc, \map{d_n}{\op A}\V$ is given by the formula
\begin{displaymath}
	(d_1 \tens \dotsb \tens d_n)(x) = \int^{\ul y \in M_n A} A(x, \ul y_1 \tens \dotsb \tens \ul y_n) \tens d_1(\ul y_1) \tens \dotsb \tens d_n(\ul y_n).
\end{displaymath}
	This is called the \emph{Day convolution} of $d_1, \dotsc, d_n$, which was introduced in \cite[Section 4]{Day70}, and can be extended to a symmetric pseudomonoidal structure on $\fun{\op A} \V$. Moreover any $\V$-functor $\map{j^*}{\fun{\op B}\V}{\fun{\op A}\V}$, that is given by precomposition with a symmetric colax monoidal $\V$-functor $\map jAB$, can be made into a symmetric lax monoidal $\V$-functor by the following structure maps, where $\map{d_1, \dotsc, d_n}{\op B}\V$, where the coends range over $\ul y \in M_n A$ and $\ul z \in M_n B$ and where the unlabelled map is induced by the universality of coends.
\begin{align} \label{diagram:pattern coherence map}
	(j^*d_1&{} \tens \dotsb \tens j^*d_n)(x) = \int^{\ul y} A(x, \ul y_1 \tens \dotsb \tens \ul y_n) \tens d_1 \of j(\ul y_1) \tens \dotsb \tens d_n \of j(\ul y_n) \notag \\
	&\xrar{\int j \tens \id} \int^{\ul y} B\bigpars{j(x), j(\ul y_1 \tens \dotsb \tens \ul y_n)} \tens d_1 \of j(\ul y_1) \tens \dotsb \tens d_n \of j(\ul y_n) \notag \\
	&\xrar{\int B(\id, j_\tens) \tens \id} \int^{\ul y} B\bigpars{j(x), j(\ul y_1) \tens \dotsb \tens j(\ul y_n)} \tens d_1 \of j(\ul y_1) \tens \dotsb \tens d_n \of j(\ul y_n) \notag \\
	&\to \int^{\ul z} B\bigpars{j(x), \ul z_1 \tens \dotsb \tens \ul z_n} \tens d_1(\ul z_1) \tens \dotsb \tens d_n(\ul z_n) \notag \\
	&= j^*(d_1 \tens \dotsb \tens d_n)(x)
\end{align}
	Getzler calls a symmetric colax monoidal $\V$-functor $\map jAB$, for which the lax monoidal coherence map for $j^*$ above is invertible, a \emph{$\V$-pattern}. Using this notion, Proposition 2.3 of \cite{Getzler09} is stated as follows.
\begin{proposition}[Getzler] \label{Getzler's proposition}
	Let $\map jAB$ be a $\V$-pattern, and let $M$ be a cocomplete symmetric pseudomonoidal $\V$-category whose binary monoidal product $\dash \tens \dash$ preserves colimits in both variables. The adjunction between the categories of $\V$\ndash functors
	\begin{displaymath}
		\begin{tikzpicture}
			\matrix(m)[math, column sep=1.5em]{\lan_j \colon \fun AM & \fun BM \colon j^* \\ };
			\path[transform canvas={yshift=3.5pt}, map] (m-1-1) edge (m-1-2);
			\path[transform canvas={yshift=-3.5pt}, map]	(m-1-2) edge (m-1-1);
			\path[color=white] (m-1-1) edge node[color=black, font=\scriptsize] {$\bot$} (m-1-2);
		\end{tikzpicture}
	\end{displaymath}
	induces an adjunction between the categories of symmetric pseudomonoidal $\V$-functors
	\begin{displaymath}
		\begin{tikzpicture}
			\matrix(m)[math, column sep=1.5em]{\lan_j \colon \psfun AM & \psfun BM \colon j^*. \\ };
			\path[transform canvas={yshift=3.5pt}, map] (m-1-1) edge (m-1-2);
			\path[transform canvas={yshift=-3.5pt}, map]	(m-1-2) edge (m-1-1);
			\path[color=white] (m-1-1) edge node[color=black, font=\scriptsize] {$\bot$} (m-1-2);
		\end{tikzpicture}
	\end{displaymath}
\end{proposition}
	
	We claim that a symmetric colax monoidal $\V$-functor $\map jAB$ is a $\V$-pattern precisely if its companion $\hmap{B(j, \id)}AB$ is a right pseudopromorphism, so that \propref{colax monoidal functors application} specialises to Getzler's result in the case that the weight is taken to be the companion $J = B(j, \id)$. To prove the claim we consider the components of the structure map \eqref{diagram:pattern coherence map} for $j^*$ at representable $\V$-presheaves $d_i = B(\dash, \ul u_i)$, for $\ul u \in M_n B$. In that case the source of \eqref{diagram:pattern coherence map} is $(j^* d_1 \tens \dotsc \tens j^*d_n)(x) = \bigpars{A(\id, a) \hc_{MA} \fmc B(j, \id)}(x, \ul u)$, where $\map a{M_n A}A$ is the structure map of $A$, while its target is
\begin{displaymath}
	j^*(d_1 \tens \dotsb \tens d_n)(x) = \int^{\ul z} B\bigpars{j(x), \ul z_1 \tens \dotsb \tens \ul z_n} \tens M_n B(\ul z, \ul u) \iso B\bigpars{j(x), \ul u_1 \tens \dotsb \tens \ul u_n},
\end{displaymath}
	using the enriched Yoneda isomorphism. It is readily checked that, composed with this isomorphism, \eqref{diagram:pattern coherence map} coincides with the horizontal cell $\rho\ol{B(j, \id)}$, corresponding to the structure cell $B(j, \id)$ that makes $B(j, \id)$ into a symmetric lax monoidal $\V$\ndash profunctor; $\rho\ol{B(j, \id)}$ is obtained by substituting \eqref{equation:companion monoidal functor} into \eqref{equation:right pseudo monoidal profunctor}. This shows that $j$ is a $\V$-pattern, that is $j^*$ is pseudomonoidal, only if $\ol{B(j, \id)}$ is right invertible, i.e.\ $B(j, \id)$ is a right pseudopromorphism. The converse follows from the fact that any $\V$-presheaf $\map d{\op A}\V$ can be written as a colimit of representable $\V$-presheaves $d = \int^x A(\dash, x) \tens d(x)$, see \cite[Formula 3.72]{Kelly82}.

\begin{proof}[Proof of \propref{colax monoidal functors application}]
	We can apply \thmref{right colax forgetful functor lifts all weighted colimits} to obtain a colax monoidal structure vertical cell $\cell{l_\tens}{l \of b}{k \of \fmc l}$ for $l$. To see that $l_\tens$ coincides with \eqref{equation:colax monoidal functors application} recall that it corresponds to the horizontal cell $\lambda l_\tens$ given by the composite \eqref{equation:left hom lax structure cell}, with $T = \fmc$ and $K = K(d, \id)$. Working out $\lambda l_\tens$ we find that its components are the $\V$-natural compositions
	\begin{align*}
		\Bigpars{\fmc_n \bigpars{J&{} \lhom K(d, \id)} \hc K(k, \id)}(\ul z, t) \\
		&= \int^{\ul u} \int_{y_1} \!\begin{lgathered}[t]{} \brks{J(y_1, \ul z_1), K(dy_1, \ul u_1)} \\
		{}\tens \dotsb \tens \int_{y_n} \bigbrks{J(y_n, \ul z_n), K(dy_n, \ul u_n)} \tens K(\ul u_1 \tens \dotsb \tens \ul u_n, t) \end{lgathered} \\
		&\to \int_{\ul y} \bigbrks{J(\ul y_1, \ul z_1) \tens \dotsb \tens J(\ul y_n, \ul z_n), K(d\ul y_1 \tens \dotsb \tens d\ul y_n, t)} \\
		&\xrar{\int\brks{\id, K(d_\tens, \id)}} \int_{\ul y} \bigbrks{J(\ul y_1, \ul z_1) \tens \dotsb \tens J(\ul y_n, \ul z_n), K\bigpars{d(\ul y_1 \tens \dotsb \tens \ul y_n), t}} \\
		&\iso \int_{\ul y} \Bigbrks{J(\ul y_1, \ul z_1) \tens \dotsb \tens J(\ul y_n, \ul z_n), \int_x \bigbrks{A(x, \ul y_1 \tens \dotsb \tens \ul y_n), K(dx, t)}} \\
		&\iso \int_x \bigbrks{\int^{\ul y} A(x, \ul y_1 \tens \dotsb \tens \ul y_n) \tens J(\ul y_1, \ul z_1) \tens \dotsb \tens J(\ul y_n, \ul z_n), K(dx, t)} \\
		&\xrar{\int\brks{\bar J, \id}} \int_x \bigbrks{J(x, \ul z_1 \tens \dotsb \tens \ul z_n), K(dx, t)} = \bigpars{J \lhom K(d, \id)}(\ul z_1 \tens \dotsb \tens \ul z_n, t),
	\end{align*}
	where we have used that the first isomorphism in \eqref{equation:left hom lax structure cell} is in fact the composite $TJ \lhom \bigpars{A(a, \id) \hc K} \iso TJ \lhom \bigpars{A(\id, a) \lhom K} \iso \bigpars{A(\id, a) \hc TJ} \lhom K$, see \corref{companion, conjoint and left hom isomorphisms}. The source and target of the composite above define the objects $l\ul z_1 \tens \dotsb \tens l\ul z_n$ and $l(\ul z_1 \tens \dotsb \tens \ul z_n)$ in $K$, via the $\V$-natural isomorphisms
	\begin{align*}
		\Bigpars{\fmc_n \bigpars{J&{} \lhom K(d, \id)} \hc K(k, \id)}(\ul z, t) \\
		&= \int^{\ul u} \int_{y_1} \!\begin{lgathered}[t]{} \brks{J(y_1, \ul z_1), K(dy_1, \ul u_1)} \\
		{}\tens \dotsb \tens \int_{y_n} \bigbrks{J(y_n, \ul z_n), K(dy_n, \ul u_n)} \tens K(\ul u_1 \tens \dotsb \tens \ul u_n, t) \end{lgathered} \\
		&\iso \int^{\ul u} K\!\begin{lgathered}[t]{}\bigpars{\int^{y_1} dy_1 \tens J(y_1, \ul z_1), \ul u_1} \\
		{}\tens \dotsb \tens K\bigpars{\int^{y_n} dy_n \tens J(y_n, \ul z_n), \ul u_n} \tens K(\ul u_1 \tens \dotsb \tens \ul u_n, t) \end{lgathered} \\
		&\iso K\bigpars{\int^{y_1} dy_1 \tens J(y_1, \ul z_1) \tens \dotsb \tens \int^{y_n} dy_n \tens J(y_n, \ul z_n), t} \\
		&= K(l\ul z_1 \tens \dotsb \tens l\ul z_n, t)
	\end{align*}
	and
	\begin{multline*}
		\bigpars{J \lhom K(d, \id)}(\ul z_1 \tens \dotsb \tens \ul z_n, t) = \int_x \bigbrks{J(x, \ul z_1 \tens \dotsb \tens \ul z_n), K(dx, t)} \\
		\iso K\bigpars{\int^x dx \tens J(y, \ul z_1 \tens \dotsb \tens \ul z_n), t} = K\bigpars{l(\ul z_1 \tens \dotsb \tens \ul z_n), t},
	\end{multline*}
	where the first and third isomorphism are given by the defining isomorphisms of coends and copowers in $K$, while the second is the compositor of $\lambda$. It is now easily seen that, under these isomorphisms, the five maps in the composite $\lambda l_\tens$ above define, in reversed order, the five maps in the composition \eqref{equation:colax monoidal functors application} (for the unlabelled map this follows from \propref{universality result for colimits with respect to functors}). This proves the first assertion of the proposition, while the second one immediately follows from \thmref{right colax forgetful functor lifts all weighted colimits}.
	
	To prove the last assertion assume that $K$ is a pseudomonoidal $\V$-category, with invertible associator. We have to show that the canonical maps
	\begin{multline} \label{equation:761211}
		\int^{\ul y} k_n(d\ul y_1, \dotsc, d\ul y_n) \tens \bigpars{J(\ul y_1, \ul z_1) \tens \dotsb \tens J(\ul y_n, \ul z_n)} \\
		\to k_n\bigpars{\int^{y_1} dy_1 \tens J(y_1, \ul z_1), \dotsc, \int^{y_n} dy_n \tens J(y_n, \ul z_n)} 
	\end{multline}
	are isomorphisms for all $n \geq 1$ where, for clarity, we have refrained from using the tensor product notations for the $k_n$-images. They are induced by the maps
	\begin{multline*}
		k_n(d\ul y_1, \dotsc, d\ul y_n) \tens \bigpars{J(\ul y_1, \ul z_1) \tens \dotsb \tens J(\ul y_n, \ul z_n)} \\
		\to k_n\bigpars{d \ul y_1 \tens J(\ul y_1, \ul z_1), \dotsc, d\ul y_n \tens J(\ul y_n, \ul z_n)},
	\end{multline*}
	given by the universality of the copowers, followed by the $k_n$-image of the insertions into the coends $\int^{y_i} dy_i \tens J(y_i, \ul z_i)$. A straightforward calculation shows that, when $n \geq 2$, the canonical map above can be rewritten as (here $n' = n - 1$)
	\begin{align*}
		&\int^{\ul y} k_n(d\ul y_1, \dotsc, d\ul y_n) \tens \bigpars{J(\ul y_1, \ul z_1) \tens \dotsb \tens J(\ul y_n, \ul z_n)} \\
		&\overset{(1)}\iso \int^{\ul y} \Bigpars{k_2\bigpars{d\ul y_1, k_{n'}(d\ul y_2, \dotsc, d\ul y_n)} \tens \bigpars{J(\ul y_2, \ul z_2) \tens \dotsb \tens J(\ul y_n, \ul z_n)}} \tens J(\ul y_1, \ul z_1) \\
		&\overset{(2)}\iso \int^{\ul y} k_2\bigpars{d\ul y_1 \tens J(\ul y_1, \ul z_1), k_{n'}(d\ul y_2, \dotsc, d\ul y_n)} \tens \bigpars{J(\ul y_2, \ul z_2) \tens \dotsb \tens J(\ul y_n, \ul z_n)} \\
		&\overset{(2)}\iso \int^{\ul y} k_2\Bigpars{d\ul y_1 \tens J(\ul y_1, \ul z_1), k_{n'}(d\ul y_2, \dotsc, d\ul y_n) \tens \bigpars{J(\ul y_2, \ul z_2) \tens \dotsb \tens J(\ul y_n, \ul z_n)}} \\
		&\overset{(3)}\iso \int\limits^{y_1} \int\limits^{y_2, \dotsc, y_n} k_2\Bigpars{d\ul y_1 \tens J(\ul y_1, \ul z_1), k_{n'}(d\ul y_2, \dotsc, d\ul y_n) \tens \bigpars{J(\ul y_2, \ul z_2) \tens \dotsb \tens J(\ul y_n, \ul z_n)}} \\
		&\overset{(2)}\iso \int\limits^{y_1} k_2\Bigpars{d\ul y_1 \tens J(\ul y_1, \ul z_1), \int\limits^{y_2, \dotsc, y_n} k_{n'}(d\ul y_2, \dotsc, d\ul y_n) \tens \bigpars{J(\ul y_2, \ul z_2) \tens \dotsb \tens J(\ul y_n, \ul z_n)}} \\
		&\to \int^{y_1} k_2\Bigpars{d\ul y_1 \tens J(\ul y_1, \ul z_1), k_{n'}\bigpars{\int^{y_2} dy_2 \tens J(dy_2, \ul z_2), \dotsc, \int^{y_n} dy_n \tens J(dy_n, \ul z_n)}} \\
		&\overset{(2)}\iso k_2\Bigpars{\int^{y_1} d\ul y_1 \tens J(\ul y_1, \ul z_1), k_{n'}\bigpars{\int^{y_2} dy_2 \tens J(dy_2, \ul z_2), \dotsc, \int^{y_n} dy_n \tens J(dy_n, \ul z_n)}} \\
		&\overset{(1)}\iso k_n\bigpars{\int^{y_1} dy_1 \tens J(y_1, \ul z_1), \dotsc, \int^{y_n} dy_n \tens J(y_n, \ul z_n)},
	\end{align*}
	where the isomorphisms are induced by (1) the restriction $k_n \iso k_2 \of (\id \tens k_{n'})$ of the invertible associator of $K$, (2) the assumption that $k_2$ preserves weighted colimits in both variables and (3) Fubini's theorem for coends. The unlabelled map is the canonical map \eqref{equation:761211} for $n'$, and we conclude that \eqref{equation:761211} is an isomorphism for $n$ whenever it is one for $n'$. Since all monoidal $\V$-categories are assumed to be normal, we have $k_1 = \id$, so that \eqref{equation:761211} is the identity for $n = 1$; that it is an isomorphism for all $n \geq 1$ follows by induction.
\end{proof}

\subsection{Double categories}
	Applying the main theorem to the `free strict double category'-monad $D$ on the equipment $\psProf{\GG_1}$ of $\GG_1$-indexed categories we find that the forgetful functor $\map{U^D}{\rcDblProf}{\psProf{\GG_1}}$ lifts all weighted colimits, where $\rcDblProf$ is the pseudo double category of right colax double profunctors, see \exref{example:right colax double profunctors}. Furthermore, the equipment $\psProf{\GG_1}$ has strong double comma objects by \propref{internal double comma objects are strong} so that we can apply \corref{right colax forgetful functor lifts all pointwise weighted colimits} as well, obtaining the following.
\begin{proposition} \label{application for double categories}
	The forgetful functor $\map{U^D}{\rcDblProf}{\psProf{\GG_1}}$ lifts all pointwise weighted colimits. In particular, it lifts pointwise weighted left Kan extensions $\lan_j d$ (\defref{definition:pointwise algebraic weighted colimits}) where $d$ and $j$ are colax double functors such that $B(j, \id)$ is a right pseudo double profunctor.
\end{proposition}

	The second statement here, about pointwise weighted left Kan extensions, can be compared (but seems unlikely to be equivalent) to the main result of \cite{Grandis-Pare07}. To discuss it we need the notion of `double colimits' in double categories, that were introduced in \cite[Section 4]{Grandis-Pare99}. First, a \emph{double cocone} $x$ for a colax double functor $\map pAM$ consists of an object $x$ in $M$ equipped with a vertical map $\map{x_a}{pa}x$ for every object $a$ of $A$ and a cell
	\begin{displaymath}
		\begin{tikzpicture}
			\matrix(m)[math175em]{ pa & pb \\ x & x \\ };
			\path[map]	(m-1-1) edge[barred] node[above] {$pj$} (m-1-2)
													edge node[left] {$x_a$} (m-2-1)
									(m-1-2) edge node[right] {$x_b$} (m-2-2);
			\path	(m-2-1) edge[eq] (m-2-2);
			\path[map, transform canvas={shift=($(m-1-2)!0.5!(m-2-2)$)}]	(m-1-1) edge[cell] node[right] {$x_j$} (m-2-1);
		\end{tikzpicture}
	\end{displaymath}
	for every horizontal morphism $\hmap jab$ in $A$, satisfying the following naturality conditions:
	\begin{itemize}[label=-]
		\item	$x_b \of pf = x_a$ for every vertical cell $\map fab$ in $A$; 
		\item $x_k \of pu = x_j$ for every cell $\cell ujk$ in $A$;
		\item $(x_j \hc x_k) \of p_\hc = x_{j \hc k}$ for composable horizontal morphisms $\hmap jab$ and $\hmap kbc$, where $\cell{p_\hc}{p(j \hc k)}{pj \hc pk}$ is the compositor of $p$.
	\end{itemize}
	The \emph{double colimit} of $\map pAM$ is a double cocone $\colim p = c$ for $p$ that is universal in the sense that, for any other double cocone $x$, there exists a unique vertical map $\map{x'}cx$ such that $x' \of c_a = x_a$ and $1_{x'} \of c_j = x_j$, for every object $a$ and horizontal morphism $\hmap jab$, where $1_{x'}$ denotes the horizontal identity for $x'$. Dually one can define double limits; the double comma objects that we considered in \chapref{chapter:weighted colimits} are examples of double limits.
	
	The universality of double colimits implies that any transformation $\nat\xi pq$ (\defref{definition:transformation}) of colax functors $p$ and $\map qAM$ induces a canonical vertical morphism $\map{\colim\xi}{\colim p}{\colim q}$ in $M$, whenever these colimits exist. Remember that such transformations $\xi$ consist of natural vertical maps $\map{\xi_a}{pa}{qa}$ for all $a$ in $A$, as well as natural cells $\cell{\xi_j}{pj}{qj}$ for every $\hmap jab$ in $A$. Symmetrically one can consider \emph{horizontal transformations} $\nat\chi pq$, which consist of horizontal maps $\hmap{\chi_a}{pa}{qa}$ for every object $a$ in $A$ as well as cells
	\begin{displaymath}
		\begin{tikzpicture}
			\matrix(m)[math175em]{ pa & qa \\ pb & qb \\ };
			\path[map]	(m-1-1) edge[barred] node[above] {$\chi_a$} (m-1-2)
													edge node[left] {$pf$} (m-2-1)
									(m-1-2) edge node[right] {$qf$} (m-2-2)
									(m-2-1) edge[barred] node[below] {$\chi_b$} (m-2-2);
			\path[map, transform canvas={shift=($(m-1-2)!0.5!(m-2-2)$)}]	(m-1-1) edge[cell] node[right] {$\chi_f$} (m-2-1);
		\end{tikzpicture}
	\end{displaymath}
	for every vertical morphism $\map fab$ in $A$, satisfying certain coherence conditions. Taking double colimits need not be functorial with respect to these horizontal transformations; consequently Grandis and Par\'e consider in \cite[Section 4]{Grandis-Pare07} cocomplete double categories $M$ equipped with a \emph{colax functorial choice} of double colimits. Roughly speaking, such a choice picks out a horizontal morphism $\hmap{\colim \chi}{\colim p}{\colim q}$ in $M$ for every horizontal transformation $\nat \chi pq$, where $p$ and $\map q AM$ are colax double functors, in a coherent way.
	
	After recalling Grandis and Par\'e's notion of `pointwise normal left Kan extension' of a pair of colax double functors, that we discussed following \defref{definition:pointwise algebraic weighted colimits}, the colax variant of the main theorem of \cite{Grandis-Pare07} can now be stated, in our terms, as follows.
\begin{theorem}[Grandis and Par\'e]
	Let $M$ be a normal pseudo double category in which every vertical isomorphism has a companion. The following are equivalent:
	\begin{itemize}[label=-]
		\item for any pair of colax double functors $\map jAB$ and $\map dAM$ the pointwise normal left Kan extension $\map{\lan_j d}BM$ exists whenever the companion $B(j, \id)$ is a right pseudo double profunctor;
		\item $M$ is an equipment that admits a colax functorial choice of double colimits.
	\end{itemize}
\end{theorem}
	The author does not know how to compare this result with the second assertion of \propref{application for double categories} above. In particular the respective pointwise left Kan extensions, that are considered in the two different results, belong to, what seem to be, very different double categories. However, the condition of the companion $B(j, \id)$ being a right pseudo double profunctor, in both results, is striking.

\subsection{Comparison to the work of Melli\`es and Tabareau}
	While writing up this thesis the author discovered the unpublished paper \cite{Mellies-Tabareau08} by Melli\`es and Tabareau, that is also devoted to producing algebraic left Kan extensions. Fortunately their approach, which we discuss below, is very different to the one presented here, leading to many differences in the results obtained.
	
	Whereas we assume and use a closed structure on our equipments, their approach is based on the notion of `Yoneda situations', as follows. In our terms, a morphism $\map yM{\ps M}$ in an equipment $\K$ is a \emph{Yoneda situation} if
\begin{itemize}[label=-]
	\item $y$ is fully faithful, that is the unit $\cell{\ls y\eta_y}{U_M}{\ps M(y, y)}$ of the companion-conjoint adjunction (\propref{companion-conjoint adjunction}) is invertible in $H(\K)$;
	\item for each object $A$ the composition
	\begin{displaymath}
		V(\K)(A, \ps M) \xrar\rho H(\K)(\ps M, A) \xrar{\ps M(y, \id) \hc \dash} H(\K)(A, M),
	\end{displaymath}
	where $\rho$ is given by \propref{companion and conjoint pseudofunctors}, is fully faithful.
\end{itemize}
	In $\K = \Prof$ the Yoneda embedding $\map yM{\ps M}$, with $\ps M = \fun{\op M}\Set$ the category of presheaves on $M$, is a Yoneda situation.
	Propositions 1 and 2 of \cite{Mellies-Tabareau08} imply the following result.
\begin{proposition}[Melli\`es and Tabareau]
	In an equipment $\K$ consider morphisms $j$, $d$, $y$ and $k$ as below, such that $y$ is a Yoneda situation. If
\begin{itemize}[label=-]
	\item $y$ has a left adjoint $\map{\colim}{\ps M}M$ and
	\item there exists an isomorphism $M(\id, d) \hc B(j, \id) \iso \ps M(y, \id) \hc \ps M(\id, k)$ in $H(\K)$,
\end{itemize}
	then the composite $\map{\mathrm{\colim} \of k}BM$ is the (ordinary) left Kan extension of $d$ along $j$.
	\begin{displaymath}
		\begin{tikzpicture}
			\matrix(m)[math2em]{A & B \\ M & \\};
			\path[map]	(m-1-1) edge node[above] {$j$} (m-1-2)
													edge node[left] {$d$} (m-2-1);
			\matrix(m)[math2em, xshift=1.5em, yshift=-1.5em]{& B \\ M & \ps M \\};
			\path[map]	(m-1-2) edge[shorten >=3pt] node[right] {$k$} (m-2-2)
									(m-2-1) edge node[below] {$y$} (m-2-2);
		\end{tikzpicture}
	\end{displaymath}
\end{proposition}
	
 In our terms, their main theorem \cite[Theorem 1]{Mellies-Tabareau08} can now be stated as follows. Given a normal pseudomonad $T$ on an equipment $\K$ we denote by $\psAlg{V(T)}$ the $2$-category of pseudo $V(T)$-algebras, pseudo $V(T)$-morphisms and $V(T)$-cells.
\begin{theorem}[Melli\`es and Tabareau]
	Let $T$ be a normal pseudomonad on an equipment $\K$ and consider the situation of the previous proposition in $\K$, such that
	\begin{itemize}[label=-]
		\item	$A$, $B$, $M$ and $\ps M$ are pseudo $T$-algebras;
		\item $d$, $j$, $\colim$ and $k$ are pseudo $T$-morphisms;
		\item the induced lax structures on $B(j, \id)$ and $\ps M(y, \id)$ are right pseudo.
	\end{itemize}
	Then $\colim \of k$, which is the left Kan extension of $d$ along $j$ by the previous proposition, can be given the structure of a pseudo $T$-morphism, which makes it the left Kan extension of $d$ along $j$ in $\psAlg{V(T)}$.
\end{theorem}
	
	Thus the main difference between our main result and that of Melli\`es and Tabareau is that where we assume a closed structure on equipments, Melli\`es and Tabareau assume the existence of Yoneda situations. Moreover, where they work with ordinary left Kan extensions, we consider the more general, but also stronger, notion of weighted colimits. Finally, and perhaps most importantly, our main theorem works in the case of normal monads, while Melli\`es and Tabareau consider only pseudomonads.

	\appendix
	\chapter{Double comma objects} \label{appendix:double comma object proofs}
	This appendix collects some proofs of results concering comma objects that we have used. Each of them long, the author feared they would have been distracting when part of the main text.
	
	Recall that the double comma object $J \slash f$ of a pair of morphisms $\hmap JAB$ and $\map fCB$, in a pseudo double category $\K$, consists of a cell $\pi$ as below that satisfies the following universal properties.
\begin{displaymath}
	\begin{tikzpicture}[baseline]
		\matrix(m)[math175em]{J \slash f & J \slash f \\ & C \\ A & B \\};
		\path[map]	(m-1-1) edge node[left] {$\pi_A$} (m-3-1)
								(m-1-2)	edge node[right] {$\pi_C$} (m-2-2)
								(m-2-2) edge node[right] {$f$} (m-3-2)
								(m-3-1) edge[barred] node[below] {$J$} (m-3-2);
		\path[transform canvas={shift={($(m-2-2)!0.5!(m-3-2)$)}}] (m-1-1) edge[cell] node[right] {$\pi$} (m-2-1);
		\path				(m-1-1) edge[eq] (m-1-2);
	\end{tikzpicture}
	\quad\quad\qquad\qquad\begin{tikzpicture}[baseline]
		\matrix(m)[math175em]{X & X \\ \phantom C & C \\ A & B \\};
		\path[map]	(m-1-1) edge node[left] {$\phi_A$} (m-3-1)
								(m-1-2)	edge node[right] {$\phi_C$} (m-2-2)
								(m-2-2) edge node[right] {$f$} (m-3-2)
								(m-3-1) edge[barred] node[below] {$J$} (m-3-2);
		\path[transform canvas={shift={($(m-2-2)!0.5!(m-3-2)$)}}] (m-1-1) edge[cell] node[right] {$\phi$} (m-2-1);
		\path				(m-1-1) edge[eq] (m-1-2);
	\end{tikzpicture}
\end{displaymath}
	Firstly, any other cell $\phi$ as above factors uniquely through $\pi$ as a map $\map{\phi'}X{J \slash f}$, such that $\pi_A \of \phi' = \phi_A$, $\pi_C \of \phi' = \phi_C$ and $\pi \of U_{\phi'}$ = $\phi$, and, secondly, if the cell $\psi$ below likewise factors as $\map{\psi'}Y{J \slash f}$, then any pair $(\xi_A, \xi_C)$ of cells satisfying the identity below corresponds to a unique cell $\xi'$, as on the right below, such that $U_{\pi_A} \of \xi' = \xi_A$ and $U_{\pi_C} \of \xi' = \xi_C$.
\begin{displaymath}	
	\begin{tikzpicture}[textbaseline]
		\matrix(m)[math175em]{X & Y & Y \\ \phantom Y & & C \\ A & A & B \\};
		\path[map]	(m-1-1) edge[barred] node[above] {$K$} (m-1-2)
												edge node[left] {$\phi_A$} (m-3-1)
								(m-1-2) edge node[right] {$\psi_A$} (m-3-2)
								(m-1-3) edge node[right] {$\psi_C$} (m-2-3)
								(m-2-3) edge node[right] {$f$} (m-3-3)
								(m-3-2) edge[barred] node[below] {$J$} (m-3-3);
		\path				(m-1-2) edge[eq] (m-1-3)
								(m-3-1) edge[eq] (m-3-2);
		\path[transform canvas={shift={($(m-2-3)!0.5!(m-3-2)$)}}]	(m-1-1) edge[cell] node[right] {$\xi_A$} (m-2-1);
		\path[transform canvas={shift={($(m-2-3)!0.5!(m-3-2)$)}, xshift=2pt}]	(m-1-2) edge[cell] node[right] {$\psi$} (m-2-2);
		\end{tikzpicture}
	= \begin{tikzpicture}[textbaseline]
		\matrix(m)[math175em]{X & X & Y \\ \phantom Y & C & C \\ A & B & B \\};
		\path[map]	(m-1-1) edge node[left] {$\phi_A$} (m-3-1)
								(m-1-2) edge[barred] node[above] {$K$} (m-1-3)
												edge node[left] {$\phi_C$} (m-2-2)
								(m-1-3) edge node[right] {$\psi_C$} (m-2-3)
								(m-2-2) edge node[right] {$f$} (m-3-2)
								(m-2-3) edge node[right] {$f$} (m-3-3)
								(m-3-1) edge[barred] node[below] {$J$} (m-3-2);
		\path				(m-1-1) edge[eq] (m-1-2)
								(m-2-2) edge[eq] (m-2-3)
								(m-3-2) edge[eq] (m-3-3);
		\path[transform canvas={shift={($(m-2-3)!0.5!(m-3-2)$)}}]	(m-1-1) edge[cell] node[right] {$\phi$} (m-2-1);
		\path[transform canvas={shift={($(m-2-2)!0.5!(m-2-3)$)}}]	(m-1-2) edge[cell] node[right] {$\xi_C$} (m-2-2);
	\end{tikzpicture}
	\qquad\qquad\begin{tikzpicture}[baseline]
			\matrix(m)[math175em]{X & Y \\ J \slash f & J \slash f \\};
			\path[map]	(m-1-1) edge[barred] node[above] {$K$} (m-1-2)
													edge node[left] {$\phi'$} (m-2-1)
									(m-1-2) edge node[right] {$\psi'$} (m-2-2);
			\path				(m-2-1) edge[eq] (m-2-2);
			\path[transform canvas={shift=($(m-1-2)!0.5!(m-2-2)$)}] (m-1-1) edge[cell] node[right] {$\xi'$} (m-2-1);
		\end{tikzpicture}
\end{displaymath}
	
	Informally, it is useful to think of the cells $\phi = (\phi_A, \phi, \phi_C)$ and $\psi = (\psi_A, \psi, \psi_C)$ above as objects, and of a pair $\xi = (\xi_A, \xi_C)$ as a morphism $\phi \to \psi$. Then the universality of $\pi$ gives assignments $\phi \mapsto \phi'$ and $\xi \mapsto \xi'$, the latter with $L\xi' = \phi'$ and $R\xi' = \psi'$, that map $\phi$ and $\xi$ to their factorisations through the universal cell $\pi$. The uniqueness of these factorisations implies that the assignments $\phi \mapsto \phi'$ and $\xi \mapsto \xi'$ are functorial in the following sense, where $\map hSX$ is a vertical map and $\zeta$ and $\theta$ are cells of the following form.
\begin{displaymath}
	\begin{tikzpicture}[baseline]
		\matrix(m)[math175em]{S & T \\ X & Y \\};
		\path[map]	(m-1-1) edge[barred] node[above] {$L$} (m-1-2)
												edge node[left] {$h$} (m-2-1)
								(m-1-2) edge node[right] {$k$} (m-2-2)
								(m-2-1) edge[barred] node[below] {$K$} (m-2-2);
		\path[transform canvas={shift=($(m-1-2)!0.5!(m-2-2)$)}] (m-1-1) edge[cell] node[right] {$\zeta$} (m-2-1);
	\end{tikzpicture}
	\qquad\qquad\qquad\qquad\begin{tikzpicture}[baseline]
		\matrix(m)[math175em]{S & T \\ X & X \\};
		\path[map]	(m-1-1) edge[barred] node[above] {$L$} (m-1-2)
												edge node[left] {$h$} (m-2-1)
								(m-1-2) edge node[right] {$l$} (m-2-2);
		\path				(m-2-1) edge[eq] (m-2-2);
		\path[transform canvas={shift=($(m-1-2)!0.5!(m-2-2)$)}] (m-1-1) edge[cell] node[right] {$\theta$} (m-2-1);
	\end{tikzpicture}
\end{displaymath}
	\begin{itemize}[label=-]
		\item If $(\phi_A, \phi, \phi_C) \mapsto \phi'$ then $\brks{\phi \of h = (\phi_A \of h, \phi \of U_h, \phi_C \of h)} \mapsto \phi' \of h$;
		\item if $(\xi_A, \xi_C) \mapsto \xi'$ then $\brks{\map{(\xi_A \of \zeta, \xi_C \of \zeta)}{\phi \of h}{\psi \of k}} \mapsto \xi' \of \zeta$;
		\item if $(\phi_A, \phi, \phi_C) \mapsto \phi'$ then $\brks{\map{(U_{\phi_A} \of \theta, U_{\phi_C} \of \theta)}{\phi \of h}{\phi \of l}} \mapsto U_{\phi'} \of  \theta$;
		\item if $\brks{\map{(\rho_A, \rho_C)}\phi\psi}\mapsto \rho'$ and $\brks{\map{(\sigma_A, \sigma_C)}\psi\chi}\mapsto \sigma'$ then
		\begin{displaymath}
		\brks{\map{(\rho_A \hc \sigma_A, \rho_C \hc \sigma_C)}\phi\chi} \mapsto \rho' \hc \sigma'.
		\end{displaymath}
	\end{itemize}
	This functoriality will be useful, especially in the proof of the second proposition below.
	
	We first turn to the proof of \propref{internal double comma objects}, which describes double comma objects in the equipment $\inProf\E$ of internal profunctors in a category $\E$ with finite limits. Given an internal profunctor $\hmap JAB$ and an internal functor $\map fCB$ we have already constructed an $\E$-span $J \slash f \to (J \slash f)_0 \times (J \slash f)_0$, in the discussion preceeding \propref{internal double comma objects}, which asserts that this span underlies the internal category that is the double comma object of $J$ and $f$. We shall first recall the construction of $(J \slash f)_0$ and $J \slash f$. The former we took to be the wide pullback $(J \slash f)_0 = A_0 \times_{A_0} J \times_{B_0} C_0$ of the diagram $A_0 \xrar{\id} A_0 \xlar{d_0} J \xrar{d_1} B_0 \xlar{f_0} C_0$, while for the latter we chose the pullback
	\begin{displaymath}
		\begin{tikzpicture}
			\matrix(m)[math2em, column sep=0.75em]{J \slash f & (J \slash f)_0 \times_{C_0} C \\ A \times_{A_0} (J \slash f)_0 & (J \slash f)_0, \\};
			\path[map]	(m-1-1)	edge (m-1-2)
													edge (m-2-1)
									(m-1-2) edge node[right] {$r$} (m-2-2)
									(m-2-1) edge node[below] {$l$} (m-2-2);
			\coordinate (hook) at ($(m-1-1)+(0.4,-0.4)$);
			\draw (hook)+(0,0.17) -- (hook) -- +(-0.17,0);
		\end{tikzpicture}
	\end{displaymath}
	where the corners are pullbacks of $(J \slash f)_0 \to C_0 \xlar{d_0} C$ and $A \xrar{d_1} A_0 \leftarrow (J \slash f)_0$ respectively, and the maps $r$ and $l$ are the compositions
	\begin{flalign*}
		&& (J \slash f)_0 \times_{C_0} C &\to A_0 \times_{A_0} J \times_{B_0} B \times_{C_0} C_0 \xrar{\id \times r \times \id} (J \slash f)_0 & \\
		\text{and} && A \times_{A_0} (J \slash f)_0 &\iso A_0 \times_{A_0} A \times_{A_0} J \times_{B_0} C_0 \xrar{\id \times l \times \id} (J \slash f)_0. & 
	\end{flalign*}
	Here the unlabelled map is given by applying $f$ to $C$ after using $C_0 \times_{C_0} C \iso C \times_{C_0} C_0$, while the isomorphism is induced by $A \times_{A_0} A_0 \iso A_0 \times_{A_0} A$ and the maps $l$ and $r$ denote the actions of $A$ and $B$ on $J$. The source and target maps \mbox{$\map{d = (d_0, d_1)}{J \slash f}{(J \slash f)_0 \times (J \slash f)_0}$} are given by the projections
	\begin{flalign*}
		&& d_0 &= \bigbrks{J \slash f \to (J \slash f)_0 \times_{C_0} C \to (J \slash f)_0} & \\
		\text{and} && d_1 & = \bigbrks{J \slash f \to A \times_{A_0} (J \slash f)_0 \to (J \slash f)_0}. &
	\end{flalign*}
\begin{internalcommaobjects}
	The $\E$-span $\map d{J \slash f}{(J \slash f)_0 \times (J \slash f)_0}$ above can be given the structure of a category internal to $\E$. Moreover, $J \slash f$ is made into the double comma object of $J$ and $f$ in $\inProf\E$ by the internal functors $\map{\pi_A}{J \slash f}A$, given by the projections
	\begin{displaymath}
		\map{(\pi_A)_0}{(J \slash f)_0}{A_0} \qquad \text{and} \qquad \pi_A = \bigbrks{J \slash f \to A \times_{A_0} (J \slash f)_0 \to A},
	\end{displaymath}
	and $\map{\pi_C}{J \slash f}C$, given similarly, together with the cell
	\begin{displaymath}
		\begin{tikzpicture}
			\matrix(m)[math175em]{J \slash f & J \slash f \\ & C \\ A & B \\};
			\path[map]	(m-1-1) edge node[left] {$\pi_A$} (m-3-1)
									(m-1-2)	edge node[right] {$\pi_C$} (m-2-2)
									(m-2-2) edge node[right] {$f$} (m-3-2)
									(m-3-1) edge[barred] node[below] {$J$} (m-3-2);
			\path[transform canvas={shift={($(m-2-2)!0.5!(m-3-2)$)}}] (m-1-1) edge[cell] node[right] {$\pi$} (m-2-1);
			\path				(m-1-1) edge[eq] (m-1-2);
		\end{tikzpicture}
	\end{displaymath}
	of internal profunctors that, under the correspondence of \propref{internal transformations}, is given by the projection $(J \slash f)_0 \to J$.
\end{internalcommaobjects}
\begin{proof}
	To give the multiplication $\map m{J \slash f \times_{(J \slash f)_0} J \slash f}{J \slash f}$, notice that the source $W = J \slash f \times_{(J \slash f)_0} J \slash f$ of $m$ forms the wide pullback of the diagram
	\begin{equation} \label{equation:402213}
		\begin{split}
		\begin{tikzpicture}
			\matrix(m)[math175em, column sep=-0.95em]
			{	(J \slash f)_0 \times_{C_0} C \nc \nc A \times_{A_0} (J \slash f)_0 \nc \nc (J \slash f)_0 \times_{C_0} C \nc \nc A \times_{A_0} (J \slash f)_0 \\
			\nc (J \slash f)_0 \nc \nc (J \slash f)_0 \nc \nc (J \slash f)_0, \nc \\}; 
			\path[map]	(m-1-1) edge node[below left] {$r$} (m-2-2)
									(m-1-3) edge node[below right] {$l$} (m-2-2)
													edge (m-2-4)
									(m-1-5) edge (m-2-4)
													edge node[below left] {$r$} (m-2-6)
									(m-1-7) edge node[below right] {$l$} (m-2-6);
			\draw[font=\scriptsize]	(m-1-1) node[above=3pt] {(1)}
															(m-1-3) node[above=3pt] {(2)}
															(m-1-5) node[above=3pt] {(3)}
															(m-1-7) node[above=3pt] {(4)};
		\end{tikzpicture}
		\end{split}
	\end{equation}
	where the first factor of $W$ is the pullback of the objects labelled (1) and (2), the second factor is the pullback of those labelled (3) and (4), and where the unlabelled maps are projections. In the case of $\E = \Set$, $W$ consists of quadruples of pairs of maps, one in each of the sets on the top row above, such that the pairs in (1) and (2), as well as those in (3) and (4), form commutative squares \eqref{diagram:morphism in a double comma category}, and such that the bottom map of the first square coincides with the top map of the second. Let $m$ be induced by the compositions
	\begin{flalign*}
		&& W &\to \overset{\raisebox{0.75ex}{\scriptsize{(2)}}}A \times_{A_0} \overbrace{A \times_{A_0} (J \slash f)_0}^{\scriptsize{(4)}} \xrar{m_A \times \id} A \times_{A_0} (J \slash f)_0 & \\
		\text{and} && W &\to \overbrace{(J \slash f)_0 \times_{C_0} C}^{\scriptsize{(1)}} \times_{C_0} \overset{\raisebox{0.75ex}{\scriptsize{(3)}}}C \xrar{\id \times m_C} (J \slash f)_0 \times_{C_0} C, &
	\end{flalign*}
	where the unlabelled maps are induced by projections from $W$ onto (factors of) the pullbacks in the top row of \eqref{equation:402213}, as indicated by the labels. When $\E = \Set$, this simply composes the two pairs of maps making up the sides of two adjacent commutative squares \eqref{diagram:morphism in a double comma category}. That the maps above indeed induce a map into the pullback $J \slash f$ follows from
	\begin{align*}
		\bigbrks{&W \to \overset{\raisebox{0.75ex}{\scriptsize{(2)}}}A \times_{A_0} \overbrace{A \times_{A_0} (J \slash f)_0}^{\scriptsize{(4)}} \xrar{m_A \times \id} A \times_{A_0} (J \slash f)_0 \xrar l (J \slash f)_0} \\
		=\bigbrks{&W \to \overset{\raisebox{0.75ex}{\scriptsize{(2)}}}A \times_{A_0} \overbrace{(J \slash f)_0 \times_{C_0} C}^{\scriptsize{(3)}} \xrar{\id \times r} A \times_{A_0} (J \slash f)_0 \xrar l (J \slash f)_0} \\
		=\bigbrks{&W \to \overbrace{A \times_{A_0} (J \slash f)_0}^{\scriptsize{(2)}} \times_{C_0} \overset{\raisebox{0.75ex}{\scriptsize{(3)}}}C \xrar{l \times \id} (J \slash f)_0 \times_{C_0} C \xrar r (J \slash f)_0} \\ 
		=\bigbrks{&W \to \overbrace{(J \slash f)_0 \times_{C_0} C}^{\scriptsize{(1)}} \times_{C_0} \overset{\raisebox{0.75ex}{\scriptsize{(3)}}}C \xrar{\id \times m_C} (J \slash f)_0 \times_{C_0} C \xrar r (J \slash f)_0}.
	\end{align*}
	Here the first equality follows from the fact that we can replace $m_A \times \id$ by $\id \times l$ (by associativity of the action of $A$ on $J$) and that precomposing $l$ with the projection onto (4) coincides with precomposing $r$ with the projection onto (3), by definition of $W$. Likewise the projection of $W$ onto the factor $(J \slash f)_0$ of (3) equals the projection onto the same factor of (2) and, together with the fact that $r$ and $l$ commute, the second equality follows. Finally the third equality is symmetric to the first.
	
	That the multiplication $\map m{J \slash f \times_{(J \slash f)_0} J \slash f}{J \slash f}$ thus defined is compatible with the source and target maps is clear; that it is associative follows from the associativity of $m_A$ and $m_C$. Moreover one readily verifies that the unit \mbox{$\map e{(J \slash f)_0}{J \slash f}$} can be taken to be induced by the compositions
	\begin{flalign*}
		&&(J \slash f)_0 &\iso (J \slash f)_0 \times_{C_0} C_0 \xrar{\id \times e_C} (J \slash f)_0 \times_{C_0} C & \\
		\text{and} &&(J \slash f)_0 &\iso A_0 \times_{A_0} (J \slash f)_0 \xrar{e_A \times \id} A \times_{A_0} (J \slash f)_0, &
	\end{flalign*}
	which completes the definition of the structure making $J \slash f$ into a category internal to $\E$. Checking that the projections $\map{\pi_A}{J \slash f}A$ and $\map{\pi_C}{J \slash f}C$, as well as the cell $\nat\pi{U_{J \slash f}}J$, are compatible with this structure is straightforward, which leaves us to prove that $\pi$ satisfies the universal properties that define $J \slash f$ as the double comma object.
	
	So let us consider a second internal transformation $\phi$ as on the left below, given by a map $\map\phi{X_0}J$; we have to show that there exists a unique factorisation $\map{\phi'}X{J \slash f}$ such that $\pi_A \of \phi' = \phi_A$, $\pi_C \of \phi' = \phi_C$ and $\pi \of U_{\phi'} = \phi$.
	\begin{displaymath}
		\begin{tikzpicture}[baseline]
			\matrix(m)[math175em]{X & X \\ \phantom C & C \\ A & B \\};
			\path[map]	(m-1-1) edge node[left] {$\phi_A$} (m-3-1)
									(m-1-2)	edge node[right] {$\phi_C$} (m-2-2)
									(m-2-2) edge node[right] {$f$} (m-3-2)
									(m-3-1) edge[barred] node[below] {$J$} (m-3-2);
			\path[transform canvas={shift={($(m-2-2)!0.5!(m-3-2)$)}}] (m-1-1) edge[cell] node[right] {$\phi$} (m-2-1);
			\path				(m-1-1) edge[eq] (m-1-2);
		\end{tikzpicture}
		\qquad\qquad\qquad\qquad\begin{tikzpicture}[baseline]
			\matrix(m)[math175em]{Y & Y \\ & C \\ A & B \\};
			\path[map]	(m-1-1) edge node[left] {$\psi_A$} (m-3-1)
									(m-1-2)	edge node[right] {$\psi_C$} (m-2-2)
									(m-2-2) edge node[right] {$f$} (m-3-2)
									(m-3-1) edge[barred] node[below] {$J$} (m-3-2);
			\path[transform canvas={shift={($(m-2-2)!0.5!(m-3-2)$)}}] (m-1-1) edge[cell] node[right] {$\psi$} (m-2-1);
			\path				(m-1-1) edge[eq] (m-1-2);
		\end{tikzpicture}
	\end{displaymath}
	Using the fact that, under \propref{internal transformations}, the composite $\map{\pi \of U_{\phi'}}XJ$ corresponds to $X_0 \xrar{\phi'_0} (J \slash f)_0 \xrar\pi J$, it is readily checked that these identities are satisfied by taking
	\begin{flalign} \label{equation:523838}
		&&\map{\phi'_0 &= \bigpars{(\phi_A)_0, \phi, (\phi_C)_0}}{X_0}{(J \slash f)_0} & \notag \\
		\text{and} &&\map{\phi' &= \bigpars{(\phi'_0 \of d_0, \phi_C), (\phi_A, \phi'_0 \of d_1)}}X{J \slash f}. &
	\end{flalign}
	Here the factors $(\phi_A)_0$ in $\phi'_0$ and $\phi_A$ in $\phi'$ are forced by the identity $\pi_A \of \phi' = \phi_A$, while the factors $(\phi_C)_0$ and $\phi_C$ are forced by $\pi_C \of \phi' = \phi_C$. Similarly the factor $\phi$ in $\phi'_0$ is forced by $\pi \of U_{\phi'} = \phi$, while compatibility with the source and target maps forces the factors $\phi'_0 \of d_i$ in $\phi'$. We conclude that $\phi'$ is uniquely determined.
	
	Finally, to prove that the 2-dimensional universal property holds as well, consider a second transformation $\psi$, as on the right above, and suppose it factors through $\pi$ as $\map{\psi'}Y{J \slash f}$. Given cells
	\begin{displaymath}
		\begin{tikzpicture}[baseline]
			\matrix(m)[math175em]{X & Y \\ A & A \\};
			\path[map]	(m-1-1) edge[barred] node[above] {$K$} (m-1-2)
													edge node[left] {$\phi_A$} (m-2-1)
									(m-1-2) edge node[right] {$\psi_A$} (m-2-2);
			\path				(m-2-1) edge[eq] (m-2-2);
			\path[transform canvas={shift=($(m-1-2)!0.5!(m-2-2)$)}] (m-1-1) edge[cell] node[right] {$\xi_A$} (m-2-1);
		\end{tikzpicture}
		\qquad\qquad\qquad\qquad\begin{tikzpicture}[baseline]
			\matrix(m)[math175em]{X & Y \\ C & C \\};
			\path[map]	(m-1-1) edge[barred] node[above] {$K$} (m-1-2)
													edge node[left] {$\phi_C$} (m-2-1)
									(m-1-2) edge node[right] {$\psi_C$} (m-2-2);
			\path				(m-2-1) edge[eq] (m-2-2);
			\path[transform canvas={shift=($(m-1-2)!0.5!(m-2-2)$)}] (m-1-1) edge[cell] node[right] {$\xi_C$} (m-2-1);
		\end{tikzpicture}
	\end{displaymath}
	with $\xi_A \hc \psi = \phi \hc (U_f \of \xi_C)$, it is easy to see that the unique cell $\nat{\xi'}K{J \slash f}$ such that $U_{\pi_A} \of \xi' = \xi_A$ and $U_{\pi_C} \of \xi' = \xi_C$, is given by $\xi' = \bigpars{(\phi'_0 \of d_0, \xi_C), (\xi_A, \psi'_0 \of d_1)}$.
\end{proof}

	The second proposition whose proof we have postponed is \propref{lifting double comma objects}. Given a right suitable normal monad $T$ on an equipment $\K$, it concerns the forgetful functor $\map{U^T}{\rcProm T}\K$, where $\rcProm T$ is the pseudo double category of colax $T$-algebras, colax $T$-morphisms, right colax $T$-promorphisms and $T$-cells, as defined in \secref{section:algebraic promorphisms}.
\begin{liftingcommaobjects}
	Consider a right suitable normal monad $T$ on an equipment $\K$. The forgetful functor $\map{U^T}{\rcProm T}\K$ lifts double comma objects $J \slash f$ where $\hmap JAB$ is a right pseudopromorphism and $\map fCB$ is a pseudomorphism. The projections of $J \slash f$ are strict $T$-morphisms, and $J \slash f$ is a pseudoalgebra whenever both $A$ and $C$ are.
\end{liftingcommaobjects}
\begin{proof}
	For a right pseudopromorphism $\hmap JAB$ and a pseudomorphism \mbox{$\map fCB$} suppose that the double comma object $J \slash f$, as on the left below, exists in $\K$.
	\begin{displaymath}
		\begin{tikzpicture}[baseline]
			\matrix(m)[math175em]{J \slash f & J \slash f \\ & C \\ A & B \\};
			\path[map]	(m-1-1) edge node[left] {$\pi_A$} (m-3-1)
									(m-1-2)	edge node[right] {$\pi_C$} (m-2-2)
									(m-2-2) edge node[right] {$f$} (m-3-2)
									(m-3-1) edge[barred] node[below] {$J$} (m-3-2);
			\path[transform canvas={shift={($(m-2-2)!0.5!(m-3-2)$)}}] (m-1-1) edge[cell] node[right] {$\pi$} (m-2-1);
			\path				(m-1-1) edge[eq] (m-1-2);
		\end{tikzpicture}
		\qquad\qquad\qquad\qquad\begin{tikzpicture}[baseline]
			\matrix(m)[math175em]{TA & TB \\ A & B \\};
			\path[map]  (m-1-1) edge[barred] node[above] {$TJ$} (m-1-2)
													edge node[left] {$a$} (m-2-1)
									(m-1-2) edge node[right] {$b$} (m-2-2)
									(m-2-1) edge[barred] node[below] {$J$} (m-2-2);
			\path[transform canvas={shift={($(m-1-2)!(0,0)!(m-2-2)$)}}] (m-1-1) edge[cell] node[right] {$\tilde J$} (m-2-1);			
		\end{tikzpicture}
	\end{displaymath}
	Remember that the right colax algebra structure on $J$ is given by a horizontal cell $\cell{\bar J}{J \hc B(\id, b)}{A(\id, A) \hc TJ}$ (\defref{definition:right colax promorphisms}) which, because $J$ is a right pseudopromorphism, has an inverse $\inv{\bar J}$. Under \propref{left and right structure cells} this inverse corresponds to a cell $\tilde J$ above that makes $J$ into a lax promorphism (\defref{definition:lax promorphism}). Likewise the inverse of the vertical cell $\cell{\bar f}{f \of c}{b \of Tf}$, that makes $f$ into a colax morphism, gives $f$ the structure of a lax morphism; we will write $\tilde f = \inv{\bar f}$ as well.
	
	We will show that the algebra structures of $A$ and $C$ induce an algebra structure on $J \slash f$, that this makes $\pi_A$ and $\pi_B$ into strict $T$-morphisms and $\pi$ into a $T$-cell, and that $\pi$ satisfies the universal properties making $J \slash f$ into a double comma object in $\rcProm T$. As structure map $\map r{T J \slash f}{J \slash f}$ we take the factorisation of the composition
	\begin{displaymath}
		\begin{tikzpicture}
      \matrix(m)[math175em]
      { TJ \slash f & TJ \slash f & TJ \slash f \\
				& TC & TC \\
				TA & TB & C \\
				A & B & B \\ };
			\path[map]	(m-1-1) edge node[left] {$T\pi_A$} (m-3-1)
									(m-1-2) edge node[right] {$T\pi_C$} (m-2-2)
									(m-1-3) edge node[right] {$T\pi_C$} (m-2-3)
									(m-2-2) edge node[right] {$Tf$} (m-3-2)
									(m-2-3) edge node[right] {$c$} (m-3-3)
									(m-3-1) edge[barred] node[above] {$TJ$} (m-3-2)
													edge node[left] {$a$} (m-4-1)
									(m-3-2) edge node[right] {$b$} (m-4-2)
									(m-3-3) edge node[right] {$f$} (m-4-3)
									(m-4-1) edge[barred] node[below] {$J$} (m-4-2);
			\path[eq]		(m-1-1) edge[eq] (m-1-2)
									(m-1-2) edge[eq] (m-1-3)
									(m-2-2) edge[eq] (m-2-3)
									(m-4-2) edge[eq] (m-4-3);
			\path[transform canvas={shift=($(m-3-2)!0.5!(m-3-3)$)}]	(m-1-1) edge[cell] node[right] {$T\pi$} (m-2-1)
									(m-2-2) edge[cell] node[right] {$\tilde f$} (m-3-2);
			\path[transform canvas={shift=($(m-2-2)!0.5!(m-3-3)$)}] (m-3-1) edge[cell] node[right] {$\tilde J$} (m-4-1);
    \end{tikzpicture}
	\end{displaymath}
	through the cell $\pi$, so that $\pi_A \of r = a \of T\pi_A$ and $\pi_C \of r = c \of T\pi_C$, while $\pi \of U_r$ equals the composite above. Note that the first two identities mean that $\pi_A$ and $\pi_C$ are strict morphisms. That $\pi$ is a $T$-cell (\defref{definition:colax cell}) is shown by the following chain of equalities, where the identity cell right of $\pi$ is that of $\pi_C \of r = c \of T\pi_C$ and that left of $T\pi$ is that of $\pi_A \of r = a \of T\pi_A$. The equalities follow from the definition of $\rho_1$ (see the discussion preceeding \defref{definition:cells of right colax promorphisms}); that $\pi \of U_r$ equals the composite above; the definition of $\rho_1$ and $\tilde f = \inv{\bar f}$; that $\inv{(\bar J)} = \eps_a \hc \tilde J \hc \eta_b$.
	\begin{align*}
		\begin{tikzpicture}[textbaseline, x=0.75cm, y=0.6cm, font=\scriptsize]
			\draw	(0,1) -- (2,1) -- (2,0) -- (0,0) -- (0,3) -- (1,3) -- (1,1)
				(1,2) -- (2,2) -- (2,1)
				(1,3) -- (2,3) -- (2,2);
			\draw[shift={(0,0.5)}]
				(1,0) node {$\bar J$};
			\draw[shift={(0.5,0)}]
				(0,2) node {$\pi$};
			\draw[shift={(0.5,0.5)}]
				(1,1) node {$\rho_1\bar f$}
				(1,2) node {$\rho_1\id$};
		\end{tikzpicture}
		&\quad = \quad\begin{tikzpicture}[textbaseline, x=0.75cm, y=0.6cm, font=\scriptsize]
			\draw (4,2) -- (5,2) -- (5,1) -- (3,1) -- (3,3) -- (4,3) -- (4,1)
				(3,1) -- (1,1) -- (1,0) -- (5,0) -- (5,1)
				(3,3) -- (3,4) -- (2,4) -- (2,2) -- (3,2)
				(2,2) -- (2,1)
				(1,1) -- (1,4)
				(2,3) -- (0,3) -- (0,4) -- (2,4);
			\draw[shift={(0,0.5)}]
				(3,0) node {$\bar J$};
			\draw[shift={(0.5,0)}]
				(1,2) node {$\pi$}
				(2,3) node {$\id$}
				(3,2) node {$\bar f$};
			\draw[shift={(0.5,0.5)}]
				(0,3) node {$\eps_r$}
				(1,3) node {$U_r$}
				(4,1) node {$\eta_b$};
		\end{tikzpicture}
		\quad = \quad\begin{tikzpicture}[textbaseline, x=0.75cm, y=0.6cm, font=\scriptsize]
			\draw (5,2) -- (6,2) -- (6,1) -- (4,1) -- (4,3) -- (5,3) -- (5,1)
				(4,3) -- (3,3) -- (3,1)
				(3,3) -- (3,4) -- (2,4) -- (2,2) -- (3,2)
				(2,2) -- (2,1)
				(2,4) -- (1,4) -- (1,0) -- (6,0) -- (6,1)
				(1,4) -- (0,4) -- (0,3) -- (1,3)
				(1,1) -- (4,1);
			\draw[shift={(0.5,0)}]
				(2,3) node {$T\pi$}
				(3,2) node {$\tilde f$}
				(4,2) node {$\bar f$};
			\draw[shift={(0.5,0.5)}]
				(0,3) node {$\eps_r$}
				(1,2) node {$\id$}
				(2,1) node {$\tilde J$}
				(5,1) node {$\eta_b$}
				(3,0) node {$\bar J$};
		\end{tikzpicture} \\[1em]
		&\quad = \quad\begin{tikzpicture}[textbaseline, x=0.75cm, y=0.6cm, font=\scriptsize]
			\draw	(0,1) -- (0,0) -- (3,0) -- (3,2) -- (0,2) -- (0,1) -- (3,1)
				(1,1) -- (1,4) -- (2,4) -- (2,1)
				(0,2) -- (0,4) -- (1,4);
			\draw[shift={(0.5,0)}]
				(0,3) node {$\rho_1\id$}
				(1,3) node {$T\pi$};
			\draw[shift={(0.5,0.5)}]
				(1,0) node {$\bar J$}
				(0,1) node {$\eps_a$}
				(1,1) node {$\tilde J$}
				(2,1) node {$\eta_b$};
		\end{tikzpicture}
		\quad = \quad\begin{tikzpicture}[textbaseline, x=0.75cm, y=0.6cm, font=\scriptsize]
			\draw	(1,2) -- (2,2) -- (2,4) -- (0,4) -- (0,2) -- (1,2) -- (1,4)
				(0,2) -- (0,1) -- (2,1) -- (2,2)
				(0,1) -- (0,0) -- (2,0) -- (2,1);
			\draw[shift={(0.5,0)}]
				(0,3) node {$\rho_1\id$}
				(1,3) node {$T\pi$};
			\draw[shift={(0,0.5)}]
				(1,0) node {$\bar J$}
				(1,1) node {$\inv{(\bar J)}$};
		\end{tikzpicture}
		\quad = \quad\begin{tikzpicture}[textbaseline, x=0.75cm, y=0.6cm, font=\scriptsize]
			\draw	(1,0) -- (2,0) -- (2,2) -- (0,2) -- (0,0) -- (1,0) -- (1,2);
			\draw[shift={(0.5,0)}]
				(0,1) node {$\rho_1\id$}
				(1,1) node {$T\pi$};
		\end{tikzpicture}
	\end{align*}
	
	Having chosen a structure map we have to give the associator and unitor of $J \slash f$, that is cells
	\begin{displaymath}
		\begin{tikzpicture}[textbaseline]
			\matrix(m)[math175em, row sep=2.25em]{T^2J \slash f & TJ \slash f \\ TJ \slash f & J \slash f \\};
			\path[map]	(m-1-1) edge node[above] {$\mu_{J \slash f}$} (m-1-2)
													edge node[left] {$Tr$} (m-2-1)
									(m-1-2) edge node[right] {$r$} (m-2-2)
									(m-2-1) edge node[below] {$r$} (m-2-2);
			\path	(m-1-2) edge[cell, shorten >= 13pt, shorten <= 13pt] node[below right] {$\rho$} (m-2-1);
		\end{tikzpicture}
		\qquad\qquad\text{and}\qquad\qquad\begin{tikzpicture}[textbaseline]
			\matrix(m)[math175em, row sep=2.25em]{J \slash f & TJ \slash f \\ & J \slash f, \\};
			\path[map]	(m-1-1) edge node[above] {$\eta_{J \slash f}$} (m-1-2)
													edge[bend right=25] node[below left] {$\id$} (m-2-2)
									(m-1-2) edge node[right] {$r$} (m-2-2);
			\path	(m-1-2) edge[cell, shorten >= 23pt, shorten <= 8pt] node[below right, xshift=5pt, yshift=2.5pt] {$\rho_0$} (m-2-1);
		\end{tikzpicture}
	\end{displaymath}
	such that $\rho$ satisfies the associativity axiom, displayed below, and $\rho_0$ satisfies the unit axioms (see \defref{definition:colax algebra}).
	\begin{equation} \label{equation:902348}
		\begin{split}
		\begin{tikzpicture}[textbaseline]
      \matrix(m)[math, column sep=0em, row sep=2em]
      { \nc T^2 J \slash f \nc \phantom{T^2J \slash f} \nc[-0.9em] TJ \slash f \nc \phantom{T^2J \slash f}\\
        T^3 J \slash f \nc \nc TJ \slash f \nc \phantom{T^2J \slash f} \nc J \slash f \\
        \nc T^2 J \slash f \nc \nc TJ \slash f \nc \\ };
      \path[map]  (m-1-2) edge node[above] {$\mu_{J \slash f}$} (m-1-4)
                          edge node[below left] {$Tr$} (m-2-3)
                  (m-1-4) edge node[above right] {$r$} (m-2-5)
                  (m-2-1) edge node[above left] {$\mu_{TJ \slash f}$} (m-1-2)
                          edge node[below left] {$T^2 r$} (m-3-2)
                  (m-2-3) edge node[below] {$r$} (m-2-5)
                  (m-3-2) edge node[above left] {$\mu_{J \slash f}$} (m-2-3)
                          edge node[below] {$Tr$} (m-3-4)
                  (m-3-4) edge node[below right] {$r$} (m-2-5)
                  (m-1-4) edge[cell, shorten >= 7pt, shorten <= 7pt] node[below right] {$\rho$} (m-2-3)
                  (m-2-3) edge[cell, shorten >= 7pt, shorten <= 7pt] node[below left] {$\rho$} (m-3-4);
    \end{tikzpicture}
    = \begin{tikzpicture}[textbaseline]
      \matrix(m)[math, column sep=0em, row sep=2em]
      { \nc T^2 J \slash f \nc[-0.9em] \phantom{T^2J \slash f} \nc TJ \slash f \nc \phantom{T^2J \slash f} \\
        T^3 J \slash f \nc \nc T^2 J \slash f \nc \phantom{T^2J \slash f} \nc J \slash f \\
        \nc T^2 J \slash f \nc \nc TJ \slash f \nc \\ };
      \path[map]  (m-1-2) edge node[above] {$\mu_{J \slash f}$} (m-1-4)
                  (m-1-4) edge node[above right] {$r$} (m-2-5)
                  (m-2-1) edge node[above left] {$\mu_{TJ \slash f}$} (m-1-2)
                          edge node[above] {$T\mu_{J \slash f}$} (m-2-3)
                          edge node[below left] {$T^2 r$} (m-3-2)
                  (m-2-3) edge node[above left] {$\mu_{J \slash f}$} (m-1-4)
                          edge node[below left] {$Tr$} (m-3-4)
                  (m-3-2) edge node[below] {$Tr$} (m-3-4)
                  (m-3-4) edge node[below right] {$r$} (m-2-5)
                  (m-2-3) edge[cell, shorten >= 7pt, shorten <= 7pt] node[above left] {$T\rho$} (m-3-2)
                  ($(m-2-3)!0.5!(m-2-5)+(0,0.75em)$) edge[cell] node[right] {$\rho$} ($(m-2-3)!0.5!(m-2-5)-(0,0.75em)$);
    \end{tikzpicture}
    \end{split}
  \end{equation}
	Both the associator and unitor are obtained by using the $2$-dimensional universal property of $J \slash f$, as follows. Notice that the source and target morphisms of the associator correspond to the cells $\pi \of U_r \of U_{\mu_{J \slash f}}$ and $\pi \of U_r \of U_{Tr}$, which equal the composites of the solid subdiagrams of the respective diagrams below. After adding the dashed cells $\alpha$ and $\gamma$ (the associators for $A$ and $C$), we claim that the composites of the full diagrams coincide. Indeed, using the associativity axiom for $\tilde f$ (which can be obtained from that of $\bar f = \inv{(\tilde f)}$, given in \defref{definition:colax morphism}), we can replace, in the diagram on the left-hand side, the subdiagram consisting of the identity cell for $\mu_B \of T^2 f = Tf \of \mu_C$ and the cells $\tilde f$ and $\gamma$ by the likewise shaped diagram consisting of the cells $\beta$, $T\tilde f$ and $\tilde f$. Similarly, in the diagram thus obtained we can replace the subdiagram consisting of $\beta$, $\mu_J$ and $\tilde J$ by subdiagram consisting of the cells $\alpha$, $T\tilde J$ and $\tilde J$, now by using the associativity axiom for $\tilde J$ (\defref{definition:lax promorphism}). The resulting diagram is that on the right-hand side.
	
	We conclude, using the $2$-dimensional universal property of $J \slash f$, that there exists a unique cell $r \of \mu_{J \slash f} \Rar r \of Tr$, which we take $\rho$ to be, such that $U_{\pi_A} \of \rho = \alpha \of U_{T^2 \pi_A}$ and $U_{\pi_C} \of \rho = \gamma \of U_{T^2 \pi_B}$.	Notice that $\rho$ is invertible if $\alpha$ and $\gamma$ are, because the assignment $(\xi_A, \xi_B) \mapsto \xi'$ induced by the $2$-dimensional universal property of $\pi$ is functorial with respect to horizontal composition, as we have seen in the introduction to this appendix. Thus, once we have defined the algebra structure on $J \slash f$, the final assertion of the proposition follows.
	
	\begin{multline*}
		\begin{tikzpicture}[textbaseline]
      \matrix(m)[math175em]
      {	T^2J \slash f \nc T^2J \slash f \nc T^2J \slash f \nc T^2J \slash f \nc T^2J \slash f \\
				\nc T^2C \nc T^2C \nc T^2C \nc T^2C \\
				T^2A \nc T^2B \nc TC \nc TC \nc TC \\
				TA \nc TB \nc TB \nc C \nc C \\
				A \nc B \nc B \nc B \nc B \\ };
			\path[map]	(m-1-1) edge node[left] {$T^2\pi_A$} (m-3-1)
									(m-1-2) edge node[right] {$T^2\pi_C$} (m-2-2)
									(m-1-3) edge node[right] {$T^2\pi_C$} (m-2-3)
									(m-1-4) edge node[right] {$T^2\pi_C$} (m-2-4)
									(m-1-5) edge[dashed] node[right] {$T^2\pi_C$} (m-2-5)
									(m-2-2) edge node[right] {$T^2f$} (m-3-2)
									(m-2-3) edge node[right] {$\mu_C$} (m-3-3)
									(m-2-4) edge node[left] {$\mu_C$} (m-3-4)
									(m-2-5) edge[dashed] node[right] {$Tc$} (m-3-5)
									(m-3-1) edge[barred] node[below] {$T^2J$} (m-3-2)
													edge node[left] {$\mu_A$} (m-4-1)
									(m-3-2) edge node[right] {$\mu_B$} (m-4-2)
									(m-3-3) edge node[left] {$Tf$} (m-4-3)
									(m-3-4) edge node[right] {$c$} (m-4-4)
									(m-3-5) edge[dashed] node[right] {$c$} (m-4-5)
									(m-4-1) edge[barred] node[below] {$TJ$} (m-4-2)
													edge node[left] {$a$} (m-5-1)
									(m-4-2) edge node[right] {$b$} (m-5-2)
									(m-4-3) edge node[left] {$b$} (m-5-3)
									(m-4-4) edge node[right] {$f$} (m-5-4)
									(m-4-5) edge[dashed] node[right] {$f$} (m-5-5)
									(m-5-1) edge[barred] node[below] {$J$} (m-5-2);
			\path				(m-1-1) edge[eq] (m-1-2)
									(m-1-2) edge[eq] (m-1-3)
									(m-1-3) edge[eq] (m-1-4)
									(m-1-4) edge[eq, dashed] (m-1-5)
									(m-2-2) edge[eq] (m-2-3)
									(m-2-3) edge[eq] (m-2-4)
									(m-2-4) edge[eq, dashed] (m-2-5)
									(m-3-3) edge[eq] (m-3-4)
									(m-4-2) edge[eq] (m-4-3)
									(m-4-4) edge[eq, dashed] (m-4-5)
									(m-5-2) edge[eq] (m-5-3)
									(m-5-3) edge[eq] (m-5-4)
									(m-5-4) edge[eq, dashed] (m-5-5);
			\path[transform canvas={shift=($(m-3-4)!0.5!(m-4-3)$), xshift=-8pt}]	(m-1-1) edge[cell] node[right] {$T^2\pi$} (m-2-1);
			\path[transform canvas={shift=($(m-3-4)!0.5!(m-4-3)$)}]	(m-3-3) edge[cell] node[right] {$\tilde f$} (m-4-3)
									(m-2-4) edge[cell, dashed] node[right] {$\gamma$} (m-3-4);
			\path[transform canvas={shift=($(m-3-3)!0.5!(m-3-4)$), yshift=-3pt}]	(m-3-1) edge[cell] node[right] {$\mu_J$} (m-4-1)
									(m-4-1) edge[cell] node[right] {$\tilde J$} (m-5-1);
    \end{tikzpicture} \\
    = \begin{tikzpicture}[textbaseline]
			\matrix(m)[math175em]
			{	T^2J \slash f \nc T^2J \slash f \nc T^2J \slash f \nc T^2J \slash f \nc T^2J \slash f \\
				\nc \nc T^2C \nc T^2C \nc T^2C \\
				T^2A \nc T^2A \nc T^2B \nc TC \nc TC \\
				TA \nc TA \nc TB \nc TB \nc C \\
				A \nc A \nc B \nc B \nc B \\ };
			\path[map]	(m-1-1) edge[dashed] node[left] {$T^2\pi_A$} (m-3-1)
									(m-1-2) edge node[left] {$T^2\pi_A$} (m-3-2)
									(m-1-3) edge node[right] {$T^2\pi_C$} (m-2-3)
									(m-1-4) edge node[right] {$T^2\pi_C$} (m-2-4)
									(m-1-5) edge node[right] {$T^2\pi_C$} (m-2-5)
									(m-2-3) edge node[right] {$T^2f$} (m-3-3)
									(m-2-4) edge node[right] {$Tc$} (m-3-4)
									(m-2-5) edge node[right] {$Tc$} (m-3-5)
									(m-3-1) edge[dashed] node[left] {$\mu_A$} (m-4-1)
									(m-3-2) edge[barred] node[below] {$T^2J$} (m-3-3)
													edge node[left] {$Ta$} (m-4-2)
									(m-3-3) edge node[right] {$Tb$} (m-4-3)
									(m-3-4) edge node[right] {$Tf$} (m-4-4)
									(m-3-5) edge node[right] {$c$} (m-4-5)
									(m-4-1) edge[dashed] node[left] {$a$} (m-5-1)
									(m-4-2) edge[barred] node[below] {$TJ$} (m-4-3)
													edge node[left] {$a$} (m-5-2)
									(m-4-3) edge node[right] {$b$} (m-5-3)
									(m-4-4) edge node[right] {$b$} (m-5-4)
									(m-4-5) edge node[right] {$f$} (m-5-5)
									(m-5-2) edge[barred] node[below] {$J$} (m-5-3);
			\path				(m-1-1) edge[eq, dashed] (m-1-2)
									(m-1-2) edge[eq] (m-1-3)
									(m-1-3) edge[eq] (m-1-4)
									(m-1-4) edge[eq] (m-1-5)
									(m-2-3) edge[eq] (m-2-4)
									(m-2-4) edge[eq] (m-2-5)
									(m-3-1) edge[eq, dashed] (m-3-2)
									(m-3-4) edge[eq] (m-3-5)
									(m-4-3) edge[eq] (m-4-4)
									(m-5-1) edge[eq, dashed] (m-5-2)
									(m-5-3) edge[eq] (m-5-4)
									(m-5-4) edge[eq] (m-5-5);
			\path[transform canvas={shift=($(m-3-4)!0.5!(m-4-3)$)}]	(m-3-1) edge[cell, dashed] node[right] {$\alpha$} (m-4-1)
									(m-2-3) edge[cell] node[right] {$T\tilde f$} (m-3-3)
									(m-3-4) edge[cell] node[right] {$\tilde f$} (m-4-4);
			\path[transform canvas={shift=($(m-3-4)!0.5!(m-4-3)$), xshift=-8pt}]	(m-1-2) edge[cell] node[right] {$T^2\pi$} (m-2-2);
			\path[transform canvas={shift=($(m-3-3)!0.5!(m-3-4)$), yshift=-3pt}]	(m-3-2) edge[cell] node[right] {$T\tilde J$} (m-4-2)
									(m-4-2) edge[cell] node[right] {$\tilde J$} (m-5-2);
    \end{tikzpicture}
	\end{multline*}
	
	To show that $\rho$ satisfies the associativity axiom \eqref{equation:902348} above, consider first the vertical cells making up its right-hand side: the functoriality of the factorisations through $\pi$ and the naturality of $\mu$ mean that
	\begin{itemize}
		\item[-]	the composite $\rho \of U_{T\mu_{J \slash f}}$ corresponds to $\alpha \of U_{T^2\pi_A \of T\mu_{J \slash f}} = \alpha \of U_{T\mu_A \of T^3\pi_A}$ and $\gamma \of U_{T\mu_C \of T^3\pi_C}$;
		\item[-]	the composite $U_r \of T\rho$ corresponds to $U_{a \of T\pi_A} \of T\rho = U_a \of T\alpha \of U_{T^3\pi_A}$ and $U_c \of T\gamma \of U_{T^3\pi_C}$;
		\item[-]	denoting by $\id_{\mu_{J \slash f}}$ the identity cell for $\mu_{J \slash f} \of \mu_{TJ \slash f} = \mu_{J \slash f} \of T\mu_{J \slash f}$, the composite $U_r \of \id_{\mu_{J \slash f}}$ corresponds to $U_{a \of T\pi_A} \of \id_{\mu_{J \slash f}} = U_a \of \id_{\mu_A} \of U_{T^3 \pi_A}$ and $U_c \of \id_{\mu_C} \of U_{T^3 \pi_C}$.
	\end{itemize}
	That is each of the cells in the right-hand side of \eqref{equation:902348} are factorisations of the corresponding cells in the right-hand side of the associativity axiom for $A$ and $C$ and, since factorisation through $\pi$ is functorial with respect to horizontal composition, we can conclude that the right-hand side of \eqref{equation:902348} corresponds to the right-hand sides of the associativity axioms of $\alpha$ and $\gamma$. In the same way the left-hand side of \eqref{equation:902348} corresponds to the left-hand sides of the associativity axioms for $\alpha$ and $\gamma$. Therefore, since for $\alpha$ and $\gamma$ both the left-hand and right-hand sides coincide, it follows that the left-hand and right-hand sides of \eqref{equation:902348} coincide too. This shows that $\rho$ satisfies the associativity axiom.
	
	The unitor $\cell{\rho_0}{r \of \eta_{J \slash f}}{\id_{J \slash f}}$ is obtained similarly: its source and target morphisms correspond to the composites of the solid subdiagrams of the respective diagrams below. Adding the dashed unitors $\gamma_0$ and $\alpha_0$, the full diagrams coincide because of the unit axioms for $\tilde f$ (that can be obtained from that for $\bar f = \inv{(\tilde f)}$) and $\tilde J$. It follows that there is a unique cell $\cell{\rho_0}{r \of \eta_{J \slash f}}{\id_{J \slash f}}$ such that $U_{\pi_A} \of \rho_0 = \alpha_0 \of U_{\pi_A}$ and $U_{\pi_C} \of \rho_0 = \gamma_0 \of U_{\pi_C}$. That $\rho_0$ satisfies the unit axioms (\defref{definition:colax algebra}) follows from the fact that $\alpha_0$ and $\gamma_0$ do, in the same way that the associativity axiom for $\rho$ followed from those for $\alpha$ and $\gamma$. This completes the definition of the colax algebra structure on $J \slash f$.
	\begin{displaymath}
		\begin{tikzpicture}[textbaseline]
      \matrix(m)[math175em, column sep=1.35em]
      { TJ \slash f & TJ \slash f & TJ \slash f & TJ \slash f & TJ \slash f \\
				& C & C & C & C \\
				A & B & TC & TC & \\
				TA & TB & TB & C & C \\
				A & B & B & B & B \\ };
			\path[map]	(m-1-1) edge node[left] {$\pi_A$} (m-3-1)
									(m-1-2) edge node[right] {$\pi_C$} (m-2-2)
									(m-1-3) edge node[right] {$\pi_C$} (m-2-3)
									(m-1-4) edge node[right] {$\pi_C$} (m-2-4)
									(m-1-5) edge[dashed] node[right] {$\pi_C$} (m-2-5)
									(m-2-2) edge node[right] {$f$} (m-3-2)
									(m-2-3) edge node[left] {$\eta_C$} (m-3-3)
									(m-2-4) edge node[left] {$\eta_C$} (m-3-4)
									(m-2-5) edge[dashed] node[right] {$\id_C$} (m-4-5)
									(m-3-1) edge[barred] node[below] {$J$} (m-3-2)
													edge node[left] {$\eta_A$} (m-4-1)
									(m-3-2) edge node[right] {$\eta_B$} (m-4-2)
									(m-3-3) edge node[left] {$Tf$} (m-4-3)
									(m-3-4) edge node[right] {$c$} (m-4-4)
									(m-4-1) edge[barred] node[below] {$TJ$} (m-4-2)
													edge node[left] {$a$} (m-5-1)
									(m-4-2) edge node[right] {$b$} (m-5-2)
									(m-4-3) edge node[left] {$b$} (m-5-3)
									(m-4-4) edge node[right] {$f$} (m-5-4)
									(m-4-5) edge[dashed] node[right] {$f$} (m-5-5)
									(m-5-1) edge[barred] node[below] {$J$} (m-5-2);
			\path				(m-1-1) edge[eq] (m-1-2)
									(m-1-2) edge[eq] (m-1-3)
									(m-1-3) edge[eq] (m-1-4)
									(m-1-4) edge[eq, dashed] (m-1-5)
									(m-2-2) edge[eq] (m-2-3)
									(m-2-3) edge[eq] (m-2-4)
									(m-2-4) edge[eq, dashed] (m-2-5)
									(m-3-3) edge[eq] (m-3-4)
									(m-4-2) edge[eq] (m-4-3)
									(m-4-4) edge[eq, dashed] (m-4-5)
									(m-5-2) edge[eq] (m-5-3)
									(m-5-3) edge[eq] (m-5-4)
									(m-5-4) edge[eq, dashed] (m-5-5);
			\path[transform canvas={shift=($(m-3-4)!0.5!(m-4-3)$)}]	(m-1-1) edge[cell] node[right] {$\pi$} (m-2-1)
									(m-3-3) edge[cell] node[right] {$\tilde f$} (m-4-3)
									(m-2-4) edge[cell, dashed] node[right] {$\gamma_0$} (m-3-4);
			\path[transform canvas={shift=($(m-3-3)!0.5!(m-3-4)$), yshift=-3pt}]	(m-3-1) edge[cell] node[right] {$\eta_J$} (m-4-1)
									(m-4-1) edge[cell] node[right] {$\tilde J$} (m-5-1);
    \end{tikzpicture}
		=	\begin{tikzpicture}[textbaseline]
			\matrix(m)[math175em, column sep=1.35em]
      {	TJ \slash f & TJ \slash f & TJ \slash f \\
				& & C \\
				A & A & B \\
				TA & & \\
				A & A & B \\ };
      \path[map]  (m-1-1) edge[dashed] node[left] {$\pi_A$} (m-3-1)
									(m-1-2) edge node[left] {$\pi_A$} (m-3-2)
									(m-1-3) edge node[right] {$\pi_C$} (m-2-3)
									(m-2-3) edge node[right] {$f$} (m-3-3)
									(m-3-1) edge[dashed] node[left] {$\eta_A$} (m-4-1)
									(m-3-2) edge[barred] node[below] {$J$} (m-3-3)
													edge node[right] {$\id_A$} (m-5-2)
									(m-3-3) edge node[right] {$\id_B$} (m-5-3)
									(m-4-1) edge[dashed] node[left] {$a$} (m-5-1)
									(m-5-2) edge[barred] node[below] {$J$} (m-5-3);
			\path				(m-1-1) edge[eq, dashed] (m-1-2)
									(m-1-2) edge[eq] (m-1-3)
									(m-3-1) edge[eq, dashed] (m-3-2)
									(m-5-1) edge[eq, dashed] (m-5-2);
			\path[transform canvas={shift=($(m-3-3)!0.5!(m-4-2)$)}]	(m-1-2) edge[cell] node[right] {$\pi$} (m-2-2)
									(m-3-1) edge[cell, dashed] node[right] {$\alpha_0$} (m-4-1);
    \end{tikzpicture}
	\end{displaymath}
	
	We have already seen that the induced algebra structure on $J \slash f$ makes $\pi_A$ and $\pi_C$ into strict morphisms, and $\pi$ into a $T$-cell: this leaves showing that $\pi$ satisfies the universal properties making $J \slash f$ into a double comma object of $\rcProm T$. Thus consider a second $T$-cell $\phi$ below. Since $J \slash f$ is a double comma object in $\K$, $\phi$ factors through $\pi$ as a morphism $\map{\phi'}X{J \slash f}$, which we have to give a structure cell $\bar\phi'$ as on the right.
	\begin{displaymath}
		\begin{tikzpicture}[baseline]
			\matrix(m)[math175em]{X & X \\ \phantom C & C \\ A & B \\};
			\path[map]	(m-1-1) edge node[left] {$\phi_A$} (m-3-1)
									(m-1-2)	edge node[right] {$\phi_C$} (m-2-2)
									(m-2-2) edge node[right] {$f$} (m-3-2)
									(m-3-1) edge[barred] node[below] {$J$} (m-3-2);
			\path[transform canvas={shift={($(m-2-2)!0.5!(m-3-2)$)}}] (m-1-1) edge[cell] node[right] {$\phi$} (m-2-1);
			\path				(m-1-1) edge[eq] (m-1-2);
		\end{tikzpicture}
		\qquad\qquad\qquad\qquad\begin{tikzpicture}[textbaseline]
			\matrix(m)[math175em]{TX & X \\ TJ \slash f & J \slash f \\};
			\path[map]	(m-1-1) edge node[above] {$x$} (m-1-2)
													edge node[left] {$T\phi'$} (m-2-1)
									(m-1-2) edge node[right] {$\phi'$} (m-2-2)
									(m-2-1) edge node[below] {$r$} (m-2-2);
			\path				(m-1-2) edge[cell, shorten >= 11pt, shorten <= 11pt] node[below right] {$\bar\phi'$} (m-2-1);
		\end{tikzpicture}
	\end{displaymath}
	Again we can follow the same recipe: the source and target morphisms of $\bar\phi'$ correspond to the solid subdiagrams in
	\begin{displaymath}
		\begin{tikzpicture}[textbaseline]
      \matrix(m)[math175em]
      { TX & TX & TX \\
				X & X & TC \\
				& C & C \\
				A & B & B \\ };
      \path[map]  (m-1-1) edge node[left] {$x$} (m-2-1)
									(m-1-2) edge node[left] {$x$} (m-2-2)
									(m-1-3) edge[dashed] node[right] {$T\phi_C$} (m-2-3)
									(m-2-1) edge node[left] {$\phi_A$} (m-4-1)
									(m-2-2) edge node[right] {$\phi_C$} (m-3-2)
									(m-2-3) edge[dashed] node[right] {$c$} (m-3-3)
									(m-3-2) edge node[right] {$f$} (m-4-2)
									(m-3-3) edge[dashed] node[right] {$f$} (m-4-3)
									(m-4-1) edge[barred] node[below] {$J$} (m-4-2);
			\path				(m-1-1) edge[eq] (m-1-2)
									(m-1-2) edge[eq, dashed] (m-1-3)
									(m-2-1) edge[eq] (m-2-2)
									(m-3-2) edge[eq, dashed] (m-3-3)
									(m-4-2) edge[eq, dashed] (m-4-3);
			\path[transform canvas={shift=($(m-3-2)!0.5!(m-3-3)$), xshift=-6pt}]	(m-1-2) edge[cell, dashed] node[right] {$\bar\phi_C$} (m-2-2);
			\path[transform canvas={shift=($(m-3-2)!0.5!(m-3-3)$)}]	(m-2-1) edge[cell] node[right] {$\phi$} (m-3-1);
    \end{tikzpicture}
    = \begin{tikzpicture}[textbaseline]
      \matrix(m)[math175em]
      { TX & TX & TX & TX \\
				\phantom{TB} & \phantom{TB} & TC & TC \\
				X & TA & TB & C \\
				A & A & B & B, \\ };
      \path[map]  (m-1-1) edge[dashed] node[left] {$x$} (m-3-1)
									(m-1-2) edge node[left] {$T\phi_A$} (m-3-2)
									(m-1-3) edge node[right] {$T\phi_C$} (m-2-3)
									(m-1-4) edge node[right] {$T\phi_C$} (m-2-4)
									(m-2-3) edge node[right] {$Tf$} (m-3-3)
									(m-2-4) edge node[right] {$c$} (m-3-4)
									(m-3-1) edge[dashed] node[left] {$\phi_A$} (m-4-1)
									(m-3-2) edge[barred] node[below] {$TJ$} (m-3-3)
													edge node[left] {$a$} (m-4-2)
									(m-3-3) edge node[right] {$b$} (m-4-3)
									(m-3-4) edge node[right] {$f$} (m-4-4)
									(m-4-2) edge[barred] node[below] {$J$} (m-4-3);
			\path				(m-1-1) edge[eq, dashed] (m-1-2)
									(m-1-2) edge[eq] (m-1-3)
									(m-1-3) edge[eq] (m-1-4)
									(m-2-3) edge[eq] (m-2-4)
									(m-4-1) edge[eq, dashed] (m-4-2)
									(m-4-3) edge[eq] (m-4-4);
			\path[transform canvas={shift=($(m-2-3)!0.5!(m-3-3)$)}]	(m-2-1) edge[cell, dashed] node[right] {$\bar\phi_A$} (m-3-1);
			\path[transform canvas={shift=($(m-2-3)!0.5!(m-3-3)$), yshift=-3pt}] (m-3-2) edge[cell] node[right] {$\tilde J$} (m-4-2);
			\path[transform canvas={shift=(m-3-3), xshift=-6pt}]	(m-1-2) edge[cell] node[right] {$T\phi$} (m-2-2);
			\path[transform canvas={shift=(m-3-3)}]	(m-2-3) edge[cell] node[right] {$\tilde f$} (m-3-3);
    \end{tikzpicture}
	\end{displaymath}
	and it is shown below that the full diagrams coincide, where the identities follow from the definition of $\rho_1\bar\phi_A$ (see \defref{definition:cells of right colax promorphisms}); the $T$-cell axiom for $\phi$ and the definition of $\tilde J$; cancelling $\bar J$ against its inverse as well as the definitions of $\rho_1\bar\phi_C$ and $\rho_1\bar f$; the fact that $\tilde f = \inv{\bar f}$. Using the $2$-dimensional universal property of $J \slash f$, we obtain a unique cell $\bar\phi'$, as above, such that $U_{\pi_A} \of \bar\phi' = \bar\phi_A$ and $U_{\pi_C} \of \bar\phi' = \bar\phi_C$. That $\bar\phi'$ satisfies the associativity and unit axioms follows from the fact that $\bar\phi_A$ and $\bar\phi_C$ do. We conclude that there is a unique way of making $\phi'$ into a $T$-cell such that $U_{\pi_A} \of \phi' = \phi_A$ and $U_{\pi_C} \of \phi' = \phi_C$, as colax $T$-morphisms. This completes the proof showing that $\pi$ satisfies the $1$-dimensional universal property in $\rcProm T$.
	\begin{align*}
		\begin{tikzpicture}[textbaseline, x=0.75cm, y=0.6cm, font=\scriptsize]
			\draw	(1,1) -- (2,1) -- (2,3) -- (0,3) -- (0,0) -- (1,0) -- (1,3)
				(1,0) -- (2,0) -- (2,1)
				(2,0) -- (3,0) -- (3,2) -- (2,2);
			\draw[shift={(0.5,0)}]
				(1,2) node {$T\phi$}
				(2,1) node {$\tilde f$};
			\draw[shift={(0.5,0.5)}]
				(0,1) node {$\bar\phi_A$}
				(1,0) node {$\tilde J$};
		\end{tikzpicture}
		&\quad = \quad\begin{tikzpicture}[textbaseline, x=0.75cm, y=0.6cm, font=\scriptsize]
			\draw	(0,1) -- (2,1) -- (2,3) -- (0,3) -- (0,0) -- (1,0) -- (1,3)
				(1,0) -- (2,0) -- (2,1)
				(2,0) -- (3,0) -- (3,2) -- (2,2)
				(0,3) -- (0,4) -- (1,4) -- (1,3);
			\draw[shift={(0.5,0)}]
				(0,2) node {$\rho_1\bar\phi_A$}
				(1,2) node {$T\phi$}
				(2,1) node {$\tilde f$};
			\draw[shift={(0.5,0.5)}]
				(0,0) node {$\eps_a$}
				(0,3) node {$\eta_x$}
				(1,0) node {$\tilde J$};
		\end{tikzpicture}
		\quad = \quad\begin{tikzpicture}[textbaseline, x=0.75cm, y=0.6cm, font=\scriptsize]
			\draw	(0,2) -- (2,2) -- (2,1) -- (0,1) -- (0,3) -- (2,3) -- (2,2)
				(2,3) -- (2,4) -- (3,4) -- (3,0) -- (2,0) -- (2,1)
				(1,1) -- (1,0) -- (2,0)
				(1,3) -- (1,4) -- (2,4)
				(0,3) -- (0,5) -- (2,5)
				(1,4) -- (1,6) -- (2,6) -- (2,4);
			\draw[shift={(0.5,0)}]
				(0,4) node {$\phi$}
				(2,2) node {$\tilde f$};
			\draw[shift={(0,0.5)}]
				(1,1) node {$\inv{(\bar J)}$}
				(1,2) node {$\bar J$};
			\draw[shift={(0.5,0.5)}]
				(1,0) node {$\eps_b$}
				(1,3) node {$\rho_1\bar f$}
				(1,4) node {$\rho_1\bar\phi_C$}
				(1,5) node {$\eta_x$};
		\end{tikzpicture} \\[1em]
		&\quad = \quad\,\begin{tikzpicture}[textbaseline, x=0.75cm, y=0.6cm, font=\scriptsize]
			\draw	(1,2) -- (0,2) -- (0,0) -- (1,0) -- (1,3) -- (2,3) -- (2,1) -- (1,1)
				(2,1) -- (2,0) -- (4,0) -- (4,2) -- (3,2)
				(2,2) -- (3,2) -- (3,0);
			\draw[shift={(0.5,0)}]
				(0,1) node {$\phi$}
				(1,2) node {$\bar\phi_C$}
				(2,1) node {$\bar f$}
				(3,1) node {$\tilde f$};
		\end{tikzpicture}
		\quad = \quad\begin{tikzpicture}[textbaseline, x=0.75cm, y=0.6cm, font=\scriptsize]
			\draw	(1,2) -- (0,2) -- (0,0) -- (1,0) -- (1,3) -- (2,3) -- (2,1) -- (1,1);
			\draw[shift={(0.5,0)}]
				(0,1) node {$\phi$}
				(1,2) node {$\bar\phi_C$};
		\end{tikzpicture}
	\end{align*}
	
	Finally let $\psi = (\psi_A, \psi, \psi_C)$ be a second $T$-cell, with $0$-dimensional source $Y$, that factors through $\pi$ as the colax $T$-morphism $\map{\psi'}Y{J \slash f}$. Consider cells $\map{(\xi_A, \xi_C)}\phi\psi$, in $\K$, as in the introduction to this appendix, so that there exists a unique cell $\xi'$ with $U_{\pi_A} \of \xi' = \xi_A$ and $U_{\pi_C} \of \xi' = \xi_C$, because $J \slash f$ is a double comma object in $\K$. To prove that $J \slash f$ satisfies the $2$-dimensional universal property in $\rcProm T$ as well, we have to show that if $\xi_A$ and $\xi_B$ are $T$-cells then so is $\xi$. To see this consider the following identity, which is equivalent to the $T$-cell axiom (\defref{definition:cells of right colax promorphisms}) for $\xi$ and is obtained by vertically postcomposing the latter with the companion cell $\eps_r$. 
	\begin{displaymath}
		\begin{tikzpicture}[textbaseline]
			\matrix(m)[math175em, column sep=1.6em]
			{	\phantom{TJ \slash f} & \phantom{TJ \slash f} & & TY \\
				X & TX & TX & TY \\
				X & X & TJ \slash f & TJ \slash f \\
				J\slash f & J \slash f & J \slash f & J \slash f \\ };
			\draw[text height=1.5ex, text depth=0.25ex]	(m-1-1) node(p) {$X$}
				($(m-1-1)!0.5!(m-1-4)$) node(q) {$Y$};
			\path[map]	(p) edge[barred] node[above] {$K$} (q)
									(q) edge[barred] node[above] {$Y(\id, y)$} (m-1-4)
									(m-2-1) edge[barred] node[above] {$X(\id, x)$} (m-2-2)
									(m-2-2) edge node[right] {$x$} (m-3-2)
									(m-2-3) edge[barred] node[above] {$TK$} (m-2-4)
													edge node[left] {$T\phi'$} (m-3-3)
									(m-2-4) edge node[right] {$T\psi'$} (m-3-4)
									(m-3-1) edge node[left] {$\phi'$} (m-4-1)
									(m-3-2) edge node[left] {$\phi'$} (m-4-2)
									(m-3-3) edge node[right] {$r$} (m-4-3)
									(m-3-4) edge node[right] {$r$} (m-4-4);
			\path				(p) edge[eq] (m-2-1)
									(m-1-4) edge[eq] (m-2-4)
									(m-2-1) edge[eq] (m-3-1)
									(m-2-2) edge[eq] (m-2-3)
									(m-3-1) edge[eq] (m-3-2)
									(m-3-3) edge[eq] (m-3-4)
									(m-4-1) edge[eq] (m-4-2)
									(m-4-2) edge[eq] (m-4-3)
									(m-4-3) edge[eq] (m-4-4);
			\path[transform canvas={shift=($(m-2-3)!0.5!(m-3-3)$)}]	(m-1-2) edge[cell] node[right] {$\bar K$} (m-2-2)
									(m-2-1) edge[cell] node[right] {$\eps_x$} (m-3-1)
									(m-2-3) edge[cell] node[right] {$T\xi'$} (m-3-3);
			\path[transform canvas={shift=(m-3-3)}]	(m-2-2) edge[cell] node[right] {$\bar\phi'$} (m-3-2);
		\end{tikzpicture}
		= \begin{tikzpicture}[textbaseline]
			\matrix(m)[math175em, column sep=1.6em]
			{	X & Y & TY & TY \\
				X & Y & Y & TJ \slash f \\
				J \slash f & J \slash f & J \slash f & J \slash f \\ };
			\path[map]	(m-1-1) edge[barred] node[above] {$K$} (m-1-2)
									(m-1-2) edge[barred] node[above] {$Y(\id, y)$} (m-1-3)
									(m-1-3) edge node[right] {$y$} (m-2-3)
									(m-1-4) edge node[right] {$T\psi'$} (m-2-4)
									(m-2-1) edge[barred] node[above] {$K$} (m-2-2)
													edge node[left] {$\phi'$} (m-3-1)
									(m-2-2) edge node[right] {$\psi'$} (m-3-2)
									(m-2-3) edge node[left] {$\psi'$} (m-3-3)
									(m-2-4) edge node[right] {$r$} (m-3-4);
			\path				(m-1-1) edge[eq] (m-2-1)
									(m-1-2) edge[eq] (m-2-2)
									(m-1-3) edge[eq] (m-1-4)
									(m-2-2) edge[eq] (m-2-3)
									(m-3-1) edge[eq] (m-3-2)
									(m-3-2) edge[eq] (m-3-3)
									(m-3-3) edge[eq] (m-3-4);
			\path[transform canvas={shift=(m-2-3)}]	(m-1-2) edge[cell] node[right] {$\eps_y$} (m-2-2)
									(m-2-1) edge[cell] node[right] {$\xi'$} (m-3-1);
			\path[transform canvas={shift=($(m-2-3)!0.5!(m-3-3)$)}]	(m-1-3) edge[cell] node[right] {$\bar\psi'$} (m-2-3);
		\end{tikzpicture}
	\end{displaymath}
	To see that the identity above holds, we again use the functoriality of the $2$\ndash di\-men\-sio\-nal factorisations through $\pi$ and find that, by postcomposing the composites above with $U_{\pi_A}$, we obtain the two sides of the corresponding $T$-cell axiom for $\xi_A$, postcomposed with $\eps_a$. Similarly, when postcomposed with $U_{\pi_C}$, we obtain the two sides of the corresponding $T$-cell axiom for $\xi_C$, postcomposed with $\eps_c$. Therefore, since the $T$-cell axioms for $\xi_A$ and $\xi_C$ hold, and since factorisations through $\pi$ are unique, it follows that the identity above, and thus the $T$-cell axiom for $\xi'$, hold as well. This shows that $J \slash f$ satisfies the $2$-dimensional property as well, in $\rcProm T$, which finishes the proof.
\end{proof}

\chapter{Proof of \lemref{APPLEMMA}} \label{appendix:lemma}
In this appendix we prove \lemref{lemma:left hom lax structure cell}, reproduced below, which was used in proving the main theorem.
\begin{lefthomlaxstructurecell}
	Let $T$ be a right suitable normal monad on a closed equipment $\K$ (\defref{definition:right suitable monads}) and consider a right colax $T$\ndash pro\-mor\-phism $\hmap JAB$ (\defref{definition:right colax promorphisms}) as well as a lax $T$-promorphism $\hmap KAC$ (\defref{definition:lax promorphism}). The left hom $\hmap{J \lhom K}BC$ admits a lax $T$-promorphism structure cell $\ol{J \lhom K}$ that corresponds, under \propref{left and right cells}, to the horizontal cell $\lambda\ol{J \lhom K}$ that is given by
	\begin{multline} \label{equation:left hom lax structure cell in appendix}
		T(J \lhom K) \hc C(c, \id) \Rar TJ \lhom \bigpars{TK \hc C(c, \id)} \xRar{\id \lhom \lambda\bar K} TJ \lhom \bigpars{A(a, \id) \hc K} \\
		\iso \bigpars{A(\id, a) \hc TJ} \lhom K \xRar{\bar J \lhom \id} \bigpars{J \hc B(\id, b)} \lhom K \iso B(b, \id) \hc J \lhom K,
	\end{multline}
	where the first cell is given by \propref{lax functors induce left hom coherence cell} and where the two isomorphisms are given by \corref{companion, conjoint and left hom isomorphisms}.
\end{lefthomlaxstructurecell}
	
	In the proof we use the facts recorded in the following lemma. Each of them follows easily from the naturality of the isomorphisms defining $J \lhom \dash$ as the right adjoint of $J \hc \dash$, and the naturality of the associators $F_\hc$ of lax functors.
\begin{lemma} \label{lemma:left hom facts}
	In each of the statements below $J$, $H$, $K$, $L$ and $M$ are horizontal morphisms in a closed equipment $\K$, while $\phi$ and $\psi$ are horizontal cells. In the last statement $\map F\K\L$ is assumed to be a lax functor between closed equipments.
	\begin{itemize}
		\item[\textup{(a)}] If $L \Rar K \lhom M$ is adjoint to $\cell\phi{K \hc L}M$ and $H \hc K \lhom M \Rar J \lhom M$ is adjoint to $J \hc H \hc K \lhom M \xRar{\psi \hc \id} K \hc K \lhom M \xRar{\ev} M$, then the composition $H \hc L \Rar H \hc K \lhom M \Rar J \lhom M$ is adjoint to $J \hc H \hc L \xRar{\psi \hc \id} K \hc L \xRar{\phi} M$.
		\item[\textup{(b)}] If $K \Rar H \lhom L$ is adjoint to $\cell\phi{H \hc K}L$ and $H \lhom L \Rar (J \hc H) \lhom M$ is adjoint to $J \hc H \hc H \lhom L \xRar{\id \hc \ev} J \hc L \xRar{\psi} M$, then the composition $K \Rar (J \hc H) \lhom M$ is adjoint to $J \hc H \hc K \xRar{\id \hc \phi} J \hc L \xRar{\psi} M$.
		\item[\textup{(c)}] If $H \Rar J \lhom K$ is adjoint to $\cell\phi{J \hc H}K$ and $F(J \lhom K) \hc L \Rar FJ \lhom (FK \hc L)$ is the horizontal cell given by \propref{lax functors induce left hom coherence cell}, then the composition \begin{displaymath}
			FH \hc L \Rar F(J \lhom K) \hc L \Rar FJ \lhom (FK \hc L)
		\end{displaymath}
		is adjoint to $FJ \hc FH \hc L \xRar{F_\hc} F(J \hc H) \hc L \xRar{F\phi \hc \id} FK \hc L$.
	\end{itemize}
\end{lemma}
\begin{proof}[Proof of \lemref{lemma:left hom lax structure cell}]
	Recall that we denote the first four cells of \eqref{equation:left hom lax structure cell in appendix} by $\omega(\bar J, \bar K)$, so that the lemma states that the left hom $\hmap{J \lhom K}BC$ can be given the structure of a lax promorphism whose structure cell $\ol{J \lhom K}$ corresponds, under \propref{left and right cells}, to the composite
	\begin{displaymath}
		\lambda\ol{J \lhom K} = \bigbrks{T(J \lhom K) \hc C(c, \id) \xRar{\omega(\bar J, \bar K)} \bigpars{J \hc B(\id, b)} \lhom K \iso B(b, \id) \hc J \lhom K},
	\end{displaymath}
	where the isomorphism, from right to left, is adjoint to the counit $\ls b\eps_b$ of $B(b, \id) \ladj B(\id, b)$, followed by evaluation (\corref{companion, conjoint and left hom isomorphisms}). Applying \lemref{lemma:left hom facts}(b) to $\omega(\bar J, \bar K)$ we find that its adjoint is the composite
	\begin{align} \label{equation:891123}
		&{}J \hc B(\id, b) \hc T(J \lhom K) \hc C(c, \id) \xRar{\bar J \hc \id} A(\id, a) \hc TJ \hc T(J \lhom K) \hc C(c, \id) \notag \\
		&\xRar{\id \hc T_\hc \hc \id} A(\id, a) \hc T(J \hc J \lhom K) \hc C(c, \id) \xRar{\id \hc T\ev \hc \id} A(\id, a) \hc TK \hc C(c, \id) \notag \\
		&\xRar{\id \hc \lambda\bar K} A(\id, a) \hc A(a, \id) \hc K \xRar{\ls a\eps_a \hc \id} K.
	\end{align}
	We will show that $\ol{J \lhom K}$ satisfies the associativity and unit axioms of \defref{definition:lax promorphism} by proving that the corresponding axioms for $\lambda\ol{J \lhom K}$ hold, as given in \propref{left and right structure cells}.
	
	We start with the associativity axiom for $\lambda\ol{J \lhom K}$, whose right-hand side is the following composite
	\begin{align} \label{equation:587709}
		&T^2(J \lhom K) \hc TC(c, \id) \hc C(c, \id) \xRar{T_\hc \hc \id} T\bigpars{T(J \lhom K) \hc C(c, \id)} \hc C(c, \id) \notag \\
		&\xRar{T\omega(\bar J, \bar K) \hc \id} T\Bigpars{\bigpars{J \hc B(\id, b)} \lhom K} \hc C(c, \id) \iso T\bigpars{B(b, \id) \hc J \lhom K} \hc C(c, \id) \notag \\
		&\xRar{\inv T_\hc \hc \id} TB(b, \id) \hc T(J \lhom K) \hc C(c, \id) \xRar{\id \hc \omega(\bar J, \bar K)} TB(b, \id) \hc \bigpars{J \hc B(\id, b)} \lhom K \notag \\
		&\iso TB(b, \id) \hc B(b, \id) \hc J \lhom K \xRar{\lambda\beta \hc \id} TB(\mu_B, \id) \hc B(b, \id) \hc J \lhom K,
	\end{align}
	where the inverse of the $T_\hc$-component exists by \propref{normal functors preserve companions}. In this chain of horizontal cells we consider the subchain $T\Bigpars{\bigpars{J \hc B(\id, b)} \lhom K} \hc C(c, \id) \Rar TB(\mu_B, \id) \hc B(b, \id) \hc J \lhom K$ of the last five cells, which also occurs as the top leg of the diagram below.
	\begin{displaymath}
		\begin{tikzpicture}
			\matrix(k)[math, column sep=2em, yshift=7.5em]
				{ T\Bigpars{\bigpars{J \hc B(\id, b)} \lhom K} \hc C(c, \id) & T\bigpars{B(b, \id) \hc J \lhom K} \hc C(c, \id) \\};
			\matrix(l)[math, row sep=2.5em, column sep=6em]
				{ T\bigpars{J \hc B(\id, b)} \lhom \bigpars{TK \hc C(c, \id)} & TB(b, \id) \hc T(J \lhom K) \hc C(c, \id) \\
					\bigpars{TJ \hc TB(\id, b)} \lhom \bigpars{A(a, \id) \hc K} & TB(b, \id) \hc \bigpars{J \hc B(\id, b)} \lhom K \\
					\bigpars{A(\id, a) \hc TJ \hc B(\id, b)} \lhom K & TB(b, \id) \hc B(b, \id) \hc J \lhom K \\ };
			\matrix(m)[math, yshift=-7.5em, xshift=2.1em, column sep=3.5em]
				{	\bigpars{J \hc B(\id, b) \hc TB(\id, b)} \lhom K & TB(\mu_B, \id) \hc B(b, \id) \hc J \lhom K \\};
			\matrix(n)[math, yshift=-11em]
				{ \bigpars{J \hc B(\id, b) \hc TB(\id, \mu_B)} \lhom K \\};
			\path	(k-1-1) edge[cell] (l-1-1)
						(k-1-2) edge[cell] node[iso] {$\iso$} (k-1-1)
						(l-1-1) edge[cell] node[left] {$T_\hc \lhom \lambda\bar K$} (l-2-1)
						(l-1-2) edge[cell] node[right] {$T_\hc \hc \id$} (k-1-2)
										edge[cell] node[right] {$\id \hc \omega(\bar J, \bar K)$} (l-2-2)
						(l-2-1) edge[cell] node[iso] {$\iso$} (l-3-1)
						(l-2-2.south west) edge[cell] node[iso] {$\iso$} (m-1-1)
						(l-3-1) edge[cell] node[left, inner sep=5pt] {$(\bar J \hc \id) \lhom \id$} (m-1-1)
						(l-3-2) edge[cell] node[iso] {$\iso$} (l-2-2)
										edge[cell] node[right] {$\lambda\beta \hc \id$} (m-1-2)
						(l-3-2.south west) edge[cell] node[iso] {$\iso$} (m-1-1.north east)				
						(m-1-1) edge[cell] node[left, inner sep=8pt] {$(\id \hc \rho\beta) \lhom \id$} (n-1-1)
						(m-1-2) edge[cell] node[iso] {$\iso$} (n-1-1);
			\draw	(0,2em) node {I}
						(3em,-3em) node {II}
						(3em,-7em) node {III};
		\end{tikzpicture}
	\end{displaymath}
	Here the single unlabelled cell is given by \propref{lax functors induce left hom coherence cell} and each of the isomorphisms is given by \corref{companion, conjoint and left hom isomorphisms}, using the counits for the several companion-conjoint adjunctions. For example, the top isomorphism of the subdiagram labelled III is adjoint to the counit $\ls{b \of Tb}\eps_{b \of Tb}$, for the adjunction $TB(b, \id) \hc B(b, \id) \ladj B(\id, b) \hc TB(\id, b)$ of the companion and conjoint of $b \of Tb$, followed by evaluation (remember our convention, preceeding \propref{left and right structure cells}, concerning the companions and conjoints of composites like $b \of Tb$).
	
	We will show that the diagram above commutes. First consider the bottom leg of the subdiagram labelled I, which consists of $\id \hc \omega(\bar J, \bar K)$ followed by the isomorphism that is adjoint to $\ls{Tb}\eps_{Tb}$ and followed by evaluation. Here $\omega(\bar J, \bar K)$ is adjoint to the composite \eqref{equation:891123} so that we can apply \lemref{lemma:left hom facts}(a). We conclude that the bottom leg of I is adjoint to
	\begin{align*}
		J \hc B(\id, b)&{} \hc TB(\id, b) \hc TB(b, \id) \hc T(J \lhom K) \hc C(c, \id) \\
		&\xRar{\id \hc \ls{Tb}\eps_{Tb} \hc \id} J \hc B(\id, b) \hc T(J \lhom K) \hc C(c, \id) \\
		&\xRar{\bar J \hc \id} A(\id, a) \hc TJ \hc T(J \lhom K) \hc C(c, \id) \\
		&\xRar{\id \hc T_\hc \hc \id} A(\id, a) \hc T(J \hc J \lhom K) \hc C(c, \id) \\
		&\xRar{\id \hc T\ev \hc \id} A(\id, a) \hc TK \hc C(c, \id) \\
		&\xRar{\id \hc \bar K} A(\id, a) \hc A(a, \id) \hc K \xRar{\ls a\eps_a \hc \id} K
	\end{align*}
	On the other hand the adjoint of the top leg of the subdiagram I can be computed by first applying part (b) of \lemref{lemma:left hom facts} to its second isomorphism and then part (c) to the unlabelled cell. Thus it is adjoint to
	\begin{align*}
		J \hc B(\id, b)&{} \hc TB(\id, b) \hc TB(b, \id) \hc T(J \lhom K) \hc C(c, \id) \\
		&\xRar{\bar J \hc \id} A(\id, a) \hc TJ \hc TB(\id, b) \hc TB(b, \id) \hc T(J \lhom K) \hc C(c, \id) \\
		&\xRar{\id \hc T_\hc \hc \id} A(\id, a) \hc T\bigpars{J \hc B(\id, b) \hc B(b, \id) \hc J \lhom K} \hc C(c, \id) \\
		&\xRar{\id \hc T(\id \hc \ls b\eps_b \hc \id) \hc \id} A(\id, a) \hc T(J \hc J \lhom K) \hc C(c, \id) \\
		&\xRar{\id \hc T\ev \hc \id} A(\id, a) \hc TK \hc C(c, \id) \\
		&\xRar{\id \hc \lambda\bar K} A(\id, a) \hc A(a, \id) \hc K \xRar{\ls a\eps_a \hc \id} K.
	\end{align*}
	That the two adjoints above coincide follows from the associativity axiom for $T$ and from $T(\ls b\eps_b) \of T_\hc = T(\eps_b \hc \ls b\eps) \of T_\hc = T\eps_b \hc T\ls b\eps = \eps_{Tb} \hc \ls{Tb}\eps = \ls{Tb}\eps_{Tb}$; we conclude that the subdiagram I commutes. Similarly we can use \lemref{lemma:left hom facts}(a) to compute the adjoint of the top leg of the subdiagram labelled II, obtaining the composite
	\begin{multline*}
		J \hc B(\id, b) \hc TB(\id, b) \hc TB(b, \id) \hc B(b, \id) \hc J \lhom K \\
		\xRar{\id \hc \ls{Tb}\eps_{Tb} \hc \id} J \hc B(\id, b) \hc B(b, \id) \hc J \lhom K \xRar{\id \hc \ls b\eps_b \hc \id} J \hc J \lhom K \xRar\ev K.
	\end{multline*}
	The composite of the first two cells here coincides with $\id \hc \ls{b \of Tb}\eps_{b \of Tb} \hc \id$, so that the full composite above coincides with the adjoint of the bottom cell of II: the subdiagram II commutes too. Finally the top leg of the subdiagram labelled III is adjoint to
	\begin{align*}
		J \hc B(\id, b)&{} \hc TB(\id, \mu_B) \hc TB(b, \id) \hc B(b, \id) \hc J \lhom K \\
		&\xRar{\id \hc \rho\beta \hc \id} J \hc B(\id, b) \hc TB(\id, b) \hc TB(b, \id) \hc B(b, \id) \hc J \lhom K \\
		&\xRar{\id \hc \ls{b \of Tb}\eps_{b \of Tb} \hc \id} J \hc J \lhom K \xRar\ev K.
	\end{align*}
	It follows from the definitions of $\rho\beta$ and $\lambda\beta$ (\propref{left and right cells}), as well as the vertical companion and conjoint identities, that $\ls{b \of Tb}\eps_{b \of Tb} \of (\rho\beta \hc \id) = \ls{b \of \mu_B}\eps_{b \of \mu_B} \of (\id \hc \lambda\beta)$, so that the composite above equals the adjoint of the bottom leg of the subdiagram III, which therefore commutes as well.
	
	Replacing in \eqref{equation:587709} the top leg of the diagram above with the bottom one, as well as expanding $T\omega(\bar J, \bar K)$, we obtain the composite
	\begin{align} \label{equation:256479}
		&T^2(J \lhom K) \hc TC(c, \id) \hc C(c, \id) \xRar{T_\hc \hc \id} T\bigpars{T(J \lhom K) \hc C(c, \id)} \hc C(c, \id) \notag \\
		&\Rar T\Bigpars{TJ \lhom \bigpars{TK \hc C(c, \id)}} \hc C(c, \id) \notag \\
		&\xRar{T(\id \lhom \lambda\bar K) \hc \id} T\Bigpars{TJ \lhom \bigpars{A(a, \id) \hc K}} \hc C(c, \id) \notag \\
		&\iso T\Bigpars{\bigpars{A(\id, a) \hc TJ} \lhom K} \hc C(c, \id) \xRar{T(\bar J \lhom \id) \hc \id} T\Bigpars{\bigpars{J \hc B(\id, b)} \lhom K} \hc C(c, \id) \notag \\
		&\Rar T\bigpars{J \hc B(\id, b)} \lhom \bigpars{TK \hc C(c, \id)} \xRar{\id \lhom \lambda\bar K} T\bigpars{J \hc B(\id, b)} \lhom \bigpars{A(a, \id) \hc K} \notag \\
		&\xRar{T_\hc \lhom \id} \bigpars{TJ \hc TB(\id, b)} \lhom \bigpars{A(a, \id) \hc K} \iso \bigpars{A(\id, a) \hc TJ \hc TB(\id, b)} \lhom K \notag \\
		&\xRar{(\bar J \hc \id) \lhom \id} \bigpars{J \hc B(\id, b) \hc TB(\id, b)} \lhom K \notag \\
		&\xRar{(\id \hc \rho\beta) \lhom \id} \bigpars{J \hc B(\id, b) \hc TB(\id, \mu_B)} \lhom K \notag \\
		&\iso TB(\mu_B, \id) \hc B(b, \id) \hc J \lhom K.
	\end{align}
	Our aim is to move $T(\bar J \lhom \id) \hc \id$ past $\id \lhom \lambda\bar K$, so that the rear of the composite above consists of the left hom of the right-hand side of the associativity axiom for $\bar J$ (\defref{definition:right colax promorphisms}) and the identity on $K$. To do this, consider the subchain $T\Bigpars{\bigpars{A(\id, a) \hc TJ} \lhom K} \hc C(c, \id) \Rar \bigpars{A(\id, a) \hc TJ \hc TB(\id, b)} \lhom K$ in the chain above, that starts with $T(\bar J \lhom \id) \hc \id$ and ends with the second isomorphism. Applying part (b) of \lemref{lemma:left hom facts} to the latter, followed by applying part (c) to the unlabelled cell, we find that this subchain is adjoint to the composite
	\begin{align*}
		&A(\id, a) \hc TJ \hc TB(\id, b) \hc T\Bigpars{\bigpars{A(\id, a) \hc TJ} \lhom K} \hc C(c, \id) \\
		&\xRar{\id \hc T_\hc \hc \id} A(\id, a) \hc T\Bigpars{J \hc B(\id, b) \hc \bigpars{A(\id, a) \hc TJ} \lhom K} \hc C(c, \id) \\
		&\xRar{\id \hc T(\bar J \hc \id) \hc \id} A(\id, a) \hc T\Bigpars{A(a, \id) \hc TJ \hc \bigpars{A(a, \id) \hc TJ} \lhom K} \hc C(c, \id) \\
		&\xRar{\id \hc T\ev \hc \id} A(\id, a) \hc TK \hc C(c, \id) \xRar{\id \hc \lambda\bar K} A(\id, a) \hc A(a, \id) \hc K \xRar{\ls a\eps_a \hc \id} K.
	\end{align*}
	Using the associativity and naturality of $T_\hc$ it is straightforward to show that the composite above is also the adjoint of the composite below, which therefore coincides with the subchain of \eqref{equation:256479} that we are considering.
	\begin{align*}
		&T\Bigpars{\bigpars{A(\id, a) \hc TJ} \lhom K} \hc C(c, \id) \Rar T\bigpars{A(\id, a) \hc TJ} \lhom \bigpars{TK \hc C(c, \id)} \\
		&\xRar{\id \lhom \lambda\bar K} T\bigpars{A(\id, a) \hc TJ} \lhom \bigpars{A(a, \id) \hc K} \iso \Bigpars{A(\id, a) \hc T\bigpars{A(\id, a) \hc TJ}} \lhom K \\
		&\xRar{(\id \hc T\bar J) \lhom \id} \Bigpars{A(\id, a) \hc T\bigpars{J \hc B(\id, b)}} \lhom K \\
		&\xRar{(\id \hc T_\hc) \lhom \id} \bigpars{A(\id, a) \hc TJ \hc TB(\id, b)} \lhom K
	\end{align*}
	As was our aim, substituting the above into \eqref{equation:256479} we obtain a chain whose rear forms the left hom of the right-hand side of the associativity axiom for $\bar J$. Replacing the latter with the left hom of the corresponding left-hand side we obtain
	\begin{align} \label{equation:411726}
		&T^2(J \lhom K) \hc TC(c, \id) \hc C(c, \id) \xRar{T_\hc \hc \id} T\bigpars{T(J \lhom K) \hc C(c, \id)} \hc C(c, \id) \notag \\
		&\Rar T\Bigpars{TJ \lhom \bigpars{TK \hc C(c, \id)}} \hc C(c, \id) \notag \\
		&\xRar{T(\id \lhom \lambda\bar K) \hc \id} T\Bigpars{TJ \lhom \bigpars{A(a, \id) \hc K}} \hc C(c, \id) \notag \\
		&\iso T\Bigpars{\bigpars{A(\id, a) \hc TJ} \lhom K} \hc C(c, \id) \Rar T\bigpars{A(\id, a) \hc TJ} \lhom \bigpars{TK \hc C(c, \id)} \notag \\
		&\xRar{\id \lhom \lambda\bar K} T\bigpars{A(\id, a) \hc TJ} \lhom \bigpars{A(a, \id) \hc K} \iso \Bigpars{A(\id, a) \hc T\bigpars{A(\id, a) \hc TJ}} \lhom K \notag \\
		&\xRar{(\id \hc T_\hc) \lhom \id} \bigpars{A(\id, a) \hc TA(\id, a) \hc T^2J} \lhom K \notag \\
		&\xRar{(\rho\alpha \hc \id) \lhom \id} \bigpars{A(\id, a) \hc TA(\id, \mu_A) \hc T^2J} \lhom K \notag \\
		&\xRar{(\id \hc \inv{(\rho\mu_J)}) \lhom \id} \bigpars{A(\id, a) \hc TJ \hc TB(\id, \mu_B)} \lhom K \notag \\
		&\xRar{(\bar J \hc \id) \lhom \id} \bigpars{J \hc B(\id, b) \hc TB(\id, \mu_B)} \lhom K \notag \\
		&\iso TB(\mu_B, \id) \hc B(b, \id) \hc J \lhom K.
	\end{align}
	
	We claim that the subchain consisting of the first eight cells here, with target $\bigpars{A(\id, a) \hc TA(\id, \mu_A) \hc T^2J} \lhom K$, equals the composite
	\begin{align} \label{equation:714342}
		T^2(J \lhom K)&{} \hc TC(c, \id) \hc C(c, \id) \Rar T^2 J \lhom \bigpars{T^2 K \hc TC(c, \id) \hc C(c, \id)} \notag \\
		&\xRar{\id \lhom (\lambda T\bar K \hc \id)} T^2J \lhom \bigpars{TA(a, \id) \hc TK \hc C(c, \id)} \notag \\
		&\xRar{\id \lhom (\id \hc \lambda\bar K)} T^2J \lhom \bigpars{TA(a, \id) \hc A(a, \id) \hc K} \notag \\
		&\xRar{\id \lhom (\lambda\alpha \hc \id)} T^2J \lhom \bigpars{TA(\mu_A, \id) \hc A(a, \id) \hc K} \notag \\
		&\iso \bigpars{A(\id, a) \hc TA(\id, \mu_A) \hc T^2J} \lhom K,
	\end{align}
	where the first cell is given by \propref{lax functors induce left hom coherence cell}, now for $F = T^2$. Notice that the composite above contains the right-hand side of the associativity axiom for $\lambda\bar K$ (\propref{left and right structure cells}). To see that the claim holds we again compute adjoints: that of the composite above, using \lemref{lemma:left hom facts}(b), equals
	\begin{align*}
		&A(\id, a) \hc TA(\id, \mu_A) \hc T^2 J \hc T^2(J \lhom K) \hc TC(c, \id) \hc C(c, \id) \\
		&\xRar{\id \hc T^2_\hc \hc \id} A(\id, a) \hc TA(\id, \mu_A) \hc T^2(J \hc J \lhom K) \hc TC(c, \id) \\
		&\xRar{\id \hc T^2\ev \hc \id} A(\id, a) \hc TA(\id, \mu_A) \hc T^2 K \hc TC(c, \id) \hc C(c, \id) \\
		&\xRar{\id \lambda T\bar K \hc \id} A(\id, a) \hc TA(\id, \mu_A) \hc TA(a, \id) \hc TK \hc C(c, \id) \\
		&\xRar{\id \hc \lambda\bar K} A(\id, a) \hc TA(\id, \mu_A) \hc TA(a, \id) \hc A(a, \id) \hc K \\
		&\xRar{\id \hc \lambda\alpha \hc \id} A(\id, a) \hc TA(\id, \mu_A) \hc TA(\mu_A, \id) \hc A(a, \id) \hc K \xRar{\ls{a \of \mu_A}\eps_{a \of \mu_A} \hc \id} K
	\end{align*}
	As we have seen before, in showing that the subdiagram III commutes, we have $\ls{a \of \mu_A}\eps_{a \of \mu_A} \of (\id \hc \lambda\alpha) = \ls{a \of Ta}\eps_{a \of Ta} \of (\rho\alpha \hc \id)$. After substituting this into the composite above, we move $\rho\alpha$ all the way to the front, and rewrite $\ls{a \of Ta}\eps_{a \of Ta} = \ls a\eps_a \of (\id \hc T\ls a\eps_a \hc \id) \of (\id \hc T_\hc \hc \id)$. Following this we also rewrite
	\begin{multline*}
		\cell{(T\ls a\eps_a \hc \id) \of (T_\hc \hc \id) \of (\id \hc \lambda T\bar K) \\
		= T(\ls a\eps_a \hc \id) \of T(\id \hc \lambda\bar K) \of T_\hc}{TA(\id, a) \hc T^2 K \hc TC(c, \id)}{TK},
	\end{multline*}
	where we have used that $\lambda T\bar K = \inv T_\hc \of T\lambda\bar K \of T_\hc$, see \propref{left and right structure cells}, so that the adjoint above is equal to
	\begin{align*}
		A(\id, a)&{} \hc TA(\id, \mu_A) \hc T^2 J \hc T^2(J \lhom K) \hc TC(c, \id) \hc C(c, \id) \\
		&\xRar{\rho\alpha \hc T^2_\hc \hc \id} A(\id, a) \hc TA(\id, a) \hc T^2(J \hc J \lhom K) \hc TC(c, \id) \hc C(c, \id) \\
		&\xRar{\id \hc T^2\ev \hc \id} A(\id, a) \hc TA(\id, a) \hc T^2 K \hc TC(c, \id) \hc C(c, \id) \\
		&\xRar{\id \hc T_\hc \hc \id} A(\id, a) \hc T\bigpars{A(\id, a) \hc TK \hc C(c, \id)} \hc C(c, \id) \\
		&\xRar{\id \hc T(\id \hc \lambda\bar K) \hc \id} A(\id, a) \hc T\bigpars{A(\id, a) \hc A(a, \id) \hc K} \hc C(c, \id) \\
		&\xRar{\id \hc T(\ls a\eps_a \hc \id) \hc \id} A(\id, a) \hc TK \hc C(c, \id) \\
		&\xRar{\id \hc \lambda\bar K} A(\id, a) \hc A(a, \id) \hc K \xRar{\ls a\eps_a \hc \id} K.
	\end{align*}
	On the other hand, computing the adjoint of the first eight cells of \eqref{equation:411726}, using parts (b) and (c) of \lemref{lemma:left hom facts}, we find
	\begin{align*}
		&A(\id, a) \hc TA(\id, \mu_A) \hc T^2 J \hc T^2(J \lhom K) \hc TC(c, \id) \hc C(c, \id) \\
		&\xRar{\rho\alpha \hc T_\hc \hc \id} A(\id, a) \hc TA(\id, a) \hc T^2J \hc T\bigpars{T(J \lhom K) \hc C(c, \id)} \hc C(c, \id) \\
		&\xRar{\id \hc T_\hc \hc \id} A(\id, a) \hc T\bigpars{A(\id, a) \hc TJ \hc T(J \lhom K) \hc C(c, \id)} \hc C(c, \id) \\
		&\xRar{\id \hc T(\id \hc T_\hc \id) \hc \id)} A(\id, a) \hc T\bigpars{A(\id, a) \hc T(J \hc J \lhom K) \hc C(c, \id)} \hc C(c, \id) \\
		&\xRar{\id \hc T(\id \hc T\ev \hc \id) \hc \id} A(\id, a) \hc T\bigpars{A(\id, a) \hc TK \hc C(c, \id)} \hc C(c, \id) \\
		&\xRar{\id \hc T(\id \hc \lambda\bar K) \hc \id} A(\id, a) \hc T\bigpars{A(\id, a) \hc A(a, \id) \hc K} \hc C(c, \id) \\
		&\xRar{\id \hc T(\ls a\eps_a \hc \id) \hc \id} A(\id, a) \hc TK \hc C(c, \id) \\
		&\xRar{\id \hc \lambda\bar K} A(\id, a) \hc A(a, \id) \hc K \xRar{\ls a\eps_a \hc \id} K.
	\end{align*}
	Using the associativity and naturality of $T_\hc$ it is easily seen that the two composites above coincide, which proves the claim.
	
	Next we replace in \eqref{equation:714342} the right-hand side of the associativity axiom for $\lambda\bar K$ by the corresponding left-hand side (see \propref{left and right structure cells}), and then substitute it for the first eight cells in \eqref{equation:411726}. This gives the composite
	\begin{align} \label{equation:457489}
		T^2(J \lhom K)&{} \hc TC(c, \id) \hc C(c, \id) \Rar T^2J \lhom \bigpars{T^2 K \hc TC(c, \id) \hc C(c, \id)} \notag \\
		&\xRar{\id \lhom (\id \hc \lambda\gamma)} T^2J \lhom \bigpars{T^2K \hc TC(\mu_C, \id) \hc C(c, \id)} \notag \\
		&\xRar{\id \lhom (\lambda\mu_K \hc \id)} T^2J \lhom \bigpars{TA(\mu_A, \id) \hc TK \hc C(c, \id)} \notag \\
		&\xRar{\id \lhom (\id \hc \lambda\bar K)} T^2J \lhom \bigpars{TA(\mu_A, \id) \hc A(a, \id) \hc K} \notag \\
		&\iso \bigpars{A(\id, a) \hc TA(\id, \mu_A) \hc T^2J} \lhom K \notag \\
		&\xRar{(\id \hc \inv{(\rho\mu_J)}) \lhom \id} \bigpars{A(\id, a) \hc TJ \hc TB(\id, \mu_B)} \lhom K \notag \\
		&\xRar{(\bar J \hc \id) \lhom \id} \bigpars{J \hc B(\id, b) \hc TB(\id, \mu_B)} \lhom K \notag \\
		&\iso TB(\mu_B, \id) \hc B(b, \id) \hc J \lhom K.
	\end{align}
	Remember that the composite above equals the right-hand side of the associativity axiom for $\lambda\ol{J \lhom K}$ (\propref{left and right structure cells}). On the other hand, the left-hand side is the composite
	\begin{align} \label{equation:151984}
		T^2(J \lhom K)&{} \hc TC(c, \id) \hc C(c, \id) \xRar{\id \hc \lambda\gamma} T^2(J \lhom K) \hc TC(\mu_C, \id) \hc C(c, \id) \notag \notag \\
		&\xRar{\lambda\mu_{J\lhom K} \hc \id} TB(\mu_B, \id) \hc T(J \lhom K) \hc C(c, \id) \notag \notag \\
		&\xRar{\id \hc \omega(\bar J, \bar K)} TB(\mu_B, \id) \hc \bigpars{J \hc B(\id, b)} \lhom K \notag \notag \\
		&\iso \bigpars{J \hc B(\id, b) \hc TB(\id, \mu_B)} \lhom K \notag \notag \\
		&\iso TB(\mu_B, \id) \hc B(b, \id) \hc J \lhom K,
	\end{align}
	where we have replaced the single isomorphism
	\begin{displaymath}
		TB(\mu_B, \id) \hc \bigpars{J \hc B(\id, b)} \lhom K \iso TB(\mu_B, \id) \hc B(b, \id) \hc J \lhom K
	\end{displaymath}
	by the pair of isomorphisms above (proving that these are equal is analogous to proving that the subdiagram II above commutes). Notice that both sides above have their final isomorphism in common; we will show that they coincide by proving that, after taking these isomorphisms away, the adjoints of the remaining subchains coincide. Computing the adjoint of the first four cells of the composite above, we obtain
	\begin{align} \label{equation:168498}
		J&{} \hc B(\id, b) \hc TB(\id, \mu_B) \hc T^2(J \lhom K) \hc TC(c, \id) \hc C(c, \id) \notag \\
		&\xRar{\bar J \hc \id \hc \lambda\gamma} A(\id, a) \hc TJ \hc TB(\id, \mu_B) \hc T^2(J \lhom K) \hc TC(\mu_C, \id) \hc C(c, \id) \notag \\
		&\xRar{\id \hc \lambda\mu_{J \lhom K} \hc \id} A(\id, a) \hc TJ \hc TB(\id, \mu_B) \hc TB(\mu_B, \id) \hc T(J \lhom K) \hc C(c, \id) \notag \\
		&\xRar{\id \hc \ls{\mu_B}\eps_{\mu_B} \hc \id} A(\id, a) \hc TJ \hc T(J \lhom K) \hc C(c, \id) \notag \\
		&\xRar{\id \hc T_\hc \hc \id} A(\id, a) \hc T(J \hc J \lhom K) \hc C(c, \id) \notag \\
		&\xRar{\id \hc T\ev \hc \id} A(\id, a) \hc TK \hc C(c, \id) \notag \\
		&\xRar{\id \hc \lambda\bar K} A(\id,a) \hc A(a, \id) \hc K \xRar{\ls a\eps_a \hc \id} K,
	\end{align}
	where we have used \lemref{lemma:left hom facts}(a) and the fact that $\omega(\bar J, \bar K)$ is adjoint to \eqref{equation:891123}. As before we have $(\ls{\mu_B}\eps_{\mu_B} \hc \id) \of (\id \hc \lambda\mu_{J \lhom K}) = (\id \hc \ls{\mu_C}\eps_{\mu_C}) \of (\rho\mu_{J \lhom K} \hc \id)$, where $\rho\mu_{J \lhom K}$ can be rewritten as follows. The composition axiom (\defref{definition:transformation}) and the naturality of $\mu$ imply that
	\begin{align} \label{equation:264886}
		\bigbrks{T^2 J&{} \hc T^2(J \lhom K) \xRar{\mu_J \hc \mu_{J \lhom K}} TJ \hc T(J \lhom K) \xRar{T_\hc} T(J \hc J \lhom K) \xRar{T\ev} TK} \notag \\
		&=\bigbrks{T^2 J \hc T^2(J \lhom K) \xRar{T^2_\hc} T^2(J \hc J \lhom K) \xRar{\mu_{J \hc J \lhom K}} T(J \hc J \lhom K) \xRar{T\ev} TK} \notag \\
		&=\bigbrks{T^2 J \hc T^2(J \lhom K) \xRar{T^2_\hc} T^2(J \hc J \lhom K) \xRar{T^2\ev} T^2 K \xRar{\mu_K} TK}
	\end{align}
	which, after applying $\rho$ and precomposing with $\inv{(\rho\mu_J)} \hc \id$, gives the identity
	\begin{multline*}
		(T\ev \hc \id) \of (T_\hc \hc \id) \of (\id \hc \rho\mu_{J\lhom K}) = \rho\mu_K \of (\id \hc T^2\ev) \of (\id \hc T^2_\hc) \of (\inv{(\rho\mu_J)} \hc \id)
	\end{multline*}
	of horizontal cells $TJ \hc TB(\id, \mu_B) \hc T^2(J \lhom K) \Rar TK \hc TC(\mu_C, \id)$, where we have also used the functoriality of $\rho$, see \propref{left and right cells}. Substituting this into \eqref{equation:168498} we obtain
	\begin{align*}
		J&{} \hc B(\id, b) \hc TB(\id, \mu_B) \hc T^2(J \lhom K) \hc TC(c, \id) \hc C(c, \id) \\
		&\xRar{\bar J \hc \id \hc \lambda\gamma} A(\id, a) \hc TJ \hc TB(\id, \mu_B) \hc T^2(J \lhom K) \hc TC(\mu_C, \id) \hc C(c, \id) \\
		&\xRar{\id \hc \inv{(\rho\mu_J)} \hc \id} A(\id, a) \hc TA(\id, \mu_A) \hc T^2 J \hc T^2(J \lhom K) \hc TC(\mu_C, \id) \hc C(c, \id) \\
		&\xRar{\id \hc T^2_\hc \hc \id} A(\id, a) \hc TA(\id, \mu_A) \hc T^2(J \hc J \lhom K) \hc TC(\mu_C, \id) \hc C(c, \id) \\
		&\xRar{\id \hc T^2\ev \hc \id} A(\id, a) \hc TA(\id, \mu_A) \hc T^2K \hc TC(\mu_C, \id) \hc C(c, \id) \\
		&\xRar{\id \hc \rho\mu_K \hc \id} A(\id, a) \hc TK \hc TC(\id, \mu_C) \hc TC(\mu_C, \id) \hc C(c, \id) \\
		&\xRar{\id \hc \ls{\mu_C}\eps_{\mu_C} \hc \id} A(\id, a) \hc TK \hc C(c, \id) \xRar{\id \hc \lambda\bar K} A(\id,a) \hc A(a, \id) \hc K \xRar{\ls a\eps_a \hc \id} K.
	\end{align*}
	On the other hand, the adjoint of the right-hand side \eqref{equation:457489}, after removing the final isomorphism, equals
	\begin{align*}
		J&{} \hc B(\id, b) \hc TB(\id, \mu_B) \hc T^2(J \lhom K) \hc TC(c, \id) \hc C(c, \id) \\
		&\xRar{\bar J \hc \id \hc \lambda\gamma} A(\id, a) \hc TJ \hc TB(\id, \mu_B) \hc T^2(J \lhom K) \hc TC(\mu_C, \id) \hc C(c, \id) \\
		&\xRar{\id \hc \inv{(\rho\mu_J)} \hc \id} A(\id, a) \hc TA(\id, \mu_A) \hc T^2 J \hc T^2(J \lhom K) \hc TC(\mu_C, \id) \hc C(c, \id) \\
		&\xRar{\id \hc T^2_\hc \hc \id} A(\id, a) \hc TA(\id, \mu_A) \hc T^2(J \hc J \lhom K) \hc TC(\mu_C, \id) \hc C(c, \id) \\
		&\xRar{\id \hc T^2\ev \hc \id} A(\id, a) \hc TA(\id, \mu_A) \hc T^2K \hc TC(\mu_C, \id) \hc C(c, \id) \\
		&\xRar{\id \hc \lambda\mu_K \hc \id} A(\id, a) \hc TA(\id, \mu_A) \hc TA(\mu_A, \id) \hc TK \hc C(c, \id) \\
		&\xRar{\id \hc \lambda\bar K} A(\id, a) \hc TA(\id, \mu_A) \hc TA(\mu_A, \id) \hc A(a, \id) \hc K \xRar{\ls{a \of \mu_A}\eps_{a \of \mu_A} \hc \id} K.
	\end{align*}
	That the two adjoints above coincide follows from $\ls{a \of \mu_A}\eps_{a \of \mu_A} = \ls a\eps_a \of (\id \hc \ls{\mu_A}\eps_{\mu_A} \hc \id)$ and $(\ls{\mu_A}\eps_{\mu_A} \hc \id) \of (\id \hc \lambda\mu_K) = (\id \hc \ls{\mu_C}\eps_{\mu_C}) \of (\rho\mu_K \hc \id)$. It follows that composites \eqref{equation:457489} and \eqref{equation:151984}, which equal respectively the right-hand and left-hand side of the associativity axiom for $\lambda\ol{J \lhom K}$, coincide, and we conclude that the lax structure cell $\ol{J \lhom K}$, as it was defined in \eqref{equation:left hom lax structure cell in appendix}, satisfies the associativity axiom.
	
	The unit axiom for $\lambda\ol{J \lhom K}$, given in \propref{left and right structure cells}, is proved similarly, as follows. Its left-hand side is the composite
	\begin{align} \label{equation:854143}
		&J \lhom K \xRar{\id \hc \lambda\gamma_0} J \lhom K \hc TC(\eta_C, \id) \hc C(c, \id) \notag \\
		&\xRar{\lambda\eta_{J \lhom K} \hc \id} TB(\eta_B, \id) \hc T(J \lhom K) \hc C(c, \id) \Rar TB(\eta_B, \id) \hc TJ \lhom \bigpars{TK \hc C(c, \id)} \notag \\
		&\xRar{\id \hc \id \lhom \lambda\bar K} TB(\eta_B, \id) \hc TJ \lhom \bigpars{A(a, \id) \hc K} \iso TB(\eta_B, \id) \hc \bigpars{A(\id, a) \hc TJ} \lhom K \notag \\
		&\xRar{\id \hc \bar J \lhom \id} TB(\eta_B, \id) \hc \bigpars{J \hc B(\id ,b)} \lhom K \iso \bigpars{J \hc B(\id, b) \hc TB(\id, \eta_B)} \lhom K \notag \\
		&\iso TB(\eta_B, \id) \hc B(b, \id) \hc J \lhom K
	\end{align}
	where, like in \eqref{equation:151984}, we have replaced the single isomorphism
	\begin{displaymath}
		TB(\eta_B, \id) \hc \bigpars{J \hc B(\id, b)} \lhom K \iso TB(\eta_B, \id) \hc B(b, \id) \hc J \lhom K
	\end{displaymath}
	by the pair of isomorphisms above. Using parts (a) and (b) of \lemref{lemma:left hom facts} we find that the adjoint of the composite above, after removing the last isomorphism, equals
	\begin{align*}
		&J \hc B(\id, b) \hc TB(\id, \eta_B) \hc J \lhom K \\
		&\xRar{\bar J \hc \id \hc \lambda\gamma_0} A(\id, a) \hc TJ \hc TB(\id, \eta_B) \hc J \lhom K \hc TC(\eta_C, \id) \hc C(c, \id) \\
		&\xRar{\id \hc \lambda\eta_{J \lhom K} \hc \id} A(\id, a) \hc TJ \hc TB(\id, \eta_B) \hc TB(\eta_B, \id) \hc T(J \lhom K) \hc C(c, \id) \\
		&\xRar{\id \hc \ls{\eta_B}\eps_{\eta_B} \hc \id} A(\id, a) \hc TJ \hc T(J \lhom K) \hc C(c, \id) \\
		&\xRar{\id \hc T_\hc \hc \id} A(\id, a) \hc T(J \hc J \lhom K) \hc C(c, \id) \xRar{\id \hc T\ev \hc \id} A(\id, a) \hc TK \hc C(c, \id) \\
		&\xRar{\id \hc \lambda\bar K} A(\id, a) \hc A(a, \id) \hc K \xRar{\ls a\eps_a \hc \id} K.
	\end{align*}
	As before, here $(\ls{\eta_B} \eps_{\eta_B} \hc \id) \of (\id \hc \lambda\eta_{J\lhom K}) = (\id \hc \ls{\eta_C} \eps_{\eta_C}) \of (\rho\eta_{J\lhom K} \hc \id)$ and, like \eqref{equation:264886}, we have $\cell{T\ev \of T_\hc \of (\eta_J \hc \eta_{J \lhom K}) = \eta_K \of \ev}{J \hc J \lhom K}{TK}$, so that the composite above can be rewritten as
	\begin{align*}
		J&{} \hc B(\id, b) \hc TB(\id, \eta_B) \hc J \lhom K \xRar{\bar J \hc \id} A(\id, a) \hc TJ \hc TB(\id, \eta_B) \hc J \lhom K \\
		&\xRar{\id \hc \inv{(\rho\eta_J)} \hc \id} A(\id, a) \hc TA(\id, \eta_A) \hc J \hc J \lhom K \\
		&\xRar{\id \hc \ev \hc \lambda\gamma_0} A(\id, a) \hc TA(\id, \eta_A) \hc K \hc TC(\eta_C, \id) \hc C(c, \id) \\
		&\xRar{\id \hc \rho\eta_K \hc \id} A(\id, a) \hc TK \hc TC(\id, \eta_C) \hc TC(\eta_C, \id) \hc C(c, \id) \\
		&\xRar{\id \hc \ls{\eta_C} \eps_{\eta_C} \hc \id} A(\id, a) \hc TK \hc C(c, \id) \xRar{\id \hc \lambda\bar K} A(\id, a) \hc A(a, \id) \hc K \xRar{\ls a\eps_a \hc \id} K.
	\end{align*}
	Here we can first replace $(\id \hc \ls{\eta_C}\eps_{\eta_C}) \of (\rho\eta_K \hc \id)$ by $(\ls{\eta_A}\eps_{\eta_A} \hc \id) \of (\id \hc \lambda\eta_K)$ and then $\ls a\eps_a \of (\id \hc \ls{\eta_A}\eps_{\eta_A} \hc \id)$ by $\ls{a \of \eta_A} \eps_{a \of \eta_A}$, thus obtaining
	\begin{align*}
		J&{} \hc B(\id, b) \hc TB(\id, \eta_B) \hc J \lhom K \xRar{\bar J \hc \id} A(\id, a) \hc TJ \hc TB(\id, \eta_B) \hc J \lhom K \\
		&\xRar{\id \hc \inv{(\rho\eta_J)} \hc \id} A(\id, a) \hc TA(\id, \eta_A) \hc J \hc J \lhom K \\
		&\xRar{\id \hc \ev \hc \lambda\gamma_0} A(\id, a) \hc TA(\id, \eta_A) \hc K \hc TC(\eta_C, \id) \hc C(c, \id) \\
		&\xRar{\id \hc \lambda\eta_K \hc \id} A(\id, a) \hc TA(\id, \eta_A) \hc TA(\eta_A, \id) \hc TK \hc C(c, \id) \\
		&\xRar{\id \hc \lambda\bar K} A(\id, a) \hc TA(\id, \eta_A) \hc TA(\eta_A, \id) \hc A(a, \id) \hc K \xRar{\ls{a \of \eta_A}\eps_{a \of \eta_A} \hc \id} K,
	\end{align*}
	Now the unit axiom for $\lambda\bar K$(\propref{left and right structure cells}) means that we can substitute $\lambda\alpha_0$ for the composite of $\lambda\bar K$, $\lambda\eta_K$ and $\lambda\gamma_0$. Following this, we can rewrite $\ls{a \of \eta_A}\eps_{a \of \eta_A} \of (\id \hc \lambda\alpha_0) = \rho\alpha_0$ and then use the unit axiom for $\bar J$ (\defref{definition:right colax promorphisms}) to substitute $\rho\beta_0$ for the composite of $\rho\alpha_0$, $\inv{(\rho\eta_J)}$ and $\bar J$. We conclude that the composite above equals
	\begin{displaymath}
		J \hc B(\id, b) \hc TB(\id, \eta_B) \hc J \lhom K \xRar{\id \hc \rho\beta_0 \hc \id} J \hc J \lhom K \xRar{\ev} K,
	\end{displaymath}
	which coincides with the adjoint of
	\begin{displaymath}
		J \lhom K \xRar{\lambda\beta_0 \hc \id} TB(\eta_B, \id) \hc B(b, \id) \hc J \lhom K \iso \bigpars{J \hc B(\id, b) \hc TB(\id, \eta_B)} \lhom K.
	\end{displaymath}
	From this we conclude that $\lambda\beta_0 \hc \id$ equals the composite \eqref{equation:854143}, thus proving that the unit axiom for $\lambda\ol{J \lhom K}$ holds. This completes the proof.
\end{proof}

	\backmatter
  \bibliographystyle{halpha}
	\bibliography{main}

\begin{thebibliography}{MLM92}

\bibitem[Bat98]{Batanin98}
M.~A. Batanin.
\newblock Monoidal globular categories as a natural environment for the theory
  of weak {$n$}-categories.
\newblock {\em Advances in Mathematics}, 136(1):39--103, 1998.

\bibitem[CS10]{Cruttwell-Shulman10}
G.~S.~H. Cruttwell and M.~A. Shulman.
\newblock A unified framework for generalized multicategories.
\newblock {\em Theory and Applications of Categories}, 24(21):580--655, 2010.

\bibitem[Day70]{Day70}
B.~Day.
\newblock On closed categories of functors.
\newblock In {\em Reports of the {M}idwest {C}ategory {S}eminar, {IV}}, volume
  137 of {\em Lecture Notes in Mathematics}, pages 1--38. Springer, Berlin,
  1970.

\bibitem[DPP06]{Dawson-Pare-Pronk06}
R.~J.~Macg. Dawson, R.~Par{\'e}, and D.~A. Pronk.
\newblock Paths in double categories.
\newblock {\em Theory and Applications of Categories}, 16(18):460--521, 2006.

\bibitem[Dub70]{Dubuc70}
E.~J. Dubuc.
\newblock {\em Kan extensions in enriched category theory}, volume 145 of {\em
  Lecture Notes in Mathematics}.
\newblock Springer-Verlag, Berlin, 1970.

\bibitem[Get09]{Getzler09}
E.~Getzler.
\newblock Operads revisited.
\newblock In {\em Algebra, arithmetic, and geometry: in honor of {Y}u. {I}.
  {M}anin. {V}olume {I}}, volume 269 of {\em Progress in Mathematics}, pages
  675--698. Birkh\"auser Boston Inc., 2009.

\bibitem[GP99]{Grandis-Pare99}
M.~Grandis and R.~Par{\'e}.
\newblock Limits in double categories.
\newblock {\em Cahiers de Topologie et G\'eom\'etrie Diff\'erentielle
  Cat\'egoriques}, 40(3):162--220, 1999.

\bibitem[GP04]{Grandis-Pare04}
M.~Grandis and R.~Par{\'e}.
\newblock Adjoint for double categories.
\newblock {\em Cahiers de Topologie et G\'eom\'etrie Diff\'erentielle
  Cat\'egoriques}, 45(3):193--240, 2004.

\bibitem[GP07]{Grandis-Pare07}
M.~Grandis and R.~Par{\'e}.
\newblock Lax {K}an extensions for double categories (on weak double
  categories, part {IV}).
\newblock {\em Cahiers de Topologie et G\'eom\'etrie Diff\'erentielle
  Cat\'egoriques}, 48(3):163--199, 2007.

\bibitem[GP08]{Grandis-Pare08}
M.~Grandis and R.~Par{\'e}.
\newblock Kan extensions in double categories (on weak double categories, part
  {III}).
\newblock {\em Theory and Applications of Categories}, 20(8):152--185, 2008.

\bibitem[Kel82]{Kelly82}
G.~M. Kelly.
\newblock {\em Basic concepts of enriched category theory}, volume~64 of {\em
  London Mathematical Society Lecture Note Series}.
\newblock Cambridge University Press, Cambridge, 1982.

\bibitem[Lac02]{Lack02}
S.~Lack.
\newblock Codescent objects and coherence.
\newblock {\em Journal of Pure and Applied Algebra}, 175(1-3):223--241, 2002.
\newblock Special volume celebrating the 70th birthday of Professor Max Kelly.

\bibitem[Lac05]{Lack05}
S.~Lack.
\newblock Limits for lax morphisms.
\newblock {\em Applied Categorical Structures}, 13(3):189--203, 2005.

\bibitem[Lei04]{Leinster04}
T.~Leinster.
\newblock {\em Higher operads, higher categories}, volume 298 of {\em London
  Mathematical Society Lecture Note Series}.
\newblock Cambridge University Press, Cambridge, 2004.

\bibitem[LS12]{Lack-Shulman12}
S.~Lack and M.~Shulman.
\newblock Enhanced 2-categories and limits for lax morphisms.
\newblock {\em Advances in Mathematics}, 229(1):294--356, 2012.

\bibitem[Man69]{Manes69}
E.~Manes.
\newblock A triple theoretic construction of compact algebras.
\newblock In {\em Seminar on {T}riples and {C}ategorical {H}omology {T}heory
  ({ETH}, {Z}\"urich, 1966/67)}, pages 91--118. Springer, Berlin, 1969.

\bibitem[ML98]{MacLane98}
S.~Mac~Lane.
\newblock {\em Categories for the working mathematician}, volume~5 of {\em
  Graduate Texts in Mathematics}.
\newblock Springer-Verlag, New York, second edition, 1998.

\bibitem[MLM92]{MacLane-Moerdijk92}
S.~Mac~Lane and I.~Moerdijk.
\newblock {\em Sheaves in geometry and logic}.
\newblock Universitext. Springer-Verlag, New York, 1992.

\bibitem[MT08]{Mellies-Tabareau08}
P.-A. Melli{\`e}s and N.~Tabareau.
\newblock Free models of {$T$}-algebraic theories computed as {K}an extensions.
\newblock Unpublished, available as
  \url{http://hal.archives-ouvertes.fr/hal-00339331/PDF/free-models.pdf}, 2008.

\bibitem[Rez96]{Rezk96}
C.~W. Rezk.
\newblock {\em Spaces of algebra structures and cohomology of operads}.
\newblock PhD thesis, Massachusetts Institute of Technology, 1996.

\bibitem[Shu08]{Shulman08}
M.~A. Shulman.
\newblock Framed bicategories and monoidal fibrations.
\newblock {\em Theory and Applications of Categories}, 20(18):650--738, 2008.

\bibitem[Shu10]{nLab_on_double_profunctors}
M.~A. Shulman.
\newblock Double profunctor, 2010.
\newblock Fourth revision of the $n$Lab article available at
  \url{http://nlab.mathforge.org/nlab/show/double+profunctor}.

\bibitem[Shu11]{nLab_on_equipments}
M.~A. Shulman.
\newblock 2-category equipped with proarrows, 2011.
\newblock 24th revision of the $n$Lab article available at
  \url{http://nlab.mathforge.org/nlab/show/2-category+equipped+with+proarrows}.

\bibitem[Str74]{Street74}
R.~Street.
\newblock Fibrations and {Y}oneda's lemma in a {$2$}-category.
\newblock In {\em Proceedings of the {S}ydney {C}ategory {T}heory {S}eminar,
  1972/1973}, volume 420 of {\em Lecture Notes in Mathematics}, pages 104--133.
  Springer, Berlin, 1974.

\bibitem[Tak71]{Takeuchi71}
M.~Takeuchi.
\newblock Free {H}opf algebras generated by coalgebras.
\newblock {\em Journal of the Mathematical Society of Japan}, 23:561--582,
  1971.

\bibitem[Tho09]{Tholen09}
W.~Tholen.
\newblock Ordered topological structures.
\newblock {\em Topology and its Applications}, 156(12):2148--2157, 2009.

\bibitem[Wad08]{Wadsley08}
S.~Wadsley, 2008.
\newblock Unpublished note, available as
  \url{http://www.dpmms.cam.ac.uk/~sjw47/PROP2.pdf}.

\bibitem[Woo82]{Wood82}
R.~J. Wood.
\newblock Abstract proarrows {I}.
\newblock {\em Cahiers de Topologie et G\'eom\'etrie Diff\'erentielle},
  23(3):279--290, 1982.

\end{thebibliography}
\end{document}